\documentclass[oneside,english,reqno]{amsart}
\usepackage[T1]{fontenc}
\usepackage[utf8]{inputenc}
\usepackage{geometry}
\usepackage{color}
\usepackage{babel}
\usepackage{mathrsfs}
\usepackage{mathtools}
\usepackage{enumitem}
\usepackage{amstext}
\usepackage{amsthm}
\usepackage{amssymb}
\usepackage[unicode=true,pdfusetitle,
 bookmarks=true,bookmarksnumbered=false,bookmarksopen=false,
 breaklinks=false,pdfborder={0 0 1},backref=false,colorlinks=false]
 {hyperref}

\makeatletter
\numberwithin{equation}{section}
\numberwithin{figure}{section}
  \theoremstyle{plain}
  \newtheorem*{thm*}{\protect\theoremname}
\theoremstyle{plain}
\newtheorem{thm}{\protect\theoremname}[section]
  \theoremstyle{definition}
  \newtheorem{defn}[thm]{\protect\definitionname}
  \theoremstyle{remark}
  \newtheorem*{rem*}{\protect\remarkname}
  \theoremstyle{plain}
  \newtheorem{assumption}[thm]{\protect\assumptionname}
  \theoremstyle{plain}
  \newtheorem{prop}[thm]{\protect\propositionname}
  \theoremstyle{plain}
  \newtheorem{lem}[thm]{\protect\lemmaname}
  \theoremstyle{plain}
  \newtheorem{cor}[thm]{\protect\corollaryname}
  \theoremstyle{remark}
  \newtheorem{rem}[thm]{\protect\remarkname}
 \newlist{casenv}{enumerate}{4}
 \setlist[casenv]{leftmargin=*,align=left,widest={iiii}}
 \setlist[casenv,1]{label={{\itshape\ \casename} \arabic*.},ref=\arabic*}
 \setlist[casenv,2]{label={{\itshape\ \casename} \roman*.},ref=\roman*}
 \setlist[casenv,3]{label={{\itshape\ \casename\ \alph*.}},ref=\alph*}
 \setlist[casenv,4]{label={{\itshape\ \casename} \arabic*.},ref=\arabic*}
  \theoremstyle{definition}
  \newtheorem{example}[thm]{\protect\examplename}
  \theoremstyle{plain}
  \newtheorem*{lem*}{\protect\lemmaname}

\usepackage[T1]{fontenc}
\usepackage{dsfont}

\usepackage{epigraph}

\let\emptyset\varnothing

\DeclareMathOperator*{\esssup}{ess\,sup}

\theoremstyle{remark}
\newtheorem{remx}[thm]{Remark}
\renewenvironment{rem}
  {\pushQED{\qed}\remx}
  {\popQED\endremx}

\theoremstyle{remark}
\newtheorem*{remxx}{Remark}
\renewenvironment{rem*}
  {\pushQED{\qed}\remxx}
  {\popQED\endremxx}

\theoremstyle{definition}
\newtheorem{examplex}[thm]{Example}
\renewenvironment{example}
  {\pushQED{\qed}\examplex}
  {\popQED\endexamplex}

\theoremstyle{definition}
\newtheorem{assumex}[thm]{Assumption}
\renewenvironment{assumption}
  {\pushQED{\qed}\assumex}
  {\popQED\endassumex}

\theoremstyle{definition}
\newtheorem{defnx}[thm]{Definition}
\renewenvironment{defn}
  {\pushQED{\qed}\defnx}
  {\popQED\enddefnx}

\theoremstyle{plain}
\newtheorem{lemx}[thm]{Lemma}
\renewenvironment{lem}
  {\pushQED{\qed}\lemx}
  {\popQED\endlemx}

\theoremstyle{plain}
\newtheorem{corx}[thm]{Corollary}
\renewenvironment{cor}
  {\pushQED{\qed}\corx}
  {\popQED\endcorx}

\theoremstyle{plain}
\newtheorem{propx}[thm]{Proposition}
\renewenvironment{prop}
  {\pushQED{\qed}\propx}
  {\popQED\endpropx}

\theoremstyle{definition}
\newtheorem{thmx}[thm]{Theorem}
\renewenvironment{thm}
  {\pushQED{\qed}\thmx}
  {\popQED\endthmx}

\usepackage[ddmmyyyy,hhmmss]{datetime}
\newdateformat{daymonthyeardate}{%
  \THEDAY.\THEMONTH.\THEYEAR}
\usepackage{fancyhdr}
\makeatletter

\newsavebox{\@linebox}
\savebox{\@linebox}[3em][t]{\parbox[t]{3em}{%
  \@tempcnta\@ne\relax
  \loop{\underline{\scriptsize\the\@tempcnta}}\\
    \advance\@tempcnta by \@ne\ifnum\@tempcnta<55\repeat}}

\makeatother

\makeatother

  \providecommand{\assumptionname}{Assumption}
  \providecommand{\corollaryname}{Corollary}
  \providecommand{\definitionname}{Definition}
  \providecommand{\examplename}{Example}
  \providecommand{\lemmaname}{Lemma}
  \providecommand{\propositionname}{Proposition}
  \providecommand{\remarkname}{Remark}
  \providecommand{\theoremname}{Theorem}
 \providecommand{\casename}{Case}
\providecommand{\theoremname}{Theorem}

\begin{document}

\title{Analysis vs.\@ Synthesis Sparsity for \texorpdfstring{{\larger[3]$\alpha$}}{α}-Shearlets}

\date{}{}


\author{Felix Voigtlaender, Anne Pein}

\global\long\def\vertiii#1{{\left\vert \kern-0.25ex  \left\vert \kern-0.25ex  \left\vert #1\right\vert \kern-0.25ex  \right\vert \kern-0.25ex  \right\vert }}
\global\long\def\essup{\esssup}
\global\long\def\with{\,\middle|\,}
\global\long\def\DecompSp#1#2#3#4{{\mathcal{D}\left({#1},L_{#4}^{#2},{#3}\right)}}
\global\long\def\FourierDecompSp#1#2#3#4{{\mathcal{D}_{\mathcal{F}}\left({#1},L_{#4}^{#2},{#3}\right)}}
\global\long\def\BAPUFourierDecompSp#1#2#3#4{{\mathcal{D}_{\mathcal{F},{#4}}\left({#1},L^{#2},{#3}\right)}}
\global\long\def\R{\mathbb{R}}
\global\long\def\Compl{\mathbb{C}}
\global\long\def\N{\mathbb{N}}
\global\long\def\Z{\mathbb{Z}}
\global\long\def\CalQ{\mathcal{Q}}
\global\long\def\CalP{\mathcal{P}}
\global\long\def\CalR{\mathcal{R}}
\global\long\def\CalS{\mathcal{S}}
\global\long\def\CalD{\mathcal{D}}
\global\long\def\BesovInhom#1#2#3{\mathcal{B}_{#3}^{#1,#2}}
\global\long\def\BesovHom#1#2#3{\dot{\mathcal{B}}_{#3}^{#1,#2}}
\global\long\def\dimension{d}
\global\long\def\ModSpace#1#2#3{M_{#3}^{#1,#2}}
\global\long\def\AlphaModSpace#1#2#3#4{M_{#3,#4}^{#1,#2}}
\global\long\def\ShearletSmoothness#1#2#3{S_{#3}^{#1,#2}}
\global\long\def\GL{\mathrm{GL}}
\global\long\def\CalO{\mathcal{O}}
\global\long\def\Coorbit{\mathrm{Co}}
\global\long\def\d{\operatorname{d}}
\global\long\def\TestFunctionSpace#1{C_{c}^{\infty}\left(#1\right)}
\global\long\def\DistributionSpace#1{\mathcal{D}'\left(#1\right)}
\global\long\def\SpaceTestFunctions#1{Z\left(#1\right)}
\global\long\def\SpaceReservoir#1{Z'\left(#1\right)}
\global\long\def\Schwartz{\mathcal{S}}
\global\long\def\Fourier{\mathcal{F}}
\global\long\def\supp{\operatorname{supp}}
\global\long\def\dist{\operatorname{dist}}
\global\long\def\Xhookrightarrow#1{\xhookrightarrow{#1}}
\global\long\def\Xmapsto#1{\xmapsto{#1}}
\global\long\def\Indicator{{\mathds{1}}}
\global\long\def\identity{\operatorname{id}}
\global\long\def\RaisedSup#1{\raisebox{0.15cm}{\mbox{\ensuremath{{\displaystyle \sup_{#1}}}}}}
\global\long\def\Raised#1{\raisebox{0.12cm}{\mbox{\ensuremath{#1}}}}
\global\long\def\osc#1{\operatorname{osc}_{#1}}
\global\long\def\mybullet{\bullet}
\global\long\def\uin{\left(n,m,\delta\right)}
\global\long\def\cin{\left(n,m,\varepsilon,\delta\right)}

\begin{abstract}
There are two notions of sparsity associated to a frame $\Psi=\left(\psi_{i}\right)_{i\in I}$:
\emph{Analysis sparsity} of $f$ means that the analysis coefficients
$\left(\left\langle f,\,\psi_{i}\right\rangle \right)_{i\in I}$ are
sparse, while \emph{synthesis sparsity} means that we can write $f=\sum_{i\in I}c_{i}\psi_{i}$
with sparse synthesis coefficients $\left(c_{i}\right)_{i\in I}$.
Here, \emph{sparsity} of a sequence $c=\left(c_{i}\right)_{i\in I}$
means $c\in\ell^{p}\left(I\right)$ for a given $p<2$. We show that
both notions of sparsity coincide if $\Psi={\rm SH}\left(\varphi,\psi;\,\delta\right)=\left(\psi_{i}\right)_{i\in I}$
is a discrete (cone-adapted) shearlet frame with sufficiently nice
generators $\varphi,\psi$ and sufficiently small sampling density
$\delta>0$. The required 'niceness' of $\varphi,\psi$ is explicitly
quantified in terms of Fourier-decay and vanishing moment conditions.
In addition to $\ell^{p}$-sparsity, we even allow \emph{weighted}
$\ell^{p}$-spaces $\ell_{w^{s}}^{p}$ as a sparsity measure, with
weights of the form $w^{s}=\left(2^{js}\right)_{\left(j,\ell,\delta,k\right)}$
where $j$ encodes the \emph{scale} of the corresponding shearlet
elements.

More precisely, we show that the \emph{shearlet smoothness spaces}
$\mathscr{S}_{s}^{p,q}\left(\R^{2}\right)$ introduced by Labate et
al\@.\@ \emph{simultaneously} characterize analysis and synthesis
sparsity with respect to a shearlet frame, in the sense that—for suitable
$\varphi,\psi,\delta$—the following are equivalent: 1) $f\in\mathscr{S}_{s+\frac{3}{2}\left(p^{-1}-2^{-1}\right)}^{p,p}\left(\R^{2}\right)$;
\quad{}2) $\left(\left\langle f,\psi_{i}\right\rangle \right)_{i\in I}\in\ell_{w^{s}}^{p}$;
\quad{}3) $f=\sum_{i\in I}c_{i}\psi_{i}$ for suitable coefficients
$c=\left(c_{i}\right)_{i\in I}\in\ell_{w^{s}}^{p}$.

As an application, we prove that \emph{shearlets yield (almost) optimal
approximation rates for the class of cartoon-like functions}: If $f$
is cartoon-like and $\varepsilon>0$, then $\left\Vert f-f_{N}\right\Vert _{L^{2}}\lesssim N^{-\left(1-\varepsilon\right)}$,
where $f_{N}$ is a linear combination of $N$ shearlets. This might
appear to be a well-known statement, but an inspection of the existing
proofs reveals that these only establish \emph{analysis sparsity}
of cartoon-like functions, which implies $\left\Vert f-g_{N}\right\Vert _{L^{2}}\lesssim N^{-1}\cdot\left(1+\log N\right)^{3/2}$,
where $g_{N}$ is a linear combination of $N$ elements of the \emph{dual
frame} $\widetilde{\Psi}$ to the shearlet frame $\Psi$. This is
not completely satisfying, since only limited knowledge about the
structure and properties of $\widetilde{\Psi}$ is available.

In addition to classical shearlets, we also consider more general
$\alpha$-shearlet systems. For these, the parabolic scaling is replaced
by $\alpha$-parabolic scaling. The resulting systems range from ridgelet-like
systems (for $\alpha=0$) over classical shearlets ($\alpha=\frac{1}{2}$)
to wavelet-like systems ($\alpha=1$). In this more general case,
the shearlet smoothness spaces $\mathscr{S}_{s}^{p,q}\left(\R^{2}\right)$
have to be replaced by the \emph{$\alpha$-shearlet smoothness spaces}
$\mathscr{S}_{\alpha,s}^{p,q}\left(\R^{2}\right)$. We completely
characterize the existence of embeddings between these spaces for
\emph{different} values of $\alpha$. This allows us to decide whether
sparsity with respect to $\alpha_{1}$-shearlets implies sparsity
with respect to $\alpha_{2}$-shearlets, even for $\alpha_{1}\neq\alpha_{2}$.
\end{abstract}

\maketitle

\section{Introduction}

\label{sec:Introduction}A cone-adapted \textbf{shearlet system}\cite{CompactlySupportedShearletFrames,CompactlySupportedShearletsAreOptimallySparse,OptimallySparseMultidimensionalRepresentationUsingShearlets,shearlet_book,ConeAdaptedShearletFirstPaper}
${\rm SH}\left(\varphi,\psi,\theta;\delta\right)$ is a \emph{directional}
multiscale system in $L^{2}\left(\R^{2}\right)$ that is obtained
by applying suitable translations, shearings and parabolic dilations
to the generators $\varphi,\psi,\theta$. The shearings are utilized
to obtain elements with different \emph{orientations}; precisely,
the number of different orientations on scale $j$ is approximately
$2^{j/2}$, in stark contrast to wavelet-like systems which only employ
a constant number of directions per scale. We refer to Definition
\ref{def:AlphaShearletSystem} for a more precise description of shearlet
systems.

One of the most celebrated properties of shearlets is their ability
to provide ``optimally sparse approximations'' for functions that
are governed by directional features like edges. This can be made
more precise by introducing the class $\mathcal{E}^{2}\left(\R^{2}\right)$
of \textbf{$C^{2}$-cartoon-like functions}; roughly, these are all
compactly supported functions that are \emph{$C^{2}$ away from a
$C^{2}$ edge}\cite{OptimallySparseMultidimensionalRepresentationUsingShearlets}.
More rigorously, the class $\mathcal{E}^{2}\left(\R^{2}\right)$ consists
of all functions $f$ that can be written as $f=f_{1}+\Indicator_{B}\cdot f_{2}$
with $f_{1},f_{2}\in C_{c}^{2}\left(\R^{2}\right)$ and a compact
set $B\subset\R^{2}$ whose boundary $\partial B$ is a $C^{2}$ Jordan
curve; see also Definition \ref{def:CartoonLikeFunction} for a completely
formal description of the class of cartoon-like functions. With this
notion, the (almost) optimal sparse approximation of cartoon-like
functions as understood in \cite{OptimallySparseMultidimensionalRepresentationUsingShearlets,CompactlySupportedShearletsAreOptimallySparse}
means that
\begin{equation}
\left\Vert f-f_{N}\right\Vert _{L^{2}}\lesssim N^{-1}\cdot\left(1+\log N\right)^{3/2}\qquad\forall N\in\N\text{ and }f\in\mathcal{E}^{2}\left(\R^{2}\right).\label{eq:IntroductionShearletApproximationRate}
\end{equation}
Here, the \textbf{$N$-term approximation} $f_{N}$ is obtained by
retaining only the $N$ largest coefficients in the expansion $f=\sum_{i\in I}\left\langle f,\psi_{i}\right\rangle \widetilde{\psi_{i}}$,
where $\widetilde{\Psi}=\left(\smash{\widetilde{\psi_{i}}}\right)_{i\in I}$
is a dual frame for the shearlet frame $\Psi={\rm SH}\left(\varphi,\psi,\theta;\delta\right)=\left(\psi_{i}\right)_{i\in I}$.
Formally, this means $f_{N}=\sum_{i\in I_{N}}\left\langle f,\psi_{i}\right\rangle \widetilde{\psi_{i}}$,
where the set $I_{N}\subset I$ satisfies $\left|I_{N}\right|=N$
and $\left|\left\langle f,\psi_{i}\right\rangle \right|\geq\left|\left\langle f,\psi_{j}\right\rangle \right|$
for all $i\in I_{N}$ and $j\in I\setminus I_{N}$.

One can even show that the approximation rate in equation \eqref{eq:IntroductionShearletApproximationRate}
is optimal up to log factors; i.e., up to log factors, no reasonable
system $\left(\varrho_{n}\right)_{n\in\N}$ can achieve a better approximation
rate for the whole class $\mathcal{E}^{2}\left(\R^{2}\right)$. The
restriction to ``reasonable'' systems is made to exclude pathological
cases like dense subsets of $L^{2}\left(\R^{2}\right)$ and involves
a restriction of the \emph{search depth}: The $N$-term approximation
$f_{N}=\sum_{n\in J_{N}}c_{n}\varrho_{n}$ has to satisfy $\left|J_{N}\right|=N$
and furthermore $J_{N}\subset\left\{ 1,\dots,\pi\left(N\right)\right\} $
for a fixed polynomial $\pi$. For more details on this restriction,
we refer to \cite[Section 2.1.1]{CartoonApproximationWithAlphaCurvelets}.

The approximation rate achieved by shearlets is precisely the same
as that obtained by (second generation) \textbf{curvelets}\cite{CandesDonohoCurvelets}.
Note, however, that the construction of curvelets in \cite{CandesDonohoCurvelets}
uses \emph{bandlimited} frame elements, while shearlet frames can
be chosen to have compact support\cite{CompactlySupportedShearletsAreOptimallySparse,CompactlySupportedShearletFrames}.
A frame with compactly supported elements is potentially advantageous
for implementations, but also for theoretical considerations, since
localization arguments are highly simplified and since compactly supported
frames can be adapted to frames on bounded domains, see e.g.\@ \cite{AnisotropicMultiscaleSystemsOnBoundedDomains,ShearletFramesForSobolevSpaces}.
A further advantage of shearlets over curvelets is that curvelets
are defined using \emph{rotations}, while shearlets employ \emph{shearings}
to change the orientation; in contrast to rotations, these shearings
leave the digital grid $\Z^{2}$ invariant, which is beneficial for
implementations.

\subsection{Cartoon approximation \emph{by shearlets}}

Despite its great utility, the approximation result in equation \eqref{eq:IntroductionShearletApproximationRate}
has one remaining issue: It yields a rapid approximation of $f$ by
a linear combination of $N$ elements of the \emph{dual frame} $\widetilde{\Psi}$
of the shearlet frame $\Psi$, \emph{not} by a linear combination
of $N$ elements of $\Psi$ itself. If $\Psi$ is a tight frame, this
is no problem, but the only known construction of tight cone-adapted
shearlet frames uses \emph{bandlimited} generators. In case of a \emph{non-tight}
cone-adapted shearlet frame, the only knowledge about $\widetilde{\Psi}$
that is available is that $\widetilde{\Psi}$ is a frame with dual
$\Psi$; but nothing seems to be known\cite{IntrinsicLocalizationOfAnisotropicFrames}
about the support, the smoothness, the decay or the frequency localization
of the elements of $\widetilde{\Psi}$. Thus, it is highly desirable
to have an approximation result similar to equation \eqref{eq:IntroductionShearletApproximationRate},
but with $f_{N}$ being a linear combination of $N$ elements \emph{of
the shearlet frame $\Psi={\rm SH}\left(\varphi,\psi,\theta;\delta\right)$
itself}.

We will provide such a result by showing that \textbf{analysis sparsity}
with respect to a (suitable) shearlet frame ${\rm SH}\left(\varphi,\psi,\theta;\delta\right)$
is \emph{equivalent} to \textbf{synthesis sparsity} with respect to
the same frame, cf.\@ Theorem \ref{thm:AnalysisAndSynthesisSparsityAreEquivalent}.
Here, \emph{analysis sparsity} with respect to a frame $\Psi=\left(\psi_{i}\right)_{i\in I}$
means that the \textbf{analysis coefficients} $A_{\Psi}f=\left(\left\langle f,\psi_{i}\right\rangle \right)_{i\in I}$
are \emph{sparse}, i.e., they satisfy $A_{\Psi}f\in\ell^{p}\left(I\right)$
for some fixed $p\in\left(0,2\right)$. Note that an arbitrary function
$f\in L^{2}\left(\R^{2}\right)$ always satisfies $A_{\Psi}f\in\ell^{2}\left(I\right)$
by the frame property. \emph{Synthesis sparsity} means that we can
write $f=S_{\Psi}c=\sum_{i\in I}c_{i}\psi_{i}$ for a \emph{sparse}
sequence $c=\left(c_{i}\right)_{i\in I}$, i.e., $c\in\ell^{p}\left(I\right)$.
For general frames, these two properties need not be equivalent, as
shown in Section \ref{sec:AnalysisSynthesisSparsityNotEquivalentInGeneral}.

Note though that such an equivalence would indeed imply the desired
result, since the proof of equation \eqref{eq:IntroductionShearletApproximationRate}
given in \cite{CompactlySupportedShearletsAreOptimallySparse} proceeds
by a careful analysis of the analysis coefficients $A_{\Psi}f$ of
a cartoon-like function $f$: By counting how many shearlets intersect
the ``problematic'' region $\partial B$ where $f=f_{1}+\Indicator_{B}\cdot f_{2}$
is not $C^{2}$ and by then distinguishing whether the orientation
of the shearlet is aligned with the boundary curve $\partial B$ or
not, the authors show $\sum_{n>N}\left|\theta_{n}\left(f\right)\right|^{2}\lesssim N^{-2}\cdot\left(1+\log N\right)^{3}$,
where $\left(\theta_{n}\left(f\right)\right)_{n\in\N}$ is the \emph{nonincreasing
rearrangement} of the shearlet analysis coefficients $A_{\Psi}f$.
It is not too hard to see (see e.g.\@ the proof of Theorem \ref{thm:CartoonApproximationWithAlphaShearlets})
that this implies $A_{\Psi}f\in\ell^{p}\left(I\right)$ for all $p>\frac{2}{3}$.
Assuming that analysis sparsity with respect to the shearlet frame
$\Psi$ is indeed equivalent to synthesis sparsity, this implies $f=\sum_{i\in I}c_{i}\psi_{i}$
for a sequence $c=\left(c_{i}\right)_{i\in I}\in\ell^{p}\left(I\right)$.
Then, simply by taking only the $N$ largest coefficients of the sequence
$c$ and by using that the synthesis map $S_{\Psi}:\ell^{2}\left(I\right)\to L^{2}\left(\R^{2}\right),\left(e_{i}\right)_{i\in I}\mapsto\sum_{i\in I}e_{i}\psi_{i}$
is bounded, it is not hard to see $\left\Vert f-f_{N}\right\Vert _{L^{2}}\lesssim\left\Vert c-c\cdot\Indicator_{I_{N}}\right\Vert _{\ell^{2}}\lesssim N^{-\left(p^{-1}-2^{-1}\right)}$,
where $I_{N}\subset I$ is a set containing $N$ largest coefficients
of $c$.

Thus, once we know that analysis sparsity with respect to a (suitable)
shearlet frame is equivalent to synthesis sparsity, we only need to
make the preceding argument completely rigorous.

\subsection{Previous results concerning the equivalence of analysis and synthesis
sparsity for shearlets}

As noted above, analysis sparsity and synthesis sparsity need not
be equivalent for general frames. To address this and other problems,
Gröchenig\cite{LocalizationOfFrames} and Gröchenig \& Cordero\cite{LocalizationOfFrames2},
as well as Gröchenig \& Fornasier\cite{IntrinsicLocalizationOfFrames}
introduced the concept of \textbf{(intrinsically) localized frames}
for which these two properties are indeed equivalent, cf.\@ \cite[Proposition 2]{GribonvalNielsenHighlySparseRepresentations}.

In contrast to Gabor- and wavelet frames, it is quite nontrivial,
however, to verify that a shearlet or curvelet frame is intrinsically
localized: To our knowledge, the only papers discussing \emph{a variant}
of this property are \cite{IntrinsicLocalizationOfAnisotropicFrames,IntrinsicLocalizationOfAnisotropicFrames2},
where the results from \cite{IntrinsicLocalizationOfAnisotropicFrames}
about curvelets and shearlets are generalized in \cite{IntrinsicLocalizationOfAnisotropicFrames2}
to the setting of $\alpha$-molecules; a generalization that we will
discuss below in greater detail. For now, let us stick to the setting
of \cite{IntrinsicLocalizationOfAnisotropicFrames}. In that paper,
Grohs considers a certain \textbf{distance function} $\omega:\Lambda^{S}\times\Lambda^{S}\to\left[1,\infty\right)$
(cf.\@ \cite[Definition 3.9]{IntrinsicLocalizationOfAnisotropicFrames}
for the precise formula) on the index set
\[
\Lambda^{S}:=\left\{ \left(j,\ell,k,\delta\right)\in\N_{0}\times\Z\times\Z^{2}\times\left\{ 0,1\right\} \with-2^{\left\lfloor j/2\right\rfloor }\leq\ell<2^{\left\lfloor j/2\right\rfloor }\right\} ,\vspace{-0.05cm}
\]
which is (a slightly modified version of) the index set that is used
for shearlet frames. A shearlet frame $\Psi=\left(\psi_{\lambda}\right)_{\lambda\in\Lambda^{S}}$
is called \textbf{$N$-localized with respect to $\omega$} if the
associated \textbf{Gramian matrix} $\mathbf{A}:=\mathbf{A}_{\Psi}:=\left(\left\langle \psi_{\lambda},\psi_{\lambda'}\right\rangle \right)_{\lambda,\lambda'\in\Lambda^{S}}$
satisfies
\begin{equation}
\left|\left\langle \psi_{\lambda},\,\psi_{\lambda'}\right\rangle \right|\leq\left\Vert \mathbf{A}\right\Vert _{\mathcal{B}_{N}}\cdot\left[\omega\left(\lambda,\lambda'\right)\right]^{-N}\qquad\forall\lambda,\lambda'\in\Lambda^{S},\label{eq:GrohsLocalizationDefinition}
\end{equation}
where $\left\Vert \mathbf{A}\right\Vert _{\mathcal{B}_{N}}$ is chosen
to be the optimal constant in the preceding inequality.

Then, if $\Psi$ is a frame with \textbf{frame bounds} $A,B>0$, i.e.,
if $A^{2}\cdot\left\Vert f\right\Vert _{L^{2}}^{2}\leq\sum_{\lambda\in\Lambda^{S}}\left|\left\langle f,\psi_{\lambda}\right\rangle \right|^{2}\leq B^{2}\cdot\left\Vert f\right\Vert _{L^{2}}^{2}$
for all $f\in L^{2}\left(\R^{2}\right)$, \cite[Lemma 3.3]{IntrinsicLocalizationOfAnisotropicFrames}
shows that the infinite matrix $\mathbf{A}$ induces a bounded, positive
semi-definite operator $\mathbf{A}:\ell^{2}\left(\Lambda^{S}\right)\to\ell^{2}\left(\Lambda^{S}\right)$
that furthermore satisfies $\sigma\left(\mathbf{A}\right)\subset\left\{ 0\right\} \cup\left[A,B\right]$
and the \textbf{Moore-Penrose pseudoinverse} $\mathbf{A}^{+}$ of
$\mathbf{A}$ is the Gramian associated to the \textbf{canonical dual
frame} $\widetilde{\Psi}$ of $\Psi$. This is important, since \cite[Theorem 3.11]{IntrinsicLocalizationOfAnisotropicFrames}
now yields the following:
\begin{thm*}
Assume that $\Psi=\left(\psi_{\lambda}\right)_{\lambda\in\Lambda^{S}}$
is a shearlet frame with sampling density $\delta>0$ and frame bounds
$A,B>0$. Furthermore, assume that $\Psi$ is $N+L$-localized with
respect to $\omega$, where
\[
N>2\qquad\text{ and }\qquad L>2\cdot\frac{\ln\left(10\right)}{\ln\left(5/4\right)}.
\]
Then the canonical dual frame $\widetilde{\Psi}$ of $\Psi$ is $N^{+}$-localized
with respect to $\omega$, with
\begin{equation}
N^{+}=N\cdot\left(1+\frac{\log\left(1+\frac{2}{A^{2}+B^{2}}\left\Vert \mathbf{A}\right\Vert _{N+L}\cdot\left[1+C_{\delta}\cdot\left(\frac{2}{1-2^{-L/2-2}}+\frac{8}{3}+\frac{1}{1-2^{2-L/2}}+\frac{1}{1-2^{-L/2}}\right)\right]^{2}\right)}{\log\left(\frac{B^{2}+A^{2}}{B^{2}-A^{2}}\right)}\right)^{-1},\label{eq:GrohsNPlusDefinition}
\end{equation}
where the constant $C_{\delta}>0$ only depends on the sampling density
$\delta>0$.
\end{thm*}
To see how this theorem could in principle be used, note that the
dual frame coefficients satisfy
\[
\left(\left\langle f,\,\smash{\widetilde{\psi_{\lambda}}}\right\rangle \right)_{\lambda\in\Lambda^{S}}=\mathbf{A}^{+}\left(\left\langle f,\psi_{\lambda}\right\rangle \right)_{\lambda\in\Lambda}.
\]
Consequently, if(!)\@ the Gramian $\mathbf{A}^{+}$ of the canonical
dual frame $\widetilde{\Psi}$ of $\Psi$ restricts to a well-defined
and bounded operator $\mathbf{A}^{+}:\ell^{p}\left(\Lambda^{S}\right)\to\ell^{p}\left(\Lambda^{S}\right)$,
then analysis sparsity with respect to $\Psi$ would imply analysis
sparsity with respect to $\widetilde{\Psi}$ and thus synthesis sparsity
with respect to $\Psi$, as desired. In fact, \cite[Proposition 3.5]{IntrinsicLocalizationOfAnisotropicFrames}
shows that if $\mathbf{A}^{+}$ is $N^{+}$-localized with respect
to $\omega$, then $\mathbf{A}^{+}:\ell^{p}\left(\Lambda^{S}\right)\to\ell^{p}\left(\Lambda^{S}\right)$
is bounded \emph{as long as} $N^{+}>2p^{-1}$.

\medskip{}

Thus, it seems that all is well, in particular since a combination
of \cite[Theorem 2.9 and Proposition 3.11]{ParabolicMolecules} provides\footnote{Strictly speaking, \cite[Definition 2.4]{ParabolicMolecules} uses
the index distance $\omega\left(\lambda,\lambda'\right)=2^{\left|s_{\lambda}-s_{\lambda'}\right|}\left(1+\smash{2^{\min\left\{ s_{\lambda},s_{\lambda'}\right\} }}d\left(\lambda,\lambda'\right)\right)$
which is \emph{different} from the distance $\omega\left(\lambda,\lambda'\right)=2^{\left|s_{\lambda}-s_{\lambda'}\right|}\left(1+d\left(\lambda,\lambda'\right)\right)$
used in \cite[Definition 3.9]{IntrinsicLocalizationOfAnisotropicFrames}.
Luckily, this inconsistency is no serious problem, since the distance
in \cite{ParabolicMolecules} dominates the distance from \cite{IntrinsicLocalizationOfAnisotropicFrames},
so that $N$-localization with respect to the \cite{ParabolicMolecules}-distance
implies $N$-localization with respect to the \cite{IntrinsicLocalizationOfAnisotropicFrames}-distance.} readily verifiable conditions on the generators $\varphi,\psi,\theta$
which ensure that the shearlet frame $\Psi={\rm SH}\left(\varphi,\psi,\theta;\delta\right)$
is $N$-localized with respect to $\omega$.

There is, however, a well-hidden remaining problem which is also the
reason why the equivalence of analysis and synthesis sparsity is not
explicitly claimed in any of the papers \cite{IntrinsicLocalizationOfAnisotropicFrames,IntrinsicLocalizationOfAnisotropicFrames2,ParabolicMolecules,AlphaMolecules}:
As seen above, we need $N^{+}>2p^{-1}$, but it is not clear at all
that this can be achieved with $N^{+}$ as in equation \eqref{eq:GrohsNPlusDefinition}:
There are strong interdependencies between the different quantities
on the right-hand side of equation \eqref{eq:GrohsNPlusDefinition}
which make it next to impossible to verify $N^{+}>2p^{-1}$. Indeed,
the results in \cite{ParabolicMolecules} only yield $\left\Vert \mathbf{A}\right\Vert _{N+L}<\infty$
under certain assumptions (which depend on $N+L$) concerning $\varphi,\psi,\theta$,
but \emph{no explicit control} over $\left\Vert \mathbf{A}\right\Vert _{N+L}$
is given. Thus, it is not at all clear that increasing $N$ (or $L$)
will increase $N^{+}$. Likewise, the frame bounds $A,B$ only depend
on $\varphi,\psi,\theta$ (which are more or less fixed) and on the
sampling density $\delta$. Thus, one could be tempted to change $\delta$
to influence $A,B$ in equation \eqref{eq:GrohsNPlusDefinition} and
thus to achieve $N^{+}>2p^{-1}$. But the sampling density $\delta$
also influences $C_{\delta}$ and $\left\Vert \mathbf{A}\right\Vert _{N+L}$,
so that it is again not clear at all whether one can ensure $N^{+}>2p^{-1}$
by modifying $\delta$.

\medskip{}

A further framework for deriving the equivalence between analysis
and synthesis sparsity for frames is provided by \textbf{(generalized)
coorbit theory}\cite{FeichtingerCoorbit0,FeichtingerCoorbit1,FeichtingerCoorbit2,RauhutCoorbitQuasiBanach,GeneralizedCoorbit1,GeneralizedCoorbit2}.
Here, one starts with a \emph{continuous} frame $\Psi=\left(\psi_{x}\right)_{x\in X}$
which is indexed by a locally compact measure space $\left(X,\mu\right)$.
In the case of classical, group-based coorbit theory\cite{FeichtingerCoorbit0,FeichtingerCoorbit1,FeichtingerCoorbit2},
it is even required that $\left(\psi_{x}\right)_{x\in G}=\left(\pi\left(x\right)\psi\right)_{x\in G}$
arises from an integrable, irreducible \textbf{unitary representation}
of a locally compact topological group $G$, although one can weaken
certain of these conditions\cite{CoorbitOnHomogenousSpaces,CoorbitOnHomogenousSpaces2,CoorbitWithVoiceInFrechetSpace,CoorbitSpacesForDualPairs}.

Based on the continuous frame $\Psi$, one can then introduce so-called
\textbf{coorbit spaces} ${\rm Co}\left(Y\right)$ which are defined
in terms of decay conditions (specified by the function space $Y$)
concerning the \textbf{voice transform} $V_{\Psi}f\left(x\right):=\left\langle f,\,\psi_{x}\right\rangle $
of a function or distribution $f$. Coorbit theory then provides conditions
under which one can sample the continuous frame $\Psi$ to obtain
a discrete frame $\Psi_{d}=\left(\smash{\psi_{x_{i}}}\right)_{i\in I}$,
but such that membership of a distribution $f$ in ${\rm Co}\left(Y\right)$
is \emph{simultaneously} equivalent to analysis sparsity and to synthesis
sparsity of $f$ with respect to $\Psi_{d}$.

Thus, if one could find a \emph{continuous frame} $\Psi$ such that
the prerequisites of coorbit theory are satisfied and such that the
discretized frame $\Psi_{d}$ coincides with a discrete, \emph{cone-adapted}
shearlet frame, one would obtain the desired equivalence between analysis
sparsity and synthesis sparsity. There is, however, no known construction
of such a frame $\Psi$: Although there is a rich theory of \textbf{shearlet
coorbit spaces}\cite{Dahlke_etal_sh_coorbit1,Dahlke_etal_sh_coorbit2,DahlkeShearletArbitraryDimension,DahlkeShearletCoorbitEmbeddingsInHigherDimensions,DahlkeToeplitzShearletTransform,MR2896277,FuehrSimplifiedVanishingMomentCriteria}
which fits into the more general framework of \textbf{wavelet-type
coorbit spaces}\cite{FuehrContinuousWaveletTransformsFromSemidirectProducts,FuehrContinuousWaveletTransformsSemidirectProducts,FuehrCoorbit1,FuehrCoorbit2,FuehrGeneralizedCalderonConditions,FuehrSimplifiedVanishingMomentCriteria,FuehrVoigtlaenderCoorbitSpacesAsDecompositionSpaces,FuehrWaveletFramesAndAdmissibility},
the resulting discretized frames are \emph{not} cone-adapted shearlet
frames; instead, they are highly \emph{directionally biased} (i.e.,
they treat the $x$ and $y$ direction in very different ways) and
the number of directions per scale is \emph{infinite} for each scale;
therefore, these systems are unsuitable for most practical applications
and for the approximation of cartoon-like functions, cf.\@ \cite[Section 3.3]{ConeAdaptedShearletFirstPaper}.
Hence—at least using the currently known constructions of continuous
shearlet frames—coorbit theory can \emph{not} be used to derive the
desired equivalence of analysis and synthesis sparsity with respect
to cone-adapted shearlet frames.

\subsection{Our approach for proving the equivalence of analysis and synthesis
sparsity for shearlets}

In this paper, we use the recently introduced theory of \textbf{structured
Banach frame decompositions of decomposition spaces}\cite{StructuredBanachFrames}
to obtain the desired equivalence between analysis and synthesis sparsity
for (cone-adapted) shearlet frames. A more detailed and formal exposition
of this theory will be given in Section \ref{sec:BanachFrameDecompositionCrashCourse};
for this introduction, we restrict ourselves to the bare essentials.

The starting point in \cite{StructuredBanachFrames} is a \emph{covering}
$\CalQ=\left(Q_{i}\right)_{i\in I}$ of the frequency space $\R^{\dimension}$,
where it is assumed that each $Q_{i}$ is of the form $Q_{i}=T_{i}Q+b_{i}$
for a fixed \emph{base set} $Q\subset\R^{\dimension}$ and certain
linear maps $T_{i}\in\GL\left(\R^{\dimension}\right)$ and $b_{i}\in\R^{\dimension}$.
Then, using a suitable \emph{partition of unity} $\Phi=\left(\varphi_{i}\right)_{i\in I}$
subordinate to $\CalQ$ and a suitable \emph{weight} $w=\left(w_{i}\right)_{i\in I}$
on the index set $I$ of the covering $\CalQ$, one defines the associated
\textbf{decomposition space (quasi)-norm}
\[
\left\Vert g\right\Vert _{\DecompSp{\CalQ}p{\ell_{w}^{q}}{}}:=\left\Vert \left(w_{i}\cdot\left\Vert \Fourier^{-1}\left(\varphi_{i}\cdot\widehat{g}\right)\right\Vert _{L^{p}}\right)_{i\in I}\right\Vert _{\ell^{q}},
\]
while the associated \textbf{decomposition space} $\DecompSp{\CalQ}p{\ell_{w}^{q}}{}$
contains exactly those distributions $g$ for which this quasi-norm
is finite.

Roughly speaking, the decomposition space (quasi)-norm measures the
size of the distribution $g$ by frequency-localizing $g$ to each
of the sets $Q_{i}$ (using the partition of unity $\Phi$), where
each of these frequency-localized pieces is measured in $L^{p}\left(\R^{\dimension}\right)$,
while the individual contributions are aggregated using a certain
weighted $\ell^{q}$-norm. The underlying idea in \cite{StructuredBanachFrames}
is to ask whether the \emph{strict} frequency localization using the
compactly supported partition of unity $\Phi$ can be replaced by
a soft, qualitative frequency localization: Indeed, if $\psi\in L^{1}\left(\R^{\dimension}\right)$
has \emph{essential} frequency support in the base set $Q$, then
it is not hard to see that the function
\[
\psi^{\left[i\right]}:=\left|\det T_{i}\right|^{-1/2}\cdot\Fourier^{-1}\left(L_{b_{i}}\left[\smash{\widehat{\psi}}\circ T_{i}^{-1}\right]\right)=\left|\det T_{i}\right|^{1/2}\cdot M_{b_{i}}\left[\psi\circ T_{i}^{T}\right]
\]
has essential frequency support in $Q_{i}=T_{i}Q+b_{i}$, for arbitrary
$i\in I$. Here, $L_{x}$ and $M_{\xi}$ denote the usual translation
and modulation operators, cf.\@ Section \ref{subsec:Notation}.

Using this notation, the theory developed in \cite{StructuredBanachFrames}
provides criteria pertaining to the \textbf{generator} $\psi$ which
guarantee that the generalized shift-invariant system
\begin{equation}
\Psi_{\delta}:=\left(L_{\delta\cdot T_{i}^{-T}k}\:\psi^{\left[i\right]}\right)_{i\in I,\,k\in\Z^{\dimension}}\label{eq:IntroductionStructuredFrameDefinition}
\end{equation}
forms, respectively, a \textbf{Banach frame} or an \textbf{atomic
decomposition} for the decomposition space $\DecompSp{\CalQ}p{\ell_{w}^{q}}{}$,
for sufficiently fine sampling density $\delta>0$. The notions of
Banach frames and atomic decompositions generalize the concept of
frames for Hilbert spaces to the setting of (Quasi)-Banach spaces.
The precise definitions of these two concepts, however, are outside
the scope of this introduction; see e.g.\@ \cite{GroechenigDescribingFunctions}
for a lucid exposition.

For us, the most important conclusion is the following: If $\Psi_{\delta}$
\emph{simultaneously} forms a Banach space and an atomic decomposition
for $\DecompSp{\CalQ}p{\ell_{w}^{q}}{}$, then there is an explicitly
known (Quasi)-Banach space of sequences $C_{w}^{p,q}\leq\Compl^{I\times\Z^{\dimension}}$,
called the \textbf{coefficient space}, such that the following are
equivalent for a distribution $g$:
\begin{enumerate}
\item $g\in\DecompSp{\CalQ}p{\ell_{w}^{q}}{}$,
\item the analysis coefficients $\left(\left\langle g,\,L_{\delta\cdot T_{i}^{-T}k}\:\psi^{\left[i\right]}\right\rangle \right)_{i\in I,\,k\in\Z^{\dimension}}$
belong to $C_{w}^{p,q}$,
\item we can write $g=\sum_{i\in I}\,\sum_{k\in\Z^{\dimension}}\left(\smash{c_{k}^{\left(i\right)}}\cdot\psi^{\left[i\right]}\right)$
for a sequence $\left(\smash{c_{k}^{\left(i\right)}}\right)_{i\in I,\,k\in\Z^{\dimension}}\in C_{w}^{p,q}$.
\end{enumerate}
One can even derive slightly stronger conclusions which make these
purely qualitative statements quantitative. Now, if one chooses $p=q\in\left(0,2\right)$
and a suitable weight $w=\left(w_{i}\right)_{i\in I}$ depending on
$p$, one can achieve $C_{w}^{p,q}=\ell^{p}\left(I\times\Z^{\dimension}\right)$.
Thus, in this case, the preceding equivalence can be summarized as
follows:
\[
\text{If }\psi\text{ is nice and }\delta>0\text{ is small, then \textbf{analysis sparsity is equivalent to synthesis sparsity} w.r.t. }\Psi_{\delta}.
\]

In fact, the theory developed in \cite{StructuredBanachFrames} even
allows the base set $Q$ to vary with $i\in I$, i.e., $Q_{i}=T_{i}Q_{i}'+b_{i}$,
at least as long as the family $\left\{ Q_{i}'\with i\in I\right\} $
of different base sets remains finite. Similarly, the generator $\psi$
is allowed to vary with $i\in I$, so that $\psi^{\left[i\right]}=\left|\det T_{i}\right|^{1/2}\cdot M_{b_{i}}\left[\psi_{i}\circ T_{i}^{T}\right]$,
again with the provision that the set $\left\{ \psi_{i}\with i\in I\right\} $
of generators is finite.

As we will see, one can choose a suitable covering $\CalQ=\CalS$—the
so-called \textbf{shearlet covering} of the frequency space $\R^{2}$—such
that the system $\Psi_{\delta}$ from above coincides with a shearlet
frame. The resulting decomposition spaces $\DecompSp{\CalS}p{\ell_{w}^{q}}{}$
are then (slight modifications of) the \textbf{shearlet smoothness
spaces} as introduced by Labate et al.\cite{Labate_et_al_Shearlet}.

In summary, the theory of \textbf{structured Banach frame decompositions
of decomposition spaces} will imply the desired equivalence of analysis
and synthesis sparsity with respect to cone-adapted shearlet frames.
To this end, however, we first need to show that the technical conditions
on the generators that are imposed in \cite{StructuredBanachFrames}
are indeed satisfied if the generators of the shearlet system are
sufficiently smooth and satisfy certain vanishing moment conditions.
As we will see, this is by no means trivial and requires a huge amount
of technical estimates.

\medskip{}

Finally, we remark that spaces similar to the shearlet smoothness
spaces have also been considered by Vera: In \cite{VeraShearBesovSpaces},
he introduced so-called \textbf{shear anisotropic inhomogeneous Besov
spaces}, which are essentially a generalization of the shearlet smoothness
spaces to $\R^{\dimension}$. Vera then shows that the analysis and
synthesis operators with respect to certain \emph{bandlimited} shearlet
systems are bounded between the shear anisotropic inhomogeneous Besov
spaces and certain sequence spaces. Note that the assumption of bandlimited
frame elements excludes the possibility of having compact support
in space. Furthermore, boundedness of the analysis and synthesis operators
alone does \emph{not} imply that the \emph{bandlimited} shearlet systems
form Banach frames or atomic decompositions for the shear anisotropic
Besov spaces, since this requires existence of a certain \emph{reproducing
formula}.  In \cite{VeraShearTriebelLizorkin}, Vera also considers
\emph{Triebel-Lizorkin type} shearlet smoothness spaces and again
derives similar boundedness results for the analysis and synthesis
operators. Finally, in both papers \cite{VeraShearBesovSpaces,VeraShearTriebelLizorkin},
certain embedding results between the classical Besov or Triebel-Lizorkin
spaces and the new ``shearlet adapted'' smoothness spaces are considered,
similarly to our results in Section \ref{sec:EmbeddingsBetweenAlphaShearletSmoothness}.
Note though that we are able to completely characterize the existence
of such embeddings, while \cite{VeraShearBesovSpaces} only establishes
certain necessary and certain sufficient conditions, without achieving
a characterization.

\subsection{\texorpdfstring{$\alpha$}{α}-shearlets and cartoon-like functions
of different regularity}

The usual construction of shearlets employs the \textbf{parabolic
dilations} ${\rm diag}\left(2^{j},2^{j/2}\right)$ and (the dual frames
of) the resulting shearlet systems turn out to be (almost) optimal
for the approximation of functions that are $C^{2}$ away from a $C^{2}$
edge. Beginning with the paper \cite{OptimallySparse3D}, it was realized
that different regularities—i.e., ``functions that are $C^{\beta}$
away from a $C^{\beta}$ edge''—can be handled by employing a different
type of dilations, namely\footnote{In fact, in \cite[Section 4.1]{OptimallySparse3D} the three-dimensional
counterparts of the scaling matrices ${\rm diag}\left(2^{\beta j/2},\,2^{j/2}\right)$
are used, but the resulting hybrid shearlet systems have the same
approximation properties as those defined using the $\alpha$-parabolic
dilations ${\rm diag}\left(2^{j},\,2^{\alpha j}\right)$ with $\alpha=\beta^{-1}$;
see Section \ref{sec:CartoonLikeFunctionsAreBoundedInAlphaShearletSmoothness}
for more details.} the \textbf{$\alpha$-parabolic dilations} ${\rm diag}\left(2^{j},\,2^{\alpha j}\right)$,
with the specific choice $\alpha=\beta^{-1}$.

These modified shearlet systems were called \textbf{hybrid shearlets}
in \cite{OptimallySparse3D}, where they were introduced in the three-dimensional
setting. In the Bachelor's thesis \cite{SandraBachelorArbeit}, precisely
in \cite[Section 4]{SandraBachelorArbeit}, it was then shown also
in the two-dimensional setting that shearlet systems using $\alpha$-parabolic
scaling—from now on called \textbf{$\alpha$-shearlet systems}—indeed
yield (almost) optimal approximation rates for the model class of
\textbf{$C^{\beta}$-cartoon-like functions}, if $\alpha=\beta^{-1}$.
Again, this comes with the caveat that the approximation is actually
performed using the \emph{dual frame} of the $\alpha$-shearlet frame.

Note, however, that the preceding result requires the regularity $\beta$
of the $C^{\beta}$-cartoon-like functions to satisfy $\beta\in\left(1,2\right]$.
Outside of this range, the arguments in \cite{SandraBachelorArbeit}
are not applicable; in fact, it was shown in \cite{RoleOfAlphaScaling}
that the result concerning the optimal approximation rate fails for
$\beta>2$, at least for \textbf{$\alpha$-curvelets\cite{CartoonApproximationWithAlphaCurvelets}}
instead of $\alpha$-shearlets.

These $\alpha$-curvelets are related to $\alpha$-shearlets in the
same way that shearlets and curvelets are related\cite{ParabolicMolecules},
in the sense that the associated coverings of the Fourier domain are
equivalent and in that they agree with respect to \emph{analysis}
sparsity: If $f$ is $\ell^{p}$-analysis sparse with respect to a
(reasonable) $\alpha$-curvelet system, then the same holds with respect
to any (reasonable) $\alpha$-shearlet system and vice versa. This
was derived in \cite{AlphaMolecules} as an application of the framework
of \textbf{$\alpha$-molecules}, a common generalization of $\alpha$-shearlets
and $\alpha$-curvelets; see also \cite{MultivariateAlphaMolecules}
for a generalization to dimensions larger than two.

\medskip{}

As we will see, one can modify the shearlet covering $\CalS$ slightly
to obtain the so-called \textbf{$\alpha$-shearlet covering} $\CalS^{\left(\alpha\right)}$.
The systems $\Psi_{\delta}$ (cf.\@ equation \eqref{eq:IntroductionStructuredFrameDefinition})
that result from an application of the theory of structured Banach
frame decompositions with the covering $\CalS^{\left(\alpha\right)}$
then turn out to be $\alpha$-shearlet systems. Therefore, we will
be able to establish the equivalence of analysis and synthesis sparsity
not only for classical cone-adapted shearlet systems, but in fact
for cone-adapted $\alpha$-shearlet systems for arbitrary $\alpha\in\left[0,1\right]$,
essentially without additional effort.

Even more, recall from above that the theory of structured Banach
frame decompositions not only yields equivalence of analysis and synthesis
sparsity, but also shows that each of these properties is equivalent
to membership of the distribution $f$ under consideration in a suitable
decomposition space $\DecompSp{\CalS^{\left(\alpha\right)}}p{\ell_{w}^{q}}{}$.
We will call these spaces \textbf{$\alpha$-shearlet smoothness spaces}
and denote them by $\mathscr{S}_{\alpha,s}^{p,q}\left(\R^{2}\right)$,
where the \emph{smoothness parameter} $s$ determines the weight $w$.
Using a recently developed theory for embeddings between decomposition
spaces\cite{DecompositionEmbedding}, we are then able to completely
characterize the existence of embeddings between $\alpha$-shearlet
smoothness spaces for different values of $\alpha$. Roughly, such
an embedding $\mathscr{S}_{\alpha_{1},s_{1}}^{p_{1},q_{1}}\hookrightarrow\mathscr{S}_{\alpha_{2},s_{2}}^{p_{2},q_{2}}$
means that sparsity (in a certain sense) with respect to $\alpha_{1}$-shearlets
implies sparsity (in a possibly different sense) with respect to $\alpha_{2}$-shearlets.

In a way, this extends the results of \cite{AlphaMolecules}, where
it is shown that analysis sparsity transfers from one $\alpha$-scaled
system to another (e.g.\@ from $\alpha$-curvelets to $\alpha$-shearlets);
in contrast, our embedding theory characterizes the possibility of
transferring such results from $\alpha_{1}$-shearlet systems to $\alpha_{2}$-shearlet
systems, even for $\alpha_{1}\neq\alpha_{2}$. It will turn out, however,
that simple $\ell^{p}$-sparsity with respect to $\alpha_{1}$-shearlets
\emph{never} yields a nontrivial $\ell^{q}$-sparsity with respect
to $\alpha_{2}$-shearlets, if $\alpha_{1}\neq\alpha_{2}$. Luckily,
one can remedy this situation by requiring $\ell^{p}$-sparsity in
conjunction with a certain decay of the coefficients with the scale.
Fore more details, we refer to Section \ref{sec:EmbeddingsBetweenAlphaShearletSmoothness}.

\subsection{Structure of the paper}

\label{subsec:Structure}Before we properly start the paper, we introduce
several standard and non-standard notations in the next subsection.

In Section \ref{sec:BanachFrameDecompositionCrashCourse}, we give
an overview over the main aspects of the theory of \emph{structured
Banach frame decompositions of decomposition spaces} that was recently
developed by one of the authors in \cite{StructuredBanachFrames}.

The most important ingredient for the application of this theory is
a suitable covering $\CalQ=\left(Q_{i}\right)_{i\in I}=\left(T_{i}Q_{i}'+b_{i}\right)_{i\in I}$
of the frequency space $\R^{2}$ such that the provided Banach frames
and atomic decompositions are of the desired form; in our case we
want to obtain cone-adapted $\alpha$-shearlet systems. Thus, in Section
\ref{sec:AlphaShearletSmoothnessDefinition}, we introduce the so-called
\textbf{$\alpha$-shearlet coverings} $\CalS^{\left(\alpha\right)}$
for $\alpha\in\left[0,1\right]$ and we verify that these coverings
fulfill the standing assumptions from \cite{StructuredBanachFrames}.
The more technical parts of this verification are deferred to Section
\ref{sec:AlphaShearletCoveringAlmostStructured} in order to not disrupt
the flow of the paper. Furthermore, Section \ref{sec:AlphaShearletSmoothnessDefinition}
also contains the definition of the \textbf{$\alpha$-shearlet smoothness
spaces} $\mathscr{S}_{\alpha,s}^{p,q}\left(\R^{2}\right)=\DecompSp{\smash{\CalS^{\left(\alpha\right)}}}p{\ell_{w^{s}}^{q}}{}$
and an analysis of their basic properties.

Section \ref{sec:CompactlySupportedShearletFrames} contains the main
results of the paper. Here, we provide readily verifiable conditions—smoothness,
decay and vanishing moments—concerning the generators $\varphi,\psi$
of the $\alpha$-shearlet system ${\rm SH}_{\alpha}^{\left(\pm1\right)}\left(\varphi,\psi;\delta\right)$
which ensure that this $\alpha$-shearlet system forms, respectively,
a Banach frame or an atomic decomposition for the $\alpha$-shearlet
smoothness space $\mathscr{S}_{\alpha,s}^{p,q}\left(\R^{2}\right)$.
This is done by verifying the technical conditions of the theory of
structured Banach frame decompositions. All of these results rely
on one technical lemma whose proof is extremely lengthy and therefore
deferred to Section \ref{sec:MegaProof}.

For $\alpha$-shearlet systems, it is expected that $\frac{1}{2}$-shearlets
are identical to the classical cone-adapted shearlet systems. This
is not quite the case, however, for the shearlet systems ${\rm SH}_{1/2}^{\left(\pm1\right)}\left(\varphi,\psi;\delta\right)$
considered in Section \ref{sec:CompactlySupportedShearletFrames}.
The reason for this is that the $\alpha$-shearlet covering $\CalS^{\left(\alpha\right)}$
divides the frequency plane into \emph{four} conic regions (the top,
bottom, left, and right frequency cones) and a low-frequency region,
while the usual definition of shearlets only divides the frequency
plane into two cones (horizontal and vertical) and a low-frequency
region. To remedy this fact, Section \ref{sec:UnconnectedAlphaShearletCovering}
introduces a slightly modified covering, the so-called \textbf{unconnected
$\alpha$-shearlet covering} $\CalS_{u}^{\left(\alpha\right)}$; the
reason for this terminology being that the individual sets of the
covering are not connected anymore. Essentially, $\CalS_{u}^{\left(\alpha\right)}$
is obtained by combining each pair of opposing sets of the $\alpha$-shearlet
covering $\CalS^{\left(\alpha\right)}$ into one single set. We then
verify that the associated decomposition spaces coincide with the
previously defined $\alpha$-shearlet smoothness spaces. Finally,
we show that the Banach frames and atomic decompositions obtained
by applying the theory of structured Banach frame decompositions with
the covering $\CalS_{u}^{\left(1/2\right)}$ indeed yield conventional
cone-adapted shearlet systems.

In Section \ref{sec:CartoonLikeApproximation}, we apply the equivalence
of analysis and synthesis sparsity for $\alpha$-shearlets to prove
that $\alpha$-shearlet frames with sufficiently nice generators indeed
yield (almost) optimal $N$-term approximations for the class $\mathcal{E}^{\beta}\left(\R^{2}\right)$
of $C^{\beta}$-cartoon-like functions, for $\beta\in\left(1,2\right]$
and $\alpha=\beta^{-1}$. In case of usual shearlets (i.e., for $\alpha=\frac{1}{2}$),
this is a straightforward application of the analysis sparsity of
$C^{2}$-cartoon-like functions with respect to shearlet systems.
But in case of $\alpha\neq\frac{1}{2}$, our $\alpha$-shearlet systems
use the $\alpha$-parabolic scaling matrices ${\rm diag}\left(2^{j},2^{\alpha j}\right)$,
while analysis sparsity of $C^{\beta}$-cartoon-like functions is
only known with respect to $\beta$-shearlet systems, which use the
scaling matrices ${\rm diag}\left(2^{\beta j/2},\,2^{j/2}\right)$.
Bridging the gap between these two different shearlet systems is not
too hard, but cumbersome, so that part of the proof for $\alpha\neq\frac{1}{2}$
is deferred to Section \ref{sec:CartoonLikeFunctionsAreBoundedInAlphaShearletSmoothness},
since most readers are probably mainly interested in the (easier)
case of classical shearlets (i.e., $\alpha=\frac{1}{2}$). The obtained
approximation rate is almost \emph{optimal} (cf.\@ \cite[Theorem 2.8]{CartoonApproximationWithAlphaCurvelets})
if one restricts to systems where the $N$-term approximation is formed
under a certain \emph{polynomial search depth restriction}. But in
the main text of the paper, we  just construct \emph{some} $N$-term
approximation, which not necessarily fulfills this restriction concerning
the search depth. In Section \ref{sec:PolynomialSearchDepth}, we
give a modified proof which shows that one can indeed retain the same
approximation rate, \emph{even under a polynomial search depth restriction}.

Finally, in Section \ref{sec:EmbeddingsBetweenAlphaShearletSmoothness}
we \emph{completely} characterize the existence of embeddings $\mathscr{S}_{\alpha_{1},s_{1}}^{p_{1},q_{1}}\left(\R^{2}\right)\hookrightarrow\mathscr{S}_{\alpha_{2},s_{2}}^{p_{2},q_{2}}\left(\R^{2}\right)$
between $\alpha$-shearlet smoothness spaces for different values
of $\alpha$. Effectively, this characterizes the cases in which one
can obtain sparsity with respect to $\alpha_{2}$-shearlets \emph{when
the only knowledge available is a certain sparsity with respect to
$\alpha_{1}$-shearlets}.

\subsection{Notation}

\label{subsec:Notation}We write $\N=\Z_{\geq1}$ for the set of \textbf{natural
numbers} and $\N_{0}=\Z_{\geq0}$ for the set of natural numbers including
$0$. For a matrix $A\in\R^{\dimension\times\dimension}$, we denote
by $A^{T}$ the transpose of $A$. The norm $\left\Vert A\right\Vert $
of $A$ is the usual \textbf{operator norm} of $A$, acting on $\R^{\dimension}$
equipped with the usual euclidean norm $\left|\mybullet\right|=\left\Vert \mybullet\right\Vert _{2}$.
The \textbf{open euclidean ball} of radius $r>0$ around $x\in\R^{\dimension}$
is denoted by $B_{r}\left(x\right)$. For a linear (bounded) operator
$T:X\to Y$ between (quasi)-normed spaces $X,Y$, we denote the \textbf{operator
norm} of $T$ by 
\[
\vertiii T:=\vertiii T_{X\to Y}:=\sup_{\left\Vert x\right\Vert _{X}\leq1}\left\Vert Tx\right\Vert _{Y}.\vspace{-0.05cm}
\]

For an arbitrary set $M$, we let $\left|M\right|\in\N_{0}\cup\left\{ \infty\right\} $
denote the number of elements of the set. For $n\in\N_{0}$, we write
$\underline{n}:=\left\{ 1,\dots,n\right\} $; in particular, $\underline{0}=\emptyset$.
For the \textbf{closure} of a subset $M$ of some topological space,
we write $\overline{M}$.

The $\dimension$-dimensional \textbf{Lebesgue measure} of a (measurable)
set $M\subset\R^{\dimension}$ is denoted by $\lambda\left(M\right)$
or by $\lambda_{\dimension}\left(M\right)$. Occasionally, we will
also use the constant $s_{\dimension}:=\mathcal{H}^{\dimension-1}\left(S^{\dimension-1}\right)$,
the \textbf{surface area of the euclidean unit-sphere} $S^{\dimension-1}\subset\R^{\dimension}$.
The \textbf{complex conjugate} of $z\in\Compl$ is denoted by $\overline{z}$.
We use the convention $x^{0}=1$ for all $x\in\left[0,\infty\right)$,
even for $x=0$.

For a subset $M\subset B$ of a fixed \emph{base set} $B$ (which
is usually implied by the context), we define the \textbf{indicator
function} (or \textbf{characteristic function}) $\Indicator_{M}$
of the set $M$ by
\[
\Indicator_{M}:B\to\left\{ 0,1\right\} ,x\mapsto\begin{cases}
1, & \text{if }x\in M,\\
0, & \text{otherwise}.
\end{cases}
\]

The \textbf{translation} and \textbf{modulation} of a function $f:\R^{\dimension}\to\Compl^{k}$
by $x\in\R^{\dimension}$ or $\xi\in\R^{\dimension}$ are, respectively,
denoted by 
\[
L_{x}f:\R^{\dimension}\to\Compl^{k},y\mapsto f\left(y-x\right),\qquad\text{ and }\qquad M_{\xi}f:\R^{\dimension}\to\Compl^{k},y\mapsto e^{2\pi i\left\langle \xi,y\right\rangle }f\left(y\right).
\]
Furthermore, for $g:\R^{\dimension}\to\Compl^{k}$, we use the notation
$\widetilde{g}$ for the function $\widetilde{g}:\R^{\dimension}\to\Compl^{k},x\mapsto g\left(-x\right)$.

For the \textbf{Fourier transform}, we use the convention $\widehat{f}\left(\xi\right):=\left(\Fourier f\right)\left(\xi\right):=\int_{\R^{\dimension}}f\left(x\right)\cdot e^{-2\pi i\left\langle x,\xi\right\rangle }\d x$
for $f\in L^{1}\left(\R^{\dimension}\right)$. It is well-known that
the Fourier transform extends to a unitary automorphism $\Fourier:L^{2}\left(\R^{\dimension}\right)\to L^{2}\left(\R^{\dimension}\right)$.
The inverse of this map is the continuous extension of the inverse
Fourier transform, given by $\left(\Fourier^{-1}f\right)\left(x\right)=\int_{\R^{\dimension}}f\left(\xi\right)\cdot e^{2\pi i\left\langle x,\xi\right\rangle }\d\xi$
for $f\in L^{1}\left(\R^{\dimension}\right)$. We will make frequent
use of the space $\Schwartz\left(\R^{\dimension}\right)$ of \textbf{Schwartz
functions} and its topological dual space $\Schwartz'\left(\R^{\dimension}\right)$,
the space of \textbf{tempered distributions}. For more details on
these spaces, we refer to \cite[Section 9]{FollandRA}; in particular,
we note that the Fourier transform restricts to a linear homeomorphism
$\Fourier:\Schwartz\left(\R^{\dimension}\right)\to\Schwartz\left(\R^{\dimension}\right)$;
by duality, we can thus define $\Fourier:\Schwartz'\left(\R^{\dimension}\right)\to\Schwartz'\left(\R^{\dimension}\right)$
by $\Fourier\varphi=\varphi\circ\Fourier$ for $\varphi\in\Schwartz'\left(\R^{\dimension}\right)$.

Given an open subset $U\subset\R^{\dimension}$, we let $\DistributionSpace U$
denote the space of \textbf{distributions} on $U$, i.e., the topological
dual space of $\DistributionSpace U:=\TestFunctionSpace U$. For the
precise definition of the topology on $\TestFunctionSpace U$, we
refer to \cite[Chapter 6]{RudinFA}. We remark that the dual pairings
$\left\langle \cdot,\cdot\right\rangle _{\CalD',\CalD}$ and $\left\langle \cdot,\cdot\right\rangle _{\Schwartz',\Schwartz}$
are always taken to be \emph{bilinear} instead of sesquilinear.

Occasionally, we will make use of the \textbf{Sobolev space}
\[
W^{N,p}\left(\smash{\R^{\dimension}}\right)=\left\{ f\in L^{p}\left(\smash{\R^{\dimension}}\right)\with\forall\alpha\in\N_{0}^{\dimension}\text{ with }\left|\alpha\right|\leq N:\quad\partial^{\alpha}f\in L^{p}\left(\smash{\R^{\dimension}}\right)\right\} \qquad\text{ with }p\in\left[1,\infty\right].
\]
Here, as usual for Sobolev spaces, the partial derivatives $\partial^{\alpha}f$
have to be understood in the distributional sense.

Furthermore, we will use the notations $\left\lceil x\right\rceil :=\min\left\{ k\in\Z\with k\geq x\right\} $
and $\left\lfloor x\right\rfloor :=\max\left\{ k\in\Z\with k\leq x\right\} $
for $x\in\R$. We observe $\left\lfloor x\right\rfloor \leq x<\left\lfloor x\right\rfloor +1$
and $\left\lceil x\right\rceil -1<x\leq\left\lceil x\right\rceil $.
Sometimes, we also write $x_{+}:=\left(x\right)_{+}:=\max\left\{ 0,x\right\} $
for $x\in\R$.

Finally, we will frequently make use of the \textbf{shearing matrices}
$S_{x}$, the \textbf{$\alpha$-parabolic dilation matrices} $D_{b}^{\left(\alpha\right)}$
and the involutive matrix $R$, given by
\begin{equation}
S_{x}:=\left(\begin{matrix}1 & x\\
0 & 1
\end{matrix}\right),\quad\text{ and }\quad D_{b}^{\left(\alpha\right)}:=\left(\begin{matrix}b & 0\\
0 & b^{\alpha}
\end{matrix}\right),\quad\text{ as well as }\quad R:=\left(\begin{matrix}0 & 1\\
1 & 0
\end{matrix}\right),\label{eq:StandardMatrices}
\end{equation}
for $x\in\R$ and $\alpha,b\in\left[0,\infty\right)$.

\section{Structured Banach frame decompositions of decomposition spaces —
A crash course}

\label{sec:BanachFrameDecompositionCrashCourse}In this section, we
give a brief introduction to the theory of structured Banach frames
and atomic decompositions for decomposition spaces that was recently
developed by one of the authors in \cite{StructuredBanachFrames}.

We start with a crash course on decomposition spaces. These are defined
using a suitable covering $\CalQ=\left(Q_{i}\right)_{i\in I}$ of
(a subset of) the \emph{frequency space} $\R^{\dimension}$. For the
decomposition spaces to be well-defined and for the theory in \cite{StructuredBanachFrames}
to be applicable, the covering $\CalQ$ needs to be a \textbf{semi-structured
covering} for which a \textbf{regular partition of unity} exists.
For this, it suffices if $\CalQ$ is an \textbf{almost structured
covering}. Since the notion of almost structured coverings is somewhat
easier to understand than general semi-structured coverings, we will
restrict ourselves to this concept.
\begin{defn}
\label{def:AlmostStructuredCovering}Let $\emptyset\neq\CalO\subset\R^{\dimension}$
be open. A family $\CalQ=\left(Q_{i}\right)_{i\in I}$ is called an
\textbf{almost structured covering} of $\CalO$, if for each $i\in I$,
there is an invertible matrix $T_{i}\in\GL\left(\R^{\dimension}\right)$,
a translation $b_{i}\in\R^{\dimension}$ and an open, bounded set
$Q_{i}'\subset\R^{\dimension}$ such that the following conditions
are fulfilled:

\begin{enumerate}
\item We have $Q_{i}=T_{i}Q_{i}'+b_{i}$ for all $i\in I$.
\item We have $Q_{i}\subset\CalO$ for all $i\in I$.
\item $\CalQ$ is \textbf{admissible}, i.e., there is some $N_{\CalQ}\in\N$
satisfying $\left|i^{\ast}\right|\leq N_{\CalQ}$ for all $i\in I$,
where the \textbf{index-cluster} $i^{\ast}$ is defined as
\begin{equation}
i^{\ast}:=\left\{ \ell\in I\with Q_{\ell}\cap Q_{i}\neq\emptyset\right\} \qquad\text{ for }i\in I.\label{eq:IndexClusterDefinition}
\end{equation}
\item There is a constant $C_{\CalQ}>0$ satisfying $\left\Vert T_{i}^{-1}T_{j}\right\Vert \leq C_{\CalQ}$
for all $i\in I$ and all $j\in i^{\ast}$.
\item For each $i\in I$, there is an open set $P_{i}'\subset\R^{\dimension}$
with the following additional properties:

\begin{enumerate}
\item $\overline{P_{i}'}\subset Q_{i}'$ for all $i\in I$.
\item The sets $\left\{ P_{i}'\with i\in I\right\} $ and $\left\{ Q_{i}'\with i\in I\right\} $
are finite.
\item We have $\CalO\subset\bigcup_{i\in I}\left(T_{i}P_{i}'+b_{i}\right)$.\qedhere
\end{enumerate}
\end{enumerate}
\end{defn}
\begin{rem*}

\begin{itemize}[leftmargin=0.4cm]
\item In the following, if we require $\CalQ=\left(Q_{i}\right)_{i\in I}=\left(T_{i}Q_{i}'+b_{i}\right)_{i\in I}$
to be an almost structured covering of $\CalO$, it is always implicitly
understood that $T_{i},Q_{i}'$ and $b_{i}$ are chosen in such a
way that the conditions in Definition \ref{def:AlmostStructuredCovering}
are satisfied. 
\item Since each set $Q_{i}'$ is bounded and since the set $\left\{ Q_{i}'\with i\in I\right\} $
is finite, the family $\left(Q_{i}'\right)_{i\in I}$ is uniformly
bounded, i.e., there is some $R_{\CalQ}>0$ satisfying $Q_{i}'\subset\overline{B_{R_{\CalQ}}}\left(0\right)$
for all $i\in I$.\qedhere
\end{itemize}
\end{rem*}
A crucial property of almost structured coverings is that these always
admit a \textbf{regular partition of unity}, a notion which was originally
introduced in \cite[Definition 2.4]{DecompositionIntoSobolev}.
\begin{defn}
\label{def:RegularPartitionOfUnity}Let $\CalQ=\left(Q_{i}\right)_{i\in I}=\left(T_{i}Q_{i}'+b_{i}\right)_{i\in I}$
be an almost structured covering of the open set $\emptyset\neq\CalO\subset\R^{\dimension}$.
We say that the family $\Phi=\left(\varphi_{i}\right)_{i\in I}$ is
a \textbf{regular partition of unity} subordinate to $\CalQ$ if the
following hold:

\begin{enumerate}
\item We have $\varphi_{i}\in\TestFunctionSpace{\CalO}$ with $\supp\varphi_{i}\subset Q_{i}$
for all $i\in I$.
\item We have $\sum_{i\in I}\varphi_{i}\equiv1$ on $\CalO$.
\item For each $\alpha\in\N_{0}^{\dimension}$, the constant
\[
C^{\left(\alpha\right)}:=\sup_{i\in I}\left\Vert \partial^{\alpha}\smash{\varphi_{i}^{\natural}}\right\Vert _{\sup}
\]
is finite, where for each $i\in I$, the \textbf{normalized version}
$\varphi_{i}^{\natural}$ of $\varphi_{i}$ is defined as
\[
\varphi_{i}^{\natural}:\R^{\dimension}\to\Compl,\xi\mapsto\varphi_{i}\left(T_{i}\xi+b_{i}\right).\qedhere
\]
\end{enumerate}
\end{defn}
\begin{thm}
(cf.\@ \cite[Theorem 2.8]{DecompositionIntoSobolev} and see \cite[Proposition 1]{BorupNielsenDecomposition}
for a similar statement)

Every almost structured covering $\CalQ$ of an open subset $\emptyset\neq\CalO\subset\R^{\dimension}$
admits a regular partition of unity $\Phi=\left(\varphi_{i}\right)_{i\in I}$
subordinate to $\CalQ$.
\end{thm}
Before we can give the formal definition of decomposition spaces,
we need one further notion:
\begin{defn}
\label{def:QModerateWeightClusteringMap}(cf.\@ \cite[Definition 3.1]{DecompositionSpaces1})
Let $\emptyset\neq\CalO\subset\R^{\dimension}$ be open and assume
that $\CalQ=\left(Q_{i}\right)_{i\in I}$ is an almost structured
covering of $\CalO$. A \textbf{weight} $w$ on the index set $I$
is simply a sequence $w=\left(w_{i}\right)_{i\in I}$ of positive
numbers $w_{i}>0$. The weight $w$ is called \textbf{$\CalQ$-moderate}
if there is a constant $C_{\CalQ,w}>0$ satisfying
\begin{equation}
w_{j}\leq C_{\CalQ,w}\cdot w_{i}\qquad\forall\:i\in I\text{ and all }j\in i^{\ast}.\label{eq:ModerateWeightDefinition}
\end{equation}

For an arbitrary weight $w=\left(w_{i}\right)_{i\in I}$ on $I$ and
$q\in\left(0,\infty\right]$ we define the \textbf{weighted $\ell^{q}$
space} $\ell_{w}^{q}\left(I\right)$ as
\[
\ell_{w}^{q}\left(I\right):=\left\{ \left(c_{i}\right)_{i\in I}\in\Compl^{I}\with\left(w_{i}\cdot c_{i}\right)_{i\in I}\in\ell^{q}\left(I\right)\right\} ,
\]
equipped with the natural (quasi)-norm $\left\Vert \left(c_{i}\right)_{i\in I}\right\Vert _{\ell_{w}^{q}}:=\left\Vert \left(w_{i}\cdot c_{i}\right)_{i\in I}\right\Vert _{\ell{}^{q}}$.
We will also use the notation $\left\Vert c\right\Vert _{\ell_{w}^{q}}$
for arbitrary sequences $c=\left(c_{i}\right)_{i\in I}\in\left[0,\infty\right]^{I}$
with the understanding that $\left\Vert c\right\Vert _{\ell_{w}^{q}}=\infty$
if $c_{i}=\infty$ for some $i\in I$ or if $c\notin\ell_{w}^{q}\left(I\right)$.
\end{defn}
Now, we can finally give a precise definition of decomposition spaces.
We begin with the (easier) case of the so-called \textbf{Fourier-side
decomposition spaces}.
\begin{defn}
\label{def:FourierSideDecompositionSpaces}Let $\CalQ=\left(Q_{i}\right)_{i\in I}$
be an almost structured covering of the open set $\emptyset\neq\CalO\subset\R^{\dimension}$,
let $w=\left(w_{i}\right)_{i\in I}$ be a $\CalQ$-moderate weight
on $I$ and let $p,q\in\left(0,\infty\right]$. Finally, let $\Phi=\left(\varphi_{i}\right)_{i\in I}$
be a regular partition of unity subordinate to $\CalQ$. We then define
the associated \textbf{Fourier-side decomposition space (quasi)-norm}
as
\[
\left\Vert g\right\Vert _{\FourierDecompSp{\CalQ}p{\ell_{w}^{q}}{}}:=\left\Vert \left(\left\Vert \Fourier^{-1}\left(\varphi_{i}\cdot g\right)\right\Vert _{L^{p}}\right)_{i\in I}\right\Vert _{\ell_{w}^{q}}\in\left[0,\infty\right]\qquad\text{ for each distribution }g\in\DistributionSpace{\CalO}.
\]
The associated \textbf{Fourier-side decomposition space} is simply
\[
\FourierDecompSp{\CalQ}p{\ell_{w}^{q}}{}:=\left\{ g\in\DistributionSpace{\CalO}\with\left\Vert g\right\Vert _{\FourierDecompSp{\CalQ}p{\ell_{w}^{q}}{}}<\infty\right\} .\qedhere
\]
\end{defn}
\begin{rem*}
Before we continue with the definition of the actual (space-side)
decomposition spaces, a few remarks are in order:

\begin{itemize}[leftmargin=0.4cm]
\item The expression $\left\Vert \Fourier^{-1}\left(\varphi_{i}\cdot g\right)\right\Vert _{L^{p}}\in\left[0,\infty\right]$
makes sense for each $i\in I$, since $\varphi_{i}\in\TestFunctionSpace{\CalO}$,
so that $\varphi_{i}\cdot g$ is a compactly supported distribution
on $\R^{\dimension}$ (and thus also a tempered distribution), so
that the Paley-Wiener theorem (see e.g.\@ \cite[Theorem 7.23]{RudinFA})
shows that the tempered distribution $\Fourier^{-1}\left(\varphi_{i}\cdot g\right)$
is given by (integration against) a smooth function of which we can
take the $L^{p}$ quasi-norm.
\item The notations $\left\Vert g\right\Vert _{\FourierDecompSp{\CalQ}p{\ell_{w}^{q}}{}}$
and $\FourierDecompSp{\CalQ}p{\ell_{w}^{q}}{}$ both suppress the
specific regular partition of unity $\Phi$ that was chosen. This
is justified, since \cite[Corollary 3.18]{DecompositionEmbedding}
shows that any two $L^{p}$-BAPUs\footnote{The exact definition of an $L^{p}$-BAPU is not important for us.
The interested reader can find the definition in \cite[Definition 3.5]{DecompositionEmbedding}.} $\Phi,\Psi$ yield equivalent quasi-norms and thus the same (Fourier-side)
decomposition spaces. This suffices, since \cite[Corollary 2.7]{DecompositionIntoSobolev}
shows that every regular partition of unity is also an $L^{p}$-BAPU
for $\CalQ$, for arbitrary $p\in\left(0,\infty\right]$.
\item Finally, \cite[Theorem 3.21]{DecompositionEmbedding} shows that $\FourierDecompSp{\CalQ}p{\ell_{w}^{q}}{}$
is a Quasi-Banach space.\qedhere
\end{itemize}
\end{rem*}
\begin{defn}
\label{def:SpaceSideDecompositionSpaces}For an open set $\emptyset\neq\CalO\subset\R^{\dimension}$,
let $Z\left(\CalO\right):=\Fourier\left(\TestFunctionSpace{\CalO}\right)\subset\Schwartz\left(\R^{\dimension}\right)$
and equip this space with the unique topology which makes the Fourier
transform $\Fourier:\TestFunctionSpace{\CalO}\to Z\left(\CalO\right),\varphi\mapsto\widehat{\varphi}$
into a homeomorphism. The topological dual space of $Z\left(\CalO\right)$
is denoted by $Z'\left(\CalO\right)$. By duality, we define the Fourier
transform on $Z'\left(\CalO\right)$ by $\widehat{g}:=\Fourier g:=g\circ\Fourier\in\DistributionSpace{\CalO}$
for $g\in Z'\left(\CalO\right)$.

Finally, under the assumptions of Definition \ref{def:FourierSideDecompositionSpaces},
we define the \textbf{(space-side) decomposition space} associated
to the parameters $\CalQ,p,q,w$ as
\[
\DecompSp{\CalQ}p{\ell_{w}^{q}}{}:=\left\{ g\in Z'\left(\CalO\right)\with\left\Vert g\right\Vert _{\DecompSp{\CalQ}p{\ell_{w}^{q}}{}}:=\left\Vert \widehat{g}\right\Vert _{\FourierDecompSp{\CalQ}p{\ell_{w}^{q}}{}}<\infty\right\} .
\]
It is not hard to see that the Fourier transform $\Fourier:Z'\left(\CalO\right)\to\DistributionSpace{\CalO}$
is an isomorphism which restricts to an isometric isomorphism $\Fourier:\DecompSp{\CalQ}p{\ell_{w}^{q}}{}\to\FourierDecompSp{\CalQ}p{\ell_{w}^{q}}{}$.
\end{defn}
\begin{rem*}
For an explanation why the reservoirs $\DistributionSpace{\CalO}$
and $Z'\left(\CalO\right)$ are the correct choices for defining $\FourierDecompSp{\CalQ}p{\ell_{w}^{q}}{}$
and $\DecompSp{\CalQ}p{\ell_{w}^{q}}{}$, even in case of $\CalO=\R^{\dimension}$,
we refer to \cite[Remark 3.13]{DecompositionEmbedding}.
\end{rem*}
Now that we have formally introduced the notion of decomposition spaces,
we present the framework developed in \cite{StructuredBanachFrames}
for the construction of Banach frames and atomic decompositions for
these spaces. To this end, we introduce the following set of notations
and standing assumptions:
\begin{assumption}
\label{assu:CrashCourseStandingAssumptions}We fix an almost structured
covering $\CalQ=\left(T_{i}Q_{i}'+b_{i}\right)_{i\in I}$ with associated
regular partition of unity $\Phi=\left(\varphi_{i}\right)_{i\in I}$
for the remainder of the section. By definition of an almost structured
covering, the set $\left\{ Q_{i}'\with i\in I\right\} $ is finite.
Hence, we have $\left\{ Q_{i}'\with i\in I\right\} =\left\{ \smash{Q_{0}^{\left(1\right)}},\dots,\smash{Q_{0}^{\left(n\right)}}\right\} \vphantom{Q_{0}^{\left(n\right)}}$
for certain (not necessarily distinct) open, bounded subsets $Q_{0}^{\left(1\right)},\dots,Q_{0}^{\left(n\right)}\subset\R^{\dimension}$.
In particular, for each $i\in I$, there is some $k_{i}\in\underline{n}$
satisfying $Q_{i}'=Q_{0}^{\left(k_{i}\right)}$.

We fix the choice of $n\in\N$, of the sets $Q_{0}^{\left(1\right)},\dots,Q_{0}^{\left(n\right)}$
and of the map $I\to\underline{n},i\mapsto k_{i}$ for the remainder
of the section.
\end{assumption}
Finally, we need a suitable \textbf{coefficient space} for our Banach
frames and atomic decompositions:
\begin{defn}
\label{def:CoefficientSpace}For given $p,q\in\left(0,\infty\right]$
and a given weight $w=\left(w_{i}\right)_{i\in I}$ on $I$, we define
the associated \textbf{coefficient space} as
\[
\begin{split}C_{w}^{p,q} & :=\ell_{\left(\left|\det T_{i}\right|^{\frac{1}{2}-\frac{1}{p}}\cdot w_{i}\right)_{i\in I}}^{q}\!\!\!\!\!\left(\left[\ell^{p}\left(\Z^{\dimension}\right)\right]_{i\in I}\right)\\
 & :=\left\{ c=\left(\smash{c_{k}^{\left(i\right)}}\right)_{i\in I,k\in\Z^{\dimension}}\with\left\Vert c\right\Vert _{C_{w}^{p,q}}:=\left\Vert \left(\left|\det T_{i}\right|^{\frac{1}{2}-\frac{1}{p}}\cdot w_{i}\cdot\left\Vert \left(\smash{c_{k}^{\left(i\right)}}\right)_{k\in\Z^{\dimension}}\right\Vert _{\ell^{p}}\right)_{i\in I}\right\Vert _{\ell^{q}}<\infty\right\} \leq\Compl^{I\times\Z^{\dimension}}.\qedhere
\end{split}
\]
\end{defn}
\begin{rem*}
Observe that if $w_{i}=\left|\det T_{i}\right|^{\frac{1}{p}-\frac{1}{2}}$
and if $p=q$, then $C_{w}^{p,q}=\ell^{p}\left(I\times\Z^{\dimension}\right)$,
with equal (quasi)-norms.
\end{rem*}
Now that we have introduced the coefficient space $C_{w}^{p,q}$,
we are in a position to discuss the existence criteria for Banach
frames and atomic decompositions that were derived in \cite{StructuredBanachFrames}.
We begin with the case of \textbf{Banach frames}.
\begin{thm}
\label{thm:BanachFrameTheorem}Let $w=\left(w_{i}\right)_{i\in I}$
be a $\CalQ$-moderate weight, let $\varepsilon,p_{0},q_{0}\in\left(0,1\right]$
and let $p,q\in\left(0,\infty\right]$ with $p\geq p_{0}$ and $q\geq q_{0}$.
Define
\[
N:=\left\lceil \frac{\dimension+\varepsilon}{\min\left\{ 1,p\right\} }\right\rceil ,\qquad\tau:=\min\left\{ 1,p,q\right\} \qquad\text{ and }\qquad\sigma:=\tau\cdot\left(\frac{\dimension}{\min\left\{ 1,p\right\} }+N\right).
\]
Let $\gamma_{1}^{\left(0\right)},\dots,\gamma_{n}^{\left(0\right)}:\R^{\dimension}\to\Compl$
be given and define $\gamma_{i}:=\gamma_{k_{i}}^{\left(0\right)}$
for $i\in I$. Assume that the following conditions are satisfied:

\begin{enumerate}
\item We have $\gamma_{k}^{\left(0\right)}\in L^{1}\left(\R^{\dimension}\right)$
and $\Fourier\gamma_{k}^{\left(0\right)}\in C^{\infty}\left(\R^{\dimension}\right)$
for all $k\in\underline{n}$, where all partial derivatives of $\Fourier\gamma_{k}^{\left(0\right)}$
are polynomially bounded.
\item We have $\gamma_{k}^{\left(0\right)}\in C^{1}\left(\R^{\dimension}\right)$
and $\nabla\gamma_{k}^{\left(0\right)}\in L^{1}\left(\R^{\dimension}\right)\cap L^{\infty}\left(\R^{\dimension}\right)$
for all $k\in\underline{n}$.
\item We have $\left[\Fourier\gamma_{k}^{\left(0\right)}\right]\left(\xi\right)\neq0$
for all $\xi\in\overline{Q_{0}^{\left(k\right)}}$ and all $k\in\underline{n}$.
\item We have
\[
C_{1}:=\sup_{i\in I}\:\sum_{j\in I}M_{j,i}<\infty\quad\text{ and }\quad C_{2}:=\sup_{j\in I}\:\sum_{i\in I}M_{j,i}<\infty,
\]
where
\[
\qquad\qquad M_{j,i}:=\left(\frac{w_{j}}{w_{i}}\right)^{\tau}\cdot\left(1+\left\Vert T_{j}^{-1}T_{i}\right\Vert \right)^{\sigma}\cdot\max_{\left|\beta\right|\leq1}\left(\left|\det T_{i}\right|^{-1}\cdot\int_{Q_{i}}\:\max_{\left|\alpha\right|\leq N}\left|\left(\left[\partial^{\alpha}\widehat{\partial^{\beta}\gamma_{j}}\right]\left(T_{j}^{-1}\left(\xi-b_{j}\right)\right)\right)\right|\d\xi\right)^{\tau}.
\]
\end{enumerate}
Then there is some $\delta_{0}=\delta_{0}\left(p,q,w,\varepsilon,\left(\gamma_{i}\right)_{i\in I}\right)>0$
such that for arbitrary $0<\delta\leq\delta_{0}$, the family 
\[
\left(L_{\delta\cdot T_{i}^{-T}k}\:\widetilde{\gamma^{\left[i\right]}}\right)_{i\in I,k\in\Z^{\dimension}}\quad\text{ with }\quad\gamma^{\left[i\right]}=\left|\det T_{i}\right|^{1/2}\cdot M_{b_{i}}\left[\gamma_{i}\circ T_{i}^{T}\right]\quad\text{ and }\quad\widetilde{\gamma^{\left[i\right]}}\left(x\right)=\gamma^{\left[i\right]}\left(-x\right)
\]
forms a \textbf{Banach frame} for $\DecompSp{\CalQ}p{\ell_{w}^{q}}{}$.
Precisely, this means the following:

\begin{itemize}[leftmargin=0.7cm]
\item The \textbf{analysis operator} 
\[
A^{\left(\delta\right)}:\DecompSp{\CalQ}p{\ell_{w}^{q}}{}\to C_{w}^{p,q},f\mapsto\left(\left[\smash{\gamma^{\left[i\right]}}\ast f\right]\left(\delta\cdot T_{i}^{-T}k\right)\right)_{i\in I,k\in\Z^{\dimension}}
\]
is well-defined and bounded for each $\delta\in\left(0,1\right]$.
Here, the convolution $\gamma^{\left[i\right]}\ast f$ is defined
as
\begin{equation}
\left(\gamma^{\left[i\right]}\ast f\right)\left(x\right)=\sum_{\ell\in I}\Fourier^{-1}\left(\widehat{\gamma^{\left[i\right]}}\cdot\varphi_{\ell}\cdot\widehat{f}\:\right)\left(x\right)\qquad\forall x\in\R^{\dimension},\label{eq:SpecialConvolutionDefinition}
\end{equation}
where the series converges normally in $L^{\infty}\left(\R^{\dimension}\right)$
and thus absolutely and uniformly, for each $f\in\DecompSp{\CalQ}p{\ell_{w}^{q}}{}$.
For a more convenient expression of $\left(\gamma^{\left[i\right]}\ast f\right)\left(x\right)$,
at least for $f\in L^{2}\left(\R^{\dimension}\right)\subset Z'\left(\CalO\right)$,
see Lemma \ref{lem:SpecialConvolutionClarification}. 
\item For $0<\delta\leq\delta_{0}$, there is a bounded linear \textbf{reconstruction
operator} $R^{\left(\delta\right)}:C_{w}^{p,q}\to\DecompSp{\CalQ}p{\ell_{w}^{q}}{}$
satisfying $R^{\left(\delta\right)}\circ A^{\left(\delta\right)}=\identity_{\DecompSp{\CalQ}p{\ell_{w}^{q}}{}}$.
\item We have the following \textbf{consistency property}: If $\CalQ$-moderate
weights $w^{\left(1\right)}=\left(\smash{w_{i}^{\left(1\right)}}\right)_{i\in I}$
and $w^{\left(2\right)}=\left(\smash{w_{i}^{\left(2\right)}}\right)_{i\in I}$
and exponents $p_{1},p_{2},q_{1},q_{2}\in\left(0,\infty\right]$ are
chosen such that the assumptions of the current theorem are satisfied
for $p_{1},q_{1},w^{\left(1\right)}$, as well as for $p_{2},q_{2},w^{\left(2\right)}$
and if $0<\delta\leq\min\left\{ \delta_{0}\left(p_{1},q_{1},w^{\left(1\right)},\varepsilon,\left(\gamma_{i}\right)_{i\in I}\right),\delta_{0}\left(p_{2},q_{2},w^{\left(2\right)},\varepsilon,\left(\gamma_{i}\right)_{i\in I}\right)\right\} $
then we have the following equivalence:
\[
\forall f\in\DecompSp{\CalQ}{p_{2}}{\ell_{w^{\left(2\right)}}^{q_{2}}}{}:\quad f\in\DecompSp{\CalQ}{p_{1}}{\ell_{w^{\left(1\right)}}^{q_{1}}}{}\Longleftrightarrow\left(\left[\smash{\gamma^{\left[i\right]}}\ast f\right]\left(\delta\cdot T_{i}^{-T}k\right)\right)_{i\in I,k\in\Z^{\dimension}}\in C_{w^{\left(1\right)}}^{p_{1},q_{1}}.
\]
\end{itemize}
Finally, there is an estimate for the size of $\delta_{0}$ which
is independent of the choice of $p\geq p_{0}$ and $q\geq q_{0}$:
There is a constant $K=K\left(p_{0},q_{0},\varepsilon,\dimension,\CalQ,\Phi,\smash{\gamma_{1}^{\left(0\right)},\dots,\gamma_{n}^{\left(0\right)}}\right)>0$
such that we can choose 
\[
\delta_{0}=1\big/\left[1+K\cdot C_{\CalQ,w}^{4}\cdot\smash{\left(C_{1}^{1/\tau}+C_{2}^{1/\tau}\right)^{2}}\:\vphantom{\sum}\right].\qedhere
\]
\end{thm}
\begin{proof}
This is a special case of Theorem \ref{thm:WeightedBanachFrameTheorem},
for $\Omega_{0}=\Omega_{1}=1$, $K=0$ and $v=v_{0}\equiv1$.
\end{proof}
Now, we provide criteria which ensure that a given family of prototypes
generates \textbf{atomic decompositions}.
\begin{thm}
\label{thm:AtomicDecompositionTheorem}Let $w=\left(w_{i}\right)_{i\in I}$
be a $\CalQ$-moderate weight, let $\varepsilon,p_{0},q_{0}\in\left(0,1\right]$
and let $p,q\in\left(0,\infty\right]$ with $p\geq p_{0}$ and $q\geq q_{0}$.
Define
\[
N:=\left\lceil \frac{\dimension+\varepsilon}{\min\left\{ 1,p\right\} }\right\rceil ,\qquad\tau:=\min\left\{ 1,p,q\right\} ,\qquad\vartheta:=\left(\frac{1}{p}-1\right)_{+}\:,\qquad\text{ and }\qquad\varUpsilon:=1+\frac{\dimension}{\min\left\{ 1,p\right\} },
\]
as well as
\[
\sigma:=\begin{cases}
\tau\cdot\left(\dimension+1\right), & \text{if }p\in\left[1,\infty\right],\\
\tau\cdot\left(p^{-1}\cdot\dimension+\left\lceil p^{-1}\cdot\left(\dimension+\varepsilon\right)\right\rceil \right), & \text{if }p\in\left(0,1\right).
\end{cases}
\]
Let $\gamma_{1}^{\left(0\right)},\dots,\gamma_{n}^{\left(0\right)}:\R^{\dimension}\to\Compl$
be given and define $\gamma_{i}:=\gamma_{k_{i}}^{\left(0\right)}$
for $i\in I$. Assume that there are functions $\gamma_{1}^{\left(0,j\right)},\dots,\gamma_{n}^{\left(0,j\right)}$
for $j\in\left\{ 1,2\right\} $ such that the following conditions
are satisfied:

\begin{enumerate}
\item We have $\gamma_{k}^{\left(0,1\right)}\in L^{1}\left(\R^{\dimension}\right)$
for all $k\in\underline{n}$.
\item We have $\gamma_{k}^{\left(0,2\right)}\in C^{1}\left(\R^{\dimension}\right)$
for all $k\in\underline{n}$.
\item We have
\[
\Omega^{\left(p\right)}:=\max_{k\in\underline{n}}\left\Vert \gamma_{k}^{\left(0,2\right)}\right\Vert _{\varUpsilon}+\max_{k\in\underline{n}}\left\Vert \nabla\gamma_{k}^{\left(0,2\right)}\right\Vert _{\varUpsilon}<\infty,
\]
where $\left\Vert f\right\Vert _{\varUpsilon}=\sup_{x\in\R^{\dimension}}\left(1+\left|x\right|\right)^{\varUpsilon}\cdot\left|f\left(x\right)\right|$
for $f:\R^{\dimension}\to\Compl^{\ell}$ and (arbitrary) $\ell\in\N$.
\item We have $\Fourier\gamma_{k}^{\left(0,j\right)}\in C^{\infty}\left(\R^{\dimension}\right)$
and all partial derivatives of $\Fourier\gamma_{k}^{\left(0,j\right)}$
are polynomially bounded for all $k\in\underline{n}$ and $j\in\left\{ 1,2\right\} $.
\item We have $\gamma_{k}^{\left(0\right)}=\gamma_{k}^{\left(0,1\right)}\ast\gamma_{k}^{\left(0,2\right)}$
for all $k\in\underline{n}$.
\item We have $\left[\Fourier\gamma_{k}^{\left(0\right)}\right]\left(\xi\right)\neq0$
for all $\xi\in\overline{Q_{0}^{\left(k\right)}}$ and all $k\in\underline{n}$.
\item We have $\left\Vert \gamma_{k}^{\left(0\right)}\right\Vert _{\varUpsilon}<\infty$
for all $k\in\underline{n}$.
\item We have
\[
K_{1}:=\sup_{i\in I}\:\sum_{j\in I}N_{i,j}<\infty\qquad\text{ and }\qquad K_{2}:=\sup_{j\in I}\:\sum_{i\in I}N_{i,j}<\infty,
\]
where $\gamma_{j,1}:=\gamma_{k_{j}}^{\left(0,1\right)}$ for $j\in I$
and
\[
\qquad\qquad N_{i,j}:=\left(\frac{w_{i}}{w_{j}}\cdot\left(\left|\det T_{j}\right|\big/\left|\det T_{i}\right|\right)^{\vartheta}\right)^{\tau}\!\!\cdot\left(1\!+\!\left\Vert T_{j}^{-1}T_{i}\right\Vert \right)^{\sigma}\!\cdot\left(\left|\det T_{i}\right|^{-1}\!\cdot\int_{Q_{i}}\:\max_{\left|\alpha\right|\leq N}\left|\left[\partial^{\alpha}\widehat{\gamma_{j,1}}\right]\left(T_{j}^{-1}\left(\xi\!-\!b_{j}\right)\right)\right|\d\xi\right)^{\tau}.
\]
\end{enumerate}
Then there is some $\delta_{0}\in\left(0,1\right]$ such that the
family 
\[
\Psi_{\delta}:=\left(L_{\delta\cdot T_{i}^{-T}k}\:\gamma^{\left[i\right]}\right)_{i\in I,\,k\in\Z^{\dimension}}\qquad\text{ with }\qquad\gamma^{\left[i\right]}=\left|\det T_{i}\right|^{1/2}\cdot M_{b_{i}}\left[\gamma_{i}\circ T_{i}^{T}\right]
\]
forms an \textbf{atomic decomposition} of $\DecompSp{\CalQ}p{\ell_{w}^{q}}{}$,
for all $\delta\in\left(0,\delta_{0}\right]$. Precisely, this means
the following:

\begin{itemize}[leftmargin=0.7cm]
\item The \textbf{synthesis map}
\[
S^{\left(\delta\right)}:C_{w}^{p,q}\to\DecompSp{\CalQ}p{\ell_{w}^{q}}{},\left(\smash{c_{k}^{\left(i\right)}}\right)_{i\in I,\,k\in\Z^{\dimension}}\mapsto\sum_{i\in I}\:\sum_{k\in\Z^{\dimension}}\left[c_{k}^{\left(i\right)}\cdot L_{\delta\cdot T_{i}^{-T}k}\:\gamma^{\left[i\right]}\right]
\]
is well-defined and bounded for every $\delta\in\left(0,1\right]$.
\item For $0<\delta\leq\delta_{0}$, there is a bounded linear \textbf{coefficient
map} $C^{\left(\delta\right)}:\DecompSp{\CalQ}p{\ell_{w}^{q}}{}\to C_{w}^{p,q}$
satisfying 
\[
S^{\left(\delta\right)}\circ C^{\left(\delta\right)}=\identity_{\DecompSp{\CalQ}p{\ell_{w}^{q}}{}}.
\]
\end{itemize}
Finally, there is an estimate for the size of $\delta_{0}$ which
is independent of $p\geq p_{0}$ and $q\geq q_{0}$: There is a constant
$K=K\left(p_{0},q_{0},\varepsilon,\dimension,\CalQ,\Phi,\smash{\gamma_{1}^{\left(0\right)},\dots,\gamma_{n}^{\left(0\right)}}\right)>0$
such that we can choose
\[
\delta_{0}=\min\left\{ 1,\,\left[K\cdot\Omega^{\left(p\right)}\cdot\left(K_{1}^{1/\tau}+K_{2}^{1/\tau}\right)\right]^{-1}\right\} .\qedhere
\]
\end{thm}
\begin{rem*}

\begin{itemize}[leftmargin=0.4cm]
\item Convergence of the series defining $S^{\left(\delta\right)}$ has
to be understood as follows: For each $i\in I$, the series
\[
\sum_{k\in\Z^{\dimension}}\left[c_{k}^{\left(i\right)}\cdot L_{\delta\cdot T_{i}^{-T}k}\:\gamma^{\left[i\right]}\right]
\]
converges pointwise absolutely to a function $g_{j}\in L_{{\rm loc}}^{1}\left(\R^{\dimension}\right)\cap\Schwartz'\left(\R^{\dimension}\right)$
and the series $\sum_{j\in I}\,g_{j}=S^{\left(\delta\right)}\left(\smash{c_{k}^{\left(i\right)}}\right)_{i\in I,k\in\Z^{\dimension}}$
converges unconditionally in the weak-$\ast$-sense in $Z'\left(\CalO\right)$,
i.e., for every $\phi\in Z\left(\CalO\right)=\Fourier\left(\TestFunctionSpace{\CalO}\right)$,
the series $\sum_{j\in I}\left\langle g_{j},\,\phi\right\rangle _{\Schwartz',\Schwartz}$
converges absolutely and the functional $\phi\mapsto\sum_{j\in I}\left\langle g_{j},\,\phi\right\rangle _{\Schwartz',\Schwartz}$
is continuous on $Z\left(\CalO\right)$.
\item The action of $C^{\left(\delta\right)}$ on a given $f\in\DecompSp{\CalQ}p{\ell_{w}^{q}}{}$
is \emph{independent} of the precise choice of $p,q,w$, as long as
$C^{\left(\delta\right)}f$ is defined at all.\qedhere
\end{itemize}
\end{rem*}
\begin{proof}[Proof of Theorem \ref{thm:AtomicDecompositionTheorem}]
This is a special case of Theorem \ref{thm:WeightedAtomicDecompositionTheorem},
for $\Omega_{0}=\Omega_{1}=1$ and $v=v_{0}\equiv1$.
\end{proof}
The main limitation of Theorem \ref{thm:AtomicDecompositionTheorem}—in
comparison to Theorem \ref{thm:BanachFrameTheorem}—is that we require
each $\gamma_{k}^{\left(0\right)}$ to be factorized as a convolution
product $\gamma_{k}^{\left(0\right)}=\gamma_{k}^{\left(0,1\right)}\ast\gamma_{k}^{\left(0,2\right)}$,
which is tedious to verify. To simplify such verifications, the following
result is helpful:
\begin{prop}
\label{prop:ConvolutionFactorization}(cf.\@ \cite[Lemma 6.9]{StructuredBanachFrames})

Let $\varrho\in L^{1}\left(\R^{\dimension}\right)$ with $\varrho\geq0$.
Let $N\in\N$ with $N\geq\dimension+1$ and assume that $\gamma\in L^{1}\left(\R^{\dimension}\right)$
satisfies $\widehat{\gamma}\in C^{N}\left(\R^{\dimension}\right)$
with
\[
\left|\partial^{\alpha}\widehat{\gamma}\left(\xi\right)\right|\leq\varrho\left(\xi\right)\cdot\left(1+\left|\xi\right|\right)^{-\left(\dimension+1+\varepsilon\right)}\qquad\forall\xi\in\R^{\dimension}\quad\forall\alpha\in\N_{0}^{\dimension}\text{ with }\left|\alpha\right|\leq N
\]
for some $\varepsilon\in\left(0,1\right]$.

Then there are functions $\gamma_{1}\in C_{0}\left(\R^{\dimension}\right)\cap L^{1}\left(\R^{\dimension}\right)$
and $\gamma_{2}\in C^{1}\left(\R^{\dimension}\right)\cap W^{1,1}\left(\R^{\dimension}\right)$
with $\gamma=\gamma_{1}\ast\gamma_{2}$ and with the following additional
properties:

\begin{enumerate}
\item We have $\left\Vert \gamma_{2}\right\Vert _{K}\leq s_{\dimension}\cdot2^{1+\dimension+3K}\cdot K!\cdot\left(1+\dimension\right)^{1+2K}$
and $\left\Vert \nabla\gamma_{2}\right\Vert _{K}\leq\frac{s_{\dimension}}{\varepsilon}\cdot2^{4+\dimension+3K}\cdot\left(1+\dimension\right)^{2\left(1+K\right)}\cdot\left(K+1\right)!$
for all $K\in\N_{0}$, where $\left\Vert g\right\Vert _{K}:=\sup_{x\in\R^{\dimension}}\left(1+\left|x\right|\right)^{K}\left|g\left(x\right)\right|$.
\item We have $\widehat{\gamma_{2}}\in C^{\infty}\left(\R^{\dimension}\right)$
with all partial derivatives of $\widehat{\gamma_{2}}$ being polynomially
bounded (even bounded).
\item If $\widehat{\gamma}\in C^{\infty}\left(\R^{\dimension}\right)$ with
all partial derivatives being polynomially bounded, the same also
holds for $\widehat{\gamma_{1}}$.
\item We have $\left\Vert \gamma_{1}\right\Vert _{N}\leq\left(1+\dimension\right)^{1+2N}\cdot2^{1+\dimension+4N}\cdot N!\cdot\left\Vert \varrho\right\Vert _{L^{1}}$
and $\left\Vert \gamma\right\Vert _{N}\leq\left(1+\dimension\right)^{N+1}\cdot\left\Vert \varrho\right\Vert _{L^{1}}$.
\item We have $\left|\partial^{\alpha}\widehat{\gamma_{1}}\left(\xi\right)\right|\leq2^{1+\dimension+4N}\cdot N!\cdot\left(1+\dimension\right)^{N}\cdot\varrho\left(\xi\right)$
for all $\xi\in\R^{\dimension}$ and $\alpha\in\N_{0}^{\dimension}$
with $\left|\alpha\right|\leq N$.\qedhere
\end{enumerate}
\end{prop}

\section{Definition and basic properties of \texorpdfstring{$\alpha$}{α}-shearlet
smoothness spaces}

\label{sec:AlphaShearletSmoothnessDefinition}In this section, we
introduce the class of \textbf{$\alpha$-shearlet smoothness spaces}.
These spaces are a generalization of the ``ordinary'' shearlet smoothness
spaces as introduced by Labate et al.\cite{Labate_et_al_Shearlet}.
Later on (cf.\@ Theorem \ref{thm:AnalysisAndSynthesisSparsityAreEquivalent}),
it will turn out that these spaces simultaneously describe analysis
and synthesis sparsity with respect to (suitable) $\alpha$-shearlet
frames.

We will define the $\alpha$-shearlet smoothness spaces as certain
decomposition spaces. Thus, we first have to define the associated
covering and the weight for the sequence space $\ell_{w}^{q}\left(I\right)$
that we will use:
\begin{defn}
\label{def:AlphaShearletCovering}Let $\alpha\in\left[0,1\right]$.
The \textbf{$\alpha$-shearlet covering} $\CalS^{\left(\alpha\right)}$
is defined as 
\[
\CalS^{\left(\alpha\right)}:=\left(\smash{S_{i}^{\left(\alpha\right)}}\right)_{i\in I^{\left(\alpha\right)}}=\left(\smash{T_{i}}Q_{i}'\right)_{i\in I^{\left(\alpha\right)}}=\left(\smash{T_{i}}Q_{i}'+b_{i}\right)_{i\in I^{\left(\alpha\right)}},
\]
where:

\begin{itemize}[leftmargin=0.6cm]
\item The \emph{index set} $I^{\left(\alpha\right)}$ is given by $I:=I^{\left(\alpha\right)}:=\left\{ 0\right\} \cup I_{0}$,
where 
\[
\qquad I_{0}:=I_{0}^{(\alpha)}:=\left\{ \left(n,m,\varepsilon,\delta\right)\in\N_{0}\times\Z\times\left\{ \pm1\right\} \times\left\{ 0,1\right\} \with\left|m\right|\leq G_{n}\right\} \quad\text{ with }\quad G_{n}:=G_{n}^{\left(\alpha\right)}:=\left\lceil \smash{2^{n\left(1-\alpha\right)}}\right\rceil .
\]
\item The \emph{basic sets} $\left(Q_{i}'\right)_{i\in I^{\left(\alpha\right)}}$
are given by $Q_{0}':=\left(-1,1\right)^{2}$ and by $Q_{i}':=Q:=U_{\left(-1,1\right)}^{\left(3^{-1},3\right)}$
for $i\in I_{0}^{\left(\alpha\right)}$, where we used the notation
\begin{equation}
U_{(a,b)}^{\left(\gamma,\mu\right)}:=\left\{ \begin{pmatrix}\xi\\
\eta
\end{pmatrix}\in\left(\gamma,\mu\right)\times\R\left|\frac{\eta}{\xi}\in\left(a,b\right)\right.\right\} \quad\text{ for }a,b\in\R\text{ and }\gamma,\mu\in\left(0,\infty\right).\label{eq:BasicShearletSet}
\end{equation}
\item The \emph{matrices} $\left(T_{i}\right)_{i\in I^{\left(\alpha\right)}}$
are given by $T_{0}:=\identity$ and by $T_{i}:=T_{i}^{\left(\alpha\right)}:=R^{\delta}\cdot A_{n,m,\varepsilon}^{\left(\alpha\right)}$,
with $A_{n,m,\varepsilon}^{\left(\alpha\right)}:=\varepsilon\cdot D_{2^{n}}^{\left(\alpha\right)}\cdot S_{m}^{T}$
for $i=\left(n,m,\varepsilon,\delta\right)\in I_{0}^{\left(\alpha\right)}$.
Here, the matrices $R,S_{x}$ and $D_{b}^{\left(\alpha\right)}$ are
as in equation \eqref{eq:StandardMatrices}.
\item The \emph{translations} $\left(b_{i}\right)_{i\in I^{\left(\alpha\right)}}$
are given by $b_{i}:=0$ for all $i\in I^{\left(\alpha\right)}$.
\end{itemize}
Finally, we define the \emph{weight} $w=\left(w_{i}\right)_{i\in I}$
by $w_{0}:=1$ and $w_{n,m,\varepsilon,\delta}:=2^{n}$ for $\left(n,m,\varepsilon,\delta\right)\in I_{0}$.
\end{defn}
Our first goal is to show that the covering $\CalS^{\left(\alpha\right)}$
is an almost structured covering of $\R^{2}$ (cf.\@ Definition \ref{def:AlmostStructuredCovering}).
To this end, we begin with the following auxiliary lemma:
\begin{lem}
\label{lem:AlphaShearletCoveringAuxiliary}

\begin{enumerate}[leftmargin=0.6cm]
\item Using the notation $U_{\left(a,b\right)}^{\left(\gamma,\mu\right)}$
from equation \eqref{eq:BasicShearletSet} and the shearing matrices
$S_{x}$ from equation \eqref{eq:StandardMatrices}, we have for arbitrary
$m,a,b\in\R$ and $\kappa,\lambda,\gamma,\mu>0$ that
\begin{equation}
S_{m}^{T}U_{\left(a,b\right)}^{\left(\gamma,\mu\right)}=U_{\left(m+a,m+b\right)}^{\left(\gamma,\mu\right)}\qquad\text{ and }\qquad{\rm diag}\left(\lambda,\,\kappa\right)U_{\left(a,b\right)}^{\left(\gamma,\mu\right)}=U_{\left(\frac{\kappa}{\lambda}a,\frac{\kappa}{\lambda}b\right)}^{\left(\lambda\gamma,\lambda\mu\right)}.\label{eq:BaseSetTransformationRules}
\end{equation}
Consequently,
\begin{equation}
T_{i}^{\left(\alpha\right)}U_{\left(a,b\right)}^{\left(\gamma,\mu\right)}=\varepsilon\cdot U_{\left(2^{n\left(\alpha-1\right)}\left(m+a\right),2^{n\left(\alpha-1\right)}\left(m+b\right)\right)}^{\left(2^{n}\gamma,\,2^{n}\mu\right)}\qquad\text{ for all }\quad i=\left(n,m,\varepsilon,0\right)\in I_{0}.\label{eq:deltazeroset}
\end{equation}
In particular, $S_{n,m,\varepsilon,0}^{\left(\alpha\right)}=\varepsilon\cdot U_{\left(2^{n\left(\alpha-1\right)}\left(m-1\right),2^{n\left(\alpha-1\right)}\left(m+1\right)\right)}^{\left(2^{n}/3,\,3\cdot2^{n}\right)}$.
\item Let $i=\left(n,m,\varepsilon,\delta\right)\in I_{0}$ and let $\left(\begin{smallmatrix}\xi\\
\eta
\end{smallmatrix}\right)\in S_{i}^{\left(\alpha\right)}$ be arbitrary. Then the following hold:

\begin{enumerate}
\item If $i=\left(n,m,\varepsilon,0\right)$, we have $\left|\eta\right|<3\cdot\left|\xi\right|$.
\item If $i=\left(n,m,\varepsilon,1\right)$, we have $\left|\xi\right|<3\cdot\left|\eta\right|$.
\item We have $2^{n-2}<\frac{2^{n}}{3}<\left|\left(\begin{smallmatrix}\xi\\
\eta
\end{smallmatrix}\right)\right|<12\cdot2^{n}<2^{n+4}$.\qedhere
\end{enumerate}
\end{enumerate}
\end{lem}
\begin{proof}
We establish the different claims individually:

\begin{enumerate}[leftmargin=0.6cm]
\item The following is essentially identical with the proof of \cite[Lemma 6.3.4]{VoigtlaenderPhDThesis}
and is only given here for the sake of completeness. We first observe
the following equivalences:
\begin{align*}
\left(\begin{matrix}\xi\\
\eta
\end{matrix}\right)\in U_{\left(m+a,m+b\right)}^{\left(\gamma,\delta\right)} & \Longleftrightarrow\xi\in\left(\gamma,\delta\right)\quad\text{ and }\quad m+a<\frac{\eta}{\xi}<m+b\\
 & \Longleftrightarrow\xi\in\left(\gamma,\delta\right)\quad\text{ and }\quad a<\frac{\eta-m\xi}{\xi}<b\\
 & \Longleftrightarrow\left(\begin{matrix}1 & 0\\
m & 1
\end{matrix}\right)^{-1}\left(\begin{matrix}\xi\\
\eta
\end{matrix}\right)=\left(\begin{matrix}\xi\\
\eta-m\xi
\end{matrix}\right)\in U_{\left(a,b\right)}^{\left(\gamma,\delta\right)}
\end{align*}
and
\begin{align*}
\left(\begin{matrix}\xi\\
\eta
\end{matrix}\right)\in U_{\left(\frac{\kappa}{\lambda}a,\frac{\kappa}{\lambda}b\right)}^{\left(\lambda\gamma,\lambda\mu\right)} & \Longleftrightarrow\xi\in\left(\lambda\gamma,\lambda\mu\right)\quad\text{ and }\quad\frac{\kappa}{\lambda}a<\frac{\eta}{\xi}<\frac{\kappa}{\lambda}b\\
 & \Longleftrightarrow\lambda^{-1}\xi\in\left(\gamma,\mu\right)\quad\text{ and }\quad a<\frac{\kappa^{-1}\eta}{\lambda^{-1}\xi}<b\\
 & \Longleftrightarrow\left(\begin{matrix}\lambda & 0\\
0 & \kappa
\end{matrix}\right)^{-1}\left(\begin{matrix}\xi\\
\eta
\end{matrix}\right)=\left(\begin{matrix}\lambda^{-1}\xi\\
\kappa^{-1}\eta
\end{matrix}\right)\in U_{\left(a,b\right)}^{\left(\gamma,\mu\right)}.
\end{align*}
These equivalences show ${\rm diag}\left(\lambda,\kappa\right)U_{\left(a,b\right)}^{\left(\gamma,\mu\right)}=U_{\left(\frac{\kappa}{\lambda}a,\frac{\kappa}{\lambda}b\right)}^{\left(\lambda\gamma,\lambda\mu\right)}$
and $S_{m}^{T}U_{\left(a,b\right)}^{\left(\gamma,\mu\right)}=U_{\left(m+a,m+b\right)}^{\left(\gamma,\mu\right)}$.
But for $i=\left(n,m,\varepsilon,0\right)$, we have $T_{i}^{\left(\alpha\right)}=R^{0}\cdot A_{n,m,\varepsilon}^{\left(\alpha\right)}=\varepsilon\cdot{\rm diag}\left(2^{n},2^{n\alpha}\right)\cdot S_{m}^{T}$.
This easily yields the claim.
\item We again show the three claims individually:

\begin{enumerate}
\item For $i=\left(n,m,\varepsilon,0\right)\in I_{0}$, equation \eqref{eq:deltazeroset}
yields for $\left(\begin{smallmatrix}\xi\\
\eta
\end{smallmatrix}\right)\in S_{i}^{(\alpha)}$ that 
\[
\frac{\eta}{\xi}\in\left(2^{n\left(\alpha-1\right)}\left(m-1\right),2^{n\left(\alpha-1\right)}\left(m+1\right)\right)\subset\left(-2^{n\left(\alpha-1\right)}\left(\left|m\right|+1\right),\,2^{n\left(\alpha-1\right)}\left(\left|m\right|+1\right)\right),
\]
since $2^{n\left(\alpha-1\right)}\left(m+1\right)\leq2^{n\left(\alpha-1\right)}\left(\left|m\right|+1\right)$
and 
\[
2^{n\left(\alpha-1\right)}\left(m-1\right)\geq2^{n\left(\alpha-1\right)}\left(-|m|-1\right)=-2^{n\left(\alpha-1\right)}\left(|m|+1\right).
\]
Because of $|m|\leq G_{n}=\lceil2^{n(1-\alpha)}\rceil<2^{n(1-\alpha)}+1$
and $\left|\xi\right|>0$, it follows that
\[
\begin{aligned}\left|\eta\right|=\left|\xi\right|\cdot\left|\frac{\eta}{\xi}\right|\leq\left|\xi\right|\cdot2^{n(\alpha-1)}\cdot\left(\left|m\right|+1\right) & <\left|\xi\right|\cdot2^{-n(1-\alpha)}\cdot\left(2^{n(1-\alpha)}+2\right)\\
 & \leq\left|\xi\right|\cdot\left(1+2\cdot2^{-n(1-\alpha)}\right)\leq3\cdot\left|\xi\right|.
\end{aligned}
\]
\item For $i=\left(n,m,\varepsilon,1\right)\in I_{0}$ we have 
\[
\begin{pmatrix}\eta\\
\xi
\end{pmatrix}=R\cdot\begin{pmatrix}\xi\\
\eta
\end{pmatrix}\in RS_{n,m,\varepsilon,1}^{(\alpha)}=RT_{n,m,\varepsilon,1}^{\left(\alpha\right)}Q=RRA_{n,m,\varepsilon}^{(\alpha)}Q=A_{n,m,\varepsilon}^{(\alpha)}Q=S_{n,m,\varepsilon,0}^{(\alpha)},
\]
so that we get $\left|\xi\right|<3\cdot\left|\eta\right|$ from the
previous case.
\item To prove this claim, we again distinguish two cases:

\begin{enumerate}
\item For $i=\left(n,m,\varepsilon,0\right)$, equation \eqref{eq:deltazeroset}
yields $\varepsilon\xi\in(2^{n}/3,3\cdot2^{n})$ and thus $\frac{2^{n}}{3}<|\xi|<3\cdot2^{n}$.
Moreover, we know from a previous part of the lemma that $|\eta|<3\cdot|\xi|$.
Thus 
\begin{align*}
\frac{2^{n}}{3}<|\xi|\leq\left|\begin{pmatrix}\xi\\
\eta
\end{pmatrix}\right| & \leq|\xi|+|\eta|<|\xi|+3|\xi|=4|\xi|<12\cdot2^{n}.
\end{align*}
\item For $i=\left(n,m,\varepsilon,1\right)$ we have $\varepsilon\eta\in(2^{n}/3,3\cdot2^{n})$
and thus $\frac{2^{n}}{3}<|\eta|<3\cdot2^{n}$. Moreover, we know
from the previous part of the lemma that $|\xi|<3\cdot|\eta|$. Thus
\begin{align*}
\frac{2^{n}}{3}<|\eta|\leq\left|\begin{pmatrix}\xi\\
\eta
\end{pmatrix}\right| & \leq|\xi|+|\eta|<3|\eta|+|\eta|=4|\eta|<12\cdot2^{n}.\qedhere
\end{align*}
\end{enumerate}
\end{enumerate}
\end{enumerate}
\end{proof}
Using the preceding lemma—which will also be frequently useful elsewhere—one
can show the following:
\begin{lem}
\noindent \label{lem:AlphaShearletCoveringIsAlmostStructured}The
$\alpha$-shearlet covering $\CalS^{(\alpha)}$ from Definition \ref{def:AlphaShearletCovering}
is an almost structured covering of $\R^{2}$.
\end{lem}
Since the proof of Lemma \ref{lem:AlphaShearletCoveringIsAlmostStructured}
is quite lengthy, although it does not yield too much insight, we
postpone it to the appendix (Section \ref{sec:AlphaShearletCoveringAlmostStructured}).

Finally, before we can formally define the $\alpha$-shearlet smoothness
spaces, we still need to verify that the weight $w$ from Definition
\ref{def:AlphaShearletCovering} is $\CalS^{\left(\alpha\right)}$-moderate
(cf.\@ Definition \ref{def:QModerateWeightClusteringMap}).
\begin{lem}
\label{lem:AlphaShearletWeightIsModerate}For arbitrary $s\in\R$,
the weight $w^{s}=\left(w_{i}^{s}\right)_{i\in I}$, with $w=\left(w_{i}\right)_{i\in I}$
as in Definition \ref{def:AlphaShearletCovering}, is $\CalS^{\left(\alpha\right)}$-moderate
(cf.\@ equation \eqref{eq:ModerateWeightDefinition}) with
\[
C_{\CalS^{\left(\alpha\right)},w^{s}}\leq39^{\left|s\right|}.
\]
Furthermore, we have
\[
\frac{1}{3}\cdot w_{i}\leq1+\left|\xi\right|\leq13\cdot w_{i}\qquad\forall\:i\in I\text{ and all }\xi\in S_{i}^{\left(\alpha\right)}.\qedhere
\]
\end{lem}
\begin{proof}
First, let $i=\left(n,m,\varepsilon,\delta\right)\in I_{0}$ be arbitrary.
By Lemma \ref{lem:AlphaShearletCoveringAuxiliary}, we get
\[
\frac{1}{3}\cdot w_{i}=\frac{2^{n}}{3}\leq\left|\xi\right|\leq1+\left|\xi\right|\leq1+12\cdot2^{n}\leq13\cdot2^{n}=13\cdot w_{i}\qquad\forall\xi\in S_{i}^{\left(\alpha\right)}.
\]
Furthermore, for $i=0$, we have $S_{i}^{\left(\alpha\right)}=\left(-1,1\right)^{2}$
and thus
\[
\frac{1}{3}\cdot w_{i}\leq w_{i}=1\leq1+\left|\xi\right|\leq3=3\cdot w_{i}\leq13\cdot w_{i}\qquad\forall\xi\in S_{i}^{\left(\alpha\right)}.
\]
This establishes the second part of the lemma.

Next, let $i,j\in I$ with $S_{i}^{\left(\alpha\right)}\cap S_{j}^{\left(\alpha\right)}\neq\emptyset$.
Pick an arbitrary $\xi\in S_{i}^{\left(\alpha\right)}\cap S_{j}^{\left(\alpha\right)}$
and note as a consequence of the preceding estimates that
\[
\frac{w_{i}}{w_{j}}\leq\frac{3\cdot\left(1+\left|\xi\right|\right)}{\frac{1}{13}\cdot\left(1+\left|\xi\right|\right)}=39.
\]
By symmetry, this implies $\frac{1}{39}\leq\frac{w_{i}}{w_{j}}\leq39$
and thus also
\[
\frac{w_{i}^{s}}{w_{j}^{s}}=\left(\frac{w_{i}}{w_{j}}\right)^{s}\leq39^{\left|s\right|}.\qedhere
\]
\end{proof}
Now, we can finally formally define the $\alpha$-shearlet smoothness
spaces:
\begin{defn}
\label{def:AlphaShearletSmoothnessSpaces}For $\alpha\in\left[0,1\right]$,
$p,q\in\left(0,\infty\right]$ and $s\in\R$, we define the \textbf{$\alpha$-shearlet
smoothness space} $\mathscr{S}_{\alpha,s}^{p,q}\left(\R^{2}\right)$
associated to these parameters as
\[
\mathscr{S}_{\alpha,s}^{p,q}\left(\R^{2}\right):=\DecompSp{\smash{\CalS^{\left(\alpha\right)}}}p{\ell_{w^{s}}^{q}}{},
\]
where the covering $\CalS^{\left(\alpha\right)}$ and the weight $w^{s}$
are as in Definition \ref{def:AlphaShearletCovering} and Lemma \ref{lem:AlphaShearletWeightIsModerate},
respectively.
\end{defn}
\begin{rem*}
Since $\CalS^{\left(\alpha\right)}$ is an almost structured covering
by Lemma \ref{lem:AlphaShearletCoveringIsAlmostStructured} and since
$w^{s}$ is $\CalS^{\left(\alpha\right)}$-moderate by Lemma \ref{lem:AlphaShearletWeightIsModerate},
Definition \ref{def:FourierSideDecompositionSpaces} and the associated
remark show that $\mathscr{S}_{\alpha,s}^{p,q}\left(\R^{2}\right)$
is indeed well-defined, i.e., independent of the chosen regular partition
of unity subordinate to $\CalS^{\left(\alpha\right)}$. The same remark
also implies that $\mathscr{S}_{\alpha,s}^{p,q}\left(\R^{2}\right)$
is a Quasi-Banach space.
\end{rem*}
Recall that with our definition of decomposition spaces, $\mathscr{S}_{\alpha,s}^{p,q}\left(\R^{2}\right)$
is a subspace of $Z'\left(\R^{2}\right)=\left[\Fourier\left(\TestFunctionSpace{\R^{2}}\right)\right]'$.
But as our next result shows, each $f\in\mathscr{S}_{\alpha,s}^{p,q}\left(\R^{2}\right)$
actually extends to a tempered distribution:
\begin{lem}
\label{lem:AlphaShearletIntoTemperedDistributions}Let $\alpha\in\left[0,1\right]$,
$p,q\in\left(0,\infty\right]$ and $s\in\R$. Then
\[
\mathscr{S}_{\alpha,s}^{p,q}\left(\smash{\R^{2}}\right)\hookrightarrow\Schwartz'\left(\smash{\R^{2}}\right),
\]
in the sense that each $f\in\mathscr{S}_{\alpha,s}^{p,q}\left(\smash{\R^{2}}\right)$
extends to a uniquely determined tempered distribution $f_{\Schwartz}\in\Schwartz'\left(\R^{2}\right)$.
Furthermore, the map $\mathscr{S}_{\alpha,s}^{p,q}\left(\smash{\R^{2}}\right)\hookrightarrow\Schwartz'\left(\smash{\R^{2}}\right),f\mapsto f_{\Schwartz}$
is linear and continuous with respect to the weak-$\ast$-topology
on $\Schwartz'\left(\R^{2}\right)$.
\end{lem}
\begin{proof}
It is well known (cf.\@ \cite[Proposition 9.9]{FollandRA}) that
$\TestFunctionSpace{\R^{2}}\leq\Schwartz\left(\R^{2}\right)$ is dense.
Since $\Fourier:\Schwartz\left(\R^{2}\right)\to\Schwartz\left(\R^{2}\right)$
is a homeomorphism, we see that $Z\left(\R^{2}\right)=\Fourier\left(\TestFunctionSpace{\R^{2}}\right)\leq\Schwartz\left(\R^{2}\right)$
is dense, too. Hence, for arbitrary $f\in\mathscr{S}_{\alpha,s}^{p,q}\left(\smash{\R^{2}}\right)$,
if there is \emph{any} extension $g\in\Schwartz'\left(\R^{2}\right)$
of $f\in Z'\left(\R^{2}\right)$, then $g$ is uniquely determined.

Next, by Lemma \ref{lem:AlphaShearletCoveringIsAlmostStructured},
$\CalS^{\left(\alpha\right)}$ is almost structured, so that \cite[Theorem 8.2]{DecompositionEmbedding}
shows that $\CalS^{\left(\alpha\right)}$ is a regular covering of
$\R^{2}$. Thus, once we verify that there is some $N\in\N_{0}$ such
that the sequence $w^{\left(N\right)}=\left(\smash{w_{i}^{\left(N\right)}}\right)_{i\in I}$
defined by
\[
w_{i}^{\left(N\right)}:=\left|\det\smash{T_{i}^{\left(\alpha\right)}}\vphantom{T_{i}}\right|^{1/p}\cdot\max\left\{ 1,\,\left\Vert T_{i}^{-1}\right\Vert ^{2+1}\right\} \cdot\left[\vphantom{\sum_{i}}\smash{\inf_{\xi\in\vphantom{S_{i}^{\left(\alpha\right)}}\left(\smash{S_{i}^{\left(\alpha\right)}}\right)^{\ast}}}\left(1+\left|\xi\right|\right)\right]^{-N}
\]
satisfies $w^{\left(N\right)}\in\ell_{1/w^{s}}^{q'}\left(I\right)$
with $q'=\infty$ in case of $q\in\left(0,1\right)$, then the claim
of the present lemma is a consequence of \cite[Theorem 8.3]{DecompositionEmbedding}
and the associated remark. Here, $\vphantom{S_{i}^{\left(\alpha\right)}}\left(\smash{S_{i}^{\left(\alpha\right)}}\right)^{\ast}=\bigcup_{j\in i^{\ast}}S_{j}^{\left(\alpha\right)}$.

Since $I=\left\{ 0\right\} \cup I_{0}$ and since the single (finite(!))\@
term $w_{0}^{\left(N\right)}$ does not influence membership of $w^{\left(N\right)}$
in $\ell_{1/w^{s}}^{q'}$, we only need to show $w^{\left(N\right)}|_{I_{0}}\in\ell_{1/w^{s}}^{q'}\left(I_{0}\right)$.
But for $i=\left(n,m,\varepsilon,\delta\right)\in I_{0}$, we have
\[
\left\Vert T_{i}^{-1}\right\Vert =\left\Vert \left(\begin{matrix}2^{-n} & 0\\
-2^{-n}m & 2^{-\alpha n}
\end{matrix}\right)\right\Vert \leq3.
\]
Here, the last step used that $\left|2^{-n}\right|\leq1$, $\left|2^{-\alpha n}\right|\leq1$
and that $\left|m\right|\leq G_{n}=\left\lceil 2^{n\left(1-\alpha\right)}\right\rceil \leq\left\lceil 2^{n}\right\rceil =2^{n}$,
so that $\left|-2^{-n}m\right|\leq1$ as well.

Furthermore, Lemma \ref{lem:AlphaShearletCoveringAuxiliary} shows
$\frac{2^{n}}{3}\leq\left|\xi\right|\leq12\cdot2^{n}$ for all $\xi\in S_{i}^{\left(\alpha\right)}$.
In particular, since we have $\left|\xi\right|\leq2$ for arbitrary
$\xi\in S_{0}^{\left(\alpha\right)}=\left(-1,1\right)^{2}$, we have
$i^{\ast}\subset I_{0}$ as soon as $\frac{2^{n}}{3}>2$, i.e., for
$n\geq3$. Now, for $n\geq3$ and $j=\left(\nu,\mu,e,d\right)\in i^{\ast}\subset I_{0}$,
there is some $\eta\in S_{i}^{\left(\alpha\right)}\cap S_{j}^{\left(\alpha\right)}$,
so that Lemma \ref{lem:AlphaShearletCoveringAuxiliary} yields $\frac{2^{n}}{3}\leq\left|\eta\right|\leq12\cdot2^{\nu}$.
Another application of Lemma \ref{lem:AlphaShearletCoveringAuxiliary}
then shows $\left|\xi\right|\geq\frac{2^{\nu}}{3}\geq\frac{1}{3^{2}\cdot12}\cdot2^{n}=\frac{2^{n}}{108}$
for all $\xi\in S_{j}^{\left(\alpha\right)}$. All in all, we have
shown $1+\left|\xi\right|\geq\left|\xi\right|\geq\frac{2^{n}}{108}$
for all $\xi\in\left(\smash{S_{i}^{\left(\alpha\right)}}\right)^{\ast}$
for arbitrary $i=\left(n,m,\varepsilon,\delta\right)\in I_{0}$ with
$n\geq3$. But in case of $n\leq2$, we simply have $1+\left|\xi\right|\geq1\geq\frac{2^{n}}{108}$,
so that this estimate holds for all $i=\left(n,m,\varepsilon,\delta\right)\in I_{0}$.

Overall, we conclude
\[
w_{i}^{\left(N\right)}\leq3^{3}\cdot2^{\left(1+\alpha\right)\frac{n}{p}}\cdot\left(\frac{2^{n}}{108}\right)^{-N}=3^{3}\cdot108^{N}\cdot2^{n\left(\frac{1+\alpha}{p}-N\right)}\qquad\forall\:i=\left(n,m,\varepsilon,\delta\right)\in I_{0}.
\]
For arbitrary $\theta\in\left(0,1\right]$, this implies
\begin{align*}
\sum_{i=\left(n,m,\varepsilon,\delta\right)\in I_{0}}\left[\frac{1}{w_{i}^{s}}\cdot w_{i}^{\left(N\right)}\right]^{\theta} & \leq4\cdot\left(3^{3}\cdot108^{N}\right)^{\theta}\cdot\sum_{n=0}^{\infty}\:\sum_{\left|m\right|\leq G_{n}}2^{n\theta\left(\frac{1+\alpha}{p}-s-N\right)}\\
\left({\scriptstyle \text{since }G_{n}\leq2^{n}}\right) & \leq12\cdot\left(3^{3}\cdot108^{N}\right)^{\theta}\cdot\sum_{n=0}^{\infty}\:2^{\theta n\left(\frac{1}{\theta}+\frac{1+\alpha}{p}-s-N\right)}<\infty
\end{align*}
as soon as $N>\frac{1}{\theta}+\frac{1+\alpha}{p}-s$, which can always
be satisfied. Since we have $\ell^{\theta}\left(I_{0}\right)\hookrightarrow\ell^{q'}\left(I_{0}\right)$
for $\theta\leq q'$, this shows that we always have $w^{\left(N\right)}\in\ell_{1/w^{s}}^{q'}\left(I\right)$,
for sufficiently large $N\in\N_{0}$. As explained above, we can thus
invoke \cite[Theorem 8.3]{DecompositionEmbedding} to complete the
proof.
\end{proof}
Now that we have verified that the $\alpha$-shearlet smoothness spaces
are indeed well-defined (Quasi)-Banach spaces, our next goal is to
verify that the theory of structured Banach frame decompositions for
decomposition spaces—as outlined in Section \ref{sec:BanachFrameDecompositionCrashCourse}—applies
to these spaces. This is the goal of the next section. As we will
see (see e.g.\@ Theorem \ref{thm:AnalysisAndSynthesisSparsityAreEquivalent}),
this implies that the $\alpha$-shearlet smoothness spaces \emph{simultaneously}
characterize analysis sparsity and synthesis sparsity with respect
to (suitable) $\alpha$-shearlet systems.

\section{Construction of Banach frame decompositions for \texorpdfstring{$\alpha$}{α}-shearlet
smoothness spaces}

\label{sec:CompactlySupportedShearletFrames}We now want to verify
the pertinent conditions from Theorems \ref{thm:BanachFrameTheorem}
and \ref{thm:AtomicDecompositionTheorem} for the $\alpha$-shearlet
smoothness spaces. To this end, first recall from Definition \ref{def:AlphaShearletCovering}
that we have $Q_{i}'=Q$ for all $i\in I_{0}$ and furthermore $Q_{0}'=\left(-1,1\right)^{2}$.
Consequently, in the notation of Assumption \ref{assu:CrashCourseStandingAssumptions},
we can choose $n=2$ and $Q_{0}^{\left(1\right)}:=Q=U_{\left(-1,1\right)}^{\left(3^{-1},3\right)}$,
as well as $Q_{0}^{\left(2\right)}:=\left(-1,1\right)^{2}$.

We fix a \textbf{low-pass filter} $\varphi\in W^{1,1}\left(\R^{2}\right)\cap C^{1}\left(\R^{2}\right)$
and a \textbf{mother shearlet} $\psi\in W^{1,1}\left(\R^{2}\right)\cap C^{1}\left(\R^{2}\right)$.
Then we set (again in the notation of Assumption \ref{assu:CrashCourseStandingAssumptions})
$\gamma_{1}^{\left(0\right)}:=\psi$ and $\gamma_{2}^{\left(0\right)}:=\varphi$,
as well as $k_{0}:=2$ and $k_{i}:=1$ for $i\in I_{0}$. With these
choices, the family $\Gamma=\left(\gamma_{i}\right)_{i\in I}$ introduced
in Theorems \ref{thm:BanachFrameTheorem} and \ref{thm:AtomicDecompositionTheorem}
satisfies $\gamma_{i}=\gamma_{k_{i}}^{\left(0\right)}=\gamma_{1}^{\left(0\right)}=\psi$
for $i\in I_{0}$ and $\gamma_{0}=\gamma_{k_{0}}^{\left(0\right)}=\gamma_{2}^{\left(0\right)}=\varphi$,
so that the family $\Gamma$ is completely determined by $\varphi$
and $\psi$.

Our main goal in this section is to derive readily verifiable conditions
on $\varphi,\psi$ which guarantee that the generalized shift-invariant
system $\Psi_{\delta}:=\left(L_{\delta\cdot T_{i}^{-T}k}\:\gamma^{\left[i\right]}\right)_{i\in I,\,k\in\Z^{2}}$,
with $\gamma^{\left[i\right]}=\left|\det T_{i}\right|^{1/2}\cdot\gamma_{i}\circ T_{i}^{T}$,
generates, respectively, a Banach frame or an atomic decomposition
for the $\alpha$-shearlet smoothness space $\mathscr{S}_{\alpha,s}^{p,q}\left(\R^{2}\right)$,
for sufficiently small $\delta>0$.

Precisely, we assume $\widehat{\psi},\widehat{\varphi}\in C^{\infty}\left(\R^{2}\right)$,
where all partial derivatives of these functions are assumed to be
polynomially bounded. Furthermore, we assume (at least for the application
Theorem \ref{thm:BanachFrameTheorem}) that
\begin{equation}
\begin{split}\max_{\left|\beta\right|\leq1}\max_{\left|\theta\right|\leq N}\left|\left(\partial^{\theta}\widehat{\partial^{\beta}\psi}\right)\left(\xi\right)\right| & \leq C\cdot\min\left\{ \left|\xi_{1}\right|^{M_{1}}\!\!,\left(1+\left|\xi_{1}\right|\right)^{-M_{2}}\right\} \cdot\left(1+\left|\xi_{2}\right|\right)^{-K}\!=C\cdot\theta_{1}\left(\xi_{1}\right)\cdot\theta_{2}\left(\xi_{2}\right)=C\cdot\varrho\left(\xi\right),\\
\max_{\left|\beta\right|\leq1}\max_{\left|\theta\right|\leq N}\left|\left(\partial^{\theta}\widehat{\partial^{\beta}\varphi}\right)\left(\xi\right)\right| & \leq C\cdot\left(1+\left|\xi\right|\right)^{-H}=C\cdot\varrho_{0}\left(\xi\right)
\end{split}
\label{eq:MotherShearletMainEstimate}
\end{equation}
for all $\xi=\left(\xi_{1},\xi_{2}\right)\in\R^{2}$, a suitable constant
$C>0$ and certain $M_{1},M_{2},K,H\in\left[0,\infty\right)$ and
$N\in\N$. To be precise, we note that equation \eqref{eq:MotherShearletMainEstimate}
employed the abbreviations 
\[
\theta_{1}\left(\xi_{1}\right):=\min\left\{ \left|\xi_{1}\right|^{M_{1}},\left(1+\left|\xi_{1}\right|\right)^{-M_{2}}\right\} \quad\text{ and }\quad\theta_{2}\left(\xi_{2}\right):=\left(1+\left|\xi_{2}\right|\right)^{-K}\qquad\text{ for }\xi_{1},\xi_{2}\in\R,
\]
as well as $\varrho\left(\xi\right):=\theta_{1}\left(\xi_{1}\right)\cdot\theta_{2}\left(\xi_{2}\right)$
and $\varrho_{0}\left(\xi\right):=\left(1+\left|\xi\right|\right)^{-H}$
for $\xi=\left(\begin{smallmatrix}\xi_{1}\\
\xi_{2}
\end{smallmatrix}\right)\in\R^{2}$.

Our goal in the following is to derive conditions on $N,M_{1},M_{2},K,H$
(depending on $p,q,s,\alpha$) which ensure that the family $\Psi_{\delta}$
indeed forms a Banach frame or an atomic decomposition for $\mathscr{S}_{\alpha,s}^{p,q}\left(\R^{2}\right)=\DecompSp{\CalS^{\left(\alpha\right)}}p{\ell_{w^{s}}^{q}}{}$.

\medskip{}

To verify the conditions of Theorem \ref{thm:BanachFrameTheorem}
(recalling that $b_{j}=0$ for all $j\in I$), we need to estimate
the quantity
\begin{align}
M_{j,i} & :=\left(\frac{w_{j}^{s}}{w_{i}^{s}}\right)^{\tau}\cdot\left(1+\left\Vert T_{j}^{-1}T_{i}\right\Vert \right)^{\sigma}\cdot\max_{\left|\beta\right|\leq1}\left(\left|\det T_{i}\right|^{-1}\cdot\int_{S_{i}^{\left(\alpha\right)}}\max_{\left|\theta\right|\leq N}\left|\left(\partial^{\theta}\widehat{\partial^{\beta}\gamma_{j}}\right)\left(T_{j}^{-1}\xi\right)\right|\d\xi\right)^{\tau}\nonumber \\
\left({\scriptstyle \text{eq. }\eqref{eq:MotherShearletMainEstimate}}\right) & \leq C^{\tau}\cdot\left(\frac{w_{j}^{s}}{w_{i}^{s}}\right)^{\tau}\cdot\left(1+\left\Vert T_{j}^{-1}T_{i}\right\Vert \right)^{\sigma}\cdot\left(\left|\det T_{i}\right|^{-1}\cdot\int_{S_{i}^{\left(\alpha\right)}}\varrho_{j}\left(T_{j}^{-1}\xi\right)\d\xi\right)^{\tau}=:C^{\tau}\cdot M_{j,i}^{\left(0\right)}\label{eq:AlphaShearletConditionTargetTerm}
\end{align}
with $\sigma,\tau>0$ and $N\in\N$ as in Theorem \ref{thm:BanachFrameTheorem}
and arbitrary $i,j\in I$, where we defined $\varrho_{j}:=\varrho$
for $j\in I_{0}$, with $\varrho$ and $\varrho_{0}$ as defined in
equation \eqref{eq:MotherShearletMainEstimate}.

In view of equation \eqref{eq:AlphaShearletConditionTargetTerm},
the following—highly nontrivial—lemma is crucial:
\begin{lem}
\label{lem:MainShearletLemma}Let $\alpha\in\left[0,1\right]$ and
$\tau_{0},\omega,c\in\left(0,\infty\right)$. Furthermore, let $K,H,M_{1},M_{2}\in\left[0,\infty\right)$.
Then there is a constant $C_{0}=C_{0}\left(\alpha,\tau_{0},\omega,c,K,H,M_{1},M_{2}\right)>0$
with the following property:

If $\sigma,\tau\in\left(0,\infty\right)$ and $s\in\R$ satisfy $\tau\geq\tau_{0}$
and $\frac{\sigma}{\tau}\leq\omega$ and if we have $K\geq K_{0}+c$
, $M_{1}\geq M_{1}^{(0)}+c$, and $M_{2}\geq M_{2}^{(0)}+c$, as well
as $H\geq H_{0}+c$ for

\begin{align*}
K_{0} & :=\begin{cases}
\max\left\{ \frac{\sigma}{\tau}-s,\,\frac{2+\sigma}{\tau}\right\} , & \text{if }\alpha=1,\\
\max\left\{ \frac{1-\alpha}{\tau}+2\frac{\sigma}{\tau}-s,\,\frac{2+\sigma}{\tau}\right\} , & \text{if }\alpha\in\left[0,1\right),
\end{cases}\\
M_{1}^{(0)} & :=\begin{cases}
\frac{1}{\tau}+s, & \text{if }\alpha=1,\\
\frac{1}{\tau}+\max\left\{ s,\,0\right\} , & \text{if }\alpha\in\left[0,1\right),
\end{cases}\\
M_{2}^{(0)} & :=\left(1+\alpha\right)\frac{\sigma}{\tau}-s,\\
H_{0} & :=\frac{1-\alpha}{\tau}+\frac{\sigma}{\tau}-s,
\end{align*}
then we have
\[
\max\left\{ \sup_{i\in I}\sum_{j\in I}M_{j,i}^{\left(0\right)},\,\sup_{j\in I}\sum_{i\in I}M_{j,i}^{\left(0\right)}\right\} \leq C_{0}^{\tau},
\]
where $M_{j,i}^{\left(0\right)}$ is as in equation \eqref{eq:AlphaShearletConditionTargetTerm},
i.e.,
\[
M_{j,i}^{\left(0\right)}:=\left(\frac{w_{j}^{s}}{w_{i}^{s}}\right)^{\tau}\cdot\left(1+\left\Vert T_{j}^{-1}T_{i}\right\Vert \right)^{\sigma}\cdot\left(\left|\det T_{i}\right|^{-1}\cdot\int_{S_{i}^{\left(\alpha\right)}}\varrho_{j}\left(T_{j}^{-1}\xi\right)\d\xi\right)^{\tau},
\]
with $\varrho_{0}\left(\xi\right)=\left(1+\left|\xi\right|\right)^{-H}$
and $\varrho_{j}\left(\xi\right)=\min\left\{ \left|\xi_{1}\right|^{M_{1}},\left(1+\left|\xi_{1}\right|\right)^{-M_{2}}\right\} \cdot\left(1+\left|\xi_{2}\right|\right)^{-K}$
for arbitrary $j\in I_{0}$.
\end{lem}
The proof of Lemma \ref{lem:MainShearletLemma} is highly technical
and very lengthy. In order to not disrupt the flow of the paper too
severely, we deferred the proof to the appendix (Section \ref{sec:MegaProof}).

Using the general result of Lemma \ref{lem:MainShearletLemma}, we
can now derive convenient sufficient conditions concerning the low-pass
filter $\varphi$ and the mother shearlet $\psi$ which ensure that
$\varphi,\psi$ generate a Banach frame for $\mathscr{S}_{\alpha,s}^{p,q}\left(\R^{2}\right)$.
\begin{thm}
\label{thm:NicelySimplifiedAlphaShearletFrameConditions}Let $\alpha\in\left[0,1\right]$,
$\varepsilon,p_{0},q_{0}\in\left(0,1\right]$ and $s_{0},s_{1}\in\R$
with $s_{0}\leq s_{1}$. Assume that $\varphi,\psi:\R^{2}\rightarrow\Compl$
satisfy the following:

\begin{itemize}[leftmargin=0.6cm]
\item $\varphi,\psi\in L^{1}\left(\R^{2}\right)$ and $\widehat{\varphi},\widehat{\psi}\in C^{\infty}\left(\R^{2}\right)$,
where all partial derivatives of $\widehat{\varphi},\widehat{\psi}$
have at most polynomial growth.
\item $\varphi,\psi\in C^{1}\left(\R^{2}\right)$ and $\nabla\varphi,\nabla\psi\in L^{1}\left(\R^{2}\right)\cap L^{\infty}\left(\R^{2}\right)$.
\item We have 
\begin{align*}
\widehat{\psi}\left(\xi\right)\neq0 & \text{ for all }\xi=\left(\xi_{1},\xi_{2}\right)\in\R^{2}\text{ with }\xi_{1}\in\left[3^{-1},3\right]\text{ and }\left|\xi_{2}\right|\leq\left|\xi_{1}\right|,\\
\widehat{\varphi}\left(\xi\right)\ne0 & \text{ for all }\xi\in\left[-1,1\right]^{2}.
\end{align*}
\item There is some $C>0$ such that $\widehat{\psi}$ and $\widehat{\varphi}$
satisfy the estimates 
\begin{equation}
\begin{split}\left|\partial^{\theta}\smash{\widehat{\psi}}\left(\xi\right)\right| & \leq C\cdot\left|\xi_{1}\right|^{M_{1}}\left(1+\left|\xi_{2}\right|\right)^{-\left(1+K\right)}\qquad\forall\xi=\left(\xi_{1},\xi_{2}\right)\in\R^{2}\text{ with }\left|\xi_{1}\right|\leq1,\\
\left|\partial^{\theta}\smash{\widehat{\psi}}\left(\xi\right)\right| & \leq C\cdot\left(1+\left|\xi_{1}\right|\right)^{-\left(M_{2}+1\right)}\left(1+\left|\xi_{2}\right|\right)^{-\left(K+1\right)}\qquad\forall\xi=\left(\xi_{1},\xi_{2}\right)\in\R^{2},\\
\left|\partial^{\theta}\widehat{\varphi}\left(\xi\right)\right| & \leq C\cdot\left(1+\left|\xi\right|\right)^{-\left(H+1\right)}\qquad\forall\xi\in\R^{2}
\end{split}
\label{eq:ShearletFrameFourierDecayCondition}
\end{equation}
for all $\theta\in\N_{0}^{2}$ with $\left|\theta\right|\leq N_{0}$,
where $N_{0}:=\left\lceil p_{0}^{-1}\cdot\left(2+\varepsilon\right)\right\rceil $
and
\begin{align*}
K & :=\varepsilon+\max\left\{ \frac{1-\alpha}{\min\left\{ p_{0},q_{0}\right\} }+2\left(\frac{2}{p_{0}}+N_{0}\right)-s_{0},\,\frac{2}{\min\left\{ p_{0},q_{0}\right\} }+\frac{2}{p_{0}}+N_{0}\right\} ,\\
M_{1} & :=\varepsilon+\frac{1}{\min\left\{ p_{0},q_{0}\right\} }+\max\left\{ s_{1},\,0\right\} ,\\
M_{2} & :=\max\left\{ 0,\,\varepsilon+\left(1+\alpha\right)\left(\frac{2}{p_{0}}+N_{0}\right)-s_{0}\right\} ,\\
H & :=\max\left\{ 0,\,\varepsilon+\frac{1-\alpha}{\min\left\{ p_{0},q_{0}\right\} }+\frac{2}{p_{0}}+N_{0}-s_{0}\right\} .
\end{align*}
\end{itemize}
Then there is some $\delta_{0}\in\left(0,1\right]$ such that for
$0<\delta\leq\delta_{0}$ and all $p,q\in\left(0,\infty\right]$ and
$s\in\R$ with $p\geq p_{0}$, $q\geq q_{0}$ and $s_{0}\leq s\leq s_{1}$,
the following is true: The family 
\[
\widetilde{{\rm SH}}_{\alpha,\varphi,\psi,\delta}^{\left(\pm1\right)}:=\left(L_{\delta\cdot T_{i}^{-T}k}\widetilde{\gamma^{\left[i\right]}}\right)_{i\in I,k\in\Z^{2}}\quad\text{ with }\quad\widetilde{\gamma^{\left[i\right]}}(x)=\gamma^{\left[i\right]}(-x)\quad\text{ and }\quad\gamma^{\left[i\right]}:=\begin{cases}
\left|\det T_{i}\right|^{1/2}\cdot\left(\psi\circ T_{i}^{T}\right), & \text{if }i\in I_{0},\\
\varphi, & \text{if }i=0
\end{cases}
\]
forms a Banach frame for $\mathscr{S}_{\alpha,s}^{p,q}\left(\R^{2}\right)=\mathcal{D}\left(\CalS^{\left(\alpha\right)},L^{p},\ell_{w^{s}}^{q}\right)$.
Precisely, this means the following:

\begin{enumerate}[leftmargin=0.6cm]
\item The \textbf{analysis operator} 
\[
A^{(\delta)}:\mathscr{S}_{\alpha,s}^{p,q}\left(\R^{2}\right)\to C_{w^{s}}^{p,q},f\mapsto\left[\left(\smash{\gamma^{\left[i\right]}\ast f}\right)\left(\delta\cdot T_{i}^{-T}k\right)\right]_{i\in I,k\in\Z^{2}}
\]
is well-defined and bounded for arbitrary $\delta\in\left(0,1\right]$,
with the coefficient space $C_{w^{s}}^{p,q}$ from Definition \ref{def:CoefficientSpace}.
The convolution $\gamma^{\left[i\right]}\ast f$ has to be understood
as explained in equation \eqref{eq:SpecialConvolutionDefinition};
see Lemma \ref{lem:SpecialConvolutionClarification} for a more convenient
expression for this convolution, for $f\in L^{2}\left(\R^{2}\right)$.
\item For $0<\delta\leq\delta_{0}$, there is a bounded linear \textbf{reconstruction
operator}
\[
R^{(\delta)}:C_{w^{s}}^{p,q}\to\mathscr{S}_{\alpha,s}^{p,q}\left(\R^{2}\right)
\]
satisfying $R^{\left(\delta\right)}\circ A^{\left(\delta\right)}=\identity_{\mathscr{S}_{\alpha,s}^{p,q}\left(\R^{2}\right)}$.
\item For $0<\delta\leq\delta_{0}$, we have the following \textbf{consistency
statement}: If $f\in\mathscr{S}_{\alpha,s}^{p,q}\left(\R^{2}\right)$
and if $p_{0}\leq\tilde{p}\leq\infty$, $q_{0}\leq\tilde{q}\leq\infty$
and $s_{0}\leq\tilde{s}\leq s_{1}$, then the following equivalence
holds: 
\[
f\in\mathscr{S}_{\alpha,\tilde{s}}^{\tilde{p},\tilde{q}}\left(\R^{2}\right)\quad\Longleftrightarrow\quad\left[\left(\smash{\gamma^{\left[i\right]}\ast f}\right)\left(\delta\cdot T_{i}^{-T}k\right)\right]_{i\in I,k\in\Z^{2}}\in C_{w^{\tilde{s}}}^{\tilde{p},\tilde{q}}.\qedhere
\]
\end{enumerate}
\end{thm}
\begin{proof}
First, we show that there are constants $K_{1},K_{2}>0$ such that
\begin{equation}
\max_{\left|\beta\right|\leq1}\max_{\left|\theta\right|\leq N_{0}}\left|\left(\partial^{\theta}\smash{\widehat{\partial^{\beta}\smash{\psi}}}\right)\left(\xi\right)\right|\leq K_{1}\cdot\min\left\{ \left|\xi_{1}\right|^{M_{1}},\left(1+\left|\xi_{1}\right|\right)^{-M_{2}}\right\} \cdot\left(1+\left|\xi_{2}\right|\right)^{-K}=:K_{1}\cdot\varrho\left(\xi\right)\label{eq:NicelySimplifiedAlphaShearletFrameConditionTargetEstimate}
\end{equation}
and 
\begin{equation}
\max_{\left|\beta\right|\leq1}\max_{\left|\theta\right|\leq N_{0}}\left|\left(\partial^{\theta}\smash{\widehat{\partial^{\beta}\varphi}}\right)\left(\xi\right)\right|\leq K_{2}\cdot\left(1+\left|\xi_{2}\right|\right)^{-H}=:K_{2}\cdot\varrho_{0}\left(\xi\right)\label{eq:NiceSimplifiedAlphaShearletFrameConditionTargetEstimateLowPass}
\end{equation}
for all $\xi=\left(\xi_{1},\xi_{2}\right)\in\R^{2}$.

To this end, we recall that $\varphi,\psi\in C^{1}\left(\R^{2}\right)\cap W^{1,1}\left(\R^{2}\right)$,
so that standard properties of the Fourier transform show for $\beta=e_{\ell}$
(the $\ell$-th unit vector) that
\[
\widehat{\partial^{\beta}\psi}\left(\xi\right)=2\pi i\cdot\xi_{\ell}\cdot\widehat{\psi}\left(\xi\right)\text{ \ \ and \ \  }\widehat{\partial^{\beta}\varphi}\left(\xi\right)=2\pi i\cdot\xi_{\ell}\cdot\widehat{\varphi}\left(\xi\right)\qquad\forall\xi\in\R^{2}.
\]
Then, Leibniz's rule yields for $\beta=e_{\ell}$ and arbitrary $\theta\in\N_{0}^{2}$
with $\left|\theta\right|\leq N_{0}$ that
\begin{align}
\left|\left(\partial^{\theta}\smash{\widehat{\partial^{\beta}\psi}}\right)\left(\xi\right)\right| & =2\pi\cdot\left|\sum_{\nu\leq\theta}\binom{\theta}{\nu}\cdot\left(\partial^{\nu}\xi_{\ell}\right)\cdot\left(\partial^{\theta-\nu}\smash{\widehat{\psi}}\,\right)\left(\xi\right)\right|\nonumber \\
 & \leq2^{N_{0}+1}\pi\cdot\left(1+\left|\xi_{\ell}\right|\right)\cdot\max_{\left|\eta\right|\leq N_{0}}\left|\left(\partial^{\eta}\smash{\widehat{\psi}}\,\right)\left(\xi\right)\right|\label{eq:NicelySimplifiedShearletFrameConditionsDerivativeEstimate1}\\
 & \leq2^{N_{0}+1}\pi\cdot\left(1+\left|\xi_{\ell}\right|\right)\cdot C\cdot\left(1+\left|\xi_{1}\right|\right)^{-\left(1+M_{2}\right)}\left(1+\left|\xi_{2}\right|\right)^{-\left(1+K\right)}\nonumber \\
 & \leq2^{N_{0}+1}\pi C\cdot\left(1+\left|\xi_{1}\right|\right)^{-M_{2}}\cdot\left(1+\left|\xi_{2}\right|\right)^{-K},\label{eq:NicelySimplifiedShearletFrameConditionsDerivativeEstimate2}
\end{align}
since we have
\[
\left|\partial^{\nu}\xi_{\ell}\right|=\begin{cases}
\left|\xi_{\ell}\right|, & \text{if }\nu=0\\
1, & \text{if }\nu=e_{\ell}\\
0, & \text{otherwise}
\end{cases}\qquad\text{ and thus }\qquad\left|\partial^{\nu}\xi_{\ell}\right|\leq1+\left|\xi_{\ell}\right|\leq1+\left|\xi\right|.
\]
Above, we also used that $\sum_{\nu\leq\theta}\binom{\theta}{\nu}=\left(2,\dots,2\right)^{\theta}=2^{\left|\theta\right|}\leq2^{N_{0}}$,
as a consequence of the $\dimension$-dimensional binomial theorem
(cf.\@ \cite[Section 8.1, Exercise 2.b]{FollandRA}).

Likewise, we get 
\[
\begin{aligned}\left|\left(\partial^{\theta}\widehat{\partial^{\beta}\varphi}\right)\left(\xi\right)\right| & =2\pi\cdot\left|\sum_{\nu\leq\text{\ensuremath{\theta}}}\binom{\theta}{\nu}\cdot\left(\partial^{\nu}\xi_{\ell}\right)\cdot\left(\partial^{\theta-\nu}\widehat{\varphi}\right)\left(\xi\right)\right|\\
 & \leq2^{N_{0}+1}\pi\cdot\left(1+\left|\xi\right|\right)\cdot\max_{\left|\eta\right|\leq N_{0}}\left|\partial^{\eta}\widehat{\varphi}\left(\xi\right)\right|\\
 & \leq2^{N_{0}+1}\pi C\cdot\left(1+\left|\xi\right|\right)^{-H}\\
 & =2^{N_{0}+1}\pi C\cdot\varrho_{0}(\xi)
\end{aligned}
\]
and, by assumption,
\[
\left|\partial^{\theta}\widehat{\varphi}\left(\xi\right)\right|\leq C\cdot\left(1+\left|\xi\right|\right)^{-\left(H+1\right)}\leq C\cdot\left(1+\left|\xi\right|\right)^{-H}=C\cdot\varrho_{0}\left(\xi\right).
\]
With this, we have already established equation \eqref{eq:NiceSimplifiedAlphaShearletFrameConditionTargetEstimateLowPass}
with $K_{2}:=2^{N_{0}+1}\pi C$.

\medskip{}

To validate equation \eqref{eq:NicelySimplifiedAlphaShearletFrameConditionTargetEstimate},
we now distinguish the two cases $\left|\xi_{1}\right|>1$ and $\left|\xi_{1}\right|\leq1$:

\textbf{Case 1}: We have $\left|\xi_{1}\right|>1$. In this case,
$\varrho\left(\xi\right)=\left(1+\left|\xi_{1}\right|\right)^{-M_{2}}\left(1+\left|\xi_{2}\right|\right)^{-K}$,
so that equation \eqref{eq:NicelySimplifiedShearletFrameConditionsDerivativeEstimate2}
shows $\left|\left(\partial^{\theta}\widehat{\partial^{\beta}\psi}\right)\left(\xi\right)\right|\leq2^{N_{0}+1}\pi C\cdot\varrho\left(\xi\right)$
for $\beta=e_{\ell},$ $\ell\in\left\{ 1,2\right\} $ and arbitrary
$\theta\in\N_{0}^{2}$ with $\left|\theta\right|\leq N_{0}$. Finally,
we also have
\[
\left|\partial^{\theta}\widehat{\psi}\left(\xi\right)\right|\leq C\cdot\left(1+\left|\xi_{1}\right|\right)^{-\left(1+M_{2}\right)}\left(1+\left|\xi_{2}\right|\right)^{-\left(1+K\right)}\leq C\cdot\left(1+\left|\xi_{1}\right|\right)^{-M_{2}}\left(1+\left|\xi_{2}\right|\right)^{-K}=C\cdot\varrho\left(\xi\right)
\]
and hence $\max_{\left|\beta\right|\leq1}\max_{\left|\theta\right|\leq N_{0}}\left|\left(\partial^{\theta}\widehat{\partial^{\beta}\psi}\right)\left(\xi\right)\right|\leq2^{N_{0}+1}\pi C\cdot\varrho\left(\xi\right)$
for all $\xi\in\R^{2}$ with $\left|\xi_{1}\right|>1$.

\medskip{}

\textbf{Case 2}: We have $\left|\xi_{1}\right|\leq1$. First note
that this implies $\left(1+\left|\xi_{1}\right|\right)^{-M_{2}}\geq2^{-M_{2}}\geq2^{-M_{2}}\left|\xi_{1}\right|^{M_{1}}$
and consequently $\varrho\left(\xi\right)\geq2^{-M_{2}}\left|\xi_{1}\right|^{M_{1}}\cdot\left(1+\left|\xi_{2}\right|\right)^{-K}$.
Furthermore, we have for arbitrary $\ell\in\left\{ 1,2\right\} $
that
\[
1+\left|\xi_{\ell}\right|\leq\max\left\{ 1+\left|\xi_{1}\right|,\,1+\left|\xi_{2}\right|\right\} \leq\max\left\{ 2,\,1+\left|\xi_{2}\right|\right\} \leq2\cdot\left(1+\left|\xi_{2}\right|\right).
\]
In conjunction with equation \eqref{eq:NicelySimplifiedShearletFrameConditionsDerivativeEstimate1},
this shows for $\beta=e_{\ell}$, $\ell\in\left\{ 1,2\right\} $ and
$\theta\in\N_{0}^{2}$ with $\left|\theta\right|\leq N_{0}$ that
\begin{align*}
\left|\left(\partial^{\theta}\widehat{\partial^{\beta}\psi}\right)\left(\xi\right)\right| & \leq2^{N_{0}+1}\pi\cdot\left(1+\left|\xi_{\ell}\right|\right)\cdot\max_{\left|\eta\right|\leq N_{0}}\left|\partial^{\eta}\widehat{\psi}\left(\xi\right)\right|\\
 & \leq2^{N_{0}+2}\pi C\cdot\left(1+\left|\xi_{2}\right|\right)\cdot\left|\xi_{1}\right|^{M_{1}}\cdot\left(1+\left|\xi_{2}\right|\right)^{-\left(1+K\right)}\\
 & \leq2^{2+M_{2}+N_{0}}\pi C\cdot\varrho\left(\xi\right).
\end{align*}
Finally, we also have
\[
\left|\partial^{\theta}\widehat{\psi}\left(\xi\right)\right|\leq C\cdot\left|\xi_{1}\right|^{M_{1}}\left(1+\left|\xi_{2}\right|\right)^{-\left(1+K\right)}\leq C\cdot\left|\xi_{1}\right|^{M_{1}}\left(1+\left|\xi_{2}\right|\right)^{-K}\leq2^{M_{2}}C\cdot\varrho\left(\xi\right).
\]
All in all, we have shown $\max_{\left|\beta\right|\leq1}\max_{\left|\theta\right|\leq N_{0}}\left|\left(\partial^{\theta}\widehat{\partial^{\beta}\psi}\right)\left(\xi\right)\right|\leq2^{2+M_{2}+N_{0}}\pi C\cdot\varrho\left(\xi\right)$
for all $\xi\in\R^{2}$ with $\left|\xi_{1}\right|\leq1$.

\medskip{}

All together, we have thus established eq.\@ \eqref{eq:NicelySimplifiedAlphaShearletFrameConditionTargetEstimate}
with $K_{1}:=2^{2+M_{2}+N_{0}}\pi C$. Now, define $C_{\diamondsuit}:=\max\left\{ K_{1},K_{2}\right\} =K_{1}$.

Now, for proving the current theorem, we want to apply Theorem \ref{thm:BanachFrameTheorem}
with $\gamma_{1}^{\left(0\right)}:=\psi$, $\gamma_{2}^{\left(0\right)}:=\varphi$
and $k_{i}:=1$ for $i\in I_{0}$ and $k_{0}:=2$, as well as $Q_{0}^{\left(1\right)}:=Q=U_{\left(-1,1\right)}^{\left(3^{-1},3\right)}$
and $Q_{0}^{\left(2\right)}:=\left(-1,1\right)^{2}$, cf.\@ Assumption
\ref{assu:CrashCourseStandingAssumptions} and Definition \ref{def:AlphaShearletCovering}.
In the notation of Theorem \ref{thm:BanachFrameTheorem}, we then
have $\gamma_{i}=\gamma_{k_{i}}^{\left(0\right)}$ for all $i\in I$,
i.e., $\gamma_{i}=\psi$ for $i\in I_{0}$ and $\gamma_{0}=\varphi$.
Using this notation and setting furthermore $\varrho_{i}:=\varrho$
for $i\in I_{0}$, we have thus shown for arbitrary $N\in\N_{0}$
with $N\leq N_{0}$ that
\[
\begin{aligned}M_{j,i}: & =\left(\frac{w_{j}^{s}}{w_{i}^{s}}\right)^{\tau}\cdot\left(1+\left\Vert T_{j}^{-1}T_{i}\right\Vert \right)^{\sigma}\cdot\max_{\left|\beta\right|\leq1}\left(\left|\det T_{i}\right|^{-1}\cdot\int_{S_{i}^{\left(\alpha\right)}}\max_{\left|\theta\right|\leq N}\left|\left(\partial^{\theta}\widehat{\partial^{\beta}\gamma_{j}}\right)\left(T_{j}^{-1}\xi\right)\right|\d\xi\right)^{\tau}\\
 & \leq C_{\diamondsuit}^{\tau}\cdot\left(\frac{w_{j}^{s}}{w_{i}^{s}}\right)^{\tau}\cdot\left(1+\left\Vert T_{j}^{-1}T_{i}\right\Vert \right)^{\sigma}\cdot\left(\left|\det T_{i}\right|^{-1}\cdot\int_{S_{i}^{\left(\alpha\right)}}\varrho_{j}\left(T_{j}^{-1}\xi\right)\d\xi\right)^{\tau}=:C_{\diamondsuit}^{\tau}\cdot M_{j,i}^{\left(0\right)}
\end{aligned}
\]
for arbitrary $\sigma,\tau>0$, $s\in\R$ and the $\mathcal{S}^{\left(\alpha\right)}$-moderate
weight $w^{s}$ (cf. Lemma \ref{lem:AlphaShearletWeightIsModerate}). 

In view of the assumptions of the current theorem, the prerequisites
(1)-(3) of Theorem \ref{thm:BanachFrameTheorem} are clearly fulfilled,
but we still need to verify 
\[
C_{1}:=\sup_{i\in I}\:\sum_{j\in I}M_{j,i}<\infty\quad\text{ and }\quad C_{2}:=\sup_{j\in I}\:\sum_{i\in I}M_{j,i}<\infty,
\]
with $M_{j,i}$ as above, $\tau:=\min\left\{ 1,p,q\right\} \geq\min\left\{ p_{0},q_{0}\right\} =:\tau_{0}$,
and 
\begin{equation}
N:=\left\lceil \frac{2+\varepsilon}{\min\left\{ 1,p\right\} }\right\rceil \leq\left\lceil \frac{2+\varepsilon}{p_{0}}\right\rceil =N_{0},\quad\text{ as well as }\quad\sigma:=\tau\cdot\left(\frac{2}{\min\left\{ 1,p\right\} }+N\right)\leq\tau\cdot\left(\frac{2}{p_{0}}+N_{0}\right).\label{eq:NicelySimplifiedAlphaShearletFrameConditionSigmaDefinition}
\end{equation}
In particular, we have $\frac{\sigma}{\tau}\leq\frac{2}{p_{0}}+N_{0}=\frac{2}{p_{0}}+\left\lceil \frac{2+\varepsilon}{p_{0}}\right\rceil =:\omega$.

Hence, Lemma \ref{lem:MainShearletLemma} (with $c=\varepsilon$)
yields a constant $C_{0}=C_{0}\left(\alpha,\tau_{0},\omega,\varepsilon,K,H,M_{1},M_{2}\right)$
with $\max\left\{ C_{1},C_{2}\right\} \leq C_{\diamondsuit}^{\tau}C_{0}^{\tau}$,
provided that we can show $H\geq H_{0}+\varepsilon$, $K\geq K_{0}+\varepsilon$
and $M_{\ell}\geq M_{\ell}^{\left(0\right)}+\varepsilon$ for $\ell\in\left\{ 1,2\right\} $,
with $H_{0},K_{0},M_{1}^{(0)},M_{2}^{(0)}$ as defined in Lemma \ref{lem:MainShearletLemma}.
But we have
\[
\begin{aligned}H_{0} & =\frac{1-\alpha}{\tau}+\frac{\sigma}{\tau}-s\leq\frac{1-\alpha}{\tau_{0}}+\omega-s_{0}\\
 & =\frac{1-\alpha}{\min\left\{ p_{0},q_{0}\right\} }+\frac{2}{p_{0}}+N_{0}-s_{0}\\
 & \leq H-\varepsilon.
\end{aligned}
\]
Furthermore, 
\[
\begin{aligned}M_{2}^{(0)} & =(1+\alpha)\frac{\sigma}{\tau}-s\leq(1+\alpha)\omega-s_{0}\\
 & =\left(1+\alpha\right)\left(\frac{2}{p_{0}}+N_{0}\right)-s_{0}\\
 & \leq M_{2}-\varepsilon
\end{aligned}
\]
and
\[
M_{1}^{(\text{0})}\leq\frac{1}{\tau}+\max\left\{ s,0\right\} \leq\frac{1}{\min\left\{ p_{0},q_{0}\right\} }+\max\left\{ s_{1},0\right\} =M_{1}-\varepsilon,
\]
as well as 
\[
\begin{aligned}K_{0} & \leq\max\left\{ \frac{1-\alpha}{\tau}+2\frac{\sigma}{\tau}-s,\,\frac{2+\sigma}{\tau}\right\} \\
 & \leq\max\left\{ \frac{1-\alpha}{\tau_{0}}+2\omega-s_{0},\frac{2}{\tau_{0}}+\omega\right\} \\
 & =\max\left\{ \frac{1-\alpha}{\min\left\{ p_{0},q_{0}\right\} }+2\left(\frac{2}{p_{0}}+N_{0}\right)-s_{0},\frac{2}{\min\left\{ p_{0},q_{0}\right\} }+\frac{2}{p_{0}}+N_{0}\right\} \\
 & =K-\varepsilon.
\end{aligned}
\]

Thus, Lemma \ref{lem:MainShearletLemma} is applicable, so that 
\[
C_{1}^{1/\tau}=\left(\sup_{i\in I}\:\smash{\sum_{j\in I}}M_{j,i}\right)^{1/\tau}\leq C_{\diamondsuit}C_{0},
\]
where the right-hand side is independent of $p,q$ and $s$, since
$C_{0}$ is independent of $p,q$ and $s$ and since 
\[
C_{\diamondsuit}=C_{\diamondsuit}\left(\varepsilon,p_{0},M_{2},C\right)=2^{2+M_{2}+N_{0}}\pi C=2^{2+M_{2}+\left\lceil \frac{2+\varepsilon}{p_{0}}\right\rceil }\pi C.
\]
The exact same estimate holds for $C_{2}$.

We have shown that all prerequisites for Theorem \ref{thm:BanachFrameTheorem}
are fulfilled. Hence, the theorem implies that there is a constant
$K_{\diamondsuit}=K_{\diamondsuit}\left(p_{0},q_{0},\varepsilon,\mathcal{S}^{(\alpha)},\varphi,\psi\right)>0$
(independent of $p,q,s$) such that the family $\widetilde{{\rm SH}}_{\alpha,\varphi,\psi,\delta}^{\left(\pm1\right)}$
forms a Banach frame for $\mathscr{S}_{\alpha,s}^{p,q}\left(\R^{2}\right)$,
as soon as $\delta\leq\delta_{00}$, where 
\[
\delta_{00}:=\left(1+K_{\diamondsuit}\cdot C_{\mathcal{S}^{\left(\alpha\right)},w^{s}}^{4}\cdot\left(C_{1}^{1/\tau}+C_{2}^{1/\tau}\right)^{2}\right)^{-1}.
\]
From Lemma \ref{lem:AlphaShearletWeightIsModerate} we know that $C_{\mathcal{S}^{\left(\alpha\right)},w^{s}}\leq39^{\left|s\right|}\leq39^{s_{2}}$
where $s_{2}:=\max\left\{ \left|s_{0}\right|,\left|s_{1}\right|\right\} $.
Hence, choosing 
\[
\delta_{0}:=\left(1+4\cdot K_{\diamondsuit}\cdot C_{\diamondsuit}^{2}\cdot C_{0}^{2}\cdot39^{4s_{2}}\right)^{-1},
\]
we get $\delta_{0}\leq\delta_{00}$ and $\delta_{0}$ is independent
of the precise choice of $p,q,s$, as long as $p\geq p_{\text{0}},$
$q\geq q_{0}$ and $s_{0}\leq s\leq s_{1}$. Thus, for $0<\delta\leq\delta_{0}$
and arbitrary $p,q\in\left(0,\infty\right]$, $s\in\R$ with $p\geq p_{0}$,
$q\geq q_{0}$ and $s_{0}\leq s\leq s_{1}$, the family $\widetilde{{\rm SH}}_{\alpha,\varphi,\psi,\delta}^{\left(\pm1\right)}$
forms a Banach frame for $\mathscr{S}_{\alpha,s}^{p,q}\left(\R^{2}\right)$.
\end{proof}
Finally, we also come to verifiable sufficient conditions which ensure
that the low-pass $\varphi$ and the mother shearlet $\psi$ generate
atomic decompositions for $\mathscr{S}_{\alpha,s}^{p,q}\left(\R^{2}\right)$.
\begin{thm}
\label{thm:ReallyNiceShearletAtomicDecompositionConditions}Let $\alpha\in\left[0,1\right]$,
$\varepsilon,p_{0},q_{0}\in\left(0,1\right]$ and $s_{0},s_{1}\in\R$
with $s_{0}\leq s_{1}$. Assume that $\varphi,\psi\in L^{1}\left(\R^{2}\right)$
satisfy the following properties:

\begin{itemize}[leftmargin=0.6cm]
\item We have $\left\Vert \varphi\right\Vert _{1+\frac{2}{p_{0}}}<\infty$
and $\left\Vert \psi\right\Vert _{1+\frac{2}{p_{0}}}<\infty$, where
$\left\Vert g\right\Vert _{\Lambda}=\sup_{x\in\R^{2}}\left(1+\left|x\right|\right)^{\Lambda}\left|g\left(x\right)\right|$
for $g:\R^{2}\to\Compl^{\ell}$ (with arbitrary $\ell\in\N$) and
$\Lambda\geq0$.
\item We have $\widehat{\varphi},\widehat{\psi}\in C^{\infty}\left(\R^{2}\right)$,
where all partial derivatives of $\widehat{\varphi},\widehat{\psi}$
are polynomially bounded.
\item We have
\begin{align*}
\widehat{\psi}\left(\xi\right)\neq0 & \text{ for all }\xi=\left(\xi_{1},\xi_{2}\right)\in\R^{2}\text{ with }\xi_{1}\in\left[3^{-1},3\right]\text{ and }\left|\xi_{2}\right|\leq\left|\xi_{1}\right|,\\
\widehat{\varphi}\left(\xi\right)\ne0 & \text{ for all }\xi\in\left[-1,1\right]^{2}.
\end{align*}
\item We have
\begin{equation}
\begin{split}\left|\partial^{\beta}\widehat{\varphi}\left(\xi\right)\right| & \lesssim\left(1+\left|\xi\right|\right)^{-\Lambda_{0}},\\
\left|\partial^{\beta}\smash{\widehat{\psi}}\left(\xi\right)\right| & \lesssim\min\left\{ \left|\xi_{1}\right|^{\Lambda_{1}},\left(1+\left|\xi_{1}\right|\right)^{-\Lambda_{2}}\right\} \cdot\left(1+\left|\xi_{2}\right|\right)^{-\Lambda_{3}}\cdot\left(1+\left|\xi\right|\right)^{-\left(3+\varepsilon\right)}
\end{split}
\label{eq:ShearletAtomicDecompositionFourierDecayCondition}
\end{equation}
for all $\xi=\left(\xi_{1},\xi_{2}\right)\in\R^{2}$ and all $\beta\in\N_{0}^{2}$
with $\left|\beta\right|\leq\left\lceil p_{0}^{-1}\cdot\left(2+\varepsilon\right)\right\rceil $,
where
\begin{align*}
\qquad\qquad\Lambda_{0} & :=\begin{cases}
3+2\varepsilon+\max\left\{ \frac{1-\alpha}{\min\left\{ p_{0},q_{0}\right\} }+3+s_{1},\,2\right\} , & \text{if }p_{0}=1,\\
3+2\varepsilon+\max\left\{ \frac{1-\alpha}{\min\left\{ p_{0},q_{0}\right\} }+\frac{1-\alpha}{p_{0}}+1+\alpha+\left\lceil \frac{2+\varepsilon}{p_{0}}\right\rceil +s_{1},\,2\right\} , & \text{if }p_{0}\in\left(0,1\right),
\end{cases}\\
\qquad\qquad\Lambda_{1} & :=\varepsilon+\frac{1}{\min\left\{ p_{0},q_{0}\right\} }+\max\left\{ 0,\,\left(1+\alpha\right)\left(\frac{1}{p_{0}}-1\right)-s_{0}\right\} ,\\
\qquad\qquad\Lambda_{2} & :=\begin{cases}
\varepsilon+\max\left\{ 2,\,3\left(1+\alpha\right)+s_{1}\right\} , & \text{if }p_{0}=1,\\
\varepsilon+\max\left\{ 2,\,\left(1+\alpha\right)\left(1+\frac{1}{p_{0}}+\left\lceil \frac{2+\varepsilon}{p_{0}}\right\rceil \right)+s_{1}\right\} , & \text{if }p_{0}\in\left(0,1\right),
\end{cases}\\
\qquad\qquad\Lambda_{3} & :=\begin{cases}
\varepsilon+\max\left\{ \frac{1-\alpha}{\min\left\{ p_{0},q_{0}\right\} }+6+s_{1},\,\frac{2}{\min\left\{ p_{0},q_{0}\right\} }+3\right\} , & \text{if }p_{0}=1,\\
\varepsilon+\max\left\{ \frac{1-\alpha}{\min\left\{ p_{0},q_{0}\right\} }+\frac{3-\alpha}{p_{0}}+2\left\lceil \frac{2+\varepsilon}{p_{0}}\right\rceil +1+\alpha+s_{1},\,\frac{2}{\min\left\{ p_{0},q_{0}\right\} }+\frac{2}{p_{0}}+\left\lceil \frac{2+\varepsilon}{p_{0}}\right\rceil \right\} , & \text{if }p_{0}\in\left(0,1\right).
\end{cases}
\end{align*}
\end{itemize}
Then there is some $\delta_{0}\in\left(0,1\right]$ such that for
all $0<\delta\leq\delta_{0}$ and all $p,q\in\left(0,\infty\right]$
and $s\in\R$ with $p\geq p_{0}$, $q\geq q_{0}$ and $s_{0}\leq s\leq s_{1}$,
the following is true: The family 
\[
{\rm SH}_{\alpha,\varphi,\psi,\delta}^{\left(\pm1\right)}:=\left(L_{\delta\cdot T_{i}^{-T}k}\gamma^{\left[i\right]}\right)_{i\in I,\,k\in\Z^{2}}\quad\text{ with }\quad\gamma^{\left[i\right]}:=\begin{cases}
\left|\det T_{i}\right|^{1/2}\cdot\left(\psi\circ T_{i}^{T}\right), & \text{if }i\in I_{0},\\
\varphi, & \text{if }i=0
\end{cases}
\]
forms an atomic decomposition for $\mathscr{S}_{\alpha,s}^{p,q}\left(\R^{2}\right)$.
Precisely, this means the following:

\begin{enumerate}[leftmargin=0.6cm]
\item The \textbf{synthesis map}
\[
S^{\left(\delta\right)}:C_{w^{s}}^{p,q}\to\mathscr{S}_{\alpha,s}^{p,q}\left(\R^{2}\right),\left(\smash{c_{k}^{\left(i\right)}}\right)_{i\in I,k\in\Z^{2}}\mapsto\sum_{i\in I}\:\sum_{k\in\Z^{2}}\left(c_{k}^{\left(i\right)}\cdot L_{\delta\cdot T_{i}^{-T}k}\gamma^{\left[i\right]}\right)
\]
is well-defined and bounded for all $\delta\in\left(0,1\right]$,
where the \emph{coefficient space} $C_{w^{s}}^{p,q}$ is as in Definition
\ref{def:CoefficientSpace}. Convergence of the series has to be understood
as described in the remark to Theorem \ref{thm:AtomicDecompositionTheorem}.
\item For $0<\delta\leq\delta_{0}$, there is a bounded linear \textbf{coefficient
map}
\[
C^{\left(\delta\right)}:\mathscr{S}_{\alpha,s}^{p,q}\left(\R^{2}\right)\to C_{w^{s}}^{p,q}
\]
satisfying $S^{(\delta)}\circ C^{\left(\delta\right)}=\identity_{\mathscr{S}_{\alpha,s}^{p,q}\left(\R^{2}\right)}$.

Furthermore, the action of $C^{\left(\delta\right)}$ is \emph{independent}
of the precise choice of $p,q,s$. Precisely, if $p_{1},p_{2}\geq p_{0}$,
$q_{1},q_{2}\geq q_{0}$ and $s^{\left(1\right)},s^{\left(2\right)}\in\left[s_{0},s_{1}\right]$
and if $f\in\mathscr{S}_{\alpha,s^{\left(1\right)}}^{p_{1},q_{1}}\cap\mathscr{S}_{\alpha,s^{\left(2\right)}}^{p_{2},q_{2}}$,
then $C_{1}^{\left(\delta\right)}f=C_{2}^{\left(\delta\right)}f$,
where $C_{i}^{\left(\delta\right)}$ denotes the coefficient operator
for the choices $p=p_{i}$, $q=q_{i}$ and $s=s^{\left(i\right)}$
for $i\in\left\{ 1,2\right\} $.\qedhere

\end{enumerate}
\end{thm}
\begin{proof}
Later in the proof, we will apply Theorem \ref{thm:AtomicDecompositionTheorem}
to the decomposition space $\mathscr{S}_{\alpha,s}^{p,q}\left(\smash{\R^{2}}\right)=\DecompSp{\CalS^{\left(\alpha\right)}}p{\ell_{w^{s}}^{q}}{}$
with $w$ and $w^{s}$ as in Lemma \ref{lem:AlphaShearletWeightIsModerate},
while Theorem \ref{thm:AtomicDecompositionTheorem} itself considers
the decomposition space $\DecompSp{\CalQ}p{\ell_{w}^{q}}{}$. To avoid
confusion between these two different choices of the weight $w$,
we will write $v$ for the weight defined in Lemma \ref{lem:AlphaShearletWeightIsModerate},
so that we get $\mathscr{S}_{\alpha,s}^{p,q}\left(\smash{\R^{2}}\right)=\DecompSp{\CalS^{\left(\alpha\right)}}p{\ell_{v^{s}}^{q}}{}$.
For the application of Theorem \ref{thm:AtomicDecompositionTheorem},
we will thus choose $\CalQ=\CalS^{\left(\alpha\right)}$ and $w=v^{s}$.

Our assumptions on $\varphi$ show that there is a constant $C_{1}>0$
satisfying $\left|\partial^{\beta}\widehat{\varphi}\left(\xi\right)\right|\leq C_{1}\cdot\left(1+\left|\xi\right|\right)^{-\Lambda_{0}}$
for all $\beta\in\N_{0}^{2}$ with $\left|\beta\right|\leq N_{0}:=\left\lceil p_{0}^{-1}\cdot\left(2+\varepsilon\right)\right\rceil $.
We first apply Proposition \ref{prop:ConvolutionFactorization} (with
$N=N_{0}\geq\left\lceil 2+\varepsilon\right\rceil =3=\dimension+1$,
with $\gamma=\varphi$ and with $\varrho=\varrho_{1}$ for $\varrho_{1}\left(\xi\right):=C_{1}\cdot\left(1+\left|\xi\right|\right)^{3+\varepsilon-\Lambda_{0}}$,
where we note $\Lambda_{0}-3-\varepsilon\geq2+\varepsilon$, so that
$\varrho_{1}\in L^{1}\left(\R^{2}\right)$). We indeed have $\left|\partial^{\beta}\widehat{\varphi}\left(\xi\right)\right|\leq C_{1}\cdot\left(1+\left|\xi\right|\right)^{-\Lambda_{0}}=\varrho_{1}\left(\xi\right)\cdot\left(1+\left|\xi\right|\right)^{-\left(\dimension+1+\varepsilon\right)}$
for all $\left|\beta\right|\leq N_{0}$, since we are working in $\R^{\dimension}=\R^{2}$.
Consequently, Proposition \ref{prop:ConvolutionFactorization} provides
functions $\varphi_{1}\in C_{0}\left(\R^{2}\right)\cap L^{1}\left(\R^{2}\right)$
and $\varphi_{2}\in C^{1}\left(\R^{2}\right)\cap W^{1,1}\left(\R^{2}\right)$
with $\varphi=\varphi_{1}\ast\varphi_{2}$ and with the following
additional properties:

\begin{enumerate}
\item We have $\left\Vert \varphi_{2}\right\Vert _{\Lambda}<\infty$ and
$\left\Vert \nabla\varphi_{2}\right\Vert _{\Lambda}<\infty$ for all
$\Lambda\in\N_{0}$.
\item We have $\widehat{\varphi_{2}}\in C^{\infty}\left(\R^{2}\right)$,
where all partial derivatives of $\widehat{\varphi_{2}}$ are polynomially
bounded.
\item We have $\widehat{\varphi_{1}}\in C^{\infty}\left(\R^{2}\right)$,
where all partial derivatives of $\widehat{\varphi_{1}}$ are polynomially
bounded. This uses that $\widehat{\varphi}\in C^{\infty}\left(\R^{2}\right)$
with all partial derivatives being polynomially bounded.
\item We have 
\begin{equation}
\left|\partial^{\beta}\widehat{\varphi_{1}}\left(\xi\right)\right|\leq\frac{C_{2}}{C_{1}}\cdot\varrho_{1}\left(\xi\right)=C_{2}\cdot\left(1+\left|\xi\right|\right)^{3+\varepsilon-\Lambda_{0}}\quad\forall\xi\in\R^{2}\text{ and }\beta\in\N_{0}^{2}\text{ with }\left|\beta\right|\leq N_{0}.\label{eq:AlphaShearletAtomicDecompositionPhiFactorizationEstimate}
\end{equation}
Here, $C_{2}$ is given by $C_{2}:=C_{1}\cdot2^{3+4N_{0}}\cdot N_{0}!\cdot3^{N_{0}}$.
\end{enumerate}
Likewise, our assumptions on $\psi$ show that there is a constant
$C_{3}>0$ satisfying 
\[
\left|\partial^{\beta}\widehat{\psi}\left(\xi\right)\right|\leq C_{3}\cdot\min\left\{ \left|\xi_{1}\right|^{\Lambda_{1}},\left(1+\left|\xi_{1}\right|\right)^{-\Lambda_{2}}\right\} \cdot\left(1+\left|\xi_{2}\right|\right)^{-\Lambda_{3}}\cdot\left(1+\left|\xi\right|\right)^{-\left(3+\varepsilon\right)}\quad\forall\xi\in\R^{2}\:\forall\beta\in\N_{0}^{2}\text{ with }\left|\beta\right|\leq N_{0}.
\]
Now, we again apply Proposition \ref{prop:ConvolutionFactorization},
but this time with $N=N_{0}\geq\dimension+1$, with $\gamma=\psi$
and with $\varrho=\varrho_{2}$ for $\varrho_{2}\left(\xi\right):=C_{3}\cdot\min\left\{ \left|\xi_{1}\right|^{\Lambda_{1}},\left(1+\left|\xi_{1}\right|\right)^{-\Lambda_{2}}\right\} \cdot\left(1+\left|\xi_{2}\right|\right)^{-\Lambda_{3}}$,
where we note that $\Lambda_{2}\geq2+\varepsilon$ and $\Lambda_{3}\geq3\geq2+\varepsilon$,
so that 
\begin{align*}
\varrho_{2}\left(\xi\right) & \leq C_{3}\cdot\left(1+\left|\xi_{1}\right|\right)^{-\left(2+\varepsilon\right)}\cdot\left(1+\left|\xi_{2}\right|\right)^{-\left(2+\varepsilon\right)}\\
 & \leq C_{3}\cdot\left[\max\left\{ 1+\left|\xi_{1}\right|,\,1+\left|\xi_{2}\right|\right\} \right]^{-\left(2+\varepsilon\right)}\\
 & \leq C_{3}\cdot\left(1+\left\Vert \xi\right\Vert _{\infty}\right)^{-\left(2+\varepsilon\right)}\in L^{1}\left(\smash{\R^{2}}\right).
\end{align*}
As we just saw, we indeed have $\left|\partial^{\beta}\widehat{\psi}\left(\xi\right)\right|\leq\varrho_{2}\left(\xi\right)\cdot\left(1+\left|\xi\right|\right)^{-\left(\dimension+1+\varepsilon\right)}$
for all $\left|\beta\right|\leq N_{0}$, since we are working in $\R^{\dimension}=\R^{2}$.
Consequently, Proposition \ref{prop:ConvolutionFactorization} provides
functions $\psi_{1}\in C_{0}\left(\R^{2}\right)\cap L^{1}\left(\R^{2}\right)$
and $\psi_{2}\in C^{1}\left(\R^{2}\right)\cap W^{1,1}\left(\R^{2}\right)$
with $\psi=\psi_{1}\ast\psi_{2}$ and with the following additional
properties:

\begin{enumerate}
\item We have $\left\Vert \psi_{2}\right\Vert _{\Lambda}<\infty$ and $\left\Vert \nabla\psi_{2}\right\Vert _{\Lambda}<\infty$
for all $\Lambda\in\N_{0}$.
\item We have $\widehat{\psi_{2}}\in C^{\infty}\left(\R^{2}\right)$, where
all partial derivatives of $\widehat{\psi_{2}}$ are polynomially
bounded.
\item We have $\widehat{\psi_{1}}\in C^{\infty}\left(\R^{2}\right)$, where
all partial derivatives of $\widehat{\psi_{1}}$ are polynomially
bounded. This uses that $\widehat{\psi}\in C^{\infty}\left(\R^{2}\right)$
with all partial derivatives being polynomially bounded.
\item We have
\begin{equation}
\begin{split}\quad\qquad\left|\partial^{\beta}\,\smash{\widehat{\psi_{1}}}\,\left(\xi\right)\right| & \!\leq\!\frac{C_{4}}{C_{3}}\cdot\varrho_{2}\left(\xi\right)\\
 & \!=\!C_{4}\cdot\min\left\{ \left|\xi_{1}\right|^{\Lambda_{1}},\left(1+\left|\xi_{1}\right|\right)^{-\Lambda_{2}}\right\} \cdot\left(1+\left|\xi_{2}\right|\right)^{-\Lambda_{3}}\quad\forall\xi\in\R^{2}\text{ and }\beta\in\N_{0}^{2}\text{ with }\left|\beta\right|\leq N_{0}.
\end{split}
\label{eq:AlphaShearletAtomicDecompositionPsiFactorizationEstimate}
\end{equation}
 Here, $C_{4}$ is given by $C_{4}:=C_{3}\cdot2^{3+4N_{0}}\cdot N_{0}!\cdot3^{N_{0}}$.
\end{enumerate}
In summary, if we define $M_{1}:=\Lambda_{1}$, $M_{2}:=\Lambda_{2}$
and $K:=\Lambda_{3}$, as well as $H:=\Lambda_{0}-3-\varepsilon$,
then we have $M_{1},M_{2},K,H\geq0$ and
\begin{equation}
\begin{split}\max_{\left|\beta\right|\leq N_{0}}\left|\partial^{\beta}\widehat{\psi_{1}}\left(\xi\right)\right| & \leq C_{5}\cdot\min\left\{ \left|\xi_{1}\right|^{M_{1}},\,\left(1+\left|\xi_{1}\right|\right)^{-M_{2}}\right\} \cdot\left(1+\left|\xi_{2}\right|\right)^{-K}=:C_{5}\cdot\varrho\left(\xi\right),\\
\max_{\left|\beta\right|\leq N_{0}}\left|\partial^{\beta}\widehat{\varphi_{1}}\left(\xi\right)\right| & \leq C_{5}\cdot\left(1+\left|\xi\right|\right)^{-H}=:C_{5}\cdot\varrho_{0}\left(\xi\right),
\end{split}
\label{eq:AlphaShearletAtomicDecompositionGeneratorsMainEstimate}
\end{equation}
where we defined $C_{5}:=\max\left\{ C_{2},C_{4}\right\} $ for brevity.
For consistency with Lemma \ref{lem:MainShearletLemma}, we define
$\varrho_{j}:=\varrho$ for arbitrary $j\in I_{0}$.

Now, define $n:=2$, $\gamma_{1}^{\left(0\right)}:=\psi$ and $\gamma_{2}^{\left(0\right)}:=\varphi$,
as well as $\gamma_{1}^{\left(0,j\right)}:=\psi_{j}$ and $\gamma_{2}^{\left(0,j\right)}:=\varphi_{j}$
for $j\in\left\{ 1,2\right\} $. We want to verify the assumptions
of Theorem \ref{thm:AtomicDecompositionTheorem} for these choices
and for $\mathscr{S}_{\alpha,s}^{p,q}\left(\R^{2}\right)=\DecompSp{\CalS^{\left(\alpha\right)}}p{\ell_{v^{s}}^{q}}{}=\DecompSp{\CalQ}p{\ell_{w}^{q}}{}$.
To this end, we recall from Definition \ref{def:AlphaShearletCovering}
that $\CalQ:=\CalS^{\left(\alpha\right)}=\left(T_{i}Q_{i}'+b_{i}\right)_{i\in I}$,
with $Q_{i}'=U_{\left(-1,1\right)}^{\left(3^{-1},3\right)}=Q=:Q_{0}^{\left(1\right)}=Q_{0}^{\left(k_{i}\right)}$
for all $i\in I_{0}$, where $k_{i}:=1$ for $i\in I_{0}$ and with
$Q_{0}'=\left(-1,1\right)^{2}=:Q_{0}^{\left(2\right)}=Q_{0}^{\left(k_{0}\right)}$,
where $k_{0}:=2$ and $n:=2$, cf.\@ Assumption \ref{assu:CrashCourseStandingAssumptions}.

Now, let us verify the list of prerequisites of Theorem \ref{thm:AtomicDecompositionTheorem}:

\begin{enumerate}
\item We have $\gamma_{k}^{\left(0,1\right)}\in\left\{ \varphi_{1},\psi_{1}\right\} \subset L^{1}\left(\R^{2}\right)$
for $k\in\left\{ 1,2\right\} $ by the properties of $\varphi_{1},\psi_{1}$
from above.
\item Likewise, we have $\gamma_{k}^{\left(0,2\right)}\in\left\{ \varphi_{2},\psi_{2}\right\} \subset C^{1}\left(\R^{2}\right)$
by the properties of $\varphi_{2},\psi_{2}$ from above.
\item Next, with $\varUpsilon=1+\frac{\dimension}{\min\left\{ 1,p\right\} }$
as in Theorem \ref{thm:AtomicDecompositionTheorem}, we have $\varUpsilon\leq1+\frac{2}{p_{0}}=:\varUpsilon_{0}$
and thus, with $\Omega^{\left(p\right)}$ as in Theorem \ref{thm:AtomicDecompositionTheorem},
\begin{equation}
\begin{split}\Omega^{\left(p\right)} & =\max_{k\in\underline{n}}\left\Vert \gamma_{k}^{\left(0,2\right)}\right\Vert _{\varUpsilon}+\max_{k\in\underline{n}}\left\Vert \nabla\gamma_{k}^{\left(0,2\right)}\right\Vert _{\varUpsilon}\\
 & =\max\left\{ \left\Vert \varphi_{2}\right\Vert _{\varUpsilon},\left\Vert \psi_{2}\right\Vert _{\varUpsilon}\right\} +\max\left\{ \left\Vert \nabla\varphi_{2}\right\Vert _{\varUpsilon},\left\Vert \nabla\psi_{2}\right\Vert _{\varUpsilon}\right\} \\
 & \leq\max\left\{ \left\Vert \varphi_{2}\right\Vert _{\left\lceil \varUpsilon_{0}\right\rceil },\left\Vert \psi_{2}\right\Vert _{\left\lceil \varUpsilon_{0}\right\rceil }\right\} +\max\left\{ \left\Vert \nabla\varphi_{2}\right\Vert _{\left\lceil \varUpsilon_{0}\right\rceil },\left\Vert \nabla\psi_{2}\right\Vert _{\left\lceil \varUpsilon_{0}\right\rceil }\right\} =:C_{6}<\infty
\end{split}
\label{eq:AtomicDecompositionSecondConvolutionFactorEstimate}
\end{equation}
 by the properties of $\varphi_{2},\psi_{2}$ from above.
\item We have $\Fourier\gamma_{k}^{\left(0,j\right)}\in\left\{ \widehat{\varphi_{1}},\widehat{\psi_{1}},\widehat{\varphi_{2}},\widehat{\psi_{2}}\right\} \subset C^{\infty}\left(\R^{2}\right)$
and all partial derivatives of these functions are polynomially bounded.
\item We have $\gamma_{1}^{\left(0\right)}=\psi=\psi_{1}\ast\psi_{2}=\gamma_{1}^{\left(0,1\right)}\ast\gamma_{1}^{\left(0,2\right)}$
and $\gamma_{2}^{\left(0\right)}=\varphi=\varphi_{1}\ast\varphi_{2}=\gamma_{2}^{\left(0,1\right)}\ast\gamma_{2}^{\left(0,2\right)}$.
\item By assumption, we have $\Fourier\gamma_{1}^{\left(0\right)}\left(\xi\right)=\widehat{\psi}\left(\xi\right)\neq0$
for all $\xi\in\overline{Q}=\overline{Q_{0}^{\left(1\right)}}$. Likewise,
we have $\Fourier\gamma_{2}^{\left(0\right)}\left(\xi\right)=\widehat{\varphi}\left(\xi\right)\neq0$
for all $\xi\in\left[-1,1\right]^{2}=\overline{\left(-1,1\right)^{2}}=\overline{Q_{0}^{\left(2\right)}}$.
\item We have $\left\Vert \smash{\gamma_{1}^{\left(0\right)}}\right\Vert _{\varUpsilon}=\left\Vert \psi\right\Vert _{\varUpsilon}\leq\left\Vert \psi\right\Vert _{\varUpsilon_{0}}=\left\Vert \psi\right\Vert _{1+\frac{2}{p_{0}}}<\infty$
and $\left\Vert \smash{\gamma_{2}^{\left(0\right)}}\right\Vert _{\varUpsilon}=\left\Vert \varphi\right\Vert _{\varUpsilon}\leq\left\Vert \varphi\right\Vert _{1+\frac{2}{p_{0}}}<\infty$,
thanks to our assumptions on $\varphi,\psi$.
\end{enumerate}
Thus, as the last prerequisite of Theorem \ref{thm:AtomicDecompositionTheorem},
we have to verify
\[
K_{1}:=\sup_{i\in I}\:\sum_{j\in I}N_{i,j}<\infty\qquad\text{ and }\qquad K_{2}:=\sup_{j\in I}\:\sum_{i\in I}N_{i,j}<\infty,
\]
where $\gamma_{j,1}:=\gamma_{k_{j}}^{\left(0,1\right)}$ for $j\in I$
(i.e., $\gamma_{0,1}=\gamma_{2}^{\left(0,1\right)}=\varphi_{1}$ and
$\gamma_{j,1}=\gamma_{1}^{\left(0,1\right)}=\psi_{1}$ for $j\in I_{0}$)
and
\begin{align*}
N_{i,j} & :=\left(\frac{w_{i}}{w_{j}}\cdot\left[\frac{\left|\det T_{j}\right|}{\left|\det T_{i}\right|}\right]^{\vartheta}\right)^{\tau}\!\!\cdot\left(1\!+\!\left\Vert T_{j}^{-1}T_{i}\right\Vert \right)^{\sigma}\!\cdot\left(\left|\det T_{i}\right|^{-1}\!\cdot\int_{Q_{i}}\:\max_{\left|\beta\right|\leq N}\left|\left[\partial^{\beta}\widehat{\gamma_{j,1}}\right]\left(T_{j}^{-1}\left(\xi\!-\!b_{j}\right)\right)\right|\d\xi\right)^{\tau}\\
\left({\scriptstyle \text{since }b_{j}=0\text{ for all }j\in I}\right) & \overset{\left(\ast\right)}{\leq}\left(\frac{v_{j}^{\left(1+\alpha\right)\vartheta-s}}{v_{i}^{\left(1+\alpha\right)\vartheta-s}}\right)^{\tau}\cdot\left(1\!+\!\left\Vert T_{j}^{-1}T_{i}\right\Vert \right)^{\sigma}\!\cdot\left(\left|\det T_{i}\right|^{-1}\!\cdot\int_{S_{i}^{\left(\alpha\right)}}\:\max_{\left|\beta\right|\leq N}\left|\left[\partial^{\beta}\widehat{\gamma_{j,1}}\right]\left(T_{j}^{-1}\xi\right)\right|\d\xi\right)^{\tau}\\
\left({\scriptstyle \text{eq. }\eqref{eq:AlphaShearletAtomicDecompositionGeneratorsMainEstimate}\text{ and }N\leq N_{0}}\right) & \leq C_{5}^{\tau}\cdot M_{j,i}^{\left(0\right)},
\end{align*}
where the quantity $M_{j,i}^{\left(0\right)}$ is defined as in Lemma
\ref{lem:MainShearletLemma}, but with $s^{\natural}:=\left(1+\alpha\right)\vartheta-s$
instead of $s$. At the step marked with $\left(\ast\right)$, we
used that we have $w=v^{s}$ and $\left|\det T_{i}\right|=v_{i}^{1+\alpha}$
for all $i\in I$.

To be precise, we recall from Theorem \ref{thm:AtomicDecompositionTheorem}
that the quantities $N,\tau,\sigma,\vartheta$ from above are given
(because of $\dimension=2$) by $\vartheta=\left(p^{-1}-1\right)_{+}$,
\[
\tau=\min\left\{ 1,p,q\right\} \geq\min\left\{ p_{0},q_{0}\right\} =:\tau_{0}\qquad\text{ and }\qquad N=\left\lceil \left(\dimension+\varepsilon\right)\big/\min\left\{ 1,p\right\} \right\rceil \leq\left\lceil p_{0}^{-1}\cdot\left(2+\varepsilon\right)\right\rceil =N_{0},
\]
as well as
\[
\sigma=\begin{cases}
\tau\cdot\left(\dimension+1\right)=3\cdot\tau, & \text{if }p\in\left[1,\infty\right],\\
\tau\cdot\left(\frac{\dimension}{p}+\left\lceil \frac{\dimension+\varepsilon}{p}\right\rceil \right)=\tau\cdot\left(\frac{2}{p}+N\right)\leq\tau\cdot\left(\frac{2}{p_{0}}+N_{0}\right), & \text{if }p\in\left(0,1\right).
\end{cases}
\]
In particular, we have $\frac{\sigma}{\tau}\leq\frac{2}{p_{0}}+N_{0}=:\omega$,
even in case of $p\in\left[1,\infty\right]$, since $\frac{2}{p_{0}}+N_{0}\geq N_{0}\geq\left\lceil 2+\varepsilon\right\rceil \geq3$.

Now, Lemma \ref{lem:MainShearletLemma} (with $c=\varepsilon$) yields
a constant 
\[
C_{0}=C_{0}\left(\alpha,\tau_{0},\omega,\varepsilon,K,H,M_{1},M_{2}\right)=C_{0}\left(\alpha,p_{0},q_{0},\varepsilon,\Lambda_{0},\Lambda_{1},\Lambda_{2},\Lambda_{3}\right)>0
\]
satisfying $\max\left\{ K_{1},K_{2}\right\} \leq C_{5}^{\tau}C_{0}^{\tau}$,
provided that we can show $H\geq H_{0}+\varepsilon$, $K\geq K_{0}+\varepsilon$
and $M_{\ell}\geq M_{\ell}^{\left(0\right)}+\varepsilon$ for $\ell\in\left\{ 1,2\right\} $,
where
\begin{align*}
K_{0} & :=\begin{cases}
\max\left\{ \frac{\sigma}{\tau}-s^{\natural},\,\frac{2+\sigma}{\tau}\right\} , & \text{if }\alpha=1,\\
\max\left\{ \frac{1-\alpha}{\tau}+2\frac{\sigma}{\tau}-s^{\natural},\,\frac{2+\sigma}{\tau}\right\} , & \text{if }\alpha\in\left[0,1\right),
\end{cases}\\
M_{1}^{(0)} & :=\begin{cases}
\frac{1}{\tau}+s^{\natural}, & \text{if }\alpha=1,\\
\frac{1}{\tau}+\max\left\{ s^{\natural},\,0\right\} , & \text{if }\alpha\in\left[0,1\right),
\end{cases}\\
M_{2}^{(0)} & :=\left(1+\alpha\right)\frac{\sigma}{\tau}-s^{\natural},\\
H_{0} & :=\frac{1-\alpha}{\tau}+\frac{\sigma}{\tau}-s^{\natural}.
\end{align*}
But we have
\begin{align*}
H_{0} & =\begin{cases}
\frac{1-\alpha}{\tau}+3+s, & \text{if }p\in\left[1,\infty\right],\\
\frac{1-\alpha}{\tau}+\frac{2}{p}+N-\left[\left(1+\alpha\right)\left(\frac{1}{p}-1\right)-s\right], & \text{if }p\in\left(0,1\right)
\end{cases}\\
 & =\begin{cases}
\frac{1-\alpha}{\tau}+3+s, & \text{if }p\in\left[1,\infty\right],\\
\frac{1-\alpha}{\tau}+\frac{1-\alpha}{p}+1+\alpha+\left\lceil \frac{2+\varepsilon}{p}\right\rceil +s, & \text{if }p\in\left(0,1\right)
\end{cases}\\
 & \leq\begin{cases}
\frac{1-\alpha}{\tau_{0}}+3+s_{1}, & \text{if }p\in\left[1,\infty\right],\\
\frac{1-\alpha}{\tau_{0}}+\frac{1-\alpha}{p_{0}}+1+\alpha+\left\lceil \frac{2+\varepsilon}{p_{0}}\right\rceil +s_{1}, & \text{if }p\in\left(0,1\right)
\end{cases}\\
 & \leq\Lambda_{0}-3-2\varepsilon=H-\varepsilon,
\end{align*}
as an easy case distinction (using $\left\lceil p_{0}^{-1}\cdot\left(2+\varepsilon\right)\right\rceil \geq\left\lceil 2+\varepsilon\right\rceil \geq3$
and the observation that $p\in\left(0,1\right)$ entails $p_{0}\in\left(0,1\right)$)
shows.

Furthermore,
\begin{align*}
M_{2}^{(0)} & =\begin{cases}
3\cdot\left(1+\alpha\right)+s, & \text{if }p\in\left[1,\infty\right],\\
\left(1+\alpha\right)\left(\frac{2}{p}+N\right)-\left[\left(1+\alpha\right)\left(\frac{1}{p}-1\right)-s\right], & \text{if }p\in\left(0,1\right)
\end{cases}\\
 & =\begin{cases}
3\cdot\left(1+\alpha\right)+s, & \text{if }p\in\left[1,\infty\right],\\
\left(1+\alpha\right)\left(1+\frac{1}{p}+N\right)+s, & \text{if }p\in\left(0,1\right)
\end{cases}\\
 & \leq\begin{cases}
3\cdot\left(1+\alpha\right)+s_{1}, & \text{if }p\in\left[1,\infty\right],\\
\left(1+\alpha\right)\left(1+\frac{1}{p_{0}}+\left\lceil \frac{2+\varepsilon}{p_{0}}\right\rceil \right)+s_{1}, & \text{if }p\in\left(0,1\right)
\end{cases}\\
 & \leq\Lambda_{2}-\varepsilon=M_{2}-\varepsilon,
\end{align*}
as one can see again using an easy case distinction, since $\left\lceil p_{0}^{-1}\cdot\left(2+\varepsilon\right)\right\rceil \geq\left\lceil 2+\varepsilon\right\rceil \geq3$.

Likewise,
\begin{align*}
M_{1}^{\left(0\right)}\leq\frac{1}{\tau}+\max\left\{ s^{\natural},\,0\right\}  & \leq\frac{1}{\tau_{0}}+\max\left\{ 0,\,\left(1+\alpha\right)\left(\frac{1}{p}-1\right)_{+}-s\right\} \\
 & \leq\frac{1}{\tau_{0}}+\max\left\{ 0,\,\left(1+\alpha\right)\left(\frac{1}{p_{0}}-1\right)-s_{0}\right\} \\
 & =\Lambda_{1}-\varepsilon=M_{1}-\varepsilon.
\end{align*}

Finally, we also have
\begin{align*}
K_{0} & \leq\max\left\{ \frac{1-\alpha}{\tau}+2\frac{\sigma}{\tau}-s^{\natural},\,\frac{2+\sigma}{\tau}\right\} \\
 & =\begin{cases}
\max\left\{ \frac{1-\alpha}{\tau}+6+s,\,\frac{2}{\tau}+3\right\} , & \text{if }p\in\left[1,\infty\right],\\
\max\left\{ \frac{1-\alpha}{\tau}+2\left(\frac{2}{p}+N\right)-\left[\left(1+\alpha\right)\left(\frac{1}{p}-1\right)-s\right],\,\frac{2}{\tau}+\left(\frac{2}{p}+N\right)\right\} , & \text{if }p\in\left(0,1\right)
\end{cases}\\
 & =\begin{cases}
\max\left\{ \frac{1-\alpha}{\tau}+6+s,\,\frac{2}{\tau}+3\right\} , & \text{if }p\in\left[1,\infty\right],\\
\max\left\{ \frac{1-\alpha}{\tau}+\frac{3-\alpha}{p}+2N+1+\alpha+s,\,\frac{2}{\tau}+\frac{2}{p}+N\right\} , & \text{if }p\in\left(0,1\right)
\end{cases}\\
 & \leq\begin{cases}
\max\left\{ \frac{1-\alpha}{\tau_{0}}+6+s_{1},\,\frac{2}{\tau_{0}}+3\right\} , & \text{if }p\in\left[1,\infty\right],\\
\max\left\{ \frac{1-\alpha}{\tau_{0}}+\frac{3-\alpha}{p_{0}}+2\left\lceil \frac{2+\varepsilon}{p_{0}}\right\rceil +1+\alpha+s_{1},\,\frac{2}{\tau_{0}}+\frac{2}{p_{0}}+\left\lceil \frac{2+\varepsilon}{p_{0}}\right\rceil \right\} , & \text{if }p\in\left(0,1\right)
\end{cases}\\
 & \leq\Lambda_{3}-\varepsilon=K-\varepsilon,
\end{align*}
as one can see again using an easy case distinction and the estimate
$\left\lceil p_{0}^{-1}\cdot\left(2+\varepsilon\right)\right\rceil \geq\left\lceil 2+\varepsilon\right\rceil \geq3$.

Consequently, Lemma \ref{lem:MainShearletLemma} is indeed applicable
and yields $\max\left\{ K_{1},K_{2}\right\} \leq C_{5}^{\tau}C_{0}^{\tau}$.
We have thus verified all assumptions of Theorem \ref{thm:AtomicDecompositionTheorem},
which yields a constant
\[
K=K\left(p_{0},q_{0},\varepsilon,\dimension,\CalQ,\Phi,\gamma_{1}^{\left(0\right)},\dots,\gamma_{n}^{\left(0\right)}\right)=K\left(p_{0},q_{0},\varepsilon,\alpha,\varphi,\psi\right)>0
\]
such that the family ${\rm SH}_{\alpha,\varphi,\psi,\delta}^{\left(\pm1\right)}$
from the statement of the current theorem yields an atomic decomposition
of the $\alpha$-shearlet smoothness space $\mathscr{S}_{\alpha,s}^{p,q}\left(\R^{2}\right)=\DecompSp{\CalQ}p{\ell_{v^{s}}^{q}}{}$,
as soon as
\[
0<\delta\leq\delta_{00}:=\min\left\{ 1,\,\left[K\cdot\Omega^{\left(p\right)}\cdot\left(K_{1}^{1/\tau}+K_{2}^{1/\tau}\right)\right]^{-1}\right\} .
\]
But in equation \eqref{eq:AtomicDecompositionSecondConvolutionFactorEstimate}
we saw $\Omega^{\left(p\right)}\leq C_{6}$ independently of $p\geq p_{0}$,
$q\geq q_{0}$ and of $s\in\left[s_{0},s_{1}\right]$, so that 
\[
\delta_{00}\geq\delta_{0}:=\min\left\{ 1,\,\left[2K\cdot C_{0}C_{5}C_{6}\right]^{-1}\right\} ,
\]
where $\delta_{0}>0$ is independent of the precise choice of $p,q,s$,
as long as $p\geq p_{0}$, $q\geq q_{0}$ and $s\in\left[s_{0},s_{1}\right]$.
The claims concerning the notion of convergence for the series defining
$S^{\left(\delta\right)}$ and concerning the independence of the
action of $C^{\left(\delta\right)}$ from the choice of $p,q,s$ are
consequences of the remark after Theorem \ref{thm:AtomicDecompositionTheorem}.
\end{proof}
If $\varphi,\psi$ are compactly supported and if the mother shearlet
$\psi$ is a tensor product, the preceding conditions can be simplified
significantly:
\begin{cor}
\label{cor:ReallyNiceAlphaShearletTensorAtomicDecompositionConditions}Let
$\alpha\in\left[0,1\right]$, $\varepsilon,p_{0},q_{0}\in\left(0,1\right]$
and $s_{0},s_{1}\in\R$ with $s_{0}\leq s_{1}$. Let $\Lambda_{0},\dots,\Lambda_{3}$
as in Theorem \ref{thm:ReallyNiceShearletAtomicDecompositionConditions}
and set $N_{0}:=\left\lceil p_{0}^{-1}\cdot\left(2+\varepsilon\right)\right\rceil $.

Assume that the mother shearlet $\psi$ can be written as $\psi=\psi_{1}\otimes\psi_{2}$
and that $\varphi,\psi_{1},\psi_{2}$ satisfy the following:

\begin{enumerate}[leftmargin=0.6cm]
\item We have $\varphi\in C_{c}^{\left\lceil \Lambda_{0}\right\rceil }\left(\R^{2}\right)$,
$\psi_{1}\in C_{c}^{\left\lceil \Lambda_{2}+3+\varepsilon\right\rceil }\left(\R\right)$,
and $\psi_{2}\in C_{c}^{\left\lceil \Lambda_{3}+3+\varepsilon\right\rceil }\left(\R\right)$.
\item We have $\frac{\d^{\ell}}{\d\xi^{\ell}}\widehat{\psi_{1}}\left(0\right)=0$
for $\ell=0,\dots,N_{0}+\left\lceil \Lambda_{1}\right\rceil -1$.
\item We have $\widehat{\varphi}\left(\xi\right)\neq0$ for all $\xi\in\left[-1,1\right]^{2}$.
\item We have $\widehat{\psi_{1}}\left(\xi\right)\neq0$ for all $\xi\in\left[3^{-1},3\right]$
and $\widehat{\psi_{2}}\left(\xi\right)\neq0$ for all $\xi\in\left[-3,3\right]$.
\end{enumerate}
Then, $\varphi,\psi$ satisfy all assumptions of Theorem \ref{thm:ReallyNiceShearletAtomicDecompositionConditions}.
\end{cor}
\begin{proof}
Since $\varphi,\psi\in L^{1}\left(\R^{2}\right)$ are compactly supported,
it is well known that $\widehat{\varphi},\widehat{\psi}\in C^{\infty}\left(\R^{2}\right)$
with all partial derivatives being polynomially bounded (in fact bounded).
Thanks to the compact support and boundedness of $\varphi,\psi$,
we also clearly have $\left\Vert \varphi\right\Vert _{1+\frac{2}{p_{0}}}<\infty$
and $\left\Vert \psi\right\Vert _{1+\frac{2}{p_{0}}}<\infty$.

Next, if $\xi=\left(\xi_{1},\xi_{2}\right)\in\R^{2}$ satisfies $\xi_{1}\in\left[3^{-1},3\right]$
and $\left|\xi_{2}\right|\le\left|\xi_{1}\right|$, then $\left|\xi_{2}\right|\leq\left|\xi_{1}\right|\leq3$,
i.e., $\xi_{2}\in\left[-3,3\right]$. Thus $\widehat{\psi}\left(\xi\right)=\widehat{\psi_{1}}\left(\xi_{1}\right)\cdot\widehat{\psi_{2}}\left(\xi_{2}\right)\neq0$,
as required in Theorem \ref{thm:ReallyNiceShearletAtomicDecompositionConditions}.

Hence, it only remains to verify
\[
\left|\partial^{\beta}\widehat{\varphi}\left(\xi\right)\right|\lesssim\left(1+\left|\xi\right|\right)^{-\Lambda_{0}}\quad\text{ and }\quad\left|\partial^{\beta}\widehat{\psi}\left(\xi\right)\right|\lesssim\min\left\{ \left|\xi_{1}\right|^{\Lambda_{1}},\left(1+\left|\xi_{1}\right|\right)^{-\Lambda_{2}}\right\} \cdot\left(1+\left|\xi_{2}\right|\right)^{-\Lambda_{3}}\cdot\left(1+\left|\xi\right|\right)^{-\left(3+\varepsilon\right)}
\]
for all $\xi\in\R^{2}$ and all $\beta\in\N_{0}^{2}$ with $\left|\beta\right|\leq N_{0}$.
To this end, we first recall that differentiation under the integral
shows for $g\in C_{c}\left(\R^{\dimension}\right)$ that $\widehat{g}\in C^{\infty}\left(\R^{\dimension}\right)$,
where the derivatives are given by
\begin{equation}
\partial^{\beta}\widehat{g}\left(\xi\right)=\int_{\R^{\dimension}}g\left(x\right)\cdot\partial_{\xi}^{\beta}e^{-2\pi i\left\langle x,\xi\right\rangle }\d x=\int_{\R^{\dimension}}\left(-2\pi ix\right)^{\beta}g\left(x\right)\cdot e^{-2\pi i\left\langle x,\xi\right\rangle }\d x=\left(\Fourier\left[x\mapsto\left(-2\pi ix\right)^{\beta}g\left(x\right)\right]\right)\left(\xi\right).\label{eq:DerivativeOfFourierTransform}
\end{equation}
Furthermore, the usual mantra that ``smoothness of $f$ implies decay
of $\widehat{f}$'' shows that every $g\in W^{N,1}\left(\R^{\dimension}\right)$
satisfies $\left|\widehat{g}\left(\xi\right)\right|\lesssim\left(1+\left|\xi\right|\right)^{-N}$,
see e.g.\@ \cite[Lemma 6.3]{StructuredBanachFrames}.

Now, because of $\varphi\in C_{c}^{\left\lceil \Lambda_{0}\right\rceil }\left(\R^{2}\right)$,
we also have $\left[x\mapsto\left(-2\pi ix\right)^{\beta}\varphi\left(x\right)\right]\in C_{c}^{\left\lceil \Lambda_{0}\right\rceil }\left(\R^{2}\right)\hookrightarrow W^{\left\lceil \Lambda_{0}\right\rceil ,1}\left(\R^{2}\right)$
and thus
\[
\left|\partial^{\beta}\widehat{\varphi}\left(\xi\right)\right|=\left|\left(\Fourier\left[x\mapsto\left(-2\pi ix\right)^{\beta}\varphi\left(x\right)\right]\right)\left(\xi\right)\right|\lesssim\left(1+\left|\xi\right|\right)^{-\left\lceil \Lambda_{0}\right\rceil }\leq\left(1+\left|\xi\right|\right)^{-\Lambda_{0}},
\]
as desired.

For the estimate concerning $\widehat{\psi}$, we have to work slightly
harder: With the same arguments as for $\varphi$, we get $\left|\partial^{\beta}\widehat{\psi_{1}}\left(\xi\right)\right|\lesssim\left(1+\left|\xi\right|\right)^{-\left(\Lambda_{2}+3+\varepsilon\right)}$
and $\left|\partial^{\beta}\widehat{\psi_{2}}\left(\xi\right)\right|\lesssim\left(1+\left|\xi\right|\right)^{-\left(\Lambda_{3}+3+\varepsilon\right)}$
for all $\left|\beta\right|\leq N_{0}$. Now, in case of $\left|\xi_{1}\right|\geq1$,
we have $\left|\xi_{1}\right|^{\Lambda_{1}}\geq1\geq\left(1+\left|\xi_{1}\right|\right)^{-\Lambda_{2}}$
and thus
\begin{align*}
\left|\partial^{\beta}\widehat{\psi}\left(\xi\right)\right| & =\left|\left(\partial^{\beta_{1}}\widehat{\psi_{1}}\right)\left(\xi_{1}\right)\cdot\left(\partial^{\beta_{2}}\widehat{\psi_{2}}\right)\left(\xi_{2}\right)\right|\\
 & \lesssim\left(1+\left|\xi_{1}\right|\right)^{-\left(\Lambda_{2}+3+\varepsilon\right)}\cdot\left(1+\left|\xi_{2}\right|\right)^{-\left(\Lambda_{3}+3+\varepsilon\right)}\\
 & =\min\left\{ \left|\xi_{1}\right|^{\Lambda_{1}},\,\left(1+\left|\xi_{1}\right|\right)^{-\Lambda_{2}}\right\} \cdot\left(1+\left|\xi_{2}\right|\right)^{-\Lambda_{3}}\cdot\left[\left(1+\left|\xi_{1}\right|\right)\left(1+\left|\xi_{2}\right|\right)\right]^{-\left(3+\varepsilon\right)}\\
 & \leq\min\left\{ \left|\xi_{1}\right|^{\Lambda_{1}},\,\left(1+\left|\xi_{1}\right|\right)^{-\Lambda_{2}}\right\} \cdot\left(1+\left|\xi_{2}\right|\right)^{-\Lambda_{3}}\cdot\left(1+\left|\xi\right|\right)^{-\left(3+\varepsilon\right)},
\end{align*}
as desired. Here, the last step used that $\left(1+\left|\xi_{1}\right|\right)\left(1+\left|\xi_{2}\right|\right)\geq1+\left|\xi_{1}\right|+\left|\xi_{2}\right|\geq1+\left|\xi\right|$.

It remains to consider the case $\left|\xi_{1}\right|\leq1$. But
for arbitrary $\beta_{1}\in\N_{0}$ with $\beta_{1}\leq N_{0}$, our
assumptions on $\widehat{\psi_{1}}$ ensure $\partial^{\theta}\left[\partial^{\beta_{1}}\widehat{\psi_{1}}\right]\left(0\right)=0$
for all $\theta\in\left\{ 0,\dots,\left\lceil \Lambda_{1}\right\rceil -1\right\} $,
where we note $\Lambda_{1}>0$, so that $\left\lceil \Lambda_{1}\right\rceil -1\geq0$.
But as the Fourier transform of a compactly supported function, $\widehat{\psi_{1}}$
(and thus also $\partial^{\beta_{1}}\widehat{\psi_{1}}$) can be extended
to an entire function on $\Compl$. In particular,
\begin{align}
\partial^{\beta_{1}}\widehat{\psi_{1}}\left(\xi_{1}\right) & =\sum_{\theta=0}^{\infty}\frac{\partial^{\theta}\left[\partial^{\beta_{1}}\widehat{\psi_{1}}\right]\left(0\right)}{\theta!}\cdot\xi_{1}^{\theta}=\sum_{\theta=\left\lceil \Lambda_{1}\right\rceil }^{\infty}\frac{\partial^{\theta}\left[\partial^{\beta_{1}}\widehat{\psi_{1}}\right]\left(0\right)}{\theta!}\cdot\xi_{1}^{\theta}\nonumber \\
\left({\scriptstyle \text{with }\ell=\theta-\left\lceil \Lambda_{1}\right\rceil }\right) & =\xi_{1}^{\left\lceil \Lambda_{1}\right\rceil }\cdot\sum_{\ell=0}^{\infty}\frac{\partial^{\ell+\left\lceil \Lambda_{1}\right\rceil }\left[\partial^{\beta_{1}}\widehat{\psi_{1}}\right]\left(0\right)}{\left(\ell+\left\lceil \Lambda_{1}\right\rceil \right)!}\cdot\xi_{1}^{\ell}\label{eq:VanishingFourierDerivativesYieldFourierDecayAtOrigin}
\end{align}
for all $\xi\in\R$, where the power series in the last line converges
absolutely on all of $\R$. In particular, the (continuous(!))\@
function defined by the power series is bounded on $\left[-1,1\right]$,
so that we get $\left|\partial^{\beta_{1}}\widehat{\psi_{1}}\left(\xi_{1}\right)\right|\lesssim\left|\xi_{1}\right|^{\left\lceil \Lambda_{1}\right\rceil }\leq\left|\xi_{1}\right|^{\Lambda_{1}}$
for $\xi_{1}\in\left[-1,1\right]$. Furthermore, note $\left(1+\left|\xi_{1}\right|\right)^{-\Lambda_{2}}\geq2^{-\Lambda_{2}}\geq2^{-\Lambda_{2}}\cdot\left|\xi_{1}\right|^{\Lambda_{1}}$,
so that
\begin{align*}
\left|\partial^{\beta}\widehat{\psi}\left(\xi\right)\right| & =\left|\left(\partial^{\beta_{1}}\widehat{\psi_{1}}\right)\left(\xi_{1}\right)\cdot\left(\partial^{\beta_{2}}\widehat{\psi_{2}}\right)\left(\xi_{2}\right)\right|\\
 & \lesssim\left|\xi_{1}\right|^{\Lambda_{1}}\cdot\left(1+\left|\xi_{2}\right|\right)^{-\left(\Lambda_{3}+3+\varepsilon\right)}\\
 & \leq2^{\Lambda_{2}}\cdot\min\left\{ \left|\xi_{1}\right|^{\Lambda_{1}},\,\left(1+\left|\xi_{1}\right|\right)^{-\Lambda_{2}}\right\} \cdot\left(1+\left|\xi_{2}\right|\right)^{-\Lambda_{3}}\cdot2^{3+\varepsilon}\cdot\left[\left(1+\left|\xi_{1}\right|\right)\left(1+\left|\xi_{2}\right|\right)\right]^{-\left(3+\varepsilon\right)}\\
 & \leq2^{3+\varepsilon+\Lambda_{2}}\cdot\min\left\{ \left|\xi_{1}\right|^{\Lambda_{1}},\,\left(1+\left|\xi_{1}\right|\right)^{-\Lambda_{2}}\right\} \cdot\left(1+\left|\xi_{2}\right|\right)^{-\Lambda_{3}}\cdot\left(1+\left|\xi\right|\right)^{-\left(3+\varepsilon\right)}.\qedhere
\end{align*}
\end{proof}
Finally, we provide an analogous simplification of the conditions
of Theorem \ref{thm:NicelySimplifiedAlphaShearletFrameConditions}:
\begin{cor}
\label{cor:ReallyNiceAlphaShearletTensorBanachFrameConditions}Let
$\alpha\in\left[0,1\right]$, $\varepsilon,p_{0},q_{0}\in\left(0,1\right]$
and $s_{0},s_{1}\in\R$ with $s_{0}\leq s_{1}$. Let $K,M_{1},M_{2},H$
as in Theorem \ref{thm:NicelySimplifiedAlphaShearletFrameConditions}
and set $N_{0}:=\left\lceil p_{0}^{-1}\cdot\left(2+\varepsilon\right)\right\rceil $.

The functions $\varphi,\psi$ fulfill all assumption of Theorem \ref{thm:NicelySimplifiedAlphaShearletFrameConditions}
if the mother shearlet $\psi$ can be written as $\psi=\psi_{1}\otimes\psi_{2}$,
where $\varphi,\psi_{1},\psi_{2}$ satisfy the following:

\begin{enumerate}[leftmargin=0.6cm]
\item We have $\varphi\in C_{c}^{\left\lceil H+1\right\rceil }\left(\R^{2}\right)$,
$\psi_{1}\in C_{c}^{\left\lceil M_{2}+1\right\rceil }\left(\R\right)$,
and $\psi_{2}\in C_{c}^{\left\lceil K+1\right\rceil }\left(\R\right)$.
\item We have $\frac{\d^{\ell}}{\d\xi^{\ell}}\widehat{\psi_{1}}\left(0\right)=0$
for $\ell=0,\dots,N_{0}+\left\lceil M_{1}\right\rceil -1$.
\item We have $\widehat{\varphi}\left(\xi\right)\neq0$ for all $\xi\in\left[-1,1\right]^{2}$.
\item We have $\widehat{\psi_{1}}\left(\xi\right)\neq0$ for all $\xi\in\left[3^{-1},3\right]$
and $\widehat{\psi_{2}}\left(\xi\right)\neq0$ for all $\xi\in\left[-3,3\right]$.\qedhere
\end{enumerate}
\end{cor}
\begin{proof}
Observe $\varphi,\psi\in C_{c}\left(\R^{2}\right)\subset L^{1}\left(\R^{2}\right)$
and note $\widehat{\varphi},\widehat{\psi}\in C^{\infty}\left(\R^{2}\right)$,
where all partial derivatives of these functions are bounded (and
thus polynomially bounded), since $\varphi,\psi$ are compactly supported.
Next, since $K,H,M_{1},M_{2}\geq0$, our assumptions clearly entail
$\varphi,\psi\in C_{c}^{1}\left(\R^{2}\right)$, so that $\nabla\varphi,\nabla\psi\in L^{1}\left(\R^{2}\right)\cap L^{\infty}\left(\R^{2}\right)$.
Furthermore, we see exactly as in the proof of Corollary \ref{cor:ReallyNiceAlphaShearletTensorAtomicDecompositionConditions}
that $\widehat{\psi}\left(\xi\right)\neq0$ for all $\xi=\left(\xi_{1},\xi_{2}\right)\in\R^{2}$
with $\xi_{1}\in\left[3^{-1},3\right]$ and $\left|\xi_{2}\right|\leq\left|\xi_{1}\right|$.

Thus, it remains to verify that $\widehat{\varphi},\widehat{\psi}$
satisfy the decay conditions in equation \eqref{eq:ShearletFrameFourierDecayCondition}.
But we see exactly as in the proof of Corollary \ref{cor:ReallyNiceAlphaShearletTensorAtomicDecompositionConditions}
(cf.\@ the argument around equation \eqref{eq:DerivativeOfFourierTransform})
that $\left|\partial^{\beta}\widehat{\varphi}\left(\xi\right)\right|\lesssim\left(1+\left|\xi\right|\right)^{-\left\lceil H+1\right\rceil }\leq\left(1+\left|\xi\right|\right)^{-\left(H+1\right)}$,
as well as $\left|\partial^{\beta_{1}}\widehat{\psi_{1}}\left(\xi_{1}\right)\right|\lesssim\left(1+\left|\xi_{1}\right|\right)^{-\left\lceil M_{2}+1\right\rceil }\leq\left(1+\left|\xi_{1}\right|\right)^{-\left(M_{2}+1\right)}$
and $\left|\partial^{\beta_{2}}\widehat{\psi_{2}}\left(\xi_{2}\right)\right|\lesssim\left(1+\left|\xi_{2}\right|\right)^{-\left\lceil K+1\right\rceil }\leq\left(1+\left|\xi_{2}\right|\right)^{-\left(K+1\right)}$
for all $\beta\in\N_{0}^{2}$ and $\beta_{1},\beta_{2}\in\N_{0}$.
This establishes the last two lines of equation \eqref{eq:ShearletFrameFourierDecayCondition}.

For the first line of equation \eqref{eq:ShearletFrameFourierDecayCondition},
we see as in the proof of Corollary \ref{cor:ReallyNiceAlphaShearletTensorAtomicDecompositionConditions}
(cf.\@ the argument around equation \eqref{eq:VanishingFourierDerivativesYieldFourierDecayAtOrigin})
that $\left|\partial^{\beta_{1}}\widehat{\psi_{1}}\left(\xi_{1}\right)\right|\lesssim\left|\xi_{1}\right|^{\left\lceil M_{1}\right\rceil }\leq\left|\xi_{1}\right|^{M_{1}}$
for all $\xi_{1}\in\left[-1,1\right]$. Since we saw above that $\left|\partial^{\beta_{2}}\widehat{\psi_{2}}\left(\xi_{2}\right)\right|\lesssim\left(1+\left|\xi_{2}\right|\right)^{-\left(K+1\right)}$
for all $\xi_{2}\in\R$, we have thus also established the first line
of equation \eqref{eq:ShearletFrameFourierDecayCondition}.
\end{proof}

\section{The unconnected \texorpdfstring{$\alpha$}{α}-shearlet covering}

\label{sec:UnconnectedAlphaShearletCovering}The $\alpha$-shearlet
covering as introduced in Definition \ref{def:AlphaShearletCovering}
divides the frequency space $\R^{2}$ into a low-frequency part and
into \emph{four} different frequency cones: the top, bottom, left
and right cones. But for real-valued functions, the absolute value
of the Fourier transform is symmetric. Consequently, there is no non-zero
real-valued function with Fourier transform essentially supported
in the top (or left, ...) cone.

For this reason, it is customary to divide the frequency plane into
a low-frequency part and \emph{two} different frequency cones: the
horizontal and the vertical frequency cone. In this section, we account
for this slightly different partition of the frequency plane, by introducing
the so-called \textbf{unconnected $\alpha$-shearlet covering}. The
reason for this nomenclature is that the \emph{connected} base set
$Q=U_{\left(-1,1\right)}^{\left(3^{-1},3\right)}$ from Definition
\ref{def:AlphaShearletCovering} is replaced by the \emph{unconnected}
set $Q\cup\left(-Q\right)$. We then show that all results from the
preceding two sections remain true for this modified covering, essentially
since the associated decomposition spaces are identical, cf.\@ Lemma
\ref{lem:UnconnectedAlphaShearletSmoothnessIsBoring}.
\begin{defn}
\label{def:UnconnectedAlphaShearletCovering}Let $\alpha\in\left[0,1\right]$.
The \textbf{unconnected $\alpha$-shearlet covering} $\CalS_{u}^{\left(\alpha\right)}$
is defined as 
\[
\mathcal{\mathcal{\CalS}}_{u}^{(\alpha)}\::=\:\left(\smash{W_{v}^{\left(\alpha\right)}}\right)_{v\in V^{\left(\alpha\right)}}\::=\:\left(\smash{W_{v}}\right)_{v\in V^{\left(\alpha\right)}}\::=\:\left(\smash{B_{v}}W_{v}'\right)_{v\in V^{\left(\alpha\right)}}=\left(B_{v}W_{v}'+b_{v}\right)_{v\in V^{\left(\alpha\right)}}\:,
\]
where:

\begin{itemize}[leftmargin=0.6cm]
\item The \emph{index set} $V^{\left(\alpha\right)}$ is given by $V:=V^{\left(\alpha\right)}:=\left\{ 0\right\} \cup V_{0}$,
where 
\[
\qquad V_{0}:=V_{0}^{(\alpha)}:=\left\{ \left(n,m,\delta\right)\in\N_{0}\times\Z\times\left\{ 0,1\right\} \with\left|m\right|\leq G_{n}\right\} \quad\text{ with }\quad G_{n}:=G_{n}^{\left(\alpha\right)}:=\left\lceil \smash{2^{n\left(1-\alpha\right)}}\right\rceil .
\]
\item The \emph{basic sets} $\left(W_{v}'\right)_{v\in V^{\left(\alpha\right)}}$
are given by $W_{0}':=\left(-1,1\right)^{2}$ and by $W_{v}':=Q_{u}:=U_{\left(-1,1\right)}^{\left(3^{-1},3\right)}\cup\left[-\smash{U_{\left(-1,1\right)}^{\left(3^{-1},3\right)}}\vphantom{U^{\left(3\right)}}\right]$
for $v\in V_{0}^{\left(\alpha\right)}$. The notation $U_{\left(a,b\right)}^{\left(\gamma,\mu\right)}$
used here is as defined in equation \eqref{eq:BasicShearletSet}.
\item The \emph{matrices} $\left(B_{v}\right)_{v\in V^{\left(\alpha\right)}}$
are given by $B_{0}:=B_{0}^{\left(\alpha\right)}:=\identity$ and
by $B_{v}:=B_{v}^{\left(\alpha\right)}:=R^{\delta}\cdot A_{n,m}^{\left(\alpha\right)}$,
where we define $A_{n,m}^{\left(\alpha\right)}:=D_{2^{n}}^{\left(\alpha\right)}\cdot S_{m}^{T}$
for $v=\left(n,m,\delta\right)\in V_{0}$. Here, the matrices $R,S_{x}$
and $D_{b}^{\left(\alpha\right)}$ are as in equation \eqref{eq:StandardMatrices}.
\item The \emph{translations} $\left(b_{v}\right)_{v\in V^{\left(\alpha\right)}}$
are given by $b_{v}:=0$ for all $v\in V^{\left(\alpha\right)}$.
\end{itemize}
Finally, we define the \emph{weight} $u=\left(u_{v}\right)_{v\in V}$
by $u_{0}:=1$ and $u_{n,m,\delta}:=2^{n}$ for $\left(n,m,\delta\right)\in V_{0}$.
\end{defn}
The unconnected $\alpha$-shearlet covering $\CalS_{u}^{\left(\alpha\right)}$
is highly similar to the (connected) $\alpha$-shearlet covering $\CalS^{\left(\alpha\right)}$
from Definition \ref{def:AlphaShearletCovering}. In particular, we
have $Q_{u}=Q\cup\left(-Q\right)$ with $Q=U_{\left(-1,1\right)}^{\left(3^{-1},3\right)}$
as in Definition \ref{def:AlphaShearletCovering}. To further exploit
this connection between the two coverings, we define the \textbf{projection
map}
\[
\pi:I^{\left(\alpha\right)}\to V^{\left(\alpha\right)},i\mapsto\begin{cases}
0, & \text{if }i=0,\\
\left(n,m,\delta\right), & \text{if }i=\left(n,m,\varepsilon,\delta\right)\in I_{0}^{\left(\alpha\right)}.
\end{cases}
\]
Likewise, for $\varepsilon\in\left\{ \pm1\right\} $, we define the
\textbf{$\varepsilon$-injection}
\[
\iota_{\varepsilon}:V^{\left(\alpha\right)}\to I^{\left(\alpha\right)},v\mapsto\begin{cases}
0, & \text{if }v=0,\\
\left(n,m,\varepsilon,\delta\right), & \text{if }v=\left(n,m,\delta\right)\in V_{0}^{\left(\alpha\right)}.
\end{cases}
\]
Note that $B_{v}^{\left(\alpha\right)}=\varepsilon\cdot T_{\iota_{\varepsilon}\left(v\right)}^{\left(\alpha\right)}$
for all $v\in V_{0}^{\left(\alpha\right)}$, so that
\begin{equation}
W_{v}^{\left(\alpha\right)}=S_{\iota_{1}\left(v\right)}^{\left(\alpha\right)}\cup S_{\iota_{-1}\left(v\right)}^{\left(\alpha\right)}=\bigcup_{\varepsilon\in\left\{ \pm1\right\} }S_{\iota_{\varepsilon}\left(v\right)}^{\left(\alpha\right)}\qquad\forall v\in V_{0}^{\left(\alpha\right)},\label{eq:UnconnectedCoveringAsUnionOfConnectedCovering}
\end{equation}
since $B_{v}^{\left(\alpha\right)}\left[-\smash{U_{\left(-1,1\right)}^{\left(3,3^{-1}\right)}}\vphantom{U^{\left(3\right)}}\right]=-B_{v}^{\left(\alpha\right)}\smash{U_{\left(-1,1\right)}^{\left(3,3^{-1}\right)}}=T_{\iota_{-1}\left(v\right)}^{\left(\alpha\right)}Q=S_{\iota_{-1}\left(v\right)}^{\left(\alpha\right)}$.
Because of $W_{0}^{\left(\alpha\right)}=\left(-1,1\right)^{2}=S_{0}^{\left(\alpha\right)}$,
equation \eqref{eq:UnconnectedCoveringAsUnionOfConnectedCovering}
remains valid for $v=0$. Using these observations, we can now prove
the following lemma:
\begin{lem}
\noindent \label{lem:UnconnectedAlphaShearletCoveringIsAlmostStructured}The
unconnected $\alpha$-shearlet covering $\CalS_{u}^{(\alpha)}$ 
is an almost structured covering of $\R^{2}$.
\end{lem}
\begin{proof}
In Lemma \ref{lem:AlphaShearletCoveringIsAlmostStructured}, we showed
that the (connected) $\alpha$-shearlet covering $\CalS^{\left(\alpha\right)}$
is almost structured. Thus, for the proof of the present lemma, we
will frequently refer to the proof of Lemma \ref{lem:AlphaShearletCoveringIsAlmostStructured}. 

First of all, recall from the proof of Lemma \ref{lem:AlphaShearletCoveringIsAlmostStructured}
the notation $P_{(n,m,\varepsilon,\delta)}'=U_{(-3/4,3/4)}^{(1/2,5/2)}$
for arbitrary $\left(n,m,\varepsilon,\delta\right)\in I_{0}$. Then,
for $v=\left(n,m,\delta\right)\in V_{0}$ let us define $R_{(n,m,\delta)}':=P_{(n,m,1,\delta)}'\cup\left(-P_{(n,m,1,\delta)}'\right)$.
Furthermore, set $R_{0}':=P_{0}'$, again with $P_{0}'=\left(-\frac{3}{4},\frac{3}{4}\right)^{2}$
as in the proof of Lemma \ref{lem:AlphaShearletCoveringIsAlmostStructured}.
Then it is not hard to verify $\overline{R_{v}'}\subset W_{v}'$ for
all $v\in V$.

Furthermore, in the proof of Lemma \ref{lem:AlphaShearletCoveringIsAlmostStructured},
we showed $\bigcup_{i\in I}T_{i}P_{i}'=\R^{2}$. But this implies
\begin{align*}
\bigcup_{v\in V}\left(B_{v}R_{v}'+b_{v}\right) & =R_{0}'\cup\bigcup_{\left(n,m,\delta\right)\in V_{0}}B_{\left(n,m,\delta\right)}R_{\left(n,m,\delta\right)}'\\
 & =P_{0}'\cup\bigcup_{\left(n,m,\delta\right)\in V_{0}}\left[B_{\left(n,m,\delta\right)}P_{\left(n,m,1,\delta\right)}'\cup-B_{\left(n,m,\delta\right)}P_{\left(n,m,1,\delta\right)}'\right]\\
\left({\scriptstyle \text{since }P_{\left(n,m,1,\delta\right)}'=P_{\left(n,m,-1,\delta\right)}'}\right) & =P_{0}'\cup\bigcup_{\left(n,m,\delta\right)\in V_{0}}\left[T_{\left(n,m,1,\delta\right)}P_{\left(n,m,1,\delta\right)}'\cup T_{\left(n,m,-1,\delta\right)}P_{\left(n,m,-1,\delta\right)}'\right]\\
 & =\bigcup_{i\in I}\left(T_{i}P_{i}'+b_{i}\right)=\R^{2}.
\end{align*}

Next, if $W_{\left(n,m,\delta\right)}^{\left(\alpha\right)}\cap W_{\left(k,\ell,\gamma\right)}^{\left(\alpha\right)}\neq\emptyset$,
then equation \eqref{eq:UnconnectedCoveringAsUnionOfConnectedCovering}
yields certain $\varepsilon,\beta\in\left\{ \pm1\right\} $ such that
$S_{\left(n,m,\varepsilon,\delta\right)}^{\left(\alpha\right)}\cap S_{\left(k,\ell,\beta,\gamma\right)}^{(\alpha)}\neq\emptyset$.
But this implies $\left(k,\ell,\gamma\right)=\pi\left(\left(k,\ell,\beta,\gamma\right)\right)$,
where $\left(k,\ell,\beta,\gamma\right)\in I_{0}\cap\left(n,m,\varepsilon,\delta\right)^{\ast}$
and where the index cluster is formed with respect to the covering
$\CalS^{\left(\alpha\right)}$. Consequently, we have shown
\begin{equation}
\left(n,m,\delta\right)^{\ast}\subset\left\{ 0\right\} \cup\bigcup_{\varepsilon\in\left\{ \pm1\right\} }\pi\left(I_{0}\cap\left(n,m,\varepsilon,\delta\right)^{\ast}\right).\label{eq:UnconnectedShearletCoveringClusterInclusion}
\end{equation}
But since $\CalS^{\left(\alpha\right)}$ is admissible, the constant
$N:=\sup_{i\in I}\left|i^{\ast}\right|$ is finite. But by what we
just showed, we have $\left|\left(n,m,\delta\right)^{\ast}\right|\leq1+2N$
for all $\left(n,m,\delta\right)\in V_{0}$. Finally, using a very
similar argument one can show 
\[
0^{\ast_{\CalS_{u}^{\left(\alpha\right)}}}\subset\left\{ 0\right\} \cup\pi\left(I_{0}\cap0^{\ast_{\CalS^{\left(\alpha\right)}}}\right),
\]
where the index-cluster is taken with respect to $\CalS_{u}^{\left(\alpha\right)}$
on the left-hand side and with respect to $\CalS^{\left(\alpha\right)}$
on the right-hand side. Thus, $\left|0^{\ast_{\CalS_{u}^{\left(\alpha\right)}}}\right|\leq1+N$,
so that $\sup_{v\in V}\left|v^{\ast}\right|\leq1+2N<\infty$. All
in all, we have thus shown that $\CalS_{u}^{\left(\alpha\right)}$
is an admissible covering of $\R^{2}$.

It remains to verify $\sup_{v\in V}\sup_{r\in v^{\ast}}\left\Vert B_{v}^{-1}B_{r}\right\Vert <\infty$.
To this end, recall that $C:=\sup_{i\in I}\sup_{j\in i^{\ast}}\left\Vert T_{i}^{-1}T_{j}\right\Vert $
is finite, since $\CalS^{\left(\alpha\right)}$ is an almost structured
covering. Now, let $v\in V$ and $r\in v^{\ast}$ be arbitrary. We
distinguish several cases:

\textbf{Case 1}: We have $v=\left(n,m,\delta\right)\in V_{0}$ and
$r=\left(k,\ell,\gamma\right)\in V_{0}$. As above, there are thus
certain $\varepsilon,\beta\in\left\{ \pm1\right\} $ such that $\left(k,\ell,\beta,\gamma\right)\in\left(n,m,\varepsilon,\delta\right)^{\ast}$.
Hence,
\[
\left\Vert B_{v}^{-1}B_{r}\right\Vert =\left\Vert \left(\varepsilon\cdot T_{n,m,\varepsilon,\delta}\right)^{-1}\cdot\beta\cdot T_{k,\ell,\beta,\gamma}\right\Vert =\left\Vert \left(T_{n,m,\varepsilon,\delta}\right)^{-1}\cdot T_{k,\ell,\beta,\gamma}\right\Vert \leq C.
\]

\textbf{Case 2}: We have $v=0$ and $r=\left(k,\ell,\gamma\right)\in V_{0}$.
There is then some $\beta\in\left\{ \pm1\right\} $ satisfying $\left(k,\ell,\beta,\gamma\right)\in0^{\ast}$,
where the index-cluster is taken with respect to $\CalS^{\left(\alpha\right)}$.
Hence, we get again that
\[
\left\Vert B_{v}^{-1}B_{r}\right\Vert =\left\Vert T_{0}^{-1}\cdot\beta\cdot T_{k,\ell,\beta,\gamma}\right\Vert =\left\Vert T_{0}^{-1}\cdot T_{k,\ell,\beta,\gamma}\right\Vert \leq C.
\]

\textbf{Case 3}: We have $v=\left(n,m,\delta\right)\in V_{0}$ and
$r=0$. Hence, $0\in\left(n,m,\varepsilon,\delta\right)^{\ast}$ for
some $\varepsilon\in\left\{ \pm1\right\} $, so that
\[
\left\Vert B_{v}^{-1}B_{r}\right\Vert =\left\Vert \left(\varepsilon\cdot T_{n,m,\varepsilon,\delta}\right)^{-1}\cdot T_{0}\right\Vert =\left\Vert T_{n,m,\varepsilon,\delta}^{-1}\cdot T_{0}\right\Vert \leq C.
\]

\textbf{Case 4}: We have $v=r=0$. In this case, $\left\Vert B_{v}^{-1}B_{r}\right\Vert =1\leq C$.

Hence, we have verified $\sup_{v\in V}\sup_{r\in v^{\ast}}\left\Vert B_{v}^{-1}B_{r}\right\Vert <\infty$.
Since the sets $\left\{ \smash{W_{v}'}\with v\in\smash{V}\right\} $
and $\left\{ \smash{R_{v}'}\with v\in\smash{V}\right\} $ are finite
families of bounded, open sets (in fact, each of these families only
has two elements), we have shown that $\CalS_{u}^{\left(\alpha\right)}$
is an almost structured covering of $\R^{2}$.
\end{proof}
Before we can define the decomposition spaces associated to the unconnected
$\alpha$-shearlet covering $\CalS_{u}^{\left(\alpha\right)}$, we
need to verify that the weights that we want to use are $\CalS_{u}^{\left(\alpha\right)}$-moderate.
\begin{lem}
\label{lem:WeightUnconnectedModerate}Let $u=\left(u_{v}\right)_{v\in V}$
as in Definition \ref{def:UnconnectedAlphaShearletCovering}. Then
$u^{s}=\left(u_{v}^{s}\right)_{v\in V}$ is $\CalS_{u}^{\left(\alpha\right)}$-moderate
with $C_{\CalS_{u}^{\left(\alpha\right)},u^{s}}\leq39^{\left|s\right|}$.
\end{lem}
\begin{proof}
As seen in equation \eqref{eq:UnconnectedCoveringAsUnionOfConnectedCovering},
we have $W_{v}^{\left(\alpha\right)}=\bigcup_{\varepsilon\in\left\{ \pm1\right\} }S_{\iota_{\varepsilon}\left(v\right)}^{\left(\alpha\right)}$
for arbitrary $v\in V$ (also for $v=0$). Furthermore, it is easy
to see $u_{v}=w_{\iota_{\varepsilon}\left(v\right)}$ for arbitrary
$\varepsilon\in\left\{ \pm1\right\} $ and $v\in V$.

Thus, if $W_{v}^{\left(\alpha\right)}\cap W_{r}^{\left(\alpha\right)}\neq\emptyset$
for certain $v,r\in V$, there are $\varepsilon,\beta\in\left\{ \pm1\right\} $
such that $S_{\iota_{\varepsilon}\left(v\right)}^{\left(\alpha\right)}\cap S_{\iota_{\beta}\left(r\right)}^{\left(\alpha\right)}\neq\emptyset$.
But Lemma \ref{lem:AlphaShearletWeightIsModerate} shows that $w^{s}$
is $\CalS^{\left(\alpha\right)}$-moderate with $C_{\CalS^{\left(\alpha\right)},w^{s}}\leq39^{\left|s\right|}$.
Hence,
\[
u_{v}^{s}/u_{r}^{s}=w_{\iota_{\varepsilon}\left(v\right)}^{s}/w_{\iota_{\beta}\left(r\right)}^{s}\leq39^{\left|s\right|}.\qedhere
\]
\end{proof}
Since we now know that $\CalS_{u}^{\left(\alpha\right)}$ is an almost
structured covering of $\R^{2}$ and since $u^{s}$ is $\CalS_{u}^{\left(\alpha\right)}$-moderate,
we see precisely as in the remark after Definition \ref{def:AlphaShearletSmoothnessSpaces}
that the \emph{unconnected} $\alpha$-shearlet smoothness spaces that
we now define are well-defined Quasi-Banach spaces. We emphasize that
the following definition will only be of transitory relevance, since
we will immediately show that the newly defined \emph{unconnected}
$\alpha$-shearlet smoothness spaces are identical with the previously
defined $\alpha$-shearlet smoothness spaces.
\begin{defn}
\label{def:UnconnectedAlphaShearletSmoothness}For $\alpha\in\left[0,1\right]$,
$p,q\in\left(0,\infty\right]$ and $s\in\R$, we define the \textbf{unconnected
$\alpha$-shearlet smoothness space} $\mathscr{D}_{\alpha,s}^{p,q}\left(\R^{2}\right)$
associated to these parameters as
\[
\mathscr{D}_{\alpha,s}^{p,q}\left(\R^{2}\right):=\DecompSp{\CalS_{u}^{\left(\alpha\right)}}p{\ell_{u^{s}}^{q}}{},
\]
where the covering $\CalS_{u}^{\left(\alpha\right)}$ and the weight
$u^{s}$ are as in Definition \ref{def:UnconnectedAlphaShearletCovering}
and Lemma \ref{lem:WeightUnconnectedModerate}, respectively.
\end{defn}
\begin{lem}
\label{lem:UnconnectedAlphaShearletSmoothnessIsBoring}We have
\[
\mathscr{S}_{\alpha,s}^{p,q}\left(\smash{\R^{2}}\right)=\mathscr{D}_{\alpha,s}^{p,q}\left(\smash{\R^{2}}\right)\qquad\forall\alpha\in\left[0,1\right],\quad p,q\in\left(0,\infty\right]\quad\text{ and }\quad s\in\R,
\]
with equivalent quasi-norms.
\end{lem}
\begin{proof}
We will derive the claim from \cite[Lemma 6.11, part (2)]{DecompositionEmbedding},
with the choice $\CalQ:=\CalS_{u}^{\left(\alpha\right)}$ and $\CalP:=\CalS^{\left(\alpha\right)}$,
recalling that $\mathscr{S}_{\alpha,s}^{p,q}\left(\R^{2}\right)=\DecompSp{\CalS^{\left(\alpha\right)}}p{\ell_{w^{s}}^{q}}{}=\Fourier^{-1}\left[\FourierDecompSp{\CalS^{\left(\alpha\right)}}p{\ell_{w^{s}}^{q}}{}\right]$
and likewise $\mathscr{D}_{\alpha,s}^{p,q}\left(\R^{2}\right)=\Fourier^{-1}\left[\FourierDecompSp{\CalS_{u}^{\left(\alpha\right)}}p{\ell_{u^{s}}^{q}}{}\right]$.

To this end, we first have to verify that the coverings $\CalS^{\left(\alpha\right)}$
and $\CalS_{u}^{\left(\alpha\right)}$ are \textbf{weakly equivalent}.
This means that
\[
\sup_{i\in I}\left|\left\{ v\in V\with\smash{W_{v}^{\left(\alpha\right)}}\cap S_{i}^{\left(\alpha\right)}\neq\emptyset\right\} \right|<\infty\qquad\text{ and }\qquad\sup_{v\in V}\left|\left\{ i\in I\with S_{i}^{\left(\alpha\right)}\cap W_{v}^{\left(\alpha\right)}\neq\emptyset\right\} \right|<\infty.
\]
We begin with the first claim and thus let $i\in I$ be arbitrary.
It is easy to see $S_{i}^{\left(\alpha\right)}\subset W_{\pi\left(i\right)}^{\left(\alpha\right)}$.
Consequently, if $v\in V$ satisfies $W_{v}^{\left(\alpha\right)}\cap S_{i}^{\left(\alpha\right)}\neq\emptyset$,
then $\emptyset\subsetneq W_{v}^{\left(\alpha\right)}\cap S_{i}^{\left(\alpha\right)}\subset W_{v}^{\left(\alpha\right)}\cap W_{\pi\left(i\right)}^{\left(\alpha\right)}$
and thus $v\in\left[\pi\left(i\right)\right]^{\ast}$, where the index-cluster
is formed with respect to $\CalS_{u}^{\left(\alpha\right)}$. On the
one hand, this implies 
\begin{equation}
w_{i}^{t}=u_{\pi\left(i\right)}^{t}\:\asymp_{t}\:u_{v}^{t}\qquad\text{ if }S_{i}^{\left(\alpha\right)}\cap W_{v}^{\left(\alpha\right)}\neq\emptyset,\text{ for arbitrary }t\in\R,\label{eq:ConnectedUnconnectedCoveringWeightEquivalence}
\end{equation}
since $u^{t}$ is $\CalS_{u}^{\left(\alpha\right)}$-moderate by Lemma
\ref{lem:WeightUnconnectedModerate}. On the other hand, we get
\[
\sup_{i\in I}\left|\left\{ v\in V\with\smash{W_{v}^{\left(\alpha\right)}}\cap S_{i}^{\left(\alpha\right)}\neq\emptyset\right\} \right|\leq\sup_{i\in I}\left|\left[\pi\left(i\right)\right]^{\ast}\right|\leq\sup_{v\in V}\left|v^{\ast}\right|<\infty,
\]
since we know that $\CalS_{u}^{\left(\alpha\right)}$ is admissible
(cf.\@ Lemma \ref{lem:UnconnectedAlphaShearletCoveringIsAlmostStructured}).

Now, let us verify the second claim. To this end, let $v\in V$ be
arbitrary. For $i\in I$ with $S_{i}^{\left(\alpha\right)}\cap W_{v}^{\left(\alpha\right)}\neq\emptyset$,
equation \eqref{eq:UnconnectedCoveringAsUnionOfConnectedCovering}
shows $\emptyset\neq\bigcup_{\varepsilon\in\left\{ \pm1\right\} }\left(S_{i}^{\left(\alpha\right)}\cap S_{\iota_{\varepsilon}\left(v\right)}^{\left(\alpha\right)}\right)$
and thus $i\in\bigcup_{\varepsilon\in\left\{ \pm1\right\} }\left[\iota_{\varepsilon}\left(v\right)\right]^{\ast}$,
where the index-cluster is formed with respect to $\CalS^{\left(\alpha\right)}$.
As above, this yields
\[
\sup_{v\in V}\left|\left\{ i\in I\with S_{i}^{\left(\alpha\right)}\cap W_{v}^{\left(\alpha\right)}\neq\emptyset\right\} \right|\leq\sup_{v\in V}\left|\left[\iota_{1}\left(v\right)\right]^{\ast}\right|+\left|\left[\iota_{-1}\left(v\right)\right]^{\ast}\right|\leq2\cdot\sup_{i\in I}\left|i^{\ast}\right|<\infty,
\]
since $\CalS^{\left(\alpha\right)}$ is admissible (cf.\@ Lemma \ref{lem:AlphaShearletCoveringIsAlmostStructured}).

\medskip{}

We have thus verified the two main assumptions of \cite[Lemma 6.11]{DecompositionEmbedding},
namely that $\CalQ,\CalP$ are weakly equivalent and that $u_{v}^{s}\asymp w_{i}^{s}$
if $W_{v}^{\left(\alpha\right)}\cap S_{i}^{\left(\alpha\right)}\neq\emptyset$,
thanks to equation \eqref{eq:ConnectedUnconnectedCoveringWeightEquivalence}.
But since we also want to get the claim for $p\in\left(0,1\right)$,
we have to verify the additional condition (2) from \cite[Lemma 6.11]{DecompositionEmbedding},
i.e., that $\CalP=\CalS^{\left(\alpha\right)}=\left(\smash{S_{j}^{\left(\alpha\right)}}\right)_{j\in I}=\left(T_{j}Q_{j}'\right)_{j\in I}$
is almost subordinate to $\CalQ=\CalS_{u}^{\left(\alpha\right)}=\left(W_{v}\right)_{v\in V}=\left(B_{v}W_{v}'\right)_{v\in V}$
and that $\left|\det\left(T_{j}^{-1}B_{v}\right)\right|\lesssim1$
if $W_{v}\cap S_{j}^{\left(\alpha\right)}\neq\emptyset$. But we saw
in equation \eqref{eq:ConnectedUnconnectedCoveringWeightEquivalence}
that if $W_{v}^{\left(\alpha\right)}\cap S_{j}^{\left(\alpha\right)}\neq\emptyset$,
then
\[
\left|\det\left(T_{j}^{-1}B_{v}\right)\right|=\left(w_{j}^{1+\alpha}\right)^{-1}\cdot u_{v}^{1+\alpha}\:\asymp_{\alpha}\:1.
\]
Furthermore, $S_{j}^{\left(\alpha\right)}\subset W_{\pi\left(j\right)}^{\left(\alpha\right)}$
for all $j\in I$, so that $\CalP=\CalS^{\left(\alpha\right)}$ is
subordinate (and thus also almost subordinate, cf.\@ \cite[Definition 2.10]{DecompositionEmbedding})
to $\CalQ=\CalS_{u}^{\left(\alpha\right)}$, as required. The claim
is now an immediate consequence of \cite[Lemma 6.11]{DecompositionEmbedding}.
\end{proof}
In order to allow for a more succinct formulation of our results about
Banach frames and atomic decompositions in the setting of the \emph{unconnected}
$\alpha$-shearlet covering, we now introduce the notion of \textbf{cone-adapted
$\alpha$-shearlet systems}. As we will see in Section \ref{sec:CartoonLikeFunctionsAreBoundedInAlphaShearletSmoothness},
these systems are different, but intimately connected to the \textbf{cone-adapted
$\beta$-shearlet systems} (with $\beta\in\left(1,\infty\right)$)
as introduced in \cite[Definition 3.10]{AlphaMolecules}. There are
three main reasons why we think that the new definition is preferable
to the old one:
\begin{enumerate}
\item With the new definition, a family $\left(L_{\delta k}\,\varphi\right)_{k\in\Z^{2}}\cup\left(\psi_{j,\ell,\delta,k}\right)_{j,\ell,\delta,k}$
of $\alpha$-shearlets has the property that the shearlets $\psi_{j,\ell,\delta,k}$
of scale $j$ have essential frequency support in the dyadic corona
$\left\{ \xi\in\R^{2}\with2^{j-c}<\left|\xi\right|<2^{j+c}\right\} $
for suitable $c>0$. In contrast, for $\beta$-shearlets, the shearlets
of scale $j$ have essential frequency support in $\left\{ \xi\in\R^{2}\with2^{\frac{\beta}{2}\left(j-c\right)}<\left|\xi\right|<2^{\frac{\beta}{2}\left(j+c\right)}\right\} $,
cf.\@ Lemma \ref{lem:ReciprocalShearletCoveringIsAlmostStructuredGeneralized}.
\item With the new definition, a family of cone-adapted $\alpha$-shearlets
is also a family of $\alpha$-molecules, if the generators are chosen
suitably. In contrast, for $\beta$-shearlets, one has the slightly
inconvenient fact that a family of cone-adapted $\beta$-shearlets
is a family of $\beta^{-1}$-molecules, cf.\@ \cite[Proposition 3.11]{AlphaMolecules}.
\item The new definition includes the two boundary values $\alpha\in\left\{ 0,1\right\} $
which correspond to ridgelet-like systems and to wavelet-like systems,
respectively. In contrast, for $\beta$-shearlets, the boundary values
$\beta\in\left\{ 1,\infty\right\} $ are excluded from the definition.
\end{enumerate}
We remark that a very similar definition to the one given here is
already introduced in \cite[Definition 5.1]{MultivariateAlphaMolecules},
even generally in $\R^{\dimension}$ for $\dimension\geq2$.
\begin{defn}
\label{def:AlphaShearletSystem}Let $\alpha\in\left[0,1\right]$.
For generators $\varphi,\psi\in L^{1}\left(\R^{2}\right)+L^{2}\left(\R^{2}\right)$
and a given sampling density $\delta>0$, we define the \textbf{cone-adapted
$\alpha$-shearlet system} with sampling density $\delta$ generated
by $\varphi,\psi$ as
\[
{\rm SH}_{\alpha}\left(\varphi,\psi;\,\delta\right):=\left(\gamma^{\left[v,k\right]}\right)_{v\in V,\,k\in\Z^{2}}:=\left(L_{\delta\cdot B_{v}^{-T}k}\:\gamma^{\left[v\right]}\right)_{v\in V,\,k\in\Z^{2}}\quad\text{ with }\quad\gamma^{\left[v\right]}:=\begin{cases}
\left|\det\smash{B_{v}}\right|^{1/2}\cdot\left(\psi\circ B_{v}^{T}\right), & \text{if }v\in V_{0},\\
\varphi, & \text{if }v=0,
\end{cases}
\]
where $V,V_{0}$ and $B_{v}$ are as in Definition \ref{def:UnconnectedAlphaShearletCovering}.
Note that the notation $\gamma^{\left[v,k\right]}$ suppresses the
sampling density $\delta>0$. If we want to emphasize this sampling
density, we write $\gamma^{\left[v,k,\delta\right]}$ instead of $\gamma^{\left[v,k\right]}$.
\end{defn}
\begin{rem}
\label{rem:AlphaShearletsYieldUsualShearlets}In case of $\alpha=\frac{1}{2}$,
the preceding definition yields special cone-adapted shearlet systems:
As defined in \cite[Definition 1.2]{CompactlySupportedShearletsAreOptimallySparse},
the cone-adapted shearlet system ${\rm SH}\left(\varphi,\psi,\theta;\,\delta\right)$
with sampling density $\delta>0$ generated by $\varphi,\psi,\theta\in L^{2}\left(\R^{2}\right)$
is ${\rm SH}\left(\varphi,\psi,\theta;\,\delta\right)=\Phi\left(\varphi;\,\delta\right)\cup\Psi\left(\psi;\,\delta\right)\cup\Theta\left(\theta;\,\delta\right)$,
where
\begin{align*}
\Phi\left(\varphi;\,\delta\right) & :=\left\{ \varphi_{k}:=\varphi\left(\mybullet-\delta k\right)\with k\in\Z^{2}\right\} ,\\
\Psi\left(\psi;\,\delta\right) & :=\left\{ \psi_{j,\ell,k}:=2^{\frac{3}{4}j}\cdot\psi\left(S_{\ell}A_{2^{j}}\mybullet-\delta k\right)\with j\in\N_{0},\ell\in\Z\text{ with }\left|\ell\right|\leq\left\lceil \smash{2^{j/2}}\right\rceil \text{ and }k\in\Z^{2}\right\} ,\\
\Theta\left(\theta;\,\delta\right) & :=\left\{ \theta_{j,\ell,k}:=2^{\frac{3}{4}j}\cdot\theta\left(S_{\ell}^{T}\widetilde{A}_{2^{j}}\mybullet-\delta k\right)\with j\in\N_{0},\ell\in\Z\text{ with }\left|\ell\right|\leq\left\lceil \smash{2^{j/2}}\right\rceil \text{ and }k\in\Z^{2}\right\} ,
\end{align*}
with $S_{k}=\left(\begin{smallmatrix}1 & k\\
0 & 1
\end{smallmatrix}\right)$, $A_{2^{j}}={\rm diag}\left(2^{j},\,2^{j/2}\right)$ and $\widetilde{A}_{2^{j}}={\rm diag}\left(2^{j/2},\,2^{j}\right)$.

Now, the most common choice for $\theta$ is $\theta=\psi\circ R$
for $R=\left(\begin{smallmatrix}0 & 1\\
1 & 0
\end{smallmatrix}\right)$. With this choice, we observe in the notation of Definitions \ref{def:AlphaShearletSystem}
and \ref{def:UnconnectedAlphaShearletCovering} that
\[
\gamma^{\left[0,k\right]}=L_{\delta\cdot B_{0}^{-T}k}\:\gamma^{\left[0\right]}=L_{\delta k}\:\varphi=\varphi\left(\mybullet-\delta k\right)=\varphi_{k}\qquad\forall k\in\Z^{2}.
\]
Furthermore, we note because of $\alpha=\frac{1}{2}$ that
\[
B_{j,\ell,0}^{T}=\left[\left(\begin{matrix}2^{j} & 0\\
0 & 2^{j/2}
\end{matrix}\right)\cdot\left(\begin{matrix}1 & 0\\
\ell & 1
\end{matrix}\right)\right]^{T}=S_{\ell}\cdot A_{2^{j}}\:,
\]
with $\left|\det\smash{B_{j,\ell,0}}\right|=2^{\frac{3}{2}j}$, so
that
\[
\gamma^{\left[\left(j,\ell,0\right),k\right]}=L_{\delta\cdot\left[S_{\ell}A_{2^{j}}\right]^{-1}k}\:\gamma^{\left[\left(j,\ell,0\right)\right]}=2^{\frac{3}{4}j}\cdot\psi\left(S_{\ell}\cdot A_{2^{j}}\mybullet-\delta k\right)=\psi_{j,\ell,k}\qquad\forall\left(j,\ell,0\right)\in V_{0}\text{ and }k\in\Z^{2}.
\]
Finally, we observe $\theta\left(S_{\ell}^{T}\widetilde{A}_{2^{j}}\mybullet-\delta k\right)=\psi\left(RS_{\ell}^{T}\widetilde{A}_{2^{j}}\mybullet-\delta Rk\right)$,
as well as
\[
R\cdot S_{\ell}^{T}\cdot\widetilde{A}_{2^{j}}=\left(\begin{matrix}0 & 1\\
1 & 0
\end{matrix}\right)\left(\begin{matrix}1 & 0\\
\ell & 1
\end{matrix}\right)\left(\begin{matrix}2^{j/2} & 0\\
0 & 2^{j}
\end{matrix}\right)=\left(\begin{matrix}2^{j/2}\ell & 2^{j}\\
2^{j/2} & 0
\end{matrix}\right)
\]
and
\[
B_{j,\ell,1}^{T}=\left[\left(\begin{matrix}0 & 1\\
1 & 0
\end{matrix}\right)\left(\begin{matrix}2^{j} & 0\\
0 & 2^{j/2}
\end{matrix}\right)\left(\begin{matrix}1 & 0\\
\ell & 1
\end{matrix}\right)\right]^{T}=\left[\left(\begin{matrix}0 & 2^{j/2}\\
2^{j} & 0
\end{matrix}\right)\left(\begin{matrix}1 & 0\\
\ell & 1
\end{matrix}\right)\right]^{T}=\left(\begin{matrix}2^{j/2}\ell & 2^{j/2}\\
2^{j} & 0
\end{matrix}\right)^{T}=R\cdot S_{\ell}^{T}\cdot\widetilde{A}_{2^{j}}.
\]
Consequently, we also get
\[
\gamma^{\left[\left(j,\ell,1\right),k\right]}=L_{\delta\cdot\left[R\cdot S_{\ell}^{T}\cdot\widetilde{A}_{2^{j}}\right]^{-1}k}\:\gamma^{\left[\left(j,\ell,1\right)\right]}=2^{\frac{3}{4}j}\cdot\psi\left(R\cdot S_{\ell}^{T}\cdot\widetilde{A}_{2^{j}}\mybullet-\delta k\right)=2^{\frac{3}{4}j}\cdot\psi\left(R\cdot S_{\ell}^{T}\cdot\widetilde{A}_{2^{j}}\mybullet-\delta RRk\right)=\theta_{j,\ell,Rk}
\]
for arbitrary $\left(j,\ell,1\right)\in V_{0}$ and $k\in\Z^{2}$.
Since $\Z^{2}\to\Z^{2},k\mapsto Rk$ is bijective, this implies 
\[
{\rm SH}\left(\varphi,\psi,\theta;\,\delta\right)={\rm SH}_{1/2}\left(\varphi,\psi;\,\delta\right)\text{ up to a reordering in the translation variable }k\qquad\text{ if }\text{\ensuremath{\theta}}=\psi\circ R.\qedhere
\]
\end{rem}
We now want to transfer Theorems \ref{thm:NicelySimplifiedAlphaShearletFrameConditions}
and \ref{thm:ReallyNiceShearletAtomicDecompositionConditions} to
the setting of the \emph{unconnected} $\alpha$-shearlet covering.
The link between the connected and the unconnected setting is provided
by the following lemma:
\begin{lem}
\label{lem:MEstimate}With $\varrho$, $\varrho_{0}$ as in equation
\eqref{eq:MotherShearletMainEstimate}, set $\widetilde{\varrho}_{0}:=\varrho_{0}$,
as well as $\widetilde{\varrho}_{v}:=\varrho$ for $v\in V_{0}$.
Moreover, set
\[
\widetilde{M}_{r,v}^{(0)}:=\left(\frac{u_{r}^{s}}{u_{v}^{s}}\right)^{\tau}\left(1\!+\!\left\Vert B_{r}^{-1}B_{v}\right\Vert \right)^{\sigma}\left(\left|\det\smash{B_{v}}\right|^{-1}\cdot\int_{W_{v}^{\left(\alpha\right)}}\widetilde{\varrho}_{r}\left(B_{r}^{-1}\xi\right)\d\xi\right)^{\!\tau}
\]
for $v,r\in V$. Then we have
\[
\widetilde{M}_{r,v}^{\left(0\right)}\leq2^{\tau}\cdot M_{\iota_{1}\left(r\right),\iota_{1}\left(v\right)}^{\left(0\right)}
\]
for all $v,r\in V$, where $M_{\iota_{1}\left(r\right),\iota_{1}\left(v\right)}^{\left(0\right)}$
is as in Lemma \ref{lem:MainShearletLemma}.
\end{lem}
\begin{proof}
First of all, recall 
\[
W_{v}'=\begin{cases}
\vphantom{\sum_{j}}U_{\left(-1,1\right)}^{\left(3^{-1},3\right)}\cup\left[\vphantom{U^{\left(\gamma\right)}}-\smash{U_{\left(-1,1\right)}^{\left(3^{-1},3\right)}}\right]=Q_{\iota_{1}\left(v\right)}'\cup\left[\vphantom{U^{\left(\gamma\right)}}-\smash{Q_{\iota_{1}\left(v\right)}'}\right], & \text{if }v\in V_{0},\\
\left(-1,1\right)^{2}=\left(-1,1\right)^{2}\cup\left[-\left(-1,1\right)^{2}\right]=Q_{\iota_{1}\left(v\right)}'\cup\left[\vphantom{U^{\left(\gamma\right)}}-\smash{Q_{\iota_{1}\left(v\right)}'}\right], & \text{if }v=0
\end{cases}
\]
and $B_{v}=T_{\iota_{1}\left(v\right)}$, as well as $u_{v}=w_{\iota_{1}\left(v\right)}$
and $\widetilde{\varrho}_{v}=\varrho_{\iota_{1}\left(v\right)}$ for
all $v\in V$. Thus,
\begin{align*}
\widetilde{M}_{r,v}^{\left(0\right)} & =\left(\frac{u_{r}^{s}}{u_{v}^{s}}\right)^{\tau}\cdot\left(1+\left\Vert B_{r}^{-1}B_{v}\right\Vert \right)^{\sigma}\cdot\left(\left|\det\smash{B_{v}}\right|^{-1}\cdot\int_{W_{v}^{\left(\alpha\right)}}\widetilde{\varrho}_{r}\left(B_{r}^{-1}\xi\right)\d\xi\right)^{\tau}\\
\left({\scriptstyle \text{with }\zeta=B_{v}^{-1}\xi}\right) & =\left(\frac{w_{\iota_{1}\left(r\right)}^{s}}{w_{\iota_{1}\left(v\right)}^{s}}\right)^{\tau}\cdot\left(1+\left\Vert T_{\iota_{1}\left(r\right)}^{-1}T_{\iota_{1}\left(v\right)}\right\Vert \right)^{\sigma}\cdot\left(\int_{W_{v}'}\widetilde{\varrho}_{r}\left(B_{r}^{-1}B_{v}\zeta\right)\d\zeta\right)^{\tau}\\
 & =\left(\frac{w_{\iota_{1}\left(r\right)}^{s}}{w_{\iota_{1}\left(v\right)}^{s}}\right)^{\tau}\cdot\left(1+\left\Vert T_{\iota_{1}\left(r\right)}^{-1}T_{\iota_{1}\left(v\right)}\right\Vert \right)^{\sigma}\cdot\left(\int_{Q_{\iota_{1}\left(v\right)}'\cup\left[-Q_{\iota_{1}\left(v\right)}'\right]}\varrho_{\iota_{1}\left(r\right)}\left(T_{\iota_{1}\left(r\right)}^{-1}T_{\iota_{1}\left(v\right)}\zeta\right)\d\zeta\right)^{\tau}\\
 & \leq\left(\frac{w_{\iota_{1}\left(r\right)}^{s}}{w_{\iota_{1}\left(v\right)}^{s}}\right)^{\tau}\cdot\left(1+\left\Vert T_{\iota_{1}\left(r\right)}^{-1}T_{\iota_{1}\left(v\right)}\right\Vert \right)^{\sigma}\\
 & \phantom{\leq}\qquad\cdot\left(\int_{Q_{\iota_{1}\left(v\right)}'}\varrho_{\iota_{1}\left(r\right)}\left(T_{\iota_{1}\left(r\right)}^{-1}T_{\iota_{1}\left(v\right)}\zeta\right)\d\zeta+\int_{-Q_{\iota_{1}\left(v\right)}'}\varrho_{\iota_{1}\left(r\right)}\left(T_{\iota_{1}\left(r\right)}^{-1}T_{\iota_{1}\left(v\right)}\zeta\right)\d\zeta\right)^{\tau}\\
\left({\scriptstyle \text{since }\varrho_{\iota_{1}\left(r\right)}\left(-\xi\right)=\varrho_{\iota_{1}\left(r\right)}\left(\xi\right)}\right) & =\left(\frac{w_{\iota_{1}\left(r\right)}^{s}}{w_{\iota_{1}\left(v\right)}^{s}}\right)^{\tau}\cdot\left(1+\left\Vert T_{\iota_{1}\left(r\right)}^{-1}T_{\iota_{1}\left(v\right)}\right\Vert \right)^{\sigma}\cdot\left(2\cdot\int_{Q_{\iota_{1}\left(v\right)}'}\varrho_{\iota_{1}\left(r\right)}\left(T_{\iota_{1}\left(r\right)}^{-1}T_{\iota_{1}\left(v\right)}\zeta\right)\d\zeta\right)^{\tau}\\
\left({\scriptstyle \text{with }\xi=T_{\iota_{1}\left(v\right)}\zeta}\right) & =2^{\tau}\cdot\left(\frac{w_{\iota_{1}\left(r\right)}^{s}}{w_{\iota_{1}\left(v\right)}^{s}}\right)^{\tau}\cdot\left(1+\left\Vert T_{\iota_{1}\left(r\right)}^{-1}T_{\iota_{1}\left(v\right)}\right\Vert \right)^{\sigma}\cdot\left(\left|\det T_{\iota_{1}\left(v\right)}\right|^{-1}\cdot\int_{S_{\iota_{1}\left(v\right)}^{\left(\alpha\right)}}\varrho_{\iota_{1}\left(r\right)}\left(T_{\iota_{1}\left(r\right)}^{-1}\xi\right)\d\xi\right)^{\tau}\\
 & =2^{\tau}\cdot M_{\iota_{1}\left(r\right),\iota_{1}\left(v\right)}^{\left(0\right)}.\qedhere
\end{align*}
\end{proof}
Since the map $\iota_{1}:V\to I$ is injective, Lemma \ref{lem:MEstimate}
implies
\[
\max\left\{ \left(\sup_{v\in V}\,\sum_{r\in V}\widetilde{M}_{r,v}^{\left(0\right)}\right)^{1/\tau},\,\left(\sup_{r\in V}\,\sum_{v\in V}\widetilde{M}_{r,v}^{\left(0\right)}\right)^{1/\tau}\right\} \leq2\cdot\max\left\{ \sup_{i\in I}\,\sum_{j\in I}M_{j,i}^{\left(0\right)},\,\sup_{j\in I}\,\sum_{i\in I}M_{j,i}^{\left(0\right)}\right\} .
\]
Then, recalling Lemma \ref{lem:UnconnectedAlphaShearletSmoothnessIsBoring}
and using \emph{precisely} the same arguments as for proving Theorems
\ref{thm:NicelySimplifiedAlphaShearletFrameConditions} and \ref{thm:ReallyNiceShearletAtomicDecompositionConditions},
one can prove the following two theorems:
\begin{thm}
\label{thm:NicelySimplifiedUnconnectedAlphaShearletFrameConditions}Theorem
\ref{thm:NicelySimplifiedAlphaShearletFrameConditions} remains essentially
valid if the family $\widetilde{{\rm SH}}_{\alpha,\varphi,\psi,\delta}^{\left(\pm1\right)}$
is replaced by the $\alpha$-shearlet system
\[
{\rm SH}_{\alpha}\left(\smash{\widetilde{\varphi},\widetilde{\psi}};\,\delta\right)=\left(L_{\delta\cdot B_{v}^{-T}k}\:\widetilde{\gamma^{\left[v\right]}}\right)_{v\in V,\,k\in\Z^{2}}\quad\text{ with }\quad\gamma^{\left[v\right]}:=\begin{cases}
\left|\det\smash{B_{v}}\right|^{1/2}\cdot\left(\psi\circ B_{v}^{T}\right), & \text{if }v\in V_{0},\\
\varphi, & \text{if }v=0,
\end{cases}
\]
where $\widetilde{\varphi}\left(x\right)=\varphi\left(-x\right)$
and $\widetilde{\psi}\left(x\right)=\psi\left(-x\right)$. The only
two necessary changes are the following:

\begin{enumerate}
\item The assumption $\widehat{\psi}\left(\xi\right)\neq0$ for $\xi=\left(\xi_{1},\xi_{2}\right)\in\R^{2}$
with $\xi_{1}\in\left[3^{-1},3\right]$ and $\left|\xi_{2}\right|\leq\left|\xi_{1}\right|$
has to be replaced by
\[
\widehat{\psi}\left(\xi\right)\neq0\text{ for }\xi=\left(\xi_{1},\xi_{2}\right)\in\R^{2}\text{ with }\frac{1}{3}\leq\left|\xi_{1}\right|\leq3\text{ and }\left|\xi_{2}\right|\leq\left|\xi_{1}\right|.
\]
\item For the definition of the analysis operator $A^{\left(\delta\right)}$,
the convolution $\gamma^{\left[v\right]}\ast f$ has to be defined
as in equation \eqref{eq:SpecialConvolutionDefinition}, but using
a regular partition of unity $\left(\varphi_{v}\right)_{v\in V}$
for $\CalS_{u}^{\left(\alpha\right)}$, i.e.,
\[
\left(\gamma^{\left[v\right]}\ast f\right)\left(x\right)=\sum_{\ell\in V}\Fourier^{-1}\left(\widehat{\gamma^{\left[v\right]}}\cdot\varphi_{\ell}\cdot\widehat{f}\:\right)\left(x\right)\qquad\forall x\in\R^{\dimension},
\]
where the series converges normally in $L^{\infty}\left(\R^{2}\right)$
and thus absolutely and uniformly, for all $f\in\mathscr{S}_{\alpha,s}^{p,q}\left(\R^{2}\right)$.
For a more convenient expression for this convolution—at least for
$f\in L^{2}\left(\R^{2}\right)$—see Lemma \ref{lem:SpecialConvolutionClarification}
below.\qedhere
\end{enumerate}
\end{thm}

\begin{thm}
\label{thm:ReallyNiceUnconnectedShearletAtomicDecompositionConditions}Theorem
\ref{thm:ReallyNiceShearletAtomicDecompositionConditions} remains
essentially valid if the family ${\rm SH}_{\varphi,\psi,\delta}^{\left(\pm1\right)}$
is replaced by the $\alpha$-shearlet system
\[
{\rm SH}_{\alpha}\left(\varphi,\psi;\,\delta\right)=\left(L_{\delta\cdot B_{v}^{-T}k}\:\gamma^{\left[v\right]}\right)_{v\in V,\,k\in\Z^{2}}\quad\text{ with }\quad\gamma^{\left[v\right]}:=\begin{cases}
\left|\det\smash{B_{v}}\right|^{1/2}\cdot\left(\psi\circ B_{v}^{T}\right), & \text{if }v\in V_{0},\\
\varphi, & \text{if }v=0.
\end{cases}
\]
The only necessary change is that the assumption $\widehat{\psi}\left(\xi\right)\neq0$
for $\xi=\left(\xi_{1},\xi_{2}\right)\in\R^{2}$ with $\xi_{1}\in\left[3^{-1},3\right]$
and $\left|\xi_{2}\right|\leq\left|\xi_{1}\right|$ has to be replaced
by
\[
\widehat{\psi}\left(\xi\right)\neq0\text{ for }\xi=\left(\xi_{1},\xi_{2}\right)\in\R^{2}\text{ with }\frac{1}{3}\leq\left|\xi_{1}\right|\leq3\text{ and }\left|\xi_{2}\right|\leq\left|\xi_{1}\right|.\qedhere
\]
\end{thm}
\begin{rem}
\label{rem:NiceTensorConditionsForUnconnectedCovering}With the exact
same reasoning, one can also show that Corollaries \ref{cor:ReallyNiceAlphaShearletTensorAtomicDecompositionConditions}
and \ref{cor:ReallyNiceAlphaShearletTensorBanachFrameConditions}
remain valid with the obvious changes. Again, one now has to require
\[
\widehat{\psi_{1}}\left(\xi\right)\neq0\text{ for }\frac{1}{3}\leq\left|\xi\right|\leq3.
\]
instead of $\widehat{\psi_{1}}\left(\xi\right)\neq0$ for $\xi\in\left[3^{-1},3\right]$.
\end{rem}
The one remaining limitation of Theorems \ref{thm:NicelySimplifiedAlphaShearletFrameConditions}
and \ref{thm:NicelySimplifiedUnconnectedAlphaShearletFrameConditions}
is their somewhat strange definition of the convolution $\left(\gamma^{\left[i\right]}\ast f\right)\left(x\right)$.
The following lemma makes this definition more concrete, under the
\emph{assumption} that we already know $f\in L^{2}\left(\R^{2}\right)$.
For general $f\in\mathscr{S}_{\alpha,s}^{p,q}\left(\R^{2}\right)$,
this need not be the case, but for suitable values of $p,q,s$, we
have $\mathscr{S}_{\alpha,s}^{p,q}\left(\R^{2}\right)\hookrightarrow L^{2}\left(\R^{2}\right)$,
as we will see in Theorem \ref{thm:AnalysisAndSynthesisSparsityAreEquivalent}.
\begin{lem}
\label{lem:SpecialConvolutionClarification}Let $\left(\varphi_{i}\right)_{i\in I}$
be a regular partition of unity subordinate to some almost structured
covering $\CalQ=\left(Q_{i}\right)_{i\in I}$ of $\R^{\dimension}$.
Assume that $\gamma\in L^{1}\left(\R^{\dimension}\right)\cap L^{2}\left(\R^{\dimension}\right)$
with $\widehat{\gamma}\in C^{\infty}\left(\R^{\dimension}\right)$,
where all partial derivatives of $\widehat{\gamma}$ are polynomially
bounded. Let $f\in L^{2}\left(\R^{\dimension}\right)\hookrightarrow\Schwartz'\left(\R^{\dimension}\right)\hookrightarrow Z'\left(\R^{\dimension}\right)$
be arbitrary. Then we have
\[
\sum_{\ell\in I}\Fourier^{-1}\left(\widehat{\gamma}\cdot\varphi_{\ell}\cdot\widehat{f}\right)\left(x\right)=\left\langle f,\,L_{x}\widetilde{\gamma}\right\rangle \qquad\forall x\in\R^{\dimension},
\]
where $\widetilde{\gamma}\left(x\right)=\gamma\left(-x\right)$ and
where $\left\langle f,g\right\rangle =\int_{\R^{\dimension}}f\left(x\right)\cdot g\left(x\right)\d x$.
\end{lem}
\begin{proof}
In the expression $\Fourier^{-1}\left(\widehat{\gamma}\cdot\varphi_{\ell}\cdot\widehat{f}\right)\left(x\right)$,
the inverse Fourier transform is the inverse Fourier transform of
the compactly supported, tempered distribution $\widehat{\gamma}\cdot\varphi_{\ell}\cdot\widehat{f}\in\Schwartz'\left(\R^{\dimension}\right)$.
But by the Paley-Wiener theorem (see e.g.\@ \cite[Theorem 7.23]{RudinFA}),
the tempered distribution $\Fourier^{-1}\left(\widehat{\gamma}\cdot\varphi_{\ell}\cdot\widehat{f}\right)$
is given by (integration against) a (uniquely determined) smooth function,
whose value at $x\in\R^{\dimension}$ we denote by $\Fourier^{-1}\left(\widehat{\gamma}\cdot\varphi_{\ell}\cdot\widehat{f}\right)\left(x\right)$.
Precisely, we have
\[
\Fourier^{-1}\left(\widehat{\gamma}\cdot\varphi_{\ell}\cdot\widehat{f}\right)\left(x\right)=\left\langle \widehat{\gamma}\cdot\widehat{f},\,\varphi_{\ell}\cdot e^{2\pi i\left\langle x,\mybullet\right\rangle }\right\rangle _{\DistributionSpace{\R^{\dimension}},\TestFunctionSpace{\R^{\dimension}}}=\int_{\R^{\dimension}}\widehat{\gamma}\left(\xi\right)\cdot\widehat{f}\left(\xi\right)\cdot e^{2\pi i\left\langle x,\xi\right\rangle }\cdot\varphi_{\ell}\left(\xi\right)\d\xi.
\]
But since $\CalQ$ is an admissible covering of $\R^{\dimension}$
and since $\left(\varphi_{\ell}\right)_{\ell\in I}$ is a regular
partition of unity subordinate to $\CalQ$, we have
\begin{align*}
\sum_{\ell\in I}\left|\widehat{\gamma}\left(\xi\right)\cdot\widehat{f}\left(\xi\right)\cdot e^{2\pi i\left\langle x,\xi\right\rangle }\cdot\varphi_{\ell}\left(\xi\right)\right| & \leq\left|\widehat{\gamma}\left(\xi\right)\cdot\smash{\widehat{f}}\left(\xi\right)\right|\cdot\sum_{\ell\in I}\left|\varphi_{\ell}\left(\xi\right)\right|\\
 & \leq\sup_{\ell\in I}\left\Vert \varphi_{\ell}\right\Vert _{\sup}\cdot\left|\widehat{\gamma}\left(\xi\right)\cdot\smash{\widehat{f}}\left(\xi\right)\right|\cdot\sum_{\ell\in I}\Indicator_{Q_{\ell}}\left(\xi\right)\\
 & \leq N_{\CalQ}\cdot\sup_{\ell\in I}\left\Vert \varphi_{\ell}\right\Vert _{\sup}\cdot\left|\widehat{\gamma}\left(\xi\right)\cdot\smash{\widehat{f}}\left(\xi\right)\right|\in L^{1}\left(\R^{\dimension}\right),
\end{align*}
since $\widehat{\gamma},\widehat{f}\in L^{2}\left(\R^{\dimension}\right)$.
Since we also have $\sum_{\ell\in I}\varphi_{\ell}\equiv1$ on $\R^{\dimension}$,
we get by the dominated convergence theorem that
\begin{align*}
\sum_{\ell\in I}\Fourier^{-1}\left(\widehat{\gamma}\cdot\varphi_{\ell}\cdot\widehat{f}\right)\left(x\right) & =\int_{\R^{\dimension}}\widehat{\gamma}\left(\xi\right)\cdot\widehat{f}\left(\xi\right)\cdot e^{2\pi i\left\langle x,\xi\right\rangle }\cdot\sum_{\ell\in I}\varphi_{\ell}\left(\xi\right)\d\xi\\
 & =\int_{\R^{\dimension}}\widehat{\gamma}\left(\xi\right)\cdot\widehat{f}\left(\xi\right)\cdot e^{2\pi i\left\langle x,\xi\right\rangle }\d\xi=\Fourier^{-1}\left(\smash{\widehat{\gamma}\cdot\widehat{f}}\right)\left(x\right),
\end{align*}
where $\Fourier^{-1}\left(\smash{\widehat{\gamma}\cdot\widehat{f}}\right)\in L^{2}\left(\R^{\dimension}\right)\cap C_{0}\left(\R^{\dimension}\right)$
by the Riemann-Lebesgue Lemma and Plancherel's theorem, because of
$\widehat{\gamma}\cdot\widehat{f}\in L^{1}\left(\R^{\dimension}\right)\cap L^{2}\left(\R^{\dimension}\right)$.
But Young's inequality shows $\gamma\ast f\in L^{2}\left(\R^{\dimension}\right)$,
while the convolution theorem yields $\Fourier\left[\gamma\ast f\right]=\widehat{\gamma}\cdot\widehat{f}$.
Hence, $\gamma\ast f=\Fourier^{-1}\left(\smash{\widehat{\gamma}\cdot\widehat{f}}\right)$
almost everywhere. But both sides of the identity are continuous functions,
since the convolution of two $L^{2}$ functions is continuous. Thus,
the equality holds everywhere, so that we finally get
\[
\sum_{\ell\in I}\Fourier^{-1}\left(\widehat{\gamma}\cdot\varphi_{\ell}\cdot\widehat{f}\right)\left(x\right)=\Fourier^{-1}\left(\smash{\widehat{\gamma}\cdot\widehat{f}}\right)\left(x\right)=\left(\gamma\ast f\right)\left(x\right)=\int_{\R^{\dimension}}f\left(y\right)\cdot\gamma\left(x-y\right)\d y=\left\langle f,\,L_{x}\widetilde{\gamma}\right\rangle .\qedhere
\]
\end{proof}
We close this section with a theorem that justifies the title of the
paper: It formally encodes the fact that \emph{analysis sparsity is
equivalent to synthesis sparsity} for (suitable) $\alpha$-shearlet
systems.
\begin{thm}
\label{thm:AnalysisAndSynthesisSparsityAreEquivalent}Let $\alpha\in\left[0,1\right]$,
$\varepsilon,p_{0}\in\left(0,1\right]$ and $s^{\left(0\right)}\geq0$
be arbitrary. Assume that $\varphi,\psi\in L^{1}\left(\R^{2}\right)$
satisfy the assumptions of Theorems \ref{thm:NicelySimplifiedUnconnectedAlphaShearletFrameConditions}
and \ref{thm:ReallyNiceUnconnectedShearletAtomicDecompositionConditions}
with $q_{0}=p_{0}$ and $s_{0}=0$, as well as $s_{1}=s^{\left(0\right)}+\left(1+\alpha\right)\left(p_{0}^{-1}-2^{-1}\right)$.
For $\delta>0$, denote by ${\rm SH}_{\alpha}\left(\varphi,\psi;\delta\right)=\left(\gamma^{\left[v,k,\delta\right]}\right)_{v\in V,\,k\in\Z^{2}}$
the $\alpha$-shearlet system generated by $\varphi,\psi$, as in
Definition \ref{def:AlphaShearletSystem}.

Then there is some $\delta_{0}\in\left(0,1\right]$ with the following
property: For all $p\in\left[p_{0},2\right]$ and all $s\in\left[0,s^{\left(0\right)}\right]$,
we have
\begin{align*}
\mathscr{S}_{\alpha,s+\left(1+\alpha\right)\left(p^{-1}-2^{-1}\right)}^{p,p}\left(\R^{2}\right) & =\left\{ f\in L^{2}\left(\R^{2}\right)\with\left(u_{v}^{s}\cdot\left\langle f,\,\smash{\gamma^{\left[v,k,\delta\right]}}\right\rangle _{L^{2}}\vphantom{\gamma^{\left[v,k\right]}}\right)_{v\in V,\,k\in\Z^{2}}\in\ell^{p}\left(V\times\Z^{2}\right)\right\} \\
 & =\left\{ \sum_{\left(v,k\right)\in V\times\Z^{2}}c_{k}^{\left(v\right)}\cdot\gamma^{\left[v,k,\delta\right]}\with\left(u_{v}^{s}\cdot\smash{c_{k}^{\left(v\right)}}\right)_{v\in V,\,k\in\Z^{2}}\in\ell^{p}\left(V\times\Z^{2}\right)\right\} ,
\end{align*}
as long as $0<\delta\leq\delta_{0}$. Here, the weight $u=\left(u_{v}\right)_{v\in V}$
is as in Definition \ref{def:UnconnectedAlphaShearletCovering}, i.e.,
$u_{n,m,\delta}=2^{n}$ and $u_{0}=1$.

In fact, for $f\in\mathscr{S}_{\alpha,s+\left(1+\alpha\right)\left(p^{-1}-2^{-1}\right)}^{p,p}\left(\R^{2}\right)$,
we even have a (quasi)-norm equivalence
\begin{align*}
\left\Vert f\right\Vert _{\mathscr{S}_{\alpha,s+\left(1+\alpha\right)\left(p^{-1}-2^{-1}\right)}^{p,p}} & \asymp\left\Vert \left(u_{v}^{s}\cdot\left\langle f,\,\smash{\gamma^{\left[v,k,\delta\right]}}\right\rangle _{L^{2}}\vphantom{\gamma^{\left[v,k,\delta\right]}}\right)_{v\in V,\,k\in\Z^{2}}\right\Vert _{\ell^{p}}\\
 & \asymp\inf\left\{ \!\left\Vert \left(u_{v}^{s}\cdot\smash{c_{k}^{\left(v\right)}}\right)_{v\in V,\,k\in\Z^{2}}\right\Vert _{\ell^{p}}\with f=\!\!\!\sum_{\left(v,k\right)\in V\times\Z^{2}}\!c_{k}^{\left(v\right)}\cdot\gamma^{\left[v,k,\delta\right]}\text{ with uncond. conv. in }L^{2}\left(\R^{2}\right)\right\} \!.
\end{align*}
In particular, $\mathscr{S}_{\alpha,s+\left(1+\alpha\right)\left(p^{-1}-2^{-1}\right)}^{p,p}\left(\R^{2}\right)\hookrightarrow L^{2}\left(\R^{2}\right)$
and ${\rm SH}_{\alpha}\left(\varphi,\psi;\delta\right)$ is a frame
for $L^{2}\left(\R^{2}\right)$.
\end{thm}
\begin{rem*}
As one advantage of the decomposition space point of view, we observe
that $\mathscr{S}_{\alpha,s}^{p,q}\left(\R^{2}\right)$ is \emph{easily}
seen to be translation invariant, while this is not so easy to see
in the characterization via analysis or synthesis sparsity in terms
of a discrete $\alpha$-shearlet system.
\end{rem*}
\begin{proof}
We start with a few preparatory definitions and observations. For
brevity, we set
\begin{equation}
\left\Vert f\right\Vert _{\ast,p,s,\delta}:=\inf\left\{ \left\Vert \left(u_{v}^{s}\cdot\smash{c_{k}^{\left(v\right)}}\right)_{v\in V,\,k\in\Z^{2}}\right\Vert _{\ell^{p}}\with f=\!\!\sum_{\left(v,k\right)\in V\times\Z^{2}}c_{k}^{\left(v\right)}\cdot\gamma^{\left[v,k,\delta\right]}\text{ with uncond. conv. in }L^{2}\left(\R^{2}\right)\right\} \label{eq:AnalysisSynthesisSparsitySpecialNormDefinition}
\end{equation}
for $f\in\mathscr{S}_{\alpha,s+\left(1+\alpha\right)\left(p^{-1}-2^{-1}\right)}^{p,p}\left(\R^{2}\right)$
and $s\in\left[0,s^{\left(0\right)}\right]$, as well as $p\in\left[p_{0},2\right]$.

Next, our assumptions entail that $\varphi,\psi$ satisfy the assumptions
of Theorem \ref{thm:NicelySimplifiedUnconnectedAlphaShearletFrameConditions}
(and thus equation \eqref{eq:ShearletFrameFourierDecayCondition})
for $s_{0}=0$ and $s_{1}=s^{\left(0\right)}+\left(1+\alpha\right)\left(p_{0}^{-1}-2^{-1}\right)\geq0$.
But this implies (in the notation of Theorem \ref{thm:NicelySimplifiedAlphaShearletFrameConditions})
that $K,H,M_{2}\geq2+\varepsilon$. Hence,
\[
\left(1+\left|\xi_{1}\right|\right)^{-\left(M_{2}+1\right)}\left(1+\left|\xi_{2}\right|\right)^{-\left(K+1\right)}\leq\left[\left(1+\left|\xi_{1}\right|\right)\left(1+\left|\xi_{2}\right|\right)\right]^{-\left(2+\varepsilon\right)}\leq\left(1+\left|\xi\right|\right)^{-\left(2+\varepsilon\right)}\in L^{1}\left(\R^{2}\right).
\]
Therefore, equation \eqref{eq:ShearletFrameFourierDecayCondition}
entails $\widehat{\varphi},\widehat{\psi}\in L^{1}\left(\R^{2}\right)$,
so that Fourier inversion yields $\varphi,\psi\in L^{1}\left(\R^{2}\right)\cap C_{0}\left(\R^{2}\right)\hookrightarrow L^{2}\left(\R^{2}\right)$.
Consequently, $\gamma^{\left[v\right]}\in L^{1}\left(\R^{2}\right)\cap L^{2}\left(\R^{2}\right)$
for all $v\in V$, which will be important for our application of
Lemma \ref{lem:SpecialConvolutionClarification} later in the proof.

Finally, for $g:\R^{2}\to\Compl$, set $g^{\ast}:\R^{2}\to\Compl,x\mapsto\overline{g\left(-x\right)}$.
For $g\in L^{1}\left(\R^{2}\right)$, we then have $\widehat{g^{\ast}}\left(\xi\right)=\overline{\widehat{g}\left(\xi\right)}$
for all $\xi\in\R^{2}$. Therefore, in case of $g\in C^{1}\left(\R^{2}\right)$
with $g,\nabla g\in L^{1}\left(\R^{2}\right)\cap L^{\infty}\left(\R^{2}\right)$
and with $\widehat{g}\in C^{\infty}\left(\R^{2}\right)$, this implies
that $g^{\ast}$ satisfies the same properties and that $\left|\partial^{\theta}\widehat{g^{\ast}}\right|=\left|\partial^{\theta}\widehat{g}\right|$
for all $\theta\in\N_{0}^{2}$. These considerations easily show that
since $\varphi,\psi$ satisfy the assumptions of Theorem \ref{thm:NicelySimplifiedUnconnectedAlphaShearletFrameConditions}
(with $q_{0}=p_{0}$ and $s_{0}=0$, as well as $s_{1}=s^{\left(0\right)}+\left(1+\alpha\right)\left(p_{0}^{-1}-2^{-1}\right)$),
so do $\varphi^{\ast},\psi^{\ast}$.

Thus, Theorem \ref{thm:NicelySimplifiedUnconnectedAlphaShearletFrameConditions}
yields a constant $\delta_{1}\in\left(0,1\right]$ such that the $\alpha$-shearlet
system ${\rm SH}_{\alpha}\left(\overline{\varphi},\overline{\psi};\delta\right)={\rm SH}_{\alpha}\left(\smash{\widetilde{\varphi^{\ast}}},\smash{\widetilde{\psi^{\ast}}};\delta\right)$
forms a Banach frame for $\mathscr{S}_{\alpha,s}^{p,q}\left(\R^{2}\right)$,
for all $p,q\in\left[p_{0},\infty\right]$ and all $s\in\R$ with
$0\leq s\leq s^{\left(0\right)}+\left(1+\alpha\right)\left(p_{0}^{-1}-2^{-1}\right)$,
as long as $0<\delta\leq\delta_{1}$. Likewise, Theorem \ref{thm:ReallyNiceUnconnectedShearletAtomicDecompositionConditions}
yields a constant $\delta_{2}\in\left(0,1\right]$ such that ${\rm SH}_{\alpha}\left(\varphi,\psi;\delta\right)$
yields an atomic decomposition of $\mathscr{S}_{\alpha,s}^{p,q}\left(\R^{2}\right)$
for the same range of parameters, as long as $0<\delta\leq\delta_{2}$.
Now, let us set $\delta_{0}:=\min\left\{ \delta_{1},\delta_{2}\right\} \in\left(0,1\right]$.

\medskip{}

Let $p\in\left[p_{0},2\right]$ and $s\in\left[0,s^{\left(0\right)}\right]$
be arbitrary and set $s^{\natural}:=s+\left(1+\alpha\right)\left(p^{-1}-2^{-1}\right)$.
It is not hard to see directly from Definition \ref{def:CoefficientSpace}—and
because of $\left|\det B_{v}\right|=u_{v}^{1+\alpha}$ for all $v\in V$—that
the quasi-norm of the coefficient space $C_{u^{s^{\natural}}}^{p,p}$
satisfies
\[
\left\Vert \left(\smash{c_{k}^{\left(v\right)}}\right)_{v\in V,k\in\Z^{2}}\right\Vert _{C_{u^{s^{\natural}}}^{p,p}}=\left\Vert \left(\left|\det B_{v}\right|^{\frac{1}{2}-\frac{1}{p}}\cdot u_{v}^{s^{\natural}}\cdot\left\Vert \left(\smash{c_{k}^{\left(v\right)}}\right)_{k\in\Z^{2}}\right\Vert _{\ell^{p}}\right)_{v\in V}\right\Vert _{\ell^{p}}=\left\Vert \left(u_{v}^{s}\cdot\smash{c_{k}^{\left(v\right)}}\right)_{v\in V,\,k\in\Z^{2}}\right\Vert _{\ell^{p}}\in\left[0,\infty\right]
\]
for arbitrary sequences $\left(\smash{c_{k}^{\left(v\right)}}\right)_{v\in V,k\in\Z^{2}}$,
and $C_{u^{s^{\natural}}}^{p,p}$ contains exactly those sequences
for which this (quasi)-norm is finite. Now, note because of $s\geq0$
and $p\leq2$ that $C_{u^{s^{\natural}}}^{p,p}\hookrightarrow\ell^{2}\left(V\times\Z^{2}\right)$,
since $u_{v}\geq1$ for all $v\in V$ and since $\ell^{p}\hookrightarrow\ell^{2}$. 

Next, note that we have 
\[
s_{0}=0\leq s\leq s^{\natural}\leq s^{\left(0\right)}+\left(1+\alpha\right)\left(p_{0}^{-1}-2^{-1}\right)=s_{1},
\]
so that ${\rm SH}_{\alpha}\left(\varphi,\psi;\delta\right)$ forms
an atomic decomposition of $\mathscr{S}_{\alpha,s^{\natural}}^{p,p}\left(\R^{2}\right)$
for all $0<\delta\leq\delta_{0}$. This means that the synthesis operator
\[
S^{\left(\delta\right)}:C_{u^{s^{\natural}}}^{p,p}\to\mathscr{S}_{\alpha,s^{\natural}}^{p,p}\left(\R^{2}\right),\left(\smash{c_{k}^{\left(v\right)}}\right)_{v\in V,\,k\in\Z^{2}}\mapsto\sum_{\left(v,k\right)\in V\times\Z^{2}}c_{k}^{\left(v\right)}\cdot\gamma^{\left[v,k,\delta\right]}
\]
is well-defined and bounded with unconditional convergence of the
series in $\mathscr{S}_{\alpha,s^{\natural}}^{p,p}\left(\R^{2}\right)$.
This implicitly uses that the synthesis operator $S^{\left(\delta\right)}$
as defined in Theorem \ref{thm:ReallyNiceShearletAtomicDecompositionConditions}
is bounded and satisfies $S^{\left(\delta\right)}\left(\delta_{v,k}\right)=\gamma^{\left[v,k,\delta\right]}$
for all $\left(v,k\right)\in V\times\Z^{2}$ and that we have $c=\left(\smash{c_{k}^{\left(v\right)}}\right)_{v\in V,\,k\in\Z^{2}}=\sum_{\left(v,k\right)\in V\times\Z^{2}}c_{k}^{\left(v\right)}\cdot\delta_{v,k}$
for all $c\in C_{u^{s^{\natural}}}^{p,p}$, with unconditional convergence
in $C_{u^{s^{\natural}}}^{p,p}$, since $p\leq2<\infty$. This immediately
yields
\begin{equation}
\Omega_{1}:=\left\{ \sum_{\left(v,k\right)\in V\times\Z^{2}}c_{k}^{\left(v\right)}\cdot\gamma^{\left[v,k,\delta\right]}\with\left(u_{v}^{s}\cdot\smash{c_{k}^{\left(v\right)}}\right)_{v\in V,\,k\in\Z^{2}}\in\ell^{p}\left(V\times\Z^{2}\right)\right\} ={\rm range}\left(\smash{S^{\left(\delta\right)}}\right)\subset\mathscr{S}_{\alpha,s^{\natural}}^{p,p}\left(\R^{2}\right).\label{eq:AnalysisSynthesisSparsityEquivalentOmegaDefinition}
\end{equation}
Further, if $f\in\mathscr{S}_{\alpha,s^{\natural}}^{p,p}\left(\R^{2}\right)$
and if $c=\left(\smash{c_{k}^{\left(v\right)}}\right)_{v\in V,\,k\in\Z^{2}}$
is an arbitrary sequence satisfying $f=\sum_{\left(v,k\right)\in V\times\Z^{2}}c_{k}^{\left(v\right)}\cdot\gamma^{\left[v,k,\delta\right]}$
with unconditional convergence in $L^{2}\left(\R^{2}\right)$, there
are two cases:

\begin{casenv}
\item We have $\left\Vert \left(u_{v}^{s}\cdot\smash{c_{k}^{\left(v\right)}}\right)_{v\in V,\,k\in\Z^{2}}\right\Vert _{\ell^{p}}=\infty$.
In this case, $\left\Vert f\right\Vert _{\mathscr{S}_{\alpha,s^{\natural}}^{p,p}\left(\R^{2}\right)}\leq\vertiii{\smash{S^{\left(\delta\right)}}}\cdot\left\Vert \left(u_{v}^{s}\cdot\smash{c_{k}^{\left(v\right)}}\right)_{v\in V,\,k\in\Z^{2}}\right\Vert _{\ell^{p}}$
is trivial.
\item We have $\left\Vert \left(u_{v}^{s}\cdot\smash{c_{k}^{\left(v\right)}}\right)_{v\in V,\,k\in\Z^{2}}\right\Vert _{\ell^{p}}<\infty$.
In this case, we get $c\in C_{u^{s^{\natural}}}^{p,p}$ and $f=S^{\left(\delta\right)}c$.
Therefore, we see $\left\Vert f\right\Vert _{\mathscr{S}_{\alpha,s^{\natural}}^{p,p}\left(\R^{2}\right)}\leq\vertiii{\smash{S^{\left(\delta\right)}}}\cdot\left\Vert c\right\Vert _{C_{u^{s^{\natural}}}^{p,p}}=\vertiii{\smash{S^{\left(\delta\right)}}}\cdot\left\Vert \left(u_{v}^{s}\cdot\smash{c_{k}^{\left(v\right)}}\right)_{v\in V,\,k\in\Z^{2}}\right\Vert _{\ell^{p}}$.
\end{casenv}
All in all, we have thus established

\[
\left\Vert f\right\Vert _{\mathscr{S}_{\alpha,s^{\natural}}^{p,p}\left(\R^{2}\right)}\leq\vertiii{\smash{S^{\left(\delta\right)}}}\cdot\left\Vert f\right\Vert _{\ast,p,s,\delta}\qquad\forall f\in\mathscr{S}_{\alpha,s^{\natural}}^{p,p}\left(\R^{2}\right).
\]

Next, note that the considerations from the preceding paragraph with
the choice $p=2$ and $s=0$ also show that $S^{\left(\delta\right)}:\ell^{2}\left(V\times\Z^{2}\right)\to\mathscr{S}_{\alpha,0}^{2,2}\left(\R^{2}\right)$
is well-defined and bounded. But \cite[Lemma 6.10]{DecompositionEmbedding}
yields $\mathscr{S}_{\alpha,0}^{2,2}\left(\R^{2}\right)=L^{2}\left(\R^{2}\right)$
with equivalent norms. Since we saw above that $C_{u^{s^{\natural}}}^{p,p}\hookrightarrow\ell^{2}\left(V\times\Z^{2}\right)$
for all $p\leq2$ and $s\geq0$, this implies in particular that the
series defining $S^{\left(\delta\right)}c$ converges unconditionally
in $L^{2}\left(\R^{2}\right)$ for arbitrary $c\in C_{u^{s^{\natural}}}^{p,p}$,
for arbitrary $s\in\left[0,s^{\left(0\right)}\right]$ and $p\in\left[p_{0},2\right]$.

But from the atomic decomposition property of ${\rm SH}_{\alpha}\left(\varphi,\psi;\delta\right)$,
we also know that there is a bounded coefficient operator $C^{\left(\delta\right)}:\mathscr{S}_{\alpha,s^{\natural}}^{p,p}\left(\R^{2}\right)\to C_{u^{s^{\natural}}}^{p,p}$
satisfying $S^{\left(\delta\right)}\circ C^{\left(\delta\right)}=\identity_{\mathscr{S}_{\alpha,s^{\natural}}^{p,p}}$.
Thus, for arbitrary $f\in\mathscr{S}_{\alpha,s^{\natural}}^{p,p}\left(\R^{2}\right)$
and $e=\left(e_{v,k}\right)_{v\in V,k\in\Z^{2}}:=C^{\left(\delta\right)}f\in C_{u^{s^{\natural}}}^{p,p}$,
we have $f=S^{\left(\delta\right)}e=\sum_{\left(v,k\right)\in V\times\Z^{2}}e_{k}^{\left(v\right)}\cdot\gamma^{\left[v,k,\delta\right]}\in\Omega_{1}$,
where the series converges unconditionally in $L^{2}\left(\R^{2}\right)$
(and in $\mathscr{S}_{\alpha,s^{\natural}}^{p,p}\left(\R^{2}\right)$).
In particular, we get
\[
\left\Vert f\right\Vert _{\ast,p,s,\delta}\leq\left\Vert \left(u_{v}^{s}\cdot\smash{e_{k}^{\left(v\right)}}\right)_{v\in V,\,k\in\Z^{2}}\right\Vert _{\ell^{p}}=\left\Vert e\right\Vert _{C_{u^{s^{\natural}}}^{p,p}}\leq\vertiii{\smash{C^{\left(\delta\right)}}}\cdot\left\Vert f\right\Vert _{\mathscr{S}_{\alpha,s^{\natural}}^{p,p}}<\infty,
\]
as well as 
\begin{align*}
\left\Vert f\right\Vert _{L^{2}\left(\R^{2}\right)} & \lesssim\left\Vert f\right\Vert _{\mathscr{S}_{\alpha,0}^{2,2}}\leq\vertiii{\smash{S^{\left(\delta\right)}}}_{\ell^{2}\to\mathscr{S}_{\alpha,0}^{2,2}}\cdot\left\Vert e\right\Vert _{C_{u^{0}}^{2,2}}=\vertiii{\smash{S^{\left(\delta\right)}}}_{\ell^{2}\to\mathscr{S}_{\alpha,0}^{2,2}}\cdot\left\Vert e\right\Vert _{\ell^{2}}\\
 & \leq\vertiii{\smash{S^{\left(\delta\right)}}}_{\ell^{2}\to\mathscr{S}_{\alpha,0}^{2,2}}\cdot\left\Vert e\right\Vert _{C_{u^{s^{\natural}}}^{p,p}}\leq\vertiii{\smash{S^{\left(\delta\right)}}}_{\ell^{2}\to\mathscr{S}_{\alpha,0}^{2,2}}\cdot\vertiii{\smash{C^{\left(\delta\right)}}}\cdot\left\Vert f\right\Vert _{\mathscr{S}_{\alpha,s^{\natural}}^{p,p}}<\infty
\end{align*}
for all $f\in\mathscr{S}_{\alpha,s^{\natural}}^{p,p}\left(\R^{2}\right)$.
Up to now, we have thus shown $\mathscr{S}_{\alpha,s^{\natural}}^{p,p}\left(\R^{2}\right)=\Omega_{1}$
(with $\Omega_{1}$ as in equation \eqref{eq:AnalysisSynthesisSparsityEquivalentOmegaDefinition})
and $\left\Vert f\right\Vert _{\ast,p,s,\delta}\asymp\left\Vert f\right\Vert _{\mathscr{S}_{\alpha,s^{\natural}}^{p,p}}$
for all $f\in\mathscr{S}_{\alpha,s^{\natural}}^{p,p}\left(\R^{2}\right)$,
with $\left\Vert f\right\Vert _{\ast,p,s,\delta}$ as in equation
\eqref{eq:AnalysisSynthesisSparsitySpecialNormDefinition}. Finally,
we have also shown $\mathscr{S}_{\alpha,s^{\natural}}^{p,p}\left(\R^{2}\right)\hookrightarrow L^{2}\left(\R^{2}\right)$.

\medskip{}

Thus, it remains to show
\[
\Omega_{2}:=\left\{ f\in L^{2}\left(\R^{2}\right)\with\left(u_{v}^{s}\cdot\left\langle f,\,\smash{\gamma^{\left[v,k,\delta\right]}}\right\rangle _{L^{2}}\vphantom{\gamma^{\left[v,k,\delta\right]}}\right)_{v\in V,\,k\in\Z^{2}}\in\ell^{p}\left(V\times\Z^{2}\right)\right\} \overset{!}{=}\mathscr{S}_{\alpha,s^{\natural}}^{p,p}\left(\R^{2}\right),
\]
as well as $\left\Vert f\right\Vert _{\mathscr{S}_{\alpha,s^{\natural}}^{p,p}}\asymp\left\Vert \left(u_{v}^{s}\cdot\left\langle f,\,\smash{\gamma^{\left[v,k,\delta\right]}}\right\rangle _{L^{2}}\vphantom{\gamma^{\left[v,k,\delta\right]}}\right)_{v\in V,\,k\in\Z^{2}}\right\Vert _{\ell^{p}}$
for $f\in\mathscr{S}_{\alpha,s^{\natural}}^{p,p}\left(\R^{2}\right)$.
But Theorem \ref{thm:NicelySimplifiedUnconnectedAlphaShearletFrameConditions}
(applied with $\varphi^{\ast},\psi^{\ast}$ instead of $\varphi,\psi$,
see above) shows that the analysis operator
\[
A^{\left(\delta\right)}:\mathscr{S}_{\alpha,s^{\natural}}^{p,p}\left(\R^{2}\right)\to C_{u^{s^{\natural}}}^{p,p},f\mapsto\left[\left(\smash{\varrho^{\left[v\right]}}\ast f\right)\left(\delta\cdot B_{v}^{-T}k\right)\right]_{v\in V,\,k\in\Z^{2}}
\]
is well-defined and bounded, where (cf.\@ Theorem \ref{thm:NicelySimplifiedAlphaShearletFrameConditions}),
the family $\left(\varrho^{\left[v\right]}\right)_{v\in V}$ is given
by $\varrho^{\left[v\right]}=\left|\det B_{v}\right|^{1/2}\cdot\left(\psi^{\ast}\circ B_{v}^{T}\right)$
for $v\in V_{0}$ and by $\varrho^{\left[0\right]}=\varphi^{\ast}$.
Note that this yields $\widetilde{\varrho^{\left[v\right]}}=\overline{\gamma^{\left[v\right]}}$,
where the family $\left(\gamma^{\left[v\right]}\right)_{v\in V}$
is as in Definition \ref{def:AlphaShearletSystem}.

Now, since we already showed $\mathscr{S}_{\alpha,s^{\natural}}^{p,p}\left(\R^{2}\right)\hookrightarrow L^{2}\left(\R^{2}\right)$
and since $\varrho^{\left[v\right]}\in L^{1}\left(\R^{2}\right)\cap L^{2}\left(\R^{2}\right)$
for all $v\in V$, as we saw at the start of the proof, Lemma \ref{lem:SpecialConvolutionClarification}
yields
\[
\left(\smash{\varrho^{\left[v\right]}\ast f}\right)\left(\delta\cdot B_{v}^{-T}k\right)=\left\langle f,\,L_{\delta\cdot B_{v}^{-T}k}\:\widetilde{\varrho^{\left[v\right]}}\right\rangle =\left\langle f,\,L_{\delta\cdot B_{v}^{-T}k}\:\overline{\gamma^{\left[v\right]}}\right\rangle =\left\langle f,\,L_{\delta\cdot B_{v}^{-T}k}\:\gamma^{\left[v\right]}\right\rangle _{L^{2}}=\left\langle f,\,\gamma^{\left[v,k,\delta\right]}\right\rangle _{L^{2}}
\]
for all $f\in\mathscr{S}_{\alpha,s^{\natural}}^{p,p}\left(\R^{2}\right)$
and $\left(v,k\right)\in V\times\Z^{2}$. We thus see $\mathscr{S}_{\alpha,s^{\natural}}^{p,p}\left(\R^{2}\right)\subset\Omega_{2}$
and 
\[
\left\Vert \left(u_{v}^{s}\cdot\left\langle f,\,\smash{\gamma^{\left[v,k,\delta\right]}}\right\rangle _{L^{2}}\vphantom{\gamma^{\left[v,k,\delta\right]}}\right)_{v\in V,\,k\in\Z^{2}}\right\Vert _{\ell^{p}}=\left\Vert \smash{A^{\left(\delta\right)}}f\right\Vert _{C_{u^{s^{\natural}}}^{p,p}}\leq\vertiii{\smash{A^{\left(\delta\right)}}}\cdot\left\Vert f\right\Vert _{\mathscr{S}_{\alpha,s^{\natural}}^{p,p}}\qquad\forall f\in\mathscr{S}_{\alpha,s^{\natural}}^{p,p}\left(\R^{2}\right).
\]

Conversely, let $f\in\Omega_{2}$ be arbitrary, i.e., $f\in L^{2}\left(\R^{2}\right)$
with $\left(u_{v}^{s}\cdot\left\langle f,\,\smash{\gamma^{\left[v,k,\delta\right]}}\right\rangle _{L^{2}}\vphantom{\gamma^{\left[v,k,\delta\right]}}\right)_{v\in V,\,k\in\Z^{2}}\in\ell^{p}\left(V\times\Z^{2}\right)$.
This means $f\in L^{2}\left(\R^{2}\right)=\mathscr{S}_{\alpha,0}^{2,2}\left(\R^{2}\right)$
and $\left[\left(\varrho^{\left[v\right]}\ast f\right)\left(\delta\cdot B_{v}^{-T}k\right)\right]_{v\in V,\,k\in\Z^{2}}\in C_{u^{s^{\natural}}}^{p,p}$,
again by Lemma \ref{lem:SpecialConvolutionClarification}. Thus, the
consistency statement of Theorem \ref{thm:NicelySimplifiedAlphaShearletFrameConditions}
shows $f\in\mathscr{S}_{\alpha,s^{\natural}}^{p,p}\left(\R^{2}\right)$.
Therefore, $f=R^{\left(\delta\right)}A^{\left(\delta\right)}f$ for
the reconstruction operator $R^{\left(\delta\right)}:C_{u^{s^{\natural}}}^{p,p}\to\mathscr{S}_{\alpha,s^{\natural}}^{p,p}\left(\R^{2}\right)$
that is provided by Theorem \ref{thm:NicelySimplifiedUnconnectedAlphaShearletFrameConditions}
(applied with $\varphi^{\ast},\psi^{\ast}$ instead of $\varphi,\psi$).
Thus,
\[
\left\Vert f\right\Vert _{\mathscr{S}_{\alpha,s^{\natural}}^{p,p}}\leq\vertiii{\smash{R^{\left(\delta\right)}}}\cdot\left\Vert \smash{A^{\left(\delta\right)}}f\right\Vert _{C_{u^{s^{\natural}}}^{p,p}}=\vertiii{\smash{R^{\left(\delta\right)}}}\cdot\left\Vert \left(u_{v}^{s}\cdot\left\langle f,\,\smash{\gamma^{\left[v,k,\delta\right]}}\right\rangle _{L^{2}}\vphantom{\gamma^{\left[v,k,\delta\right]}}\right)_{v\in V,\,k\in\Z^{2}}\right\Vert _{\ell^{p}}.
\]

If we apply the preceding considerations for $s=0$ and $p=2$, we
in particular get
\[
\left\Vert f\right\Vert _{L^{2}}\asymp\left\Vert f\right\Vert _{\mathscr{S}_{\alpha,0}^{2,2}}\asymp\left\Vert \left(\left\langle f,\,\smash{\gamma^{\left[v,k,\delta\right]}}\right\rangle _{L^{2}}\vphantom{\gamma^{\left[v,k,\delta\right]}}\right)_{v\in V,\,k\in\Z^{2}}\right\Vert _{\ell^{2}}\qquad\forall f\in L^{2}\left(\R^{2}\right)=\mathscr{S}_{\alpha,0}^{2,2}\left(\R^{2}\right),
\]
which implies that the $\alpha$-shearlet system ${\rm SH}_{\alpha}\left(\varphi,\psi;\delta\right)=\left(\gamma^{\left[v,k,\delta\right]}\right)_{v\in V,\,k\in\Z^{2}}$
is a frame for $L^{2}\left(\R^{2}\right)$.
\end{proof}

\section{Approximation of cartoon-like functions using \texorpdfstring{$\alpha$}{α}-shearlets}

\label{sec:CartoonLikeApproximation}One of the most celebrated properties
of shearlet systems is that they provide (almost) optimal \emph{approximation
rates} for the model class $\mathcal{E}^{2}\left(\R^{2};\nu\right)$
of \textbf{cartoon-like functions}, which we introduce formally in
Definition \ref{def:CartoonLikeFunction} below. More precisely, this
means (cf.\@ \cite[Theorem 1.3]{CompactlySupportedShearletsAreOptimallySparse}
for the case of compactly supported shearlets) that 
\begin{equation}
\left\Vert f-\smash{f^{\left(N\right)}}\right\Vert _{L^{2}}\leq C\cdot N^{-1}\cdot\left(1+\log N\right)^{3/2}\qquad\forall N\in\N\text{ and }f\in\mathcal{E}^{2}\left(\R^{2};\nu\right),\label{eq:UsualShearletApproximation}
\end{equation}
where $f^{\left(N\right)}$ is the so-called \textbf{$N$-term approximation
of $f$}.

The exact interpretation of this $N$-term approximation, however,
requires some explanation, as was briefly discussed in the introduction:
In general, given a dictionary $\Psi=\left(\psi_{i}\right)_{i\in I}$
in a Hilbert space $\mathcal{H}$ (which is assumed to satisfy $\overline{{\rm span}\left\{ \psi_{i}\with i\in I\right\} }=\mathcal{H}$),
we let
\begin{equation}
\mathcal{H}_{\Psi}^{\left(N\right)}:=\left\{ \sum_{i\in J}\alpha_{i}\psi_{i}\with J\subset I\text{ with }\left|J\right|\leq N\text{ and }\left(\alpha_{i}\right)_{i\in J}\in\Compl^{J}\right\} \label{eq:NElementsLinearCombinationSpaceDefinition}
\end{equation}
denote the sub\emph{set} (which is in general \emph{not} a subspace)
of $\mathcal{H}$ consisting of linear combinations of (at most) $N$
elements of $\Psi$. The usual definition of a (in general non-unique)
best \textbf{$N$-term approximation} to $f\in\mathcal{H}$ is any
$f_{\Psi}^{\left(N\right)}\in\mathcal{H}_{\Psi}^{\left(N\right)}$
satisfying
\[
\left\Vert f-\smash{f_{\Psi}^{\left(N\right)}}\right\Vert =\inf_{g\in\mathcal{H}_{\Psi}^{\left(N\right)}}\left\Vert f-g\right\Vert .
\]
This definition is given for example in \cite[Section 3.1]{ShearletsAndOptimallySparseApproximation}.
Note, however, that in general, it is not clear whether such a best
$N$-term approximation exists. But regardless of whether a best $N$-term
approximation exists or not, we can always define the \textbf{$N$-term
approximation error} as
\begin{equation}
\alpha_{\Psi}^{\left(N\right)}\left(f\right):=\inf_{g\in\mathcal{H}_{\Psi}^{\left(N\right)}}\left\Vert f-g\right\Vert .\label{eq:GeneralNTermApproximationErrorDefinition}
\end{equation}

All in all, the goal of (nonlinear) $N$-term approximations is to
approximate an element $f\in\mathcal{H}$ using only a fixed number
of elements \emph{from the dictionary $\Psi$}. Thus, when one reads
the usual statement that \emph{shearlets provide (almost) optimal
$N$-term approximation rates for cartoon-like functions}, one could
be tempted to think that equation \eqref{eq:UsualShearletApproximation}
has to be understood as
\begin{equation}
\alpha_{\Psi}^{\left(N\right)}\left(f\right)\leq C\cdot N^{-1}\cdot\left(1+\log N\right)^{3/2}\qquad\forall N\in\N\text{ and }f\in\mathcal{E}^{2}\left(\R^{2};\nu\right),\label{eq:UsualShearletApproximationFormal}
\end{equation}
where the dictionary \emph{$\Psi$ is a (suitable) shearlet system}.
This, however, is \emph{not} what is shown e.g.\@ in \cite{ShearletsAndOptimallySparseApproximation}.
What is shown there, instead, is that if $\widetilde{\Psi}=\left(\smash{\widetilde{\psi_{i}}}\right)_{i\in I}$
denotes \emph{the (canonical) dual frame} (in fact, any dual frame
will do) of a suitable shearlet system $\Psi$, then we have
\[
\alpha_{\widetilde{\Psi}}^{\left(N\right)}\left(f\right)\leq C\cdot N^{-1}\cdot\left(1+\log N\right)^{3/2}\qquad\forall N\in\N\text{ and }f\in\mathcal{E}^{2}\left(\R^{2};\nu\right).
\]
This approximation rate \emph{using the dual frame} $\widetilde{\Psi}$
is not completely satisfactory, since for non-tight shearlet systems
$\Psi$, the properties of $\widetilde{\Psi}$ (like smoothness, decay,
etc) are largely unknown. Note that there is no known construction
of a \emph{tight, compactly supported} cone-adapted shearlet frame.
Furthermore, to our knowledge, there is—up to now—nothing nontrivial\footnote{Of course, one knows $\alpha_{\Psi}^{\left(N\right)}\left(f\right)\to0$
as $N\to\infty$, but this holds for every $f\in L^{2}\left(\R^{2}\right)$
and every frame $\Psi$ of $L^{2}\left(\R^{2}\right)$.} known about $\alpha_{\Psi}^{\left(N\right)}\left(f\right)$ for $f\in\mathcal{E}^{2}\left(\R^{2}\right)$
in the case that $\Psi$ is itself a shearlet system, unless $\Psi$
is a \emph{tight} shearlet frame.

\medskip{}

This difference between approximation using the primal and the dual
frame is essentially a difference between analysis and synthesis sparsity:
The usual proof strategy to obtain the approximation rate with respect
to the \emph{dual} frame is to show that the analysis coefficients
$\left(\left\langle f,\,\psi_{i}\right\rangle \right)_{i\in I}$ are
\emph{sparse} in the sense that they lie in some (weak) $\ell^{p}$
space. Then one uses the reconstruction formula
\[
f=\sum_{i\in I}\left\langle f,\,\psi_{i}\right\rangle \widetilde{\psi_{i}}\qquad\text{ using the dual frame }\widetilde{\Psi}=\left(\smash{\widetilde{\psi_{i}}}\right)_{i\in I}
\]
and truncates this series to the $N$ terms with the largest coefficients
$\left|\left\langle f,\,\psi_{i}\right\rangle \right|$. Using the
sparsity of the coefficients, one then obtains the claim. In other
words, since the analysis coefficients with respect to $\Psi=\left(\psi_{i}\right)_{i\in I}$
are the synthesis coefficients with respect to $\widetilde{\Psi}$,
analysis sparsity with respect to $\Psi$ yields synthesis sparsity
with respect to $\widetilde{\Psi}$. Conversely, analysis sparsity
with respect to $\widetilde{\Psi}$ yields synthesis sparsity with
respect to $\Psi$ itself. But since only limited knowledge about
$\widetilde{\Psi}$ is available, this fact is essentially impossible
to apply.

\medskip{}

But our preceding results concerning Banach frames and atomic decompositions
for ($\alpha$)-shearlet smoothness spaces show that \emph{analysis
sparsity is equivalent to synthesis sparsity} (cf.\@ Theorem \ref{thm:AnalysisAndSynthesisSparsityAreEquivalent})
for sufficiently nice and sufficiently densely sampled $\alpha$-shearlet
frames. Using this fact, we will show in this section that we indeed
have
\[
\alpha_{\Psi}^{\left(N\right)}\left(f\right)\leq C_{\varepsilon}\cdot N^{-\left(1-\varepsilon\right)}\qquad\forall N\in\N\text{ and }f\in\mathcal{E}^{2}\left(\R^{2};\nu\right),
\]
where $\varepsilon\in\left(0,1\right)$ can be chosen arbitrarily
and where $\Psi$ is a (suitable) shearlet frame. In fact, we will
also obtain a corresponding statement for $\alpha$-shearlet frames.
Note though that the approximation rate $N^{-\left(1-\varepsilon\right)}$
is slightly inferior to the rate of decay in equation \eqref{eq:UsualShearletApproximationFormal}.
Nevertheless—to the best of our knowledge—this is still the best result
on approximating cartoon-like functions \emph{by shearlets} (instead
of using the \emph{dual frame} of a shearlet frame) which is known.

Our proof strategy is straightforward: The known \emph{analysis-sparsity}
results, in conjunction with our results about Banach frames for shearlet
smoothness spaces, show that $\mathcal{E}^{2}\left(\R^{2};\nu\right)$
is a bounded subset of a certain range of shearlet smoothness spaces.
Thus, using our results about atomic decompositions for these shearlet
smoothness spaces, we get \emph{synthesis sparsity} with respect to
the (primal(!))\@ shearlet frame. We then truncate this (quickly
decaying) series to obtain a good $N$-term approximation.

\medskip{}

We begin our considerations by recalling the notion of $C^{\beta}$-cartoon-like
functions, which were originally introduced (in a preliminary form)
in \cite{DonohoSparseComponentsOfImages}.\pagebreak{}
\begin{defn}
\label{def:CartoonLikeFunction}Fix parameters $0<\varrho_{0}<\varrho_{1}<1$
once and for all.

\begin{itemize}[leftmargin=0.6cm]
\item For $\nu>0$ and $\beta\in\left(1,2\right]$, the set ${\rm STAR}^{\beta}\left(\nu\right)$
is the family of all subsets $\mathcal{B}\subset\left[0,1\right]^{2}$
for which there is some $x_{0}\in\R^{2}$ and a $2\pi$-periodic function
$\varrho:\R\to\left[\varrho_{0},\varrho_{1}\right]$ with $\varrho\in C^{\beta}\left(\R\right)$
such that
\[
\mathcal{B}-x_{0}=\left\{ r\cdot\left(\begin{matrix}\cos\phi\\
\sin\phi
\end{matrix}\right)\with\phi\in\left[0,2\pi\right]\text{ and }0\leq r\leq\varrho\left(\phi\right)\right\} 
\]
and such that the $\beta-1$ \textbf{Hölder semi-norm} $\left[\varrho'\right]_{\beta-1}=\sup_{\phi,\varphi\in\R,\phi\neq\varphi}\smash{\frac{\left|\varrho'\left(\phi\right)-\varrho'\left(\varphi\right)\right|}{\left|\phi-\varphi\right|^{\beta-1}}}\vphantom{\sum^{m}}$
satisfies $\left[\varrho'\right]_{\beta-1}\leq\nu$.
\item For $\nu>0$ and $\beta\in\left(1,2\right]$, the class $\mathcal{E}^{\beta}\left(\R^{2};\nu\right)$
of \textbf{cartoon-like functions with regularity $\beta$} is defined
as
\[
\qquad\mathcal{E}^{\beta}\left(\R^{2};\nu\right):=\left\{ f_{1}+\Indicator_{\mathcal{B}}\cdot f_{2}\with\vphantom{C^{\gamma}}\mathcal{B}\in\smash{{\rm STAR}^{\beta}\left(\nu\right)}\text{ and }f_{i}\in\smash{C_{c}^{\beta}}\left(\smash{\left[0,1\right]^{2}}\right)\text{ with }\left\Vert f_{i}\right\Vert _{C^{\beta}}\leq\min\left\{ 1,\nu\right\} \text{ for }i\in\underline{2}\right\} ,
\]
where $\left\Vert f\right\Vert _{C^{\beta}}=\left\Vert f\right\Vert _{\sup}+\left\Vert \nabla f\right\Vert _{\sup}+\left[\nabla f\right]_{\beta-1}$
and $\left[g\right]_{\beta-1}=\sup_{x,y\in\R^{2},x\neq y}\frac{\left|g\left(x\right)-g\left(y\right)\right|}{\left|x-y\right|^{\beta-1}}$
for $g:\R^{2}\to\Compl^{\ell}$, as well as
\[
C_{c}^{\beta}\left(\smash{\left[0,1\right]^{2}}\right)=\left\{ f\in\smash{C^{\left\lfloor \beta\right\rfloor }}\left(\R^{2}\right)\with\supp f\subset\left[0,1\right]^{2}\text{ and }\left\Vert f\right\Vert _{C^{\beta}}<\infty\right\} .
\]
Finally, we set $\mathcal{E}^{\beta}\left(\R^{2}\right):=\bigcup_{\nu>0}\mathcal{E}^{\beta}\left(\R^{2};\nu\right)$.\qedhere
\end{itemize}
\end{defn}
\begin{rem*}
The definition of ${\rm STAR}^{\beta}\left(\nu\right)$ given here
is slightly more conservative than in \cite[Definition 2.5]{CartoonApproximationWithAlphaCurvelets},
where it is only assumed that $\varrho:\R\to\left[0,\varrho_{1}\right]$
with $0<\varrho_{1}<1$, instead of $\varrho:\R\to\left[\varrho_{0},\varrho_{1}\right]$.
We also note that $\left[\varrho'\right]_{\beta-1}=\left\Vert \varrho''\right\Vert _{\sup}$
in case of $\beta=2$. This is a simple consequence of the definition
of the derivative and of the mean-value theorem. Hence, in case of
$\beta=2$, the definition given here is consistent with (in fact,
slightly stronger than) the one used in \cite[Definition 1.1]{ShearletsAndOptimallySparseApproximation}.

Further, we note that in \cite[Definition 5.9]{AlphaMolecules}, the
class $\mathcal{E}^{\beta}\left(\R^{2}\right)$ is simply defined
as
\[
\left\{ f_{1}+\Indicator_{B}\cdot f_{2}\with f_{1},f_{2}\in C_{c}^{\beta}\left(\smash{\left[0,1\right]^{2}}\right),\:B\subset\left[0,1\right]^{2}\text{ Jordan dom. with regular closed piecewise }C^{\beta}\text{ boundary curve}\right\} .
\]
Even for this—much more general—definition, the authors of \cite{AlphaMolecules}
then invoke the results which are derived in \cite{CartoonApproximationWithAlphaCurvelets}
under the more restrictive assumptions.

This is somewhat unpleasant, but does not need to concern us: In fact,
in the following, we will frequently use the notation $\mathcal{E}^{\beta}\left(\R^{2};\nu\right)$,
but the precise definition of this space is not really used; all that
we need to know is that if $\varphi,\psi$ are suitable shearlet generators,
then the $\beta$-shearlet coefficients $c=\left(c_{j,k,\varepsilon,m}\right)_{j,k,\varepsilon,m}$
of $f\in\mathcal{E}^{\beta}\left(\R^{2};\nu\right)$ satisfy $c\in\ell^{\frac{2}{1+\beta}+\varepsilon}$
for all $\varepsilon>0$, with $\left\Vert f\right\Vert _{\ell^{\frac{2}{1+\beta}+\varepsilon}}\leq C_{\varepsilon,\nu,\beta,\varphi,\psi}$.
Below, we will derive this by combining \cite[Theorem 4.2]{CartoonApproximationWithAlphaCurvelets}
with \cite[Theorem 5.6]{AlphaMolecules}, where \cite[Theorem 5.6]{AlphaMolecules}
does not use the notion of cartoon-like functions at all.
\end{rem*}
As our first main technical result in this section, we show that the
$C^{\beta}$-cartoon-like functions are bounded subsets of suitably
chosen $\alpha$-shearlet smoothness spaces. Once we have developed
this property, we obtain the claimed approximation rate by invoking
the atomic decomposition results from Theorem \ref{thm:ReallyNiceUnconnectedShearletAtomicDecompositionConditions}.
\begin{prop}
\label{prop:CartoonLikeFunctionsBoundedInAlphaShearletSmoothness}Let
$\nu>0$ and $\beta\in\left(1,2\right]$ be arbitrary and let $p\in\left(2/\left(1+\beta\right),\:2\right]$.
Then 
\[
\mathcal{E}^{\beta}\left(\R^{2};\nu\right)\quad\text{ is a bounded subset of }\quad\mathscr{S}_{\beta^{-1},\left(1+\beta^{-1}\right)\left(p^{-1}-2^{-1}\right)}^{p,p}\left(\R^{2}\right).\qedhere
\]
\end{prop}
\begin{proof}
Here, we only give the proof for the case $\beta=2$. For $\beta\in\left(1,2\right)$,
the proof is more involved and thus postponed to the appendix (Section
\ref{sec:CartoonLikeFunctionsAreBoundedInAlphaShearletSmoothness}).
The main reason for the additional complications in case of $\beta\in\left(1,2\right)$
is that our proof essentially requires that we already know that there
is some sufficiently nice, cone-adapted $\alpha$-shearlet system
with respect to which the $C^{\beta}$-cartoon-like functions are
analysis sparse (in a suitable ``almost $\ell^{2/\left(1+\beta\right)}$''
sense). In case of $\beta=2$, this is known, since we then have $\alpha=\beta^{-1}=\frac{1}{2}$,
so that the $\alpha$-shearlet systems from Definition \ref{def:AlphaShearletSystem}
coincide with the usual cone-adapted shearlets, cf.\@ Remark \ref{rem:AlphaShearletsYieldUsualShearlets}.
But in case of $\beta\in\left(1,2\right)$, it is only known (cf.\@
\cite[Theorem 5.6]{AlphaMolecules}) that $C^{\beta}$-cartoon-like
functions are analysis sparse with respect to suitable \textbf{$\beta$-shearlet
systems} (cf.\@ Definition \ref{def:BetaShearletSystem} and note
$\beta\notin\left[0,1\right]$, so that the notion of $\beta$-shearlets
does not collide with our notion of $\alpha$-shearlets for $\alpha\in\left[0,1\right]$)
which are different, but closely related to the $\beta^{-1}$-shearlet
systems from Definition \ref{def:AlphaShearletSystem}. Making this
close connection precise is what mainly makes the proof in case of
$\beta\in\left(1,2\right)$ more involved, cf.\@ Section \ref{sec:CartoonLikeFunctionsAreBoundedInAlphaShearletSmoothness}.

Thus, let us consider the case $\beta=2$. Choose $\phi_{0}\in\TestFunctionSpace{\R}$
with $\phi_{0}\geq0$ and $\phi_{0}\not\equiv0$, so that $\widehat{\phi_{0}}\left(0\right)=\left\Vert \phi_{0}\right\Vert _{L^{1}}>0$.
By continuity of $\widehat{\phi_{0}}$, there is thus some $\nu>0$
with $\widehat{\phi_{0}}\left(\xi\right)\neq0$ on $\left[-\nu,\nu\right]$.
Now, define $\phi_{1}:=\phi_{0}\left(3\mybullet/\nu\right)$ and note
that $\phi_{1}\in\TestFunctionSpace{\R}$ with $\widehat{\phi_{1}}\left(\xi\right)=\frac{\nu}{3}\cdot\widehat{\phi_{0}}\left(\nu\xi/3\right)\neq0$
for $\xi\in\left[-3,3\right]$.

Now, set $\varphi:=\phi_{1}\otimes\phi_{1}\in\TestFunctionSpace{\R^{2}}$
and $\psi_{2}:=\phi_{1}$, as well as $\psi_{1}:=\phi_{1}^{\left(8\right)}$,
the $8$-th derivative of $\phi_{1}$. By differentiating under the
integral and by performing partial integration, we get for $0\leq k\leq7$
that
\begin{equation}
\frac{\d^{k}}{\d\xi^{k}}\bigg|_{\xi=0}\widehat{\psi_{1}}=\frac{\d^{k}}{\d\xi^{k}}\bigg|_{\xi=0}\widehat{\phi_{1}^{\left(8\right)}}=\int_{\R}\phi_{1}^{\left(8\right)}\left(x\right)\cdot\left(-2\pi ix\right)^{k}\d x=\left(-1\right)^{8}\cdot\int_{\R}\phi_{1}\left(x\right)\cdot\frac{\d^{8}\left(-2\pi ix\right)^{k}}{\d x^{8}}\d x=0,\label{eq:CartoonLikeFunctionsBoundedVanishingMoments}
\end{equation}
since $\frac{\d^{8}\left(-2\pi ix\right)^{k}}{\d x^{8}}\equiv0$ for
$0\leq k\leq7$. Next, observe $\widehat{\varphi}\left(\xi\right)=\widehat{\phi_{1}}\left(\xi_{1}\right)\cdot\widehat{\phi_{1}}\left(\xi_{2}\right)\neq0$
for $\xi\in\left[-3,3\right]^{2}\supset\left[-1,1\right]^{2}$, as
well as $\widehat{\psi_{2}}\left(\xi\right)\neq0$ for $\xi\in\left[-3,3\right]$
and finally 
\[
\widehat{\psi_{1}}\left(\xi\right)=\left(2\pi i\xi\right)^{8}\cdot\widehat{\phi_{1}}\left(\xi\right)=\left(2\pi\right)^{8}\cdot\xi^{8}\cdot\widehat{\phi_{1}}\left(\xi\right)\neq0\text{ for }\xi\in\left[-3,3\right]\setminus\left\{ 0\right\} ,
\]
which in particular implies $\widehat{\psi_{1}}\left(\xi\right)\neq0$
for $\frac{1}{3}\leq\left|\xi\right|\leq3$.

Now, setting $\psi:=\psi_{1}\otimes\psi_{2}$, we want to verify that
$\varphi,\psi$ satisfy the assumptions of Theorem \ref{thm:AnalysisAndSynthesisSparsityAreEquivalent}
with the choices $\varepsilon=\frac{1}{4}$, $p_{0}=\frac{2}{3}$,
$s^{\left(0\right)}=0$ and $\alpha=\frac{1}{2}$. Since we have $\varphi\in\TestFunctionSpace{\R^{2}}$
and $\psi_{1},\psi_{2}\in\TestFunctionSpace{\R}$ and since $\widehat{\psi_{2}}\left(\xi\right)\neq0$
for $\xi\in\left[-3,3\right]$ and $\widehat{\psi_{1}}\left(\xi\right)\neq0$
for $\frac{1}{3}\leq\left|\xi\right|\leq3$ and since finally $\widehat{\varphi}\left(\xi\right)\neq0$
for $\xi\in\left[-1,1\right]^{2}$, Remark \ref{rem:NiceTensorConditionsForUnconnectedCovering}
and Corollaries \ref{cor:ReallyNiceAlphaShearletTensorAtomicDecompositionConditions}
and \ref{cor:ReallyNiceAlphaShearletTensorBanachFrameConditions}
show that all we need to check is $\frac{\d^{\ell}}{\d\xi^{\ell}}\big|_{\xi=0}\widehat{\psi_{1}}=0$
for all $\ell=0,\dots,N_{0}+\left\lceil \Lambda_{1}\right\rceil -1$
and all $\ell=0,\dots,N_{0}+\left\lceil M_{1}\right\rceil -1$, where
$N_{0}=\left\lceil p_{0}^{-1}\cdot\left(2+\varepsilon\right)\right\rceil =\left\lceil 27/8\right\rceil =4$,
\begin{align*}
\Lambda_{1} & =\varepsilon+p_{0}^{-1}+\max\left\{ 0,\,\left(1+\alpha\right)\left(p_{0}^{-1}-1\right)\right\} =\frac{5}{2}\leq3,\\
\text{and }M_{1} & =\varepsilon+p_{0}^{-1}+\max\left\{ 0,\,\left(1+\alpha\right)\left(p_{0}^{-1}-2^{-1}\right)\right\} =\frac{1}{4}+3\leq4,
\end{align*}
cf.\@ Theorems \ref{thm:AnalysisAndSynthesisSparsityAreEquivalent},
\ref{thm:NicelySimplifiedAlphaShearletFrameConditions}, and \ref{thm:ReallyNiceShearletAtomicDecompositionConditions}.
Hence, $N_{0}+\left\lceil \Lambda_{1}\right\rceil -1\leq7$ and $N_{0}+\left\lceil M_{1}\right\rceil -1\leq7$,
so that equation \eqref{eq:CartoonLikeFunctionsBoundedVanishingMoments}
shows that $\varphi,\psi$ indeed satisfy the assumptions of Theorem
\ref{thm:AnalysisAndSynthesisSparsityAreEquivalent}. That theorem
yields because of $\alpha=\frac{1}{2}$ some $\delta_{0}\in\left(0,1\right]$
such that the following hold for all $0<\delta\leq\delta_{0}$:

\begin{itemize}[leftmargin=0.6cm]
\item The shearlet system ${\rm SH}_{1/2}\left(\varphi,\psi;\delta\right)=\left(\gamma^{\left[v,k,\delta\right]}\right)_{v\in V,k\in\Z^{2}}$
is a frame for $L^{2}\left(\R^{2}\right)$.
\item Since $p\in\left(2/\left(1+\beta\right),\:2\right]\subset\left[\frac{2}{3},2\right]=\left[p_{0},2\right]$,
we have 
\[
\qquad\mathscr{S}_{\beta^{-1},\,\left(1+\beta^{-1}\right)\left(\frac{1}{p}-\frac{1}{2}\right)}^{p,p}\left(\R^{2}\right)=\mathscr{S}_{\alpha,\left(1+\alpha\right)\left(\frac{1}{p}-\frac{1}{2}\right)}^{p,p}\left(\R^{2}\right)=\left\{ f\in L^{2}\left(\R^{2}\right)\with\left(\left\langle f,\,\smash{\gamma^{\left[v,k,\delta\right]}}\right\rangle _{L^{2}}\vphantom{\gamma^{\left[v,k,\delta\right]}}\right)_{v\in V,\,k\in\Z^{2}}\in\ell^{p}\left(V\times\Z^{2}\right)\right\} 
\]
and there is a constant $C_{p}=C_{p}\left(\varphi,\psi,\delta\right)>0$
such that
\[
\left\Vert f\right\Vert _{\mathscr{S}_{\beta^{-1},\,\left(1+\beta^{-1}\right)\left(p^{-1}-2^{-1}\right)}^{p,p}}\leq C_{p}\cdot\left\Vert \left(\left\langle f,\,\smash{\gamma^{\left[v,k,\delta\right]}}\right\rangle _{L^{2}}\vphantom{\gamma^{\left[v,k,\delta\right]}}\right)_{v\in V,\,k\in\Z^{2}}\right\Vert _{\ell^{p}}\qquad\forall f\in\mathscr{S}_{\beta^{-1},\,\left(1+\beta^{-1}\right)\left(p^{-1}-2^{-1}\right)}^{p,p}\left(\R^{2}\right).
\]
\end{itemize}
Thus, since we clearly have $\mathcal{E}^{2}\left(\R^{2};\nu\right)\subset L^{2}\left(\R^{2}\right)$,
it suffices to show that there is a constant $C=C\left(p,\nu,\delta,\varphi,\psi\right)>0$
such that $\left\Vert A^{\left(\delta\right)}f\right\Vert _{\ell^{p}}\leq C<\infty$
for all $f\in\mathcal{E}^{2}\left(\R^{2};\nu\right)$, where $A^{\left(\delta\right)}f:=\left(\left\langle f,\,\gamma^{\left[v,k,\delta\right]}\right\rangle _{L^{2}}\right)_{v\in V,\,k\in\Z^{2}}$.
Here, we note that the sequence $A^{\left(\delta\right)}f$ just consists
of the shearlet coefficients of $f$ (up to a trivial reordering in
the translation variable $k$) with respect to the shearlet frame
${\rm SH}_{\frac{1}{2}}\left(\varphi,\psi;\,\delta\right)={\rm SH}\left(\varphi,\psi,\smash{\theta};\,\delta\right)$
with $\theta\left(x,y\right)=\psi\left(y,x\right)$, cf.\@ Remark
\ref{rem:AlphaShearletsYieldUsualShearlets}. Hence, there is hope
to derive the estimate $\left\Vert A^{\left(\delta\right)}f\right\Vert _{\ell^{p}}\leq C$
as a consequence of \cite[equation (3)]{CompactlySupportedShearletsAreOptimallySparse},
which states that
\begin{equation}
\sum_{n>N}\left|\lambda\left(f\right)\right|_{n}^{2}\leq C\cdot N^{-2}\cdot\left(1+\log N\right)^{3}\qquad\forall N\in\N\text{ and }f\in\mathcal{E}^{2}\left(\R^{2};\nu\right),\label{eq:ShearletCoefficientDecay}
\end{equation}
where $\left(\left|\lambda\left(f\right)\right|_{n}\right)_{n\in\N}$
are the absolute values of the shearlet coefficients of $f$ with
respect to the shearlet frame ${\rm SH}\left(\varphi,\psi,\theta;\,\delta\right)$,
ordered nonincreasingly. In particular, $\left\Vert A^{\left(\delta\right)}f\right\Vert _{\ell^{p}}=\left\Vert \left[\left|\lambda\left(f\right)\right|_{n}\right]_{n\in\N}\right\Vert _{\ell^{p}}$.

\medskip{}

Note though that in order for \cite[equation (3)]{CompactlySupportedShearletsAreOptimallySparse}
to be applicable, we need to verify that $\varphi,\psi,\theta$ satisfy
the assumptions of \cite[Theorem 1.3]{CompactlySupportedShearletsAreOptimallySparse},
i.e., $\varphi,\psi,\theta$ need to be compactly supported (which
is satisfied) and

\begin{enumerate}
\item $\left|\widehat{\psi}\left(\xi\right)\right|\lesssim\min\left\{ 1,\left|\xi_{1}\right|^{\sigma}\right\} \cdot\min\left\{ 1,\left|\xi_{1}\right|^{-\tau}\right\} \cdot\min\left\{ 1,\left|\xi_{2}\right|^{-\tau}\right\} $
and
\item $\left|\frac{\partial}{\partial\xi_{2}}\widehat{\psi}\left(\xi\right)\right|\leq\left|h\left(\xi_{1}\right)\right|\cdot\left(1+\frac{\left|\xi_{2}\right|}{\left|\xi_{1}\right|}\right)^{-\tau}$
for some $h\in L^{1}\left(\R\right)$
\end{enumerate}
for certain (arbitrary) $\sigma>5$ and $\tau\geq4$. Furthermore,
$\theta$ needs to satisfy the same estimate with interchanged roles
of $\xi_{1},\xi_{2}$. But in view of $\theta\left(x,y\right)=\psi\left(y,x\right)$,
it suffices to establish the estimates for $\psi$. To this end, recall
from above that $\widehat{\psi_{1}}\in C^{\infty}\left(\R\right)$
is analytic with $\frac{\d^{k}}{\d\xi^{k}}\bigg|_{\xi=0}\widehat{\psi_{1}}=0$
for $0\leq k\leq7$. This easily implies $\left|\widehat{\psi_{1}}\left(\xi\right)\right|\lesssim\left|\xi\right|^{8}$
for $\left|\xi\right|\leq1$, see e.g.\@ the proof of Corollary \ref{cor:ReallyNiceAlphaShearletTensorAtomicDecompositionConditions},
in particular equation \eqref{eq:VanishingFourierDerivativesYieldFourierDecayAtOrigin}.
Furthermore, since $\psi_{1},\psi_{2}\in\TestFunctionSpace{\R}$,
we get for arbitrary $K\in\N$ that $\left|\widehat{\psi_{i}}\left(\xi\right)\right|\lesssim\left(1+\left|\xi\right|\right)^{-K}$
for $i\in\left\{ 1,2\right\} $. Altogether, we conclude $\left|\widehat{\psi_{1}}\left(\xi\right)\right|\lesssim\min\left\{ 1,\left|\xi\right|^{8}\right\} \cdot\left(1+\left|\xi\right|\right)^{-8}$
and likewise $\left|\widehat{\psi_{2}}\left(\xi\right)\right|\lesssim\left(1+\left|\xi\right|\right)^{-8}$
for all $\xi\in\R$, so that the first estimate is fulfilled for $\sigma:=8>5$
and $\tau:=8\geq4$.

Next, we observe for $\xi\in\R^{2}$ with $\xi_{1}\neq0$ that
\[
1+\frac{\left|\xi_{2}\right|}{\left|\xi_{1}\right|}\leq\left(1+\left|\xi_{2}\right|\right)\cdot\left(1+\left|\xi_{1}\right|^{-1}\right)\leq2\cdot\left(1+\left|\xi_{2}\right|\right)\cdot\max\left\{ 1,\,\left|\xi_{1}\right|^{-1}\right\} 
\]
and thus
\[
\left(1+\left|\xi_{2}\right|/\left|\xi_{1}\right|\right)^{-8}\geq2^{-8}\cdot\left(1+\left|\xi_{2}\right|\right)^{-8}\cdot\min\left\{ 1,\,\left|\xi_{1}\right|^{8}\right\} .
\]
But since we have $\widehat{\psi_{2}}\in\Schwartz\left(\R\right)$
and thus $\left|\widehat{\psi_{2}}'\left(\xi\right)\right|\lesssim\left(1+\left|\xi\right|\right)^{-8}$,
this implies
\begin{align*}
\left|\frac{\partial}{\partial\xi_{2}}\widehat{\psi}\left(\xi\right)\right|=\left|\widehat{\psi_{1}}\left(\xi_{1}\right)\right|\cdot\left|\widehat{\psi_{2}}'\left(\xi_{2}\right)\right| & \lesssim\left(1+\left|\xi_{1}\right|\right)^{-8}\cdot\left(1+\left|\xi_{2}\right|\right)^{-8}\cdot\min\left\{ 1,\left|\xi_{1}\right|^{8}\right\} \\
 & \lesssim\left(1+\left|\xi_{1}\right|\right)^{-8}\cdot\left(1+\left|\xi_{2}\right|/\left|\xi_{1}\right|\right)^{-8},
\end{align*}
so that the second condition from above is satisfied for our choice
$\tau=8$, with $h\left(\xi_{1}\right)=\left(1+\left|\xi_{1}\right|\right)^{-8}$.

\medskip{}

Consequently, we conclude from \cite[equation (3)]{CompactlySupportedShearletsAreOptimallySparse}
that equation \eqref{eq:ShearletCoefficientDecay} is satisfied. Now,
for arbitrary $M\in\N_{\geq4}$, we apply equation \eqref{eq:ShearletCoefficientDecay}
with $N=\left\lceil \frac{M}{2}\right\rceil \geq2$, noting that $\left\lceil \frac{M}{2}\right\rceil \leq\frac{M}{2}+1\leq\frac{M}{2}+\frac{M}{4}=\frac{3}{4}M\leq M$
to deduce
\begin{align*}
\frac{1}{4}M\cdot\left|\lambda\left(f\right)\right|_{M}^{2}\leq\left|\lambda\left(f\right)\right|_{M}^{2}\cdot\left(M-\left\lceil M/2\right\rceil \right) & \leq\sum_{\left\lceil M/2\right\rceil <n\leq M}\left|\lambda\left(f\right)\right|_{n}^{2}\leq\sum_{n>\left\lceil M/2\right\rceil }\left|\lambda\left(f\right)\right|_{n}^{2}\\
 & \leq C\cdot\left\lceil M/2\right\rceil ^{-2}\cdot\left(1+\log\left\lceil M/2\right\rceil \right)^{3}\\
 & \leq4C\cdot M^{-2}\cdot\left(1+\log M\right)^{3},
\end{align*}
which implies $\left|\lambda\left(f\right)\right|_{M}\leq\sqrt{16C}\cdot\left[M^{-1}\cdot\left(1+\log M\right)\right]^{3/2}$
for $M\in\N_{\geq4}$. But since $\mathcal{E}^{2}\left(\R^{2};\nu\right)\subset L^{2}\left(\R^{2}\right)$
is bounded and since the elements of the shearlet frame ${\rm SH}\left(\varphi,\psi,\smash{\theta};\,\delta\right)$
are $L^{2}$-bounded, we have $\left\Vert \left[\left|\lambda\left(f\right)\right|_{n}\right]_{n\in\N}\right\Vert _{\ell^{\infty}}\lesssim1$,
so that we get $\left|\lambda\left(f\right)\right|_{M}\lesssim\left[M^{-1}\cdot\left(1+\log M\right)\right]^{3/2}$
for all $M\in\N$ and all $f\in\mathcal{E}^{2}\left(\R^{2};\nu\right)$,
where the implied constant is independent of the precise choice of
$f$. But this easily yields $\left\Vert \left[\left|\lambda\left(f\right)\right|_{M}\right]_{M\in\N}\right\Vert _{\ell^{p}}\lesssim1$,
since $p\in\left(\frac{2}{3},2\right]=\left(2/\left(1+\beta\right),\,2\right]$.
Here, the implied constant might depend on $\varphi,\psi,\delta,p,\nu$,
but not on $f\in\mathcal{E}^{2}\left(\R^{2};\nu\right)$.
\end{proof}
We can now easily derive the claimed statement about the approximation
rate of functions $f\in\mathcal{E}^{\beta}\left(\R^{2};\nu\right)$
with respect to $\beta^{-1}$-shearlet systems.
\begin{thm}
\label{thm:CartoonApproximationWithAlphaShearlets}Let $\beta\in\left(1,2\right]$
be arbitrary. Assume that $\varphi,\psi\in L^{1}\left(\R^{2}\right)$
satisfy the conditions of Theorem \ref{thm:ReallyNiceUnconnectedShearletAtomicDecompositionConditions}
for $\alpha=\beta^{-1}$, $p_{0}=q_{0}=\frac{2}{1+\beta}$, $s_{0}=0$
and $s_{1}=\frac{1}{2}\left(1+\beta\right)$ and some $\varepsilon\in\left(0,1\right]$
(see Remark \ref{rem:CartoonApproximationConstantSimplification}
for simplified conditions which ensure that these assumptions are
satisfied).

Then there is some $\delta_{0}=\delta_{0}\left(\varepsilon,\beta,\varphi,\psi\right)>0$
such that for all $0<\delta\leq\delta_{0}$ and arbitrary $f\in\mathcal{E}^{\beta}\left(\R^{2}\right)$
and $N\in\N$, there is a function $f^{\left(N\right)}\in L^{2}\left(\R^{2}\right)$
which is a linear combination of $N$ elements of the $\beta^{-1}$-shearlet
frame $\Psi={\rm SH}_{\beta^{-1}}\left(\varphi,\psi;\delta\right)=\left(\gamma^{\left[v,k,\delta\right]}\right)_{v\in V,k\in\Z^{2}}$
such that the following holds:

For arbitrary $\sigma,\nu>0$, there is a constant $C=C\left(\beta,\delta,\nu,\sigma,\varphi,\psi\right)>0$
satisfying
\[
\left\Vert f-\smash{f^{\left(N\right)}}\right\Vert _{L^{2}}\leq C\cdot N^{-\left(\frac{\beta}{2}-\sigma\right)}\qquad\forall f\in\mathcal{E}^{\beta}\left(\R^{2};\nu\right)\text{ and }N\in\N.\qedhere
\]
\end{thm}
\begin{rem*}
It was shown in \cite[Theorem 2.8]{CartoonApproximationWithAlphaCurvelets}
that \emph{no} dictionary $\Phi$ can achieve an error $\alpha_{\Phi}^{\left(N\right)}\left(f\right)\leq C\cdot N^{-\theta}$
for all $N\in\N$ and $f\in\mathcal{E}^{\beta}\left(\R^{2};\nu\right)$
with $\theta>\frac{\beta}{2}$, as long as one insists on a \emph{polynomial
depth restriction} for forming the $N$-term approximation. In this
sense, the resulting approximation rate is almost optimal. We remark,
however, that it is \emph{not immediately clear} whether the $N$-term
approximation whose existence is claimed by the theorem above can
be chosen to satisfy the polynomial depth search restriction. There
is a long-standing tradition\cite{CandesDonohoCurvelets,CartoonApproximationWithAlphaCurvelets,OptimallySparseMultidimensionalRepresentationUsingShearlets,AlphaMolecules,CompactlySupportedShearletsAreOptimallySparse}
to omit further considerations concerning this question; therefore,
we deferred to Section \ref{sec:PolynomialSearchDepth} the proof
that the above approximation rate can also be achieved using a polynomially
restricted search depth.

For more details on the technical assumption of polynomial depth restriction
in $N$-term approximations, we refer to \cite[Section 2.1.1]{CartoonApproximationWithAlphaCurvelets}.
\end{rem*}
\begin{proof}
Set $\alpha:=\beta^{-1}$. Under the given assumptions, Theorem \ref{thm:ReallyNiceUnconnectedShearletAtomicDecompositionConditions}
ensures that ${\rm SH}_{\alpha}\left(\varphi,\psi;\,\delta\right)$
forms an atomic decomposition for $\mathscr{S}_{\alpha,s}^{p,q}\left(\R^{2}\right)$
for all $p\geq p_{0}$, $q\geq q_{0}$ and $s_{0}\leq s\leq s_{1}$,
for arbitrary $0<\delta\leq\delta_{0}$, where the constant $\delta_{0}=\delta_{0}\left(\alpha,\varepsilon,p_{0},q_{0},s_{0},s_{1},\varphi,\psi\right)=\delta_{0}\left(\varepsilon,\beta,\varphi,\psi\right)>0$
is provided by Theorem \ref{thm:ReallyNiceUnconnectedShearletAtomicDecompositionConditions}.
Fix some $0<\delta\leq\delta_{0}$.

Let $S^{\left(\delta\right)}:C_{u^{0}}^{2,2}\to\mathscr{S}_{\alpha,0}^{2,2}\left(\R^{2}\right)$
and $C^{\left(\delta\right)}:\mathscr{S}_{\alpha,0}^{2,2}\left(\R^{2}\right)\to C_{u^{0}}^{2,2}$
be the synthesis map and the coefficient map whose existence and boundedness
is guaranteed by Theorem \ref{thm:ReallyNiceUnconnectedShearletAtomicDecompositionConditions},
since $2\geq p_{0}=q_{0}$ and since $s_{0}\leq0\leq s_{1}$. Note
directly from Definition \ref{def:CoefficientSpace} that $C_{u^{0}}^{2,2}=\ell^{2}\left(V\times\Z^{2}\right)$
and that $\mathscr{S}_{\alpha,0}^{2,2}\left(\R^{2}\right)=L^{2}\left(\R^{2}\right)$
(cf.\@ \cite[Lemma 6.10]{DecompositionEmbedding}). Now, for arbitrary
$f\in\mathcal{E}^{\beta}\left(\R^{2}\right)\subset L^{2}\left(\R^{2}\right)$,
let 
\[
\left(\smash{c_{j}^{\left(f\right)}}\right)_{j\in V\times\Z^{2}}:=c^{\left(f\right)}:=C^{\left(\delta\right)}f\in\ell^{2}\left(V\times\Z^{2}\right).
\]
Furthermore, for $f\in\mathcal{E}^{\beta}\left(\R^{2}\right)$ and
$N\in\N$, choose a set $J_{N}^{\left(f\right)}\subset V\times\Z^{2}$
with $\left|\smash{J_{N}^{\left(f\right)}}\vphantom{J_{N}}\right|=N$
and such that $\left|\smash{c_{j}^{\left(f\right)}}\right|\geq\left|\smash{c_{i}^{\left(f\right)}}\right|$
for all $j\in J_{N}^{\left(f\right)}$ and all $i\in\left(V\times\Z^{2}\right)\setminus J_{N}^{\left(f\right)}$.
For a general sequence, such a set need not exist, but since we have
$c^{\left(f\right)}\in\ell^{2}$, a moment's thought shows that it
does, since for each $\varepsilon>0$, there are only finitely many
indices $i\in V\times\Z^{2}$ satisfying $\left|\smash{c_{i}^{\left(f\right)}}\right|\geq\varepsilon$.

Finally, set $f^{\left(N\right)}:=S^{\left(\delta\right)}\left[c^{\left(f\right)}\cdot\Indicator_{J_{N}^{\left(f\right)}}\right]\in L^{2}\left(\R^{2}\right)$
and note that $f^{\left(N\right)}$ is indeed a linear combination
of (at most) $N$ elements of ${\rm SH}_{\beta^{-1}}\left(\varphi,\psi;\delta\right)=\left(\gamma^{\left[v,k,\delta\right]}\right)_{v\in V,k\in\Z^{2}}$,
by definition of $S^{\left(\delta\right)}$. Moreover, note that the
so-called \textbf{Stechkin lemma} (see e.g.\@ \cite[Lemma 3.3]{StechkinLemma})
shows
\begin{equation}
\left\Vert c^{\left(f\right)}-\Indicator_{J_{N}^{\left(f\right)}}\cdot c^{\left(f\right)}\right\Vert _{\ell^{2}}\leq N^{-\left(\frac{1}{p}-\frac{1}{2}\right)}\cdot\left\Vert \smash{c^{\left(f\right)}}\right\Vert _{\ell^{p}}\qquad\forall N\in\N\text{ and }p\in\left(0,2\right]\text{ for which }\left\Vert \smash{c^{\left(f\right)}}\right\Vert _{\ell^{p}}<\infty.\label{eq:StechkinEstimate}
\end{equation}

It remains to verify that the $f^{\left(N\right)}$ satisfy the stated
approximation rate. To show this, let $\sigma,\nu>0$ be arbitrary.
Because of $\frac{1}{p}-\frac{1}{2}\to\frac{\beta}{2}$ as $p\downarrow2/\left(1+\beta\right)$,
there is some $p\in\left(2/\left(1+\beta\right),\,2\right)$ satisfying
$\frac{1}{p}-\frac{1}{2}\geq\frac{\beta}{2}-\sigma$. Set $s:=\left(1+\alpha\right)\left(p^{-1}-2^{-1}\right)$.
Observe that $p\geq p_{0}=q_{0}$, as well as 
\[
s_{0}=0\leq s=\left(1+\alpha\right)\left(p^{-1}-2^{-1}\right)\leq\left(1+\alpha\right)\left(\frac{1+\beta}{2}-\frac{1}{2}\right)=\frac{\beta}{2}\left(1+\beta^{-1}\right)=\frac{1}{2}\left(1+\beta\right)=s_{1}.
\]
Now, observe $\vphantom{B_{v}^{\left(\alpha\right)}}\left|\det\smash{B_{v}^{\left(\alpha\right)}}\right|=u_{v}^{1+\alpha}$
for all $v\in V$, so that the remark after Definition \ref{def:CoefficientSpace}
shows that the coefficient space $C_{u^{s}}^{p,p}$ satisfies $C_{u^{s}}^{p,p}=\ell^{p}\left(V\times\Z^{2}\right)\hookrightarrow\ell^{2}\left(V\times\Z^{2}\right)=C_{u^{0}}^{2,2}$.
Therefore, Theorem \ref{thm:ReallyNiceUnconnectedShearletAtomicDecompositionConditions}
and the associated remark (and the inclusion $\mathscr{S}_{\alpha,s}^{p,p}\left(\R^{2}\right)\hookrightarrow L^{2}\left(\R^{2}\right)$
from Theorem \ref{thm:AnalysisAndSynthesisSparsityAreEquivalent})
show that the synthesis map and the coefficient map from above restrict
to bounded linear operators
\[
S^{\left(\delta\right)}:\ell^{p}\left(V\times\Z^{2}\right)\to\mathscr{S}_{\alpha,s}^{p,p}\left(\R^{2}\right)\quad\text{ and }\quad C^{\left(\delta\right)}:\mathscr{S}_{\alpha,s}^{p,p}\left(\R^{2}\right)\to\ell^{p}\left(V\times\Z^{2}\right).
\]

Next, Proposition \ref{prop:CartoonLikeFunctionsBoundedInAlphaShearletSmoothness}
shows $\mathcal{E}^{\beta}\left(\R^{2}\right)\subset\mathscr{S}_{\alpha,s}^{p,p}\left(\R^{2}\right)$
and even yields a constant $C_{1}=C_{1}\left(\beta,\nu,p\right)>0$
satisfying $\left\Vert f\right\Vert _{\mathscr{S}_{\alpha,s}^{p,p}}\leq C_{1}$
for all $f\in\mathcal{E}^{\beta}\left(\R^{2};\nu\right)$.  This
implies 
\begin{equation}
\left\Vert \smash{c^{\left(f\right)}}\right\Vert _{\ell^{p}}=\left\Vert \smash{C^{\left(\delta\right)}}f\right\Vert _{\ell^{p}}\leq\vertiii{\smash{C^{\left(\delta\right)}}}_{\mathscr{S}_{\alpha,s}^{p,p}\to\ell^{p}}\cdot\left\Vert f\right\Vert _{\mathscr{S}_{\alpha,s}^{p,p}}\leq C_{1}\cdot\vertiii{\smash{C^{\left(\delta\right)}}}_{\mathscr{S}_{\alpha,s}^{p,p}\to\ell^{p}}<\infty\qquad\forall f\in\mathcal{E}^{\beta}\left(\R^{2};\nu\right).\label{eq:NTermApproximationellPEstimate}
\end{equation}
By putting everything together and recalling $S^{\left(\delta\right)}\circ C^{\left(\delta\right)}=\identity_{\mathscr{S}_{\alpha,0}^{2,2}}=\identity_{L^{2}}$,
we finally arrive at
\begin{align*}
\left\Vert f-\smash{f^{\left(N\right)}}\right\Vert _{L^{2}} & =\left\Vert S^{\left(\delta\right)}C^{\left(\delta\right)}f-S^{\left(\delta\right)}\left[\Indicator_{J_{N}^{\left(f\right)}}\cdot c^{\left(f\right)}\right]\right\Vert _{L^{2}}\\
\left({\scriptstyle \text{since }\mathscr{S}_{\alpha,0}^{2,2}\left(\smash{\R^{2}}\right)=L^{2}\left(\smash{\R^{2}}\right)\text{ with equivalent norms}}\right) & \asymp\left\Vert S^{\left(\delta\right)}\left[c^{\left(f\right)}-\Indicator_{J_{N}^{\left(f\right)}}\cdot c^{\left(f\right)}\right]\right\Vert _{\mathscr{S}_{\alpha,0}^{2,2}}\\
 & \leq\vertiii{\smash{S^{\left(\delta\right)}}}_{C_{u^{0}}^{2,2}\to\mathscr{S}_{\alpha,0}^{2,2}}\cdot\left\Vert c^{\left(f\right)}-\Indicator_{J_{N}^{\left(f\right)}}\cdot c^{\left(f\right)}\right\Vert _{\ell^{2}}\\
\left({\scriptstyle \text{eq. }\eqref{eq:StechkinEstimate}}\right) & \leq\vertiii{\smash{S^{\left(\delta\right)}}}_{C_{u^{0}}^{2,2}\to\mathscr{S}_{\alpha,0}^{2,2}}\cdot\left\Vert \smash{c^{\left(f\right)}}\right\Vert _{\ell^{p}}\cdot N^{-\left(\frac{1}{p}-\frac{1}{2}\right)}\\
\left({\scriptstyle \text{eq. }\eqref{eq:NTermApproximationellPEstimate}}\right) & \leq C_{1}\cdot\vertiii{\smash{C^{\left(\delta\right)}}}_{\mathscr{S}_{\alpha,s}^{p,p}\to\ell^{p}}\cdot\vertiii{\smash{S^{\left(\delta\right)}}}_{C_{u^{0}}^{2,2}\to\mathscr{S}_{\alpha,0}^{2,2}}\cdot N^{-\left(\frac{1}{p}-\frac{1}{2}\right)}\\
\left({\scriptstyle \text{since }\frac{1}{p}-\frac{1}{2}\geq\frac{\beta}{2}-\sigma}\right) & \leq C_{1}\cdot\vertiii{\smash{C^{\left(\delta\right)}}}_{\mathscr{S}_{\alpha,s}^{p,p}\to\ell^{p}}\cdot\vertiii{\smash{S^{\left(\delta\right)}}}_{C_{u^{0}}^{2,2}\to\mathscr{S}_{\alpha,0}^{2,2}}\cdot N^{-\left(\frac{\beta}{2}-\sigma\right)}
\end{align*}
for all $N\in\N$ and $f\in\mathcal{E}^{\beta}\left(\R^{2};\nu\right)$.
Since $p$ only depends on $\sigma,\beta$, this easily yields the
desired claim.
\end{proof}
We close this section by making the assumptions of Theorem \ref{thm:CartoonApproximationWithAlphaShearlets}
more transparent:
\begin{rem}
\label{rem:CartoonApproximationConstantSimplification}With the choices
of $\alpha,p_{0},q_{0},s_{0},s_{1}$ from Theorem \ref{thm:CartoonApproximationWithAlphaShearlets},
one can choose $\varepsilon=\varepsilon\left(\beta\right)\in\left(0,1\right]$
such that the constants $\left\lceil p_{0}^{-1}\cdot\left(2+\varepsilon\right)\right\rceil $
and $\Lambda_{0},\dots,\Lambda_{3}$ from Theorem \ref{thm:ReallyNiceShearletAtomicDecompositionConditions}
satisfy $\Lambda_{1}\leq3$, as well as
\[
\left\lceil \frac{2+\varepsilon}{p_{0}}\right\rceil =\begin{cases}
3, & \text{if }\beta<2,\\
4, & \text{if }\beta=2,
\end{cases}\quad\Lambda_{0}\leq\begin{cases}
11, & \text{if }\beta<2,\\
12, & \text{if }\beta=2,
\end{cases}\quad\Lambda_{2}<\begin{cases}
11, & \text{if }\beta<2,\\
12, & \text{if }\beta=2
\end{cases}\quad\text{ and }\quad\Lambda_{3}<\begin{cases}
14, & \text{if }\beta<2,\\
16, & \text{if }\beta=2.
\end{cases}
\]
Thus, in view of Remark \ref{rem:NiceTensorConditionsForUnconnectedCovering}
(which refers to Corollary \ref{cor:ReallyNiceAlphaShearletTensorAtomicDecompositionConditions}),
it suffices in \emph{every} case to have $\varphi\in C_{c}^{12}\left(\R^{2}\right)$
and $\psi=\psi_{1}\otimes\psi_{2}$ with $\psi_{1}\in C_{c}^{15}\left(\R\right)$
and $\psi_{2}\in C_{c}^{19}\left(\R\right)$ and with the following
additional properties:

\begin{enumerate}[leftmargin=0.6cm]
\item $\widehat{\varphi}\left(\xi\right)\neq0$ for all $\xi\in\left[-1,1\right]^{2}$,
\item $\widehat{\psi_{1}}\left(\xi\right)\neq0$ for $\frac{1}{3}\leq\left|\xi\right|\leq3$
and $\widehat{\psi_{2}}\left(\xi\right)\neq0$ for all $\xi\in\left[-3,3\right]$,
\item We have $\frac{\d^{\ell}}{\d\xi^{\ell}}\big|_{\xi=0}\widehat{\psi_{1}}=0$
for $0\leq\ell\leq6$. In case of $\beta<2$, it even suffices to
have this for $0\leq\ell\leq5$.\qedhere
\end{enumerate}
\end{rem}
\begin{proof}
We have $p_{0}^{-1}=\frac{1+\beta}{2}$ and thus $\frac{2}{p_{0}}=1+\beta\in\left(2,3\right)$
in case of $\beta<2$. Hence, $\frac{2+\varepsilon}{p_{0}}\in\left(2,3\right)$
for $\varepsilon=\varepsilon\left(\beta\right)$ sufficiently small.
In case of $\beta=2$, we get $\frac{2+\varepsilon}{p_{0}}=3+\frac{\varepsilon}{p_{0}}\in\left(3,4\right)$
for $\varepsilon>0$ sufficiently small. This establishes the claimed
identity for $N_{0}:=\left\lceil p_{0}^{-1}\cdot\left(2+\varepsilon\right)\right\rceil $.
For the remainder of the proof, we always assume that $\varepsilon$
is chosen small enough for this identity to hold.

Next, the constant $\Lambda_{1}$ from Theorem \ref{thm:ReallyNiceShearletAtomicDecompositionConditions}
satisfies because of $\alpha=\beta^{-1}$ that
\begin{align*}
\Lambda_{1} & =\varepsilon+\frac{1}{\min\left\{ p_{0},q_{0}\right\} }+\max\left\{ 0,\,\left(1+\alpha\right)\left(\frac{1}{p_{0}}-1\right)-s_{0}\right\} \\
 & =\varepsilon+\frac{1+\beta}{2}+\left(1+\beta^{-1}\right)\left(\frac{1+\beta}{2}-1\right)\\
 & =\varepsilon+\frac{1}{2}+\beta-\frac{\beta^{-1}}{2},
\end{align*}
which is strictly increasing with respect to $\beta>0$. Therefore,
we always have $\Lambda_{1}\leq\varepsilon+\frac{1}{2}+2-\frac{2^{-1}}{2}=2+\frac{1}{4}+\varepsilon\leq3$
for $\varepsilon\leq\frac{3}{4}$.

Furthermore, the constant $\Lambda_{0}$ from Theorem \ref{thm:ReallyNiceShearletAtomicDecompositionConditions}
is—because of $p_{0}=\frac{2}{1+\beta}<1$—given by
\begin{align*}
\Lambda_{0} & =2\varepsilon+3+\max\left\{ \frac{1-\alpha}{\min\left\{ p_{0},q_{0}\right\} }+\frac{1-\alpha}{p_{0}}+1+\alpha+\left\lceil \frac{2+\varepsilon}{p_{0}}\right\rceil +s_{1},\,2\right\} \\
 & =2\varepsilon+3+\max\left\{ \frac{\left(1-\beta^{-1}\right)\left(1+\beta\right)}{2}+\frac{\left(1-\beta^{-1}\right)\left(1+\beta\right)}{2}+1+\beta^{-1}+N_{0}+\frac{1}{2}\left(1+\beta\right),\,2\right\} \\
 & =2\varepsilon+3+\max\left\{ \frac{3}{2}+\frac{3}{2}\beta+N_{0},\,2\right\} \\
 & \leq7+N_{0}+\frac{1}{2}+2\varepsilon.
\end{align*}
Since $N_{0}=3$ for $\beta<2$, this easily yields $\Lambda_{0}\leq11$
for $\varepsilon\leq\frac{1}{4}$. Similarly, we get $\Lambda_{0}\leq12$
for $\beta=2$ and $\varepsilon\leq\frac{1}{4}$.

Likewise, the constant $\Lambda_{2}$ from Theorem \ref{thm:ReallyNiceShearletAtomicDecompositionConditions}
satisfies
\begin{align*}
\Lambda_{2} & =\varepsilon+\max\left\{ 2,\,\left(1+\alpha\right)\left(1+\frac{1}{p_{0}}+\left\lceil \frac{2+\varepsilon}{p_{0}}\right\rceil \right)+s_{1}\right\} \\
 & =\varepsilon+\max\left\{ 2,\,\left(1+\beta^{-1}\right)\left(1+\frac{1+\beta}{2}+N_{0}\right)+\frac{1+\beta}{2}\right\} \\
 & =\varepsilon+\max\left\{ 2,\,\frac{5}{2}+\beta+N_{0}+\frac{3}{2}\beta^{-1}+\beta^{-1}N_{0}\right\} \\
 & =\varepsilon+\frac{5}{2}+\beta+N_{0}+\frac{3}{2}\beta^{-1}+\beta^{-1}N_{0}.
\end{align*}
Hence, in case of $\beta<2$, we thus get $\Lambda_{2}=\varepsilon+\frac{11}{2}+\beta+\frac{9}{2}\beta^{-1}=:\varepsilon+g\left(\beta\right)$,
where $g:\left(0,\infty\right)\to\R$ is \emph{strictly} convex with
$g\left(1\right)=11$ and $g\left(2\right)=\frac{39}{4}<11$, so that
$g\left(\beta\right)<11$ for all $\beta\in\left(1,2\right)$. Thus,
$\Lambda_{2}<11$ for sufficiently small $\varepsilon=\varepsilon\left(\beta\right)>0$.
Finally, for $\beta=2$, we get $\Lambda_{2}=11+\frac{1}{4}+\varepsilon<12$
for $0<\varepsilon<\frac{3}{4}$.

As the final constant, we consider
\begin{align*}
\Lambda_{3} & =\varepsilon+\max\left\{ \frac{1-\alpha}{\min\left\{ p_{0},q_{0}\right\} }+\frac{3-\alpha}{p_{0}}+2\left\lceil \frac{2+\varepsilon}{p_{0}}\right\rceil +1+\alpha+s_{1},\,\frac{2}{\min\left\{ p_{0},q_{0}\right\} }+\frac{2}{p_{0}}+\left\lceil \frac{2+\varepsilon}{p_{0}}\right\rceil \right\} \\
 & =\varepsilon+\max\left\{ \frac{\left(1-\beta^{-1}\right)\left(1+\beta\right)}{2}+\frac{\left(3-\beta^{-1}\right)\left(1+\beta\right)}{2}+2N_{0}+1+\beta^{-1}+\frac{1+\beta}{2},\,1+\beta+1+\beta+N_{0}\right\} \\
 & =\varepsilon+\max\left\{ \frac{5}{2}\beta+\frac{5}{2}+2N_{0},\,2+2\beta+N_{0}\right\} \\
 & =\varepsilon+\frac{5}{2}\beta+\frac{5}{2}+2N_{0}.
\end{align*}
In case of $\beta=2$, this means $\Lambda_{3}=\varepsilon+15+\frac{1}{2}<16$
for $0<\varepsilon<\frac{1}{2}$. Finally, for $\beta<2$, we get
$\Lambda_{3}<13+\frac{1}{2}+\varepsilon<14$ for $0<\varepsilon<\frac{1}{2}$.
\end{proof}

\section{Embeddings between \texorpdfstring{$\alpha$}{α}-shearlet smoothness
spaces}

\label{sec:EmbeddingsBetweenAlphaShearletSmoothness}In the preceding
sections, we saw that the $\alpha$-shearlet smoothness spaces $\mathscr{S}_{\alpha,s}^{p,q}\left(\R^{2}\right)$
simultaneously characterize analysis and synthesis sparsity with respect
to (sufficiently nice) $\alpha$-shearlet systems; see in particular
Theorem \ref{thm:AnalysisAndSynthesisSparsityAreEquivalent}. Since
we have a whole family of $\alpha$-shearlet systems, parametrized
by $\alpha\in\left[0,1\right]$, it is natural to ask if the different
systems are related in some way, e.g.\@ if $\ell^{p}$-sparsity,
$p\in\left(0,2\right)$, with respect to $\alpha_{1}$-shearlet systems
implies $\ell^{q}$-sparsity with respect to $\alpha_{2}$-shearlet
systems, for some $q\in\left(0,2\right)$.

In view of Theorem \ref{thm:AnalysisAndSynthesisSparsityAreEquivalent},
this is equivalent to asking whether there is an \textbf{embedding}
\begin{equation}
\mathscr{S}_{\alpha_{1},\left(1+\alpha_{1}\right)\left(p^{-1}-2^{-1}\right)}^{p,p}\left(\R^{2}\right)\hookrightarrow\mathscr{S}_{\alpha_{2},\left(1+\alpha_{2}\right)\left(q^{-1}-2^{-1}\right)}^{q,q}\left(\R^{2}\right).\label{eq:EmbeddingSectionSparsityEmbedding}
\end{equation}
Note, however, that equation \eqref{eq:EmbeddingSectionSparsityEmbedding}
is equivalent to asking whether one can deduce $\ell^{q}$-sparsity
with respect to $\alpha_{2}$-shearlets from $\ell^{p}$-sparsity
with respect to $\alpha_{1}$-shearlets \emph{without any additional
information}. If one \emph{does} have additional information, e.g.,
if one is only interested in functions $f$ with $\supp f\subset\Omega$,
where $\Omega\subset\R^{2}$ is fixed and bounded, then the embedding
in equation \eqref{eq:EmbeddingSectionSparsityEmbedding} is a sufficient,
but in general \emph{not} a necessary criterion for guaranteeing that
$f$ is $\ell^{q}$-sparse with respect to $\alpha_{2}$-shearlets
if it is $\ell^{p}$-sparse with respect to $\alpha_{1}$-shearlets.

More general than equation \eqref{eq:EmbeddingSectionSparsityEmbedding},
we will \emph{completely characterize} the existence of the embedding
\begin{equation}
\mathscr{S}_{\alpha_{1},s_{1}}^{p_{1},q_{1}}\left(\R^{2}\right)\hookrightarrow\mathscr{S}_{\alpha_{2},s_{2}}^{p_{2},q_{2}}\left(\R^{2}\right)\label{eq:EmbeddingSectionGeneralEmbedding}
\end{equation}
for arbitrary $p_{1},p_{2},q_{1},q_{2}\in\left(0,\infty\right]$,
$\alpha_{1},\alpha_{2}\in\left[0,1\right]$ and $s_{1},s_{2}\in\R$.
As an application, we will then see that the embedding \eqref{eq:EmbeddingSectionSparsityEmbedding}
is \emph{never} fulfilled for $p,q\in\left(0,2\right)$, but that
if one replaces the left-hand side of the embedding \eqref{eq:EmbeddingSectionSparsityEmbedding}
by $\mathscr{S}_{\alpha_{1},\varepsilon+\left(1+\alpha_{1}\right)\left(p^{-1}-2^{-1}\right)}^{p,p}\left(\R^{2}\right)$
for some $\varepsilon>0$, then the embedding holds for suitable $p,q\in\left(0,2\right)$.
Thus, \emph{without further information}, $\ell^{p}$-sparsity with
respect to $\alpha_{1}$-shearlets \emph{never} implies nontrivial
$\ell^{q}$-sparsity with respect to $\alpha_{2}$-shearlets; but
one can still transfer sparsity in some sense if one has $\ell^{p}$-sparsity
with respect to $\alpha_{1}$-shearlets, together with a certain \emph{decay
of the $\alpha_{1}$-shearlet coefficients with the scale}.

We remark that the results in this section can be seen as a continuation
of the work in \cite{AlphaMolecules}: In that paper, the authors
develop the framework of \textbf{$\alpha$-molecules} which allows
one to transfer (analysis) sparsity results between different systems
that employ $\alpha$-parabolic scaling; for example between $\alpha$-shearlets
and $\alpha$-curvelets. Before \cite[Theorem 4.2]{AlphaMolecules},
the authors note that ``\emph{it might though be very interesting
for future research to also let $\alpha$-molecules for different
$\alpha$'s interact.}'' In a way, this is precisely what we are
doing in this section, although we focus on the special case of $\alpha$-shearlets
instead of (more general) $\alpha$-molecules.

\medskip{}

In order to characterize the embedding \eqref{eq:EmbeddingSectionGeneralEmbedding},
we will invoke the embedding theory for decomposition spaces\cite{DecompositionEmbedding}
that was developed by one of the authors; this will greatly simplify
the proof, since we do not need to start from scratch. In order for
the theory in \cite{DecompositionEmbedding} to be applicable to an
embedding $\DecompSp{\CalQ}{p_{1}}{\ell_{w}^{q_{1}}}{}\hookrightarrow\DecompSp{\CalP}{p_{2}}{\ell_{v}^{q_{2}}}{}$,
the two coverings $\CalQ=\left(Q_{i}\right)_{i\in I}$ and $\CalP=\left(P_{j}\right)_{j\in J}$
need to be \emph{compatible} in a certain sense. For this, it suffices
if $\CalQ$ is \textbf{almost subordinate} to $\CalP$ (or vice versa);
roughly speaking, this means that the covering $\CalQ$ is \emph{finer}
than $\CalP$. Precisely, it means that each set $Q_{i}$ is contained
in $P_{j_{i}}^{n\ast}$ for some $j_{i}\in J$, where $n\in\N$ is
fixed and where $P_{j_{i}}^{n\ast}=\bigcup_{\ell\in j_{i}^{n\ast}}P_{\ell}$.
Here, the sets $j^{n\ast}$ are defined inductively, via $L^{\ast}:=\bigcup_{\ell\in L}\ell^{\ast}$
(with $\ell^{\ast}$ as in Definition \ref{def:AlmostStructuredCovering})
and with $L^{\left(n+1\right)\ast}:=\left(L^{n\ast}\right)^{\ast}$
for $L\subset J$. The following lemma establishes this \emph{compatibility}
between different $\alpha$-shearlet coverings.
\begin{lem}
\label{lem:AlphaShearletCoveringSubordinateness}Let $0\leq\alpha_{1}\leq\alpha_{2}\leq1$.
Then $\CalS^{\left(\alpha_{1}\right)}=\left(\smash{S_{i}^{\left(\alpha_{1}\right)}}\right)_{i\in I^{\left(\alpha_{1}\right)}}$
is almost subordinate to $\CalS^{\left(\alpha_{2}\right)}=\left(\smash{S_{j}^{\left(\alpha_{2}\right)}}\right)_{j\in I^{\left(\alpha_{2}\right)}}$.
\end{lem}
\begin{proof}
Since we have $\bigcup_{i\in I^{\left(\alpha_{1}\right)}}S_{i}^{\left(\alpha_{1}\right)}=\R^{2}=\bigcup_{j\in I^{\left(\alpha_{2}\right)}}S_{j}^{\left(\alpha_{2}\right)}$
and since all of the sets $S_{i}^{\left(\alpha_{1}\right)}$ and $S_{j}^{\left(\alpha_{2}\right)}$
are open and path-connected, \cite[Corollary 2.13]{DecompositionEmbedding}
shows that it suffices to show that $\CalS^{\left(\alpha_{1}\right)}$
is \textbf{weakly subordinate} to $\CalS^{\left(\alpha_{2}\right)}$.
This means that we have $\sup_{i\in I^{\left(\alpha_{1}\right)}}\left|L_{i}\right|<\infty$,
with
\[
L_{i}:=\left\{ j\in I^{\left(\alpha_{2}\right)}\with S_{j}^{\left(\alpha_{2}\right)}\cap S_{i}^{\left(\alpha_{1}\right)}\neq\emptyset\right\} \qquad\text{ for }i\in I^{\left(\alpha_{1}\right)}.
\]

To show this, we first consider only the case $i=\left(n,m,\varepsilon,0\right)\in I_{0}^{\left(\alpha_{1}\right)}$
and let $j\in L_{i}$ be arbitrary. We now distinguish several cases
regarding $j$:

\textbf{Case 1}: We have $j=\left(k,\ell,\beta,0\right)\in I_{0}^{\left(\alpha_{2}\right)}$.
Let $\left(\xi,\eta\right)\in S_{i}^{\left(\alpha_{1}\right)}\cap S_{j}^{\left(\alpha_{2}\right)}$.
In view of equation \eqref{eq:deltazeroset}, this implies $\xi\in\varepsilon\left(2^{n}/3,3\cdot2^{n}\right)\cap\beta\left(2^{k}/3,3\cdot2^{k}\right)$,
so that in particular $\varepsilon=\beta$. Furthermore, we see $2^{k}/3<\left|\xi\right|<3\cdot2^{n}$,
which yields $2^{n-k}>\frac{1}{9}>2^{-4}$. Analogously, we get $2^{n}/3<\left|\xi\right|<3\cdot2^{k}$
and thus $2^{n-k}<9<2^{4}$. Together, these considerations imply
$\left|n-k\right|\leq3$.

Furthermore, since $\left(\xi,\eta\right)\in S_{i}^{\left(\alpha_{1}\right)}\cap S_{j}^{\left(\alpha_{2}\right)}$,
equation \eqref{eq:deltazeroset} also shows
\[
\frac{\eta}{\xi}=\frac{\varepsilon\eta}{\varepsilon\xi}=\frac{\beta\eta}{\beta\xi}\in2^{n\left(\alpha_{1}-1\right)}\left(m-1,m+1\right)\cap2^{k\left(\alpha_{2}-1\right)}\left(\ell-1,\ell+1\right).
\]
Hence, we get the two inequalities
\[
2^{n\left(\alpha_{1}-1\right)}\left(m+1\right)>2^{k\left(\alpha_{2}-1\right)}\left(\ell-1\right)\qquad\text{ and }\qquad2^{n\left(\alpha_{1}-1\right)}\left(m-1\right)<2^{k\left(\alpha_{2}-1\right)}\left(\ell+1\right)
\]
and thus
\[
\ell<\left(m+1\right)2^{n\alpha_{1}-k\alpha_{2}+k-n}+1\qquad\text{ and }\qquad\ell>\left(m-1\right)2^{n\alpha_{1}-k\alpha_{2}+k-n}-1.
\]
In other words,
\[
\ell\in\left(\left(m-1\right)2^{n\alpha_{1}-k\alpha_{2}+k-n}-1,\left(m+1\right)2^{n\alpha_{1}-k\alpha_{2}+k-n}+1\right)\cap\Z=:s_{m}^{\left(n,k\right)}.
\]
But since any interval $I=\left(A,B\right)$ with $A\leq B$ satisfies
$\left|I\cap\Z\right|\leq B-A+1$, the cardinality of $s_{m}^{\left(n,k\right)}$
can be estimated by
\[
\begin{aligned}\left|s_{m}^{\left(n,k\right)}\right| & \leq\left(m+1\right)2^{n\alpha_{1}-k\alpha_{2}+k-n}+1-\left(m-1\right)2^{n\alpha_{1}-k\alpha_{2}+k-n}+1+1\\
 & =3+2\cdot2^{n\alpha_{1}-k\alpha_{2}+k-n}\\
\left({\scriptstyle \text{since }\left|n-k\right|\leq3}\right) & \leq3+2\cdot2^{n\alpha_{1}-(n-3)\alpha_{2}+3}\\
 & =3+2^{4}\cdot2^{n(\alpha_{1}-\alpha_{2})}\cdot2^{3\alpha_{2}}\\
\left({\scriptstyle \text{since }\alpha_{1}-\alpha_{2}\leq0\text{ and }\alpha_{2}\leq1}\right) & \leq3+2^{7}=131.
\end{aligned}
\]
Thus,
\[
L_{i}^{\left(0\right)}:=\left\{ j=\left(k,\ell,\beta,0\right)\in I_{0}^{\left(\alpha_{2}\right)}\left|S_{j}^{\left(\alpha_{2}\right)}\cap S_{i}^{\left(\alpha_{1}\right)}\neq\emptyset\right.\right\} \subset\bigcup_{t=n-3}^{n+3}\left(\left\{ t\right\} \times s_{m}^{\left(n,t\right)}\times\left\{ \pm1\right\} \times\left\{ 0\right\} \right),
\]
which is a finite set, with at most $7\cdot131\cdot2=1834$ elements.

\medskip{}

\textbf{Case 2}: We have $j=(k,\ell,\beta,1)\in I_{0}^{\left(\alpha_{2}\right)}$.
Let $\left(\xi,\eta\right)\in S_{i}^{\left(\alpha_{1}\right)}\cap S_{j}^{\left(\alpha_{2}\right)}$.
With similar arguments as in the previous case, this implies $\xi\in\varepsilon\left(2^{n}/3,3\cdot2^{n}\right)$,
$\eta\in\beta\left(2^{k}/3,3\cdot2^{k}\right)$ and $\frac{\eta}{\xi}\in2^{n\left(\alpha_{1}-1\right)}\left(m-1,m+1\right)$,
as well as $\frac{\xi}{\eta}\in2^{k\left(\alpha_{2}-1\right)}\left(\ell-1,\ell+1\right)$.
Furthermore since $\left(\xi,\eta\right)\in S_{n,m,\varepsilon,0}^{\left(\alpha_{1}\right)}$
and $\left(\xi,\eta\right)\in S_{k,\ell,\beta,1}^{\left(\alpha_{2}\right)}$,
we know from Lemma \ref{lem:AlphaShearletCoveringAuxiliary} that
$\left|\eta\right|<3\left|\xi\right|$ and $\left|\xi\right|<3\left|\eta\right|$.

Thus, $2^{k}/3<\left|\eta\right|<3\cdot\left|\xi\right|<3\cdot3\cdot2^{n}$
and hence $2^{k-n}<27<2^{5}$. Likewise, $2^{n}/3<\left|\xi\right|<3\cdot\left|\eta\right|<3\cdot3\cdot2^{k}$
and hence $2^{n-k}<2^{5}$, so that we get $\left|n-k\right|\leq4$.
Now, we distinguish two subcases regarding $\left|\eta/\xi\right|$:

\begin{enumerate}
\item We have $\left|\eta/\xi\right|>1$. Because of $\left|m\right|\leq\left\lceil 2^{n\left(1-\alpha_{1}\right)}\right\rceil \leq1+2^{n\left(1-\alpha_{1}\right)}$,
this implies
\[
1<\left|\frac{\eta}{\xi}\right|<2^{n\left(\alpha_{1}-1\right)}\left(\left|m\right|+1\right)\leq2^{n\left(\alpha_{1}-1\right)}\left(2^{n\left(1-\alpha_{1}\right)}+1+1\right)=1+2\cdot2^{n\left(\alpha_{1}-1\right)}
\]
and hence 
\[
\frac{1}{1+2\cdot2^{n\left(\alpha_{1}-1\right)}}<\left|\frac{\xi}{\eta}\right|<1.
\]
Furthermore, we know $\left|\xi/\eta\right|<2^{k\left(\alpha_{2}-1\right)}\left(\left|\ell\right|+1\right)$,
so that we get
\[
\frac{1}{1+2\cdot2^{n\left(\alpha_{1}-1\right)}}<\left|\frac{\xi}{\eta}\right|<2^{k\left(\alpha_{2}-1\right)}\left(\left|\ell\right|+1\right)\qquad\text{ and hence }\qquad\left|\ell\right|>\frac{2^{k\left(1-\alpha_{2}\right)}}{1+2\cdot2^{n\left(\alpha_{1}-1\right)}}-1.
\]
Thus, we have
\[
\left|\ell\right|\in\Z\cap\left(\frac{2^{k\left(1-\alpha_{2}\right)}}{1+2\cdot2^{n\left(\alpha_{1}-1\right)}}-1,\left\lceil \smash{2^{k\left(1-\alpha_{2}\right)}}\right\rceil \right]\subset\Z\cap\left(\frac{2^{k\left(1-\alpha_{2}\right)}}{1+2\cdot2^{n\left(\alpha_{1}-1\right)}}-1,2^{k\left(1-\alpha_{2}\right)}+1\right)=:s^{\left(n,k\right)},
\]
where as above 
\[
\begin{aligned}\left|s^{\left(n,k\right)}\right|\leq2^{k\left(1-\alpha_{2}\right)}+1-\frac{2^{k\left(1-\alpha_{2}\right)}}{1+2\cdot2^{n\left(\alpha_{1}-1\right)}}+1+1 & =3+2^{k\left(1-\alpha_{2}\right)}\left(1-\frac{1}{1+2\cdot2^{n\left(\alpha_{1}-1\right)}}\right)\\
 & =3+2^{k\left(1-\alpha_{2}\right)}\frac{2\cdot2^{n\left(\alpha_{1}-1\right)}}{1+2\cdot2^{n\left(\alpha_{1}-1\right)}}\\
 & \leq3+2\cdot2^{k\left(1-\alpha_{2}\right)-n\left(1-\alpha_{1}\right)}\\
\left({\scriptstyle \text{since }1-\alpha_{2}\geq0\text{ and }\left|n-k\right|\leq4}\right) & \leq3+2\cdot2^{\left(n+4\right)\left(1-\alpha_{2}\right)-n\left(1-\alpha_{1}\right)}\\
 & =3+2\cdot2^{4\left(1-\alpha_{2}\right)}2^{n\left(\alpha_{1}-\alpha_{2}\right)}\\
\left({\scriptstyle \text{since }\alpha_{1}-\alpha_{2}\leq0\text{ and }\alpha_{2}\geq0}\right) & \leq3+2\cdot2^{4}=35.
\end{aligned}
\]
Finally, note that $\left|\ell\right|\in s^{\left(n,k\right)}$ implies
$\ell\in\pm s^{\left(n,k\right)}$, with $\left|\pm s^{\left(n,k\right)}\right|\leq70$.
\item We have $\left|\eta/\xi\right|\leq1$. This yields $1\leq\left|\xi/\eta\right|<2^{k\left(\alpha_{2}-1\right)}\left(\left|\ell\right|+1\right)$
and hence $\left|\ell\right|>2^{k\left(1-\alpha_{2}\right)}-1$. Thus,
we have 
\[
\left|\ell\right|\in\Z\cap\left(2^{k\left(1-\alpha_{2}\right)}-1,\left\lceil \smash{2^{k\left(1-\alpha_{2}\right)}}\right\rceil \right]\subset\Z\cap\left(2^{k\left(1-\alpha_{2}\right)}-1,2^{k\left(1-\alpha_{2}\right)}+1\right)=:\tilde{s}^{\left(n,k\right)},
\]
where one easily sees $\left|\tilde{s}^{\left(n,k\right)}\right|\leq3$
and then $\ell\in\pm\tilde{s}^{\left(n,k\right)}$ with $\left|\pm\tilde{s}^{\left(n,k\right)}\right|\leq6$.
\end{enumerate}
All in all, we see 
\[
L_{i}^{\left(1\right)}:=\left\{ j=\left(k,\ell,\beta,1\right)\in I_{0}^{\left(\alpha_{2}\right)}\left|S_{j}^{\left(\alpha_{2}\right)}\cap S_{i}^{\left(\alpha_{1}\right)}\neq\emptyset\right.\right\} \subset\bigcup_{t=n-4}^{n+4}\left[\left\{ t\right\} \times\left(\left[\pm\smash{s^{\left(n,t\right)}}\right]\cup\left[\pm\smash{\tilde{s}^{\left(n,t\right)}}\right]\right)\times\left\{ \pm1\right\} \times\left\{ 1\right\} \right]
\]
and hence $\left|\smash{L_{i}^{(1)}}\right|\leq9\cdot\left(70+6\right)\cdot2=1368$. 

In total, Cases 1 and 2 show because of $L_{i}\subset L_{i}^{(0)}\cup L_{i}^{(1)}\cup\left\{ 0\right\} $
that $\left|L_{i}\right|\leq\left|\smash{L_{i}^{(0)}}\right|+\left|\smash{L_{i}^{(1)}}\right|+\left|\left\{ 0\right\} \right|\leq3203$
for all $i=\left(n,m,\varepsilon,0\right)\in I_{0}^{\left(\alpha_{1}\right)}$.

\medskip{}

But in case of $i=\left(n,m,\varepsilon,1\right)\in I_{0}^{\left(\alpha_{1}\right)}$,
we get the same result. Indeed, if we set $\tilde{\gamma}:=1-\gamma$
for $\gamma\in\left\{ 0,1\right\} $, then
\[
\begin{aligned}I_{0}^{\left(\alpha_{2}\right)}\cap L_{\left(n,m,\varepsilon,1\right)} & =\left\{ \left(k,\ell,\beta,\gamma\right)\in I_{0}^{\left(\alpha_{2}\right)}\left|S_{k,\ell,\beta,\gamma}^{\left(\alpha_{2}\right)}\cap S_{n,m,\varepsilon,1}^{\left(\alpha_{1}\right)}\neq\emptyset\right.\right\} \\
 & =\left\{ \left(k,\ell,\beta,\gamma\right)\in I_{0}^{\left(\alpha_{2}\right)}\left|RS_{k,\ell,\beta,\tilde{\gamma}}^{\left(\alpha_{2}\right)}\cap RS_{n,m,\varepsilon,0}^{\left(\alpha_{1}\right)}\neq\emptyset\right.\right\} \\
 & =\left\{ \left(k,\ell,\beta,\tilde{\gamma}\right)\in I_{0}^{\left(\alpha_{2}\right)}\left|S_{k,\ell,\beta,\gamma}^{\left(\alpha_{2}\right)}\cap S_{n,m,\varepsilon,0}^{\left(\alpha_{1}\right)}\neq\emptyset\right.\right\} \\
 & =\left\{ \left(k,\ell,\beta,\tilde{\gamma}\right)\with\left(k,\ell,\beta,\gamma\right)\in I_{0}^{\left(\alpha_{2}\right)}\cap L_{\left(n,m,\varepsilon,0\right)}\right\} ,
\end{aligned}
\]
and thus $\left|\smash{I_{0}^{\left(\alpha_{2}\right)}}\cap L_{\left(n,m,\varepsilon,1\right)}\right|=\left|\smash{I_{0}^{\left(\alpha_{2}\right)}}\cap L_{\left(n,m,\varepsilon,0\right)}\right|\leq3202$,
so that $\left|L_{\left(n,m,\varepsilon,1\right)}\right|\leq3203$.

\medskip{}

It remains to consider the case $i=0$. But for $\xi\in S_{0}^{\left(\alpha_{1}\right)}=\left(-1,1\right)^{2}$,
we have $1+\left|\xi\right|\leq3$. Conversely, Lemma \ref{lem:AlphaShearletWeightIsModerate}
shows $1+\left|\xi\right|\geq\frac{1}{3}\cdot w_{j}=2^{k}/3$ for
all $\xi\in S_{j}^{\left(\alpha_{2}\right)}$ and all $j=\left(k,\ell,\beta,\gamma\right)\in I_{0}^{\left(\alpha_{2}\right)}$.
Hence, $j\in L_{0}$ can only hold if $2^{k}/3\leq3$, i.e., if $k\leq3$.
Since we also have $\left|\ell\right|\leq\left\lceil 2^{k\left(1-\alpha_{2}\right)}\right\rceil \leq2^{k}\leq2^{3}=8$,
this implies
\[
L_{0}\subset\left\{ 0\right\} \cup\left[\left\{ 0,1,2,3\right\} \times\left\{ -8,\dots,8\right\} \times\left\{ \pm1\right\} \times\left\{ 0,1\right\} \right]
\]
and hence $\left|L_{0}\right|\leq1+4\cdot17\cdot2\cdot2=273\leq3203$.

\medskip{}

In total, we have shown $\sup_{i\in I^{\left(\alpha_{1}\right)}}\left|L_{i}\right|\leq3203<\infty$,
so that $\CalS^{\left(\alpha_{1}\right)}$ is weakly subordinate to
$\CalS^{\left(\alpha_{2}\right)}$. As seen at the beginning of the
proof, this suffices.
\end{proof}
Now that we have seen that $\CalS^{\left(\alpha_{1}\right)}$ is almost
subordinate to $\CalS^{\left(\alpha_{2}\right)}$ for $\alpha_{1}\leq\alpha_{2}$,
the theory from \cite{DecompositionEmbedding} is applicable. But
the resulting conditions simplify greatly, if in addition to the coverings,
also the employed weights are compatible in a certain sense. Precisely,
for two coverings $\CalQ=\left(Q_{i}\right)_{i\in I}$ and $\CalP=\left(P_{j}\right)_{j\in J}$
and for a weight $w=\left(w_{i}\right)_{i\in I}$ on the index set
of $\CalQ$, we say that $w$ is \textbf{relatively $\CalP$-moderate},
if there is a constant $C>0$ with
\[
w_{i}\leq C\cdot w_{\ell}\qquad\text{for all }i,\ell\in I\text{ with }Q_{i}\cap P_{j}\neq\emptyset\neq Q_{\ell}\cap P_{j}\text{ for some }j\in J.
\]
Likewise, the covering $\CalQ=\left(T_{i}Q_{i}'+b_{i}\right)_{i\in I}$
is called relatively $\CalP$-moderate, if the weight $\left(\left|\det T_{i}\right|\right)_{i\in I}$
is relatively $\CalP$-moderate. Our next lemma shows that these two
conditions are satisfied if $\CalQ$ and $\CalP$ are two $\alpha$-shearlet
coverings.
\begin{lem}
\label{lem:AlphaShearletRelativelyModerate}Let $0\leq\alpha_{1}\leq\alpha_{2}\leq1$
and let $\CalS^{\left(\alpha_{1}\right)}$ and $\CalS^{\left(\alpha_{2}\right)}$
be the associated $\alpha$-shearlet coverings. Then the following
hold: 

\begin{enumerate}
\item $\CalS^{\left(\alpha_{1}\right)}$ is relatively $\CalS^{\left(\alpha_{2}\right)}$-moderate. 
\item For arbitrary $s\in\R$, the weight $w^{s}=\left(w_{i}^{s}\right)_{i\in I^{(\alpha_{1})}}$
with $w=\left(w_{i}\right)_{i\in I^{\left(\alpha_{1}\right)}}$ as
in Definition \ref{def:AlphaShearletCovering} (considered as a weight
for $\CalS^{(\alpha_{1})}$) is relatively $\CalS^{\left(\alpha_{2}\right)}$-moderate.
More precisely, we have $39^{-\left|s\right|}\cdot w_{j}^{s}\leq w_{i}^{s}\leq39^{\left|s\right|}\cdot w_{j}^{s}$
for all $i\in I^{(\alpha_{1})}$ and $j\in I^{(\alpha_{2})}$ with
$S_{i}^{(\alpha_{1})}\cap S_{j}^{(\alpha_{2})}\neq\emptyset$.\qedhere
\end{enumerate}
\end{lem}
\begin{proof}
It is not hard to see $\left|\det\smash{T_{i}^{\left(\alpha_{1}\right)}}\right|=w_{i}^{1+\alpha_{1}}$
for all $i\in I^{\left(\alpha_{1}\right)}$. Thus, the second claim
implies the first one.

To prove the second one, let $i\in I^{\left(\alpha_{1}\right)}$ and
$j\in I^{\left(\alpha_{2}\right)}$ with $S_{i}^{\left(\alpha_{1}\right)}\cap S_{j}^{\left(\alpha_{2}\right)}\neq\emptyset$.
Thus, there is some $\xi\in S_{i}^{\left(\alpha_{1}\right)}\cap S_{j}^{\left(\alpha_{2}\right)}$.
In view of Lemma \ref{lem:AlphaShearletWeightIsModerate}, this implies\vspace{-0.1cm}
\[
\frac{w_{j}}{39}\leq\frac{1+\left|\xi\right|}{13}\leq w_{i}\leq3\cdot\left(1+\left|\xi\right|\right)\leq39\cdot w_{j},
\]
from which it easily follows that $39^{-\left|s\right|}\cdot w_{j}^{s}\leq w_{i}^{s}\leq39^{\left|s\right|}\cdot w_{j}^{s}$.
This establishes the second part of the second claim of the lemma.

But this easily implies that the weight $w^{s}$ is relatively $\CalS^{\left(\alpha_{2}\right)}$-moderate:
Indeed, let $i,\ell\in I^{\left(\alpha_{1}\right)}$ be arbitrary
with $S_{i}^{\left(\alpha_{1}\right)}\cap S_{j}^{\left(\alpha_{2}\right)}\neq\emptyset\neq S_{\ell}^{\left(\alpha_{1}\right)}\cap S_{j}^{\left(\alpha_{2}\right)}$
for some $j\in I^{\left(\alpha_{2}\right)}$. This implies $w_{i}^{s}\leq39^{\left|s\right|}\cdot w_{j}^{s}\leq\left(39^{2}\right)^{\left|s\right|}\cdot w_{\ell}^{s}$,
as desired.
\end{proof}
Now that we have established the strong compatibility between the
$\alpha$-shearlet coverings $\CalS^{\left(\alpha_{1}\right)}$ and
$\CalS^{\left(\alpha_{2}\right)}$ and of the associated weights,
we can easily characterize the existence of embeddings between the
$\alpha$-shearlet smoothness.
\begin{thm}
\noindent \label{thm:EmbeddingBetweenAlphaShearlets}Let $\alpha_{1},\alpha_{2}\in\left[0,1\right]$
with $\alpha_{1}\leq\alpha_{2}$. For $s,r\in\R$ and $p_{1},p_{2},q_{1},q_{2}\in\left(0,\infty\right]$,
the map 
\[
\mathscr{S}_{\alpha_{2},r}^{p_{1},q_{1}}\left(\R^{2}\right)\to\mathscr{S}_{\alpha_{1},s}^{p_{2},q_{2}}\left(\R^{2}\right),f\mapsto f
\]
is well-defined and bounded if and only if we have $p_{1}\leq p_{2}$
and
\[
\begin{cases}
r>s+\left(1+\alpha_{1}\right)\left(\frac{1}{p_{1}}-\frac{1}{p_{2}}\right)+\left(\alpha_{2}-\alpha_{1}\right)\left(\frac{1}{q_{2}}-\frac{1}{p_{1}^{\pm\triangle}}\right)_{+}+\left(1-\alpha_{2}\right)\left(\frac{1}{q_{2}}-\frac{1}{q_{1}}\right), & \text{if }q_{2}<q_{1},\\
r\geq s+\left(1+\alpha_{1}\right)\left(\frac{1}{p_{1}}-\frac{1}{p_{2}}\right)+\left(\alpha_{2}-\alpha_{1}\right)\left(\frac{1}{q_{2}}-\frac{1}{p_{1}^{\pm\triangle}}\right)_{+}, & \text{if }q_{2}\geq q_{1}.
\end{cases}
\]

\noindent Likewise, the map\vspace{-0.05cm}
\[
\mathscr{S}_{\alpha_{1},s}^{p_{1},q_{1}}\left(\R^{2}\right)\rightarrow\mathscr{S}_{\alpha_{2},r}^{p_{2},q_{2}}\left(\R^{2}\right),f\mapsto f
\]
is well-defined and bounded if and only if we have $p_{1}\leq p_{2}$
and 
\[
\begin{cases}
s>r+\left(1+\alpha_{1}\right)\left(\frac{1}{p_{1}}-\frac{1}{p_{2}}\right)+\left(\alpha_{2}-\alpha_{1}\right)\left(\frac{1}{p_{2}^{\triangledown}}-\frac{1}{q_{1}}\right)_{+}+\left(1-\alpha_{2}\right)\left(\frac{1}{q_{2}}-\frac{1}{q_{1}}\right), & \text{if }q_{2}<q_{1},\\
s\geq r+\left(1+\alpha_{1}\right)\left(\frac{1}{p_{1}}-\frac{1}{p_{2}}\right)+\left(\alpha_{2}-\alpha_{1}\right)\left(\frac{1}{p_{2}^{\triangledown}}-\frac{1}{q_{1}}\right)_{+}, & \text{if }q_{2}\geq q_{1}.
\end{cases}
\]
Here, we used the notations\vspace{-0.15cm}
\[
p^{\triangledown}:=\min\left\{ p,p'\right\} ,\qquad\text{ and }\qquad\frac{1}{p^{\pm\triangle}}:=\min\left\{ \frac{1}{p},1-\frac{1}{p}\right\} ,
\]
where the \textbf{conjugate exponent} $p'$ is defined as usual for
$p\in\left[1,\infty\right]$ and as $p':=\infty$ for $p\in\left(0,1\right)$.
\end{thm}
\begin{proof}
For the first part, we want to invoke part (4) of \cite[Theorem 7.2]{DecompositionEmbedding},
with $\CalQ=\CalS^{\left(\alpha_{2}\right)}=\left(\smash{T_{i}^{\left(\alpha_{2}\right)}}Q_{i}'\right)_{i\in I^{\left(\alpha_{2}\right)}}$
and $\CalP=\CalS^{\left(\alpha_{1}\right)}=\left(\smash{T_{i}^{\left(\alpha_{1}\right)}}Q_{i}'\right)_{i\in I^{\left(\alpha_{1}\right)}}$
and with $w=\left(w_{i}^{r}\right)_{i\in I^{\left(\alpha_{2}\right)}}$
and $v=\left(w_{i}^{s}\right)_{i\in I^{\left(\alpha_{1}\right)}}$.
To this end, we first have to verify that $\CalQ,\CalP,w,v$ satisfy
\cite[Assumption 7.1]{DecompositionEmbedding}. But we saw in Lemma
\ref{lem:AlphaShearletWeightIsModerate} that $w$ and $v$ are $\CalQ$-moderate
and $\CalP$-moderate, respectively. Furthermore, $\CalQ,\CalP$ are
almost structured coverings (cf.\@ Lemma \ref{lem:AlphaShearletCoveringIsAlmostStructured})
and thus also \textbf{semi-structured coverings} (cf.\@ \cite[Definition 2.5]{DecompositionEmbedding})
of $\CalO=\CalO'=\R^{2}$. Furthermore, since $\left\{ Q_{i}'\with i\in I^{\left(\alpha\right)}\right\} $
is a finite family of nonempty open sets (for arbitrary $\alpha\in\left[0,1\right]$),
it is not hard to see that $\CalS^{\left(\alpha\right)}$ is an open
covering of $\R^{2}$ and that there is some $\varepsilon>0$ and
for each $i\in I^{\left(\alpha\right)}$ some $\eta_{i}\in\R^{2}$
with $B_{\varepsilon}\left(\eta_{i}\right)\subset Q_{i}'$. Thus,
$\CalS^{\left(\alpha\right)}$ is a \textbf{tight}, open semi-structured
covering of $\R^{2}$ for all $\alpha\in\left[0,1\right]$. Hence,
so are $\CalQ,\CalP$. Finally, \cite[Corollary 2.7]{DecompositionIntoSobolev}
shows that if $\Phi=\left(\varphi_{i}\right)_{i\in I^{\left(\alpha_{2}\right)}}$
and $\Psi=\left(\psi_{j}\right)_{j\in I^{\left(\alpha_{1}\right)}}$
are regular partitions of unity for $\CalQ,\CalP$, respectively,
then $\Phi,\Psi$ are \textbf{$L^{p}$-BAPUs} (cf.\@ \cite[Definitions 3.5 and 3.6]{DecompositionEmbedding})
for $\CalQ,\CalP$, simultaneously for all $p\in\left(0,\infty\right]$.
Hence, all assumptions of \cite[Assumption 7.1]{DecompositionEmbedding}
are satisfied.

Next, Lemma \ref{lem:AlphaShearletCoveringSubordinateness} shows
that $\CalP=\CalS^{\left(\alpha_{1}\right)}$ is almost subordinate
to $\CalQ=\CalS^{\left(\alpha_{2}\right)}$ and Lemma \ref{lem:AlphaShearletRelativelyModerate}
shows that $\CalP$ and $v$ are relatively $\CalQ$-moderate, so
that all assumptions of \cite[Theorem 7.2, part (4)]{DecompositionEmbedding}
are satisfied.

Now, let us choose, for each $j\in I^{(\alpha_{2})}$, an arbitrary
index $i_{j}\in I^{(\alpha_{1})}$ with $S_{i_{j}}^{(\alpha_{1})}\cap S_{j}^{(\alpha_{2})}\neq\emptyset$.
Then \cite[Theorem 7.2, part (4)]{DecompositionEmbedding} shows that
the embedding $\mathscr{S}_{\alpha_{2},r}^{p_{1},q_{1}}\left(\R^{2}\right)\hookrightarrow\mathscr{S}_{\alpha_{1},s}^{p_{2},q_{2}}\left(\R^{2}\right)$
holds if and only if we have $p_{1}\leq p_{2}$ and if furthermore,
the following expression (then a constant) is finite: 
\[
\begin{aligned}K & :=\left\Vert \left(\frac{w_{i_{j}}^{s}}{w_{j}^{r}}\cdot\left|\det T_{j}^{(\alpha_{2})}\right|^{\left(\frac{1}{q_{2}}-\frac{1}{p_{1}^{\pm\triangle}}\right)_{+}}\;\cdot\;\left|\det T_{i_{j}}^{(\alpha_{1})}\right|^{\frac{1}{p_{1}}-\frac{1}{p_{2}}-\left(\frac{1}{q_{2}}-\frac{1}{p_{1}^{\pm\triangle}}\right)_{+}}\right)_{j\in I_{0}^{(\alpha_{2})}}\right\Vert _{\ell^{q_{2}\cdot(q_{1}/q_{2})'}}\\
\left({\scriptstyle \text{Lemma }\ref{lem:AlphaShearletRelativelyModerate}}\right) & \asymp\left\Vert \left(\frac{2^{ks}}{2^{kr}}\,\cdot\,2^{k\left(1+\alpha_{2}\right)\left(\frac{1}{q_{2}}-\frac{1}{p_{1}^{\pm\triangle}}\right)_{+}}\;\cdot\;2^{k\left(1+\alpha_{1}\right)\left[\frac{1}{p_{1}}-\frac{1}{p_{2}}-\left(\frac{1}{q_{2}}-\frac{1}{p_{1}^{\pm\triangle}}\right)_{+}\right]}\right)_{(k,\ell,\beta,\gamma)\in I_{0}^{(\alpha_{2})}}\right\Vert _{\ell^{q_{2}\cdot(q_{1}/q_{2})'}}\\
 & =\left\Vert \left(\raisebox{-0.2cm}{\ensuremath{2^{k\left(\left(s-r\right)+\left(1+\alpha_{2}\right)\left(\frac{1}{q_{2}}-\frac{1}{p_{1}^{\pm\triangle}}\right)_{+}+\left(1+\alpha_{1}\right)\left[\frac{1}{p_{1}}-\frac{1}{p_{2}}-\left(\frac{1}{q_{2}}-\frac{1}{p_{1}^{\pm\triangle}}\right)_{+}\right]\right)}}}\right)_{(k,\ell,\beta,\gamma)\in I_{0}^{(\alpha_{2})}}\right\Vert _{\ell^{q_{2}\cdot(q_{1}/q_{2})'}}\\
 & =\left\Vert \left(\raisebox{-0.2cm}{\ensuremath{2^{k\left(s-r+\left(\alpha_{2}-\alpha_{1}\right)\left(\frac{1}{q_{2}}-\frac{1}{p_{1}^{\pm\triangle}}\right)_{+}+\left(1+\alpha_{1}\right)\left(\frac{1}{p_{1}}-\frac{1}{p_{2}}\right)\right)}}}\right)_{(k,\ell,\beta,\gamma)\in I_{0}^{(\alpha_{2})}}\right\Vert _{\ell^{q_{2}\cdot(q_{1}/q_{2})'}}.
\end{aligned}
\]
Note that we only took the norm of the sequence with $j\in I_{0}^{\left(\alpha_{2}\right)}$,
omitting the term for $j=0$, in contrast to the definition of $K$
in \cite[Theorem 7.2]{DecompositionEmbedding}. This is justified,
since we are only interested in finiteness of the norm, for which
the single (finite(!))\@ term for $j=0$ is irrelevant.

Now, we distinguish two different cases regarding $q_{1}$ and $q_{2}$:

\textbf{Case 1}: We have $q_{2}<q_{1}$. This implies $\varrho:=q_{2}\cdot\left(q_{1}/q_{2}\right)'<\infty$,
cf.\@ \cite[Equation (4.3)]{DecompositionEmbedding}. For brevity,
let us define $\Theta:=s-r+\left(\alpha_{2}-\alpha_{1}\right)\left(\frac{1}{q_{2}}-\frac{1}{p_{1}^{\pm\triangle}}\right)_{+}+\left(1+\alpha_{1}\right)\left(\frac{1}{p_{1}}-\frac{1}{p_{2}}\right)$.
Then, we get
\[
\begin{aligned}K^{\varrho} & \asymp\left\Vert \left(2^{k\Theta}\right)_{(k,\ell,\beta,\gamma)\in I_{0}^{(\alpha_{2})}}\right\Vert _{\ell^{\varrho}}^{\varrho}=\sum_{(k,\ell,\beta,\gamma)\in I_{0}^{(\alpha_{2})}}2^{k\cdot\varrho\cdot\Theta}=\sum_{k=0}^{\infty}2^{k\cdot\varrho\cdot\Theta}\sum_{\left|\ell\right|\leq\left\lceil 2^{k\left(1-\alpha_{2}\right)}\right\rceil }\sum_{\beta\in\left\{ \pm1\right\} }\sum_{\gamma\in\left\{ 0,1\right\} }1\\
 & =4\cdot\sum_{k=0}^{\infty}2^{k\cdot\varrho\cdot\Theta}\left(1+2\cdot\left\lceil \smash{2^{k\left(1-\alpha_{2}\right)}}\right\rceil \right)\asymp\sum_{k=0}^{\infty}2^{k\left(\varrho\cdot\Theta+1-\alpha_{2}\right)}.
\end{aligned}
\]
Now, note from the remark to \cite[Lemma 4.8]{DecompositionEmbedding}
that $\frac{1}{p\cdot\left(q/p\right)'}=\left(\frac{1}{p}-\frac{1}{q}\right)_{+}$
for arbitrary $p,q\in\left(0,\infty\right]$. Hence, in the present
case, we have $\varrho^{-1}=\left(q_{2}^{-1}-q_{1}^{-1}\right)_{+}=q_{2}^{-1}-q_{1}^{-1}$.
Therefore, we see that the last sum from above—and therefore $K$—is
finite if and only if $\varrho\cdot\Theta+1-\alpha_{2}<0$. But this
is equivalent to
\[
s-r+\left(\alpha_{2}-\alpha_{1}\right)\left(\frac{1}{q_{2}}-\frac{1}{p_{1}^{\pm\triangle}}\right)_{+}+\left(1+\alpha_{1}\right)\left(\frac{1}{p_{1}}-\frac{1}{p_{2}}\right)=\Theta\overset{!}{<}\left(\alpha_{2}-1\right)\cdot\left(q_{2}^{-1}-q_{1}^{-1}\right),
\]
from which it easily follows that the claimed equivalence from the
first part of the theorem holds in case of $q_{2}<q_{1}$.

\medskip{}

\textbf{Case 2}: We have $q_{2}\geq q_{1}$. This implies $q_{2}\cdot\left(q_{1}/q_{2}\right)'=\infty$,
cf.\@ \cite[Equation (4.3)]{DecompositionEmbedding}. Thus, with
$\Theta$ as in the previous case, we have
\[
K\asymp\sup_{(k,\ell,\beta,\gamma)\in I_{0}^{(\alpha_{2})}}2^{k\Theta},
\]
so that $K$ is finite if and only if $\Theta\leq0$, which is equivalent
to
\[
r\geq s+(\alpha_{2}-\alpha_{1})\left(\frac{1}{q_{2}}-\frac{1}{p_{1}^{\pm\triangle}}\right)_{+}+\left(1+\alpha_{1}\right)\left(\frac{1}{p_{1}}-\frac{1}{p_{2}}\right).
\]
As in the previous case, this shows for $q_{2}\geq q_{1}$ that the
claimed equivalence from the first part of the theorem holds.

\medskip{}

For the second part of the theorem, we make use of part (4) of \cite[Theorem 7.4]{DecompositionEmbedding},
with $\CalQ=\CalS^{\left(\alpha_{1}\right)}=\left(\smash{T_{i}^{\left(\alpha_{1}\right)}}Q_{i}'\right)_{i\in I^{\left(\alpha_{1}\right)}}$
and $\CalP=\CalS^{\left(\alpha_{2}\right)}=\left(\smash{T_{i}^{\left(\alpha_{2}\right)}}Q_{i}'\right)_{i\in I^{\left(\alpha_{2}\right)}}$
and with $w=\left(w_{i}^{s}\right)_{i\in I^{\left(\alpha_{1}\right)}}$
and $v=\left(v_{i}^{r}\right)_{i\in I^{\left(\alpha_{2}\right)}}$.
As above, one sees that the corresponding assumptions are fulfilled.

Thus, \cite[Theorem 7.4, part (4)]{DecompositionEmbedding} shows
that the embedding $\mathscr{S}_{\alpha_{1},s}^{p_{1},q_{1}}\left(\R^{2}\right)\hookrightarrow\mathscr{S}_{\alpha_{2},r}^{p_{2},q_{2}}\left(\R^{2}\right)$
holds if and only if we have $p_{1}\leq p_{2}$ and if furthermore
the following expression (then a constant) is finite:
\[
C:=\left\Vert \left(\frac{w_{j}^{r}}{w_{i_{j}}^{s}}\cdot\left|\det T_{j}^{(\alpha_{2})}\right|{}^{\left(\frac{1}{p_{2}^{\triangledown}}-\frac{1}{q_{1}}\right)_{+}}\cdot\left|\det T_{i_{j}}^{(\alpha_{1})}\right|{}^{\frac{1}{p_{1}}-\frac{1}{p_{2}}-\left(\frac{1}{p_{2}^{\triangledown}}-\frac{1}{q_{1}}\right)_{+}}\right)_{j\in I^{(\alpha_{2})}}\right\Vert _{\ell^{q_{2}\cdot(q_{1}/q_{2})'}},
\]
where for each $j\in I^{(\alpha_{2})}$ an arbitrary index $i_{j}\in I^{(\alpha_{1})}$
with $S_{i_{j}}^{(\alpha_{1})}\cap S_{j}^{(\alpha_{2})}\neq\emptyset$
is chosen.

But in view of Lemma \ref{lem:AlphaShearletRelativelyModerate}, it
is not hard to see that $C$ satisfies
\[
\begin{aligned}C & \asymp\left\Vert \left(\frac{2^{kr}}{2^{ks}}\cdot2^{k\left(1+\alpha_{2}\right)\left(\frac{1}{p_{2}^{\triangledown}}-\frac{1}{q_{1}}\right)_{+}}\;\cdot\;2^{k\left(1+\alpha_{1}\right)\left[\frac{1}{p_{1}}-\frac{1}{p_{2}}-\left(\frac{1}{p_{2}^{\triangledown}}-\frac{1}{q_{1}}\right)_{+}\right]}\right)_{(k,\ell,\beta,\gamma)\in I_{0}^{(\alpha_{2})}}\right\Vert _{\ell^{q_{2}\cdot(q_{1}/q_{2})'}}\\
 & =\left\Vert \left(\raisebox{-0.2cm}{\ensuremath{2^{k\left(\left(1+\alpha_{1}\right)\left[\frac{1}{p_{1}}-\frac{1}{p_{2}}-\left(\frac{1}{p_{2}^{\triangledown}}-\frac{1}{q_{1}}\right)_{+}\right]+\left(1+\alpha_{2}\right)\left(\frac{1}{p_{2}^{\triangledown}}-\frac{1}{q_{1}}\right)_{+}-s+r\right)}}}\right)_{(k,\ell,\beta,\gamma)\in I_{0}^{(\alpha_{2})}}\right\Vert _{\ell^{q_{2}\cdot(q_{1}/q_{2})'}}\\
 & =\left\Vert \left(\raisebox{-0.2cm}{\ensuremath{2^{k\left(\left(1+\alpha_{1}\right)\left(\frac{1}{p_{1}}-\frac{1}{p_{2}}\right)+\left(\alpha_{2}-\alpha_{1}\right)\left(\frac{1}{p_{2}^{\triangledown}}-\frac{1}{q_{1}}\right)_{+}-s+r\right)}}}\right)_{(k,\ell,\beta,\gamma)\in I_{0}^{(\alpha_{2})}}\right\Vert _{\ell^{q_{2}\cdot(q_{1}/q_{2})'}}.
\end{aligned}
\]
As above, we distinguish two cases regarding $q_{1}$ and $q_{2}$:

\textbf{Case 1}: We have $q_{2}<q_{1}$, so that $\varrho:=q_{2}\cdot\left(q_{1}/q_{2}\right)'<\infty$.
But setting
\[
\Gamma:=\left(1+\alpha_{1}\right)\left(\frac{1}{p_{1}}-\frac{1}{p_{2}}\right)+\left(\alpha_{2}-\alpha_{1}\right)\left(\frac{1}{p_{2}^{\triangledown}}-\frac{1}{q_{1}}\right)_{+}-s+r,
\]
we have
\[
\begin{aligned}C^{\varrho} & \asymp\left\Vert \left(2^{k\Gamma}\right)_{(k,\ell,\beta,\gamma)\in I_{0}^{(\alpha_{2})}}\right\Vert _{\ell^{\varrho}}^{\varrho}=\sum_{(k,\ell,\beta,\text{\ensuremath{\gamma}})\in I_{0}^{(\alpha_{2})}}2^{k\cdot\varrho\cdot\Gamma}\\
 & =\sum_{k=0}^{\infty}2^{k\cdot\varrho\cdot\Gamma}\sum_{\left|\ell\right|\leq\left\lceil 2^{k\left(1-\alpha_{2}\right)}\right\rceil }\sum_{\beta\in\left\{ \pm1\right\} }\sum_{\gamma\in\left\{ 0,1\right\} }1\asymp\sum_{k=0}^{\infty}2^{k\left(\varrho\cdot\Gamma+1-\alpha_{2}\right)}.
\end{aligned}
\]
As above, we have $\varrho^{-1}=\left(q_{2}^{-1}-q_{1}^{-1}\right)_{+}=q_{2}^{-1}-q_{1}^{-1}$
and we see that the last sum—and thus $C$—is finite if and only if
we have $\varrho\cdot\Gamma+1-\alpha_{2}<0$, which is equivalent
to
\[
\left(1+\alpha_{1}\right)\left(\frac{1}{p_{1}}-\frac{1}{p_{2}}\right)+\left(\alpha_{2}-\alpha_{1}\right)\left(\frac{1}{p_{2}^{\triangledown}}-\frac{1}{q_{1}}\right)_{+}-s+r=\Gamma\overset{!}{<}\left(\alpha_{2}-1\right)\cdot\left(q_{2}^{-1}-q_{1}^{-1}\right).
\]
Based on this, it is not hard to see that the equivalence stated in
the second part of the theorem is valid for $q_{2}<q_{1}$.

\medskip{}

\textbf{Case 2}: We have $q_{2}\geq q_{1}$, so that $q_{2}\cdot\left(q_{1}/q_{2}\right)'=\infty$.
In this case, we have—with $\Gamma$ as above—that
\[
C\asymp\sup_{(k,\ell,\beta,\gamma)\in I_{0}^{(\alpha_{2})}}2^{k\Gamma},
\]
which is finite if and only if $\Gamma\leq0$, which is equivalent
to
\[
s\geq r+(1+\alpha_{1})\left(\frac{1}{p_{1}}-\frac{1}{p_{2}}\right)+(\alpha_{2}-\alpha_{1})\left(\frac{1}{p_{2}^{\triangledown}}-\frac{1}{q_{1}}\right)_{+}.
\]
This easily shows that the claimed equivalence from the second part
of the theorem also holds for $q_{2}\geq q_{1}$.
\end{proof}
With Theorem \ref{thm:EmbeddingBetweenAlphaShearlets}, we have established
the characterization of the general embedding from equation \eqref{eq:EmbeddingSectionGeneralEmbedding}.
Our main application, however, was to determine under which conditions
$\ell^{p}$-sparsity of $f$ with respect to $\alpha_{1}$-shearlet
systems implies $\ell^{q}$-sparsity of $f$ with respect to $\alpha_{2}$-shearlet
systems, \emph{if one has no additional information}. As discussed
around equation \eqref{eq:EmbeddingSectionSparsityEmbedding}, this
amounts to an embedding $\mathscr{S}_{\alpha_{1},\left(1+\alpha_{1}\right)\left(p^{-1}-2^{-1}\right)}^{p,p}\left(\R^{2}\right)\hookrightarrow\mathscr{S}_{\alpha_{2},\left(1+\alpha_{2}\right)\left(q^{-1}-2^{-1}\right)}^{q,q}\left(\R^{2}\right)$.
Since we are only interested in \emph{nontrivial} sparsity, and since
arbitrary $L^{2}$ functions have $\alpha$-shearlet coefficients
in $\ell^{2}$, the only interesting case is for $p,q\leq2$. This
setting is considered in our next lemma:
\begin{lem}
Let $\alpha_{1},\alpha_{2}\in\left[0,1\right]$ with $\alpha_{1}\neq\alpha_{2}$,
let $p,q\in\left(0,2\right]$ and let $\varepsilon\in\left[0,\infty\right)$.
The embedding
\[
\mathscr{S}_{\alpha_{1},\varepsilon+\left(1+\alpha_{1}\right)\left(p^{-1}-2^{-1}\right)}^{p,p}\left(\R^{2}\right)\hookrightarrow\mathscr{S}_{\alpha_{2},\left(1+\alpha_{2}\right)\left(q^{-1}-2^{-1}\right)}^{q,q}\left(\R^{2}\right)
\]
holds if and only if we have $p\leq q$ and $q\geq\left(\frac{1}{2}+\frac{\varepsilon}{\left|\alpha_{1}-\alpha_{2}\right|}\right)^{-1}$.
\end{lem}
\begin{rem*}
The case $\varepsilon=0$ corresponds to the embedding which is considered
in equation \eqref{eq:EmbeddingSectionSparsityEmbedding}. Here, the
preceding lemma shows that the embedding can only hold if $q\geq2$.
Since the $\alpha_{2}$-shearlet coefficients of every $L^{2}$ function
are $\ell^{2}$-sparse, we see that $\ell^{p}$-sparsity with respect
to $\alpha_{1}$-shearlets does not imply any nontrivial $\ell^{q}$-sparsity
with respect to $\alpha_{2}$-shearlets for $\alpha_{1}\neq\alpha_{2}$,
\emph{if no additional information than the $\ell^{p}$-sparsity with
respect to $\alpha_{1}$-shearlets is given}.

But in conjunction with Theorem \ref{thm:AnalysisAndSynthesisSparsityAreEquivalent},
we see that if the $\alpha_{1}$-shearlet coefficients $\left(\left\langle f,\psi^{\left[\left(j,\ell,\iota\right),k\right]}\right\rangle _{L^{2}}\right)_{\left(j,\ell,\iota\right)\in I^{\left(\alpha_{1}\right)},k\in\Z^{2}}$
satisfy 
\begin{equation}
\left\Vert \left(2^{\varepsilon j}\cdot\left\langle f,\psi^{\left[\left(j,\ell,\iota\right),k\right]}\right\rangle _{L^{2}}\right)_{\left(j,\ell,\iota\right)\in I^{\left(\alpha_{1}\right)},\,k\in\Z^{2}}\right\Vert _{\ell^{p}}<\infty\label{eq:SparsityTransferGoodCondition}
\end{equation}
for some $\varepsilon>0$, then one can derive $\ell^{q}$-sparsity
with respect to $\alpha_{2}$-shearlets for $q\geq\max\left\{ p,\,\left(\frac{1}{2}+\frac{\varepsilon}{\left|\alpha_{1}-\alpha_{2}\right|}\right)^{-1}\right\} $.
Observe that equation \eqref{eq:SparsityTransferGoodCondition} combines
an $\ell^{p}$-estimate with a decay of the coefficients with the
scale parameter $j\in\N_{0}$.
\end{rem*}
\begin{proof}
Theorem \ref{thm:EmbeddingBetweenAlphaShearlets} shows that the embedding
can only hold if $p\leq q$. Thus, we only need to show for $0<p\leq q\leq2$
that the stated embedding holds if and only if we have $q\geq\left(\frac{1}{2}+\frac{\varepsilon}{\left|\alpha_{1}-\alpha_{2}\right|}\right)^{-1}$.

For brevity, let $s:=\varepsilon+\left(1+\alpha_{1}\right)\left(p^{-1}-2^{-1}\right)$
and $r:=\left(1+\alpha_{2}\right)\left(q^{-1}-2^{-1}\right)$. We
start with a few auxiliary observations: Because of $p\leq q\leq2$,
we have $q^{\triangledown}=\min\left\{ q,q'\right\} =q$ and $\frac{1}{p^{\pm\triangle}}=\min\left\{ \frac{1}{p},1-\frac{1}{p}\right\} =1-\frac{1}{p}$,
as well as $\frac{1}{q^{\triangledown}}-\frac{1}{p}=\frac{1}{q}-\frac{1}{p}\leq0$
and $\frac{1}{p}+\frac{1}{q}\geq1$, so that $\frac{1}{q}-\frac{1}{p^{\pm\triangle}}=\frac{1}{q}-1+\frac{1}{p}\geq0$.

Now, let us first consider the case $\alpha_{1}<\alpha_{2}$. Since
we assume $p\leq q$, Theorem \ref{thm:EmbeddingBetweenAlphaShearlets}
shows that the embedding holds if and only if

\[
\begin{aligned} & s\overset{!}{\geq}r+(1+\alpha_{1})\left(\frac{1}{p}-\frac{1}{q}\right)+(\alpha_{2}-\alpha_{1})\left(\frac{1}{q^{\triangledown}}-\frac{1}{p}\right)_{+}\\
\Longleftrightarrow & \left(1+\alpha_{1}\right)\left(p^{-1}-2^{-1}\right)+\varepsilon\overset{!}{\geq}\left(1+\alpha_{2}\right)\left(q^{-1}-2^{-1}\right)+\left(1+\alpha_{1}\right)\left(p^{-1}-q^{-1}\right)\\
\Longleftrightarrow & \varepsilon\overset{!}{\geq}\left(1+\alpha_{2}\right)\left(q^{-1}-2^{-1}\right)+\left(1+\alpha_{1}\right)\left(2^{-1}-q^{-1}\right)=\left(\alpha_{2}-\alpha_{1}\right)\left(q^{-1}-2^{-1}\right)\\
\left({\scriptstyle \text{since }\alpha_{2}-\alpha_{1}>0}\right)\Longleftrightarrow & \frac{\varepsilon}{\alpha_{2}-\alpha_{1}}+\frac{1}{2}\overset{!}{\geq}\frac{1}{q}\\
\Longleftrightarrow & q\overset{!}{\geq}\left(\frac{1}{2}+\frac{\varepsilon}{\alpha_{2}-\alpha_{1}}\right)^{-1}=\left(\frac{1}{2}+\frac{\varepsilon}{\left|\alpha_{2}-\alpha_{1}\right|}\right)^{-1}.
\end{aligned}
\]
Finally, we consider the case $\alpha_{1}>\alpha_{2}$. Again, since
$p\leq q$, Theorem \ref{thm:EmbeddingBetweenAlphaShearlets} (with
interchanged roles of $\alpha_{1},\alpha_{2}$ and $r,s$) shows that
the desired embedding holds if and only if

\[
\begin{aligned} & s\overset{!}{\geq}r+\left(1+\alpha_{2}\right)\left(\frac{1}{p}-\frac{1}{q}\right)+\left(\alpha_{1}-\alpha_{2}\right)\left(\frac{1}{q}-\frac{1}{p^{\pm\triangle}}\right)_{+}\\
\left({\scriptstyle \text{since }q^{-1}-1+p^{-1}\geq0}\right)\Longleftrightarrow & \left(1+\alpha_{1}\right)\left(\frac{1}{p}-\frac{1}{2}\right)+\varepsilon\overset{!}{\geq}\left(1+\alpha_{2}\right)\left(\frac{1}{q}-\frac{1}{2}\right)+\left(1+\alpha_{2}\right)\left(\frac{1}{p}-\frac{1}{q}\right)+\left(\alpha_{1}-\alpha_{2}\right)\left(\frac{1}{q}-1+\frac{1}{p}\right)\\
\Longleftrightarrow & \varepsilon\overset{!}{\geq}\left(1+\alpha_{2}\right)\left(p^{-1}-2^{-1}\right)+\left(1+\alpha_{1}\right)\left(2^{-1}-p^{-1}\right)+\left(\alpha_{1}-\alpha_{2}\right)\left(q^{-1}-1+p^{-1}\right)\\
\Longleftrightarrow & \varepsilon\overset{!}{\geq}\left(\alpha_{1}-\alpha_{2}\right)\left(2^{-1}-p^{-1}+q^{-1}-1+p^{-1}\right)=\left(\alpha_{1}-\alpha_{2}\right)\left(q^{-1}-2^{-1}\right)\\
\left({\scriptstyle \text{since }\alpha_{1}-\alpha_{2}>0}\right)\Longleftrightarrow & \frac{\varepsilon}{\alpha_{1}-\alpha_{2}}+\frac{1}{2}\overset{!}{\geq}\frac{1}{q}\\
\Longleftrightarrow & q\overset{!}{\geq}\left(\frac{1}{2}+\frac{\varepsilon}{\alpha_{1}-\alpha_{2}}\right)^{-1}=\left(\frac{1}{2}+\frac{\varepsilon}{\left|\alpha_{1}-\alpha_{2}\right|}\right)^{-1}.
\end{aligned}
\]
This completes the proof.
\end{proof}

\section*{Acknowledgments}

We would like to thank Gitta Kutyniok for pushing us to improve the
statement and the proof of Lemma \ref{lem:MainShearletLemma} and
thus also of Theorems \ref{thm:NicelySimplifiedAlphaShearletFrameConditions},
\ref{thm:ReallyNiceShearletAtomicDecompositionConditions} and Remark
\ref{rem:CartoonApproximationConstantSimplification}. Without her
positive insistence, the proof of Lemma \ref{lem:MainShearletLemma}
would be about $5$ pages longer and Theorem \ref{thm:CartoonApproximationWithAlphaShearlets}
concerning the approximation of $C^{2}$-cartoon-like functions with
shearlets would require $\approx40$ vanishing moments and generators
in $C_{c}^{M}\left(\R^{2}\right)$ with $M\approx150$, while our
new improved conditions only require $7$ vanishing moments and generators
in $C_{c}^{19}\left(\R^{2}\right)$, cf.\@ Remark \ref{rem:CartoonApproximationConstantSimplification}.

FV would like to express warm thanks to Hartmut Führ for several fruitful
discussions and suggestions related to the present paper, in particular
for suggesting the title ``\emph{analysis vs.\@ synthesis sparsity
for shearlets}'' which we adopted nearly unchanged. FV would also
like to thank Philipp Petersen for useful discussions related to the
topics in this paper and for suggesting some changes in the notation.

Both authors would like to thank Jackie Ma for raising the question
whether membership in shearlet smoothness spaces can also be characterized
using compactly supported shearlets. We also thank Martin Schäfer
for checking parts of the introduction related to the paper \cite{RoleOfAlphaScaling}
for correctness.

Both authors acknowledge support from the European Commission through
DEDALE (contract no.\@ 665044) within the H2020 Framework Program.
AP also acknowledges partial support by the Lichtenberg Professorship
Grant of the Volkswagen Stiftung awarded to Christian Kuehn.

\newpage{}

\appendix

\section{Nonequivalence of analysis and synthesis sparsity for general frames}

\label{sec:AnalysisSynthesisSparsityNotEquivalentInGeneral}In this
section, we present two examples which show that for general frames,
neither does analysis sparsity imply synthesis sparsity, nor vice
versa. We begin with the (easier) case that synthesis sparsity does
not imply analysis sparsity:
\begin{example}
\label{exa:SynthesisSparsityDoesNOTImplyAnalysisSparsity}We consider
the Hilbert space $\ell^{2}\left(\N\right)$ with the standard orthonormal
basis given by $\left(\delta_{n}\right)_{n\in\N}$. The family $\Psi:=\left(\psi_{n}\right)_{n\in\N_{0}}$
given by $\psi_{n}:=\delta_{n}$ for $n\in\N$ and by $\psi_{0}:=\left(\frac{1}{\ell}\right)_{\ell\in\N}$
clearly forms a frame in $\ell^{2}\left(\N\right)$.

Furthermore, $f:=\psi_{0}$ is clearly $\ell^{p}$-synthesis sparse
with respect to $\Psi$ for arbitrary $p\in\left(0,2\right)$, since
we have $f=\sum_{n\in\N_{0}}c_{n}\psi_{n}$ with $\left(c_{n}\right)_{n\in\N}=\delta_{0}\in\ell^{p}\left(\N_{0}\right)$
for all $p\in\left(0,2\right)$. But the analysis coefficients are
given by $A_{\Psi}f=\left(\left\langle f,\psi_{n}\right\rangle \right)_{n\in\N_{0}}$
with $\left\langle f,\psi_{n}\right\rangle =\frac{1}{n}$ for $n\in\N$.
Hence, $A_{\Psi}f\notin\ell^{p}\left(\N_{0}\right)$ for $p\in\left(0,1\right]$.

Thus, for general frames, it is \emph{not} true that $\ell^{p}$-synthesis
sparsity implies $\ell^{p}$-analysis sparsity.
\end{example}
Finally, we give a counterexample to the reverse implication. We remark
that the counterexample constructed below is in fact a Riesz basis,
not simply a frame.
\begin{example}
\label{exa:AnalysisSparsityDoesNOTImplySynthesisSparsity}We again
consider the Hilbert space $\ell^{2}\left(\N\right)$ with the standard
orthonormal basis given by $\left(\delta_{n}\right)_{n\in\N}$. 

Choose some $N\in\N$ with $N>\sum_{n=1}^{\infty}\frac{1}{n^{2}}$
(i.e., $N>\frac{\pi^{2}}{6}\approx1.6$) and set
\[
\psi_{n}:=\delta_{n}-\frac{1}{N\cdot n^{2}}\cdot\sum_{\ell=1}^{N\cdot n^{2}}\delta_{2n+\ell}\qquad\text{ for }n\in\N.
\]
Note that $\psi_{n}\in\ell^{1}\left(\N\right)\hookrightarrow\ell^{2}\left(\N\right)$
with $\left\Vert \psi_{n}\right\Vert _{\ell^{1}}\leq1+\frac{1}{N\cdot n^{2}}\sum_{\ell=1}^{N\cdot n^{2}}1=2$.
We now want to show that the analysis operator $A_{\Psi}:\ell^{2}\left(\N\right)\to\ell^{2}\left(\N\right),x\mapsto\left(\left\langle x,\psi_{n}\right\rangle \right)_{n\in\N}$
associated to the family $\Psi=\left(\psi_{n}\right)_{n\in\N}$ is
well-defined, bounded and invertible. For this, it suffices by a Neumann
series argument to show $\sup_{\left\Vert x\right\Vert _{\ell^{2}}\leq1}\left\Vert x-A_{\Psi}x\right\Vert _{\ell^{2}}<1$.

But for arbitrary $x=\left(x_{n}\right)_{n\in\N}\in\ell^{2}\left(\N\right)$,
we have
\begin{align*}
\left\Vert x-A_{\Psi}x\right\Vert _{\ell^{2}}^{2}=\left\Vert \left(x_{n}\right)_{n\in\N}-\left(\left\langle x,\psi_{n}\right\rangle \right)_{n\in\N}\right\Vert _{\ell^{2}}^{2} & =\sum_{n=1}^{\infty}\left|\frac{1}{N\cdot n^{2}}\sum_{\ell=1}^{N\cdot n^{2}}x_{2n+\ell}\right|^{2}\\
 & \leq\sum_{n=1}^{\infty}\left(\frac{1}{N\cdot n^{2}}\sum_{\ell=1}^{N\cdot n^{2}}\left|x_{2n+\ell}\right|\right)^{2}\\
\left({\scriptstyle \text{Cauchy-Schwarz}}\right) & \leq\sum_{n=1}^{\infty}\left(\frac{1}{N\cdot n^{2}}\sqrt{\sum_{\ell=1}^{N\cdot n^{2}}\left|x_{2n+\ell}\right|^{2}}\cdot\sqrt{\sum_{\ell=1}^{N\cdot n^{2}}1^{2}}\right)^{2}\\
 & =\sum_{n=1}^{\infty}\left[\frac{1}{N\cdot n^{2}}\sum_{\ell=1}^{N\cdot n^{2}}\left|x_{2n+\ell}\right|^{2}\right]\\
 & \leq\sum_{m=1}^{\infty}\left|x_{m}\right|^{2}\cdot\frac{1}{N}\cdot\sum_{n=1}^{\infty}\frac{1}{n^{2}},
\end{align*}
so that we get $\sup_{\left\Vert x\right\Vert _{\ell^{2}}\leq1}\left\Vert x-A_{\Psi}x\right\Vert _{\ell^{2}}\leq\sqrt{\frac{1}{N}\cdot\sum_{n=1}^{\infty}n^{-2}}<1$,
as desired.

As seen above, this implies that $A_{\Psi}:\ell^{2}\left(\N\right)\to\ell^{2}\left(\N\right)$
is well-defined, bounded and boundedly invertible. Hence, so is the
synthesis operator $S_{\Psi}:\ell^{2}\left(\N\right)\to\ell^{2}\left(\N\right),\left(c_{n}\right)_{n\in\N}\mapsto\sum_{n\in\N}c_{n}\psi_{n}$,
since $S_{\Psi}=A_{\Psi}^{\ast}$. Therefore, the family $\Psi=\left(\psi_{n}\right)_{n\in\N}=\left(S_{\Psi}\delta_{n}\right)_{n\in\N}$
is the image of an orthonormal basis under an invertible linear operator,
so that $\Psi$ is a \textbf{Riesz-basis} and in particular a frame
for $\ell^{2}\left(\N\right)$, see \cite[Definition 3.6.1, Proposition 3.6.4 and Theorem 3.6.6]{ChristensenIntroductionToFramesAndRieszBases}.

Now, set $f:=\delta_{1}\in\ell^{2}\left(\N\right)$ and note $\supp\psi_{n}\subset\left\{ n,n+1,\dots\right\} $
for every $n\in\N$, so that $\left\langle f,\psi_{n}\right\rangle =0$
for all $n\geq2$. Hence, $A_{\Psi}f=\delta_{1}\in\ell^{p}\left(\N\right)$
for all $p\in\left(0,2\right)$, so that $f$ is analysis sparse with
respect to $\Psi$.

But $f$ is \emph{not} $\ell^{p}$-synthesis sparse with respect to
$\Psi$ for $p\leq1$: If $f=S_{\Psi}c$ for $c=\left(c_{n}\right)_{n\in\N}\in\ell^{p}\left(\N\right)\hookrightarrow\ell^{1}\left(\N\right)$
with $p\leq1$, then the uniform boundedness $\left\Vert \psi_{n}\right\Vert _{\ell^{1}}\leq2$
ensures that the series $f=\sum_{n\in\N}c_{n}\psi_{n}$ converges
unconditionally in $\ell^{1}\left(\N\right)$. In particular, with
the continuous linear functional $\varphi:\ell^{1}\left(\N\right)\to\Compl,\left(x_{n}\right)_{n\in\N}\mapsto\sum_{n\in\N}x_{n}$,
we would have $1=\varphi\left(f\right)=\sum_{n\in\N}c_{n}\varphi\left(\psi_{n}\right)=0$,
since $\varphi\left(\psi_{n}\right)=0$ for all $n\in\N$.

This contradiction shows that $f$ is \emph{not} $\ell^{p}$-synthesis
sparse with respect to $\Psi$ for $p\leq1$, even though $f$ is
$\ell^{p}$-analysis sparse.
\end{example}

\section{The \texorpdfstring{$\alpha$}{α}-shearlet covering is almost
structured}

\label{sec:AlphaShearletCoveringAlmostStructured}In this section,
we provide the proof of Lemma \ref{lem:AlphaShearletCoveringIsAlmostStructured},
whose statement we repeat here for the sake of convenience:
\begin{lem*}
\noindent The $\alpha$-shearlet covering $\CalS^{(\alpha)}$ from
Definition \ref{def:AlphaShearletCovering} is an almost structured
covering of $\R^{2}$.
\end{lem*}
\begin{proof}
\noindent First of all, we define the family $(T_{i}P_{i}'+b_{i})_{i\in I}$,
with $P_{i}':=U_{(-3/4,3/4)}^{(1/2,5/2)}$ for $i\in I_{0}$ and $P_{0}':=\left(-\frac{3}{4},\frac{3}{4}\right)^{2}$;
all being open sets. It is not hard to see $\overline{P_{i}'}\subset Q_{i}'$
for all $i\in I$. We now show that $\left(T_{i}P_{i}'+b_{i}\right)_{i\in I}$
covers $\R^{2}$. First, we note
\begin{align*}
\bigcup_{m=-\lceil2^{n(1-\alpha)}\rceil}^{\lceil2^{n(1-\alpha)}\rceil}\left(2^{n(\alpha-1)}\left(m-\frac{3}{4}\right),2^{n(\alpha-1)}\left(m+\frac{3}{4}\right)\right) & =2^{n(\alpha-1)}\bigcup_{m=-\lceil2^{n(1-\alpha)}\rceil}^{\lceil2^{n(1-\alpha)}\rceil}\left(m-\frac{3}{4},m+\frac{3}{4}\right)\\
 & =2^{n(\alpha-1)}\left(-\left\lceil \smash{2^{n(1-\alpha)}}\right\rceil -\frac{3}{4},\left\lceil \smash{2^{n(1-\alpha)}}\right\rceil +\frac{3}{4}\right)\\
 & \supset2^{n(\alpha-1)}\left(-2^{n(1-\alpha)}-\frac{3}{4},2^{n(1-\alpha)}+\frac{3}{4}\right)\\
 & =\left(-1-\frac{3}{4}\cdot2^{n(\alpha-1)},1+\frac{3}{4}\cdot2^{n(\alpha-1)}\right)\\
 & \supset\left[-1,1\right].
\end{align*}
Using this inclusion, as well as equation \eqref{eq:deltazeroset},
and recalling $G_{n}=\left\lceil 2^{n\left(1-\alpha\right)}\right\rceil $,
we conclude
\begin{align*}
\bigcup_{n=0}^{\infty}\bigcup_{m=-G_{n}}^{G_{n}}T_{n,m,1,0}^{\left(\alpha\right)}P_{n,m,1,0}' & =\bigcup_{n=0}^{\infty}\:\bigcup_{m=-\lceil2^{n(1-\alpha)}\rceil}^{\lceil2^{n(1-\alpha)}\rceil}U_{\left(2^{n\left(\alpha-1\right)}\left(m-3/4\right),2^{n\left(\alpha-1\right)}\left(m+3/4\right)\right)}^{\left(\frac{2^{n}}{2},\frac{5}{2}\cdot2^{n}\right)}\\
 & \supset\bigcup_{n=0}^{\infty}\left\{ \begin{pmatrix}\xi\\
\eta
\end{pmatrix}\in\left(\frac{2^{n}}{2},\frac{5}{2}\cdot2^{n}\right)\times\R\left|\frac{\eta}{\xi}\in\left[-1,1\right]\right.\right\} \\
 & \supset\left\{ \left.\begin{pmatrix}\xi\\
\eta
\end{pmatrix}\in\left(\frac{1}{2},\infty\right)\times\R\right||\eta|\leq|\xi|\right\} .
\end{align*}
Furthermore, since $T_{j,\ell,-1,0}^{\left(\alpha\right)}=-T_{j,\ell,1,0}^{\left(\alpha\right)}$,
we have 
\[
\bigcup_{n=0}^{\infty}\:\bigcup_{m=-G_{n}}^{G_{n}}\:\bigcup_{\varepsilon\in\{\pm1\}}T_{n,m,\varepsilon,0}^{\left(\alpha\right)}P_{n,m,\text{\ensuremath{\varepsilon}},0}'\supset\left\{ \left.\begin{pmatrix}\xi\\
\eta
\end{pmatrix}\in\R^{2}\right|\left|\xi\right|>\frac{1}{2}\text{ and }\left|\eta\right|\leq\left|\xi\right|\right\} ,
\]
and since $R\left(\begin{smallmatrix}\xi\\
\eta
\end{smallmatrix}\right)=\left(\begin{smallmatrix}\eta\\
\xi
\end{smallmatrix}\right)$, we finally get 
\begin{align*}
\bigcup_{i\in I_{0}}T_{i}P_{i}'\supset & \left\{ \left.\begin{pmatrix}\xi\\
\eta
\end{pmatrix}\in\R^{2}\right|\left|\xi\right|>\frac{1}{2}\text{ and }\left|\eta\right|\leq\left|\xi\right|\right\} \cup\left\{ \left.\begin{pmatrix}\xi\\
\eta
\end{pmatrix}\in\R^{2}\right|\left|\eta\right|>\frac{1}{2}\text{ and }\left|\xi\right|\leq\left|\eta\right|\right\} =:M.
\end{align*}

Since we clearly have $T_{0}P_{0}'+b_{0}=\left(-\frac{3}{4},\frac{3}{4}\right)^{2}\supset\left[-\frac{1}{2},\frac{1}{2}\right]^{2}$,
it suffices to show that each $\left(\begin{smallmatrix}\xi\\
\eta
\end{smallmatrix}\right)\in\R^{2}\backslash\left[-\frac{1}{2},\frac{1}{2}\right]^{2}$ satisfies $\left(\begin{smallmatrix}\xi\\
\eta
\end{smallmatrix}\right)\in M$, in order to prove that $\left(T_{i}P_{i}'+b_{i}\right)_{i\in I}$
covers all of $\R^{2}$. To see this, we distinguish two cases for
$\left(\begin{smallmatrix}\xi\\
\eta
\end{smallmatrix}\right)\in\R^{2}\setminus\left[-\frac{1}{2},\frac{1}{2}\right]^{2}$:

\begin{casenv}
\item $\left|\eta\right|\geq\left|\xi\right|$. Then $\left|\eta\right|>\frac{1}{2}$,
since otherwise we would have $\left|\xi\right|\leq\left|\eta\right|\leq\frac{1}{2}$,
contradicting $\left(\begin{smallmatrix}\xi\\
\eta
\end{smallmatrix}\right)\in\R^{2}\backslash\left[-\frac{1}{2},\frac{1}{2}\right]^{2}$. Hence, $\left(\begin{smallmatrix}\xi\\
\eta
\end{smallmatrix}\right)\in M$.
\item $\left|\eta\right|\leq\left|\xi\right|$. Then $\left|\xi\right|>\frac{1}{2}$,
since otherwise we would have $\left|\eta\right|\leq\left|\xi\right|\leq\frac{1}{2}$
contradicting $\left(\begin{smallmatrix}\xi\\
\eta
\end{smallmatrix}\right)\in\R^{2}\backslash\left[-\frac{1}{2},\frac{1}{2}\right]^{2}$. Hence, $\left(\begin{smallmatrix}\xi\\
\eta
\end{smallmatrix}\right)\in M$.
\end{casenv}
All in all, we have shown that $\left(T_{i}P_{i}'+b_{i}\right)_{i\in I}$
is a covering of $\R^{2}$; because of $Q_{i}=T_{i}Q_{i}'+b_{i}\supset T_{i}P_{i}'+b_{i}$
for all $i\in I$, we also see that $\CalS^{\left(\alpha\right)}$
covers all of $\R^{2}$. Moreover, the sets $\left\{ \left.P_{i}'\right|i\in I\right\} $
and $\left\{ \left.Q_{i}'\right|i\in I\right\} $ are finite; in fact,
each of these sets only has two elements. Furthermore, we clearly
have $Q_{i}=T_{i}Q_{i}'+b_{i}\subset\R^{2}$ for all $i\in I$.

Thus, to verify that $\CalS^{\left(\alpha\right)}$ is an almost structured
covering of $\R^{2}$, we only have to verify that $\CalS^{\left(\alpha\right)}$
is admissible and that $\sup_{i\in I}\sup_{j\in i^{\ast}}\left\Vert T_{i}^{-1}T_{j}\right\Vert $
is finite, cf.\@ Definition \ref{def:AlmostStructuredCovering}.
To this end, we define
\[
M_{i}:=i^{*}\cap I_{0}\quad\text{ and }\quad M_{i}^{(\nu)}:=\left\{ \left(k,\ell,\beta,\gamma\right)\in M_{i}|\gamma=\nu\right\} ,\quad\text{ as well as }\quad C_{i}^{(\nu)}:=\sup_{j\in\vphantom{M_{i}^{N}}\smash{M_{i}^{(\nu)}}}\left\Vert T_{i}^{-1}T_{j}\right\Vert 
\]
for $i\in I_{0}$ and $\nu\in\left\{ 0,1\right\} $. Note $M_{i}=M_{i}^{(0)}\uplus M_{i}^{(1)}$.
Next, for $\left(k,\ell,\beta,\gamma\right)\in I_{0}$, we define
\begin{equation}
\left(k,\ell,\beta,\gamma\right)':=\begin{cases}
\left(k,\ell,\beta,1\right), & \text{if }\gamma=0,\\
\left(k,\ell,\beta,0\right), & \text{if }\gamma=1.
\end{cases}\label{eq:AlphaCoveringAdmissibleMirroredIndex}
\end{equation}

It is not hard to see $T_{i'}=RT_{i}$ and $Q_{i'}'=Q_{i}'=Q$ for
all $i\in I_{0}$. Hence, we have the following equivalence for $i,j\in I_{0}$:
\[
\begin{aligned}\emptyset\neq S_{i}^{(\alpha)}\cap S_{j}^{(\alpha)} & \Longleftrightarrow\emptyset\neq T_{i}Q\cap T_{j}Q\\
 & \Longleftrightarrow\emptyset\neq R\left[T_{i}Q\cap T_{j}Q\right]=T_{i'}Q\cap T_{j'}Q\\
 & \Longleftrightarrow\emptyset\neq S_{i'}^{(\alpha)}\cap S_{j'}^{(\alpha)}.
\end{aligned}
\]
Furthermore, 
\[
T_{i}^{-1}T_{j}=\left(RT_{i}\right)^{-1}RT_{j}=T_{i'}^{-1}\cdot T_{j'}.
\]
Hence, $M_{i'}=\left\{ j'\left|j\in M_{i}\right.\right\} $ and $C_{i}^{(\nu)}=C_{i'}^{(\nu)}$
for $\nu\in\{0,1\}$ and all $i\in I_{0}$, so that it suffices to
consider the case $i=\left(n,m,\varepsilon,0\right)\in I_{0}$ from
now on. We distinguish two cases regarding $j\in M_{i}$:

\textbf{Case 1}: $j=\left(k,\ell,\beta,0\right)\in M_{i}^{(0)}$.
We have $\emptyset\neq S_{i}^{\left(\alpha\right)}\cap S_{j}^{\left(\alpha\right)}$.
Since $S_{i}^{\left(\alpha\right)}\subset\varepsilon\left(0,\infty\right)\times\R$
and $S_{j}^{\left(\alpha\right)}\subset\beta\left(0,\infty\right)\times\R$,
this implies $\varepsilon=\beta$, so that equation \eqref{eq:deltazeroset}
yields
\[
\emptyset\neq\varepsilon\cdot\left(S_{i}^{\left(\alpha\right)}\cap S_{j}^{\left(\alpha\right)}\right)=S_{n,m,1,0}^{\left(\alpha\right)}\cap S_{k,\ell,1,0}^{\left(\alpha\right)}=U_{\left(2^{n\left(\alpha-1\right)}\left(m-1\right),2^{n\left(\alpha-1\right)}\left(m+1\right)\right)}^{\left(2^{n}/3,\,3\cdot2^{n}\right)}\cap U_{\left(2^{k\left(\alpha-1\right)}\left(\ell-1\right),2^{k\left(\alpha-1\right)}\left(\ell+1\right)\right)}^{\left(2^{k}/3,\,3\cdot2^{k}\right)}\subset\left(0,\infty\right)\times\R.
\]
Now, we consider the diffeomorphism $\Phi:\left(0,\infty\right)\times\R\to\left(0,\infty\right)\times\R,\left(\xi,\eta\right)\mapsto\left(\xi,\frac{\eta}{\xi}\right)$
and observe the easily verifiable identity $\Phi\left(U_{\left(a,b\right)}^{\left(\gamma,\mu\right)}\right)=\left(\gamma,\mu\right)\times\left(a,b\right)$.
Consequently, we get
\[
\emptyset\neq\left[\left(\frac{2^{n}}{3},\,3\cdot2^{n}\right)\cap\left(\frac{2^{k}}{3},\,3\cdot2^{k}\right)\right]\times\left[\left(2^{n\left(\alpha-1\right)}\left(m-1\right),2^{n\left(\alpha-1\right)}\left(m+1\right)\right)\cap\left(2^{k\left(\alpha-1\right)}\left(\ell-1\right),2^{k\left(\alpha-1\right)}\left(\ell+1\right)\right)\right].
\]
In particular, $\frac{2^{k}}{3}<3\cdot2^{n}$ and $\frac{2^{n}}{3}<3\cdot2^{k}$,
which yields $2^{k-n}<9<2^{4}$ and $2^{n-k}<9<2^{4}$. Thus, $\left|k-n\right|<4$
and hence $\left|k-n\right|\leq3$, since $k-n\in\Z$.

Furthermore, we get
\[
2^{k\left(\alpha-1\right)}\left(\ell-1\right)<2^{n\left(\alpha-1\right)}\left(m+1\right)\qquad\text{ and }\qquad2^{n\left(\alpha-1\right)}\left(m-1\right)<2^{k\left(\alpha-1\right)}\left(\ell+1\right),
\]
which implies
\[
\ell-1<2^{\left(n-k\right)\left(\alpha-1\right)}\left(m+1\right)\qquad\text{ and }\qquad\ell+1>2^{\left(n-k\right)\left(\alpha-1\right)}\left(m-1\right).
\]
Because of $0\leq1-\alpha\leq1$ and $\left|k-n\right|\leq3$, we
have $2^{\left(n-k\right)\left(\alpha-1\right)}=2^{\left(1-\alpha\right)\left(k-n\right)}\leq2^{3}$
and thus
\[
2^{\left(1-\alpha\right)\left(k-n\right)}m-9\leq-1-2^{\left(1-\alpha\right)\left(k-n\right)}+2^{\left(1-\alpha\right)\left(k-n\right)}m<\ell<1+2^{\left(k-n\right)\left(1-\alpha\right)}+2^{\left(1-\alpha\right)\left(k-n\right)}m\leq2^{\left(1-\alpha\right)\left(k-n\right)}m+9.
\]
Thus, with $M_{n,m,\lambda}:=\Z\cap\left[2^{\left(1-\alpha\right)\left(\lambda-n\right)}m-9,\,2^{\left(1-\alpha\right)\left(\lambda-n\right)}m+9\right]$,
we have shown
\[
j=\left(k,\ell,\beta,0\right)\in\bigcup_{\lambda=n-3}^{n+3}\left[\left\{ \lambda\right\} \times M_{n,m,\lambda}\times\left\{ \varepsilon\right\} \times\left\{ 0\right\} \right].
\]
Because of $\left|M_{n,m,\lambda}\right|\leq19$, the set on the right-hand
side has at most $7\cdot19=133$ elements, so that we get $\left|M_{i}^{\left(0\right)}\right|\leq133=:N$.

Finally, we note
\[
\left\Vert T_{i}^{-1}T_{j}\right\Vert =\left\Vert \left(\begin{matrix}1 & 0\\
-m & 1
\end{matrix}\right)\left(\begin{matrix}2^{-n} & 0\\
0 & 2^{-n\alpha}
\end{matrix}\right)\left(\begin{matrix}2^{k} & 0\\
0 & 2^{k\alpha}
\end{matrix}\right)\left(\begin{matrix}1 & 0\\
\ell & 1
\end{matrix}\right)\right\Vert =\left\Vert \left(\begin{array}{c|c}
2^{k-n} & 0\\
2^{\alpha\left(k-n\right)}\ell-2^{k-n}m & 2^{\alpha\left(k-n\right)}
\end{array}\right)\right\Vert .
\]
Now, since $\left|k-n\right|\leq3$, we have $0\leq2^{k-n}\leq2^{3}$
and $0\leq2^{\alpha\left(k-n\right)}\leq2^{3\alpha}\leq2^{3}$. Furthermore,
we saw above that $\left|\ell-2^{\left(1-\alpha\right)\left(k-n\right)}m\right|\leq9$,
so that we get
\[
\left|2^{\alpha\left(k-n\right)}\ell-2^{k-n}m\right|=2^{\alpha\left(k-n\right)}\cdot\left|\ell-2^{\left(1-\alpha\right)\left(k-n\right)}m\right|\leq9\cdot2^{\alpha\left(k-n\right)}\leq9\cdot2^{3}.
\]
All in all, this implies $\left\Vert T_{i}^{-1}T_{j}\right\Vert \leq11\cdot2^{3}\leq2^{7}=128$.
Since $j\in M_{i}^{\left(0\right)}$ was arbitrary, we conclude $C_{i}^{\left(0\right)}\leq128=:K$.

\medskip{}

\textbf{Case 2}: $j=\left(k,\ell,\beta,1\right)\in M_{i}^{(1)}$.
By definition of $M_{i}$, there is some $\left(\begin{smallmatrix}\xi\\
\eta
\end{smallmatrix}\right)\in S_{i}^{(\alpha)}\cap S_{j}^{(\alpha)}$. Lemma \ref{lem:AlphaShearletCoveringAuxiliary} implies $2^{n-2}<\left|\left(\begin{smallmatrix}\xi\\
\eta
\end{smallmatrix}\right)\right|<2^{n+4}$, as well as $2^{k-2}<\left|\left(\begin{smallmatrix}\xi\\
\eta
\end{smallmatrix}\right)\right|<2^{k+4}$ and thus $2^{n-2}<2^{k+4}$ as well as $2^{k-2}<2^{n+4}$. Consequently,
$\left|n-k\right|<6$ and thus $\left|n-k\right|\leq5$, since $n-k\in\Z$.

Next, we explicitly compute the transition matrix $T_{i}^{-1}T_{j}$:
\begin{align}
T_{i}^{-1}T_{j}=\left(A_{n,m,\varepsilon}^{(\alpha)}\right)^{-1}RA_{k,\ell,\beta}^{(\alpha)} & =\varepsilon\beta\left(\begin{matrix}1 & 0\\
-m & 1
\end{matrix}\right)\left(\begin{matrix}2^{-n} & 0\\
0 & 2^{-\alpha n}
\end{matrix}\right)\begin{pmatrix}0 & 1\\
1 & 0
\end{pmatrix}\begin{pmatrix}2^{k} & 0\\
2^{k\alpha}\ell & 2^{k\alpha}
\end{pmatrix}\nonumber \\
 & =\varepsilon\beta\left(\begin{array}{c|c}
2^{-n} & 0\\
-2^{-n}m & 2^{-n\alpha}
\end{array}\right)\left(\begin{array}{c|c}
\ell2^{k\alpha} & 2^{k\alpha}\\
2^{k} & 0
\end{array}\right)\nonumber \\
 & =\varepsilon\beta\left(\begin{array}{c|c}
2^{k\alpha-n}\ell & 2^{k\alpha-n}\\
2^{k-n\alpha}-2^{k\alpha-n}\ell m & -2^{k\alpha-n}m
\end{array}\right).\label{eq:matrixentries}
\end{align}
Now we distinguish three different subcases regarding $\alpha\in\left[0,1\right]$
and $n\in\N_{0}$:

\textbf{Case 2(a)}: $\alpha\neq1$ and $n<\frac{12}{1-\alpha}$. Since
$|n-k|\leq5$ this implies $k\leq5+n<5+\frac{12}{1-\alpha}=\frac{17-5\alpha}{1-\alpha}\leq\frac{17}{1-\alpha}$.
We thus have 
\[
M_{i}^{(1)}\subset\bigcup_{k=0}^{\left\lceil 17/\left(1-\alpha\right)\right\rceil }\left[\left\{ k\right\} \times\left\{ -\left\lceil \smash{2^{k(1-\alpha)}}\right\rceil ,\dots,\left\lceil \smash{2^{k(1-\alpha)}}\right\rceil \right\} \times\left\{ \pm1\right\} \times\left\{ 0,1\right\} \right]=:M,
\]
and hence $\left|\smash{M_{i}^{(1)}}\right|\leq|M|\leq N_{0}$, for
some absolute constant $N_{0}=N_{0}\left(\alpha\right)\in\N$, since
$M$ is a finite set. Note also that $i\in M$, since $n<\frac{12}{1-\alpha}\leq\frac{17}{1-\alpha}$.
Consequently, 
\[
C_{i}^{(1)}=\sup_{j\in M_{i}^{(1)}}\left\Vert T_{i}^{-1}T_{j}\right\Vert \leq\max_{\gamma,\lambda\in M}\left\Vert T_{\lambda}^{-1}T_{\gamma}\right\Vert =:K_{0}.
\]

\medskip{}

\textbf{Case 2(b)}: $\alpha\neq1$ and $n\geq\frac{12}{1-\alpha}$.
Since $|n-k|\leq5$ this implies $k\geq n-5\geq\frac{12}{1-\alpha}-5=\frac{7+5\alpha}{1-\alpha}$.
We know from Lemma \ref{lem:AlphaShearletCoveringAuxiliary} that
$0<\left|\xi\right|<3\left|\eta\right|$ and $0<\left|\eta\right|<3\left|\xi\right|$,
i.e., $\frac{1}{3}<\left|\frac{\eta}{\xi}\right|<3$.

Now, we claim $\left|m\right|\geq\frac{64}{3}-1$. To see this, assume
towards a contradiction that $|m|<\frac{64}{3}-1$. This implies because
of $n\geq\frac{12}{1-\alpha}$, because of equation \eqref{eq:deltazeroset}
and because of $\left(\begin{smallmatrix}\xi\\
\eta
\end{smallmatrix}\right)\in S_{i}^{\left(\alpha\right)}=S_{n,m,\varepsilon,0}^{\left(\alpha\right)}$ that
\begin{align*}
\frac{\eta}{\xi}\in\left(2^{-n\left(1-\alpha\right)}\left(m-1\right),2^{-n\left(1-\alpha\right)}\left(m+1\right)\right)\subset\Big(-\frac{64/3}{2^{n\left(1-\alpha\right)}},\frac{64/3}{2^{n\left(1-\alpha\right)}}\Big)\subset\left(-\frac{64/3}{2^{12}},\frac{64/3}{2^{12}}\right) & \subset\left(-\frac{1}{3},\frac{1}{3}\right),
\end{align*}
in contradiction to $\left|\frac{\eta}{\xi}\right|>\frac{1}{3}$.
Thus we must have $|m|\geq\frac{64}{3}-1$.

Likewise, we have $\left|\ell\right|\geq\frac{64}{3}-1$. Indeed,
since we have $k\geq\frac{7+5\alpha}{1-\alpha}$ and $\frac{1}{3}<\left|\frac{\xi}{\eta}\right|<3$,
the assumption $|\ell|<\frac{64}{3}-1$ yields the contradiction
\begin{align*}
\frac{\xi}{\eta}\in\left(2^{-k(1-\alpha)}\left(\ell-1\right),2^{-k(1-\alpha)}\left(\ell+1\right)\right)\subset\Big(-\frac{64/3}{2^{7+5\alpha}},\frac{64/3}{2^{7+5\alpha}}\Big) & \subset\Big(-\frac{64/3}{2^{7}},\frac{64/3}{2^{7}}\Big)\subset\left(-\frac{1}{3},\frac{1}{3}\right).
\end{align*}
Consequently, we must have $|\ell|\geq\frac{64}{3}-1$.

Now, since $\left|m\right|\geq\frac{64}{3}-1$, we either have $m\geq\frac{64}{3}-1>0$
or $m\leq1-\frac{64}{3}<0$. Let us distinguish these two cases: 

\textbf{Case 2(b)(i}): $m\geq\frac{64}{3}-1$. Since $\left(\begin{smallmatrix}\xi\\
\eta
\end{smallmatrix}\right)\in S_{n,m,\varepsilon,0}^{\left(\alpha\right)}\cap S_{k,\ell,\beta,1}^{\left(\alpha\right)}=S_{n,m,\varepsilon,0}^{\left(\alpha\right)}\cap RS_{k,\ell,\beta,0}^{\left(\alpha\right)}$ and using equation \eqref{eq:deltazeroset}, we see $\frac{\eta}{\xi}>2^{-n\left(1-\alpha\right)}\left(m-1\right)>0$
and $0<\frac{\xi}{\eta}<2^{-k\left(1-\alpha\right)}\left(\ell+1\right)$.
Hence, $\ell>-1$ and since $|\ell|\geq\frac{64}{3}-1$, we have $\ell\geq\frac{64}{3}-1$.

First, we want to show $m\geq2^{n\left(1-\alpha\right)}-65$. Thus,
assume towards a contradiction that $m<2^{n(1-\alpha)}-2^{6}-1$ and
note that $2^{n\left(1-\alpha\right)}-2^{6}-1=2^{n(1-\alpha)}-65\geq2^{12}-65>0$,
since $n\geq\frac{12}{1-\alpha}$. Now, we get 
\begin{align*}
\frac{\xi}{\eta}<2^{-k(1-\alpha)}(\ell+1)\leq2^{-k(1-\alpha)}\left(\left\lceil \smash{2^{k(1-\alpha)}}\right\rceil +1\right) & <2^{-k(1-\alpha)}\left(2^{k(1-\alpha)}+1+1\right)=1+2^{-k(1-\alpha)+1}
\end{align*}
and 
\begin{align*}
\frac{\xi}{\eta}=\left(\frac{\eta}{\xi}\right)^{-1}>\left(2^{-n(1-\alpha)}(m+1)\right)^{-1}=\frac{2^{n(1-\alpha)}}{m+1}>\frac{2^{n(1-\alpha)}}{2^{n(1-\alpha)}-2^{6}} & =1+\frac{2^{6}}{2^{n(1-\alpha)}-2^{6}}>1+\frac{2^{6}}{2^{n(1-\alpha)}}.
\end{align*}
Thus $2^{-k(1-\alpha)+1}>\frac{2^{6}}{2^{n(1-\alpha)}}$ and hence
$2^{(n-k)(1-\alpha)}>2^{5}$ in contradiction to $2^{(n-k)(1-\alpha)}\leq2^{|n-k|(1-\alpha)}\leq2^{|n-k|}\leq2^{5}$.
Thus, $m\geq2^{n(1-\alpha)}-65$.

Next, we similarly show $\ell\geq2^{k\left(1-\alpha\right)}-65$.
Again, we assume towards a contradiction that $\ell<2^{k(1-\alpha)}-2^{6}-1$
and note $2^{k(1-\alpha)}-2^{6}-1\geq2^{7+5\alpha}-2^{6}-1>0$. Now,
on the one hand we get
\[
\left(\frac{\xi}{\eta}\right)^{-1}>\left(2^{-k\left(1-\alpha\right)}\left(\ell+1\right)\right)^{-1}=\frac{2^{k(1-\alpha)}}{\ell+1}\geq\frac{2^{k(1-\alpha)}}{2^{k(1-\alpha)}-2^{6}}=1+\frac{2^{6}}{2^{k(1-\alpha)}-2^{6}}>1+\frac{2^{6}}{2^{k(1-\alpha)}},
\]
but on the other hand
\[
\left(\frac{\xi}{\eta}\right)^{-1}=\frac{\eta}{\xi}<2^{-n\left(1-\alpha\right)}\left(m+1\right)\leq2^{-n\left(1-\alpha\right)}\left(\left\lceil \smash{2^{n(1-\alpha)}}\right\rceil +1\right)<2^{-n\left(1-\alpha\right)}\left(2^{n(1-\alpha)}+2\right)=1+2^{-n(1-\alpha)+1},
\]
i.e., $2^{(k-n)(1-\alpha)}>2^{5}$ in contradiction to $\left|n-k\right|\leq5$.
Thus, $\ell\geq2^{k(1-\alpha)}-2^{6}-1=2^{k(1-\alpha)}-65$.

Using these estimates for $m$ and $\ell$, we can now bound the entries
of $T_{i}^{-1}T_{j}$ (cf.\@ eq.\@ \eqref{eq:matrixentries}): We
have 
\begin{align*}
\left|2^{k\alpha-n}\ell\right|\leq2^{k\alpha-n}\left\lceil \smash{2^{k(1-\alpha)}}\right\rceil <2^{k\alpha-n}\left(2^{k(1-\alpha)}+1\right)=2^{k-n}+2^{k\alpha-n}\leq2^{5}+2^{k-n} & \leq2\cdot2^{5}
\end{align*}
and furthermore $\left|2^{k\alpha-n}\right|\leq2^{k-n}\leq2^{5}$,
as well as 
\begin{align*}
\left|-2^{k\alpha-n}m\right|=2^{k\alpha-n}\left|m\right|\leq2^{k\alpha-n}\left\lceil \smash{2^{n(1-\alpha)}}\right\rceil <2^{k\alpha-n}\left(2^{n(1-\alpha)}+1\right)=2^{(k-n)\alpha}+2^{k\alpha-n} & \leq2^{5\alpha}+2^{k-n}\leq2^{5}+2^{5}.
\end{align*}
Finally, having in mind 
\[
0\leq\ell m\leq\left(2^{k(1-\alpha)}+1\right)\left(2^{n(1-\alpha)}+1\right)=2^{n(1-\alpha)}2^{k(1-\alpha)}+2^{k(1-\alpha)}+2^{n(1-\alpha)}+1,
\]
as well as $\ell\geq2^{k(1-\alpha)}-65>0$ and $m\geq2^{n(1-\alpha)}-65>0$,
we get
\begin{align*}
\left|2^{k-n\alpha}-2^{k\alpha-n}\ell m\right| & =2^{k\alpha-n}\cdot\left|2^{n(1-\alpha)+k(1-\alpha)}-\ell m\right|\\
 & \leq2^{k\alpha-n}\cdot\left(\left|2^{n(1-\alpha)+k(1-\alpha)}+2^{k(1-\alpha)}+2^{n(1-\alpha)}+1-\ell m\right|+\left|-2^{k(1-\alpha)}-2^{n(1-\alpha)}-1\right|\right)\\
 & =2^{k\alpha-n}\cdot\left[\left(2^{n(1-\alpha)+k(1-\alpha)}+2^{k(1-\alpha)}+2^{n(1-\alpha)}+1-\ell m\right)+(2^{k(1-\alpha)}+2^{n(1-\alpha)}+1)\right]\\
 & \leq2^{k\alpha-n}\cdot\left[2^{n(1-\alpha)+k(1-\alpha)}-\left(2^{k(1-\alpha)}-65\right)\left(2^{n(1-\alpha)}-65\right)+2\cdot\left(2^{k(1-\alpha)}+2^{n(1-\alpha)}+1\right)\right]\\
 & =2^{k\alpha-n}\cdot\left[65\cdot2^{k(1-\alpha)}+65\cdot2^{n(1-\alpha)}-65^{2}+2\cdot\left(2^{k(1-\alpha)}+2^{n(1-\alpha)}+1\right)\right]\\
 & \leq2^{k\alpha-n}\cdot\left(67\cdot2^{k(1-\alpha)}+67\cdot2^{n(1-\alpha)}+2\right)\\
 & =2^{k-n}\left(67+67\cdot2^{(n-k)(1-\alpha)}+2\cdot2^{-k(1-\alpha)}\right)\\
 & \leq2^{5}\left(67+67\cdot2^{5}+2\right)=70\,816.
\end{align*}
Thus, we have $\left\Vert T_{i}^{-1}T_{j}\right\Vert \leq2^{5}+2^{6}+2^{6}+70\,816=70\,976=:K_{1}$
for all $j\in M_{i}^{\left(1\right)}$, as long as $\alpha\neq1$
and $i=\left(n,m,\varepsilon,0\right)\in I_{0}$ with $n\geq\frac{12}{1-\alpha}$
and $m\geq\frac{64}{3}-1$.

\medskip{}

\textbf{Case 2(b)(ii)}: $m\leq-\frac{64}{3}+1$. Then we have $\frac{\eta}{\xi}<2^{-n\left(1-\alpha\right)}\left(m+1\right)<0$
and $2^{-k\left(1-\alpha\right)}\left(\ell-1\right)<\frac{\xi}{\eta}<0$.
Hence, $\ell<1$ and since $|\ell|\geq\frac{64}{3}-1$, we have $\ell\leq-\frac{64}{3}+1$.
Setting $\tilde{m}:=-m$ and $\tilde{\ell}:=-\ell$ and using $-\frac{\eta}{\xi},-\frac{\xi}{\eta}$
instead of $\frac{\eta}{\xi},\frac{\xi}{\eta}$ we get, with the same
arguments as in the previous case, that $\tilde{m}\geq2^{n(1-\alpha)}-65$
and $\tilde{\ell}\geq2^{k(1-\alpha)}-65$, i.e. $m\leq-2^{n(1-\alpha)}+65$
and $\ell\leq-2^{k(1-\alpha)}+65$. Consequently, since $m\ell=\tilde{m}\tilde{\ell}$
and $|m|=|\tilde{m}|$, as well as $|\ell|=|\tilde{\ell}|$, we get
the same bounds for the matrix entries as in the previous case. Thus,
$\left\Vert T_{i}^{-1}T_{j}\right\Vert \leq K_{1}$.

\medskip{}

All in all, since the cases 2(b)(i) and 2(b)(ii) are the only ones
possible—assuming that we are in case 2(b)—we get $C_{i}^{\left(1\right)}\leq K_{1}$
if $\alpha\neq1$ and if $i=\left(n,m,\varepsilon,0\right)$ satisfies
$n\geq\frac{12}{1-\alpha}$. Finally, in both of the cases from above,
we saw that $\ell\leq-2^{k\left(1-\alpha\right)}+65\leq-\left\lceil 2^{k\left(1-\alpha\right)}\right\rceil +66$
or that $\ell\geq2^{k\left(1-\alpha\right)}-65\geq\left\lceil 2^{k\left(1-\alpha\right)}\right\rceil -66$.
Consequently, we get for the whole case 2(b) that $M_{i}^{(1)}\subset\widetilde{M}$
with
\[
\widetilde{M}:=\bigcup_{\lambda=n-5}^{n+5}\left[\left\{ \lambda\right\} \times\left(\left\{ \left\lceil \smash{2^{\lambda(1-\alpha)}}\right\rceil -66,\dots,\left\lceil \smash{2^{\lambda(1-\alpha)}}\right\rceil \right\} \cup\left\{ -\left\lceil \smash{2^{\lambda(1-\alpha)}}\right\rceil ,\dots,-\left\lceil \smash{2^{\lambda(1-\alpha)}}\right\rceil +66\right\} \right)\times\left\{ \pm1\right\} \times\left\{ 1\right\} \right]
\]
and thus $\left|\smash{M_{i}^{(1)}}\right|\leq\left|\smash{\widetilde{M}}\right|\leq11\cdot2\cdot67\cdot2=2948=:N_{1}$,
independent of $i=\left(n,m,\varepsilon,0\right)\in I_{0}$, as long
as $\alpha\neq1$ and $n\geq\frac{12}{1-\alpha}$.

\medskip{}

\textbf{Case 2(c)}: $\alpha=1$. In this case, the matrix $T_{i}^{-1}T_{j}$
from equation \eqref{eq:matrixentries} reduces to 
\[
T_{i}^{-1}T_{j}=\varepsilon\beta\left(\begin{array}{c|c}
2^{k-n}\ell & 2^{k-n}\\
2^{k-n}-2^{k-n}\ell m & -2^{k-n}m
\end{array}\right)
\]
and we have $\left|m\right|\leq G_{n}=1$, as well as $\left|\ell\right|\leq G_{k}=1$.
Thus, recalling $\left|n-k\right|\leq5$, we can easily bound all
matrix elements uniformly: We have $\left|2^{k-n}\ell\right|=2^{k-n}\left|\ell\right|\leq2^{5}$
and $\left|2^{k-n}\right|\leq2^{5}$, as well as $\left|-2^{k-n}m\right|\leq2^{k-n}\leq2^{5}$
and finally
\[
\left|2^{k-n}-2^{k-n}\ell m\right|=2^{k-n}\left|1-\ell m\right|\leq2^{k-n}\left(1+\left|\ell\right|\cdot\left|m\right|\right)\leq2^{5}\cdot2
\]
and thus $\left\Vert T_{i}^{-1}T_{j}\right\Vert \leq2^{5}+2^{5}+2^{5}+2\cdot2^{5}=160=:K_{2}$,
independent of $i=\left(n,m,\varepsilon,0\right)\in I_{0}$, as long
as $\alpha=1$.

Furthermore, since we saw above that $\left|k-n\right|\leq5$ for
$j=\left(k,\ell,\beta,1\right)\in M_{i}^{\left(1\right)}$, we get
\[
M_{i}^{(1)}\subset\bigcup_{\lambda=n-5}^{n+5}\left[\left\{ \lambda\right\} \times\left\{ -1,0,1\right\} \times\left\{ \pm1\right\} \times\left\{ 1\right\} \right]
\]
and thus $\left|\smash{M_{i}^{\left(1\right)}}\right|\leq11\cdot3\cdot2=66=:N_{2}$.

\medskip{}

All in all, the cases 2(a), 2(b) and 2(c) entail for $i=\left(n,m,\varepsilon,0\right)\in I_{0}$
that 
\[
C_{i}^{(1)}\leq K_{3}:=\begin{cases}
K_{2}, & \text{if }\alpha=1,\\
\max\left\{ K_{0},K_{1}\right\} , & \text{if }\alpha\neq1
\end{cases}\qquad\text{ and also }\qquad\left|\smash{M_{i}^{\left(1\right)}}\right|\leq N_{3}:=\begin{cases}
N_{2}, & \text{if }\alpha=1,\\
\max\left\{ N_{0},N_{1}\right\} , & \text{if }\alpha\neq1.
\end{cases}
\]
Furthermore, putting cases 1 and 2 together yields for arbitrary $i=\left(n,m,\varepsilon,0\right)\in I_{0}$
that 
\[
C_{i}:=\sup_{j\in M_{i}}\left\Vert T_{i}^{-1}T_{j}\right\Vert =\max\left\{ \smash{C_{i}^{\left(0\right)}},\smash{C_{i}^{\left(1\right)}}\right\} \leq\max\left\{ K,K_{3}\right\} =:K_{4}
\]
and 
\[
\left|M_{i}\right|=\left|\smash{M_{i}^{\left(0\right)}}\cup\smash{M_{i}^{\left(1\right)}}\right|\leq\left|\smash{M_{i}^{\left(0\right)}}\right|+\left|\smash{M_{i}^{\left(1\right)}}\right|\leq N+N_{3}=:N_{4}.
\]
As we saw above, this even holds for arbitrary $i\in I_{0}$ (i.e.,
without assuming that the last component of $i$ is $0$), since $M_{i'}=\left\{ j'\with j\in M_{i}\right\} $
and since $C_{i'}=C_{i}$, cf.\@ equation \eqref{eq:AlphaCoveringAdmissibleMirroredIndex}
and the ensuing paragraph.

\medskip{}

Next, we show that $0^{\ast}$ is finite: For $i=\left(n,m,\varepsilon,\delta\right)\in I_{0}$
and $\left(\begin{smallmatrix}\xi\\
\eta
\end{smallmatrix}\right)\in S_{i}^{\left(\alpha\right)}$, we saw in Lemma \ref{lem:AlphaShearletCoveringAuxiliary} that $\left|\left(\xi,\eta\right)\right|>2^{n-2}$.
Since we clearly have $\left|\left(\xi,\eta\right)\right|<2$ for
$\left(\begin{smallmatrix}\xi\\
\eta
\end{smallmatrix}\right)\in S_{0}^{\left(\alpha\right)}=\left(-1,1\right)^{2}$, this implies that $S_{i}^{\left(\alpha\right)}\cap S_{0}^{\left(\alpha\right)}\neq\emptyset$
can only hold if $2^{n}<2^{3}$, i.e., if $n\leq2$. This implies
\begin{align*}
0^{\ast} & \subset\left\{ 0\right\} \cup\left\{ \left(n,m,\varepsilon,\delta\right)\in I_{0}\with n\leq2\right\} \\
 & \subset\left\{ 0\right\} \cup\bigcup_{n=0}^{2}\left[\left\{ n\right\} \times\left\{ -\left\lceil \smash{2^{n\left(1-\alpha\right)}}\right\rceil ,\dots,\left\lceil \smash{2^{n\left(1-\alpha\right)}}\right\rceil \right\} \times\left\{ \pm1\right\} \times\left\{ 0,1\right\} \right],
\end{align*}
which is clearly a finite set. In fact, since $\left\lceil 2^{n\left(1-\alpha\right)}\right\rceil \leq2^{n}\leq4$,
we get $\left|0^{\ast}\right|\leq1+3\cdot2\cdot2\cdot4=49$.

Now, for $i\in I_{0}$, we have $i^{\ast}\subset M_{i}\cup\left\{ 0\right\} $
and thus $\left|i^{\ast}\right|\leq1+N_{4}$. Furthermore, for an
arbitrary $i\in I=I_{0}\cup\left\{ 0\right\} $ we have $\left|i^{\ast}\right|\leq\max\left\{ 1+N_{4},\left|0^{\ast}\right|\right\} $
and thus $N_{\CalS^{\left(\alpha\right)}}<\infty$, i.e., $\CalS^{\left(\alpha\right)}$
is admissible.

Moreover, for $i\in I_{0}\setminus0^{\ast}$, we have $0\notin i^{\ast}$
and thus 
\[
\sup_{j\in i^{*}}\left\Vert T_{i}^{-1}T_{j}\right\Vert =\sup_{j\in M_{i}}\left\Vert T_{i}^{-1}T_{j}\right\Vert \leq K_{4}.
\]
Next, for $i\in0^{\ast}\cap I_{0}$, we have 
\[
\sup_{j\in i^{*}}\left\Vert T_{i}^{-1}T_{j}\right\Vert \leq\sup_{\lambda\in I_{0}\cap0^{\ast}}\left[\max\left\{ \left\Vert T_{\lambda}^{-1}T_{0}\right\Vert ,\:C_{\lambda}\right\} \right]\leq K_{5},
\]
for some fixed constant $K_{5}$, since $0^{\ast}$ is finite. Finally,
again by finiteness of $0^{\ast}$, we also get $\sup_{j\in0^{\ast}}\left\Vert T_{0}^{-1}T_{j}\right\Vert \leq K_{6}$
for a fixed constant $K_{6}$. Thus, in total we get 
\[
C_{\CalS^{\left(\alpha\right)}}=\sup_{i\in I}\:\sup_{j\in i^{*}}\left\Vert T_{i}^{-1}T_{j}\right\Vert \leq\max\left\{ K_{4},K_{5},K_{6}\right\} <\infty.
\]
All in all, we have shown that $\CalS^{\left(\alpha\right)}$ is an
almost structured covering of $\R^{2}$, as claimed.
\end{proof}

\section{The proof of Lemma \ref{lem:MainShearletLemma}}

\label{sec:MegaProof}In this section, we provide the (highly technical
and lengthy) proof of Lemma \ref{lem:MainShearletLemma}. For this
proof, the following lemma will turn out to be extremely useful.
\begin{lem}
\label{lem:WeightedSumOfShiftedIntegrals}For $f:\R^{\dimension}\to\Compl$
and $\theta\in\left[0,\infty\right)$, define $\left\Vert f\right\Vert _{\theta}:=\sup_{x\in\R^{\dimension}}\left(1+\left|x\right|\right)^{\theta}\left|f\left(x\right)\right|\in\left[0,\infty\right]$.

Then, for each $N\in\left[0,\infty\right)$ and $p\in\left(0,\infty\right)$,
arbitrary $\beta,L>0$ and $M\in\R$ and all measurable $f:\R^{\dimension}\to\Compl$
we have
\[
\sum_{k\in\Z}\left|\beta k+M\right|^{N}\left(\int_{\beta k+M-L}^{\beta k+M+L}\left|f\left(x\right)\right|\d x\right)^{p}\leq2^{1+p}\cdot10^{N+3}\cdot\left\Vert f\right\Vert _{\frac{1}{p}\left(N+2\right)}^{p}\cdot L^{p}\cdot\left(1+L^{N}\right)\cdot\left(1+\frac{L+1}{\beta}\right).\qedhere
\]
\end{lem}
\begin{rem*}
Note that $\left(1+\theta\right)^{N}\leq\left[2\cdot\max\left\{ 1,\theta\right\} \right]^{N}\leq2^{N}\cdot\max\left\{ 1,\,\theta^{N}\right\} \leq2^{N}\cdot\left(1+\theta^{N}\right)$
for arbitrary $\theta\geq0$, so that an application of the preceding
lemma with $N=\lambda$ for $\lambda\in\left\{ 0,N\right\} $ yields
\begin{equation}
\begin{split} & \sum_{k\in\Z}\left(1+\left|\beta k+M\right|\right)^{N}\left(\int_{\beta k+M-L}^{\beta k+M+L}\left|f\left(x\right)\right|\d x\right)^{p}\\
 & \leq2^{N}\cdot\sum_{\lambda\in\left\{ 0,N\right\} }\:\sum_{k\in\Z}\left|\beta k+M\right|^{\lambda}\left(\int_{\beta k+M-L}^{\beta k+M+L}\left|f\left(x\right)\right|\d x\right)^{p}\\
 & \leq2^{N}\cdot\sum_{\lambda\in\left\{ 0,N\right\} }2^{1+p}\cdot10^{\lambda+3}\cdot\left\Vert f\right\Vert _{\frac{1}{p}\left(\lambda+2\right)}^{p}\cdot L^{p}\cdot\left(1+L^{\lambda}\right)\cdot\left(1+\frac{L+1}{\beta}\right)\\
 & \leq2^{3+p+N}\cdot10^{N+3}\cdot\left\Vert f\right\Vert _{\frac{1}{p}\left(N+2\right)}^{p}\cdot L^{p}\cdot\left(1+L^{N}\right)\cdot\left(1+\frac{L+1}{\beta}\right).
\end{split}
\label{eq:WeightedSumOfShiftedIntegralsImproved}
\end{equation}
Here, the last step used that we have $\lambda\leq N$ and hence $\left\Vert f\right\Vert _{\frac{1}{p}\left(\lambda+2\right)}^{p}\leq\left\Vert f\right\Vert _{\frac{1}{p}\left(N+2\right)}^{p}$
for $\lambda\in\left\{ 0,N\right\} $ and furthermore that $1+L^{\lambda}=2\leq2\cdot\left(1+L^{N}\right)$
for $\lambda=0$ and trivially $1+L^{\lambda}\leq2\cdot\left(1+L^{N}\right)$
for $\lambda=N$.
\end{rem*}
\begin{proof}
Since otherwise the claim is trivial, we can assume $\left\Vert f\right\Vert _{\frac{1}{p}\left(N+2\right)}<\infty$.
We distinguish three cases for $k\in\Z$:

\textbf{Case 1}: We have $\beta k+M\geq10\cdot L>0$. This implies
$x\geq\beta k+M-L\geq\frac{9}{10}\left(\beta k+M\right)>0$ for arbitrary
$x\in\left[\beta k+M-L,\beta k+M+L\right]$ and hence
\begin{align*}
\left|f\left(x\right)\right| & \leq\left\Vert f\right\Vert _{\frac{1}{p}\left(N+2\right)}\cdot\left(1+\left|x\right|\right)^{-\frac{1}{p}\left(N+2\right)}\\
 & \leq\left\Vert f\right\Vert _{\frac{1}{p}\left(N+2\right)}\cdot\left(1+\frac{9}{10}\left(\beta k+M\right)\right)^{-\frac{1}{p}\left(N+2\right)}\\
 & \leq\left(\frac{10}{9}\right)^{\frac{1}{p}\left(N+2\right)}\cdot\left\Vert f\right\Vert _{\frac{1}{p}\left(N+2\right)}\cdot\left(1+\left|\beta k+M\right|\right)^{-\frac{1}{p}\left(N+2\right)}.
\end{align*}
This yields
\begin{equation}
\left|\beta k+M\right|^{N}\left(\int_{\beta k+M-L}^{\beta k+M+L}\left|f\left(x\right)\right|\d x\right)^{p}\leq\left(\frac{10}{9}\right)^{N+2}\left\Vert f\right\Vert _{\frac{1}{p}\left(N+2\right)}^{p}\cdot\left(2L\right)^{p}\cdot\left(1+\left|\beta k+M\right|\right)^{-2}.\label{eq:IntegralSumLemmaLargeKEstimate}
\end{equation}

\textbf{Case 2}: We have $\beta k+M\leq-10\cdot L<0$. This implies
$x\leq\beta k+M+L\leq\frac{9}{10}\left(\beta k+M\right)<0$ and hence
$\left|x\right|\geq\frac{9}{10}\left|\beta k+M\right|$ for arbitrary
$x\in\left[\beta k+M-L,\beta k+M+L\right]$. This easily implies that
estimate \eqref{eq:IntegralSumLemmaLargeKEstimate} also holds in
this case.

\textbf{Case 3}: We have $\left|\beta k+M\right|\leq10\cdot L$. In
this case, we have $-10\cdot L\leq\beta k+M\leq10\cdot L$ and hence
\[
\frac{-10L-M}{\beta}\leq k\leq\frac{10L-M}{\beta},
\]
which implies $k\in\Z\cap\left[-\frac{M}{\beta}-\frac{10L}{\beta},-\frac{M}{\beta}+\frac{10L}{\beta}\right]$.
But every (closed) interval $I$ of length $R\geq0$ satisfies $\left|I\cap\Z\right|\leq1+R$,
so that there are at most $1+\frac{20L}{\beta}$ possible values of
$k$ for which the present case is satisfied. Hence,
\begin{align*}
\sum_{\substack{k\in\Z\\
\left|\beta k+M\right|\leq10L
}
}\underbrace{\left|\beta k+M\right|^{N}}_{\leq\left(10L\right)^{N}}\left(\int_{\beta k+M-L}^{\beta k+M+L}\left|f\left(x\right)\right|\d x\right)^{p} & \leq10^{N}\left\Vert f\right\Vert _{L^{\infty}}^{p}\cdot\left(1+20\frac{L}{\beta}\right)\cdot L^{N}\cdot\left(2L\right)^{p}\\
 & \leq2^{p}\cdot10^{N}\cdot\left\Vert f\right\Vert _{\frac{1}{p}\left(N+2\right)}^{p}\cdot L^{N+p}\cdot\left(1+20\frac{L}{\beta}\right).
\end{align*}

All in all, we arrive at
\begin{equation}
\begin{split} & \sum_{k\in\Z}\left|\beta k+M\right|^{N}\left(\int_{\beta k+M-L}^{\beta k+M+L}\left|f\left(x\right)\right|\d x\right)^{p}\\
 & \leq2^{p}10^{N}\left\Vert f\right\Vert _{\frac{1}{p}\left(N+2\right)}^{p}L^{N+p}\cdot\left(1+20\frac{L}{\beta}\right)+2^{p}\left(\frac{10}{9}\right)^{N+2}\left\Vert f\right\Vert _{\frac{1}{p}\left(N+2\right)}^{p}L^{p}\cdot\!\!\sum_{\substack{k\in\Z\\
\left|\beta k+M\right|\geq10L
}
}\!\!\left(1+\left|\beta k+M\right|\right)^{-2}.
\end{split}
\label{eq:WeightedSumOfShiftedIntegralsMainEstimate}
\end{equation}

Now, define $g:\R\to\left[0,\infty\right],x\mapsto\sum_{k\in\Z}\left(1+\left|\beta\left(k+x\right)\right|\right)^{-2}$
and note that $g$ is $1$-periodic and also that
\[
\sum_{\substack{k\in\Z\\
\left|\beta k+M\right|\geq10L
}
}\left(1+\left|\beta k+M\right|\right)^{-2}\leq\sum_{k\in\Z}\left(1+\left|\beta k+M\right|\right)^{-2}=\sum_{k\in\Z}\left(1+\left|\beta\left(k+\frac{M}{\beta}\right)\right|\right)^{-2}=g\left(M/\beta\right).
\]
Our next goal is to show $g\left(x\right)\leq2+\frac{10}{\beta}$
for all $x\in\R$. Since $g$ is $1$-periodic, it suffices to consider
$x\in\left[0,1\right]$. Now, we again distinguish three cases regarding
$k\in\Z$:

\textbf{Case 1}: We have $k\geq\frac{1}{\beta}$ and hence $\beta\left(k+x\right)\geq\beta k\geq1$.
This implies
\[
\sum_{k\geq1/\beta}\left(1+\left|\beta\left(k+x\right)\right|\right)^{-2}\leq\sum_{k\geq1/\beta}\left(\beta k\right)^{-2}=\beta^{-2}\cdot\sum_{k\geq1/\beta}k^{-2}.
\]
Now, note for arbitrary $y>0$ that for $n\in\Z_{\geq y}$, we have
$n\geq y>0$ and hence $n\geq1$, which implies $n+1\leq2n$, so that
we get for $z\in\left[n,n+1\right]$ the estimate $z^{-2}\geq\left(n+1\right)^{-2}\geq\left(2n\right)^{-2}=n^{-2}/4$
and hence
\[
\sum_{n\in\Z_{\geq y}}n^{-2}=\sum_{n\geq y}\int_{n}^{n+1}n^{-2}\d z\leq4\sum_{n\geq y}\int_{n}^{n+1}z^{-2}\d z\leq4\cdot\int_{y}^{\infty}z^{-2}\d z=4\cdot\frac{z^{-1}}{-1}\bigg|_{z=y}^{\infty}=\frac{4}{y}.
\]
 Thus, $\sum_{k\geq1/\beta}\left(1+\left|\beta\left(k+x\right)\right|\right)^{-2}\leq\beta^{-2}\cdot\sum_{k\geq1/\beta}k^{-2}\leq\beta^{-2}\cdot\frac{4}{1/\beta}=\frac{4}{\beta}$.

\textbf{Case 2}: We have $k\leq-\frac{1}{\beta}-1$, which entails
$-\left(k+1\right)\geq\frac{1}{\beta}$. For $x\in\left[0,1\right]$,
this implies 
\[
\beta\left(k+x\right)\leq\beta\left(k+1\right)\leq\beta\cdot\left(-\frac{1}{\beta}\right)=-1<0\quad\text{ and hence }\quad\left|\beta\left(k+x\right)\right|=-\beta\left(k+x\right)\geq-\beta\left(k+1\right)>0,
\]
so that we get
\begin{align*}
\sum_{k\in\Z_{\leq-\frac{1}{\beta}-1}}\left(1+\left|\beta\left(k+x\right)\right|\right)^{-2} & \leq\sum_{k\in\Z_{\leq-\frac{1}{\beta}-1}}\left(-\beta\left(k+1\right)\right)^{-2}\\
\left({\scriptstyle \text{with }\ell=-\left(k+1\right)}\right) & =\sum_{\ell\in\Z_{\geq1/\beta}}\left(\beta\ell\right)^{-2}\\
\left({\scriptstyle \text{as above}}\right) & \leq\beta^{-2}\cdot\frac{4}{1/\beta}=\frac{4}{\beta}.
\end{align*}

\textbf{Case 3}: We have $-\frac{1}{\beta}-1\leq k\leq\frac{1}{\beta}$
and hence $k\in\Z\cap\left[-\frac{1}{\beta}-1,\frac{1}{\beta}\right]$,
so that there are at most $2+\frac{2}{\beta}$ possible values of
$k$ for which this case holds. Hence,
\[
\sum_{\substack{k\in\Z\\
-\frac{1}{\beta}-1\leq k\leq\frac{1}{\beta}
}
}\left(1+\left|\beta\left(k+x\right)\right|\right)^{-2}\leq2\left(1+\frac{1}{\beta}\right).
\]
Summarizing all three cases, we easily see $g\left(x\right)\leq\frac{4}{\beta}+\frac{4}{\beta}+2\left(1+\frac{1}{\beta}\right)=2+\frac{10}{\beta}$
for all $x\in\R$, as claimed.

Returning to the proof of the claim of the lemma, we recall from equation
\eqref{eq:WeightedSumOfShiftedIntegralsMainEstimate} (and the displayed
equation after that) that we have
\begin{align*}
 & \sum_{k\in\Z}\left|\beta k+M\right|^{N}\left(\int_{\beta k+M-L}^{\beta k+M+L}\left|f\left(x\right)\right|\d x\right)^{p}\\
 & \leq2^{p}\cdot10^{N}\left\Vert f\right\Vert _{\frac{1}{p}\left(N+2\right)}^{p}\cdot L^{N+p}\cdot\left(1+20\frac{L}{\beta}\right)+2^{p}\left(\frac{10}{9}\right)^{N+2}\!\!\left\Vert f\right\Vert _{\frac{1}{p}\left(N+2\right)}^{p}\cdot L^{p}\cdot g\left(\frac{M}{\beta}\right)\\
 & \leq2^{p}\cdot10^{N+2}\cdot\left\Vert f\right\Vert _{\frac{1}{p}\left(N+2\right)}^{p}\cdot\left[L^{N+p}\cdot\left(1+20\frac{L}{\beta}\right)+L^{p}\cdot g\left(M/\beta\right)\right]\\
 & \leq2^{p}\cdot10^{N+2}\cdot\left\Vert f\right\Vert _{\frac{1}{p}\left(N+2\right)}^{p}\cdot L^{p}\cdot\left[L^{N}\cdot\left(1+20\frac{L}{\beta}\right)+\left(2+\frac{10}{\beta}\right)\right]\\
 & \leq2^{1+p}\cdot10^{N+3}\cdot\left\Vert f\right\Vert _{\frac{1}{p}\left(N+2\right)}^{p}\cdot L^{p}\cdot\left[L^{N}\cdot\left(1+\frac{L}{\beta}\right)+\left(1+\frac{1}{\beta}\right)\right]\\
 & \leq2^{1+p}\cdot10^{N+3}\cdot\left\Vert f\right\Vert _{\frac{1}{p}\left(N+2\right)}^{p}\cdot L^{p}\cdot\left(1+L^{N}\right)\cdot\max\left\{ 1+\frac{L}{\beta},\,1+\frac{1}{\beta}\right\} \\
 & \leq2^{1+p}\cdot10^{N+3}\cdot\left\Vert f\right\Vert _{\frac{1}{p}\left(N+2\right)}^{p}\cdot L^{p}\cdot\left(1+L^{N}\right)\cdot\left(1+\frac{L+1}{\beta}\right),
\end{align*}
which completes the proof.
\end{proof}
The proof of Lemma \ref{lem:MainShearletLemma} will occupy the whole
remainder of this section. In fact, we divide the remainder of this
section into several subsections, each of which handles a certain
subset of the whole set of pairs $\left(i,j\right)\in I^{2}$. Precisely,
we define for $\left(e,d\right)\in\left\{ \pm1\right\} \times\left\{ 0,1\right\} $
the set
\[
I^{\left(e,d\right)}:=\left\{ \left(n,m,\varepsilon,\delta\right)\in I_{0}\with\varepsilon=e\text{ and }\delta=d\right\} .
\]
Furthermore, we set $I^{\left(0\right)}:=\left\{ 0\right\} $ and
$L:=\left\{ 0\right\} \cup\left(\left\{ \pm1\right\} \times\left\{ 0,1\right\} \right)$.
Then $I=\biguplus_{\ell\in L}I^{\left(\ell\right)}$, so that
\begin{equation}
\sup_{i\in I}\:\sum_{j\in I}M_{j,i}^{\left(0\right)}\leq\sum_{\ell_{1}\in L}\:\sup_{i\in I^{\left(\ell_{1}\right)}}\:\sum_{\ell_{2}\in L}\:\sum_{j\in I^{\left(\ell_{2}\right)}}M_{j,i}^{\left(0\right)}\leq\sum_{\ell_{1},\ell_{2}\in L}\:\sup_{i\in I^{\left(\ell_{1}\right)}}\:\sum_{j\in I^{\left(\ell_{2}\right)}}M_{j,i}^{\left(0\right)}\label{eq:ShearletSchurTestSubdivision1}
\end{equation}
and likewise
\begin{equation}
\sup_{j\in I}\:\sum_{i\in I}M_{j,i}^{\left(0\right)}\leq\sum_{\ell_{1},\ell_{2}\in L}\:\sup_{j\in I^{\left(\ell_{2}\right)}}\:\sum_{i\in I^{\left(\ell_{1}\right)}}M_{j,i}^{\left(0\right)}.\label{eq:ShearletSchurTestSubdivision2}
\end{equation}

Now, each of the subsections of this section handles a specific choice
of $\ell_{1},\ell_{2}\in L$, which in principle are $25$ cases.
Luckily, it will turn out that many of these cases can be handled
completely analogously, so that the actual number of subsections is
smaller.

We first only consider the case $\ell_{1},\ell_{2}\in\left\{ \pm1\right\} \times\left\{ 0,1\right\} $.
Then, $I^{\left(\ell_{1}\right)},I^{\left(\ell_{2}\right)}\subset I_{0}$,
so that $\varrho_{j}=\varrho_{i}=\varrho$ and so that $i\in I^{\left(\ell_{1}\right)}$
and $j\in I^{\left(\ell_{2}\right)}$ are of the form $i=\left(n,m,\varepsilon,\delta\right)$
and $j=\left(\nu,\mu,e,d\right)$ for certain $n,\nu\in\N_{0}$, $m,\mu\in\Z$
with $\left|m\right|\leq G_{n}$ and $\left|\mu\right|\leq G_{\nu}$
and certain $\varepsilon,e\in\left\{ \pm1\right\} $ and $\delta,d\in\left\{ 0,1\right\} $.
We will keep this convention throughout the section, without mentioning
it explicitly.

In the remainder of the proof, the notation $x_{+}:=\left(x\right)_{+}:=\max\left\{ 0,x\right\} $
for $x\in\R$ will be frequently useful. We immediately observe $2^{x_{+}}=\max\left\{ 1,2^{x}\right\} $
and $\min\left\{ 1,2^{x}\right\} =2^{-\left(-x\right)_{+}}$.

Next, we collect two estimates concerning $\theta_{1},\theta_{2}$
that will frequently be useful: First, if $C^{-1}\leq\eta\leq C$
for some $C\geq1$, then $1+\left|\eta\xi\right|\geq1+C^{-1}\left|\xi\right|\geq C^{-1}\cdot\left(1+\left|\xi\right|\right)$
and thus
\begin{equation}
\begin{split}\theta_{1}\left(\eta\xi\right) & =\min\left\{ \left|\eta\xi\right|^{M_{1}},\left(1+\left|\eta\xi\right|\right)^{-M_{2}}\right\} \\
 & \leq\min\left\{ C^{M_{1}}\cdot\left|\xi\right|^{M_{1}},\,C^{M_{2}}\cdot\left(1+\left|\xi\right|\right)^{-M_{2}}\right\} \leq C^{M_{3}}\cdot\theta_{1}\left(\xi\right)
\end{split}
\label{eq:Theta1Rescaling}
\end{equation}
for arbitrary $\xi\in\R$ and $M_{3}:=\max\left\{ M_{1},M_{2}\right\} $.

Finally, if $\eta\geq C$ for some $C\in\left(0,1\right]$, then $1+\left|\eta\xi\right|\geq1+C\cdot\left|\xi\right|\geq C\cdot\left(1+\left|\xi\right|\right)$,
so that
\begin{equation}
\theta_{2}\left(\eta\xi\right)=\left(1+\left|\eta\xi\right|\right)^{-K}\leq C^{-K}\cdot\left(1+\left|\xi\right|\right)^{-K}=C^{-K}\cdot\theta_{2}\left(\xi\right)\qquad\forall\xi\in\R.\label{eq:Theta2Rescaling}
\end{equation}
Now, we properly start the proof of Lemma \ref{lem:MainShearletLemma}
by distinguishing the different values of $\ell_{1},\ell_{2}\in L$.

\subsection{We have \texorpdfstring{$\ell_{1}=\ell_{2}=\left(1,0\right)$}{ℓ₁=ℓ₂=(1,0)}}

\label{subsec:BothRightCone}For brevity, let $\ell:=\left(1,0\right)$.
Geometrically, the present case means that $i,j\in I^{\left(\ell\right)}$
both belong to the right cone, i.e., $\varepsilon=e=1$ and $\delta=d=0$.
Thus, we have
\begin{align*}
T_{j}^{-1}T_{i} & =\left(\begin{matrix}1 & 0\\
-\mu & 1
\end{matrix}\right)\left(\begin{matrix}2^{-\nu} & 0\\
0 & 2^{-\nu\alpha}
\end{matrix}\right)\left(\begin{matrix}2^{n} & 0\\
0 & 2^{n\alpha}
\end{matrix}\right)\left(\begin{matrix}1 & 0\\
m & 1
\end{matrix}\right)\\
 & =\left(\begin{matrix}1 & 0\\
-\mu & 1
\end{matrix}\right)\left(\begin{matrix}2^{n-\nu} & 0\\
0 & 2^{\alpha\left(n-\nu\right)}
\end{matrix}\right)\left(\begin{matrix}1 & 0\\
m & 1
\end{matrix}\right)\\
 & =\left(\begin{matrix}1 & 0\\
-\mu & 1
\end{matrix}\right)\left(\begin{matrix}2^{n-\nu} & 0\\
2^{\alpha\left(n-\nu\right)}m & 2^{\alpha\left(n-\nu\right)}
\end{matrix}\right)\\
 & =\left(\begin{array}{c|c}
2^{n-\nu} & 0\\
2^{\alpha\left(n-\nu\right)}m-2^{n-\nu}\mu & 2^{\alpha\left(n-\nu\right)}
\end{array}\right)
\end{align*}
and hence, since $2^{\alpha(n-\nu)}\leq2^{\alpha(n-\nu)_{+}}\leq2^{(n-\nu)_{+}}$
and $2^{n-\nu}\leq2^{(n-\nu)_{+}}$, 
\[
\left\Vert T_{j}^{-1}T_{i}\right\Vert \leq2\cdot\left(2^{\left(n-\nu\right)_{+}}+\omega_{n,m,\nu,\mu}\right)\leq2\cdot2^{\left(n-\nu\right)_{+}}\cdot\left(1+\omega_{n,m,\nu,\mu}\right)\quad\text{ for }\quad\omega_{n,m,\nu,\mu}:=\left|2^{\alpha\left(n-\nu\right)}m-2^{n-\nu}\mu\right|,
\]
which finally yields
\begin{equation}
\left(1+\left\Vert T_{j}^{-1}T_{i}\right\Vert \right)^{\sigma}\leq3^{\sigma}\cdot2^{\sigma\cdot\left(n-\nu\right)_{+}}\cdot\left(1+\omega_{n,m,\nu,\mu}\right)^{\sigma}.\label{eq:BothRightConeCoordinateChangeNormEstimate}
\end{equation}

On the other hand, with $\varrho,\theta_{1},\theta_{2}$ as in equation
\eqref{eq:MotherShearletMainEstimate}, we have because of $\varrho_{j}=\varrho$
that
\begin{align}
\left|\det T_{i}\right|^{-1}\!\cdot\!\int_{S_{i}^{\left(\alpha\right)}}\varrho_{j}\left(T_{j}^{-1}\xi\right)\d\xi & =\left|\det T_{i}\right|^{-1}\cdot\int_{T_{i}Q}\varrho\left(T_{j}^{-1}\xi\right)\d\xi\nonumber \\
\left({\scriptstyle \xi=T_{i}\eta}\right) & =\int_{Q}\varrho\left(T_{j}^{-1}T_{i}\eta\right)\d\eta\nonumber \\
\left({\scriptstyle \text{def. of }Q,\text{ cf. Def. }\ref{def:AlphaShearletCovering}}\right) & =\int_{1/3}^{3}\int_{\R}\Indicator_{\left(-1,1\right)}\left(\frac{\eta_{2}}{\eta_{1}}\right)\cdot\varrho\left(\begin{matrix}2^{n-\nu}\eta_{1}\\
\left(2^{\alpha\left(n-\nu\right)}m-2^{n-\nu}\mu\right)\eta_{1}+2^{\alpha\left(n-\nu\right)}\eta_{2}
\end{matrix}\right)\d\eta_{2}\d\eta_{1}\nonumber \\
\left({\scriptstyle \xi=\frac{\eta_{2}}{\eta_{1}}\text{ in inner integral}}\right) & =\!\int_{\frac{1}{3}}^{3}\!\eta_{1}\!\int_{\R}\!\Indicator_{\left(-1,1\right)}\left(\xi\right)\cdot\theta_{1}\left(2^{n-\nu}\eta_{1}\right)\cdot\theta_{2}\left(\!\left(2^{\alpha\left(n-\nu\right)}m\!-\!2^{n-\nu}\mu\right)\eta_{1}\!+\!2^{\alpha\left(n-\nu\right)}\xi\eta_{1}\right)\d\xi\d\eta_{1}\nonumber \\
 & \leq3\cdot\int_{1/3}^{3}\theta_{1}\left(2^{n-\nu}\eta_{1}\right)\cdot\int_{-1}^{1}\left(1+\eta_{1}\cdot\left|\left(2^{\alpha\left(n-\nu\right)}m-2^{n-\nu}\mu\right)+2^{\alpha\left(n-\nu\right)}\xi\right|\right)^{-K}\d\xi\d\eta_{1}\nonumber \\
\left({\scriptstyle \text{eq. }\eqref{eq:Theta2Rescaling}}\right) & \overset{}{\leq}3^{K+1}\cdot\int_{1/3}^{3}\theta_{1}\left(2^{n-\nu}\eta_{1}\right)\d\eta_{1}\cdot\int_{-1}^{1}\left(1+\left|\left(2^{\alpha\left(n-\nu\right)}m-2^{n-\nu}\mu\right)+2^{\alpha\left(n-\nu\right)}\xi\right|\right)^{-K}\d\xi\nonumber \\
\left({\scriptstyle \eta_{2}=2^{\alpha\left(n-\nu\right)}m-2^{n-\nu}\mu+2^{\alpha\left(n-\nu\right)}\xi}\right) & =3^{K+1}\cdot2^{\alpha\left(\nu-n\right)}\cdot\int_{1/3}^{3}\theta_{1}\left(2^{n-\nu}\eta_{1}\right)\d\eta_{1}\cdot\int_{2^{\alpha\left(n-\nu\right)}m-2^{n-\nu}\mu-2^{\alpha\left(n-\nu\right)}}^{2^{\alpha\left(n-\nu\right)}m-2^{n-\nu}\mu+2^{\alpha\left(n-\nu\right)}}\left(1+\left|\eta_{2}\right|\right)^{-K}\d\eta_{2}\nonumber \\
\left({\scriptstyle \text{eq. }\eqref{eq:Theta1Rescaling}}\right) & \leq3^{2+K+M_{3}}\cdot2^{\alpha\left(\nu-n\right)}\cdot\theta_{1}\left(2^{n-\nu}\right)\cdot\int_{2^{\alpha\left(n-\nu\right)}m-2^{n-\nu}\mu-2^{\alpha\left(n-\nu\right)}}^{2^{\alpha\left(n-\nu\right)}m-2^{n-\nu}\mu+2^{\alpha\left(n-\nu\right)}}\left(1+\left|\eta_{2}\right|\right)^{-K}\d\eta_{2}.\label{eq:BothRightConeIntegralCalculation}
\end{align}

\medskip{}

Now, since the assumptions of Lemma \ref{lem:MainShearletLemma} ensure
$K\geq\frac{1}{\tau}\left(\sigma+2\right)$, an application of Lemma
\ref{lem:WeightedSumOfShiftedIntegrals} and of the associated remark
(with $p=\tau\in\left(0,\infty\right)$, $\beta=2^{\alpha\left(n-\nu\right)}>0$,
$N=\sigma\geq0$, $M=-2^{n-\nu}\mu\in\R$ and $L=2^{\alpha\left(n-\nu\right)}>0$)
yields
\begin{align}
 & \sum_{m\in\Z}\!\left(1\!+\negthinspace\left|2^{\alpha\left(n-\nu\right)}m\!+\!\left(-2^{n-\nu}\mu\right)\right|\right)^{\sigma}\!\left[\!\int_{2^{\alpha\left(n-\nu\right)}m+\left(-2^{n-\nu}\mu\right)-2^{\alpha\left(n-\nu\right)}}^{2^{\alpha\left(n-\nu\right)}m+\left(-2^{n-\nu}\mu\right)+2^{\alpha\left(n-\nu\right)}}\!\!\!\left(1\!+\!\left|\eta_{2}\right|\right)^{-K}\d\eta_{2}\right]^{\!\tau}\nonumber \\
 & \leq2^{3+p+N}10^{N+3}\cdot\left\Vert \left(1+\left|\mybullet\right|\right)^{-K}\right\Vert _{\frac{1}{p}\left(N+2\right)}^{p}\cdot L^{p}\cdot\left(1+L^{N}\right)\cdot\left(1+\frac{L+1}{\beta}\right)\nonumber \\
\left({\scriptstyle \left\Vert \left(1+\left|\mybullet\right|\right)^{-K}\right\Vert _{\frac{1}{p}\left(N+2\right)}^{p}\leq1\text{ since }K\geq\frac{2+\sigma}{\tau}}\right) & \leq2^{3+\tau+\sigma}\cdot10^{\sigma+3}\cdot2^{\alpha\tau\left(n-\nu\right)}\cdot\left(1+2^{\alpha\left(n-\nu\right)\sigma}\right)\cdot\left(1+\frac{1+2^{\alpha\left(n-\nu\right)}}{2^{\alpha\left(n-\nu\right)}}\right)\nonumber \\
 & \leq2^{5+\tau+\sigma}\cdot10^{\sigma+3}\cdot2^{\alpha\tau\left(n-\nu\right)}\cdot2^{\alpha\sigma\cdot\left(n-\nu\right)_{+}}\cdot\left(1+2^{\alpha\left(\nu-n\right)}\right)\nonumber \\
 & \leq2^{6+\tau+\sigma}\cdot10^{\sigma+3}\cdot2^{\alpha\tau\left(n-\nu\right)+\alpha\sigma\left(n-\nu\right)_{+}}\cdot2^{\alpha\cdot\left(\nu-n\right)_{+}}\nonumber \\
 & \leq2^{18+\tau+5\sigma}\cdot2^{\alpha\tau\left(n-\nu\right)+\alpha\sigma\left(n-\nu\right)_{+}+\alpha\cdot\left(\nu-n\right)_{+}}.\label{eq:BothRightConeSpecialIntegralSumOverM}
\end{align}
Consequently, we get for arbitrary $j=\left(\nu,\mu,1,0\right)\in I^{\left(\ell\right)}$
the estimate
\begin{align*}
 & \sum_{i\in I^{\left(\ell\right)}}\!\left[\!\left(\frac{w_{j}^{s}}{w_{i}^{s}}\right)^{\tau}\!\!\left(1+\left\Vert T_{j}^{-1}T_{i}\right\Vert \right)^{\sigma}\left(\left|\det T_{i}\right|^{-1}\int_{S_{i}^{\left(\alpha\right)}}\varrho_{j}\left(T_{j}^{-1}\xi\right)\d\xi\right)^{\!\!\tau}\,\right]\\
\left({\scriptstyle \text{eqs. }\eqref{eq:BothRightConeCoordinateChangeNormEstimate},\,\eqref{eq:BothRightConeIntegralCalculation}}\right) & \leq\sum_{n\in\N_{0}}\left(3^{\sigma}\cdot3^{\tau\left(2+K+M_{3}\right)}\cdot2^{\left(\tau s+\tau\alpha\right)\left(\nu-n\right)+\sigma\left(n-\nu\right)_{+}}\cdot\left[\theta_{1}\left(2^{n-\nu}\right)\right]^{\tau}\cdot\vphantom{\sum_{m\in\Z}}\right.\\
 & \phantom{\lesssim\sum_{n\in\N_{0}}\bigg(}\left.\sum_{m\in\Z}\!\left[\!\left(1\!+\!\left|2^{\alpha\left(n-\nu\right)}m-2^{n-\nu}\mu\right|\right)^{\sigma}\!\cdot\!\left(\!\int_{2^{\alpha\left(n-\nu\right)}m-2^{n-\nu}\mu-2^{\alpha\left(n-\nu\right)}}^{2^{\alpha\left(n-\nu\right)}m-2^{n-\nu}\mu+2^{\alpha\left(n-\nu\right)}}\!\!\!\left(1\!+\!\left|\eta_{2}\right|\right)^{-K}\!\d\eta_{2}\!\right)^{\!\!\tau}\right]\!\right)\\
\left({\scriptstyle \text{eq. }\eqref{eq:BothRightConeSpecialIntegralSumOverM}}\right) & \leq2^{18+\tau+7\sigma}3^{\tau\left(2+K+M_{3}\right)}\cdot\sum_{n\in\N_{0}}\!\left(2^{\left(\tau s+\tau\alpha\right)\left(\nu-n\right)+\sigma\left(n-\nu\right)_{+}}\cdot\left[\theta_{1}\left(2^{n-\nu}\right)\right]^{\tau}\cdot2^{\alpha\tau\left(n-\nu\right)+\alpha\sigma\left(n-\nu\right)_{+}+\alpha\cdot\left(\nu-n\right)_{+}}\!\right)\\
 & \leq2^{18+7\sigma+\tau\left(5+2K+2M_{3}\right)}\cdot\sum_{n\in\N_{0}}\left(2^{\tau s\left(\nu-n\right)+\alpha\cdot\left(\nu-n\right)_{+}+\sigma\left(1+\alpha\right)\left(n-\nu\right)_{+}}\cdot\left[\theta_{1}\left(2^{n-\nu}\right)\right]^{\tau}\right).
\end{align*}
Now, observe $\theta_{1}\left(2^{n-\nu}\right)=\min\left\{ 2^{\left(n-\nu\right)M_{1}},\left(1+2^{n-\nu}\right)^{-M_{2}}\right\} \leq\min\left\{ 2^{M_{1}\left(n-\nu\right)},2^{-M_{2}\left(n-\nu\right)}\right\} $
and hence
\[
2^{\tau s\left(\nu-n\right)+\alpha\cdot\left(\nu-n\right)_{+}+\sigma\left(1+\alpha\right)\left(n-\nu\right)_{+}}\cdot\left[\theta_{1}\left(2^{n-\nu}\right)\right]^{\tau}\leq\begin{cases}
2^{-\left|\nu-n\right|\left(\tau M_{1}-\tau s-\alpha\right)}\leq2^{-\tau c\left|\nu-n\right|}, & \text{if }\nu\geq n,\\
2^{-\left|\nu-n\right|\left(\tau M_{2}+\tau s-\sigma\left(1+\alpha\right)\right)}\leq2^{-\tau c\left|\nu-n\right|}, & \text{if }\nu\leq n.
\end{cases}
\]
Here, we used that $M_{2}\geq M_{2}^{(0)}+c\geq\left(1+\alpha\right)\frac{\sigma}{\tau}-s+c$,
as well as $M_{1}\geq M_{1}^{(0)}+c\geq s+\frac{\alpha}{\tau}+c$
by the assumptions of Lemma \ref{lem:MainShearletLemma}. Thus, all
in all, we arrive at
\begin{align*}
 & \sum_{i\in I^{\left(\ell\right)}}\!\left[\left(\frac{w_{j}^{s}}{w_{i}^{s}}\right)^{\tau}\cdot\left(1+\left\Vert T_{j}^{-1}T_{i}\right\Vert \right)^{\sigma}\cdot\left(\left|\det T_{i}\right|^{-1}\int_{S_{i}^{\left(\alpha\right)}}\varrho_{j}\left(T_{j}^{-1}\xi\right)\d\xi\right)^{\!\!\tau}\,\right]\\
 & \leq2^{18+7\sigma+\tau\left(5+2K+2M_{3}\right)}\cdot\sum_{n\in\N_{0}}2^{-\tau c\left|\nu-n\right|}\\
 & \leq2^{18+7\sigma+\tau\left(5+2K+2M_{3}\right)}\cdot\sum_{\ell\in\Z}2^{-\tau c\left|\ell\right|}\leq2^{19+7\sigma+\tau\left(5+2K+2M_{3}\right)}/\left(1-2^{-\tau c}\right).
\end{align*}

\medskip{}

Likewise, for the summation over $j$ instead of $i$, we apply Lemma
\ref{lem:WeightedSumOfShiftedIntegrals} and the associated remark
(using the choices $p=\tau\in\left(0,\infty\right)$, $\beta=2^{n-\nu}>0$,
$N=\sigma\geq0$, $M=2^{\alpha\left(n-\nu\right)}m\in\R$ and $L=2^{\alpha\left(n-\nu\right)}>0$)
to get
\begin{align}
 & \sum_{\mu\in\Z}\left(1+\left|2^{\alpha\left(n-\nu\right)}m+\left(-2^{n-\nu}\mu\right)\right|\right)^{\sigma}\left(\int_{2^{\alpha\left(n-\nu\right)}m+\left(-2^{n-\nu}\mu\right)-2^{\alpha\left(n-\nu\right)}}^{2^{\alpha\left(n-\nu\right)}m+\left(-2^{n-\nu}\mu\right)+2^{\alpha\left(n-\nu\right)}}\left(1+\left|\eta_{2}\right|\right)^{-K}\d\eta_{2}\right)^{\tau}\nonumber \\
\left({\scriptstyle \zeta=-\mu}\right) & =\sum_{\zeta\in\Z}\left(1+\left|2^{n-\nu}\zeta+2^{\alpha\left(n-\nu\right)}m\right|\right)^{\sigma}\left(\int_{2^{n-\nu}\zeta+2^{\alpha\left(n-\nu\right)}m-2^{\alpha\left(n-\nu\right)}}^{2^{n-\nu}\zeta+2^{\alpha\left(n-\nu\right)}m+2^{\alpha\left(n-\nu\right)}}\left(1+\left|\eta_{2}\right|\right)^{-K}\d\eta_{2}\right)^{\tau}\nonumber \\
 & \leq2^{3+p+N}\cdot10^{N+3}\cdot\left\Vert \left(1+\left|\mybullet\right|\right)^{-K}\right\Vert _{\frac{1}{p}\left(N+2\right)}^{p}\cdot L^{p}\cdot\left(1+L^{N}\right)\cdot\left(1+\frac{L+1}{\beta}\right)\nonumber \\
\left({\scriptstyle \text{since }K\geq\frac{\sigma+2}{\tau}}\right) & \leq2^{3+\tau+\sigma}\cdot10^{\sigma+3}\cdot2^{\alpha\tau\left(n-\nu\right)}\cdot\left(1+2^{\sigma\alpha\left(n-\nu\right)}\right)\cdot\left(1+2^{\left(1-\alpha\right)\left(\nu-n\right)}+2^{\nu-n}\right)\nonumber \\
 & \leq2^{18+\tau+5\sigma}\cdot2^{\alpha\tau\left(n-\nu\right)}\cdot2^{\sigma\alpha\cdot\left(n-\nu\right)_{+}}\cdot2^{\left(\nu-n\right)_{+}},\label{eq:BothRightConeMuSeriesEstimate}
\end{align}
where $\left\Vert \left(1+\left|\mybullet\right|\right)^{-K}\right\Vert _{\frac{1}{p}\left(N+2\right)}^{p}\leq1$,
since $K\geq\frac{2+\sigma}{\tau}$ by the assumptions of Lemma \ref{lem:MainShearletLemma}.

Now we get as above for arbitrary $i=\left(n,m,1,0\right)\in I^{\left(\ell\right)}$
that
\begin{align*}
 & \sum_{j\in I^{\left(\ell\right)}}\left[\left(\frac{w_{j}^{s}}{w_{i}^{s}}\right)^{\tau}\left(1+\left\Vert T_{j}^{-1}T_{i}\right\Vert \right)^{\sigma}\left(\left|\det T_{i}\right|^{-1}\int_{S_{i}^{\left(\alpha\right)}}\varrho_{j}\left(T_{j}^{-1}\xi\right)\d\xi\right)^{\tau}\,\right]\\
 & \overset{\left(\ast\right)}{\leq}2^{18+\tau+7\sigma}\!\cdot\!3^{\tau\left(2+K+M_{3}\right)}\cdot\sum_{\nu\in\N_{0}}2^{\tau s\left(\nu-n\right)+\sigma\cdot\left(n-\nu\right)_{+}+\tau\alpha\left(\nu-n\right)+\alpha\tau\left(n-\nu\right)}\cdot\left[\theta_{1}\left(2^{n-\nu}\right)\right]^{\tau}\cdot2^{\sigma\alpha\cdot\left(n-\nu\right)_{+}}\cdot2^{\left(\nu-n\right)_{+}}\\
 & \leq2^{18+7\sigma+\tau\left(5+2K+2M_{3}\right)}\cdot\sum_{\nu\in\N_{0}}2^{\tau s\left(\nu-n\right)+\sigma\left(1+\alpha\right)\cdot\left(n-\nu\right)_{+}+\left(\nu-n\right)_{+}}\cdot\left[\theta_{1}\left(2^{n-\nu}\right)\right]^{\tau}.
\end{align*}
Here, the step marked with $\left(\ast\right)$ is justified by equations
\eqref{eq:BothRightConeCoordinateChangeNormEstimate}, \eqref{eq:BothRightConeIntegralCalculation},
and \eqref{eq:BothRightConeMuSeriesEstimate}.

As above, we observe
\[
2^{\tau s\left(\nu-n\right)+\sigma\left(1+\alpha\right)\cdot\left(n-\nu\right)_{+}+\left(\nu-n\right)_{+}}\cdot\left[\theta_{1}\left(2^{n-\nu}\right)\right]^{\tau}\leq\begin{cases}
2^{-\left|\nu-n\right|\left(\tau M_{1}-\tau s-1\right)}\leq2^{-\tau c\left|\nu-n\right|}, & \text{if }\nu\geq n,\\
2^{-\left|\nu-n\right|\left(\tau M_{2}+\tau s-\sigma\left(1+\alpha\right)\right)}\leq2^{-\tau c\left|\nu-n\right|}, & \text{if }\nu\leq n,
\end{cases}
\]
where we used that we have $M_{1}\geq M_{1}^{(0)}+c\geq\frac{1}{\tau}+s+c$
and $M_{2}\geq M_{2}^{(0)}+c\geq\left(1+\alpha\right)\frac{\sigma}{\tau}-s+c$
by the assumptions of Lemma \ref{lem:MainShearletLemma}. Consequently,
we conclude
\begin{align*}
 & \sum_{j\in I^{\left(\ell\right)}}\left[\left(\frac{w_{j}^{s}}{w_{i}^{s}}\right)^{\tau}\left(1+\left\Vert T_{j}^{-1}T_{i}\right\Vert \right)^{\sigma}\left(\left|\det T_{i}\right|^{-1}\int_{S_{i}^{\left(\alpha\right)}}\varrho_{j}\left(T_{j}^{-1}\xi\right)\d\xi\right)^{\tau}\,\right]\\
 & \leq2^{18+7\sigma+\tau\left(5+2K+2M_{3}\right)}\cdot\sum_{\nu\in\N_{0}}2^{-\tau c\left|\nu-n\right|}\\
 & \leq2^{18+7\sigma+\tau\left(5+2K+2M_{3}\right)}\cdot\sum_{\ell\in\Z}2^{-\tau c\left|\ell\right|}\leq2^{19+7\sigma+\tau\left(5+2K+2M_{3}\right)}/\left(1-2^{-\tau c}\right).
\end{align*}

In summary, in this subsection, we have shown for $C_{0}^{\left(1\right)}:=2^{19+7\sigma+\tau\left(5+2K+2M_{3}\right)}/\left(1\!-\!2^{-\tau c}\right)$
that
\[
\sup_{i\in I^{\left(1,0\right)}}\:\sum_{j\in I^{\left(1,0\right)}}\!\!M_{j,i}^{\left(0\right)}\leq C_{0}^{\left(1\right)}\quad\text{and}\quad\sup_{j\in I^{\left(1,0\right)}}\,\sum_{i\in I^{\left(1,0\right)}}\!\!M_{j,i}^{\left(0\right)}\leq C_{0}^{\left(1\right)}.
\]

\subsection{We have \texorpdfstring{$\ell_{1}=\left(1,1\right)$}{ℓ₁=(1,1)}
and \texorpdfstring{$\ell_{2}=\left(1,0\right)$}{ℓ₂=(1,0)}}

\label{subsec:UpperConeRightCone}Geometrically, the present case
means that $i$ belongs to the top cone, while $j$ belongs to the
right cone, i.e., $e=\varepsilon=1$, $\delta=1$ and $d=0$. In this
case, we have
\begin{align}
T_{j}^{-1}T_{i} & =\left(\begin{matrix}1 & 0\\
-\mu & 1
\end{matrix}\right)\left(\begin{matrix}2^{-\nu} & 0\\
0 & 2^{-\nu\alpha}
\end{matrix}\right)\left(\begin{matrix}0 & 1\\
1 & 0
\end{matrix}\right)\left(\begin{matrix}2^{n} & 0\\
0 & 2^{n\alpha}
\end{matrix}\right)\left(\begin{matrix}1 & 0\\
m & 1
\end{matrix}\right)\nonumber \\
 & =\left(\begin{matrix}1 & 0\\
-\mu & 1
\end{matrix}\right)\left(\begin{matrix}2^{-\nu} & 0\\
0 & 2^{-\nu\alpha}
\end{matrix}\right)\left(\begin{matrix}0 & 2^{n\alpha}\\
2^{n} & 0
\end{matrix}\right)\left(\begin{matrix}1 & 0\\
m & 1
\end{matrix}\right)\nonumber \\
 & =\left(\begin{matrix}2^{-\nu} & 0\\
-2^{-\nu}\mu & 2^{-\nu\alpha}
\end{matrix}\right)\left(\begin{matrix}2^{n\alpha}m & 2^{n\alpha}\\
2^{n} & 0
\end{matrix}\right)\nonumber \\
 & =\left(\begin{array}{c|c}
2^{n\alpha-\nu}m & 2^{n\alpha-\nu}\\
2^{n-\nu\alpha}-2^{n\alpha-\nu}\mu m & -2^{n\alpha-\nu}\mu
\end{array}\right).\label{eq:UpperConeRightConeTransitionMatrix}
\end{align}
As our first step, we want to obtain an estimate for $\left\Vert T_{j}^{-1}T_{i}\right\Vert $.

To this end, recall $\left|m\right|\leq G_{n}=\left\lceil 2^{n\left(1-\alpha\right)}\right\rceil \leq1+2^{n\left(1-\alpha\right)}$;
hence $\left|2^{n\alpha-\nu}m\right|\leq2^{n\alpha-\nu}+2^{n-\nu}\leq2\cdot2^{n-\nu}\leq2\cdot2^{\left(n-\nu\right)_{+}}$.
Likewise, $\left|2^{n\alpha-\nu}\mu\right|\leq2^{n\alpha-\nu}+2^{n\alpha-\nu\alpha}\leq2\cdot2^{\alpha\left(n-\nu\right)}\leq2\cdot2^{\alpha\left(n-\nu\right)_{+}}\leq2\cdot2^{\left(n-\nu\right)_{+}}$
and $2^{n\alpha-\nu}\leq2^{n-\nu}\leq2^{\left(n-\nu\right)_{+}}$.
Finally, setting
\begin{equation}
\kappa:=\frac{\mu}{2^{\nu\left(1-\alpha\right)}}\qquad\text{ and }\qquad\iota:=\frac{m}{2^{n\left(1-\alpha\right)}},\label{eq:KappaIotaDefinition}
\end{equation}
we have
\[
2^{n-\nu\alpha}-2^{n\alpha-\nu}\mu m=2^{n-\nu\alpha}\cdot\left(1-\frac{\mu}{2^{\nu\left(1-\alpha\right)}}\frac{m}{2^{n\left(1-\alpha\right)}}\right)=2^{n-\nu\alpha}\cdot\left(1-\kappa\iota\right)=:2^{n-\nu\alpha}\cdot\lambda_{n,m,\nu,\mu},
\]
and also
\[
\left|\kappa\right|=\frac{\left|\mu\right|}{2^{\nu\left(1-\alpha\right)}}\leq\frac{1+2^{\nu\left(1-\alpha\right)}}{2^{\nu\left(1-\alpha\right)}}=1+2^{-\nu\left(1-\alpha\right)}\leq2\qquad\text{ and }\qquad\left|\iota\right|=\frac{\left|m\right|}{2^{n\left(1-\alpha\right)}}\leq1+2^{-n\left(1-\alpha\right)}\leq2.
\]
All in all, we have shown
\begin{equation}
\begin{split}\left(1+\left\Vert T_{j}^{-1}T_{i}\right\Vert \right)^{\sigma} & \leq\left(1+5\cdot2^{\left(n-\nu\right)_{+}}+2^{n-\nu\alpha}\cdot\left|\lambda_{n,m,\nu,\mu}\right|\right)^{\sigma}\\
 & \leq\left(1+5\cdot2^{\left(n-\nu\right)_{+}}\right)^{\sigma}\cdot\left(1+2^{n-\nu\alpha}\cdot\left|\lambda_{n,m,\nu,\mu}\right|\right)^{\sigma}\\
 & \leq6^{\sigma}\cdot2^{\sigma\cdot\left(n-\nu\right)_{+}}\cdot\left(1+2^{n-\nu\alpha}\cdot\left|\lambda_{n,m,\nu,\mu}\right|\right)^{\sigma}.
\end{split}
\label{eq:UpperConeRightConeMatrixChangeNorm}
\end{equation}

Next, we consider the integral term occurring in $M_{j,i}^{\left(0\right)}$.
Precisely, with $\varrho$ and $\theta_{1}$ as in equation \eqref{eq:MotherShearletMainEstimate},
we observe
\begin{align}
\left|\det T_{i}\right|^{-1}\int_{S_{i}^{\left(\alpha\right)}}\varrho_{j}\left(T_{j}^{-1}\xi\right)\d\xi & =\int_{Q}\varrho\left(T_{j}^{-1}T_{i}\eta\right)\d\eta\nonumber \\
 & =\int_{1/3}^{3}\int_{-\eta_{1}}^{\eta_{1}}\varrho\left(\begin{matrix}2^{n\alpha-\nu}\left(m\eta_{1}+\eta_{2}\right)\\
2^{n-\nu\alpha}\cdot\lambda_{n,m,\nu,\mu}\cdot\eta_{1}-2^{n\alpha-\nu}\mu\eta_{2}
\end{matrix}\right)\d\eta_{2}\d\eta_{1}\nonumber \\
\left({\scriptstyle \text{with }\xi=\eta_{2}/\eta_{1}}\right) & =\int_{1/3}^{3}\eta_{1}\int_{-1}^{1}\varrho\left(\begin{matrix}2^{n\alpha-\nu}\left(m+\xi\right)\eta_{1}\\
\eta_{1}\cdot\left(2^{n-\nu\alpha}\lambda_{n,m,\nu,\mu}-2^{n\alpha-\nu}\mu\xi\right)
\end{matrix}\right)\d\xi\d\eta_{1}\nonumber \\
\left({\scriptstyle \text{since }\eta_{1}\leq3}\right) & \leq3\cdot\int_{1/3}^{3}\int_{-1}^{1}\theta_{1}\left(2^{n\alpha-\nu}\left(m\!+\!\xi\right)\eta_{1}\right)\cdot\left(1\!+\!\left|\eta_{1}\cdot\left(2^{n-\nu\alpha}\lambda_{n,m,\nu,\mu}\!-\!2^{n\alpha-\nu}\mu\xi\right)\right|\right)^{-K}\d\xi\d\eta_{1}\nonumber \\
\left({\scriptstyle \frac{1}{3}\leq\eta_{1}\leq3,\text{ cf. eqs. }\eqref{eq:Theta1Rescaling},\eqref{eq:Theta2Rescaling}}\right) & \leq3^{1+K+M_{3}}\cdot\int_{1/3}^{3}\int_{-1}^{1}\theta_{1}\left(2^{n\alpha-\nu}\left(m\!+\!\xi\right)\right)\cdot\left(1\!+\!\left|2^{n-\nu\alpha}\lambda_{n,m,\nu,\mu}\!-\!2^{n\alpha-\nu}\mu\xi\right|\right)^{-K}\d\xi\d\eta_{1}\nonumber \\
\left({\scriptstyle \text{eq. }\eqref{eq:KappaIotaDefinition}}\right) & \leq3^{2+K+M_{3}}\cdot\int_{-1}^{1}\theta_{1}\left(2^{n\alpha-\nu}\left(m\!+\!\xi\right)\right)\cdot\left(1\!+\!\left|2^{n-\nu\alpha}\left(1\!-\!\kappa\iota\right)\!-\!2^{\alpha\left(n-\nu\right)}\kappa\xi\right|\right)^{-K}\d\xi.\label{eq:UpperConeRightConeIntegralRewritten}
\end{align}
As our next step, we derive several basic estimates for the quantities
appearing in equation \eqref{eq:UpperConeRightConeIntegralRewritten}:
\begin{enumerate}
\item We have
\begin{equation}
\theta_{1}\left(2^{n\alpha-\nu}\cdot\left(m+\xi\right)\right)\leq4^{M_{3}}\cdot\min\left\{ 1,\,2^{M_{1}\left(n-\nu\right)}\right\} =4^{M_{3}}\cdot2^{-M_{1}\cdot\left(\nu-n\right)_{+}}\qquad\forall\xi\in\left[-1,1\right].\label{eq:Theta1StandardEstimate}
\end{equation}
To see this, we consider the cases $\left|m\right|\leq1$ and $\left|m\right|\geq2$.
In case of $\left|m\right|\leq1$, we have $\left|m+\xi\right|\leq\left|m\right|+\left|\xi\right|\leq2$
and thus
\begin{align*}
\theta_{1}\left(2^{n\alpha-\nu}\cdot\left(m+\xi\right)\right) & \leq\min\left\{ 1,\,\left|2^{n\alpha-\nu}\cdot\left(m+\xi\right)\right|^{M_{1}}\right\} \\
 & \leq2^{M_{1}}\cdot\min\left\{ 1,\,2^{M_{1}\left(n\alpha-\nu\right)}\right\} \\
\left({\scriptstyle \text{since }n\alpha\leq n\text{ and }M_{1}\geq0\text{, as well as }M_{1}\leq M_{3}}\right) & \leq2^{M_{3}}\cdot\min\left\{ 1,\,2^{M_{1}\left(n-\nu\right)}\right\} ,
\end{align*}
which is even slightly better than the estimate \eqref{eq:Theta1StandardEstimate}.
Next, in case of $\left|m\right|\geq2$, we have
\begin{equation}
\frac{\left|m\right|}{2}\leq\left|m\right|-1\leq\left|m\right|-\left|\xi\right|\leq\left|m+\xi\right|\leq\left|m\right|+\left|\xi\right|\leq1+\left|m\right|\leq2\left|m\right|\qquad\forall\xi\in\left[-1,1\right],\label{eq:AddingXiDoesNotMatterForLargeM}
\end{equation}
so that equation \eqref{eq:Theta1Rescaling} yields
\begin{align*}
\theta_{1}\left(2^{n\alpha-\nu}\cdot\left(m+\xi\right)\right) & \leq2^{M_{3}}\cdot\theta_{1}\left(2^{n\alpha-\nu}\cdot m\right)\\
\left({\scriptstyle \text{cf. eq. }\eqref{eq:KappaIotaDefinition}}\right) & =2^{M_{3}}\cdot\theta_{1}\left(2^{n-\nu}\cdot\iota\right)\\
 & \leq2^{M_{3}}\cdot\min\left\{ 1,\,\left|2^{n-\nu}\cdot\iota\right|^{M_{1}}\right\} \\
\left({\scriptstyle \text{since }\left|\iota\right|\leq2}\right) & \leq2^{M_{3}}\cdot\min\left\{ 1,2^{M_{1}}\cdot2^{M_{1}\left(n-\nu\right)}\right\} \\
\left({\scriptstyle \text{since }M_{1}\leq M_{3}}\right) & \leq4^{M_{3}}\cdot\min\left\{ 1,\,2^{M_{1}\left(n-\nu\right)}\right\} \qquad\forall\xi\in\left[-1,1\right].
\end{align*}
We have thus established equation \eqref{eq:Theta1StandardEstimate}
in both cases.
\item Next, in case of $\left|\kappa\right|\leq\frac{1}{4}$, we have
\begin{align}
\left|2^{n-\nu\alpha}\lambda_{n,m,\nu,\mu}-2^{n\alpha-\nu}\mu\xi\right| & =\left|2^{n-\nu\alpha}\left(1-\kappa\iota\right)-2^{\alpha\left(n-\nu\right)}\kappa\xi\right|\nonumber \\
 & =2^{n-\nu\alpha}\cdot\left|1-\kappa\iota-2^{-n\left(1-\alpha\right)}\kappa\xi\right|\nonumber \\
 & \geq2^{n-\nu\alpha}\cdot\left(1-\left|\kappa\iota\right|-2^{-n\left(1-\alpha\right)}\cdot\left|\kappa\xi\right|\right)\nonumber \\
\left({\scriptstyle \text{since }2^{-n\left(1-\alpha\right)}\leq1\text{ and }\left|\kappa\right|\leq\frac{1}{4},\text{ as well as }\left|\iota\right|\leq2\text{ and }\left|\xi\right|\leq1}\right) & \geq2^{n-\nu\alpha}\cdot\left(1-\frac{1}{2}-\frac{1}{4}\right)\nonumber \\
 & =\frac{2^{n-\nu\alpha}}{4}\qquad\forall\xi\in\left[-1,1\right].\label{eq:Theta2EstimateForSmallKappa}
\end{align}
\item Finally, we want to obtain an estimate similar to equation \eqref{eq:Theta2EstimateForSmallKappa}
also if $\left|\iota\right|\leq\frac{1}{4}$. To this end, we additionally
assume $\alpha<1$ and $n\geq\frac{3}{1-\alpha}$, since this ensures
$-n\left(1-\alpha\right)\leq-3$ and thus $2^{-n\left(1-\alpha\right)}\leq\frac{1}{8}$.
Consequently,
\begin{align}
\left|2^{n-\nu\alpha}\lambda_{n,m,\nu,\mu}-2^{n\alpha-\nu}\mu\xi\right| & =\left|2^{n-\nu\alpha}\left(1-\kappa\iota\right)-2^{\alpha\left(n-\nu\right)}\kappa\xi\right|\nonumber \\
 & \geq2^{n-\nu\alpha}\cdot\left(1-\left|\kappa\iota\right|-2^{-n\left(1-\alpha\right)}\cdot\left|\kappa\xi\right|\right)\nonumber \\
\left({\scriptstyle \text{since }\left|\iota\right|\leq\frac{1}{4}\text{ and }2^{-n\left(1-\alpha\right)}\leq\frac{1}{8},\text{ as well as }\left|\xi\right|\leq1\text{ and }\left|\kappa\right|\leq2}\right) & \geq2^{n-\nu\alpha}\cdot\left(1-\frac{1}{2}-\frac{1}{8}\cdot2\right)\nonumber \\
 & =\frac{2^{n-\nu\alpha}}{4}\qquad\forall\xi\in\left[-1,1\right].\label{eq:Theta2EstimateForSmallIota}
\end{align}
\end{enumerate}
For the last estimate above, we needed to assume $\alpha<1$. To avoid
cumbersome case distinctions later on, we now consider the special
case $\alpha=1$, so that we can then assume $\alpha<1$ for the remainder
of the subsection.

\subsubsection{The special case $\alpha=1$}

\label{subsec:UpperConeRIghtConeAlphaIs1}Because of $\alpha=1$,
we simply have $\kappa=\mu$ and $\iota=m$. Further, $G_{n}=\left\lceil 2^{n\left(1-\alpha\right)}\right\rceil =1$
for all $n\in\N_{0}$, i.e., $m,\mu\in\left\{ -1,0,1\right\} $. Consequently,
we also get $\lambda_{n,m,\nu,\mu}=1-m\mu\in\left\{ 0,1,2\right\} $
and estimate \eqref{eq:UpperConeRightConeIntegralRewritten} takes
the form
\begin{equation}
\left|\det T_{i}\right|^{-1}\cdot\int_{S_{i}^{\left(\alpha\right)}}\varrho_{j}\left(T_{j}^{-1}\xi\right)\d\xi\leq3^{2+K+M_{3}}\cdot\int_{-1}^{1}\theta_{1}\left(2^{n-\nu}\left(m+\xi\right)\right)\cdot\left(1+\left|2^{n-\nu}\left[1-\mu\left(m+\xi\right)\right]\right|\right)^{-K}\d\xi.\label{eq:UpperConeRightConeIntegralRewrittenForAlpha1}
\end{equation}
Finally, we get because of $\alpha=1$ and $\lambda_{n,m,\nu,\mu}\in\left\{ 0,1,2\right\} $
from equation \eqref{eq:UpperConeRightConeMatrixChangeNorm} that
\begin{align}
\left(1+\left\Vert T_{j}^{-1}T_{i}\right\Vert \right)^{\sigma} & \leq\left(1+5\cdot2^{\left(n-\nu\right)_{+}}+2^{n-\nu\alpha}\cdot\left|\lambda_{n,m,\nu,\mu}\right|\right)^{\sigma}\nonumber \\
 & \leq\left(1+5\cdot2^{\left(n-\nu\right)_{+}}+2\cdot2^{n-\nu}\right)^{\sigma}\nonumber \\
 & \leq8^{\sigma}\cdot2^{\sigma\cdot\left(n-\nu\right)_{+}}.\label{eq:UpperConeRIghtConeMatrixChangeNormAlpha1}
\end{align}

Next, we distinguish two subcases:
\begin{enumerate}
\item If $n-\nu\leq0$, then $\left|2^{n-\nu}\left(m+\xi\right)\right|\leq2\cdot2^{n-\nu}$
since $|m|\leq1$ and $\left|\xi\right|\leq1$. Hence 
\[
\theta_{1}\left(2^{n-\nu}\left(m+\xi\right)\right)\leq\left|2^{n-\nu}\left(m+\xi\right)\right|^{M_{1}}\leq2^{M_{1}}\cdot2^{M_{1}\left(n-\nu\right)}=2^{M_{1}}\cdot2^{-M_{1}\left|n-\nu\right|}\leq2^{M_{3}}\cdot2^{-M_{1}\left|n-\nu\right|}.
\]
\item Otherwise, $n-\nu\geq0$, so that there are again two subcases:

\begin{enumerate}
\item If $\left|m+\xi\right|\geq\frac{1}{2}$, then $\frac{1}{2}\leq\left|m+\xi\right|\leq2$,
so that equation \eqref{eq:Theta1Rescaling} yields
\[
\theta_{1}\left(2^{n-\nu}\left(m+\xi\right)\right)\leq2^{M_{3}}\cdot\theta_{1}\left(2^{n-\nu}\right)\leq2^{M_{3}}\cdot\left(1+\left|2^{n-\nu}\right|\right)^{-M_{2}}\leq2^{M_{3}}\cdot2^{-M_{2}\left|n-\nu\right|}.
\]
\item Otherwise, $\left|m+\xi\right|\leq\frac{1}{2}$ and hence $\left|1-\mu\left(m+\xi\right)\right|\geq1-\left|\mu\left(m+\xi\right)\right|\geq1-\left|m+\xi\right|\geq\frac{1}{2}$,
which implies
\[
\left(1+\left|2^{n-\nu}\left[1-\mu\left(m+\xi\right)\right]\right|\right)^{-K}\leq\left(\frac{1}{2}\cdot2^{n-\nu}\right)^{-K}\leq2^{K}\cdot2^{-K\left|n-\nu\right|}.
\]
\end{enumerate}
\end{enumerate}
All in all, we have for all $\xi\in\left[-1,1\right]$ that 
\[
\begin{aligned}\theta_{1}\left(2^{n-\nu}\left(m+\xi\right)\right)\cdot\left(1+\left|2^{n-\nu}\left[1-\mu\left(m+\xi\right)\right]\right|\right)^{-K} & \leq\begin{cases}
2^{M_{3}+K}\cdot2^{-M_{1}\left|n-\nu\right|}, & \text{if }n\leq\nu\\
2^{M_{3}+K}\cdot2^{-\min\left\{ M_{2},K\right\} \left|n-\nu\right|}, & \text{if }n\geq\nu
\end{cases}\\
 & =2^{M_{3}+K}\cdot2^{-M_{1}(\nu-n)_{+}}\cdot2^{-\min\left\{ M_{2},K\right\} \left(n-\nu\right)_{+}}
\end{aligned}
\]
and thus
\begin{align*}
M_{j,i}^{\left(0\right)} & =\left(\frac{w_{j}^{s}}{w_{i}^{s}}\right)^{\tau}\cdot\left(1+\left\Vert T_{j}^{-1}T_{i}\right\Vert \right)^{\sigma}\cdot\left(\left|\det T_{i}\right|^{-1}\cdot\int_{S_{i}^{\left(\alpha\right)}}\varrho_{j}\left(T_{j}^{-1}\xi\right)\d\xi\right)^{\tau}\\
\left({\scriptstyle \text{eqs. }\eqref{eq:UpperConeRightConeIntegralRewrittenForAlpha1}\text{ and }\eqref{eq:UpperConeRIghtConeMatrixChangeNormAlpha1}}\right) & \leq2^{s\tau\left(\nu-n\right)}\cdot8^{\sigma}\cdot2^{\sigma\cdot\left(n-\nu\right)_{+}}\cdot\left[3^{2+K+M_{3}}\cdot2^{1+M_{3}+K}\cdot2^{-M_{1}(\nu-n)_{+}}\cdot2^{-\min\left\{ M_{2},K\right\} \left(n-\nu\right)_{+}}\right]^{\tau}\\
 & \leq8^{\sigma}\cdot6^{\tau\left(2+K+M_{3}\right)}\cdot\begin{cases}
2^{-\left|n-\nu\right|\left[\tau\min\left\{ M_{2},K\right\} +s\tau-\sigma\right]}, & \text{if }n-\nu\geq0,\\
2^{-\left|n-\nu\right|\left[\tau M_{1}-s\tau\right]}, & \text{if }n-\nu\leq0.
\end{cases}
\end{align*}
But the assumptions of Lemma \ref{lem:MainShearletLemma} ensure that
$M_{1}\geq M_{1}^{(0)}+c\geq s+c$ and $M_{2},K\geq\frac{\sigma}{\tau}-s+c$,
which entails $\tau\min\left\{ M_{2},K\right\} +s\tau-\sigma\geq\tau c$,
as well as $\tau M_{1}-s\tau\geq\tau c$, so that $M_{j,i}^{\left(0\right)}\leq8^{\sigma}\cdot6^{\tau\left(2+K+M_{3}\right)}\cdot2^{-\tau c\left|n-\nu\right|}$
for all $i\in I^{\left(\ell_{1}\right)}$ and $j\in I^{\left(\ell_{2}\right)}$.
Consequently, we get because of $G_{n}=G_{\nu}=1$ for all $n,\nu\in\N_{0}$
that
\begin{equation}
\begin{split}\sum_{i\in I^{\left(\ell_{1}\right)}}M_{j,i}^{\left(0\right)} & \leq8^{\sigma}\cdot6^{\tau\left(2+K+M_{3}\right)}\cdot\sum_{n=0}^{\infty}\left[2^{-\tau c\left|n-\nu\right|}\cdot3\cdot G_{n}\right]\\
 & \leq3\cdot8^{\sigma}\cdot6^{\tau\left(2+K+M_{3}\right)}\cdot\sum_{\ell\in\Z}2^{-\tau c\left|\ell\right|}\\
 & \leq3\cdot8^{\sigma}\cdot6^{\tau\left(2+K+M_{3}\right)}\cdot\frac{2}{1-2^{-\tau c}}=:C_{1}\quad\text{if }\alpha=1
\end{split}
\label{eq:UpperConeRightConeAlpha1FinalEstimate}
\end{equation}
for arbitrary $j=\left(\nu,\mu,e,d\right)\in I^{\left(\ell_{2}\right)}$.
Exactly the same estimate also yields $\sum_{j\in I^{\left(\ell_{2}\right)}}M_{j,i}^{\left(0\right)}\leq C_{1}$
for arbitrary $i=\left(n,m,\varepsilon,\delta\right)\in I^{\left(\ell_{1}\right)}$,
as long as $\alpha=1$.

\subsubsection{The general case $\alpha\in\left[0,1\right)$}

\label{subsec:UpperConeRightConeGeneralAlpha}In this subsection,
we first consider two special cases and then the remaining general
case.

\textbf{Case 1}: $n\leq\frac{3}{1-\alpha}$. In this case, equation
\eqref{eq:Theta1StandardEstimate} yields
\[
\theta_{1}\left(2^{n\alpha-\nu}\cdot\left(m+\xi\right)\right)\leq4^{M_{3}}\cdot\min\left\{ 1,\,2^{M_{1}\left(n-\nu\right)}\right\} \leq4^{M_{3}}\cdot2^{\frac{3M_{1}}{1-\alpha}}\cdot2^{-M_{1}\nu}\qquad\forall\xi\in\left[-1,1\right].
\]
Furthermore, equation \eqref{eq:UpperConeRightConeMatrixChangeNorm}
entails, because of $\left|\lambda_{n,m,\nu,\mu}\right|=\left|1-\kappa\iota\right|\leq5$,
that 
\begin{align*}
\left(1+\left\Vert T_{j}^{-1}T_{i}\right\Vert \right)^{\sigma} & \leq6^{\sigma}\cdot2^{\sigma\cdot\left(n-\nu\right)_{+}}\cdot\left(1+5\cdot2^{n-\nu\alpha}\right)^{\sigma}\\
 & \leq6^{\sigma}\cdot2^{\sigma\cdot\left(n-\nu\right)_{+}}\cdot\left(1+5\cdot2^{\frac{3}{1-\alpha}}\right)^{\sigma}\\
\left({\scriptstyle \text{since }\left(n-\nu\right)_{+}=n-\nu\leq n\leq\frac{3}{1-\alpha}\text{ for }n-\nu\geq0\text{ and }\left(n-\nu\right)_{+}=0\leq\frac{3}{1-\alpha}\text{ otherwise}}\right) & \leq6^{\sigma}\cdot2^{\frac{3\sigma}{1-\alpha}}\cdot\left(1+5\cdot2^{\frac{3}{1-\alpha}}\right)^{\sigma}\\
 & \leq6^{2\sigma}\cdot2^{\frac{6\sigma}{1-\alpha}}=:C_{2}.
\end{align*}
In combination with equation \eqref{eq:UpperConeRightConeIntegralRewritten},
we conclude
\begin{align*}
M_{j,i}^{\left(0\right)} & =\left(\frac{w_{j}^{s}}{w_{i}^{s}}\right)^{\tau}\cdot\left(1+\left\Vert T_{j}^{-1}T_{i}\right\Vert \right)^{\sigma}\cdot\left(\left|\det T_{i}\right|^{-1}\cdot\int_{S_{i}^{\left(\alpha\right)}}\varrho_{j}\left(T_{j}^{-1}\xi\right)\d\xi\right)^{\tau}\\
 & \leq C_{2}\cdot2^{s\tau\left(\nu-n\right)}\cdot\left[3^{2+K+M_{3}}\cdot\int_{-1}^{1}\theta_{1}\left(2^{n\alpha-\nu}\left(m+\xi\right)\right)\cdot\left(1+\left|2^{n-\nu\alpha}\lambda_{n,m,\nu,\mu}-2^{n\alpha-\nu}\mu\xi\right|\right)^{-K}\d\xi\right]^{\tau}\\
 & \leq C_{2}\cdot2^{s\tau\nu}\cdot\left[2\cdot3^{2+K+M_{3}}\cdot4^{M_{3}}\cdot2^{\frac{3M_{1}}{1-\alpha}}\cdot2^{-M_{1}\nu}\right]^{\tau}\\
 & =C_{2}\cdot2^{\tau\nu\left(s-M_{1}\right)}\cdot\left[2\cdot3^{2+K+M_{3}}\cdot4^{M_{3}}\cdot2^{\frac{3M_{1}}{1-\alpha}}\right]^{\tau}\\
 & =:2^{\tau\nu\left(s-M_{1}\right)}\cdot C_{3}.
\end{align*}
Since our assumptions imply $M_{1}\geq M_{1}^{(0)}+c\geq s+\frac{1}{\tau}+c\geq s+\frac{1-\alpha}{\tau}+c$,
we get $1-\alpha+\tau s-\tau M_{1}\leq-\tau c$ and hence
\begin{equation}
\begin{split}\sup_{\substack{i=\left(n,m,\varepsilon,\delta\right)\in I^{\left(\ell_{1}\right)}\\
\text{with }n\leq3/\left(1-\alpha\right)
}
}\:\sum_{j\in I^{\left(\ell_{2}\right)}}M_{j,i}^{\left(0\right)} & \leq C_{3}\cdot\sum_{\nu=0}^{\infty}\:\sum_{\left|\mu\right|\leq G_{\nu}}2^{\tau\nu\left(s-M_{1}\right)}\\
\left({\scriptstyle \text{since }G_{\nu}=\left\lceil 2^{\nu\left(1-\alpha\right)}\right\rceil \leq1+2^{\nu\left(1-\alpha\right)}\leq2\cdot2^{\nu\left(1-\alpha\right)}}\right) & \leq6C_{3}\cdot\sum_{\nu=0}^{\infty}2^{\nu\left(1-\alpha+\tau s-\tau M_{1}\right)}\\
 & \leq6C_{3}\cdot\sum_{\nu=0}^{\infty}2^{-\tau c\nu}=\frac{6C_{3}}{1-2^{-\tau c}}.
\end{split}
\label{eq:UpperConeRightConeSmallNSumOverJ}
\end{equation}
Furthermore, since $\tau\left(s-M_{1}\right)\leq1-\alpha+\tau s-\tau M_{1}\leq-\tau c<0$,
we also have
\begin{equation}
\begin{split}\sup_{j\in I^{\left(\ell_{2}\right)}}\:\sum_{\substack{i=\left(n,m,\varepsilon,\delta\right)\in I^{\left(\ell_{1}\right)}\\
\text{with }n\leq3/\left(1-\alpha\right)
}
}M_{j,i}^{\left(0\right)} & \leq C_{3}\cdot\sup_{\nu\in\N_{0}}2^{\tau\nu\left(s-M_{1}\right)}\cdot\sum_{n\leq\frac{3}{1-\alpha}}\:\sum_{\left|m\right|\leq G_{n}}1\\
\left({\scriptstyle \text{since }G_{n}=\left\lceil 2^{\left(1-\alpha\right)n}\right\rceil \leq\left\lceil 2^{3}\right\rceil =8}\right) & \leq C_{3}\cdot\left(1+\frac{3}{1-\alpha}\right)\cdot\left(1+2\cdot8\right)\leq\frac{68\cdot C_{3}}{1-\alpha}.
\end{split}
\label{eq:UpperConeRightConeSmallNSumOverI}
\end{equation}
This completes our considerations for the special case $n\leq\frac{3}{1-\alpha}$.
In the remainder of this subsection, we can (and will) thus assume
$n\geq\frac{3}{1-\alpha}$.

\medskip{}

\textbf{Case 2}: We have $\left[\left|\kappa\right|\geq\frac{1}{4}\right]\wedge\left[\left|m\right|\geq2\right]\wedge\left[\left(n\leq\nu\right)\vee\left(\left|\iota\right|\geq\frac{1}{4}\right)\right]$,
as well as $n\geq\frac{3}{1-\alpha}$. We first show that these conditions
imply
\begin{equation}
\begin{split}\theta_{1}\left(2^{n\alpha-\nu}\cdot\left(m+\xi\right)\right) & \leq2^{5M_{3}}\cdot\left|\iota\right|^{M_{1}}\cdot2^{-M_{1}\left(\nu-n\right)_{+}}\cdot2^{-M_{2}\left(n-\nu\right)_{+}}\\
 & =\begin{cases}
2^{5M_{3}}\cdot\left|\iota\right|^{M_{1}}\cdot2^{-M_{1}\left|n-\nu\right|}, & \text{if }\nu\geq n\\
2^{5M_{3}}\cdot\left|\iota\right|^{M_{1}}\cdot2^{-M_{2}\left|n-\nu\right|}, & \text{if }n>\nu
\end{cases}\qquad\forall\xi\in\left[-1,1\right].
\end{split}
\label{eq:UpperConeRightConeSummationCaseTheta1Estimate}
\end{equation}
To establish equation \eqref{eq:UpperConeRightConeSummationCaseTheta1Estimate},
we first note $\frac{\left|m\right|}{2}\leq\left|m+\xi\right|\leq2\left|m\right|$,
(cf.\@ equation \eqref{eq:AddingXiDoesNotMatterForLargeM}) since
$\left|m\right|\geq2$. Hence, equations \eqref{eq:Theta1Rescaling}
and \eqref{eq:KappaIotaDefinition} yield 
\[
\theta_{1}\left(2^{n\alpha-\nu}\left(m+\xi\right)\right)\leq2^{M_{3}}\cdot\theta_{1}\left(2^{n\alpha-\nu}\left|m\right|\right)=2^{M_{3}}\cdot\theta_{1}\left(2^{n-\nu}\left|\iota\right|\right).
\]
Now, we distinguish the two cases that are suggested by equation \eqref{eq:UpperConeRightConeSummationCaseTheta1Estimate}:
\begin{enumerate}
\item In case of $n\leq\nu$, we get $n-\nu=-\left|n-\nu\right|$ and thus
\begin{align*}
\theta_{1}\left(2^{n\alpha-\nu}\left(m+\xi\right)\right) & \leq2^{M_{3}}\cdot\theta_{1}\left(2^{n-\nu}\left|\iota\right|\right)\leq2^{M_{3}}\cdot\left(2^{n-\nu}\cdot\left|\iota\right|\right)^{M_{1}}\\
 & =2^{M_{3}}\cdot\left|\iota\right|^{M_{1}}\cdot2^{-M_{1}\left|n-\nu\right|}=2^{M_{3}}\cdot\left|\iota\right|^{M_{1}}\cdot2^{-M_{1}(\nu-n)_{+}},
\end{align*}
which is even slightly better than equation \eqref{eq:UpperConeRightConeSummationCaseTheta1Estimate}.
\item In case of $n>\nu$, we have $\left|\iota\right|\geq\frac{1}{4}$,
since we assume $\left(n\leq\nu\right)\vee\left(\left|\iota\right|\geq\frac{1}{4}\right)$.
Consequently, $\frac{1}{4}\leq\left|\iota\right|\leq2\leq4$, so that
equation \eqref{eq:Theta1Rescaling} yields
\begin{align*}
\theta_{1}\left(2^{n\alpha-\nu}\left(m+\xi\right)\right) & \leq2^{M_{3}}\cdot\theta_{1}\left(2^{n-\nu}\left|\iota\right|\right)\leq2^{M_{3}}4^{M_{3}}\cdot\theta_{1}\left(2^{n-\nu}\right)\\
 & \leq8^{M_{3}}\cdot\left(1+2^{n-\nu}\right)^{-M_{2}}\\
 & \leq8^{M_{3}}\cdot2^{-M_{2}\left|n-\nu\right|}\\
\left({\scriptstyle \text{since }\left|\iota\right|\geq\frac{1}{4}}\right) & \leq8^{M_{3}}\cdot4^{M_{1}}\cdot\left|\iota\right|^{M_{1}}\cdot2^{-M_{2}\left|n-\nu\right|}\\
 & \leq2^{5M_{3}}\cdot\left|\iota\right|^{M_{1}}\cdot2^{-M_{2}\left|n-\nu\right|}=2^{5M_{3}}\cdot\left|\iota\right|^{M_{1}}\cdot2^{-M_{2}(n-\nu)_{+}}\,,
\end{align*}
which establishes equation \eqref{eq:UpperConeRightConeSummationCaseTheta1Estimate}
also in this case.
\end{enumerate}
We now properly start the proof: First, note that $\left|\iota\right|^{M_{1}}\leq2^{M_{1}}\leq2^{M_{3}}$,
so that equation \eqref{eq:UpperConeRightConeSummationCaseTheta1Estimate}
yields the estimate $\theta_{1}\left(2^{n\alpha-\nu}(m+\xi)\right)\leq2^{6M_{3}}\cdot2^{-M_{1}(\nu-n)_{+}}\cdot2^{-M_{2}(n-\nu)_{+}}$
for all $\xi\in\left[-1,1\right]$. In combination with equation \eqref{eq:UpperConeRightConeIntegralRewritten},
we conclude
\begin{align*}
 & \left|\det T_{i}\right|^{-1}\int_{S_{i}^{\left(\alpha\right)}}\varrho_{j}\left(T_{j}^{-1}\xi\right)\d\xi\\
 & \leq3^{2+K+M_{3}}\cdot\int_{-1}^{1}\theta_{1}\left(2^{n\alpha-\nu}\left(m+\xi\right)\right)\cdot\left(1+\left|2^{n-\nu\alpha}\left(1-\kappa\iota\right)-2^{\alpha\left(n-\nu\right)}\kappa\xi\right|\right)^{-K}\d\xi\\
\left({\scriptstyle \text{with }\eta=2^{\alpha\left(n-\nu\right)}\kappa\xi}\right) & \leq3^{2+K+5M_{3}}\,2^{-M_{1}(\nu-n)_{+}-M_{2}\left(n-\nu\right)_{+}}\cdot\frac{2^{\alpha\left(\nu-n\right)}}{\left|\kappa\right|}\,\int_{-2^{\alpha\left(n-\nu\right)}\left|\kappa\right|}^{2^{\alpha\left(n-\nu\right)}\left|\kappa\right|}\!\left(1\!+\!\left|2^{n-\nu\alpha}\!-\!2^{n-\nu\alpha}\kappa\iota\!-\!\eta\right|\right)^{\!-K}\!\d\eta\\
\left({\scriptstyle \text{since }\iota=m/2^{n\left(1-\alpha\right)}\text{ and }\left|\kappa\right|\geq\frac{1}{4}}\right) & \leq3^{4+K+5M_{3}}\,2^{\alpha\left(\nu-n\right)-M_{1}(\nu-n)_{+}-M_{2}\left(n-\nu\right)_{+}}\:\int_{-2^{\alpha\left(n-\nu\right)}\left|\kappa\right|}^{2^{\alpha\left(n-\nu\right)}\left|\kappa\right|}\!\left(1\!+\!\left|2^{n-\nu\alpha}\!-\!2^{\alpha\left(n-\nu\right)}\kappa m\!-\!\eta\right|\right)^{\!-K}\!\d\eta\\
\left({\scriptstyle \text{with }\xi=2^{n-\nu\alpha}-2^{\alpha\left(n-\nu\right)}\kappa m-\eta}\right) & =3^{4+K+5M_{3}}\,2^{\alpha\left(\nu-n\right)-M_{1}(\nu-n)_{+}-M_{2}\left(n-\nu\right)_{+}}\:\int_{2^{n-\nu\alpha}-2^{\alpha\left(n-\nu\right)}\kappa m-2^{\alpha\left(n-\nu\right)}\left|\kappa\right|}^{2^{n-\nu\alpha}-2^{\alpha\left(n-\nu\right)}\kappa m+2^{\alpha\left(n-\nu\right)}\left|\kappa\right|}\left(1\!+\!\left|\xi\right|\right)^{-K}\d\xi.
\end{align*}
For brevity, let us set $L:=2^{\alpha\left(n-\nu\right)}\left|\kappa\right|$
(which is independent of $m$) and $C_{4}:=6^{\sigma}\cdot3^{\tau\left(4+K+5M_{3}\right)}$,
as well as
\begin{equation}
\begin{split}\Lambda_{n,m,\nu,\mu} & :=\left(1+2^{n-\alpha\nu}\left|\lambda_{n,m,\nu,\mu}\right|\right)^{\sigma}=\left(1+2^{n-\alpha\nu}\left|1-\kappa\iota\right|\right)^{\sigma}\\
\left({\scriptstyle \text{eq. }\eqref{eq:KappaIotaDefinition}}\right) & =\left(1+2^{n-\alpha\nu}\left|1-2^{n\left(\alpha-1\right)}\kappa m\right|\right)^{\sigma}\\
 & =\left(1+\left|2^{n-\alpha\nu}-2^{\alpha\left(n-\nu\right)}\kappa m\right|\right)^{\sigma}.
\end{split}
\label{eq:CapitalLambdaDefinition}
\end{equation}
 In combination with equation \eqref{eq:UpperConeRightConeMatrixChangeNorm},
the preceding estimate yields
\begin{align*}
\sum_{\substack{\left|m\right|\leq G_{n}\\
\text{s.t. Case 2 holds}
}
}M_{j,i}^{\left(0\right)} & =\sum_{\substack{\left|m\right|\leq G_{n}\\
\text{s.t. Case 2 holds}
}
}\left(\frac{w_{j}^{s}}{w_{i}^{s}}\right)^{\tau}\cdot\left(1+\left\Vert T_{j}^{-1}T_{i}\right\Vert \right)^{\sigma}\cdot\left(\left|\det T_{i}\right|^{-1}\cdot\int_{S_{i}^{\left(\alpha\right)}}\varrho_{j}\left(T_{j}^{-1}\xi\right)\d\xi\right)^{\tau}\\
\left({\scriptstyle \text{eq. }\eqref{eq:UpperConeRightConeMatrixChangeNorm}}\right) & \leq C_{4}\cdot2^{\tau s\left(\nu-n\right)+\sigma\left(n-\nu\right)_{+}}\cdot\left[2^{-M_{1}(\nu-n)_{+}}\cdot2^{-M_{2}(n-\nu)_{+}}\cdot2^{\alpha\left(\nu-n\right)}\right]^{\tau}\\
 & \phantom{=C_{2}\cdot}\cdot\sum_{m\in\Z}\Lambda_{n,m,\nu,\mu}\left(\int_{2^{n-\nu\alpha}-2^{\alpha\left(n-\nu\right)}\kappa m-L}^{2^{n-\nu\alpha}-2^{\alpha\left(n-\nu\right)}\kappa m+L}\left(1+\left|\xi\right|\right)^{-K}\d\xi\right)^{\tau}\\
\left({\scriptstyle \ell=-{\rm sign}\left(\kappa\right)\cdot m\text{ and eq. }\eqref{eq:CapitalLambdaDefinition}}\right) & =C_{4}\cdot2^{\tau\left[\left(s+\alpha\right)\left(\nu-n\right)+\frac{\sigma}{\tau}\left(n-\nu\right)_{+}-M_{1}(\nu-n)_{+}-M_{2}(n-\nu)_{+}\right]}\\
 & \phantom{=C_{2}\cdot}\cdot\sum_{\ell\in\Z}\left(1+\left|2^{n-\nu\alpha}+2^{\alpha\left(n-\nu\right)}\left|\kappa\right|\ell\right|\right)^{\sigma}\left(\int_{2^{\alpha\left(n-\nu\right)}\left|\kappa\right|\ell+2^{n-\nu\alpha}-L}^{2^{\alpha\left(n-\nu\right)}\left|\kappa\right|\ell+2^{n-\nu\alpha}+L}\left(1+\left|\xi\right|\right)^{-K}\d\xi\right)^{\tau}.
\end{align*}
Now, an application of Lemma \ref{lem:WeightedSumOfShiftedIntegrals}
and of the associated remark (with $p=\tau$, $N=\sigma$, $\beta=2^{\alpha\left(n-\nu\right)}\left|\kappa\right|>0$
and $L=2^{\alpha\left(n-\nu\right)}\left|\kappa\right|$, as well
as $M=2^{n-\nu\alpha}$) yields
\begin{align*}
 & \sum_{\ell\in\Z}\left(1+\left|2^{n-\nu\alpha}+2^{\alpha\left(n-\nu\right)}\left|\kappa\right|\ell\right|\right)^{\sigma}\left(\int_{2^{\alpha\left(n-\nu\right)}\left|\kappa\right|\ell+2^{n-\nu\alpha}-L}^{2^{\alpha\left(n-\nu\right)}\left|\kappa\right|\ell+2^{n-\nu\alpha}+L}\left(1+\left|\xi\right|\right)^{-K}\d\xi\right)^{\tau}\\
 & \leq2^{3+\tau+\sigma}\cdot10^{3+\sigma}\cdot\left\Vert \left(1\!+\!\left|\mybullet\right|\right)^{-K}\right\Vert _{\frac{2+\sigma}{\tau}}^{\tau}\!\cdot\!\left(2^{\alpha\left(n-\nu\right)}\left|\kappa\right|\right)^{\tau}\!\cdot\!\left(1\!+\!\left[2^{\alpha\left(n-\nu\right)}\left|\kappa\right|\right]^{\sigma}\right)\!\cdot\!\left(1\!+\!1\!+\!\frac{2^{\alpha\left(\nu-n\right)}}{\left|\kappa\right|}\right)\\
\left({\scriptstyle \text{since }K\geq\frac{2+\sigma}{\tau}\text{ and }\frac{1}{4}\leq\left|\kappa\right|\leq2}\right) & \leq2^{3+2\sigma+2\tau}\cdot10^{3+\sigma}\cdot2^{\alpha\tau\left(n-\nu\right)}\cdot\left(1+2^{\alpha\sigma\left(n-\nu\right)}\right)\cdot\left(2+4\cdot2^{\alpha\left(\nu-n\right)}\right)\\
 & \leq2^{7+2\sigma+2\tau}\cdot10^{3+\sigma}\cdot2^{\alpha\tau\left(n-\nu\right)}\cdot2^{\alpha\sigma\cdot\left(n-\nu\right)_{+}}\cdot2^{\alpha\cdot\left(\nu-n\right)_{+}}.
\end{align*}
All in all, we get for $C_{5}:=C_{4}\cdot2^{7+2\sigma+2\tau}\cdot10^{3+\sigma}$
that
\begin{align*}
\smash{\sum_{\substack{\left|m\right|\leq G_{n}\\
\text{s.t. Case 2 holds}
}
}}M_{j,i}^{\left(0\right)} & \leq C_{5}\cdot2^{\alpha\tau\left(n-\nu\right)}\cdot2^{\alpha\sigma\cdot\left(n-\nu\right)_{+}}\cdot2^{\alpha\cdot\left(\nu-n\right)_{+}}\cdot2^{\tau\left[\left(s+\alpha\right)\left(\nu-n\right)+\frac{\sigma}{\tau}\left(n-\nu\right)_{+}-M_{1}(\nu-n)_{+}-M_{2}(n-\nu)_{+}\right]}\\
 & =C_{5}\cdot2^{\tau s\left(\nu-n\right)+\alpha\left(\nu-n\right)_{+}+\left(1+\alpha\right)\sigma\left(n-\nu\right)_{+}-\tau M_{1}(\nu-n)_{+}-\tau M_{2}(n-\nu)_{+}}\\
 & =C_{5}\cdot\begin{cases}
2^{-\tau\left|\nu-n\right|\left[M_{2}-\left(1+\alpha\right)\frac{\sigma}{\tau}+s\right]}, & \text{if }n\geq\nu,\\
2^{-\tau\left|\nu-n\right|\left[-s-\frac{\alpha}{\tau}+M_{1}\right]}, & \text{if }n\leq\nu
\end{cases}\\
\left({\scriptstyle \text{since }M_{1}\geq M_{1}^{(0)}+c\text{ and }M_{2}\geq M_{2}^{(0)}+c}\right) & \leq C_{5}\cdot2^{-\tau c\left|\nu-n\right|}.
\end{align*}
As usual, this implies
\begin{equation}
\sup_{j\in I^{\left(\ell_{2}\right)}}\:\sum_{\substack{i=\left(n,m,\varepsilon,\delta\right)\in I^{\left(\ell_{1}\right)}\\
\text{s.t. Case 2 holds}
}
}M_{j,i}^{\left(0\right)}\leq C_{5}\cdot\sum_{\ell\in\Z}2^{-\tau c\left|\ell\right|}\leq\frac{2C_{5}}{1-2^{-\tau c}}.\label{eq:UpperConeRightConeSummationCaseSumOverI}
\end{equation}

In addition to the preceding inequality, we also need to estimate
the corresponding expression where the sum is taken over $j$ instead
of over $i$. To this end, we set $L:=2^{1+\alpha\left(n-\nu\right)}$
for brevity and estimate similar to the preceding case
\begin{align*}
\int_{-1}^{1}\left(1+\left|2^{n-\nu\alpha}\left(1-\kappa\iota\right)-2^{\alpha\left(n-\nu\right)}\kappa\xi\right|\right)^{-K}\d\xi & =\frac{2^{\alpha\left(\nu-n\right)}}{\left|\kappa\right|}\cdot\int_{2^{n-\nu\alpha}\left(1-\kappa\iota\right)-2^{\alpha\left(n-\nu\right)}\left|\kappa\right|}^{2^{n-\nu\alpha}\left(1-\kappa\iota\right)+2^{\alpha\left(n-\nu\right)}\left|\kappa\right|}\left(1+\left|\zeta\right|\right)^{-K}\d\zeta\\
\left({\scriptstyle \text{since }\frac{1}{4}\leq\left|\kappa\right|\leq2}\right) & \leq4\cdot2^{\alpha\left(\nu-n\right)}\cdot\int_{-2^{n-\nu\alpha}\kappa\iota+2^{n-\nu\alpha}-2^{1+\alpha\left(n-\nu\right)}}^{-2^{n-\nu\alpha}\kappa\iota+2^{n-\nu\alpha}+2^{1+\alpha\left(n-\nu\right)}}\left(1+\left|\zeta\right|\right)^{-K}\d\zeta\\
\left({\scriptstyle \text{since }\kappa=2^{\nu\left(\alpha-1\right)}\mu}\right) & =4\cdot2^{\alpha\left(\nu-n\right)}\cdot\int_{-2^{n-\nu}\mu\iota+2^{n-\nu\alpha}-L}^{-2^{n-\nu}\mu\iota+2^{n-\nu\alpha}+L}\left(1+\left|\zeta\right|\right)^{-K}\d\zeta.
\end{align*}
Now, a combination of equations \eqref{eq:UpperConeRightConeIntegralRewritten}
and \eqref{eq:UpperConeRightConeSummationCaseTheta1Estimate} yields
\begin{align*}
\left|\det T_{i}\right|^{-1}\int_{S_{i}^{\left(\alpha\right)}}\varrho_{j}\left(T_{j}^{-1}\xi\right)\d\xi & \leq3^{2+K+5M_{3}}\cdot\left|\iota\right|^{M_{1}}\cdot2^{-M_{1}(\nu-n)_{+}-M_{2}(n-\nu)_{+}}\:\int_{-1}^{1}\!\left(1\!+\!\left|2^{n-\nu\alpha}\left(1\!-\!\kappa\iota\right)\!-\!2^{\alpha\left(n-\nu\right)}\kappa\xi\right|\right)^{-K}\!\d\xi\\
 & \leq3^{4+K+5M_{3}}\cdot\left|\iota\right|^{M_{1}}\cdot2^{\alpha\left(\nu-n\right)-M_{1}(\nu-n)_{+}-M_{2}(n-\nu)_{+}}\cdot\int_{-2^{n-\nu}\mu\iota+2^{n-\nu\alpha}-L}^{-2^{n-\nu}\mu\iota+2^{n-\nu\alpha}+L}\left(1+\left|\zeta\right|\right)^{-K}\d\zeta.
\end{align*}
In conjunction with equations \eqref{eq:UpperConeRightConeMatrixChangeNorm}
and \eqref{eq:CapitalLambdaDefinition}, this entails
\begin{align*}
M_{j,i}^{\left(0\right)} & =\left(\frac{w_{j}^{s}}{w_{i}^{s}}\right)^{\tau}\cdot\left(1+\left\Vert T_{j}^{-1}T_{i}\right\Vert \right)^{\sigma}\cdot\left(\left|\det T_{i}\right|^{-1}\cdot\int_{S_{i}^{\left(\alpha\right)}}\varrho_{j}\left(T_{j}^{-1}\xi\right)\d\xi\right)^{\tau}\\
 & \leq6^{\sigma}\cdot3^{\tau\left(4+K+5M_{3}\right)}\cdot2^{\tau\left(s+\alpha\right)\left(\nu-n\right)+\sigma\left(n-\nu\right)_{+}-\tau M_{1}(\nu-n)_{+}-\tau M_{2}(n-\nu)_{+}}\cdot\left|\iota\right|^{\tau M_{1}}\cdot\\
 & \phantom{\leq\cdot}\Lambda_{n,m,\nu,\mu}\cdot\left[\int_{-2^{n-\nu}\mu\iota+2^{n-\nu\alpha}-L}^{-2^{n-\nu}\mu\iota+2^{n-\nu\alpha}+L}\!\left(1\!+\!\left|\zeta\right|\right)^{-K}\d\zeta\right]^{\tau}\!\!.
\end{align*}

For brevity, set $C_{6}:=6^{\sigma}\cdot3^{\tau\left(4+K+5M_{3}\right)}$
and $C_{7}:=2^{5+2\tau+2\sigma}\cdot10^{3+\sigma}\cdot C_{6}$ and
recall from equations \eqref{eq:CapitalLambdaDefinition} and \eqref{eq:KappaIotaDefinition}
that
\begin{equation}
\Lambda_{n,m,\nu,\mu}=\left(1+2^{n-\alpha\nu}\left|1-\kappa\iota\right|\right)^{\sigma}=\left(1+\left|2^{n-\alpha\nu}-2^{n-\nu}\mu\iota\right|\right)^{\sigma}=\left(1+\left|2^{n-\alpha\nu}-2^{n-\nu}\left|\iota\right|{\rm sign}\left(\iota\right)\mu\right|\right)^{\sigma}.\label{eq:UpperConeRightConeCapitalLambdaRewritten}
\end{equation}
We now invoke Lemma \ref{lem:WeightedSumOfShiftedIntegrals} and
the associated remark (with $L=2^{1+\alpha\left(n-\nu\right)}$, $N=\sigma$,
$p=\tau$, $M=2^{n-\nu\alpha}$ and $\beta=2^{n-\nu}\left|\iota\right|$)
to justify the following estimate:
\begin{align*}
\smash{\sum_{\substack{\left|\mu\right|\leq G_{\nu}\\
\text{s.t. Case 2 holds}
}
}}M_{j,i}^{\left(0\right)} & \leq C_{6}\cdot2^{\tau\left(s+\alpha\right)\left(\nu-n\right)+\sigma\left(n-\nu\right)_{+}-\tau M_{1}(\nu-n)_{+}-\tau M_{2}(n-\nu)_{+}}\cdot\left|\iota\right|^{\tau M_{1}}\cdot\\
 & \phantom{\leq C_{5}\cdot}\sum_{\mu\in\Z}\left(\!\Lambda_{n,m,\nu,\mu}\!\cdot\!\left[\int_{-2^{n-\nu}\iota\mu+2^{n-\nu\alpha}-L}^{-2^{n-\nu}\iota\mu+2^{n-\nu\alpha}+L}\left(1\!+\!\left|\zeta\right|\right)^{-K}\d\zeta\right]^{\tau}\right)\\
\left({\scriptstyle \text{eq. }\eqref{eq:UpperConeRightConeCapitalLambdaRewritten}\text{ and }\ell=-{\rm sign}\left(\iota\right)\mu}\right) & =C_{6}\cdot2^{\tau\left(s+\alpha\right)\left(\nu-n\right)+\sigma\left(n-\nu\right)_{+}-\tau M_{1}(\nu-n)_{+}-\tau M_{2}(n-\nu)_{+}}\cdot\left|\iota\right|^{\tau M_{1}}\cdot\\
 & \phantom{=C_{5}\cdot}\sum_{\ell\in\Z}\left(\!\left(1+\left|2^{n-\alpha\nu}+2^{n-\nu}\left|\iota\right|\cdot\ell\right|\right)^{\sigma}\!\cdot\!\left[\int_{2^{n-\nu}\left|\iota\right|\ell+2^{n-\nu\alpha}-L}^{2^{n-\nu}\left|\iota\right|\ell+2^{n-\nu\alpha}+L}\left(1\!+\!\left|\zeta\right|\right)^{-K}\d\zeta\right]^{\tau}\right)\\
\left({\scriptstyle \text{Lem. }\ref{lem:WeightedSumOfShiftedIntegrals}\text{ and remark}}\right) & \leq2^{3+\tau+\sigma}\cdot10^{\sigma+3}\cdot C_{6}\cdot2^{\tau\left(s+\alpha\right)\left(\nu-n\right)+\sigma\left(n-\nu\right)_{+}-\tau M_{1}(\nu-n)_{+}-\tau M_{2}(n-\nu)_{+}}\cdot\left|\iota\right|^{\tau M_{1}}\cdot\\
 & \phantom{\leq\cdot}\left\Vert \left(1+\left|\mybullet\right|\right)^{-K}\right\Vert _{\frac{2+\sigma}{\tau}}^{\tau}\cdot2^{\tau\left[1+\alpha\left(n-\nu\right)\right]}\cdot\left(1+2^{\sigma\left[1+\alpha\left(n-\nu\right)\right]}\right)\cdot\left(1+\frac{2^{1+\alpha\left(n-\nu\right)}}{2^{n-\nu}\left|\iota\right|}+\frac{2^{\nu-n}}{\left|\iota\right|}\right)\\
\left({\scriptstyle \text{since }K\geq\frac{2+\sigma}{\tau}}\right) & \leq C_{7}\cdot2^{\tau s\left(\nu-n\right)+\left(1+\alpha\right)\sigma\left(n-\nu\right)_{+}-\tau M_{1}(\nu-n)_{+}-\tau M_{2}(n-\nu)_{+}}\cdot\left|\iota\right|^{\tau M_{1}}\cdot\left(1\!+\!\frac{2^{\left(1-\alpha\right)\left(\nu-n\right)}}{\left|\iota\right|}\!+\!\frac{2^{\nu-n}}{\left|\iota\right|}\right)\!.
\end{align*}
Here, we note that we indeed have $\beta>0$, since $\left|m\right|\geq2>0$,
so that $\iota\neq0$. Now, recall $\left|\iota\right|\leq2$ and
$M_{1}\geq\frac{1}{\tau}$, so that $\left|\iota\right|^{\tau M_{1}}\leq2^{\tau M_{1}}\leq2^{\tau M_{3}}$
and furthermore 
\[
\frac{\left|\iota\right|^{\tau M_{1}}}{\left|\iota\right|}=\left|\iota\right|^{\tau M_{1}-1}\leq2^{\tau M_{1}-1}\leq2^{\tau M_{3}}.
\]
Hence, we can continue the estimate from above as follows:
\begin{align*}
\smash{\sum_{\substack{\left|\mu\right|\leq G_{\nu}\\
\text{s.t. Case 2 holds}
}
}}M_{j,i}^{\left(0\right)} & \leq2^{\tau M_{3}}C_{7}\cdot2^{\tau s\left(\nu-n\right)+\left(1+\alpha\right)\sigma\left(n-\nu\right)_{+}-\tau M_{1}(\nu-n)_{+}-\tau M_{2}(n-\nu)_{+}}\cdot\left(1\!+\!2^{\left(1-\alpha\right)\left(\nu-n\right)}\!+\!2^{\nu-n}\right)\\
 & \leq2^{2+\tau M_{3}}C_{7}\cdot2^{\tau s\left(\nu-n\right)+\left(1+\alpha\right)\sigma\left(n-\nu\right)_{+}-\tau M_{1}(\nu-n)_{+}-\tau M_{2}(n-\nu)_{+}}\cdot2^{\left(\nu-n\right)_{+}}.
\end{align*}
Now, set $C_{8}:=2^{2+\tau M_{3}}C_{7}$ and observe
\begin{align*}
 & 2^{\tau s\left(\nu-n\right)+\left(1+\alpha\right)\sigma\left(n-\nu\right)_{+}-\tau M_{1}(\nu-n)_{+}-\tau M_{2}(n-\nu)_{+}}\cdot2^{\left(\nu-n\right)_{+}}\\
 & =\begin{cases}
2^{\tau s\left(\nu-n\right)-\tau M_{1}\left|n-\nu\right|}\cdot2^{\nu-n}=2^{-\left|\nu-n\right|\left(\tau M_{1}-\tau s-1\right)}, & \text{if }n\leq\nu,\\
2^{\tau s\left(\nu-n\right)+\left(1+\alpha\right)\sigma\left(n-\nu\right)-\tau M_{2}\left|n-\nu\right|}=2^{-\left|\nu-n\right|\left(\tau s-\left(1+\alpha\right)\sigma+\tau M_{2}\right)}, & \text{if }n\geq\nu
\end{cases}\\
 & \leq2^{-\tau c\left|\nu-n\right|},
\end{align*}
since the assumptions of Lemma \ref{lem:MainShearletLemma} ensure
$M_{1}\geq M_{1}^{(0)}+c\geq s+\frac{1}{\tau}+c$, as well as $M_{2}\geq M_{2}^{(0)}+c\geq\left(1+\alpha\right)\frac{\sigma}{\tau}-s+c$.

All in all, we finally conclude
\begin{equation}
\sup_{i=\left(n,m,\varepsilon,\delta\right)\in I^{\left(\ell_{1}\right)}}\,\sum_{\substack{j=\left(\nu,\mu,e,d\right)\in I^{\left(\ell_{2}\right)}\\
\text{s.t. Case 2 holds}
}
}M_{j,i}^{\left(0\right)}\leq\sup_{n\in\N_{0}}C_{8}\cdot\sum_{\nu=0}^{\infty}2^{-\tau c\left|\nu-n\right|}\leq C_{8}\cdot\sum_{\ell\in\Z}2^{-\tau c\left|\ell\right|}\leq\frac{2C_{8}}{1-2^{-\tau c}},\label{eq:UpperConeRightConeSummationCaseSumOverJ}
\end{equation}
which completes our considerations in the present case.

\medskip{}

\textbf{Case 3}: The remaining case, i.e., $\left[\left|\kappa\right|<\frac{1}{4}\right]\vee\left[\left|m\right|\leq1\right]\vee\left[\left(n>\nu\right)\wedge\left(\left|\iota\right|<\frac{1}{4}\right)\right]$,
as well as $n\geq\frac{3}{1-\alpha}$. Our first step is to show
\begin{equation}
\begin{split}M_{j,i}^{\left(0\right)} & \leq C_{9}\cdot2^{\tau s\left(\nu-n\right)}\cdot2^{\sigma\cdot\left(n-\nu\right)_{+}}\cdot\left(1+2^{n-\nu\alpha}\left|\lambda_{n,m,\nu,\mu}\right|\right)^{\sigma}\cdot\min\left\{ 1,\,2^{\tau M_{1}\left(n-\nu\right)}\right\} \cdot\min\left\{ 1,\,2^{\tau K\left(\nu\alpha-n\right)}\right\} \\
 & \overset{\left(\ast\right)}{\leq}C_{10}\cdot2^{\tau s\left(\nu-n\right)}\cdot2^{\sigma\cdot\left(n-\nu\right)_{+}}\cdot2^{-\tau M_{1}\left(\nu-n\right)_{+}}\cdot2^{\left(\sigma-\tau K\right)\left(n-\nu\alpha\right)_{+}}
\end{split}
\label{eq:UpperConeRightConeStandardEstimate}
\end{equation}
for $C_{9}:=6^{\sigma}\cdot\left(3^{3+K+3M_{3}}\cdot4^{K}\right)^{\tau}$
and $C_{10}:=6^{\sigma}\cdot C_{9}$. Furthermore, as an intermediate
result of independent interest, we also show
\begin{equation}
\left|2^{n-\nu\alpha}\lambda_{n,m,\nu,\mu}-2^{n\alpha-\nu}\mu\xi\right|\geq\frac{2^{n-\nu\alpha}}{4}\qquad\forall\xi\in\left[-1,1\right].\label{eq:UpperConeRightConeStandardIntegrandEstimate}
\end{equation}
Here, the step marked with $\left(\ast\right)$ in equation \eqref{eq:UpperConeRightConeStandardEstimate}
used that $\left|\lambda_{n,m,\nu,\mu}\right|\leq5$, so that 
\begin{equation}
\Lambda_{n,m,\nu,\mu}=\left(1+2^{n-\nu\alpha}\left|\lambda_{n,m,\nu,\mu}\right|\right)^{\sigma}\leq6^{\sigma}\cdot2^{\sigma\left(n-\nu\alpha\right)_{+}}.\label{eq:UpperConeRightConeLastCaseSpecialFactorBrutalEstimate}
\end{equation}

To prove equations \eqref{eq:UpperConeRightConeStandardEstimate}
and \eqref{eq:UpperConeRightConeStandardIntegrandEstimate}, we distinguish
three subcases:
\begin{enumerate}
\item We have $\left|m\right|\leq1$. Because of $n\geq\frac{3}{1-\alpha}$,
this implies
\[
\left|\iota\right|=2^{-\left(1-\alpha\right)n}\left|m\right|\leq2^{-\left(1-\alpha\right)n}\leq2^{-3}=\frac{1}{8}\leq\frac{1}{4},
\]
so that equation \eqref{eq:Theta2EstimateForSmallIota} yields $\left|2^{n-\nu\alpha}\lambda_{n,m,\nu,\mu}-2^{n\alpha-\nu}\mu\xi\right|\geq\frac{2^{n-\nu\alpha}}{4}$
for all $\xi\in\left[-1,1\right]$, i.e., equation \eqref{eq:UpperConeRightConeStandardIntegrandEstimate}
holds. Hence, a combination of equations \eqref{eq:UpperConeRightConeMatrixChangeNorm},
\eqref{eq:UpperConeRightConeIntegralRewritten}, \eqref{eq:UpperConeRightConeStandardIntegrandEstimate},
and \eqref{eq:Theta1StandardEstimate} yields
\begin{align*}
\;\quad\qquad M_{j,i}^{\left(0\right)} & =\left(\frac{w_{j}^{s}}{w_{i}^{s}}\right)^{\tau}\cdot\left(1+\left\Vert T_{j}^{-1}T_{i}\right\Vert \right)^{\sigma}\cdot\left(\left|\det T_{i}\right|^{-1}\cdot\int_{S_{i}^{\left(\alpha\right)}}\varrho_{j}\left(T_{j}^{-1}\xi\right)\d\xi\right)^{\tau}\\
 & \leq6^{\sigma}\!\cdot\!2^{\tau s\left(\nu-n\right)}\cdot2^{\sigma\cdot\left(n-\nu\right)_{+}}\cdot\left(1\!+\!2^{n-\nu\alpha}\cdot\left|\lambda_{n,m,\nu,\mu}\right|\right)^{\sigma}\cdot\left[3^{3+K+3M_{3}}\!\cdot\!\min\left\{ 1,\,2^{M_{1}\left(n-\nu\right)}\right\} \!\cdot\!\left(1\!+\!\frac{2^{n-\nu\alpha}}{4}\right)^{-K}\right]^{\tau}\\
 & \leq C_{9}\!\cdot2^{\tau s\left(\nu-n\right)}\cdot2^{\sigma\cdot\left(n-\nu\right)_{+}}\cdot\left(1\!+\!2^{n-\nu\alpha}\!\cdot\left|\lambda_{n,m,\nu,\mu}\right|\right)^{\sigma}\!\cdot\!\min\!\left\{ 1,\,2^{\tau M_{1}\left(n-\nu\right)}\right\} \!\cdot\!\min\!\left\{ 1,2^{-\tau K\left(n-\nu\alpha\right)}\right\} \!.
\end{align*}
Thus, equations \eqref{eq:UpperConeRightConeStandardEstimate} and
\eqref{eq:UpperConeRightConeStandardIntegrandEstimate} are valid
in this case.
\item We have $\left|\kappa\right|<\frac{1}{4}$. In this case, equation
\eqref{eq:Theta2EstimateForSmallKappa} yields $\left|2^{n-\nu\alpha}\lambda_{n,m,\nu,\mu}-2^{n\alpha-\nu}\mu\xi\right|\geq\frac{2^{n-\nu\alpha}}{4}$
for all $\xi\in\left[-1,1\right]$. Then, validity of equations \eqref{eq:UpperConeRightConeStandardEstimate}
and \eqref{eq:UpperConeRightConeStandardIntegrandEstimate} follows
just as in the previous case.
\item The remaining case, i.e., $\left|\kappa\right|\geq\frac{1}{4}$ and
$\left|m\right|\geq2$. Since we are in Case 3, this entails $n>\nu$
\textbf{and} $\left|\iota\right|<\frac{1}{4}$. Since we also have
$\alpha<1$ and $n\geq\frac{3}{1-\alpha}$, equation \eqref{eq:Theta2EstimateForSmallIota}
yields $\left|2^{n-\nu\alpha}\lambda_{n,m,\nu,\mu}-2^{n\alpha-\nu}\mu\xi\right|\geq\frac{2^{n-\nu\alpha}}{4}$
for all $\xi\in\left[-1,1\right]$, so that the desired estimates
follow just as in the previous two cases.
\end{enumerate}
Now, we observe that $n\geq\nu$ implies $\left(n-\nu\right)_{+}=n-\nu$,
as well as $\left(\nu-n\right)_{+}=0$ and finally $\left(n-\nu\alpha\right)_{+}=n-\nu\alpha$,
since $n\geq\nu\geq\nu\alpha$. Consequently, equation \eqref{eq:UpperConeRightConeStandardEstimate}
yields
\begin{equation}
\begin{split}\sum_{\substack{i=\left(n,m,\varepsilon,\delta\right)\in I^{\left(\ell_{1}\right)}\\
\text{s.t. }n\geq\nu\text{ and Case 3 holds}
}
}M_{j,i}^{\left(0\right)} & \leq C_{10}\cdot\sum_{n=\nu}^{\infty}\ \sum_{\left|m\right|\leq G_{n}}\left[2^{\tau s\left(\nu-n\right)}\cdot2^{\sigma\left(n-\nu\right)}\cdot2^{\left(\sigma-\tau K\right)\left(n-\nu\alpha\right)}\right]\\
\left({\scriptstyle \text{since }G_{n}=\left\lceil 2^{n\left(1-\alpha\right)}\right\rceil \leq1+2^{n\left(1-\alpha\right)}\leq2\cdot2^{n\left(1-\alpha\right)}}\right) & \leq6C_{10}\cdot\sum_{n=\nu}^{\infty}2^{n\left(1-\alpha\right)}\cdot2^{\tau s\left(\nu-n\right)}\cdot2^{\sigma\left(n-\nu\right)}\cdot2^{\left(\sigma-\tau K\right)\left(n-\nu\alpha\right)}\\
 & =6C_{10}\cdot2^{\nu\left(\tau s-\sigma+\tau\alpha K-\alpha\sigma\right)}\cdot\sum_{n=\nu}^{\infty}2^{n\left(1-\alpha-\tau s+2\sigma-\tau K\right)}\\
\left({\scriptstyle \text{eq. }\eqref{eqGeometricSumNegativeExponent}}\right) & \overset{\left(\ast\right)}{\leq}\frac{6C_{10}}{1-2^{1-\alpha-\tau s+2\sigma-\tau K}}\cdot2^{\nu\left(\tau s-\sigma+\tau\alpha K-\alpha\sigma\right)}\cdot2^{\nu\left(1-\alpha-\tau s+2\sigma-\tau K\right)}\\
 & \leq\frac{6C_{10}}{1-2^{-\tau c}}\cdot2^{\nu\left(1-\alpha\right)\left(1+\sigma-\tau K\right)}\leq\frac{6C_{10}}{1-2^{-\tau c}}.
\end{split}
\label{eq:UpperConeRightConeLastCaseNLargerThanNuSum}
\end{equation}
Here, the last step used that $K\geq\frac{2+\sigma}{\tau}\geq\frac{1+\sigma}{\tau}$
by the assumptions of Lemma \ref{lem:MainShearletLemma}, so that
$1+\sigma-\tau K\leq0$. Furthermore, the step marked with $\left(\ast\right)$
used that the assumptions of Lemma \ref{lem:MainShearletLemma} imply
$K\geq K_{0}+c\geq\frac{1-\alpha}{\tau}+2\frac{\sigma}{\tau}-s+c$
and thus $1-\alpha-\tau s+2\sigma-\tau K\leq-\tau c<0$ and finally
that
\begin{equation}
\sum_{n=\nu}^{\infty}2^{n\phi}=2^{\nu\phi}\cdot\sum_{\ell=0}^{\infty}2^{\ell\phi}=\frac{2^{\nu\phi}}{1-2^{\phi}}\qquad\text{for arbitrary }\phi\in\left(-\infty,0\right).\label{eqGeometricSumNegativeExponent}
\end{equation}

To estimate the sum over $j$ instead of over $i$, we observe again
that $n\geq\nu$ implies $n\geq\nu\geq\alpha\nu$. In combination
with equation \eqref{eq:UpperConeRightConeStandardEstimate}, this
implies
\begin{align}
\sum_{\substack{j=\left(\nu,\mu,e,d\right)\in I^{\left(\ell_{2}\right)}\\
\text{s.t. }\nu\leq n\text{ and Case 3 holds}
}
}M_{j,i}^{\left(0\right)} & \leq C_{10}\cdot\sum_{\nu=0}^{n}\:\sum_{\left|\mu\right|\leq G_{\nu}}2^{\left(\sigma-\tau s\right)\left(n-\nu\right)}\cdot2^{\left(\sigma-\tau K\right)\left(n-\nu\alpha\right)}\nonumber \\
\left({\scriptstyle \text{since }G_{\nu}=\left\lceil 2^{\nu\left(1-\alpha\right)}\right\rceil \leq1+2^{\nu\left(1-\alpha\right)}\leq2\cdot2^{\nu\left(1-\alpha\right)}}\right) & \leq6C_{10}\cdot2^{n\left(2\sigma-\tau s-\tau K\right)}\cdot\sum_{\nu=0}^{n}2^{\nu\left(1-\alpha+\tau s-\left(1+\alpha\right)\sigma+\tau\alpha K\right)}.\label{eq:UpperConeRightConeLastCaseNuSmallerThanNSumPrep}
\end{align}
To further estimate the right-hand side of this expression, we first
observe that $g:\left[0,\infty\right)\to\left[0,\infty\right),x\mapsto x\cdot2^{-x}$
is differentiable with derivative $g'\left(x\right)=2^{-x}\cdot\left(1-x\cdot\ln2\right)$.
Hence, $g'\left(x\right)>0$ for $0\leq x<\frac{1}{\ln2}$ and $g'\left(x\right)<0$
for $x>\frac{1}{\ln2}$. Consequently, $g$ attains its unique global
maximum at $x=\frac{1}{\ln2}$. But we have $\ln2=\frac{1}{2}\ln2^{2}\geq\frac{1}{2}\ln e=\frac{1}{2}$
and thus $g\left(x\right)\leq g\left(\frac{1}{\ln2}\right)=\frac{1}{\ln2}\cdot2^{-\frac{1}{\ln2}}\leq\frac{2}{e}\leq1$
for all $x\in\left[0,\infty\right)$. For arbitrary $n\in\N_{0}$
and $\phi>0$, this implies $n\cdot2^{-\phi n}=\frac{1}{\phi}\cdot\left(\phi n\cdot2^{-\phi n}\right)=\frac{1}{\phi}\cdot g\left(\phi n\right)\leq\frac{1}{\phi}$
and thus 
\[
\left(n+1\right)\cdot2^{-\phi n}\leq1+n\cdot2^{-\phi n}\leq1+\frac{1}{\phi}\qquad\forall n\in\N_{0}\text{ and }\phi\in\left(0,\infty\right).
\]
Now, set $\beta:=1-\alpha+\tau s-\left(1+\alpha\right)\sigma+\tau\alpha K$
for brevity and note
\begin{align*}
2^{n\left(2\sigma-\tau s-\tau K\right)}\cdot\sum_{\nu=0}^{n}2^{\beta\nu} & \leq\begin{cases}
2^{n\left(2\sigma-\tau s-\tau K\right)}\cdot\left(n+1\right)\cdot2^{\beta n}=\left(n+1\right)\cdot2^{n\left(1-\alpha\right)\left(\sigma+1-\tau K\right)}, & \text{if }\beta\geq0,\\
2^{n\left(2\sigma-\tau s-\tau K\right)}\cdot\left(n+1\right), & \text{if }\beta<0
\end{cases}\\
 & \overset{\left(\ast\right)}{\leq}\begin{cases}
\left(n+1\right)\cdot2^{-n\left(1-\alpha\right)\tau c}, & \text{if }\beta\geq0,\\
\left(n+1\right)\cdot2^{-n\tau c}, & \text{if }\beta<0
\end{cases}\\
\left({\scriptstyle \text{since }\left(n+1\right)\cdot2^{-\phi n}\leq1+\frac{1}{\phi}}\right) & \leq\begin{cases}
1+\frac{1}{\left(1-\alpha\right)\tau c}, & \text{if }\beta\geq0,\\
1+\frac{1}{\tau c}, & \text{if }\beta<0
\end{cases}\\
 & \leq1+\frac{1}{\left(1-\alpha\right)\tau c},
\end{align*}
where we recall that we assume $\alpha<1$ in the present case. Furthermore,
the step marked with $\left(\ast\right)$ used that the assumptions
of Lemma \ref{lem:MainShearletLemma} ensure $K\geq K_{0}+c\geq\frac{1-\alpha}{\tau}+2\frac{\sigma}{\tau}-s+c\geq2\frac{\sigma}{\tau}-s+c$
and thus $2\sigma-\tau s-\tau K\leq-\tau c$, as well as $K\geq K_{0}+c\geq\frac{1+\sigma}{\tau}+c$,
so that $\sigma+1-\tau K\leq-\tau c$.

By plugging this into equation \eqref{eq:UpperConeRightConeLastCaseNuSmallerThanNSumPrep},
we obtain
\begin{equation}
\sum_{\substack{j=\left(\nu,\mu,e,d\right)\in I^{\left(\ell_{2}\right)}\\
\text{s.t. }\nu\leq n\text{ and Case 3 holds}
}
}M_{j,i}^{\left(0\right)}\leq6C_{10}\cdot\left(1+\frac{1}{\left(1-\alpha\right)\tau c}\right)\qquad\forall i=\left(n,m,\varepsilon,\delta\right)\in I^{\left(\ell_{1}\right)}.\label{eq:UpperConeRightConeLastCaseNuSmallerThanNSum}
\end{equation}

Together, equations \eqref{eq:UpperConeRightConeLastCaseNLargerThanNuSum}
and \eqref{eq:UpperConeRightConeLastCaseNuSmallerThanNSum} take
care of the case $\nu\leq n$, under the general assumptions of the
current case. Hence, we only need to further consider the case $\nu>n$,
which we now do.

\medskip{}

Using estimate \eqref{eq:UpperConeRightConeStandardEstimate}, we
get
\begin{align*}
\sum_{\substack{i=\left(n,m,\varepsilon,\delta\right)\in I^{\left(\ell_{1}\right)}\\
\text{s.t. }n<\nu\text{ and Case 3 holds}
}
}M_{j,i}^{\left(0\right)} & \leq C_{10}\cdot\sum_{n=0}^{\nu}\:\sum_{\left|m\right|\leq G_{n}}2^{\tau s\left(\nu-n\right)}\cdot2^{-\tau M_{1}\left(\nu-n\right)}\cdot2^{\left(\sigma-\tau K\right)\left(n-\nu\alpha\right)_{+}}\\
\left({\scriptstyle \text{since }G_{n}=\left\lceil 2^{\left(1-\alpha\right)n}\right\rceil \leq1+2^{\left(1-\alpha\right)n}\leq2\cdot2^{\left(1-\alpha\right)n}}\right) & \leq6C_{10}\cdot\sum_{n=0}^{\nu}2^{\left(\tau s-\tau M_{1}\right)\left(\nu-n\right)}2^{\left(\sigma-\tau K\right)\cdot\left(n-\nu\alpha\right)_{+}}\cdot2^{n\left(1-\alpha\right)}.
\end{align*}
We now divide the sum into the two parts were we know the sign of
$n-\nu\alpha$. First, we observe that the assumptions of Lemma \ref{lem:MainShearletLemma}
entail $\sigma-\tau K\leq0$, so that $2^{\left(\sigma-\tau K\right)\cdot\left(n-\nu\alpha\right)_{+}}\leq1$.
Consequently,
\begin{align*}
\sum_{0\leq n\leq\nu\alpha}2^{\left(\tau s-\tau M_{1}\right)\left(\nu-n\right)}2^{\left(\sigma-\tau K\right)\cdot\left(n-\nu\alpha\right)_{+}}\cdot2^{n\left(1-\alpha\right)} & \leq2^{\nu\left(\tau s-\tau M_{1}\right)}\cdot\sum_{n=0}^{\left\lfloor \nu\alpha\right\rfloor }2^{\left(\tau M_{1}-\tau s+1-\alpha\right)n}\\
\left({\scriptstyle \text{eq. }\eqref{eq:GeometricSumPositiveExponent}\text{ and }\tau M_{1}-\tau s+1-\alpha\geq\tau M_{1}-\tau s>0}\right) & \overset{\left(\ast\right)}{\leq}\frac{2^{\tau M_{1}-\tau s+1-\alpha}}{2^{\tau M_{1}-\tau s+1-\alpha}-1}\cdot2^{\nu\left(\tau s-\tau M_{1}\right)}\cdot2^{\left\lfloor \nu\alpha\right\rfloor \left(\tau M_{1}-\tau s+1-\alpha\right)}\\
\left({\scriptstyle \text{since }\tau M_{1}-\tau s+1-\alpha>0\text{ and }\left\lfloor \nu\alpha\right\rfloor \leq\nu\alpha}\right) & \leq\frac{2^{\tau M_{1}-\tau s+1-\alpha}}{2^{\tau M_{1}-\tau s+1-\alpha}-1}\cdot2^{\nu\left(\tau s-\tau M_{1}\right)}\cdot2^{\nu\alpha\left(\tau M_{1}-\tau s+1-\alpha\right)}\\
 & \leq\frac{1}{1-2^{-(\tau M_{1}-\tau s+1-\alpha)}}2^{\nu(1-\alpha)(\alpha+\tau s-\tau M_{1})}\\
\left({\scriptstyle \text{since }\alpha+\tau s-\tau M_{1}\leq0\text{ and }\tau M_{1}-\tau s+1-\alpha\geq\tau M_{1}-\tau s\geq1}\right) & \leq\frac{1}{1-2^{-1}}=2.
\end{align*}
Here, the step marked with $\left(\ast\right)$ used that the geometric
sum formula shows
\begin{equation}
\sum_{\ell=0}^{n}2^{\phi\ell}=\frac{2^{\left(n+1\right)\phi}-1}{2^{\phi}-1}\leq\frac{2^{\left(n+1\right)\phi}}{2^{\phi}-1}=\frac{2^{\phi}}{2^{\phi}-1}\cdot2^{n\phi}=\frac{1}{1-2^{-\phi}}\cdot2^{n\phi}=:C^{\left(\phi\right)}\cdot2^{n\phi}\text{ for arbitrary }\phi>0.\label{eq:GeometricSumPositiveExponent}
\end{equation}

Now, we consider the remaining part of the sum. To this end, we first
observe that the assumptions of Lemma \ref{lem:MainShearletLemma}
entail $M_{1}\geq M_{0}:=s+\frac{\alpha}{\tau}$. In conjunction with
$n\leq\nu$, this implies $-\tau M_{1}\left(\nu-n\right)\leq-\tau M_{0}\left(\nu-n\right)$
and thus $2^{-\tau M_{1}\left(\nu-n\right)}\leq2^{-\tau M_{0}\left(\nu-n\right)}$.
Consequently,
\begin{align*}
\sum_{\nu\alpha<n\leq\nu}2^{\left(\tau s-\tau M_{1}\right)\left(\nu-n\right)}2^{\left(\sigma-\tau K\right)\cdot\left(n-\nu\alpha\right)_{+}}\cdot2^{n\left(1-\alpha\right)} & =\sum_{n=1+\left\lfloor \nu\alpha\right\rfloor }^{\nu}2^{\left(\tau s-\tau M_{1}\right)\left(\nu-n\right)}2^{\left(\sigma-\tau K\right)\cdot\left(n-\nu\alpha\right)}\cdot2^{n\left(1-\alpha\right)}\\
 & \leq\sum_{n=1+\left\lfloor \nu\alpha\right\rfloor }^{\nu}2^{\left(\tau s-\tau M_{0}\right)\left(\nu-n\right)}2^{\left(\sigma-\tau K\right)\cdot\left(n-\nu\alpha\right)}\cdot2^{n\left(1-\alpha\right)}\\
 & =2^{\nu\left(\tau s-\tau M_{0}-\alpha\left(\sigma-\tau K\right)\right)}\cdot\sum_{n=1+\left\lfloor \nu\alpha\right\rfloor }^{\nu}2^{n\left(1-\alpha+\tau M_{0}-\tau s+\sigma-\tau K\right)}\\
 & \leq2^{\nu\left(-\alpha-\alpha\left(\sigma-\tau K\right)\right)}\cdot\sum_{n=1+\left\lfloor \nu\alpha\right\rfloor }^{\infty}2^{n\left(1-\alpha+\tau M_{0}-\tau s+\sigma-\tau K\right)}\\
\left({\scriptstyle \text{eq. }\eqref{eqGeometricSumNegativeExponent}\text{ and }1-\alpha+\tau M_{0}-\tau s+\sigma-\tau K=1+\sigma-\tau K<0}\right) & \leq\frac{1}{1-2^{1+\sigma-\tau K}}\cdot2^{\alpha\nu\left(\tau K-\sigma-1\right)}\cdot2^{\left(1+\left\lfloor \nu\alpha\right\rfloor \right)\cdot\left(1+\sigma-\tau K\right)}\\
\left({\scriptstyle \text{since }1+\left\lfloor \nu\alpha\right\rfloor \geq\nu\alpha\text{ and }1+\sigma-\tau K<0}\right) & \leq\frac{1}{1-2^{1+\sigma-\tau K}}\cdot2^{\alpha\nu\left(\tau K-\sigma-1\right)}\cdot2^{\nu\alpha\left(1+\sigma-\tau K\right)}\\
\left({\scriptstyle \text{since }1+\sigma-\tau K\leq-1}\right) & =\frac{1}{1-2^{1+\sigma-\tau K}}\leq\frac{1}{1-2^{-1}}=2.
\end{align*}
Altogether, the preceding four displayed equations show
\begin{equation}
\sup_{j=\left(\nu,\mu,e,d\right)\in I^{\left(\ell_{2}\right)}}\:\sum_{\substack{i=\left(n,m,\varepsilon,\delta\right)\in I^{\left(\ell_{1}\right)}\\
\text{s.t. }n<\nu\text{ and Case 3 holds}
}
}M_{j,i}^{\left(0\right)}\leq6C_{10}\cdot\left(2+2\right)=24\cdot C_{10}.\label{eq:UpperConeRightConeCase3.2NSum}
\end{equation}

It remains to consider the sum over $j\in I^{\left(\ell_{2}\right)}$
instead of over $i\in I^{\left(\ell_{1}\right)}$. To this end, we
first consider the special case $\alpha=0$. In this case we have
$G_{\nu}=\left\lceil 2^{\left(1-\alpha\right)\nu}\right\rceil =2^{\nu}$
for all $\nu\in\N_{0}$, as well as $\left(n-\nu\alpha\right)_{+}=n_{+}=n$,
so that equation \eqref{eq:UpperConeRightConeStandardEstimate} implies
\begin{equation}
\begin{split}\smash{\sum_{\substack{j=\left(\nu,\mu,e,d\right)\in I^{\left(\ell_{2}\right)}\\
\text{s.t. }n<\nu\text{ and Case 3 holds}
}
}}M_{j,i}^{\left(0\right)} & \leq C_{10}\cdot\sum_{\nu=n}^{\infty}\:\sum_{\left|\mu\right|\leq G_{\nu}}2^{\left(\tau s-\tau M_{1}\right)\left(\nu-n\right)}\cdot2^{\left(\sigma-\tau K\right)\left(n-\nu\alpha\right)_{+}}\\
 & \leq3C_{10}\cdot2^{n\left(\sigma-\tau K+\tau M_{1}-\tau s\right)}\cdot\sum_{\nu=n}^{\infty}2^{\nu\left(1+\tau s-\tau M_{1}\right)}\\
\left({\scriptstyle \text{eq. }\eqref{eqGeometricSumNegativeExponent}\text{ and }1+\tau s-\tau M_{1}<0\text{ by the assump. of Lem. }\ref{lem:MainShearletLemma}}\right) & \leq\frac{3C_{10}}{1-2^{1+\tau s-\tau M_{1}}}\cdot2^{n\left(\sigma-\tau K+\tau M_{1}-\tau s+1+\tau s-\tau M_{1}\right)}\\
\left({\scriptstyle \text{since }1+\tau s-\tau M_{1}\leq-\tau c\text{ by the assump. of Lem. }\ref{lem:MainShearletLemma}}\right) & \leq\frac{3C_{10}}{1-2^{-\tau c}}\cdot2^{n\left(1+\sigma-\tau K\right)}\\
\left({\scriptstyle \text{since }1+\sigma-\tau K\leq0\text{ by the assump. of Lem. }\ref{lem:MainShearletLemma}}\right) & \leq\frac{3C_{10}}{1-2^{-\tau c}}\quad\text{ in case of }\alpha=0.
\end{split}
\label{eq:UpperConeRightConeCase3.2NuSumAlpha0}
\end{equation}

Having taken care of the case $\alpha=0$, we can now assume $\alpha>0$.
With another application of equation \eqref{eq:UpperConeRightConeStandardEstimate},
we conclude
\begin{align*}
\sum_{\substack{j=\left(\nu,\mu,e,d\right)\in I^{\left(\ell_{2}\right)}\\
\text{s.t. }n<\nu\text{ and Case 3 holds}
}
}M_{j,i}^{\left(0\right)} & \leq C_{10}\cdot\sum_{\nu=n}^{\infty}\:\sum_{\left|\mu\right|\leq G_{\nu}}2^{\left(\tau s-\tau M_{1}\right)\left(\nu-n\right)}\cdot2^{\left(\sigma-\tau K\right)\left(n-\nu\alpha\right)_{+}}\\
\left({\scriptstyle \text{since }G_{\nu}=\left\lceil 2^{\nu\left(1-\alpha\right)}\right\rceil \leq1+2^{\nu\left(1-\alpha\right)}\leq2\cdot2^{\nu\left(1-\alpha\right)}}\right) & \leq6C_{10}\cdot\sum_{\nu=n}^{\infty}\left(2^{\nu\left(1-\alpha\right)}\cdot2^{\left(\tau s-\tau M_{1}\right)\left(\nu-n\right)}\cdot2^{\left(\sigma-\tau K\right)\left(n-\nu\alpha\right)_{+}}\right).
\end{align*}
As in the previous case, we now split the series into two parts, according
to the sign of $n-\nu\alpha$. But first, we observe by the assumptions
of Lemma \ref{lem:MainShearletLemma} that $K\geq K_{00}:=\frac{1+\sigma}{\tau}$
and hence $2^{\left(\sigma-\tau K\right)\left(n-\nu\alpha\right)_{+}}\leq2^{\left(\sigma-\tau K_{00}\right)\left(n-\nu\alpha\right)_{+}}=2^{-\left(n-\nu\alpha\right)_{+}}$.
Now, for $n\leq\nu\leq\left\lfloor \frac{n}{\alpha}\right\rfloor \leq\frac{n}{\alpha}$,
we have $n-\nu\alpha\geq0$ and thus
\begin{align*}
\sum_{n\leq\nu\leq\left\lfloor n/\alpha\right\rfloor }\left(2^{\nu\left(1-\alpha\right)}\cdot2^{\left(\tau s-\tau M_{1}\right)\left(\nu-n\right)}\cdot2^{\left(\sigma-\tau K\right)\left(n-\nu\alpha\right)_{+}}\right) & \leq\sum_{n\leq\nu\leq\left\lfloor n/\alpha\right\rfloor }\left(2^{\nu\left(1-\alpha\right)}\cdot2^{\left(\tau s-\tau M_{1}\right)\left(\nu-n\right)}\cdot2^{-\left(n-\nu\alpha\right)}\right)\\
 & =2^{n\left(\tau M_{1}-\tau s-1\right)}\cdot\sum_{n\leq\nu\leq\left\lfloor n/\alpha\right\rfloor }2^{\nu\left(1+\tau s-\tau M_{1}\right)}\\
 & \leq2^{n\left(\tau M_{1}-\tau s-1\right)}\cdot\sum_{\nu=n}^{\infty}2^{\nu\left(1+\tau s-\tau M_{1}\right)}\\
\left({\scriptstyle \text{eq. }\eqref{eqGeometricSumNegativeExponent}\text{ and }1+\tau s-\tau M_{1}<0\text{ by the assump. of Lem. }\ref{lem:MainShearletLemma}}\right) & \leq\frac{1}{1-2^{1+\tau s-\tau M_{1}}}\cdot2^{n\left(\tau M_{1}-\tau s-1\right)}\cdot2^{n\left(1+\tau s-\tau M_{1}\right)}\\
\left({\scriptstyle \text{since }1+\tau s-\tau M_{1}\leq-c\tau\text{ by the assump. of Lem. }\ref{lem:MainShearletLemma}}\right) & \leq\frac{1}{1-2^{-c\tau}}.
\end{align*}
Finally, for the second part of the series, we have
\begin{align*}
\sum_{\nu>\left\lfloor n/\alpha\right\rfloor }\left(2^{\nu\left(1-\alpha\right)}\cdot2^{\left(\tau s-\tau M_{1}\right)\left(\nu-n\right)}\cdot2^{\left(\sigma-\tau K\right)\left(n-\nu\alpha\right)_{+}}\right) & \leq2^{n\left(\tau M_{1}-\tau s\right)}\cdot\sum_{\nu=1+\left\lfloor n/\alpha\right\rfloor }^{\infty}2^{\nu\left(1-\alpha+\tau s-\tau M_{1}\right)}\\
\left({\scriptstyle \text{eq. }\eqref{eqGeometricSumNegativeExponent}\text{ and }1-\alpha+\tau s-\tau M_{1}<0\text{ by the assump. of Lem. }\ref{lem:MainShearletLemma}}\right) & \leq\frac{1}{1-2^{1-\alpha+\tau s-\tau M_{1}}}\cdot2^{n\left(\tau M_{1}-\tau s\right)}\cdot2^{\left(1+\left\lfloor n/\alpha\right\rfloor \right)\cdot\left(1-\alpha+\tau s-\tau M_{1}\right)}\\
\left({\scriptstyle \text{since }1+\left\lfloor \frac{n}{\alpha}\right\rfloor \geq\frac{n}{\alpha}\text{ and }1-\alpha+\tau s-\tau M_{1}\leq-c\tau<0\text{ by assump. of Lem. }\ref{lem:MainShearletLemma}}\right) & \leq\frac{1}{1-2^{-c\tau}}\cdot2^{n\left(\tau M_{1}-\tau s\right)}\cdot2^{\frac{n}{\alpha}\cdot\left(1-\alpha+\tau s-\tau M_{1}\right)}\\
 & =\frac{1}{1-2^{-c\tau}}\cdot2^{n\cdot\frac{1-\alpha}{\alpha}\cdot\left(1+\tau s-\tau M_{1}\right)}\\
\left({\scriptstyle \text{since }1+\tau s-\tau M_{1}\leq0\text{ by the assump. of Lem. }\ref{lem:MainShearletLemma}}\right) & \leq\frac{1}{1-2^{-c\tau}}.
\end{align*}
All in all, the preceding three displayed equations show for $\alpha>0$
that 
\begin{equation}
\sup_{i=\left(n,m,\varepsilon,\delta\right)\in I^{\left(\ell_{1}\right)}}\:\sum_{\substack{j=\left(\nu,\mu,e,d\right)\in I^{\left(\ell_{2}\right)}\\
\text{s.t. }n<\nu\text{ and Case 3 holds}
}
}M_{j,i}^{\left(0\right)}\leq\frac{12\cdot C_{10}}{1-2^{-c\tau}},\label{eq:UpperConeRightConeCase3.2NuSum}
\end{equation}
and in view of equation \eqref{eq:UpperConeRightConeCase3.2NuSumAlpha0},
this estimate also holds in case of $\alpha=0$.

\medskip{}

Overall, our considerations in this subsection have established the
bound 
\[
\sup_{i\in I^{\left(\ell_{1}\right)}}\,\sum_{j\in I^{\left(\ell_{2}\right)}}M_{j,i}^{\left(0\right)}\leq C_{1}=:C_{0}^{\left(2\right)}\quad\text{ if }\alpha=1,
\]
cf.\@ equation \eqref{eq:UpperConeRightConeAlpha1FinalEstimate}.
Furthermore, in case of $\alpha\in\left[0,1\right)$, we have shown
\begin{align*}
\sup_{i\in I^{\left(\ell_{1}\right)}}\sum_{j\in I^{\left(\ell_{2}\right)}}M_{j,i}^{\left(0\right)} & \leq\sup_{i\in I^{\left(\ell_{1}\right)}}\left[\sum_{\substack{j\in I^{\left(\ell_{2}\right)}\\
\text{s.t. Case 1 holds}
}
}\!\!\!M_{j,i}^{\left(0\right)}+\sum_{\substack{j\in I^{\left(\ell_{2}\right)}\\
\text{s.t. Case 2 holds}
}
}\!\!\!M_{j,i}^{\left(0\right)}+\sum_{\substack{j\in I^{\left(\ell_{2}\right)}\\
\text{s.t. Case 3 holds}
}
}\!\!\!M_{j,i}^{\left(0\right)}\right]\\
\left({\scriptstyle \text{eqs. }\eqref{eq:UpperConeRightConeSmallNSumOverJ},\eqref{eq:UpperConeRightConeSummationCaseSumOverJ},\eqref{eq:UpperConeRightConeLastCaseNuSmallerThanNSum},\eqref{eq:UpperConeRightConeCase3.2NuSum}}\right) & \leq\frac{6C_{3}}{1-2^{-\tau c}}+\frac{2C_{8}}{1-2^{-\tau c}}+\left(6C_{10}\cdot\left(1+\frac{1}{\left(1-\alpha\right)\tau c}\right)+\frac{12\cdot C_{10}}{1-2^{-c\tau}}\right)\\
 & =:C_{0}^{\left(2\right)}.
\end{align*}
Note that the constant $C_{0}^{\left(2\right)}$ has a different value
depending on whether $\alpha=1$ or $\alpha<1$.

Likewise, we have shown
\[
\sup_{j\in I^{\left(\ell_{2}\right)}}\,\sum_{i\in I^{\left(\ell_{1}\right)}}M_{j,i}^{\left(0\right)}\leq C_{1}=:C_{0}^{\left(3\right)}\quad\text{if }\alpha=1,
\]
see again equation \eqref{eq:UpperConeRightConeAlpha1FinalEstimate}.
In case of $\alpha\in\left[0,1\right)$, we have also shown
\begin{align*}
\sup_{j\in I^{\left(\ell_{2}\right)}}\sum_{i\in I^{\left(\ell_{1}\right)}}M_{j,i}^{\left(0\right)} & \leq\sup_{j\in I^{\left(\ell_{2}\right)}}\left[\sum_{\substack{i\in I^{\left(\ell_{1}\right)}\\
\text{s.t. Case 1 holds}
}
}\!\!\!M_{j,i}^{\left(0\right)}+\sum_{\substack{i\in I^{\left(\ell_{1}\right)}\\
\text{s.t. Case 2 holds}
}
}\!\!\!M_{j,i}^{\left(0\right)}+\sum_{\substack{i\in I^{\left(\ell_{1}\right)}\\
\text{s.t. Case 3 holds}
}
}\!\!\!M_{j,i}^{\left(0\right)}\right]\\
\left({\scriptstyle \text{eqs. }\eqref{eq:UpperConeRightConeSmallNSumOverI},\eqref{eq:UpperConeRightConeSummationCaseSumOverI},\eqref{eq:UpperConeRightConeCase3.2NSum},\eqref{eq:UpperConeRightConeLastCaseNLargerThanNuSum}}\right) & \leq\frac{68\cdot C_{3}}{1-\alpha}+\frac{2C_{5}}{1-2^{-\tau c}}+\left(24\cdot C_{10}+\frac{6C_{10}}{1-2^{-\tau c}}\right)\\
 & =:C_{0}^{\left(3\right)}.
\end{align*}
Note as above that $C_{0}^{\left(3\right)}$ has a different value
depending on whether $\alpha=1$ or $\alpha<1$.

Finally, we observe that $\left(C_{0}^{\left(2\right)}\right)^{1/\tau}$
and $\left(C_{0}^{\left(3\right)}\right)^{1/\tau}$ can be estimated
solely in terms of $\alpha,\tau_{0},\omega,c,K,H,M_{1},M_{2}$: Indeed,
for arbitrary $C\geq0$, we have because of $\tau\geq\tau_{0}$ that
\[
C^{1/\tau}\leq\left[\max\left\{ 1,C\right\} \right]^{1/\tau}\leq\left[\max\left\{ 1,C\right\} \right]^{1/\tau_{0}}=\max\left\{ 1,C^{1/\tau_{0}}\right\} .
\]
Thus, if a constant $C\geq0$ can be bounded only in terms of $\alpha,\tau_{0},\omega,c,K,H,M_{1},M_{2}$,
then so can $C^{1/\tau}$. In particular, $\left(\frac{1}{\tau c}\right)^{1/\tau}\leq\max\left\{ 1,\left(\tau c\right)^{-1/\tau_{0}}\right\} \leq\max\left\{ 1,\left(\tau_{0}c\right)^{-1/\tau_{0}}\right\} =:\Omega_{1}$.
Furthermore, using again that $\tau\geq\tau_{0}$, we get $\ell^{\tau_{0}}\hookrightarrow\ell^{\tau}$,
where the embedding does not increase the norm. Hence,
\[
\left(\frac{1}{1-2^{-\tau c}}\right)^{1/\tau}=\left(\sum_{n=0}^{\infty}2^{-\tau cn}\right)^{1/\tau}\leq\left(\sum_{n=0}^{\infty}2^{-\tau_{0}cn}\right)^{1/\tau_{0}}=\left(\frac{1}{1-2^{-\tau_{0}c}}\right)^{1/\tau_{0}}=:\Omega_{2}.
\]
Similarly, since $\frac{1}{1-\alpha}\geq1$ for $\alpha<1$, we have
$\left(\frac{1}{1-\alpha}\right)^{1/\tau}\leq\left(\frac{1}{1-\alpha}\right)^{1/\tau_{0}}=:\Omega_{3}$.

Finally, using once more that $\tau\geq\tau_{0}$, we see 
\[
\left(\sum_{i=1}^{n}a_{i}\right)^{1/\tau}\leq\left(n\cdot\max\left\{ a_{1},\dots,a_{n}\right\} \right)^{1/\tau}\leq n^{1/\tau_{0}}\cdot\max\left\{ a_{1}^{1/\tau},\dots,a_{n}^{1/\tau}\right\} 
\]
for arbitrary $a_{1},\dots,a_{n}\geq0$. Thus, if $a_{1}^{1/\tau},\dots,a_{n}^{1/\tau}$
can be estimated only in terms of $\alpha,\tau_{0},\omega,c,K,H,M_{1},M_{2}$,
then so can $\left(\sum_{i=1}^{n}a_{i}\right)^{1/\tau}$. All in all,
we have shown that the set of all expressions/constants $C\geq0$
for which $C^{1/\tau}$ can be estimated only in terms of $\alpha,\tau_{0},\omega,c,K,H,M_{1},M_{2}$
is closed under multiplication and addition. Hence, it suffices to
show $C_{i}^{1/\tau}\leq L_{i}$ for $i\in\underline{10}$, where
$L_{i}$ only depends on $\alpha,\tau_{0},\omega,c,K,H,M_{1},M_{2}$.

To this end, recall that $\frac{\sigma}{\tau}\leq\omega$ and that
$M_{3}=\max\left\{ M_{1},M_{2}\right\} $ only depends on $M_{1},M_{2}$.
Hence, recalling that the constants $C_{2},\dots,C_{10}$ are only
needed in case of $\alpha\in\left[0,1\right)$, we get
\begin{align*}
C_{1}^{1/\tau} & =\left[3\cdot8^{\sigma}\cdot6^{\tau\left(2+K+M_{3}\right)}\cdot\frac{2}{1-2^{-\tau c}}\right]^{1/\tau}\leq6^{1/\tau_{0}}\cdot\Omega_{2}\cdot8^{\frac{\sigma}{\tau}}\cdot6^{2+K+M_{3}}\leq6^{1/\tau_{0}}\cdot\Omega_{2}\cdot8^{\omega}\cdot6^{2+K+M_{3}}=:L_{1},\\
C_{2}^{1/\tau} & =6^{2\sigma/\tau}\cdot2^{\frac{6\sigma/\tau}{1-\alpha}}\leq6^{2\omega}\cdot2^{\frac{6\omega}{1-\alpha}}=:L_{2},\\
C_{3}^{1/\tau} & =C_{2}^{1/\tau}\cdot2\cdot3^{2+K+M_{3}}\cdot4^{M_{3}}\cdot2^{\frac{3M_{1}}{1-\alpha}}\leq L_{2}\cdot2\cdot3^{2+K+M_{3}}\cdot4^{M_{3}}\cdot2^{\frac{3M_{1}}{1-\alpha}}=:L_{3},\\
C_{4}^{1/\tau} & =6^{\sigma/\tau}\cdot3^{4+K+5M_{3}}\leq6^{\omega}\cdot3^{4+K+5M_{3}}=:L_{4},\\
C_{5}^{1/\tau} & =C_{4}^{1/\tau}\cdot2^{\frac{7}{\tau}+2\frac{\sigma}{\tau}+2}\cdot10^{\frac{3}{\tau}+\frac{\sigma}{\tau}}\leq L_{4}\cdot2^{\frac{7}{\tau_{0}}+2\omega+2}\cdot10^{\frac{3}{\tau_{0}}+\omega}=:L_{5},\\
C_{6}^{1/\tau} & =6^{\sigma/\tau}\cdot3^{4+K+5M_{3}}\leq6^{\omega}\cdot3^{4+K+5M_{3}}=:L_{6},\\
C_{7}^{1/\tau} & =C_{6}^{1/\tau}\cdot2^{\frac{5}{\tau}+2+2\frac{\sigma}{\tau}}\cdot10^{\frac{3}{\tau}+\frac{\sigma}{\tau}}\leq L_{6}\cdot2^{\frac{5}{\tau_{0}}+2+2\omega}\cdot10^{\frac{3}{\tau_{0}}+\omega}=:L_{7},\\
C_{8}^{1/\tau} & =2^{\frac{2}{\tau}+M_{3}}C_{7}^{1/\tau}\leq L_{7}\cdot2^{\frac{2}{\tau_{0}}+M_{3}}=:L_{8},\\
C_{9}^{1/\tau} & =6^{\sigma/\tau}\cdot3^{3+K+3M_{3}}\cdot4^{K}\leq6^{\omega}\cdot3^{3+K+3M_{3}}\cdot4^{K}=:L_{9},\\
C_{10}^{1/\tau} & =6^{\sigma/\tau}\cdot C_{9}^{1/\tau}\leq6^{\omega}\cdot L_{9}=:L_{10},
\end{align*}
where the constants $L_{1},\dots,L_{10}$ only depend on $\alpha,\tau_{0},\omega,c,K,H,M_{1},M_{2}$.
Taken together, these considerations easily imply $C_{0}^{\left(2\right)}\leq\left[C_{00}^{\left(2\right)}\right]^{\tau}$
and $C_{0}^{\left(3\right)}\leq\left[C_{00}^{\left(3\right)}\right]^{\tau}$,
where $C_{00}^{\left(2\right)}$ and $C_{00}^{\left(3\right)}$ only
depend on $\alpha,\tau_{0},\omega,c,K,H,M_{1},M_{2}$.

Likewise, the constant $C_{0}^{\left(1\right)}=2^{19+7\sigma+\tau\left(5+2K+2M_{3}\right)}/\left(1\!-\!2^{-\tau c}\right)$
from Subsection \ref{subsec:BothRightCone} can be estimated by
\[
\left[C_{0}^{\left(1\right)}\right]^{1/\tau}\leq\Omega_{2}\cdot2^{\frac{19}{\tau}+7\frac{\sigma}{\tau}+5+2K+2M_{3}}\leq\Omega_{2}\cdot2^{\frac{19}{\tau_{0}}+7\omega+5+2K+2M_{3}}=:C_{00}^{\left(1\right)},
\]
where $C_{00}^{\left(1\right)}$ only depends on $\alpha,\tau_{0},\omega,c,K,H,M_{1},M_{2}$.

\subsection{We have \texorpdfstring{$\ell_{1}=\left(1,0\right)$}{ℓ₁=(1,0)}
and \texorpdfstring{$\ell_{2}=\left(1,1\right)$}{ℓ₂=(1,1)}}

\label{subsec:RightConeUpperCone}Geometrically, this case means that
$i$ belongs to the right cone, while $j$ belongs to the upper cone.
In this case, we have $i=\left(n,m,1,0\right)$ and $j=\left(\nu,\mu,1,1\right)$
and hence—because of $R=R^{-1}$—that
\[
T_{\nu,\mu,1,1}^{-1}T_{n,m,1,0}=\left(R\cdot A_{\nu,\mu,1}^{\left(\alpha\right)}\right)^{-1}A_{n,m,1}^{\left(\alpha\right)}=\left(A_{\nu,\mu,1}^{\left(\alpha\right)}\right)^{-1}\cdot R^{-1}\cdot A_{n,m,1}^{\left(\alpha\right)}=\left(A_{\nu,\mu,1}^{\left(\alpha\right)}\right)^{-1}\cdot\left(RA_{n,m,1}^{\left(\alpha\right)}\right)=T_{\nu,\mu,1,0}^{-1}T_{n,m,1,1}.
\]
Since furthermore $\varrho_{(\nu,\mu,1,1)}=\varrho_{(\nu,\mu,1,0)}=\varrho$
and since the weight $w_{\left(n,m,\varepsilon,\delta\right)}=2^{n}$
is independent of $\varepsilon,\delta$, we get
\begin{align*}
 & M_{\left(\nu,\mu,1,1\right),\left(n,m,1,0\right)}^{\left(0\right)}\\
 & =\left(\frac{w_{\nu,\mu,1,1}^{s}}{w_{n,m,1,0}^{s}}\right)^{\tau}\left(1\!+\!\left\Vert T_{\nu,\mu,1,1}^{-1}T_{n,m,1,0}\right\Vert \right)^{\sigma}\left(\left|\det T_{n,m,1,0}\right|^{-1}\int_{S_{\left(n,m,1,0\right)}^{\left(\alpha\right)}}\varrho_{(\nu,\mu,1,1)}\left(T_{\nu,\mu,1,1}^{-1}\xi\right)\d\xi\right)^{\!\tau}\\
\left({\scriptstyle \zeta=T_{n,m,1,0}^{-1}\xi}\right) & =\left(\frac{w_{\nu,\mu,1,0}^{s}}{w_{n,m,1,1}^{s}}\right)^{\tau}\left(1\!+\!\left\Vert T_{\nu,\mu,1,0}^{-1}T_{n,m,1,1}\right\Vert \right)^{\sigma}\left(\int_{Q}\varrho_{(\nu,\mu,1,1)}\left(T_{\left(\nu,\mu,1,1\right)}^{-1}T_{\left(n,m,1,0\right)}\zeta\right)\d\zeta\right)^{\tau}\\
 & =\left(\frac{w_{\nu,\mu,1,0}^{s}}{w_{n,m,1,1}^{s}}\right)^{\tau}\left(1\!+\!\left\Vert T_{\nu,\mu,1,0}^{-1}T_{n,m,1,1}\right\Vert \right)^{\sigma}\left(\int_{Q}\varrho_{(\nu,\mu,1,0)}\left(T_{\left(\nu,\mu,1,0\right)}^{-1}T_{\left(n,m,1,1\right)}\zeta\right)\d\zeta\right)^{\tau}\\
\left({\scriptstyle \xi=T_{n,m,1,1}\zeta}\right) & =\left(\frac{w_{\nu,\mu,1,0}^{s}}{w_{n,m,1,1}^{s}}\right)^{\tau}\left(1\!+\!\left\Vert T_{\nu,\mu,1,0}^{-1}T_{n,m,1,1}\right\Vert \right)^{\sigma}\left(\left|\det T_{n,m,1,1}\right|^{-1}\int_{S_{\left(n,m,1,1\right)}^{\left(\alpha\right)}}\varrho_{(\nu,\mu,1,0)}\left(T_{\nu,\mu,1,0}^{-1}\xi\right)\d\xi\right)^{\!\tau}\\
 & =M_{\left(\nu,\mu,1,0\right),\left(n,m,1,1\right)}^{\left(0\right)}.
\end{align*}

But Subsection \ref{subsec:UpperConeRightCone} shows under the assumptions
of Lemma \ref{lem:MainShearletLemma} that
\begin{align*}
\sup_{i\in I^{\left(1,0\right)}}\,\sum_{j\in I^{\left(1,1\right)}}M_{j,i}^{\left(0\right)} & =\sup_{n\in\N_{0}}\sup_{\left|m\right|\leq G_{n}}\sum_{\nu\in\N_{0}}\sum_{\left|\mu\right|\leq G_{\nu}}M_{\left(\nu,\mu,1,1\right),\left(n,m,1,0\right)}^{\left(0\right)}\\
 & =\sup_{n\in\N_{0}}\sup_{\left|m\right|\leq G_{n}}\sum_{\nu\in\N_{0}}\sum_{\left|\mu\right|\leq G_{\nu}}M_{\left(\nu,\mu,1,0\right),\left(n,m,1,1\right)}^{\left(0\right)}=\sup_{i\in I^{\left(1,1\right)}}\,\sum_{j\in I^{\left(1,0\right)}}M_{j,i}^{\left(0\right)}\leq C_{0}^{\left(2\right)}\leq\left[C_{00}^{\left(2\right)}\right]^{\tau}
\end{align*}
and
\begin{align*}
\sup_{j\in I^{\left(1,1\right)}}\,\sum_{i\in I^{\left(1,0\right)}}M_{j,i}^{\left(0\right)} & =\sup_{\nu\in\N_{0}}\sup_{\left|\mu\right|\leq G_{\nu}}\sum_{n\in\N_{0}}\sum_{\left|m\right|\leq G_{n}}M_{\left(\nu,\mu,1,1\right),\left(n,m,1,0\right)}^{\left(0\right)}\\
 & =\sup_{\nu\in\N_{0}}\sup_{\left|\mu\right|\leq G_{\nu}}\sum_{n\in\N_{0}}\sum_{\left|m\right|\leq G_{n}}M_{\left(\nu,\mu,1,0\right),\left(n,m,1,1\right)}^{\left(0\right)}=\sup_{j\in I^{\left(1,0\right)}}\,\sum_{i\in I^{\left(1,1\right)}}M_{j,i}^{\left(0\right)}\leq C_{0}^{\left(3\right)}\leq\left[C_{00}^{\left(3\right)}\right]^{\tau}.
\end{align*}

\subsection{We have \texorpdfstring{$\ell_{1}=\ell_{2}=\left(1,1\right)$}{ℓ₁=ℓ₂=(1,1)}}

\label{subsec:BothUpperCone}Geometrically, this case means that both
$i$ and $j$ belong to the upper cone. In this case, we have $i=\left(n,m,1,1\right)$
and $j=\left(\nu,\mu,1,1\right)$ and hence
\[
T_{\nu,\mu,1,1}^{-1}T_{n,m,1,1}=\left(R\cdot A_{\nu,\mu,1}^{\left(\alpha\right)}\right)^{-1}\cdot\left(R\cdot A_{n,m,1}^{\left(\alpha\right)}\right)=\left(A_{\nu,\mu,1}^{\left(\alpha\right)}\right)^{-1}\cdot A_{n,m,1}^{\left(\alpha\right)}=T_{\nu,\mu,1,0}^{-1}T_{n,m,1,0},
\]
as well as $\varrho_{\nu,\mu,1,1}=\varrho_{\nu,\mu,1,0}=\varrho$.
This implies precisely as in the preceding subsection that
\[
M_{\left(\nu,\mu,1,1\right),\left(n,m,1,1\right)}^{\left(0\right)}=M_{\left(\nu,\mu,1,0\right),\left(n,m,1,0\right)}^{\left(0\right)}.
\]

Then, we use that Subsection \ref{subsec:BothRightCone} shows under
the assumptions of Lemma \ref{lem:MainShearletLemma} that 
\[
\sup_{i\in I^{\left(1,1\right)}}\,\sum_{j\in I^{\left(1,1\right)}}M_{j,i}^{\left(0\right)}=\sup_{i\in I^{\left(1,0\right)}}\,\sum_{j\in I^{\left(1,0\right)}}M_{j,i}^{\left(0\right)}\leq C_{0}^{\left(1\right)}\leq\left[C_{00}^{\left(1\right)}\right]^{\tau}
\]
and
\[
\sup_{j\in I^{\left(1,1\right)}}\,\sum_{i\in I^{\left(1,1\right)}}M_{j,i}^{\left(0\right)}=\sup_{j\in I^{\left(1,0\right)}}\,\sum_{i\in I^{\left(1,0\right)}}M_{j,i}^{\left(0\right)}\leq C_{0}^{\left(1\right)}\leq\left[C_{00}^{\left(1\right)}\right]^{\tau}.
\]

\subsection{We have \texorpdfstring{$\ell_{1},\ell_{2}\in\left\{ -1\right\} \times\left\{ 0,1\right\} $}{ℓ₁,ℓ₂∈\{-1\}x\{0,1\}}}

This case comprises all the cases considered in Subsections \ref{subsec:BothRightCone}–\ref{subsec:BothUpperCone},
with the only difference that geometrically the lower and left cones
are considered instead of the upper and right cones. In this case,
we have $i=\left(n,m,-1,\delta\right)$ and $j=\left(\nu,\mu,-1,d\right)$
and hence 
\[
T_{\nu,\mu,-1,d}^{-1}T_{n,m,-1,\delta}=\left(-1\right)\cdot\left(-1\right)\cdot T_{\nu,\mu,1,d}^{-1}T_{n,m,1,\delta}=T_{\nu,\mu,1,d}^{-1}T_{n,m,1,\delta},
\]
as well as $\varrho_{\nu,\mu,-1,d}=\varrho_{\nu,\mu,1,d}=\varrho$.
As in Subsection \ref{subsec:RightConeUpperCone}, this implies that
\[
M_{\left(\nu,\mu,-1,d\right),\left(n,m,-1,\delta\right)}^{\left(0\right)}=M_{\left(\nu,\mu,1,d\right),\left(n,m,1,\delta\right)}^{\left(0\right)}.
\]
Hence, depending on $\delta$ and $d$ we get the same estimates as
in Subsections \ref{subsec:BothRightCone}–\ref{subsec:BothUpperCone}. 

\subsection{We have \texorpdfstring{$\ell_{1}\in\left\{ -1\right\} \times\left\{ 0,1\right\} $
and $\ell_{2}\in\left\{ 1\right\} \times\left\{ 0,1\right\} $}{ℓ₁∈\{-1\}x\{0,1\} and ℓ₂∈\{1\}x\{0,1\}}}

Geometrically this means that $i$ belongs to the left or lower cone
and $j$ belongs to the right or upper cone. In this case, we have
$i=\left(n,m,-1,\delta\right)$ and $j=\left(\nu,\mu,1,d\right)$
and hence 
\[
T_{\nu,\mu,1,d}^{-1}T_{n,m,-1,\delta}=(-1)\cdot T_{\nu,\mu,1,d}^{-1}T_{n,m,1,\delta}.
\]
Consequently, we get $\left\Vert T_{\nu,\mu,1,d}^{-1}T_{n,m,-1,\delta}\right\Vert =\left\Vert T_{\nu,\mu,1,d}^{-1}T_{n,m,1,\delta}\right\Vert $.
Now, since we have $\varrho\left(-\xi\right)=\varrho\left(\xi\right)$
for all $\xi\in\R^{2}$ and $\varrho_{\left(\nu,\mu,1,d\right)}=\varrho$,
we finally see
\[
\int_{Q}\varrho_{\nu,\mu,1,d}\left(T_{\nu,\mu,1,d}^{-1}T_{n,m,-1,\delta}\zeta\right)\d\zeta=\int_{Q}\varrho_{\nu,\mu,1,d}\left(-T_{\nu,\mu,1,d}^{-1}T_{n,m,1,\delta}\zeta\right)\d\zeta=\int_{Q}\varrho_{\nu,\mu,1,d}\left(T_{\nu,\mu,1,d}^{-1}T_{n,m,1,\delta}\zeta\right)\d\zeta.
\]
As before this implies $M_{\left(\nu,\mu,1,d\right),\left(n,m,-1,\delta\right)}^{\left(0\right)}=M_{\left(\nu,\mu,1,d\right),\left(n,m,1,\delta\right)}^{\left(0\right)}$
and depending on $\delta$ and $d$ we get the same estimates as in
Subsections \ref{subsec:BothRightCone}–\ref{subsec:BothUpperCone}.

\subsection{We have \texorpdfstring{$\ell_{1}\in\left\{ 1\right\} \times\left\{ 0,1\right\} $
and $\ell_{2}\in\left\{ -1\right\} \times\left\{ 0,1\right\} $}{ℓ₁∈\{1\}x\{0,1\} and ℓ₂∈\{-1\}x\{0,1\}}}

Geometrically this means that $i$ belongs to the right or upper cone
and $j$ belongs to the left or lower cone. In this case, we have
$i=\left(n,m,1,\delta\right)$ and $j=\left(\nu,\mu,-1,d\right)$
and hence 
\[
T_{\nu,\mu,-1,d}^{-1}T_{n,m,1,\delta}=\left(-1\right)\cdot T_{\nu,\mu,1,d}^{-1}T_{n,m,1,\delta}.
\]
Consequently, $\left\Vert T_{\nu,\mu,-1,d}^{-1}T_{n,m,1,\delta}\right\Vert =\left\Vert T_{\nu,\mu,1,d}^{-1}T_{n,m,1,\delta}\right\Vert $.
Now, since $\varrho_{\nu,\mu,-1,d}=\varrho_{\nu,\mu,1,d}=\varrho$
and since $\varrho\left(-\xi\right)=\varrho\left(\xi\right)$ for
all $\xi\in\R^{2}$, we get
\[
\int_{Q}\varrho_{\nu,\mu,-1,d)}\left(T_{\nu,\mu,-1,d}^{-1}T_{n,m,1,\delta}\zeta\right)\d\zeta=\int_{Q}\varrho_{\nu,\mu,1,d}\left(-T_{\nu,\mu,1,d}^{-1}T_{n,m,1,\delta}\zeta\right)\d\zeta=\int_{Q}\varrho_{\nu,\mu,1,d}\left(T_{\nu,\mu,1,d}^{-1}T_{n,m,1,\delta}\zeta\right)\d\zeta.
\]
As before this implies $M_{\left(\nu,\mu,-1,d\right),\left(n,m,1,\delta\right)}^{\left(0\right)}=M_{\left(\nu,\mu,1,d\right),\left(n,m,1,\delta\right)}^{\left(0\right)}$
and depending on $\delta$ and $d$ we get the same estimates as in
Subsections \ref{subsec:BothRightCone}–\ref{subsec:BothUpperCone}. 

\subsection{We have \texorpdfstring{$\ell_{1}=0$ and $\ell_{2}\in\left\{ \pm1\right\} \times\left\{ 0,1\right\} $}{ℓ₁=0 and ℓ₂∈\{1,-1\}x\{0,1\}}}

\label{subsec:LowPassAndRemainingStuff}In this case, we have for
$j=\left(\nu,\mu,e,d\right)\in I^{\left(\ell_{2}\right)}\subset I_{0}$
and $i\in I^{\left(\ell_{1}\right)}=I^{\left(0\right)}=\left\{ 0\right\} $
that
\[
\left\Vert T_{j}^{-1}T_{i}\right\Vert =\left\Vert T_{j}^{-1}\right\Vert =\left\Vert e\cdot\left(A_{\nu,\mu,1}^{\left(\alpha\right)}\right)^{-1}\cdot R^{-d}\right\Vert =\left\Vert \left(A_{\nu,\mu,1}^{\left(\alpha\right)}\right)^{-1}\right\Vert =\left\Vert \left(\begin{matrix}2^{-\nu} & 0\\
-2^{-\nu}\mu & 2^{-\alpha\nu}
\end{matrix}\right)\right\Vert \leq3\cdot\max\left\{ 1,2^{-\nu}\left|\mu\right|\right\} .
\]
But because of $\left|\mu\right|\leq G_{\nu}=\left\lceil \smash{2^{\left(1-\alpha\right)\nu}}\right\rceil \leq\left\lceil \smash{2^{\nu}}\right\rceil =2^{\nu}$,
we have $2^{-\nu}\left|\mu\right|\leq1$, which yields $\left\Vert T_{j}^{-1}T_{i}\right\Vert \leq3$.

Next, recall that $S_{0}^{\left(\alpha\right)}=Q_{0}'=\left(-1,1\right)^{2}$.
Because of $-\left(-1,1\right)^{2}=\left(-1,1\right)^{2}$, this implies
in case of $d=0$ that
\begin{align*}
T_{j}^{-1}S_{0}^{\left(\alpha\right)}=\left(A_{\nu,\mu,1}^{\left(\alpha\right)}\right)^{-1}\left(-1,1\right)^{2} & =\left\{ \left(\begin{matrix}2^{-\nu} & 0\\
-2^{-\nu}\mu & 2^{-\alpha\nu}
\end{matrix}\right)\left(\begin{matrix}\xi_{1}\\
\xi_{2}
\end{matrix}\right)\with\xi_{1},\xi_{2}\in\left(-1,1\right)\right\} \\
 & =\left\{ \left(\eta_{1},\eta_{2}\right)\in\R^{2}\with\eta_{1}\in\left(-2^{-\nu},2^{-\nu}\right),\,\eta_{2}\in\left(-2^{-\alpha\nu},2^{-\alpha\nu}\right)-\mu\eta_{1}\right\} .
\end{align*}
Likewise, since $R=R^{-1}$ and since $R\left(-1,1\right)^{2}=\left(-1,1\right)^{2}$,
we also get in case of $d=1$ that
\begin{align*}
T_{j}^{-1}S_{0}^{\left(\alpha\right)} & =\left(A_{\nu,\mu,1}^{\left(\alpha\right)}\right)^{-1}R\left(-1,1\right)^{2}=\left(A_{\nu,\mu,1}^{\left(\alpha\right)}\right)^{-1}\left(-1,1\right)^{2}\\
 & =\left\{ \left(\eta_{1},\eta_{2}\right)\in\R^{2}\with\eta_{1}\in\left(-2^{-\nu},2^{-\nu}\right),\,\eta_{2}\in\left(-2^{-\alpha\nu},2^{-\alpha\nu}\right)-\mu\eta_{1}\right\} .
\end{align*}

Consequently, we get in all cases that 
\begin{align*}
\left|\det T_{0}\right|^{-1}\cdot\int_{S_{0}^{\left(\alpha\right)}}\varrho_{j}\left(T_{j}^{-1}\xi\right)\d\xi & =\left|\det T_{j}\right|\cdot\int_{T_{j}^{-1}S_{0}^{\left(\alpha\right)}}\varrho\left(\eta\right)\d\eta\\
 & \leq2^{\left(1+\alpha\right)\nu}\cdot\int_{-2^{-\nu}}^{2^{-\nu}}\theta_{1}\left(\eta_{1}\right)\cdot\int_{-\mu\eta_{1}-2^{-\alpha\nu}}^{-\mu\eta_{1}+2^{-\alpha\nu}}\left(1+\left|\eta_{2}\right|\right)^{-K}\d\eta_{2}\d\eta_{1}.
\end{align*}
But for $\left|\eta_{1}\right|\leq2^{-\nu}$, we have $\theta_{1}\left(\eta_{1}\right)\leq\left|\eta_{1}\right|^{M_{1}}\leq2^{-M_{1}\nu}$,
so that
\[
\left|\det T_{0}\right|^{-1}\cdot\int_{S_{0}^{\left(\alpha\right)}}\varrho\left(T_{j}^{-1}\xi\right)\d\xi\leq2^{-M_{1}\nu}\cdot2^{\left(1+\alpha\right)\nu}\cdot\int_{-2^{-\nu}}^{2^{-\nu}}\int_{-\mu\eta_{1}-2^{-\alpha\nu}}^{-\mu\eta_{1}+2^{-\alpha\nu}}\d\eta_{2}\d\eta_{1}\leq4\cdot2^{-M_{1}\nu}.
\]
All in all, this implies
\[
M_{j,0}^{\left(0\right)}=\left(\frac{w_{j}^{s}}{w_{0}^{s}}\right)^{\tau}\cdot\left(1+\left\Vert T_{j}^{-1}T_{0}\right\Vert \right)^{\sigma}\cdot\left(\left|\det T_{0}\right|^{-1}\int_{S_{0}^{\left(\alpha\right)}}\varrho_{j}\left(T_{j}^{-1}\xi\right)\d\xi\right)^{\tau}\leq4^{\sigma}\cdot2^{\tau s\nu}\cdot4^{\tau}\cdot2^{-M_{1}\tau\nu},
\]
which yields
\begin{align*}
\sum_{j\in I^{\left(\ell_{2}\right)}}M_{j,0}^{\left(0\right)}= & \leq4^{\tau+\sigma}\cdot\sum_{\nu=0}^{\infty}\:\sum_{\left|\mu\right|\leq G_{\nu}}2^{\tau\nu\left(s-M_{1}\right)}\\
\left({\scriptstyle \text{since }G_{\nu}=\left\lceil \smash{2^{\left(1-\alpha\right)\nu}}\right\rceil \leq1+2^{\left(1-\alpha\right)\nu}\leq2\cdot2^{\left(1-\alpha\right)\nu}}\right) & \leq2^{3}4^{\tau+\sigma}\cdot\sum_{\nu=0}^{\infty}2^{\nu\left[\tau\left(s-M_{1}\right)+\left(1-\alpha\right)\right]}\\
 & \leq2^{3}4^{\tau+\sigma}\cdot\sum_{\nu=0}^{\infty}2^{-\nu\tau c}\\
 & \leq\frac{2^{3+2\tau+2\sigma}}{1-2^{-\tau c}}=:C_{0}^{\left(4\right)},
\end{align*}
since the assumptions of Lemma \ref{lem:MainShearletLemma} entail
$M_{1}\geq M_{1}^{\left(0\right)}+c\geq s+\frac{1}{\tau}+c\geq s+\frac{1-\alpha}{\tau}+c$.

Likewise, we get
\[
\sup_{j\in I^{\left(\ell_{2}\right)}}\,\sum_{i\in I^{\left(0\right)}}M_{j,i}^{\left(0\right)}=\sup_{j\in I^{\left(\ell_{2}\right)}}M_{j,0}^{\left(0\right)}\leq\sum_{j\in I^{\left(\ell_{2}\right)}}M_{j,0}^{\left(0\right)}\leq C_{0}^{\left(4\right)}.
\]
Finally, we see as at the end of Subsection \ref{subsec:UpperConeRightCone}
that 
\[
\left[C_{0}^{\left(4\right)}\right]^{1/\tau}\leq\left(\frac{1}{1-2^{-\tau_{0}c}}\right)^{1/\tau_{0}}\cdot2^{\frac{3}{\tau}+2+2\frac{\sigma}{\tau}}\leq\left(\frac{1}{1-2^{-\tau_{0}c}}\right)^{1/\tau_{0}}\cdot2^{\frac{3}{\tau_{0}}+2+2\omega}=:C_{00}^{\left(4\right)},
\]
where $C_{00}^{\left(4\right)}$ only depends on $\alpha,\tau_{0},\omega,c,K,H,M_{1},M_{2}$.

\subsection{We have \texorpdfstring{$\ell_{2}=0$ and $\ell_{1}\in\left\{ \pm1\right\} \times\left\{ 0,1\right\} $}{ℓ₂=0 and ℓ₁∈\{1,-1\}x\{0,1\}}}

\label{subsec:RemainingStuffAndLowPass}In this case, we have for
$i=\left(n,m,\varepsilon,\delta\right)\in I^{\left(\ell_{1}\right)}\subset I_{0}$
and $j\in I^{\left(\ell_{2}\right)}=I^{\left(0\right)}=\left\{ 0\right\} $
that
\[
1+\left\Vert T_{j}^{-1}T_{i}\right\Vert =1+\left\Vert T_{i}\right\Vert =1+\left\Vert \left(\begin{matrix}2^{n} & 0\\
2^{n\alpha}m & 2^{n\alpha}
\end{matrix}\right)\right\Vert \leq5\cdot2^{n},
\]
since $\left|2^{n\alpha}m\right|\leq2^{n\alpha}G_{n}\leq2^{n\alpha}\left(2^{n\left(1-\alpha\right)}+1\right)\leq2\cdot2^{n}$.

Furthermore, we note $\lambda_{\dimension}\left(Q\right)\leq18$,
since $Q=Q_{i}'=U_{\left(-1,1\right)}^{\left(3^{-1},3\right)}\subset\left(\frac{1}{3},3\right)\times\left(-3,3\right)$.
Thus,
\begin{equation}
\left|\det T_{i}\right|^{-1}\cdot\int_{S_{i}^{\left(\alpha\right)}}\varrho_{j}\left(T_{j}^{-1}\xi\right)\d\xi=\int_{Q}\varrho_{0}\left(T_{i}\eta\right)\d\eta\leq18\cdot\sup_{\eta\in Q}\varrho_{0}\left(T_{i}\eta\right).\label{eq:RemainingStuffAndLowPassIntegralEstimatedBySupremum}
\end{equation}
Now, we distinguish the cases $\delta=0$ and $\delta=1$:
\begin{enumerate}
\item For $\delta=0$, we have
\[
T_{i}\eta=\left(\begin{matrix}\varepsilon\cdot2^{n}\eta_{1}\\
\varepsilon\cdot\left(2^{n\alpha}m\eta_{1}+2^{n\alpha}\eta_{2}\right)
\end{matrix}\right)\quad\text{ for }\eta=\left(\eta_{1},\eta_{2}\right)\in\R^{2}.
\]
But for $\eta\in Q$, we have $\frac{1}{3}<\eta_{1}<3$ and hence
$2^{n}\eta_{1}\geq2^{n}/3$, so that we get
\[
\varrho_{0}\left(T_{i}\eta\right)\leq\left(1+\left|\varepsilon\cdot2^{n}\eta_{1}\right|\right)^{-H}\leq3^{H}\cdot2^{-Hn},
\]
which yields $\left|\det T_{i}\right|^{-1}\cdot\int_{S_{i}^{\left(\alpha\right)}}\varrho_{j}\left(T_{j}^{-1}\xi\right)\d\xi\leq18\cdot3^{H}\cdot2^{-Hn}$
by virtue of equation \eqref{eq:RemainingStuffAndLowPassIntegralEstimatedBySupremum}.
\item For $\delta=1$, we have
\[
T_{i}\eta=\left(\begin{matrix}\varepsilon\cdot\left(2^{n\alpha}m\eta_{1}+2^{n\alpha}\eta_{2}\right)\\
\varepsilon\cdot2^{n}\eta_{1}
\end{matrix}\right)\quad\text{ for }\eta=\left(\eta_{1},\eta_{2}\right)\in\R^{2}.
\]
Again, for $\eta\in Q$, we have $2^{n}\eta_{1}\geq2^{n}/3$ and hence
\[
\varrho_{0}\left(T_{i}\eta\right)\leq\left(1+\left|\varepsilon\cdot2^{n}\eta_{1}\right|\right)^{-H}\leq3^{H}\cdot2^{-Hn},
\]
which as above yields $\left|\det T_{i}\right|^{-1}\cdot\int_{S_{i}^{\left(\alpha\right)}}\varrho_{j}\left(T_{j}^{-1}\xi\right)\d\xi\leq18\cdot3^{H}\cdot2^{-Hn}$.
\end{enumerate}
In total, we get for each case the estimate
\begin{align*}
M_{0,i}^{\left(0\right)} & =\left(\frac{w_{0}^{s}}{w_{i}^{s}}\right)^{\tau}\cdot\left(1+\left\Vert T_{0}^{-1}T_{i}\right\Vert \right)^{\sigma}\cdot\left(\left|\det T_{i}\right|^{-1}\cdot\int_{S_{i}^{\left(\alpha\right)}}\varrho_{0}\left(T_{0}^{-1}\xi\right)\d\xi\right)^{\tau}\\
 & \leq2^{-s\tau n}\cdot5^{\sigma}\cdot2^{n\sigma}\cdot18^{\tau}\cdot3^{H\tau}\cdot2^{-\tau Hn}\\
 & \leq2^{3\sigma+5\tau+2H\tau}\cdot2^{n\tau\left(\frac{\sigma}{\tau}-s-H\right)}.
\end{align*}
Thus, we get on the one hand
\begin{align*}
\sum_{i\in I^{\left(\ell_{1}\right)}}M_{0,i}^{\left(0\right)} & \leq2^{3\sigma+5\tau+2H\tau}\cdot\sum_{n=0}^{\infty}\:\sum_{\left|m\right|\leq G_{n}}2^{n\tau\left(\frac{\sigma}{\tau}-s-H\right)}\\
\left({\scriptstyle \text{since }G_{n}=\left\lceil 2^{n\left(1-\alpha\right)}\right\rceil \leq1+2^{n\left(1-\alpha\right)}\leq2\cdot2^{n\left(1-\alpha\right)}}\right) & \leq2^{3+3\sigma+5\tau+2H\tau}\cdot\sum_{n=0}^{\infty}2^{n(1-\alpha)}2^{n\tau\left(\frac{\sigma}{\tau}-s-H\right)}\\
 & \leq2^{3+3\sigma+5\tau+2H\tau}\cdot\frac{1}{1-2^{-c\tau}}=:C_{0}^{\left(5\right)},
\end{align*}
since the assumptions of Lemma \ref{lem:MainShearletLemma} imply
$H\geq H_{0}+c=\frac{1-\alpha}{\tau}+\frac{\sigma}{\tau}-s+c$.

Likewise, the summation over $j$ yields
\[
\sup_{i\in I^{\left(\ell_{1}\right)}}\,\sum_{j\in I^{\left(0\right)}}M_{j,i}^{\left(0\right)}=\sup_{i\in I^{\left(\ell_{1}\right)}}M_{0,i}^{\left(0\right)}\leq\sum_{i\in I^{(\ell_{1})}}M_{0,i}^{\left(0\right)}\leq C_{0}^{\left(5\right)}.
\]
Finally, we get as at the end of Subsection \ref{subsec:UpperConeRightCone}
that 
\[
\left[C_{0}^{\left(5\right)}\right]^{1/\tau}\leq\left(\frac{1}{1-2^{-c\tau_{0}}}\right)^{1/\tau_{0}}\cdot2^{\frac{3}{\tau}+3\frac{\sigma}{\tau}+5+2H}\leq\left(\frac{1}{1-2^{-c\tau_{0}}}\right)^{1/\tau_{0}}\cdot2^{\frac{3}{\tau_{0}}+3\omega+5+2H}=:C_{00}^{\left(5\right)},
\]
where $C_{00}^{\left(5\right)}$ only depends on $\alpha,\tau_{0},\omega,c,K,H,M_{1},M_{2}$.

\subsection{We have \texorpdfstring{$\ell_{1}=\ell_{2}=0$}{ℓ₁=ℓ₂=0}}

\label{subsec:LowPassLowPass}Here, the sum and the supremum reduce
to a single term, namely to
\begin{align*}
M_{0,0}^{\left(0\right)} & =\left(\frac{w_{0}^{s}}{w_{0}^{s}}\right)^{\tau}\cdot\left(1+\left\Vert T_{0}^{-1}T_{0}\right\Vert \right)^{\sigma}\cdot\left(\left|\det T_{0}\right|^{-1}\int_{S_{0}^{\left(\alpha\right)}}\varrho_{0}\left(T_{0}^{-1}\xi\right)\d\xi\right)^{\tau}\\
\left({\scriptstyle \text{since }Q_{0}'=\left(-1,1\right)^{2}}\right) & \leq2^{\sigma}\cdot\left(\int_{Q_{0}'}\left(1+\left|\xi\right|\right)^{-H}\d\xi\right)^{\tau}\leq2^{\sigma}\cdot\left[\lambda_{\dimension}\left(Q_{0}'\right)\right]^{\tau}\leq2^{\sigma}\cdot4^{\tau}=:C_{0}^{\left(6\right)},
\end{align*}
where $\left[\smash{C_{0}^{\left(6\right)}}\right]^{1/\tau}\leq2^{\sigma/\tau}\cdot4\leq4\cdot2^{\omega}=:C_{00}^{\left(6\right)}$.

\subsection{Completing the proof of Lemma \ref{lem:MainShearletLemma}}

By recalling equations \eqref{eq:ShearletSchurTestSubdivision1} and
\eqref{eq:ShearletSchurTestSubdivision2} and by collecting our results
from Subsections \ref{subsec:BothRightCone}–\ref{subsec:LowPassLowPass},
we finally conclude that
\[
\max\left\{ \sup_{j\in I}\,\sum_{i\in I}M_{j,i}^{\left(0\right)},\,\sup_{i\in I}\,\sum_{j\in I}M_{j,i}^{\left(0\right)}\right\} \leq25\cdot\left(\max\left\{ C_{00}^{\left(1\right)},C_{00}^{\left(2\right)},C_{00}^{\left(3\right)},C_{00}^{\left(4\right)},C_{00}^{\left(5\right)},C_{00}^{\left(6\right)}\right\} \right)^{\tau},
\]
given that the assumptions of Lemma \ref{lem:MainShearletLemma} are
fulfilled. This easily yields the claim of Lemma \ref{lem:MainShearletLemma}.\hfill $\square$

\section{The proof of Proposition \ref{prop:CartoonLikeFunctionsBoundedInAlphaShearletSmoothness}
in the general case}

\label{sec:CartoonLikeFunctionsAreBoundedInAlphaShearletSmoothness}Recall
that the parameter $\alpha$ for the definition of the $\alpha$-shearlet
smoothness spaces $\mathscr{S}_{\alpha,s}^{p,q}\left(\R^{2}\right)$
satisfies $\alpha\in\left[0,1\right]$, as for the theory of $\alpha$-molecules
developed in \cite{AlphaMolecules} or as for $\alpha$-curvelets\cite{CartoonApproximationWithAlphaCurvelets}.
In contrast, there is a definition of cone-adapted $\beta$-shearlets
(cf.\@ \cite[Definition 3.10]{AlphaMolecules}) for $\beta\in\left(1,\infty\right)$.

In this section we introduce so-called \textbf{\emph{reciprocal}}\textbf{
$\beta$-shearlet smoothness spaces} $\mathscr{S}_{\beta,s}^{p,q}\left(\R^{2}\right)$
which will turn out to be the smoothness spaces associated to $\beta$-shearlets.
Our main goal is to show $\mathscr{S}_{\beta,s}^{p,q}\left(\R^{2}\right)=\mathscr{S}_{\beta^{-1},s}^{p,q}\left(\R^{2}\right)$
for $\beta\in\left(1,\infty\right)$, i.e., the reciprocal $\beta$-shearlet
smoothness spaces coincide with the usual $\alpha$-shearlet smoothness
spaces for $\alpha=\beta^{-1}$. This will allow us to transfer approximation
results that are known for $\beta$-shearlets to approximation results
for $\alpha$-shearlets, which is not entirely trivial, since the
two definitions differ quite heavily for $\beta\neq2$, see also the
discussion before Definition \ref{def:AlphaShearletSystem}. Once
this property from $\beta$-shearlets to $\alpha$-shearlets is established,
we use it to prove Proposition \ref{prop:CartoonLikeFunctionsBoundedInAlphaShearletSmoothness}
for $\beta\in\left(1,2\right)$.

We begin with the definition of the reciprocal $\beta$-shearlet covering:
\begin{defn}
\label{def:ReciprocalBetaShearletCovering}For $\beta\in\left(1,\infty\right)$,
define
\[
J_{0}:=J_{0}^{\left(\beta\right)}:=\left\{ \left(j,\ell,\delta\right)\in\N_{0}\times\Z\times\left\{ 0,1\right\} \with\left|\ell\right|\leq H_{j}\right\} \quad\text{ with }\quad H_{j}:=H_{j}^{\left(\beta\right)}:=\left\lceil \smash{2^{\frac{j}{2}\left(\beta-1\right)}}\right\rceil .
\]
Furthermore, recall the matrices $S_{x},D_{b}^{\left(\alpha\right)}$
and $R$ from equation \eqref{eq:StandardMatrices}, and define
\[
Y_{j,\ell,\delta}:=Y_{j,\ell,\delta}^{\left(\beta\right)}:=R^{\delta}\cdot D_{\left(2^{\beta j/2}\right)}^{\left(1/\beta\right)}\cdot S_{\ell}^{T}\quad\text{ for }\left(j,\ell,\delta\right)\in J_{0}
\]
and $P_{j}':=P:=U_{\left(-3,3\right)}^{\left(\mu_{0}^{-1},\mu_{0}\right)}\cup\left(\vphantom{U^{\left(\mu\right)}}-\smash{U_{\left(-3,3\right)}^{\left(\mu_{0}^{-1},\mu_{0}\right)}}\right)$
for $j\in J_{0}$ with $U_{\left(a,b\right)}^{\left(\gamma,\mu\right)}$
as in equation \eqref{eq:BasicShearletSet} and with $\mu_{0}:=\mu_{0}^{\left(\beta\right)}:=3\cdot2^{\beta/2}$.

Finally, define $J:=J^{\left(\beta\right)}:=\left\{ 0\right\} \uplus J_{0}$,
set $c_{j}:=0$ for all $j\in J$ and $Y_{0}:=Y_{0}^{\left(\beta\right)}:=\identity$,
as well as $P_{0}':=\left(-1,1\right)^{2}$. Then, the \textbf{reciprocal
$\beta$-shearlet covering} is defined as
\[
\CalS^{\left(\beta\right)}:=\left(\smash{S_{j}^{\left(\beta\right)}}\right)_{j\in J}:=\left(\smash{Y_{j}^{\left(\beta\right)}}P_{j}'\right)_{j\in J}=\left(\smash{Y_{j}^{\left(\beta\right)}}P_{j}'+c_{j}\right)_{j\in J}.\qedhere
\]
\end{defn}
\begin{rem*}
The notation $\CalS^{\left(\beta\right)}$ for the reciprocal $\beta$-shearlet
covering might appear to be ambiguous with the notation $\CalS^{\left(\alpha\right)}$
for the $\alpha$-shearlet covering introduced in Definition \ref{def:AlphaShearletCovering},
but this is no real ambiguity: The parameter $\beta$ in the preceding
definition always satisfies $\beta\in\left(1,\infty\right)$, while
the parameter $\alpha$ from Definition \ref{def:AlphaShearletCovering}
satisfies $\alpha\in\left[0,1\right]$, so that no ambiguity is possible.
\end{rem*}
As for the usual $\alpha$-shearlet covering, our first goal is to
show that $\CalS^{\left(\beta\right)}$ is an almost structured covering
of $\R^{2}$. In this case, however, it will turn out to be useful
to show the following slightly more general result:
\begin{lem}
\label{lem:ReciprocalShearletCoveringIsAlmostStructuredGeneralized}Let
$\beta\in\left(1,\infty\right)$, $a,b\in\R$ and $\gamma,\mu,A\in\left(0,\infty\right)$
be arbitrary and let $U:=U_{\left(a,b\right)}^{\left(\gamma,\mu\right)}\cup\left(\vphantom{U^{\left(\gamma\right)}}-\smash{U_{\left(a,b\right)}^{\left(\gamma,\mu\right)}}\right)$,
as well as $U_{0}':=\left(-A,A\right)^{2}$. Define $U_{j}':=U$ for
$j\in J_{0}$ and consider the family
\[
\mathcal{U}:=\left(U_{j}\right)_{j\in J}:=\left(\smash{Y_{j}^{\left(\beta\right)}}\,U_{j}'\right)_{j\in J}.
\]
Then there are constants $N\in\N$ and $C,L\geq1$ (depending on $\beta,a,b,\gamma,\mu,A$)
such that the following are true:

\begin{enumerate}
\item We have $L^{-1}\cdot2^{\frac{\beta}{2}n}\leq\left|\xi\right|\leq L\cdot2^{\frac{\beta}{2}n}$
for all $\xi\in U_{n,m,\varepsilon}$ and arbitrary $\left(n,m,\varepsilon\right)\in J_{0}$.
\item We have $\left|i^{\ast}\right|\leq N$ for all $i\in J$ and $i^{\ast}:=\left\{ j\in J\with U_{j}\cap U_{i}\neq\emptyset\right\} $.
\item We have $\left\Vert Y_{i}^{-1}Y_{j}\right\Vert \leq C$ for all $i\in J$
and $j\in i^{\ast}$.\qedhere
\end{enumerate}
\end{lem}
\begin{proof}
The proof uses the same ideas as that of Lemma \ref{lem:AlphaShearletCoveringIsAlmostStructured}
and is only provided here for completeness.

Set $c:=\max\left\{ \left|a\right|,\left|b\right|\right\} $ and note
$U_{\left(a,b\right)}^{\left(\gamma,\mu\right)}\subset U_{\left(-c,c\right)}^{\left(\gamma,\mu\right)}$,
so that we can assume $a=-c$ and $b=c$, since the claim of the lemma
is stronger the larger the set $U_{\left(a,b\right)}^{\left(\gamma,\mu\right)}$
is. By even further enlarging this set, we can also assume $c\geq1$.
With the same reasoning, we can assume $A\geq1$.

Next, note with $U_{\left(B,C\right)}^{\left(\kappa,\lambda\right)}$
as in equation \eqref{eq:BasicShearletSet} that
\begin{equation}
V_{\left(B,C\right)}^{\left(\kappa,\lambda\right)}:=U_{\left(B,C\right)}^{\left(\kappa,\lambda\right)}\cup\left(-U_{\left(B,C\right)}^{\left(\kappa,\lambda\right)}\right)=\left\{ \left(\begin{matrix}\xi\\
\eta
\end{matrix}\right)\in\R^{\ast}\times\R\with\left|\xi\right|\in\left(\kappa,\lambda\right)\text{ and }\frac{\eta}{\xi}\in\left(B,C\right)\right\} \label{eq:UnconnectedShearletBaseSetDefinition}
\end{equation}
for arbitrary $B,C\in\R$ and $\kappa,\lambda>0$. It is now an easy
consequence of equation \eqref{eq:BaseSetTransformationRules} and
of $a=-c$ and $b=c$ that
\begin{equation}
U_{n,m,0}=V_{\left(2^{n\left(1-\beta\right)/2}\left(m-c\right),2^{n\left(1-\beta\right)/2}\left(m+c\right)\right)}^{\left(2^{\beta n/2}\gamma,2^{\beta n/2}\mu\right)}\qquad\forall\left(n,m,0\right)\in J_{0}.\label{eq:ReciprocalShearletCoveringNiceExpression}
\end{equation}
Now, since we have $m+c\leq\left|m\right|+c$ and $m-c\geq-\left|m\right|-c=-\left(\left|m\right|+c\right)$,
we get for arbitrary $\left(\begin{smallmatrix}\xi\\
\eta
\end{smallmatrix}\right)\in U_{\left(n,m,0\right)}$ because of $\left|m\right|\leq\left\lceil 2^{\frac{n}{2}\left(\beta-1\right)}\right\rceil \leq2^{\frac{n}{2}\left(\beta-1\right)}+1$
that
\begin{equation}
\left|\frac{\eta}{\xi}\right|<2^{\frac{n}{2}\left(1-\beta\right)}\left(\left|m\right|+c\right)\leq2^{\frac{n}{2}\left(1-\beta\right)}\left(2^{\frac{n}{2}\left(\beta-1\right)}+1+c\right)\leq c+2\leq3c.\label{eq:ReciprocalShearletCoveringConeLikeCorridor}
\end{equation}
Here, we used that $2^{\frac{n}{2}\left(1-\beta\right)}\leq1$, since
$\beta>1$. Consequently, we get
\[
\gamma\cdot2^{\frac{\beta}{2}n}\leq\left|\xi\right|\leq\left|\left(\begin{matrix}\xi\\
\eta
\end{matrix}\right)\right|\leq\left|\xi\right|+\left|\eta\right|\leq\left(1+3c\right)\cdot\left|\xi\right|<2^{\frac{\beta}{2}n}\cdot4\mu c.
\]
This establishes the first part of the lemma for $L:=\max\left\{ \gamma^{-1},\,4\mu c,\,1\right\} $,
since we have $U_{n,m,1}=R\cdot U_{n,m,0}$ and $\left|R\xi\right|=\left|\xi\right|$
for all $\xi\in\R^{2}$.

\medskip{}

Now, let $i=\left(n,m,\delta\right)\in J_{0}$ be fixed and let $\left(j,\ell,\varepsilon\right)\in J_{0}$
such that there is some $\left(\begin{smallmatrix}\xi\\
\eta
\end{smallmatrix}\right)\in U_{n,m,\delta}\cap U_{j,\ell,\varepsilon}\neq\emptyset$. In the following, we want to derive conditions on $\left(j,\ell,\varepsilon\right)$
which allow us to estimate the set $i^{\ast}$, as well as the norm
$\left\Vert Y_{i}^{-1}Y_{j}\right\Vert $.

First of all, set $M:=\left\lceil \frac{2}{\beta}\cdot\log_{2}\left(L^{2}\right)\right\rceil \in\N_{0}$,
so that $2^{M}\geq2^{\frac{2}{\beta}\cdot\log_{2}\left(L^{2}\right)}$
and thus $2^{\frac{\beta}{2}M}\geq2^{\log_{2}\left(L^{2}\right)}=L^{2}$.
Consequently, the first part of the lemma implies $L^{-1}\cdot2^{\frac{\beta}{2}j}\leq\left|\left(\begin{smallmatrix}\xi\\
\eta
\end{smallmatrix}\right)\right|\leq L\cdot2^{\frac{\beta}{2}n}$ and thus $2^{\frac{\beta}{2}\left(j-n\right)}\leq L^{2}\leq2^{\frac{\beta}{2}M}$,
which entails $j-n\leq M$. By symmetry, we in fact get $\left|j-n\right|\leq M$
and thus $j\in\left\{ n-M,\dots,n+M\right\} $.

In order to establish further conditions on $\left(j,\ell,\varepsilon\right)$,
we distinguish several cases depending on $\varepsilon,\delta$:

\textbf{Case 1}: We have $\varepsilon=\delta=0$. In this case, equation
\eqref{eq:ReciprocalShearletCoveringNiceExpression} shows
\[
2^{\frac{1-\beta}{2}n}\left(m-c\right)<\frac{\eta}{\xi}<2^{\frac{1-\beta}{2}n}\left(m+c\right)\quad\text{ and }\quad2^{\frac{1-\beta}{2}j}\left(\ell-c\right)<\frac{\eta}{\xi}<2^{\frac{1-\beta}{2}j}\left(\ell+c\right).
\]
By rearranging, this implies for $C_{1}:=\left(2^{\frac{\beta-1}{2}M}+1\right)\cdot c$
that 
\[
\ell<2^{\frac{\beta-1}{2}\left(j-n\right)}\left(m+c\right)+c\leq2^{\frac{\beta-1}{2}\left(j-n\right)}m+C_{1},\quad\text{ as well as }\quad\ell>2^{\frac{\beta-1}{2}\left(j-n\right)}\left(m-c\right)-c\geq2^{\frac{\beta-1}{2}\left(j-n\right)}m-C_{1}.
\]
Consequently, with
\[
\Gamma_{n,m,t}:=\Z\cap\left[2^{\frac{\beta-1}{2}\left(t-n\right)}m-C_{1},2^{\frac{\beta-1}{2}\left(t-n\right)}m+C_{1}\right],
\]
we have established $\left(j,\ell,\varepsilon\right)\in\bigcup_{t=n-M}^{n+M}\left[\left\{ t\right\} \times\Gamma_{n,m,t}\times\left\{ 0\right\} \right]$.
But since every (closed) interval $I=\left[B,D\right]$ satisfies
$\left|I\cap\Z\right|\leq1+D-B$, we have $\left|\Gamma_{n,m,t}\right|\leq1+2C_{1}$
and thus
\begin{equation}
\left|\left\{ \left(j,\ell,0\right)\in J_{0}\with U_{j,\ell,0}\cap U_{n,m,0}\neq\emptyset\right\} \right|\leq\sum_{t=n-M}^{n+M}\left|\left\{ t\right\} \times\Gamma_{n,m,t}\times\left\{ 0\right\} \right|\leq\left(1+2M\right)\cdot\left(1+2C_{1}\right).\label{eq:ReciprocalCoveringAdmissibilityBothHorizontal}
\end{equation}

Furthermore, a direct computation shows
\[
Y_{n,m,0}^{-1}Y_{j,\ell,0}=\left(\begin{array}{c|c}
2^{\frac{\beta}{2}\left(j-n\right)} & 0\\
2^{\frac{j-n}{2}}\ell-2^{\frac{\beta}{2}\left(j-n\right)}m & 2^{\frac{j-n}{2}}
\end{array}\right).
\]
But thanks to $\left|j-n\right|\leq M$, we have $0\leq2^{\frac{\beta}{2}\left(j-n\right)}\leq2^{\frac{\beta}{2}M}$
and $0\leq2^{\frac{j-n}{2}}\leq2^{\frac{M}{2}}$. Finally, we saw
above that $\left|\ell-2^{\frac{\beta-1}{2}\left(j-n\right)}m\right|\leq C_{1}$,
so that
\[
\left|2^{\frac{j-n}{2}}\ell-2^{\frac{\beta}{2}\left(j-n\right)}m\right|=2^{\frac{j-n}{2}}\left|\ell-2^{\frac{\beta-1}{2}\left(j-n\right)}m\right|\leq2^{\frac{M}{2}}C_{1}.
\]
All in all, this implies $\left\Vert Y_{n,m,0}^{-1}\cdot Y_{j,\ell,0}\right\Vert \leq2^{\frac{\beta}{2}M}+2^{\frac{M}{2}}+2^{\frac{M}{2}}C_{1}$
and thus concludes our considerations for the present case.

\medskip{}

\textbf{Case 2}: We have $\varepsilon=1$ and $\delta=0$. In this
case, a direct calculation shows
\begin{equation}
Y_{n,m,0}^{-1}Y_{j,\ell,1}=\left(\begin{array}{c|c}
2^{\frac{1}{2}\left(j-\beta n\right)}\ell & 2^{\frac{1}{2}\left(j-\beta n\right)}\\
2^{\frac{1}{2}\left(\beta j-n\right)}-2^{\frac{1}{2}\left(j-\beta n\right)}m\ell & -2^{\frac{1}{2}\left(j-\beta n\right)}m
\end{array}\right).\label{eq:ReciprocalCoveringTransitionMatrixHorizontalVertical}
\end{equation}
We immediately recall that $\left|m\right|\leq\left\lceil 2^{\frac{n}{2}\left(\beta-1\right)}\right\rceil \leq1+2^{\frac{n}{2}\left(\beta-1\right)}\leq2\cdot2^{\frac{n}{2}\left(\beta-1\right)}$
and likewise $\left|\ell\right|\leq2\cdot2^{\frac{j}{2}\left(\beta-1\right)}$.
In conjunction with $\left|n-j\right|\leq M$ and $\beta>1$, this
implies
\begin{equation}
\begin{split}\left|2^{\frac{1}{2}\left(j-\beta n\right)}\ell\right| & \leq2\cdot2^{\frac{1}{2}\left(j-\beta n\right)}2^{\frac{j}{2}\left(\beta-1\right)}=2\cdot2^{\frac{\beta}{2}\left(j-n\right)}\leq2\cdot2^{\frac{\beta}{2}M},\\
\left|2^{\frac{1}{2}\left(j-\beta n\right)}\right| & \leq2^{\frac{1}{2}\left(j-n\right)}2^{\frac{n}{2}\left(1-\beta\right)}\leq2^{\frac{1}{2}\left(j-n\right)}\leq2^{\frac{M}{2}},\\
\left|-2^{\frac{1}{2}\left(j-\beta n\right)}m\right| & \leq2\cdot2^{\frac{1}{2}\left(j-\beta n\right)}2^{\frac{n}{2}\left(\beta-1\right)}=2\cdot2^{\frac{1}{2}\left(j-n\right)}\leq2\cdot2^{\frac{M}{2}}.
\end{split}
\label{eq:ReciprocalCoveringTransitionMatrixHorizontalVerticalEasyEstimates}
\end{equation}
In order to estimate the remaining entry of $Y_{n,m,0}^{-1}Y_{j,\ell,1}$
and to obtain an estimate similar to equation \eqref{eq:ReciprocalCoveringAdmissibilityBothHorizontal},
we have to work harder. To this end, define 
\begin{equation}
K:=\min\left\{ \left(12\cdot c^{2}\right)^{-1},\,\left(2^{\beta M}\cdot3c\right)^{-1}\right\} \in\left(0,1\right)\qquad\text{ and }\qquad n_{0}:=\frac{2}{\beta-1}\cdot\log_{2}\left(K^{-1}\right)\in\left(0,\infty\right).\label{eq:ReciprocalCoveringKN0Definition}
\end{equation}
Based on these quantities, we now distinguish two subcases:

\textbf{Case 2(a)}: We have $n\geq M+n_{0}$. First note that this
implies $j\geq n-M\geq n_{0}$. Furthermore, we have $2^{\frac{n_{0}}{2}\left(\beta-1\right)}=2^{\log_{2}\left(K^{-1}\right)}=K^{-1}$
and thus $2^{\frac{n}{2}\left(\beta-1\right)}\geq K^{-1}$ and $2^{\frac{j}{2}\left(\beta-1\right)}\geq K^{-1}$.
Next, note that equation \eqref{eq:ReciprocalShearletCoveringConeLikeCorridor}
implies because of $\left(\begin{smallmatrix}\xi\\
\eta
\end{smallmatrix}\right)\in U_{n,m,0}$ that $\left|\eta/\xi\right|<3c$. Likewise, since 
\begin{equation}
\left(\begin{matrix}\eta\\
\xi
\end{matrix}\right)=R\left(\begin{matrix}\xi\\
\eta
\end{matrix}\right)\in R\cdot U_{j,\ell,1}=RR\cdot U_{j,\ell,0}=U_{j,\ell,0},\label{eq:ReciprocalCoveringAdmissibilitySwitchedXiInUJL}
\end{equation}
another application of equation \eqref{eq:ReciprocalShearletCoveringConeLikeCorridor}
shows $\eta\neq0$ and $\left|\xi/\eta\right|<3c$, so that $\left(3c\right)^{-1}<\left|\eta/\xi\right|<3c$.

We now claim that this implies $\left|m\right|>c$. Indeed, if this
was false, we would get from equation \eqref{eq:ReciprocalShearletCoveringNiceExpression}
because of $2^{\frac{n}{2}\left(1-\beta\right)}\leq K$ that
\[
\left(3c\right)^{-1}<\left|\frac{\eta}{\xi}\right|<2^{\frac{n}{2}\left(1-\beta\right)}\cdot\left(\left|m\right|+c\right)\leq2c\cdot2^{\frac{n}{2}\left(1-\beta\right)}\leq2cK\leq\frac{2c}{12\cdot c^{2}}=\frac{1}{2}\cdot\frac{1}{3c}<\left(3c\right)^{-1},
\]
a contradiction. Because of $\left|m\right|>c$ we either have $m>c$
or $m<-c$. Let us now set $C_{2}:=2^{\beta M}\cdot3c$ and distinguish
these two subcases:

\textbf{Case 2(a)(i)}: We have $m>c$. We first claim that this implies
$m\geq2^{\frac{n}{2}\left(\beta-1\right)}-C_{2}$. To see this, assume
towards a contradiction that $m<2^{\frac{n}{2}\left(\beta-1\right)}-C_{2}$.
But equation \eqref{eq:ReciprocalShearletCoveringNiceExpression}
shows because of $\left(\begin{smallmatrix}\xi\\
\eta
\end{smallmatrix}\right)\in U_{n,m,0}$ that 
\[
0<2^{\frac{n}{2}\left(1-\beta\right)}\cdot\left(m-c\right)<\frac{\eta}{\xi}<2^{\frac{n}{2}\left(1-\beta\right)}\cdot\left(m+c\right)<2^{\frac{n}{2}\left(1-\beta\right)}\cdot\left(2^{\frac{n}{2}\left(\beta-1\right)}-C_{2}+c\right).
\]
By taking reciprocals and by noting $C_{2}\geq3c>c$, we arrive at
\[
\frac{\xi}{\eta}>\frac{2^{\frac{n}{2}\left(\beta-1\right)}}{2^{\frac{n}{2}\left(\beta-1\right)}-C_{2}+c}=1+\frac{C_{2}-c}{2^{\frac{n}{2}\left(\beta-1\right)}-C_{2}+c}>1+\frac{C_{2}-c}{2^{\frac{n}{2}\left(\beta-1\right)}}.
\]
But another application of equations \eqref{eq:ReciprocalShearletCoveringNiceExpression}
and \eqref{eq:ReciprocalCoveringAdmissibilitySwitchedXiInUJL} shows
because of $\left|\ell\right|\leq\left\lceil 2^{j\left(\beta-1\right)/2}\right\rceil \leq1+2^{j\left(\beta-1\right)/2}$
that
\[
\frac{\xi}{\eta}<2^{\frac{j}{2}\left(1-\beta\right)}\left(\ell+c\right)\leq2^{\frac{j}{2}\left(1-\beta\right)}\left(2^{\frac{j}{2}\left(\beta-1\right)}+1+c\right)\leq1+\frac{1+c}{2^{\frac{j}{2}\left(\beta-1\right)}}.
\]
A combination of the last two displayed equations finally yields
\[
\frac{C_{2}-c}{2^{\frac{n}{2}\left(\beta-1\right)}}<\frac{1+c}{2^{\frac{j}{2}\left(\beta-1\right)}}\quad\text{ and thus }\quad C_{2}<c+2^{\frac{\beta-1}{2}\left(n-j\right)}\cdot\left(1+c\right)\leq c+2^{\frac{\beta-1}{2}M}\cdot2c\leq c+2^{\beta M}\cdot2c\leq2^{\beta M}\cdot3c=C_{2},
\]
a contradiction. Here, we used that $\left|n-j\right|\leq M$ and
that $c\geq1$. This contradiction shows $m\geq2^{\frac{n}{2}\left(\beta-1\right)}-C_{2}$.

\medskip{}

Now, we claim similarly that $\ell\geq2^{\frac{j}{2}\left(\beta-1\right)}-C_{2}$.
To see this, assume towards a contradiction that $\ell<2^{\frac{j}{2}\left(\beta-1\right)}-C_{2}$.
Recall from equation \eqref{eq:ReciprocalShearletCoveringNiceExpression}
and because of $m>c$ that $\frac{\eta}{\xi}>2^{\frac{n}{2}\left(1-\beta\right)}\cdot\left(m-c\right)>0$,
so that also $\frac{\xi}{\eta}>0$. Now, an application of equations
\eqref{eq:ReciprocalShearletCoveringNiceExpression} and \eqref{eq:ReciprocalCoveringAdmissibilitySwitchedXiInUJL}
shows
\[
0<\frac{\xi}{\eta}<2^{\frac{j}{2}\left(1-\beta\right)}\cdot\left(\ell+c\right)<2^{\frac{j}{2}\left(1-\beta\right)}\cdot\left(2^{\frac{j}{2}\left(\beta-1\right)}-C_{2}+c\right).
\]
By taking reciprocals, we get as above because of $C_{2}\geq3c>c$
that
\[
\frac{\eta}{\xi}>\frac{2^{\frac{j}{2}\left(\beta-1\right)}}{2^{\frac{j}{2}\left(\beta-1\right)}-C_{2}+c}=1+\frac{C_{2}-c}{2^{\frac{j}{2}\left(\beta-1\right)}-C_{2}+c}>1+\frac{C_{2}-c}{2^{\frac{j}{2}\left(\beta-1\right)}}.
\]
But equation \eqref{eq:ReciprocalShearletCoveringNiceExpression}
shows because of $\left(\begin{smallmatrix}\xi\\
\eta
\end{smallmatrix}\right)\in U_{n,m,0}$ and since $\left|m\right|\leq\left\lceil 2^{n\left(\beta-1\right)/2}\right\rceil \leq1+2^{n\left(\beta-1\right)/2}$
that 
\[
\frac{\eta}{\xi}<2^{\frac{n}{2}\left(1-\beta\right)}\left(m+c\right)\leq2^{\frac{n}{2}\left(1-\beta\right)}\left(2^{\frac{n}{2}\left(\beta-1\right)}+1+c\right)\leq1+\frac{1+c}{2^{\frac{n}{2}\left(\beta-1\right)}}.
\]
Again, by combining the preceding two displayed equations, we obtain
a contradiction.

\medskip{}

We have thus shown $\ell\geq2^{\frac{j}{2}\left(\beta-1\right)}-C_{2}\geq\left\lceil \smash{2^{\frac{j}{2}\left(\beta-1\right)}}\right\rceil -\left(1+C_{2}\right)\geq\left\lceil \smash{2^{\frac{j}{2}\left(\beta-1\right)}}\right\rceil -\left(1+\left\lceil C_{2}\right\rceil \right)$.
Hence, setting $C_{3}:=1+\left\lceil C_{2}\right\rceil $, we have
shown for $n\geq M+n_{0}$ and $m\geq0$ (which entails $m>c$) that
\begin{equation}
\begin{split}\left|\left\{ \left(j,\ell,1\right)\in J_{0}\with U_{j,\ell,1}\cap U_{n,m,0}\neq\emptyset\right\} \right| & \leq\sum_{t=n-M}^{n+M}\left|\left\{ t\right\} \times\left\{ \left\lceil \smash{2^{\frac{t}{2}\left(\beta-1\right)}}\right\rceil -C_{3},\dots,\left\lceil \smash{2^{\frac{t}{2}\left(\beta-1\right)}}\right\rceil \right\} \times\left\{ 1\right\} \right|\\
 & \leq\left(1+2M\right)\cdot\left(1+C_{3}\right).
\end{split}
\label{eq:ReciprocalCoveringAdmissibilityHorizontalVertical-1}
\end{equation}
Now, we can finally also estimate the remaining entry of the transition
matrix $Y_{n,m,0}^{-1}Y_{j,\ell,1}$ (cf.\@ equation \eqref{eq:ReciprocalCoveringTransitionMatrixHorizontalVertical}):
Recall from the beginning of Case 2(a) and from equation \eqref{eq:ReciprocalCoveringKN0Definition}
that $2^{\frac{j}{2}\left(\beta-1\right)}\geq K^{-1}\geq2^{\beta M}\cdot3c=C_{2}$
and likewise that $2^{\frac{n}{2}\left(\beta-1\right)}\geq C_{2}$.
Hence, $\ell\geq2^{\frac{j}{2}\left(\beta-1\right)}-C_{2}\geq0$ and
similarly $m\geq0$, so that
\[
0\leq\ell m\leq\left\lceil \smash{2^{\frac{j}{2}\left(\beta-1\right)}}\right\rceil \cdot\left\lceil \smash{2^{\frac{n}{2}\left(\beta-1\right)}}\right\rceil \leq\left(1+2^{\frac{j}{2}\left(\beta-1\right)}\right)\cdot\left(1+2^{\frac{n}{2}\left(\beta-1\right)}\right)=2^{\frac{j}{2}\left(\beta-1\right)}2^{\frac{n}{2}\left(\beta-1\right)}+2^{\frac{n}{2}\left(\beta-1\right)}+2^{\frac{j}{2}\left(\beta-1\right)}+1.
\]
Consequently, we get because of $m\geq2^{\frac{n}{2}\left(\beta-1\right)}-C_{2}\geq0$
and $\ell\geq2^{\frac{j}{2}\left(\beta-1\right)}-C_{2}\geq0$ that
\begin{align}
\left|2^{\frac{1}{2}\left(\beta j-n\right)}\!-\!2^{\frac{1}{2}\left(j-\beta n\right)}m\ell\right| & =2^{\frac{1}{2}\left(j-\beta n\right)}\left|2^{\frac{n}{2}\left(\beta-1\right)}2^{\frac{j}{2}\left(\beta-1\right)}-m\ell\right|\nonumber \\
 & \leq2^{\frac{1}{2}\left(j-\beta n\right)}\!\cdot\!\left(\left|2^{\frac{n}{2}\left(\beta-1\right)}2^{\frac{j}{2}\left(\beta-1\right)}+2^{\frac{n}{2}\left(\beta-1\right)}+2^{\frac{j}{2}\left(\beta-1\right)}+1-m\ell\right|+\left|2^{\frac{n}{2}\left(\beta-1\right)}+2^{\frac{j}{2}\left(\beta-1\right)}+1\right|\right)\nonumber \\
 & =2^{\frac{1}{2}\left(j-\beta n\right)}\!\cdot\!\left(2^{\frac{n}{2}\left(\beta-1\right)}2^{\frac{j}{2}\left(\beta-1\right)}+2\cdot2^{\frac{n}{2}\left(\beta-1\right)}+2\cdot2^{\frac{j}{2}\left(\beta-1\right)}+2-m\ell\right)\nonumber \\
 & \leq2^{\frac{1}{2}\left(j-\beta n\right)}\!\cdot\!\left(2^{\frac{n}{2}\left(\beta-1\right)}2^{\frac{j}{2}\left(\beta-1\right)}\!+\!2\cdot2^{\frac{n}{2}\left(\beta-1\right)}\!+\!2\cdot2^{\frac{j}{2}\left(\beta-1\right)}\!+\!2\!-\!\left(2^{\frac{n}{2}\left(\beta-1\right)}\!-\!C_{2}\right)\left(2^{\frac{j}{2}\left(\beta-1\right)}\!-\!C_{2}\right)\right)\nonumber \\
 & =2^{\frac{1}{2}\left(j-\beta n\right)}\!\cdot\!\left(\left(2+C_{2}\right)\cdot2^{\frac{n}{2}\left(\beta-1\right)}+\left(2+C_{2}\right)\cdot2^{\frac{j}{2}\left(\beta-1\right)}+2-C_{2}^{2}\right)\nonumber \\
 & \leq2^{\frac{\beta}{2}\left(j-n\right)}\cdot2^{\frac{j}{2}\left(1-\beta\right)}\cdot\left(\left(2+C_{2}\right)\cdot2^{\frac{n}{2}\left(\beta-1\right)}+\left(2+C_{2}\right)\cdot2^{\frac{j}{2}\left(\beta-1\right)}+2\right)\nonumber \\
 & =2^{\frac{\beta}{2}\left(j-n\right)}\cdot\left(\left(2+C_{2}\right)\cdot2^{\left(n-j\right)\frac{\beta-1}{2}}+2+C_{2}+2\cdot2^{\frac{j}{2}\left(1-\beta\right)}\right)\nonumber \\
\left({\scriptstyle \text{since }\left|j-n\right|\leq M\text{ and }\beta>1}\right) & \leq2^{\frac{\beta}{2}M}\cdot\left(\left(2+C_{2}\right)\cdot2^{M\frac{\beta-1}{2}}+2+C_{2}+2\right)=:C_{4}.\label{eq:ReciprocalCoveringHorizontalVerticalDifficultMatrixEntry}
\end{align}
In conjunction with equation \eqref{eq:ReciprocalCoveringTransitionMatrixHorizontalVerticalEasyEstimates},
this implies $\left\Vert Y_{n,m,0}^{-1}Y_{j,\ell,1}\right\Vert \leq2\cdot2^{\frac{\beta}{2}M}+3\cdot2^{\frac{M}{2}}+C_{4}$.

\medskip{}

\textbf{Case 2(a)(ii)}: We have $m<-c$. Here, we set $\widetilde{m}:=-m$
and $\widetilde{\ell}:=-\ell$ and note that
\[
2^{\frac{n}{2}\left(1-\beta\right)}\left(m-c\right)<\frac{\eta}{\xi}<2^{\frac{n}{2}\left(1-\beta\right)}\left(m+c\right)\qquad\text{ implies }\qquad2^{\frac{n}{2}\left(1-\beta\right)}\left(-m-c\right)<\frac{-\eta}{\xi}<2^{\frac{n}{2}\left(1-\beta\right)}\left(-m+c\right),
\]
so that $\left(\xi,\,-\eta\right)\in U_{n,-m,0}=U_{n,\widetilde{m},0}$.
Likewise, it is not hard to see $\left(\xi,\,-\eta\right)\in U_{j,-\ell,1}=U_{j,\widetilde{\ell},1}$,
so that Case 2(a)(i) shows (because of $\widetilde{m}>c$) that $\widetilde{m}\geq2^{\frac{n}{2}\left(\beta-1\right)}-C_{2}$
and $\widetilde{\ell}\geq2^{\frac{j}{2}\left(\beta-1\right)}-C_{2}\geq\left\lceil \smash{2^{\frac{j}{2}\left(\beta-1\right)}}\right\rceil -C_{3}$,
which entails $\ell\leq-\left\lceil \smash{2^{\frac{j}{2}\left(\beta-1\right)}}\right\rceil +C_{3}$.
Hence, we have shown for $n\geq M+n_{0}$ and $m<0$ (which entails
$m<-c$) that
\begin{equation}
\begin{split}\left|\left\{ \left(j,\ell,1\right)\in J_{0}\with U_{j,\ell,1}\cap U_{n,m,0}\neq\emptyset\right\} \right| & \leq\sum_{t=n-M}^{n+M}\left|\left\{ t\right\} \times\left\{ -\left\lceil \smash{2^{\frac{t}{2}\left(\beta-1\right)}}\right\rceil ,\dots,-\left\lceil \smash{2^{\frac{t}{2}\left(\beta-1\right)}}\right\rceil +C_{3}\right\} \times\left\{ 1\right\} \right|\\
 & \leq\left(1+2M\right)\cdot\left(1+C_{3}\right),
\end{split}
\label{eq:ReciprocalCoveringAdmissibilityHorizontalVertical-2}
\end{equation}
as in the preceding case.

Finally, because of $\ell m=\widetilde{\ell}\cdot\widetilde{m}$,
we get $\left|2^{\frac{1}{2}\left(\beta j-n\right)}-2^{\frac{1}{2}\left(j-\beta n\right)}m\ell\right|=\left|2^{\frac{1}{2}\left(\beta j-n\right)}-2^{\frac{1}{2}\left(j-\beta n\right)}\widetilde{m}\widetilde{\ell}\right|\leq C_{4}$
from equation \eqref{eq:ReciprocalCoveringHorizontalVerticalDifficultMatrixEntry}
and thus $\left\Vert Y_{n,m,0}^{-1}Y_{j,\ell,1}\right\Vert \leq2\cdot2^{\frac{\beta}{2}M}+3\cdot2^{\frac{M}{2}}+C_{4}$
as in the previous case.

\medskip{}

\textbf{Case 2(b)}: We have $n\leq n_{0}+M$. This implies $j\leq n_{0}+2M$
and $\left|\ell\right|\leq\left\lceil \smash{2^{\frac{j}{2}\left(\beta-1\right)}}\right\rceil \leq\left\lceil \smash{2^{\frac{\beta}{2}j}}\right\rceil \leq\left\lceil \smash{2^{\beta\left(n_{0}+2M\right)}}\right\rceil $,
because of $\left|n-j\right|\leq M$. On the one hand, this implies
\begin{equation}
\begin{split}\left|\left\{ \left(j,\ell,1\right)\in J_{0}\with U_{j,\ell,1}\cap U_{n,m,0}\neq\emptyset\right\} \right| & \leq\left|\left\{ 0,\dots,n_{0}+2M\right\} \times\left\{ -\left\lceil \smash{2^{\beta\left(n_{0}+2M\right)}}\right\rceil ,\dots,\left\lceil \smash{2^{\beta\left(n_{0}+2M\right)}}\right\rceil \right\} \times\left\{ 1\right\} \right|\\
 & \leq\left(n_{0}+2M+1\right)\cdot\left(1+2\cdot\left\lceil \smash{2^{\beta\left(n_{0}+2M\right)}}\right\rceil \right)
\end{split}
\label{eq:ReciprocalCoveringAdmissibilityHorizontalVertical-3}
\end{equation}
and on the other hand
\[
\left\Vert Y_{n,m,0}^{-1}Y_{j,\ell,1}\right\Vert \leq\max_{n'\leq M+n_{0}}\;\max_{\left|m'\right|\leq\left\lceil 2^{n'\cdot\left(\beta-1\right)/2}\right\rceil }\:\max_{j'\leq n_{0}+2M}\;\max_{\left|\ell'\right|\leq\left\lceil 2^{j'\cdot\left(\beta-1\right)/2}\right\rceil }\left\Vert Y_{n',m',0}^{-1}\cdot Y_{j',\ell',1}\right\Vert =:C_{5}.
\]

\medskip{}

\textbf{Case 3}: We have $\varepsilon=\delta=1$. Here, we observe
that $U_{n,m,1}\cap U_{j,\ell,1}=R\cdot\left(U_{n,m,0}\cap U_{j,\ell,0}\right)$,
so that $U_{n,m,1}\cap U_{j,\ell,1}\neq\emptyset$ if and only if
$U_{n,m,0}\cap U_{j,\ell,0}\neq\emptyset$. Consequently, we get from
Case 1, equation \eqref{eq:ReciprocalCoveringAdmissibilityBothHorizontal}
that
\[
\left|\left\{ \left(j,\ell,1\right)\in J_{0}\with U_{j,\ell,1}\cap U_{n,m,1}\neq\emptyset\right\} \right|=\left|\left\{ \left(j,\ell,0\right)\in J_{0}\with U_{j,\ell,0}\cap U_{n,m,0}\neq\emptyset\right\} \right|\leq\left(1+2M\right)\cdot\left(1+2C_{1}\right).
\]
Likewise, since $Y_{n,m,1}^{-1}\cdot Y_{j,\ell,1}=Y_{n,m,0}^{-1}\cdot R^{-1}R\cdot Y_{j,\ell,0}=Y_{n,m,0}^{-1}\cdot Y_{j,\ell,0}$,
we get in case of $U_{n,m,1}\cap U_{j,\ell,1}\neq\emptyset$ that
\[
\left\Vert Y_{n,m,1}^{-1}\cdot Y_{j,\ell,1}\right\Vert =\left\Vert Y_{n,m,0}^{-1}\cdot Y_{j,\ell,0}\right\Vert \leq2^{\frac{\beta}{2}M}+2^{\frac{M}{2}}+2^{\frac{M}{2}}C_{1},
\]
since $U_{n,m,0}\cap U_{j,\ell,0}\neq\emptyset$, cf.\@ Case 1.

\textbf{Case 4}: We have $\varepsilon=0$ and $\delta=1$. As in the
previous case, we observe $U_{n,m,1}\cap U_{j,\ell,0}=R\cdot\left(U_{n,m,0}\cap U_{j,\ell,1}\right)$,
so that we can reduce the present case to the setting of Case 2, similar
to what was done in Case 3. In view of equations \eqref{eq:ReciprocalCoveringAdmissibilityHorizontalVertical-1},
\eqref{eq:ReciprocalCoveringAdmissibilityHorizontalVertical-2} and
\eqref{eq:ReciprocalCoveringAdmissibilityHorizontalVertical-3}, this
implies
\begin{align*}
\left|\left\{ \left(j,\ell,0\right)\in J_{0}\with U_{j,\ell,0}\cap U_{n,m,1}\neq\emptyset\right\} \right| & =\left|\left\{ \left(j,\ell,1\right)\in J_{0}\with U_{j,\ell,1}\cap U_{n,m,0}\neq\emptyset\right\} \right|\\
 & \leq\max\left\{ \left(1+2M\right)\cdot\left(1+C_{3}\right),\,\left(n_{0}+2M+1\right)\cdot\left(1+2\cdot\left\lceil \smash{2^{\beta\left(n_{0}+2M\right)}}\right\rceil \right)\right\} ,
\end{align*}
as well as
\[
\left\Vert Y_{n,m,1}^{-1}Y_{j,\ell,0}\right\Vert \leq\max\left\{ 2\cdot2^{\frac{\beta}{2}M}+3\cdot2^{\frac{M}{2}}+C_{4},\,C_{5}\right\} ,
\]
provided that $U_{n,m,1}\cap U_{j,\ell,0}\neq\emptyset$.

\medskip{}

It remains to consider the case $i=0$ or $j=0$. Recall from the
first part of the lemma that $\left|\xi\right|\geq L^{-1}\cdot2^{\frac{\beta}{2}n}$
for all $\xi\in U_{n,m,\varepsilon}$. Conversely, for $\xi\in U_{0}=U_{0}'$,
we have $\left|\xi\right|\leq2A$, so that $U_{0}\cap U_{n,m,\varepsilon}\neq\emptyset$
can only hold if $2^{\frac{\beta}{2}n}\leq2AL$, i.e., if $n\leq\left\lfloor \frac{2}{\beta}\cdot\log_{2}\left(2AL\right)\right\rfloor =:n_{1}\in\N_{0}$.
On the one hand, this implies because of $\left|\ell\right|\leq\left\lceil \smash{2^{\frac{j}{2}\left(\beta-1\right)}}\right\rceil \leq\left\lceil 2^{\beta j}\right\rceil $
for $\left(j,\ell,\varepsilon\right)\in J_{0}$ that
\[
\left|\left\{ j\in J\with U_{j}\cap U_{0}\neq\emptyset\right\} \right|\leq\left|\left\{ 0\right\} \cup\left(\left\{ 0,\dots,n_{1}\right\} \times\left\{ -\left\lceil \smash{2^{\beta n_{1}}}\right\rceil ,\dots,\left\lceil \smash{2^{\beta n_{1}}}\right\rceil \right\} \times\left\{ \pm1\right\} \right)\right|\leq1+2\cdot\left(1+n_{1}\right)\cdot\left(1+2\cdot\left\lceil \smash{2^{\beta n_{1}}}\right\rceil \right).
\]
On the other hand, we get in case of $U_{0}\cap U_{n,m,\varepsilon}\neq\emptyset$
for some $\left(n,m,\varepsilon\right)\in J_{0}$ that
\begin{align*}
\left\Vert Y_{0}^{-1}Y_{n,m,\varepsilon}\right\Vert  & =\left\Vert \left(\begin{matrix}2^{\frac{\beta}{2}n} & 0\\
0 & 2^{\frac{n}{2}}
\end{matrix}\right)\cdot\left(\begin{matrix}1 & 0\\
m & 1
\end{matrix}\right)\right\Vert \leq\left\Vert \left(\begin{matrix}2^{\frac{\beta}{2}n} & 0\\
0 & 2^{\frac{n}{2}}
\end{matrix}\right)\right\Vert \cdot\left\Vert \left(\begin{matrix}1 & 0\\
m & 1
\end{matrix}\right)\right\Vert \\
 & \leq\max\left\{ 2^{\frac{n}{2}},\,2^{\frac{\beta}{2}n}\right\} \cdot\left(2+\left|m\right|\right)\leq2^{\frac{\beta}{2}n_{1}}\cdot\left(2+\left\lceil 2^{\beta n_{1}}\right\rceil \right),
\end{align*}
as well as
\[
\left\Vert Y_{n,m,\varepsilon}^{-1}Y_{0}\right\Vert =\left\Vert \left(\begin{matrix}1 & 0\\
-m & 1
\end{matrix}\right)\cdot\left(\begin{matrix}2^{-\frac{\beta}{2}n} & 0\\
0 & 2^{-\frac{n}{2}}
\end{matrix}\right)\right\Vert \leq\left\Vert \left(\begin{matrix}1 & 0\\
-m & 1
\end{matrix}\right)\right\Vert \cdot\left\Vert \left(\begin{matrix}2^{-\frac{\beta}{2}n} & 0\\
0 & 2^{-\frac{n}{2}}
\end{matrix}\right)\right\Vert \leq2+\left|m\right|\leq2+\left\lceil 2^{\beta n_{1}}\right\rceil .
\]
Taken together, the preceding cases easily yield the claim of the
lemma.
\end{proof}
As a corollary of the preceding lemma, we can now easily show that
the reciprocal $\beta$-shearlet covering is indeed an almost structured
covering of $\R^{2}$.
\begin{cor}
\label{cor:ReciprocalShearletCoveringAlmostStructured}For every $\beta\in\left(1,\infty\right)$,
the family $\CalS^{\left(\beta\right)}$ from Definition \ref{def:ReciprocalBetaShearletCovering}
is an almost structured covering of $\R^{2}$.

Furthermore, if we set $v_{n,m,\varepsilon}:=2^{\frac{\beta}{2}n}$
for $\left(n,m,\varepsilon\right)\in J_{0}$ and $v_{0}:=1$, then
the weight $v^{s}=\left(v_{j}^{s}\right)_{j\in J}$ is $\CalS^{\left(\beta\right)}$-moderate
for arbitrary $s\in\R$.

Precisely, we have $C_{\CalS^{\left(\beta\right)},v^{s}}\leq K^{2\left|s\right|}$
for some absolute constant $K=K\left(\beta\right)\geq1$ which also
satisfies
\[
K^{-1}\cdot v_{j}\leq1+\left|\xi\right|\leq K\cdot v_{j}\qquad\forall\xi\in S_{j}^{\left(\beta\right)}\text{ and all }j\in J.\qedhere
\]
\end{cor}
\begin{proof}
First of all, note that an application of Lemma \ref{lem:ReciprocalShearletCoveringIsAlmostStructuredGeneralized}
with $a=-3$, $b=3$, $\mu=\mu_{0}^{\left(\beta\right)}=3\cdot2^{\beta/2}$
and $\gamma=\mu^{-1}$, as well as $A=1$ yields constants $L,N,C$
satisfying $L^{-1}\cdot2^{\frac{\beta}{2}n}\leq\left|\xi\right|\leq L\cdot2^{\frac{\beta}{2}n}$
for all $\left(n,m,\varepsilon\right)\in J_{0}^{\left(\beta\right)}$
and all $\xi\in S_{n,m,\varepsilon}^{\left(\beta\right)}$, as well
as $\left|j^{\ast}\right|\leq N$ for all $j\in J^{\left(\beta\right)}$
and finally $\left\Vert Y_{i}^{-1}Y_{j}\right\Vert \leq C$ for all
$j\in J^{\left(\beta\right)}$ and $i\in j^{\ast}$.

Thus, since we have $\CalS^{\left(\beta\right)}=\left(Y_{j}P_{j}'+c_{j}\right)_{j\in J}$
with $\left\{ P_{j}'\with j\in J\right\} $ having only two elements,
in order to establish that $\CalS^{\left(\beta\right)}$ is an almost
structured covering of $\R^{2}$ it suffices to prove $\R^{2}=\bigcup_{j\in J}T_{j}R_{j}'$
for $R_{0}':=\left(-\frac{3}{4},\frac{3}{4}\right)^{2}$ and $R_{j}':=U_{\left(-1,1\right)}^{\left(2^{-\beta/2},2^{\beta/2}\right)}\cup\left[-U_{\left(-1,1\right)}^{\left(2^{-\beta/2},2^{\beta/2}\right)}\right]$,
since clearly each $R_{j}'$ is open with $\overline{R_{j}'}\subset P_{j}'$
and since $\left\{ R_{j}'\with j\in J\right\} $ is finite. But an
analog of equation \eqref{eq:BaseSetTransformationRules} (see equations
\eqref{eq:UnconnectedShearletBaseSetDefinition} and \eqref{eq:ReciprocalShearletCoveringNiceExpression}
for more details) shows
\begin{align*}
Y_{n,m,0}R_{n,m,0}' & =V_{\left(2^{n\left(1-\beta\right)/2}\left(m-1\right),2^{n\left(1-\beta\right)/2}\left(m+1\right)\right)}^{\left(2^{\beta\left(n-1\right)/2},2^{\beta\left(n+1\right)/2}\right)}\\
 & =\left\{ \left(\begin{matrix}\xi\\
\eta
\end{matrix}\right)\in\R^{\ast}\times\R\with\left|\xi\right|\in\left(2^{\frac{\beta}{2}\left(n-1\right)},2^{\frac{\beta}{2}\left(n+1\right)}\right)\text{ and }\frac{\eta}{\xi}\in\left(2^{\frac{n}{2}\left(1-\beta\right)}\left(m-1\right),2^{\frac{n}{2}\left(1-\beta\right)}\left(m+1\right)\right)\right\} 
\end{align*}
for all $\left(n,m,0\right)\in J_{0}^{\left(\beta\right)}$. But recalling
the notation $H_{n}=H_{n}^{\left(\beta\right)}=\left\lceil 2^{n\left(\beta-1\right)/2}\right\rceil $,
we see
\begin{align*}
\bigcup_{m=-H_{n}}^{H_{n}}\left(2^{\frac{n}{2}\left(1-\beta\right)}\left(m-1\right),2^{\frac{n}{2}\left(1-\beta\right)}\left(m+1\right)\right) & =2^{\frac{n}{2}\left(1-\beta\right)}\cdot\bigcup_{m=-H_{n}}^{H_{n}}\left(m-1,\,m+1\right)\\
 & \supset2^{\frac{n}{2}\left(1-\beta\right)}\cdot\left(-\left\lceil \smash{2^{n\left(\beta-1\right)/2}}\right\rceil -1,\,\left\lceil \smash{2^{n\left(\beta-1\right)/2}}\right\rceil +1\right)\\
 & \supset2^{\frac{n}{2}\left(1-\beta\right)}\cdot\left[-2^{n\left(\beta-1\right)/2},\,2^{n\left(\beta-1\right)/2}\right]=\left[-1,1\right]
\end{align*}
and because of $\beta>1$ and since $\left(2^{-1/2}\right)^{2}=\frac{1}{2}<\frac{9}{16}=\left(\frac{3}{4}\right)^{2}$,
we also get
\[
\bigcup_{n=0}^{\infty}\left(2^{\frac{\beta}{2}\left(n-1\right)},2^{\frac{\beta}{2}\left(n+1\right)}\right)\supset\left(2^{-\frac{\beta}{2}},\infty\right)\supset\left(2^{-\frac{1}{2}},\infty\right)\supset\left[3/4,\:\infty\right).
\]
Taken together, this implies
\begin{align*}
\bigcup_{n=0}^{\infty}\;\bigcup_{m=-H_{n}}^{H_{n}}Y_{n,m,0}R_{n,m,0}' & \supset\bigcup_{n=0}^{\infty}\left\{ \left(\begin{matrix}\xi\\
\eta
\end{matrix}\right)\in\R^{\ast}\times\R\with\left|\xi\right|\in\left(2^{\frac{\beta}{2}\left(n-1\right)},2^{\frac{\beta}{2}\left(n+1\right)}\right)\text{ and }\frac{\eta}{\xi}\in\left[-1,1\right]\right\} \\
 & \supset\left\{ \left(\xi,\eta\right)\in\R^{\ast}\times\R\with\left|\xi\right|\in\left[3/4,\:\infty\right)\text{ and }\left|\eta\right|\leq\left|\xi\right|\right\} =:M_{1}
\end{align*}
and therefore also
\begin{align*}
\bigcup_{n=0}^{\infty}\;\bigcup_{m=-H_{n}}^{H_{n}}Y_{n,m,1}R_{n,m,1}' & =R\cdot\left[\bigcup_{n=0}^{\infty}\;\bigcup_{m=-H_{n}}^{H_{n}}Y_{n,m,0}R_{n,m,0}'\right]\\
 & =\left\{ \left(\xi,\eta\right)\in\R^{\ast}\times\R\with\left|\eta\right|\in\left[3/4,\:\infty\right)\text{ and }\left|\xi\right|\leq\left|\eta\right|\right\} =:M_{2}.
\end{align*}

Altogether, we see $\R^{2}=\bigcup_{j\in J}Y_{j}R_{j}'$, since for
$\left(\begin{smallmatrix}\xi\\
\eta
\end{smallmatrix}\right)\in\R^{2}\setminus\left(-\frac{3}{4},\frac{3}{4}\right)^{2}=\R^{2}\setminus\left[Y_{0}R_{0}'\right]$, there are only two cases:

\begin{casenv}
\item We have $\left|\xi\right|\leq\left|\eta\right|$. This implies $\left|\eta\right|\geq\frac{3}{4}$
and thus $\left(\begin{smallmatrix}\xi\\
\eta
\end{smallmatrix}\right)\in M_{2}$, since otherwise $\left(\begin{smallmatrix}\xi\\
\eta
\end{smallmatrix}\right)\in\left(-\frac{3}{4},\frac{3}{4}\right)^{2}$.
\item We have $\left|\eta\right|\leq\left|\xi\right|$. This yields $\left|\xi\right|\geq\frac{3}{4}$
and thus $\left(\begin{smallmatrix}\xi\\
\eta
\end{smallmatrix}\right)\in M_{1}$, since otherwise $\left(\begin{smallmatrix}\xi\\
\eta
\end{smallmatrix}\right)\in\left(-\frac{3}{4},\frac{3}{4}\right)^{2}$.
\end{casenv}
We have thus shown that $\CalS^{\left(\beta\right)}$ is an almost
structured covering of $\R^{2}$, so that it remains to verify the
part of the lemma related to the weight $v$.

But for $j=0$ and $\xi\in S_{0}^{\left(\beta\right)}=\left(-1,1\right)^{2}$,
we simply have $\left(2+L\right)^{-1}\cdot v_{j}\leq v_{j}=1\leq1+\left|\xi\right|\leq3\leq\left(2+L\right)\cdot v_{j}$
since $L\geq1$. Furthermore, for $j=\left(n,m,\varepsilon\right)\in J_{0}^{\left(\beta\right)}$,
we have 
\[
\left(2+L\right)^{-1}\cdot v_{j}\leq L^{-1}\cdot2^{\frac{\beta}{2}n}\leq\left|\xi\right|\leq1+\left|\xi\right|\leq1+L\cdot2^{\frac{\beta}{2}n}\leq\left(1+L\right)\cdot2^{\frac{\beta}{2}n}\leq\left(2+L\right)\cdot v_{j}
\]
for all $\xi\in S_{j}^{\left(\beta\right)}$. Therefore, we have shown
$K^{-1}\cdot v_{j}\leq1+\left|\xi\right|\leq K\cdot v_{j}$ for all
$j\in J^{\left(\beta\right)}$ and $\xi\in S_{j}^{\left(\beta\right)}$
with $K:=2+L$, as claimed in the last part of the lemma.

Finally, assume $S_{j}^{\left(\beta\right)}\cap S_{i}^{\left(\beta\right)}\neq\emptyset$.
For an arbitrary $\xi\in S_{j}^{\left(\beta\right)}\cap S_{i}^{\left(\beta\right)}$,
this implies $v_{i}\leq K\cdot\left(1+\left|\xi\right|\right)\leq K^{2}\cdot v_{j}$
and thus $K^{-2}\leq\frac{v_{i}}{v_{j}}\leq K^{2}$ by symmetry. This
easily yields $\frac{v_{i}^{s}}{v_{j}^{s}}\leq K^{2\left|s\right|}$,
so that $v^{s}$ is $\CalS^{\left(\beta\right)}$-moderate, with $C_{\CalS^{\left(\beta\right)},v^{s}}\leq K^{2\left|s\right|}$,
as claimed.
\end{proof}
Since we now know that $\CalS^{\left(\beta\right)}$ is an almost
structured covering of $\R^{2}$ and that $v^{s}$ is $\CalS^{\left(\beta\right)}$-moderate,
we see precisely as in the remark after Definition \ref{def:AlphaShearletSmoothnessSpaces}
that the reciprocal $\beta$-shearlet smoothness spaces that we now
define are well-defined Quasi-Banach spaces. As for the unconnected
$\alpha$-shearlet smoothness spaces, the following definition will
only be of transitory relevance, since we will immediately show that
the newly defined reciprocal $\beta$-shearlet smoothness spaces are
identical with the previously defined $\alpha$-shearlet smoothness
spaces, for $\alpha=\beta^{-1}$.
\begin{defn}
\label{def:ReciprocalShearletSmoothnessSpaces}For $\beta\in\left(1,\infty\right)$,
$p,q\in\left(0,\infty\right]$ and $s\in\R$, we define the \textbf{reciprocal
$\beta$-shearlet smoothness space} $\mathscr{S}_{\beta,s}^{p,q}\left(\R^{2}\right)$
associated to these parameters as
\[
\mathscr{S}_{\beta,s}^{p,q}\left(\R^{2}\right):=\DecompSp{\smash{\CalS^{\left(\beta\right)}}}p{\ell_{v^{s}}^{q}}{},
\]
where the covering $\CalS^{\left(\beta\right)}$ and the weight $v^{s}$
are defined as in Definition \ref{def:ReciprocalBetaShearletCovering}
and Corollary \ref{cor:ReciprocalShearletCoveringAlmostStructured},
respectively.
\end{defn}
Next, we want to show $\mathscr{S}_{\beta,s}^{p,q}\left(\R^{2}\right)=\mathscr{S}_{\beta^{-1},s}^{p,q}\left(\R^{2}\right)$.
To this end, we will utilize the general theory of embeddings between
decomposition spaces that was developed in \cite{DecompositionEmbedding}.
The main prerequisite for an application of this theory is to have
a certain compatibility between the two relevant coverings. This compatibility
is established in the next lemma:
\begin{lem}
\label{lem:AlphaShearletSubordinateToReciprocalShearlet}Let $\beta\in\left(1,\infty\right)$
and set $\alpha:=\beta^{-1}\in\left(0,1\right)$. Then, for each $i\in I^{\left(\alpha\right)}$,
there is some $j=j_{i}\in J^{\left(\beta\right)}$ satisfying $S_{i}^{\left(\alpha\right)}\subset S_{j}^{\left(\beta\right)}$.
\end{lem}
\begin{rem*}
The set $P$ in Definition \ref{def:ReciprocalBetaShearletCovering}
is chosen precisely to make the preceding lemma true. In general,
one could have chosen $P$ to be smaller.
\end{rem*}
\begin{proof}
For $i=0$, we clearly have $S_{0}^{\left(\alpha\right)}=\left(-1,1\right)^{2}=S_{0}^{\left(\beta\right)}$,
so that we can assume $i=\left(n,m,\varepsilon,\delta\right)\in I_{0}^{\left(\alpha\right)}$
in the following. Let us first consider the case $\varepsilon=1$
and $\delta=0$. Define $j:=\left\lfloor 2\alpha n\right\rfloor \in\N_{0}$
and observe $2\alpha n-1<j\leq2\alpha n$. Recall the notation $\mu_{0}=3\cdot2^{\beta/2}$
from Definition \ref{def:ReciprocalBetaShearletCovering} and note
for arbitrary $\ell\in\Z$ with $\left|\ell\right|\leq H_{j}$ that
\[
S_{\left(j,\ell,0\right)}^{\left(\beta\right)}\supset{\rm diag}\left(2^{\frac{\beta}{2}j},\,2^{\frac{j}{2}}\right)\cdot\left(\begin{matrix}1 & 0\\
\ell & 1
\end{matrix}\right)U_{\left(-3,3\right)}^{\left(\mu_{0}^{-1},\mu_{0}\right)}=U_{\left(2^{j\frac{1-\beta}{2}}\left(\ell-3\right),\,2^{j\frac{1-\beta}{2}}\left(\ell+3\right)\right)}^{\left(2^{\beta j/2}\cdot\mu_{0}^{-1},\,2^{\beta j/2}\cdot\mu_{0}\right)}\quad\text{and}\quad S_{i}^{\left(\alpha\right)}=U_{\left(2^{n\left(\alpha-1\right)}\left(m-1\right),\,2^{n\left(\alpha-1\right)}\left(m+1\right)\right)}^{\left(2^{n}/3,\,3\cdot2^{n}\right)},
\]
thanks to equation \eqref{eq:BaseSetTransformationRules}. Consequently,
it suffices to show that we have $\left(2^{\beta j/2}\cdot\mu_{0}^{-1},\,2^{\beta j/2}\cdot\mu_{0}\right)\supset\left(2^{n}/3,\,3\cdot2^{n}\right)$
and that one can choose $\ell\in\Z$ with $\left|\ell\right|\leq H_{j}$
such that 
\begin{equation}
\left(2^{j\left(1-\beta\right)/2}\left(\ell-3\right),\,2^{j\left(1-\beta\right)/2}\left(\ell+3\right)\right)\supset\left(2^{n\left(\alpha-1\right)}\left(m-1\right),\,2^{n\left(\alpha-1\right)}\left(m+1\right)\right).\label{eq:AlphaShearletSubordinateToReciprocalShearletTargetInclusion}
\end{equation}

The first of these inclusions is straightforward to verify: We have
$\mu_{0}=3\cdot2^{\beta/2}\geq3$ and $j\leq2\alpha n=\frac{2}{\beta}n$,
so that $2^{\beta j/2}\mu_{0}^{-1}\leq\frac{1}{3}\cdot2^{\beta j/2}\leq\frac{1}{3}\cdot2^{n}$.
Furthermore, since $j>2\alpha n-1$,
\[
2^{\beta j/2}\cdot\mu_{0}\geq2^{\frac{\beta}{2}\left(2\alpha n-1\right)}\cdot\mu_{0}=2^{n}\cdot2^{-\frac{\beta}{2}}\cdot3\cdot2^{\beta/2}=3\cdot2^{n}.
\]
Thus, all that remains is to show that one can choose $\ell$ suitably.
To this end, let $\ell_{0}:=\left\lfloor 2^{n\left(\alpha-1\right)+\left(\beta-1\right)\frac{j}{2}}\left(m-1\right)\right\rfloor \in\Z$
and observe
\begin{align*}
\ell_{0}\leq2^{n\left(\alpha-1\right)+\left(\beta-1\right)\frac{j}{2}}\left(m-1\right) & \leq2^{n\left(\alpha-1\right)+\left(\beta-1\right)\frac{j}{2}}\left(\left|m\right|-1\right)\\
\left({\scriptstyle \text{since }\left|m\right|-1\leq\left\lceil 2^{n\left(1-\alpha\right)}\right\rceil -1<2^{n\left(1-\alpha\right)}}\right) & \leq2^{\left(\beta-1\right)\frac{j}{2}}\leq\left\lceil \smash{2^{\left(\beta-1\right)\frac{j}{2}}}\right\rceil =H_{j}^{\left(\beta\right)}.
\end{align*}
We now distinguish two cases:

\textbf{Case 1}: We have $\ell_{0}\geq-H_{j}^{\left(\beta\right)}$.
In this case, we set $\ell:=\ell_{0}$ and note $\left|\ell\right|\leq H_{j}^{\left(\beta\right)}$.
Furthermore, we observe
\[
2^{\frac{j}{2}\left(1-\beta\right)}\left(\ell-3\right)<2^{\frac{j}{2}\left(1-\beta\right)}\ell\leq2^{n\left(\alpha-1\right)}\left(m-1\right).
\]
Finally, since we have $\ell+1>2^{n\left(\alpha-1\right)+\left(\beta-1\right)\frac{j}{2}}\left(m-1\right)$,
we get
\begin{align*}
2^{\frac{j}{2}\left(1-\beta\right)}\left(\ell+3\right) & >2^{\frac{j}{2}\left(1-\beta\right)}\left[2+2^{n\left(\alpha-1\right)+\left(\beta-1\right)\frac{j}{2}}\left(m-1\right)\right]\\
 & =2\cdot2^{\frac{j}{2}\left(1-\beta\right)}+2^{n\left(\alpha-1\right)}\left(m-1\right)\\
\left({\scriptstyle \text{since }1-\beta<0\text{ and }j\leq2\alpha n}\right) & \geq2\cdot2^{\alpha n\left(1-\beta\right)}+2^{n\left(\alpha-1\right)}\left(m-1\right)\\
 & =2\cdot2^{n\left(\alpha-1\right)}+2^{n\left(\alpha-1\right)}\left(m-1\right)=2^{n\left(\alpha-1\right)}\left(m+1\right).
\end{align*}
The last two displayed equations establish the desired inclusion \eqref{eq:AlphaShearletSubordinateToReciprocalShearletTargetInclusion},
so that indeed $S_{i}^{\left(\alpha\right)}\subset S_{j,\ell,0}^{\left(\beta\right)}$.

\medskip{}

\textbf{Case 2}: We have $\ell_{0}<-H_{j}^{\left(\beta\right)}$.
This implies $2^{n\left(\alpha-1\right)}\left(m-1\right)\leq-1$,
since we would otherwise have 
\[
\ell_{0}=\left\lfloor 2^{n\left(\alpha-1\right)+\left(\beta-1\right)\frac{j}{2}}\left(m-1\right)\right\rfloor \geq\left\lfloor -2^{\left(\beta-1\right)\frac{j}{2}}\right\rfloor =-\left\lceil 2^{\left(\beta-1\right)\frac{j}{2}}\right\rceil =-H_{j}^{\left(\beta\right)}.
\]
Consequently, we get for $\ell:=-H_{j}^{\left(\beta\right)}\in\Z$
that
\begin{align*}
2^{n\left(\alpha-1\right)}\left(m+1\right)=2^{n\left(\alpha-1\right)}\left(m-1\right)+2\cdot2^{n\left(\alpha-1\right)} & \leq-1+2\cdot2^{n\left(\alpha-1\right)}\\
\left({\scriptstyle \text{since }\alpha-1<0\text{ and }n\geq\frac{j}{2\alpha}}\right) & \leq-1+2\cdot2^{\frac{j}{2\alpha}\left(\alpha-1\right)}\\
 & =2^{\frac{j}{2}\left(1-\beta\right)}\left[-2^{\frac{j}{2}\left(\beta-1\right)}+2\right]\\
\left({\scriptstyle \text{since }2^{\frac{j}{2}\left(\beta-1\right)}\geq\left\lceil \smash{2^{\frac{j}{2}\left(\beta-1\right)}}\right\rceil -1=-\ell-1}\right) & \leq2^{\frac{j}{2}\left(1-\beta\right)}\left(\ell+3\right).
\end{align*}
Finally, recall $\left|m\right|\leq\left\lceil 2^{n\left(1-\alpha\right)}\right\rceil \leq1+2^{n\left(1-\alpha\right)}$,
so that
\begin{align*}
2^{n\left(\alpha-1\right)}\left(m-1\right)\geq-2^{n\left(\alpha-1\right)}\left(\left|m\right|+1\right) & \geq-2^{n\left(\alpha-1\right)}\left(2^{n\left(1-\alpha\right)}+2\right)\\
 & =-1-2\cdot2^{n\left(\alpha-1\right)}\\
\left({\scriptstyle \text{since }\alpha-1<0\text{ and }n\geq\frac{j}{2\alpha}}\right) & \geq-1-2\cdot2^{\frac{j}{2\alpha}\left(\alpha-1\right)}\\
 & =2^{\frac{j}{2}\left(1-\beta\right)}\cdot\left(-2^{\frac{j}{2}\left(\beta-1\right)}-2\right)\\
\left({\scriptstyle \text{since }2^{\frac{j}{2}\left(\beta-1\right)}\leq\left\lceil \smash{2^{\frac{j}{2}\left(\beta-1\right)}}\right\rceil =-\ell}\right) & \geq2^{\frac{j}{2}\left(1-\beta\right)}\cdot\left(\ell-2\right)\\
 & \geq2^{\frac{j}{2}\left(1-\beta\right)}\cdot\left(\ell-3\right).
\end{align*}
We have thus again established the inclusion \eqref{eq:AlphaShearletSubordinateToReciprocalShearletTargetInclusion},
so that $S_{i}^{\left(\alpha\right)}\subset S_{\left(j,\ell,0\right)}^{\left(\beta\right)}$.

\medskip{}

Up to now, we have constructed for $i=\left(n,m,\varepsilon,\delta\right)\in I_{0}^{\left(\alpha\right)}$
with $\varepsilon=1$ and $\delta=0$ some $\left(j,\ell,0\right)\in J_{0}^{\left(\beta\right)}$
with $S_{i}^{\left(\alpha\right)}\subset S_{j}^{\left(\beta\right)}$,
so that it remains to consider the general case $\varepsilon\in\left\{ \pm1\right\} $
and $\delta\in\left\{ 0,1\right\} $. But since the base-set $P$
from Definition \ref{def:ReciprocalBetaShearletCovering} satisfies
$P=-P$, we have $S_{j}^{\left(\beta\right)}=-S_{j}^{\left(\beta\right)}$,
so that $S_{n,m,-1,0}^{\left(\alpha\right)}=-S_{n,m,1,0}^{\left(\alpha\right)}\subset-S_{j,\ell,0}^{\left(\beta\right)}=S_{j,\ell,0}^{\left(\beta\right)}$,
assuming $S_{n,m,1,0}^{\left(\alpha\right)}\subset S_{j,\ell,0}^{\left(\beta\right)}$.
Finally, assuming that $S_{n,m,\varepsilon,0}^{\left(\alpha\right)}\subset S_{j,\ell,0}^{\left(\beta\right)}$,
we get
\[
S_{n,m,\varepsilon,1}^{\left(\alpha\right)}=R\cdot S_{n,m,\varepsilon,0}^{\left(\alpha\right)}\subset R\cdot S_{j,\ell,0}^{\left(\beta\right)}=S_{j,\ell,1}^{\left(\beta\right)}.
\]
This completes the proof.
\end{proof}
Now, we can finally show that the reciprocal $\beta$-shearlet smoothness
spaces are identical to the $\alpha$-shearlet smoothness spaces from
Section \ref{sec:AlphaShearletSmoothnessDefinition}.
\begin{lem}
\label{lem:ReciprocalShearletSmoothnessSpacesAreBoring}Let $\beta\in\left(1,\infty\right)$,
$s\in\R$ and $p,q\in\left(0,\infty\right]$. Then
\[
\mathscr{S}_{\beta,s}^{p,q}\left(\R^{2}\right)=\mathscr{S}_{\beta^{-1},s}^{p,q}\left(\R^{2}\right).\qedhere
\]
\end{lem}
\begin{proof}
Set $\alpha:=\beta^{-1}\in\left(0,1\right)$ for brevity. As in the
proof of Lemma \ref{lem:UnconnectedAlphaShearletSmoothnessIsBoring},
we want to invoke \cite[Lemma 6.11, part (2)]{DecompositionEmbedding},
with the choice $\CalP:=\CalS^{\left(\alpha\right)}$ and $\CalQ:=\CalS^{\left(\beta\right)}$,
recalling that $\mathscr{S}_{\beta^{-1},s}^{p,q}\left(\R^{2}\right)=\DecompSp{\CalS^{\left(\alpha\right)}}p{\ell_{w^{s}}^{q}}{}=\Fourier^{-1}\left[\FourierDecompSp{\CalS^{\left(\alpha\right)}}p{\ell_{w^{s}}^{q}}{}\right]$
and likewise $\mathscr{S}_{\beta,s}^{p,q}\left(\R^{2}\right)=\Fourier^{-1}\left[\FourierDecompSp{\CalS^{\left(\beta\right)}}p{\ell_{v^{s}}^{q}}{}\right]$.

To this end, we first have to verify that we have $v_{j}^{s}\asymp w_{i}^{s}$
if $S_{i}^{\left(\alpha\right)}\cap S_{j}^{\left(\beta\right)}\neq\emptyset$
and that the coverings $\CalS^{\left(\alpha\right)}$ and $\CalS^{\left(\beta\right)}$
are \textbf{weakly equivalent}. This means that
\[
\sup_{i\in I^{\left(\alpha\right)}}\left|\left\{ j\in J^{\left(\beta\right)}\with S_{j}^{\left(\beta\right)}\cap S_{i}^{\left(\alpha\right)}\neq\emptyset\right\} \right|<\infty\qquad\text{ and }\qquad\sup_{j\in J^{\left(\beta\right)}}\left|\left\{ i\in I^{\left(\alpha\right)}\with S_{i}^{\left(\alpha\right)}\cap S_{j}^{\left(\beta\right)}\neq\emptyset\right\} \right|<\infty.
\]

For the first point, let $K\geq1$ as in Corollary \ref{cor:ReciprocalShearletCoveringAlmostStructured},
i.e., such that $K^{-1}\cdot v_{j}\leq1+\left|\xi\right|\leq K\cdot v_{j}$
for all $j\in J^{\left(\beta\right)}$ and all $\xi\in S_{j}^{\left(\beta\right)}$.
Likewise, Lemma \ref{lem:AlphaShearletWeightIsModerate} shows $\frac{1}{13}\cdot w_{i}\leq1+\left|\xi\right|\leq13\cdot w_{i}$
for all $i\in I^{\left(\alpha\right)}$ and $\xi\in S_{i}^{\left(\alpha\right)}$.
Consequently, if $S_{i}^{\left(\alpha\right)}\cap S_{j}^{\left(\beta\right)}\neq\emptyset$,
we can choose some $\xi\in S_{i}^{\left(\alpha\right)}\cap S_{j}^{\left(\beta\right)}$,
so that
\[
\left(13K\right)^{-1}\cdot w_{i}\leq K^{-1}\cdot\left(1+\left|\xi\right|\right)\leq v_{j}\leq K\cdot\left(1+\left|\xi\right|\right)\leq13K\cdot w_{i}.
\]
Consequently, we get 
\begin{equation}
\left(13K\right)^{-\left|t\right|}\leq\frac{v_{j}^{t}}{w_{i}^{t}}\leq\left(13K\right)^{\left|t\right|}\qquad\forall t\in\R\quad\text{ if }i\in I^{\left(\alpha\right)}\text{ and }j\in J^{\left(\beta\right)}\text{ with }S_{i}^{\left(\alpha\right)}\cap S_{j}^{\left(\beta\right)}\neq\emptyset.\label{eq:ReciprocalShearletCoveringWeightEquivalentToUsualOne}
\end{equation}

It remains to show that $\CalS^{\left(\alpha\right)}$ and $\CalS^{\left(\beta\right)}$
are weakly equivalent. To this end, let $i\in I^{\left(\alpha\right)}$
be arbitrary and note from Lemma \ref{lem:AlphaShearletSubordinateToReciprocalShearlet}
that $S_{i}^{\left(\alpha\right)}\subset S_{j_{i}}^{\left(\beta\right)}$
for some $j_{i}\in J^{\left(\beta\right)}$. Thus, for arbitrary $j\in J^{\left(\beta\right)}$
with $\emptyset\neq S_{j}^{\left(\beta\right)}\cap S_{i}^{\left(\alpha\right)}\neq\emptyset$,
we get $\emptyset\subsetneq S_{j}^{\left(\beta\right)}\cap S_{i}^{\left(\alpha\right)}\subset S_{j}^{\left(\beta\right)}\cap S_{j_{i}}^{\left(\beta\right)}$
and thus $j\in j_{i}^{\ast}$. This implies
\[
\sup_{i\in I^{\left(\alpha\right)}}\left|\left\{ j\in J^{\left(\beta\right)}\with S_{j}^{\left(\beta\right)}\cap S_{i}^{\left(\alpha\right)}\neq\emptyset\right\} \right|\leq\sup_{i\in I^{\left(\alpha\right)}}\left|j_{i}^{\ast}\right|\leq\sup_{j\in J^{\left(\beta\right)}}\left|j^{\ast}\right|=N_{\CalS^{\left(\beta\right)}}<\infty,
\]
since $\CalS^{\left(\beta\right)}$ is an almost structured covering
of $\R^{2}$.

For the second part of weak equivalence, we have to work harder: Let
$j\in J^{\left(\beta\right)}$ be arbitrary. For each $i\in I^{\left(\alpha\right)}$
with $S_{i}^{\left(\alpha\right)}\cap S_{j}^{\left(\beta\right)}\neq\emptyset$,
Lemma \ref{lem:AlphaShearletSubordinateToReciprocalShearlet} yields
some $j_{i}\in J^{\left(\beta\right)}$ satisfying $S_{i}^{\left(\alpha\right)}\subset S_{j_{i}}^{\left(\beta\right)}$.
Hence, $\emptyset\subsetneq S_{i}^{\left(\alpha\right)}\cap S_{j}^{\left(\beta\right)}\subset S_{j_{i}}^{\left(\beta\right)}\cap S_{j}^{\left(\beta\right)}$,
so that $j_{i}\in j^{\ast}$. Thus, with $\left[\smash{S_{j}^{\left(\beta\right)}}\right]^{\ast}:=\bigcup_{\ell\in j^{\ast}}S_{\ell}^{\left(\beta\right)}$,
we have shown $S_{i}^{\left(\alpha\right)}\subset S_{j_{i}}^{\left(\beta\right)}\subset\left[\smash{S_{j}^{\left(\beta\right)}}\right]^{\ast}$
for arbitrary $i\in I^{\left(\alpha\right)}$ with $S_{i}^{\left(\alpha\right)}\cap S_{j}^{\left(\beta\right)}\neq\emptyset$.

Now, we will need the easily verifiable identities $\left|\det\smash{T_{i}^{\left(\alpha\right)}}\right|=w_{i}^{1+\alpha}$
and $\left|\det\smash{Y_{j}^{\left(\beta\right)}}\right|=v_{j}^{1+\beta^{-1}}=v_{j}^{1+\alpha}$
for $i\in I^{\left(\alpha\right)}$ and $j\in J^{\left(\beta\right)}$.
To use these identities, set $M_{j}:=\left\{ i\in I^{\left(\alpha\right)}\with S_{i}^{\left(\alpha\right)}\cap S_{j}^{\left(\beta\right)}\neq\emptyset\right\} $,
as well as 
\[
C_{1}:=\min\left\{ \lambda_{2}\left(\left(-1,1\right)^{2}\right),\lambda_{2}\left(\vphantom{U^{\left(3\right)}}\smash{U_{\left(-1,1\right)}^{\left(3^{-1},3\right)}}\right)\right\} >0\quad\text{ and }\quad C_{2}:=\max\left\{ \lambda_{2}\left(\left(-1,1\right)^{2}\right),\,\lambda_{2}\left(P\right)\right\} >0,
\]
with $P$ as in Definition \ref{def:ReciprocalBetaShearletCovering}.
We clearly have $\sum_{i\in M_{j}}\Indicator_{S_{i}^{\left(\alpha\right)}}\leq\sum_{i\in I^{\left(\alpha\right)}}\Indicator_{S_{i}^{\left(\alpha\right)}}\leq N_{\CalS^{\left(\alpha\right)}}$.
But because of $S_{i}^{\left(\alpha\right)}\subset\left[\smash{S_{j}^{\left(\beta\right)}}\right]^{\ast}$
for $i\in M_{j}$, this implies
\begin{align*}
0<C_{1}\cdot\sum_{i\in M_{j}}w_{i}^{1+\alpha}=C_{1}\cdot\sum_{i\in M_{j}}\left|\det\smash{T_{i}^{\left(\alpha\right)}}\right| & \leq\sum_{i\in M_{j}}\lambda_{2}\left(\smash{S_{i}^{\left(\alpha\right)}}\right)=\int_{\R^{2}}\sum_{i\in M_{j}}\Indicator_{S_{i}^{\left(\alpha\right)}}\left(\xi\right)\d\xi\\
 & \leq N_{\CalS^{\left(\alpha\right)}}\cdot\lambda_{2}\left(\left[\smash{S_{j}^{\left(\beta\right)}}\right]^{\ast}\right)\leq N_{\CalS^{\left(\alpha\right)}}\cdot\sum_{\ell\in j^{\ast}}\lambda_{2}\left(\smash{S_{\ell}^{\left(\beta\right)}}\right)\\
 & \leq C_{2}N_{\CalS^{\left(\alpha\right)}}\cdot\sum_{\ell\in j^{\ast}}\left|\det\smash{Y_{\ell}^{\left(\beta\right)}}\right|=C_{2}N_{\CalS^{\left(\alpha\right)}}\cdot\sum_{\ell\in j^{\ast}}v_{\ell}^{1+\alpha}\\
\left({\scriptstyle \text{Corollary }\ref{cor:ReciprocalShearletCoveringAlmostStructured}}\right) & \leq C_{2}N_{\CalS^{\left(\alpha\right)}}\cdot\left|j^{\ast}\right|\cdot K^{2\left(1+\alpha\right)}\cdot v_{j}^{1+\alpha}\\
\left({\scriptstyle \text{eq. }\eqref{eq:ReciprocalShearletCoveringWeightEquivalentToUsualOne}\text{ and }S_{i}^{\left(\alpha\right)}\cap S_{j}^{\left(\beta\right)}\neq\emptyset\text{ for }i\in M_{j}}\right) & \leq C_{2}N_{\CalS^{\left(\alpha\right)}}N_{\CalS^{\left(\beta\right)}}\cdot K^{2\left(1+\alpha\right)}\cdot\left(13K\right)^{1+\alpha}\cdot\inf_{i\in M_{j}}w_{i}^{1+\alpha}.
\end{align*}
Now, observe $\inf_{i\in M_{j}}w_{i}^{1+\alpha}\geq1>0$, so that
the preceding inequality shows that $M_{j}$ is finite with
\[
\left|M_{j}\right|\leq C_{1}^{-1}\cdot C_{2}N_{\CalS^{\left(\alpha\right)}}N_{\CalS^{\left(\beta\right)}}\cdot K^{2\left(1+\alpha\right)}\cdot\left(13K\right)^{1+\alpha},
\]
where the right-hand side is independent of $j\in J^{\left(\beta\right)}$.

\medskip{}

We have thus verified the main requirements of \cite[Lemma 6.11]{DecompositionEmbedding}.
But since we want to apply that lemma also in case of $p\in\left(0,1\right)$,
we still have to verify the extra condition that $\CalP=\CalS^{\left(\alpha\right)}=\left(\smash{T_{i}^{\left(\alpha\right)}}Q_{i}'\right)_{i\in I^{\left(\alpha\right)}}$
is almost subordinate to $\CalQ=\CalS^{\left(\beta\right)}=\left(\smash{Y_{j}^{\left(\beta\right)}}P_{j}'\right)_{j\in J^{\left(\beta\right)}}$
and that we have 
\[
\left|\det\left[\left(\smash{T_{i}^{\left(\alpha\right)}}\vphantom{T_{i}^{\alpha}}\right)^{-1}Y_{j}^{\left(\beta\right)}\right]\right|\lesssim1\quad\text{if }\quad S_{j}^{\left(\beta\right)}\cap S_{i}^{\left(\alpha\right)}\neq\emptyset
\]
But Lemma \ref{lem:AlphaShearletSubordinateToReciprocalShearlet}
shows that $\CalP=\CalS^{\left(\alpha\right)}$ is subordinate (and
thus also almost subordinate, cf.\@ \cite[Definition 2.10]{DecompositionEmbedding})
to $\CalQ=\CalS^{\left(\beta\right)}$. Furthermore, in case of $S_{j}^{\left(\beta\right)}\cap S_{i}^{\left(\alpha\right)}\neq\emptyset$,
equation \eqref{eq:ReciprocalShearletCoveringWeightEquivalentToUsualOne}
yields
\[
\left|\det\left[\left(\smash{T_{i}^{\left(\alpha\right)}}\right)^{-1}Y_{j}^{\left(\beta\right)}\right]\right|=\left(w_{i}^{1+\alpha}\right)^{-1}\cdot v_{j}^{1+\alpha}\asymp1,
\]
as desired. The claim now follows from \cite[Lemma 6.11]{DecompositionEmbedding}.
\end{proof}
Now, we show that a suitable $\beta$-shearlet system generated by
\emph{bandlimited} functions yields a Banach frame for the reciprocal
shearlet smoothness spaces. We restrict ourselves to bandlimited functions,
since this simplifies the proof.

But first, we review the precise definition of a $\beta$-shearlet
system from \cite[Definition 3.10]{AlphaMolecules}.
\begin{defn}
\label{def:BetaShearletSystem}For $c\in\left(0,\infty\right)$ and
$\beta\in\left(1,\infty\right)$ and given generators $\varphi,\psi,\theta\in L^{2}\left(\R^{2}\right)$,
the \textbf{cone-adapted $\beta$-shearlet system} ${\rm SH}\left(\varphi,\psi,\theta;c,\beta\right)$
with sampling density $c$ generated by $\varphi,\psi,\theta$ is
defined as
\[
{\rm SH}\left(\varphi,\psi,\theta;\,c,\beta\right):=\Phi\left(\varphi;\,c,\beta\right)\cup\Psi\left(\psi;\,c,\beta\right)\cup\Theta\left(\theta;\,c,\beta\right),
\]
where
\begin{align*}
\Phi\left(\varphi;\,c,\beta\right) & =\left\{ \varphi\left(\mybullet-ck\right)\with k\in\Z^{2}\right\} ,\\
\Psi\left(\psi;\,c,\beta\right) & =\left\{ 2^{j\left(\beta+1\right)/4}\cdot\psi\left(S_{\ell}A_{\beta^{-1},2^{\beta j/2}}\mybullet-ck\right)\with j\in\N_{0},k\in\Z^{2}\text{ and }\ell\in\Z\text{ with }\left|\ell\right|\leq\left\lceil \smash{2^{j\left(\beta-1\right)/2}}\right\rceil \right\} ,\\
\Theta\left(\theta;\,c,\beta\right) & =\left\{ 2^{j\left(\beta+1\right)/4}\cdot\theta\left(S_{\ell}^{T}\widetilde{A}_{\beta^{-1},2^{\beta j/2}}\mybullet-ck\right)\with j\in\N_{0},k\in\Z^{2}\text{ and }\ell\in\Z\text{ with }\left|\ell\right|\leq\left\lceil \smash{2^{j\left(\beta-1\right)/2}}\right\rceil \right\} ,
\end{align*}
where $A_{\alpha,s}={\rm diag}\left(s,\,s^{\alpha}\right)$ and $\widetilde{A}_{\alpha,s}={\rm diag}\left(s^{\alpha},\,s\right)$,
as well as $S_{\ell}=\left(\begin{smallmatrix}1 & \ell\\
0 & 1
\end{smallmatrix}\right)$ for $\alpha\in\left[0,1\right]$, $s\in\left(0,\infty\right)$ and
$\ell\in\R$.
\end{defn}
\begin{prop}
\label{prop:BanachFrameForReciprocalShearletSmoothnessSpace}Let $\varphi,\psi\in\Schwartz\left(\R^{2}\right)$
with $\widehat{\varphi},\widehat{\psi}\in\TestFunctionSpace{\R^{2}}$
and the following additional properties:

\begin{enumerate}
\item We have $\widehat{\varphi}\left(\xi\right)\neq0$ for all $\xi\in\left[-1,1\right]^{2}$.
\item We have $\widehat{\psi}\left(\xi\right)\neq0$ for all $\xi\in\overline{P}$
with $P$ as in Definition \ref{def:ReciprocalBetaShearletCovering}.
\item We have $\supp\widehat{\psi}\subset\R^{\ast}\times\R$.
\end{enumerate}
Then, for $\beta\in\left(1,\infty\right)$, $p_{0},q_{0}\in\left(0,1\right]$
and $s_{0},s_{1}\in\R$ with $s_{0}\leq s_{1}$, there is some $\delta_{0}=\delta_{0}\left(\beta,p_{0},q_{0},s_{0},s_{1},\varphi,\psi\right)>0$
such that for every $0<\delta\leq\delta_{0}$, all $p\in\left[p_{0},\infty\right]$,
all $q\in\left[q_{0},\infty\right]$ and all $s\in\R$ with $s_{0}\leq s\leq s_{1}$,
the family ${\rm SH}\left(\varphi,\psi,\theta;\delta,\beta\right)$
forms a Banach frame for $\mathscr{S}_{\beta,s}^{p,q}\left(\R^{2}\right)$,
where $\theta:=\psi\circ R$, i.e., $\theta\left(x,y\right)=\psi\left(y,x\right)$.

Precisely, this means with the coefficient space $C_{v^{s}}^{p,q}$
as in Definition \ref{def:CoefficientSpace} (with $\CalQ=\CalS^{\left(\beta\right)}$
and $w=v^{s}$) and with
\[
\gamma^{\left[i,k,\delta\right]}:=\begin{cases}
\varphi\left(\mybullet-\delta k\right), & \text{if }i=0\\
2^{j\left(\beta+1\right)/4}\cdot\psi\left(S_{\ell}A_{\beta^{-1},2^{\beta j/2}}\mybullet-\delta k\right), & \text{if }i=\left(j,\ell,0\right),\\
2^{j\left(\beta+1\right)/4}\cdot\theta\left(S_{\ell}^{T}\widetilde{A}_{\beta^{-1},2^{\beta j/2}}\mybullet-\delta k\right), & \text{if }i=\left(j,\ell,1\right)
\end{cases}
\]
for $i\in J^{\left(\beta\right)}$ and $k\in\Z^{2}$ that the following
hold:

\begin{enumerate}
\item For each $0<\delta\leq1$, the analysis map
\[
A^{\left(\delta\right)}:\mathscr{S}_{\beta,s}^{p,q}\left(\smash{\R^{2}}\right)\to C_{v^{s}}^{p,q},f\mapsto\left(\left\langle f,\,\smash{\gamma^{\left[i,k,\delta\right]}}\right\rangle _{Z'\left(\R^{2}\right),Z\left(\R^{2}\right)}\right)_{i\in J^{\left(\beta\right)},\,k\in\Z^{2}}
\]
is well-defined and bounded.
\item For all $0<\delta\leq\delta_{0}$, there is a bounded linear reconstruction
map $R^{\left(\delta\right)}:C_{v^{s}}^{p,q}\to\mathscr{S}_{\beta,s}^{p,q}\left(\smash{\R^{2}}\right)$
satisfying $R^{\left(\delta\right)}\circ A^{\left(\delta\right)}=\identity_{\mathscr{S}_{\beta,s}^{p,q}\left(\smash{\R^{2}}\right)}$.
\item We have the following consistency statement: If $f\in\mathscr{S}_{\beta,s}^{p,q}\left(\smash{\R^{2}}\right)$
and if $p_{1}\in\left[p_{0},\infty\right]$ and $q_{1}\in\left[q_{0},\infty\right]$
and $s_{0}\leq r\leq s_{1}$, then we have the following equivalence:
\[
f\in\mathscr{S}_{\beta,r}^{p_{1},q_{1}}\left(\smash{\R}^{2}\right)\quad\Longleftrightarrow\quad A^{\left(\delta\right)}f\in C_{v^{r}}^{p_{1},q_{1}}.\qedhere
\]
\end{enumerate}
\end{prop}
\begin{proof}
We want to verify that Theorem \ref{thm:BanachFrameTheorem} applies
in the current setting, i.e., with $\CalQ=\CalS^{\left(\beta\right)}=\left(\smash{Y_{j}^{\left(\beta\right)}}P_{j}'\right)_{j\in J^{\left(\beta\right)}}$.
To this end, we first recall the notation introduced in Assumption
\ref{assu:CrashCourseStandingAssumptions}: If we set $n:=2$ and
$Q_{0}^{\left(1\right)}:=P$ with $\mu_{0}=3\cdot2^{\beta/2}$ and
$P=U_{\left(-3,3\right)}^{\left(\mu_{0}^{-1},\mu_{0}\right)}\cup\left(\vphantom{U^{\left(\mu\right)}}-\smash{U_{\left(-3,3\right)}^{\left(\mu_{0}^{-1},\mu_{0}\right)}}\right)$
as in Definition \ref{def:ReciprocalBetaShearletCovering}, as well
as $Q_{0}^{\left(2\right)}:=\left(-1,1\right)^{2}$ and finally $k_{j}:=1$
for $j\in J_{0}^{\left(\beta\right)}$ and $k_{0}:=2$, then we have
$P_{j}'=Q_{0}^{\left(k_{j}\right)}$ for all $j\in J^{\left(\beta\right)}$.

Now, we set $\gamma_{1}^{\left(0\right)}:=\psi$ and $\gamma_{2}^{\left(0\right)}:=\varphi$,
as well as $\varepsilon:=1$. With these choices, we want to verify
the prerequisites of Theorem \ref{thm:BanachFrameTheorem}. We clearly
have $\gamma_{k}^{\left(0\right)},\Fourier\gamma_{k}^{\left(0\right)}\in\Schwartz\left(\R^{2}\right)\subset W^{1,1}\left(\R^{2}\right)\cap W^{1,\infty}\left(\R^{2}\right)\cap C^{\infty}\left(\R^{2}\right)$
and all partial derivatives of these functions are (polynomially)
bounded, so that the first two prerequisites of Theorem \ref{thm:BanachFrameTheorem}
clearly hold. Next, our assumptions on $\widehat{\varphi},\widehat{\psi}$
ensure that $\Fourier\gamma_{1}^{\left(0\right)}\left(\xi\right)=\widehat{\psi}\left(\xi\right)\neq0$
for all $\xi\in\overline{P}=\overline{Q_{0}^{\left(1\right)}}$ and
likewise that $\Fourier\gamma_{2}^{\left(0\right)}\left(\xi\right)=\widehat{\varphi}\left(\xi\right)\neq0$
for all $\xi\in\left[-1,1\right]^{2}=\overline{Q_{0}^{\left(2\right)}}$.

Consequently, since we are interested in the decomposition space $\mathscr{S}_{\beta,s}^{p,q}\left(\R^{2}\right)=\DecompSp{\CalS^{\left(\beta\right)}}p{\ell_{v^{s}}^{q}}{}$
in $\R^{\dimension}=\R^{2}$, it remains to verify
\[
C_{1}:=\sup_{i\in J^{\left(\beta\right)}}\:\sum_{j\in J^{\left(\beta\right)}}M_{j,i}<\infty\quad\text{ and }\quad C_{2}:=\sup_{j\in J^{\left(\beta\right)}}\:\sum_{i\in J^{\left(\beta\right)}}M_{j,i}<\infty,
\]
where
\[
M_{j,i}:=\left(\frac{v_{j}^{s}}{v_{i}^{s}}\right)^{\tau}\cdot\left(1+\left\Vert Y_{j}^{-1}Y_{i}\right\Vert \right)^{\sigma}\cdot\max_{\left|\nu\right|\leq1}\left(\left|\det Y_{i}\right|^{-1}\cdot\int_{S_{i}^{\left(\beta\right)}}\:\max_{\left|\alpha\right|\leq N}\left|\left(\left[\partial^{\alpha}\widehat{\partial^{\nu}\gamma_{j}}\right]\left(Y_{j}^{-1}\xi\right)\right)\right|\d\xi\right)^{\tau},
\]
with
\[
N:=\left\lceil \frac{\dimension+\varepsilon}{\min\left\{ 1,p\right\} }\right\rceil \leq\left\lceil \frac{3}{p_{0}}\right\rceil =:N_{0},\quad\tau:=\min\left\{ 1,p,q\right\} \geq\tau_{0}:=\min\left\{ p_{0},q_{0}\right\} \quad\text{ and }\quad\sigma:=\tau\cdot\left(\frac{\dimension}{\min\left\{ 1,p\right\} }+N\right).
\]
and where $\gamma_{j}:=\gamma_{k_{j}}^{\left(0\right)}$ for $j\in J^{\left(\beta\right)}$,
i.e., $\gamma_{j}=\psi$ for $j\in J_{0}^{\left(\beta\right)}$ and
$\gamma_{0}=\varphi$.

Now, since $\widehat{\varphi}\in\TestFunctionSpace{\R^{2}}$, there
is some $A>1$ satisfying $\supp\widehat{\varphi}\subset\left(-A,A\right)^{2}$.
Furthermore, since $\supp\widehat{\psi}\subset\R^{\ast}\times\R$
is compact, there are $0<\lambda<\mu$ and $B>0$ with 
\begin{align*}
\supp\widehat{\psi} & \subset\left\{ \left(\xi,\eta\right)\in\R^{2}\with\lambda<\left|\xi\right|<\mu\text{ and }\left|\eta\right|<\lambda B\right\} \\
 & \subset\left\{ \left(\xi,\eta\right)\in\R^{2}\with\lambda<\left|\xi\right|<\mu\text{ and }-B<\eta/\xi<B\right\} =U_{\left(-B,B\right)}^{\left(\lambda,\mu\right)}\cup\left[-U_{\left(-B,B\right)}^{\left(\lambda,\mu\right)}\right]=:U.
\end{align*}
By possibly shrinking $\lambda$ and enlarging $\mu$ and $B$, we
can assume $\lambda\leq\mu_{0}^{-1}$, $\mu\geq\mu_{0}$ and $B\geq3$,
so that $U\supset P=Q_{0}^{\left(1\right)}$. Setting $U_{0}':=\left(-A,A\right)^{2}$
and $U_{j}':=U$ for $j\in J_{0}^{\left(\beta\right)}$, we have just
shown $\supp\widehat{\gamma_{j}}\subset U_{j}'$ for all $j\in J^{\left(\beta\right)}$.
But standard properties of the Fourier transform (see e.g.\@ \cite[Theorem 8.22]{FollandRA})
show $\widehat{\partial^{\nu}\gamma_{j}}\left(\xi\right)=\left(2\pi i\xi\right)^{\nu}\cdot\widehat{\gamma_{j}}\left(\xi\right)$
and thus again $\supp\partial^{\alpha}\widehat{\partial^{\nu}\gamma_{j}}\subset U_{j}'$
for all $j\in J^{\left(\beta\right)}$ and arbitrary $\alpha,\nu\in\N_{0}^{2}$.
Therefore, we get
\[
\max_{\left|\nu\right|\leq1}\:\max_{\left|\alpha\right|\leq N}\left|\left(\left[\partial^{\alpha}\widehat{\partial^{\nu}\gamma_{j}}\right]\left(Y_{j}^{-1}\xi\right)\right)\right|\leq\left[\sup_{j\in J^{\left(\beta\right)}}\max_{\left|\nu\right|\leq1}\:\max_{\left|\alpha\right|\leq N_{0}}\left\Vert \partial^{\alpha}\widehat{\partial^{\nu}\gamma_{j}}\right\Vert _{\sup}\right]\cdot\Indicator_{U_{j}'}\left(Y_{j}^{-1}\xi\right)=:K_{1}\cdot\Indicator_{U_{j}'}\left(Y_{j}^{-1}\xi\right)=K_{1}\cdot\Indicator_{Y_{j}U_{j}'}\left(\xi\right)
\]
for all $\xi\in\R^{2}$ and $j\in J^{\left(\beta\right)}$. Here,
we emphasize that the constant $K_{1}$ is finite since $\left\{ \gamma_{j}\with j\in J^{\left(\beta\right)}\right\} =\left\{ \varphi,\psi\right\} \subset\Schwartz\left(\R^{2}\right)$
is a finite set.

Next, if we set $U_{j}:=Y_{j}^{\left(\beta\right)}U_{j}'$ for $j\in J^{\left(\beta\right)}$,
then Lemma \ref{lem:ReciprocalShearletCoveringIsAlmostStructuredGeneralized}
yields constants $L_{1},C>0$ and $M\in\N$ (depending only on $\lambda,\mu,A,B,\beta$)
such that
\begin{align*}
\varUpsilon_{j}:=\left\{ i\in J^{\left(\beta\right)}\with U_{i}\cap U_{j}\neq\emptyset\right\} \quad\text{ satisfies }\quad\left|\varUpsilon_{j}\right|\leq M & \qquad\forall j\in J^{\left(\beta\right)},\\
\left\Vert Y_{i}^{-1}Y_{j}\right\Vert \leq C & \qquad\forall i,j\in J^{\left(\beta\right)}\text{ with }U_{i}\cap U_{j}\neq\emptyset,\\
L_{1}^{-1}\cdot v_{j}\leq\left|\xi\right|\leq L_{1}\cdot v_{j} & \qquad\forall j\in J_{0}^{\left(\beta\right)}\text{ and }\xi\in U_{j}.
\end{align*}
As a slight modification, the last estimate yields because of $v_{j}\geq1$
that $\left(1+L_{1}\right)^{-1}\cdot v_{j}\leq\left|\xi\right|\leq1+\left|\xi\right|\leq\left(1+L_{1}\right)\cdot v_{j}$
for all $\xi\in U_{j}$ and $j\in J_{0}^{\left(\beta\right)}$. Likewise,
for $\xi\in U_{0}=\left(-A,A\right)^{2}$, we have 
\[
\left(1+2A\right)^{-1}\cdot v_{j}\leq1\leq1+\left|\xi\right|\leq1+2A=\left(1+2A\right)\cdot v_{j},
\]
so that there is a constant $L_{2}=L_{2}\left(A,B,\lambda,\mu,\beta\right)>0$
satisfying $L_{2}^{-1}\cdot v_{j}\leq1+\left|\xi\right|\leq L_{2}\cdot v_{j}$
for all $\xi\in U_{j}$ and $j\in J^{\left(\beta\right)}$. In particular,
for $i\in\varUpsilon_{j}$ there is some $\xi\in U_{i}\cap U_{j}\neq\emptyset$,
so that $v_{i}\leq L_{2}\cdot\left(1+\left|\xi\right|\right)\leq L_{2}^{2}\cdot v_{j}$.
By symmetry, we also get $v_{j}\leq L_{2}^{2}\cdot v_{i}$ and thus
$\left(v_{j}^{s}/v_{i}^{s}\right)\leq L_{2}^{2\left|s\right|}$ for
all $j\in J^{\left(\beta\right)}$ and $i\in\varUpsilon_{j}$.

Putting everything together and recalling $P_{j}'\subset U_{j}'$
for all $j\in J^{\left(\beta\right)}$, we thus see
\begin{align*}
M_{j,i} & =\left(\frac{v_{j}^{s}}{v_{i}^{s}}\right)^{\tau}\cdot\left(1+\left\Vert Y_{j}^{-1}Y_{i}\right\Vert \right)^{\sigma}\cdot\max_{\left|\nu\right|\leq1}\left(\left|\det Y_{i}\right|^{-1}\cdot\int_{S_{i}^{\left(\beta\right)}}\:\max_{\left|\alpha\right|\leq N}\left|\left(\left[\partial^{\alpha}\widehat{\partial^{\nu}\gamma_{j}}\right]\left(Y_{j}^{-1}\xi\right)\right)\right|\d\xi\right)^{\tau}\\
 & \leq\left(\frac{v_{j}^{s}}{v_{i}^{s}}\right)^{\tau}\cdot\left(1+\left\Vert Y_{j}^{-1}Y_{i}\right\Vert \right)^{\sigma}\cdot\left(K_{1}\cdot\left|\det Y_{i}\right|^{-1}\cdot\int_{U_{i}}\Indicator_{Y_{j}U_{j}'}\left(\xi\right)\d\xi\right)^{\tau}\\
 & =\left(\frac{v_{j}^{s}}{v_{i}^{s}}\right)^{\tau}\cdot\left(1+\left\Vert Y_{j}^{-1}Y_{i}\right\Vert \right)^{\sigma}\cdot\left(K_{1}\cdot\left|\det Y_{i}\right|^{-1}\cdot\lambda_{2}\left(U_{i}\cap U_{j}\right)\right)^{\tau}\\
 & \leq\Indicator_{\varUpsilon_{j}}\left(i\right)\cdot\left(\frac{v_{j}^{s}}{v_{i}^{s}}\right)^{\tau}\cdot\left(1+\left\Vert Y_{j}^{-1}Y_{i}\right\Vert \right)^{\sigma}\cdot\left(K_{1}\cdot\left|\det Y_{i}\right|^{-1}\cdot\lambda_{2}\left(Y_{i}U_{i}'\right)\right)^{\tau}\\
 & \leq\Indicator_{\varUpsilon_{j}}\left(i\right)\cdot L_{2}^{2\tau\left|s\right|}\cdot\left(1+C\right)^{\sigma}\cdot\left(K_{1}\cdot\sup_{i\in J^{\left(\beta\right)}}\lambda_{2}\left(U_{i}'\right)\right)^{\tau}\\
 & =\Indicator_{\varUpsilon_{j}}\left(i\right)\cdot L_{2}^{2\tau\left|s\right|}\cdot\left(1+C\right)^{\sigma}\cdot K_{2}^{\tau}
\end{align*}
for $K_{2}:=K_{1}\cdot\max\left\{ \lambda_{2}\left(\smash{\left(-A,A\right)^{2}}\right),\lambda_{2}\left(U\right)\right\} $.
Hence, using $\Indicator_{\varUpsilon_{j}}\left(i\right)=\Indicator_{\varUpsilon_{i}}\left(j\right)$,
we finally get
\begin{align*}
C_{1}^{1/\tau}=\sup_{i\in J^{\left(\beta\right)}}\:\left(\smash{\sum_{j\in J^{\left(\beta\right)}}}\vphantom{\sum}M_{j,i}\right)^{1/\tau} & \leq L_{2}^{2\left|s\right|}\cdot\left(1+C\right)^{\frac{\sigma}{\tau}}\cdot K_{2}\cdot\sup_{i\in J^{\left(\beta\right)}}\:\left|\varUpsilon_{i}\right|^{1/\tau}\\
 & \leq L_{2}^{2\left|s\right|}\cdot\left(1+C\right)^{\frac{\sigma}{\tau}}\cdot K_{2}\cdot M^{1/\tau}\\
 & \leq L_{2}^{2\max\left\{ \left|s_{0}\right|,\left|s_{1}\right|\right\} }\cdot\left(1+C\right)^{\frac{2}{p_{0}}+N_{0}}\cdot K_{2}\cdot M^{1/\tau_{0}}=:K_{3},
\end{align*}
where the last step used $\frac{\sigma}{\tau}=\frac{\dimension}{\min\left\{ 1,p\right\} }+N\leq\frac{2}{p_{0}}+N_{0}$.
Precisely the same arguments also show $C_{2}^{1/\tau}\leq K_{3}<\infty$.
Observe that $K_{3}$ is independent of $p,q,s$, as long as $p\geq p_{0}$,
$q\geq q_{0}$ and $s_{0}\leq s\leq s_{1}$.

\medskip{}

Consequently, all assumptions of Theorem \ref{thm:BanachFrameTheorem}
are satisfied. Furthermore, since the sets $\R^{\ast}\times\R$, $P$
and $\left(-1,1\right)^{2}$ are symmetric, we see analogously that
all assumptions of Theorem \ref{thm:BanachFrameTheorem} are still
satisfied (possibly with a slightly different constant $K_{3}$) if
$\varphi$ is replaced by $\widetilde{\varphi}$ and $\psi$ by $\widetilde{\psi}$,
where $\widetilde{f}\left(x\right)=f\left(-x\right)$. Thus, $\gamma_{j}=\widetilde{\psi}$
for $j\in J_{0}^{\left(\beta\right)}$ and $\gamma_{0}=\widetilde{\varphi}$.
Consequently, with a fixed regular partition of unity $\Phi=\left(\varphi_{\ell}\right)_{\ell\in J^{\left(\beta\right)}}$
for $\CalS^{\left(\beta\right)}$, Theorem \ref{thm:BanachFrameTheorem}
yields a constant $K=K\left(p_{0},q_{0},\CalS^{\left(\beta\right)},\Phi,\widetilde{\varphi},\widetilde{\psi}\right)=K\left(p_{0},q_{0},\beta,\varphi,\psi\right)>0$,
such that for arbitrary
\[
0<\delta\leq\delta_{00}=\left(1+K\cdot C_{\CalS^{\left(\beta\right)},v^{s}}^{4}\cdot\left(C_{1}^{1/\tau}+C_{2}^{1/\tau}\right)^{2}\right)^{-1},
\]
the family
\[
\left(L_{\delta\cdot Y_{i}^{-T}k}\:\widetilde{\gamma^{\left[i\right]}}\right)_{i\in J^{\left(\beta\right)},k\in\Z^{2}}\quad\text{ with }\quad\gamma^{\left[i\right]}=\left|\det Y_{i}\right|^{1/2}\cdot M_{c_{i}}\left[\gamma_{i}\circ Y_{i}^{T}\right]\quad\text{ and }\quad\widetilde{\gamma^{\left[i\right]}}\left(x\right)=\gamma^{\left[i\right]}\left(-x\right)
\]
yields a Banach frame for $\mathscr{S}_{\beta,s}^{p,q}\left(\R^{2}\right)=\DecompSp{\CalS^{\left(\beta\right)}}p{\ell_{v^{s}}^{q}}{}$,
as precisely described in Theorem \ref{thm:BanachFrameTheorem}. But
with what we just saw and thanks to Corollary \ref{cor:ReciprocalShearletCoveringAlmostStructured},
we have
\[
\delta_{00}=\!\left(1\!+\!K\cdot C_{\CalS^{\left(\beta\right)},v^{s}}^{4}\cdot\left(C_{1}^{1/\tau}\!\!+\!C_{2}^{1/\tau}\right)^{\!2}\right)^{\!-1}\geq\left(1\!+\!K\cdot K_{4}^{8\left|s\right|}\cdot\left(2K_{3}\right)^{2}\right)^{\!-1}\geq\left(1\!+\!4\cdot K\cdot K_{4}^{8\max\left\{ \left|s_{0}\right|,\left|s_{1}\right|\right\} }\cdot K_{3}^{2}\right)^{-1}\!=:\delta_{0}
\]
for a suitable constant $K_{4}=K_{4}\left(\beta\right)\geq1$ which
is provided by Corollary \ref{cor:ReciprocalShearletCoveringAlmostStructured}.

Finally, note that the coefficient map $A^{\left(\delta\right)}f=\left(\left[\gamma^{\left[i\right]}\ast f\right]\left(\delta\cdot Y_{i}^{-T}k\right)\right)_{i\in J^{\left(\beta\right)},\,k\in\Z^{2}}$
from Theorem \ref{thm:BanachFrameTheorem} uses a somewhat peculiar
definition of the convolution $\gamma^{\left[i\right]}\ast f$, cf.\@
equation \eqref{eq:SpecialConvolutionDefinition}. Precisely, with
the regular partition of unity $\Phi=\left(\varphi_{\ell}\right)_{\ell\in J^{\left(\beta\right)}}$
from above, we have
\begin{align*}
\left[\gamma^{\left[i\right]}\ast f\right]\left(x\right) & =\sum_{\ell\in J}\Fourier^{-1}\left(\widehat{\gamma^{\left[i\right]}}\cdot\varphi_{\ell}\cdot\widehat{f}\:\right)\left(x\right)\\
\left({\scriptstyle \text{series is finite sum, since }\Phi\text{ is a locally finite and }\widehat{\gamma^{\left[i\right]}}\in\TestFunctionSpace{\R^{2}}}\right) & =\Fourier^{-1}\left(\sum_{\ell\in J}\varphi_{\ell}\cdot\widehat{\gamma^{\left[i\right]}}\cdot\widehat{f}\:\right)\left(x\right)\\
\left({\scriptstyle \sum_{\ell\in J}\varphi_{\ell}\equiv1\text{ on }\R^{2}}\right) & =\Fourier^{-1}\left(\widehat{\gamma^{\left[i\right]}}\cdot\widehat{f}\right)\left(x\right)=\left\langle \widehat{f},\,e^{2\pi i\left\langle x,\mybullet\right\rangle }\cdot\widehat{\gamma^{\left[i\right]}}\right\rangle _{\DistributionSpace{\R^{2}},\TestFunctionSpace{\R^{2}}}\\
 & =\left\langle f,\,\Fourier\left[e^{2\pi i\left\langle x,\mybullet\right\rangle }\cdot\widehat{\gamma^{\left[i\right]}}\right]\right\rangle _{Z'\left(\R^{2}\right),Z\left(\R^{2}\right)}=\left\langle f,\,L_{x}\cdot\widehat{\widehat{\gamma^{\left[i\right]}}}\right\rangle _{Z'\left(\R^{2}\right),Z\left(\R^{2}\right)}\\
\left({\scriptstyle \text{Fourier inversion}}\right) & =\left\langle f,\,L_{x}\cdot\widetilde{\gamma^{\left[i\right]}}\right\rangle _{Z'\left(\R^{2}\right),Z\left(\R^{2}\right)}
\end{align*}
for all $x\in\R^{2}$ and $i\in J^{\left(\beta\right)}$.

\medskip{}

It remains to verify that the family $\left(L_{\delta\cdot Y_{i}^{-T}k}\:\widetilde{\gamma^{\left[i\right]}}\right)_{i\in J^{\left(\beta\right)},k\in\Z^{2}}$
is (almost) identical to the family $\left(\gamma^{\left[i,k,\delta\right]}\right)_{i\in J^{\left(\beta\right)},k\in\Z^{2}}$
from the statement of the theorem. Recall that $c_{i}=0$ for all
$i\in J^{\left(\beta\right)}$. Now, for $i=0$, $Y_{i}=Y_{0}=\identity$
and thus
\[
L_{\delta\cdot Y_{i}^{-T}k}\:\widetilde{\gamma^{\left[i\right]}}=L_{\delta\cdot k}\:\widetilde{\gamma_{i}}=L_{\delta\cdot k}\:\varphi=\varphi\left(\mybullet-\delta k\right)=\gamma^{\left[i,k,\delta\right]}.
\]
Next, in case of $i=\left(j,\ell,0\right)\in J_{0}^{\left(\beta\right)}$,
recall from Definition \ref{def:ReciprocalBetaShearletCovering} that
$Y_{i}^{T}=S_{\ell}\cdot{\rm diag}\left(2^{\beta j/2},\,2^{j/2}\right)=S_{\ell}\cdot A_{\beta^{-1},2^{\beta j/2}}$
and $\left|\det Y_{i}\right|=2^{\frac{j}{2}\left(1+\beta\right)}$,
so that
\[
L_{\delta\cdot Y_{i}^{-T}k}\:\widetilde{\gamma^{\left[i\right]}}=L_{\delta\cdot\left[S_{\ell}\cdot A_{\beta^{-1},2^{\beta j/2}}\right]^{-1}k}\:\widetilde{\gamma^{\left[i\right]}}=2^{\frac{j}{4}\left(1+\beta\right)}\cdot\psi\left(S_{\ell}\cdot A_{\beta^{-1},2^{\beta j/2}}\mybullet-\delta k\right)=\gamma^{\left[i,k,\delta\right]}.
\]
Finally, in case of $i=\left(j,\ell,1\right)\in J_{0}^{\left(\beta\right)}$,
a direct calculation shows
\[
Y_{i}^{T}=\left(\begin{matrix}2^{j/2}\ell & 2^{\beta j/2}\\
2^{j/2} & 0
\end{matrix}\right)=R\cdot S_{\ell}^{T}\cdot\widetilde{A}_{\beta^{-1},2^{\beta j/2}},
\]
so that
\begin{align*}
L_{\delta\cdot Y_{i}^{-T}k}\:\widetilde{\gamma^{\left[i\right]}} & =2^{\frac{j}{4}\left(1+\beta\right)}\cdot\psi\left(Y_{i}^{T}\left[\mybullet-\delta Y_{i}^{-T}k\right]\right)\\
 & =2^{\frac{j}{4}\left(1+\beta\right)}\cdot\psi\left(R\left[S_{\ell}^{T}\cdot\widetilde{A}_{\beta^{-1},2^{\beta j/2}}\mybullet-\delta Rk\right]\right)\\
 & =2^{\frac{j}{4}\left(1+\beta\right)}\cdot\theta\left(S_{\ell}^{T}\cdot\widetilde{A}_{\beta^{-1},2^{\beta j/2}}\mybullet-\delta Rk\right)\\
 & =\gamma^{\left[i,Rk,\delta\right]}.
\end{align*}
But since $\Z^{2}\to\Z^{2},k\mapsto Rk$ is bijective, it is not hard
to see directly from the definition of the coefficient space $C_{v^{s}}^{p,q}$
(cf.\@ Definition \ref{def:CoefficientSpace}) that if we set $\delta_{i}:=1$
for $i$ of the form $i=\left(j,\ell,1\right)$ and $\delta_{i}:=0$
otherwise, then
\[
\Omega:C_{v^{s}}^{p,q}\to C_{v^{s}}^{p,q},\left(\smash{c_{k}^{\left(i\right)}}\right)_{i\in J,k\in\Z^{2}}\mapsto\left(c_{R^{\delta_{i}}\cdot k}^{\left(i\right)}\right)_{i\in J,k\in\Z^{2}}
\]
is an isometric isomorphism. All in all, we have shown
\begin{align*}
A^{\left(\delta\right)}f & =\left(\left[\gamma^{\left[i\right]}\ast f\right]\left(\delta\cdot Y_{i}^{-T}k\right)\right)_{i\in J^{\left(\beta\right)},\,k\in\Z^{2}}=\left(\left\langle f,\,L_{\delta\cdot Y_{i}^{-T}k}\:\widetilde{\gamma^{\left[i\right]}}\right\rangle _{Z'\left(\R^{2}\right),Z\left(\R^{2}\right)}\right)_{i\in J^{\left(\beta\right)},\,k\in\Z^{2}}\\
 & =\left(\left\langle f,\,\gamma^{\left[i,R^{\delta_{i}}k,\delta\right]}\right\rangle _{Z'\left(\R^{2}\right),Z\left(\R^{2}\right)}\right)_{i\in J^{\left(\beta\right)},\,k\in\Z^{2}}=\Omega\left[\left(\left\langle f,\,\gamma^{\left[i,k,\delta\right]}\right\rangle _{Z'\left(\R^{2}\right),Z\left(\R^{2}\right)}\right)_{i\in J^{\left(\beta\right)},\,k\in\Z^{2}}\right].
\end{align*}
In conjunction with Theorem \ref{thm:BanachFrameTheorem}, this easily
yields all claimed properties.
\end{proof}
Now, we can finally provide the proof of Proposition \ref{prop:CartoonLikeFunctionsBoundedInAlphaShearletSmoothness}
for the general case $\beta\in\left(1,2\right]$.
\begin{proof}[Proof of Proposition \ref{prop:CartoonLikeFunctionsBoundedInAlphaShearletSmoothness}
for $\beta\in\left(1,2\right)$]
Set $\alpha:=\beta^{-1}\in\left(0,1\right)$. For $j\in\N$, let
$L_{j}:=2^{\left\lfloor j\left(1-\alpha\right)\right\rfloor }$ and
$M:=M_{0}\times\Z^{2}$, with $M_{0}:=\left\{ 0\right\} \cup\left\{ \left(j,\ell\right)\in\N\times\Z\with\ell\in\left\{ 0,\dots,L_{j}-1\right\} \right\} $.
Furthermore, let $\left(\psi_{\mu}\right)_{\mu\in M}$ be the tight
$\alpha$-curvelet frame constructed in \cite[Section 3]{CartoonApproximationWithAlphaCurvelets};
see also \cite[Definition 2.2]{AlphaMolecules}. Then, \cite[Theorem 4.2]{CartoonApproximationWithAlphaCurvelets}
yields a constant $C=C\left(\beta,\nu\right)>0$ such that we have
\[
\left|\theta_{N}^{\ast}\left(f\right)\right|\leq C\cdot\left[N^{-1}\cdot\left(1+\log N\right)\right]^{-\frac{1+\beta}{2}}\qquad\forall N\in\N\text{ and }f\in\mathcal{E}^{\beta}\left(\R^{2};\nu\right),
\]
where $\theta_{N}^{\ast}\left(f\right)$ denotes the $N$-th largest
(in absolute value) $\alpha$-curvelet coefficient of $f$ with respect
to the $\alpha$-curvelet frame $\left(\psi_{\mu}\right)_{\mu\in M}$.
This easily implies
\begin{equation}
\left\Vert \left(\left\langle f,\,\psi_{\mu}\right\rangle _{L^{2}}\right)_{\mu\in M}\right\Vert _{\ell^{p}}=\left\Vert \left(\theta_{N}^{\ast}\left(f\right)\right)_{N\in\N}\right\Vert _{\ell^{p}}\leq C_{1}^{\left(p\right)}\qquad\forall f\in\mathcal{E}^{\beta}\left(\R^{2};\nu\right)\text{ and }p>\frac{2}{1+\beta},\label{eq:CurveletAnalysisSparsity}
\end{equation}
for a suitable constant $C_{1}^{\left(p\right)}=C_{1}^{\left(p\right)}\left(\beta,\nu\right)$.

Now, let $\varphi,\psi$ be real-valued functions satisfying the requirements
of Proposition \ref{prop:BanachFrameForReciprocalShearletSmoothnessSpace}
and let $\theta:=\psi\circ R$. Let $p_{0}=q_{0}=\frac{2}{1+\beta}$,
$s_{0}=0$ and $s_{1}=\frac{1}{2}\left(1+\beta\right)$ and choose
$\delta_{0}=\delta_{0}\left(\beta,p_{0},q_{0},s_{0},s_{1},\varphi,\psi\right)>0$
as provided by Proposition \ref{prop:BanachFrameForReciprocalShearletSmoothnessSpace},
so that the cone-adapted $\beta$-shearlet system ${\rm SH}\left(\varphi,\psi,\theta;\,\delta,\beta\right)$
forms a Banach frame for $\mathscr{S}_{\beta,s}^{p,q}\left(\R^{2}\right)$
for all $0<\delta\leq\delta_{0}$, $p\geq p_{0}$, $q\geq q_{0}$
and $s_{0}\leq s\leq s_{1}$, in the sense of Proposition \ref{prop:BanachFrameForReciprocalShearletSmoothnessSpace}.

From this point on, the proof heavily uses the results and terminology
of \cite{AlphaMolecules}: Since $\varphi,\psi,\theta$ are bandlimited,
\cite[Proposition 3.11(ii)]{AlphaMolecules} shows\footnote{Before \cite[Proposition 3.11]{AlphaMolecules}, it is required that
the generators $\varphi,\psi,\theta$ of a \emph{band-limited} $\beta$-shearlet
system satisfy $\supp\varphi\subset Q$, $\supp\psi\subset W$ and
$\supp\theta\subset\widetilde{W}$, where $Q\subset\R^{2}$ is a cube
centered at the origin and $W,\widetilde{W}\subset\R^{2}$ satisfy
$W\subset\left[-a,a\right]\times\left(\left[-c,-b\right]\cup\left[b,c\right]\right)$
and $\widetilde{W}\subset\left(\left[-c,-b\right]\cup\left[b,c\right]\right)\times\left[-a,a\right]$
for certain $a>0$ and $0<b<c$. This is of course impossible, since
$\varphi,\psi,\theta$ would then need to be simultaneously bandlimited
and compactly supported. What is actually meant is $\supp\widehat{\varphi}\subset Q$,
$\supp\widehat{\psi}\subset\widetilde{W}$ and $\supp\widehat{\theta}\subset W$,
with $Q,W,\widetilde{W}$ as above. Note the interchange of the sets
$\widetilde{W}$ and $W$ compared to the condition in \cite{AlphaMolecules}.
It is not hard to see that our generators $\varphi,\psi,\theta$ satisfy
these corrected assumptions, since $\supp\widehat{\psi}\subset\R^{\ast}\times\R$.} that ${\rm SH}\left(\varphi,\psi,\theta;\,\delta,\beta\right)$ is
a system of \textbf{$\beta^{-1}$-molecules} of order $\left(\infty,\infty,\infty,\infty\right)$
with respect to the \textbf{parametrization} $\left(\Lambda^{s},\Phi^{s}\right)$
with $\tau=\delta$, $\sigma=2^{\beta/2}$, $\eta_{j}=\sigma^{-j\left(1-\alpha\right)}$
and $L_{j}=\left\lceil \sigma^{j\left(1-\alpha\right)}\right\rceil $,
cf.\@ \cite[Definitions 3.7 and 3.8]{AlphaMolecules} for details
of this parametrization. Furthermore, \cite[Proposition 3.3(iii)]{AlphaMolecules}
shows that the $\alpha$-curvelet frame $\left(\psi_{\mu}\right)_{\mu\in M}$
from above is a system of $\alpha$-molecules of order $\left(\infty,\infty,\infty,\infty\right)$
with respect to the parametrization $\left(\Lambda^{c},\Phi^{c}\right)$
given in \cite[Definition 3.2]{AlphaMolecules}, with parameters $\sigma=2$,
$\tau=1$ and $L_{j}=2^{\left\lfloor j\left(1-\alpha\right)\right\rfloor }$
as above.

Next, \cite[Theorem 5.7]{AlphaMolecules} shows that the $\alpha$-curvelet
parametrization $\left(\Lambda^{c},\Phi^{c}\right)$ (defined in \cite[Definition 3.2]{AlphaMolecules})
and the $\alpha$-shearlet parametrization $\left(\Lambda^{s},\Phi^{s}\right)$
are \textbf{$\left(\alpha,k\right)$-consistent} for all $k>2$; cf.\@
\cite[Definition 5.5]{AlphaMolecules} for the definition of $\left(\alpha,k\right)$-consistency.
Now, for arbitrary $p\in\left(2/\left(1+\beta\right),1\right]$, \cite[Theorem 5.6]{AlphaMolecules}
shows that $\left(\psi_{\mu}\right)_{\mu\in M}$ and ${\rm SH}\left(\varphi,\psi,\theta;\,\delta,\beta\right)=\left(\gamma^{\left[i,k,\delta\right]}\right)_{\left(i,k\right)\in J^{\left(\beta\right)}\times\Z^{2}}$
are \textbf{sparsity equivalent} in $\ell^{p}$, which means (cf.\@
\cite[Definition 5.3]{AlphaMolecules}) that the operator $A:\ell^{p}\left(M\right)\to\ell^{p}\left(J^{\left(\beta\right)}\times\Z^{2}\right)$
given by the infinite matrix $\left(\left\langle \psi_{\mu},\,\gamma^{\left[i,k,\delta\right]}\right\rangle _{L^{2}}\right)_{\mu\in M,\left(i,k\right)\in J^{\left(\beta\right)}\times\Z^{2}}$
is well-defined and bounded. Now, since $\left(\psi_{\mu}\right)_{\mu\in M}$
is a tight frame, we get\footnote{Note that an infinite matrix $\left(A_{i,j}\right)_{i\in I,j\in J}$
usually would yield an operator $\ell^{p}\left(J\right)\to\ell^{p}\left(I\right)$,
not $\ell^{p}\left(I\right)\to\ell^{p}\left(J\right)$. But the convention
used here is the same as in \cite{AlphaMolecules}, see e.g.\@ the
proof of \cite[Proposition 5.2]{AlphaMolecules}.}
\[
f=\sum_{\mu\in M}\left\langle f,\psi_{\mu}\right\rangle _{L^{2}}\cdot\psi_{\mu}\qquad\text{ and thus }\qquad\left\langle f,\,\smash{\gamma^{\left[i,k,\delta\right]}}\right\rangle _{L^{2}}=\sum_{\mu\in M}\left\langle \psi_{\mu},\,\smash{\gamma^{\left[i,k,\delta\right]}}\right\rangle _{L^{2}}\left\langle f,\psi_{\mu}\right\rangle _{L^{2}}=A\left[\left(\left\langle f,\psi_{\mu}\right\rangle _{L^{2}}\right)_{\mu\in M}\right].
\]
Consequently, since $\varphi,\psi$ and thus also $\gamma^{\left[i,k,\delta\right]}$
are real-valued, we get from equation \eqref{eq:CurveletAnalysisSparsity}
that
\begin{align*}
\left\Vert \left(\left\langle f,\,\smash{\gamma^{\left[i,k,\delta\right]}}\right\rangle _{Z'\left(\R^{2}\right),Z\left(\R^{2}\right)}\right)_{\left(i,k\right)\in J^{\left(\beta\right)}\times\Z^{2}}\right\Vert _{\ell^{p}} & =\left\Vert \left(\left\langle f,\,\smash{\gamma^{\left[i,k,\delta\right]}}\right\rangle _{L^{2}}\right)_{\left(i,k\right)\in J^{\left(\beta\right)}\times\Z^{2}}\right\Vert _{\ell^{p}}\\
 & =\left\Vert A\left[\left(\left\langle f,\psi_{\mu}\right\rangle _{L^{2}}\right)_{\mu\in M}\right]\right\Vert _{\ell^{p}}\leq\vertiii A_{\ell^{p}\to\ell^{p}}C_{1}^{\left(p\right)}=:C_{2}^{\left(p\right)}<\infty
\end{align*}
for all $f\in\mathcal{E}^{\beta}\left(\R^{2};\nu\right)\subset L^{2}\left(\R^{2}\right)=\mathscr{S}_{\beta,0}^{2,2}\left(\R^{2}\right)$
(cf.\@ \cite[Lemma 6.10]{DecompositionEmbedding}) and $p\in\left(2/\left(1+\beta\right),1\right]$.
But since $\ell^{1}\hookrightarrow\ell^{p}$ for $p\geq1$, this estimate
in fact holds for all $p\in\left(2/\left(1+\beta\right),\infty\right]$,
with $C_{2}^{\left(p\right)}:=C_{2}^{\left(1\right)}$ for $p\geq1$.

In view of the consistency statement in Proposition \ref{prop:BanachFrameForReciprocalShearletSmoothnessSpace},
and since the remark after Definition \ref{def:CoefficientSpace}
shows $C_{v^{\left(1+\beta^{-1}\right)\left(\frac{1}{p}-\frac{1}{2}\right)}}^{p,p}=\ell^{p}\left(\smash{J^{\left(\beta\right)}}\times\Z^{2}\right)$,
we thus get $f\in\mathscr{S}_{\beta,\left(1+\beta^{-1}\right)\left(\frac{1}{p}-\frac{1}{2}\right)}^{p,p}\left(\R^{2}\right)$,
with
\begin{align*}
\left\Vert f\right\Vert _{\mathscr{S}_{\beta,\left(1+\beta^{-1}\right)\left(\frac{1}{p}-\frac{1}{2}\right)}^{p,p}} & =\left\Vert R^{\left(\delta\right)}A^{\left(\delta\right)}f\right\Vert _{\mathscr{S}_{\beta,\left(1+\beta^{-1}\right)\left(\frac{1}{p}-\frac{1}{2}\right)}^{p,p}}\\
 & \leq\vertiii{\smash{R^{\left(\delta\right)}}}_{C_{v^{\left(1+\beta^{-1}\right)\left(\frac{1}{p}-\frac{1}{2}\right)}}^{p,p}\to\mathscr{S}_{\beta,\left(1+\beta^{-1}\right)\left(\frac{1}{p}-\frac{1}{2}\right)}^{p,p}}\cdot\left\Vert \left(\left\langle f,\,\smash{\gamma^{\left[i,k,\delta\right]}}\right\rangle _{Z'\left(\R^{2}\right),Z\left(\R^{2}\right)}\right)_{\left(i,k\right)\in J^{\left(\beta\right)}\times\Z^{2}}\right\Vert _{\ell^{p}}\leq C_{3}^{\left(p\right)}
\end{align*}
for all $f\in\mathcal{E}^{\beta}\left(\R^{2};\nu\right)$ and arbitrary
$p\in\left(2/\left(1+\beta\right),2\right]$, for a suitable constant
$C_{3}^{\left(p\right)}=C_{3}^{\left(p\right)}\left(\varphi,\psi,\beta,\nu,p,\delta\right)$.
Here, $R^{\left(\delta\right)}$ is the reconstruction operator provided
by Proposition \ref{prop:BanachFrameForReciprocalShearletSmoothnessSpace}.
This uses that we indeed have $p\geq p_{0}=q_{0}=2/\left(1+\beta\right)$
and
\[
s_{0}=0\leq\left(1+\beta^{-1}\right)\left(\frac{1}{p}-\frac{1}{2}\right)\leq\left(1+\beta^{-1}\right)\left(\frac{1+\beta}{2}-\frac{1}{2}\right)=\frac{1}{2}\left(1+\beta\right)=s_{1},
\]
so that Proposition \ref{prop:BanachFrameForReciprocalShearletSmoothnessSpace}
applies. Since Lemma \ref{lem:ReciprocalShearletSmoothnessSpacesAreBoring}
shows $\mathscr{S}_{\beta,\left(1+\beta^{-1}\right)\left(\frac{1}{p}-\frac{1}{2}\right)}^{p,p}=\mathscr{S}_{\beta^{-1},\left(1+\beta^{-1}\right)\left(\frac{1}{p}-\frac{1}{2}\right)}^{p,p}\left(\R^{2}\right)$,
the proof is complete.
\end{proof}

\section{A slight twist for achieving polynomial search depth}

\label{sec:PolynomialSearchDepth}In Theorem \ref{thm:CartoonApproximationWithAlphaShearlets},
we saw for $\beta\in\left(1,2\right]$ and $\alpha=\beta^{-1}$ that
suitable $\alpha$-shearlet systems achieve the approximation rate
$\left\Vert f-f_{N}\right\Vert _{L^{2}}\lesssim N^{-\left(\frac{\beta}{2}-\varepsilon\right)}$
for arbitrary $\varepsilon>0$ and $C^{\beta}$-cartoon-like functions
$f\in\mathcal{E}^{\beta}\left(\R^{2}\right)$. Furthermore, we recalled
from \cite[Theorem 2.8]{CartoonApproximationWithAlphaCurvelets} that
this approximation rate is essentially optimal, in the sense that
no system $\Phi=\left(\varphi_{n}\right)_{n\in\N}$ can achieve an
approximation rate better than $N^{-\beta/2}$ for the whole class
$\mathcal{E}^{\beta}\left(\R^{2};\nu\right)$, if one imposes a \emph{polynomial
search depth} for forming the $N$-term approximation $f_{N}$. This
means that $f_{N}$ is assumed to be a linear combination of $N$
elements of $\left\{ \varphi_{1},\dots,\varphi_{\pi\left(N\right)}\right\} $,
where $\pi$ is a \emph{fixed} polynomial, independent of $f$. We
did \emph{not} show, however, that the $N$-term approximations $f_{N}$
constructed in Theorem \ref{thm:CartoonApproximationWithAlphaShearlets}
satisfy such a polynomial search depth restriction. The goal of this
section is precisely to show that this is possible for a suitable
enumeration $\left(\psi_{n}\right)_{n\in\N}$ of the $\alpha$-shearlet
system under consideration.

The proof, however, is surprisingly nontrivial: In the proof of Theorem
\ref{thm:CartoonApproximationWithAlphaShearlets}, we used that $f=\sum_{i\in V\times\Z^{2}}c_{i}\psi_{i}$
for a sequence $c=\left(c_{i}\right)_{i\in V\times\Z^{2}}$ with $c\in\bigcap_{p>2/\left(1+\beta\right)}\ell^{p}\left(V\times\Z^{2}\right)$
and then truncated $c$ to $c\cdot\Indicator_{J_{N}}$ to form $f_{N}=\sum_{i\in J_{N}}c_{i}\psi_{i}$,
where $J_{N}\subset V\times\Z^{2}$ contains the indices of the $N$
largest entries of $c$. But the positions of these indices depend
heavily on $c=c\left(f\right)$ and thus on $f$, while the polynomial
search depth restriction requires us to use only indices in $\left\{ 1,\dots,\pi\left(N\right)\right\} $,
where $\pi$ is \emph{independent} of $f$.

Thus, what we essentially need is a certain (weak) decay of the coefficients,
\emph{uniformly} over the whole class $\mathcal{E}^{\beta}\left(\R^{2};\nu\right)$.
But with our present decomposition space formalism, we can not express
such a decay, cf.\@ Theorem \ref{thm:AnalysisAndSynthesisSparsityAreEquivalent}:
By choosing the exponent $s$ for the weight $u^{s}$ suitably, we
can enforce a decay of the coefficients \emph{with the scale}. But
since the weight is independent of the translation variable $k\in\Z^{2}$
and since the space $\ell^{p}\left(\Z^{2}\right)$ is permutation
invariant, the current formalism \emph{cannot} impose a decay of the
coefficients as $\left|k\right|\to\infty$.

Ultimately, this is caused by the definition of the decomposition
spaces: It is not hard to see that the spaces $\DecompSp{\CalQ}p{\ell_{w}^{q}}{}$
are \emph{isometrically} translation invariant. What we need, therefore,
is a modified type of decomposition spaces which does not have this
property. Luckily, such a type of decomposition spaces already exists.
In fact, the theory of structured Banach frame decompositions in \cite{StructuredBanachFrames}
was developed for the spaces $\DecompSp{\CalQ}p{\ell_{w}^{q}}v$,
where the Lebesgue spaces $L^{p}\left(\R^{\dimension}\right)$ are
replaced by the \textbf{weighted Lebesgue spaces} $L_{v}^{p}\left(\R^{\dimension}\right)=\left\{ f:\R^{\dimension}\to\Compl\with v\cdot f\in L^{p}\left(\R^{\dimension}\right)\right\} $
with $\left\Vert f\right\Vert _{L_{v}^{p}}=\left\Vert v\cdot f\right\Vert _{L^{p}}$,
where $v:\R^{\dimension}\to\left(0,\infty\right)$ is measurable.
This theory is briefly discussed in the next subsection.

\subsection{Structured Banach frame decompositions of weighted decomposition
spaces}

\label{subsec:StructuredWeightedBFD}The weight $v$ from above needs
to satisfy certain regularity properties to ensure that the spaces
$\DecompSp{\CalQ}p{\ell_{w}^{q}}v$ are well-defined. Precisely, we
say that a measurable weight $v:\R^{\dimension}\to\left(0,\infty\right)$
is \textbf{$v_{0}$-moderate} for some weight $v_{0}:\R^{\dimension}\to\left(0,\infty\right)$
if we have
\begin{equation}
v\left(x+y\right)\leq v\left(x\right)\cdot v_{0}\left(y\right)\qquad\forall x,y\in\R^{\dimension}.\label{eq:ModerateSpaceWeightDefinition}
\end{equation}

Now, as in Section \ref{sec:BanachFrameDecompositionCrashCourse},
let us fix an almost structured covering $\CalQ=\left(T_{i}Q_{i}'+b_{i}\right)_{i\in I}$
of an open set $\emptyset\neq\CalO\subset\R^{\dimension}$ with associated
regular partition of unity $\Phi=\left(\varphi_{i}\right)_{i\in I}$
for the remainder of the subsection and assume that $\CalQ$ satisfies
Assumption \ref{assu:CrashCourseStandingAssumptions}. The weight
$v_{0}$ is called \textbf{$\left(\CalQ,\Omega_{0},\Omega_{1},K\right)$-regular},
for $\Omega_{0},\Omega_{1}\in\left[1,\infty\right)$ and $K\in\left[0,\infty\right)$,
if it satisfies the following:
\begin{enumerate}
\item $v_{0}$ is measurable and symmetric, i.e., $v_{0}\left(-x\right)=v_{0}\left(x\right)$
for all $x\in\R^{\dimension}$.
\item $v_{0}$ is \textbf{submultiplicative}, i.e., $v_{0}\left(x+y\right)\leq v_{0}\left(x\right)\cdot v_{0}\left(y\right)$
for all $x,y\in\R^{\dimension}$.
\item We have $v_{0}\left(x\right)\leq\Omega_{1}\cdot\left(1+\left|x\right|\right)^{K}$
for all $x\in\R^{\dimension}$.
\item We have $K=0$, or $\left\Vert T_{i}^{-1}\right\Vert \leq\Omega_{0}$
for all $i\in I$.
\end{enumerate}
We note that the preceding assumptions imply $v_{0}\left(x\right)\geq1$
for all $x\in\R^{\dimension}$. Indeed, $v_{0}\left(0\right)=v_{0}\left(x+\left(-x\right)\right)\leq\left[v_{0}\left(x\right)\right]^{2}$
for all $x\in\R^{\dimension}$ by symmetry and submultiplicativity.
For $x=0$, this yields $v_{0}\left(0\right)\geq1$, since $v_{0}\left(0\right)>0$.
Finally, we then see $1\leq v_{0}\left(0\right)\leq\left[v_{0}\left(x\right)\right]^{2}$
and hence $v_{0}\left(x\right)\geq1$ for all $x\in\R^{\dimension}$.

The following example introduces the class of weights in which we
will be mainly interested.
\begin{example}
\label{exa:StandardShearletSpaceWeight}The \textbf{standard weight}
$\omega_{0}$ is given by $\omega_{0}:\R^{\dimension}\to\left(0,\infty\right),x\mapsto1+\left|x\right|$.
It is submultiplicative, since
\[
1+\left|x+y\right|\leq1+\left|x\right|+\left|y\right|\leq\left(1+\left|x\right|\right)\cdot\left(1+\left|y\right|\right)\qquad\forall x,y\in\R^{\dimension}.
\]
Hence, if we have $K=0$ and $\Omega_{0}=1$, or if $K>0$ and $\left\Vert T_{i}^{-1}\right\Vert \leq\Omega_{0}$
for all $i\in I$, then $\omega_{0}^{K}$ is $\left(\CalQ,\Omega_{0},1,K\right)$-regular.

Furthermore, if $L\in\R$ with $\left|L\right|\leq K$, then $\omega_{0}^{L}$
is $\omega_{0}^{K}$-moderate. For $L\geq0$, this follows from submultiplicativity
of $\omega_{0}$, since $\omega_{0}^{L}\left(x+y\right)\leq\omega_{0}^{L}\left(x\right)\omega_{0}^{L}\left(y\right)\leq\omega_{0}^{L}\left(x\right)\omega_{0}^{K}\left(y\right)$.
If $L<0$, then our considerations for $L\geq0$ show $\omega_{0}^{-L}\left(x\right)=\omega_{0}^{-L}\left(\left[x+y\right]-y\right)\leq\omega_{0}^{-L}\left(x+y\right)\omega_{0}^{-L}\left(-y\right)\leq\omega_{0}^{-L}\left(x+y\right)\omega_{0}^{K}\left(y\right)$.
Rearranging again yields the claim.

Finally, in case of the unconnected $\alpha$-shearlet covering $\CalQ=\CalS_{u}^{\left(\alpha\right)}=\left(B_{v}W_{v}'\right)_{v\in V^{\left(\alpha\right)}}$,
we have $\left\Vert B_{v}^{-1}\right\Vert \leq3=:\Omega_{0}$ for
all $v\in V^{\left(\alpha\right)}$. Indeed, for $v=0$, this is trivial
and for $v=\left(j,\ell,\delta\right)\in V_{0}^{\left(\alpha\right)}$,
we have
\[
\left\Vert B_{\left(j,\ell,\delta\right)}^{-1}\right\Vert =\left\Vert \left(\begin{matrix}2^{-j} & 0\\
-2^{-j}\ell & 2^{-\alpha j}
\end{matrix}\right)R^{-\delta}\right\Vert =\left\Vert \left(\begin{matrix}2^{-j} & 0\\
-2^{-j}\ell & 2^{-\alpha j}
\end{matrix}\right)\right\Vert \leq2^{-j}+2^{-\alpha j}+\left|-2^{-j}\ell\right|\leq3.
\]
Here, the last step used that $\left|\ell\right|\leq\left\lceil 2^{\left(1-\alpha\right)j}\right\rceil \leq2^{j}$.
Therefore, $\omega_{0}^{K}$ is $\left(\smash{\CalS_{u}^{\left(\alpha\right)}},3,1,K\right)$-regular
for $K\geq0$.
\end{example}
Now, we can define the modified, weighted decomposition spaces.
\begin{defn}
\label{def:SpaceWeightedDecompositionSpace}Let $p,q\in\left(0,\infty\right]$
and let $w=\left(w_{i}\right)_{i\in I}$ be $\CalQ$-moderate. Further,
let $v_{0}$ be $\left(\CalQ,\Omega_{0},\Omega_{1},K\right)$-regular
and let $v$ be $v_{0}$-moderate.

Then, the \textbf{(weighted) decomposition space (quasi)-norm} of
$g\in Z'\left(\CalO\right)$ is defined as
\[
\left\Vert g\right\Vert _{\DecompSp{\CalQ}p{\ell_{w}^{q}}v}:=\left\Vert \left(\left\Vert \Fourier^{-1}\left(\varphi_{i}\cdot\widehat{g}\right)\right\Vert _{L_{v}^{p}}\right)_{i\in I}\right\Vert _{\ell_{w}^{q}}\in\left[0,\infty\right]
\]
and the associated \textbf{(weighted) decomposition space} is $\DecompSp{\CalQ}p{\ell_{w}^{q}}v:=\left\{ g\in Z'\left(\CalO\right)\with\left\Vert g\right\Vert _{\DecompSp{\CalQ}p{\ell_{w}^{q}}v}<\infty\right\} $.
\end{defn}
\begin{rem*}
It is a consequence of \cite[Proposition 2.24, Lemma 5.5, and Corollary 6.5]{StructuredBanachFrames}
that the resulting space is a well-defined Quasi-Banach space, with
equivalent (quasi)-norms for different choices of $\Phi$. Indeed,
\cite[Proposition 2.24]{StructuredBanachFrames} shows that the definition
is independent of the $\CalQ$-$v_{0}$-BAPU $\Phi$, while \cite[Corollary 6.5]{StructuredBanachFrames}
ensures that every regular partition of unity is a $\CalQ$-$v_{0}$-BAPU.
Finally, \cite[Lemma 5.5]{StructuredBanachFrames} establishes completeness
of $\DecompSp{\CalQ}p{\ell_{w}^{q}}v$.
\end{rem*}
Recall from Section \ref{sec:BanachFrameDecompositionCrashCourse}
that the Banach frame and atomic decomposition results for $\DecompSp{\CalQ}p{\ell_{w}^{q}}{}$
were formulated in terms of the coefficient space $C_{w}^{p,q}$ from
Definition \ref{def:CoefficientSpace}. This coefficient space needs
to be slightly adjusted in the present case.
\begin{defn}
\label{def:WeightedCoefficientSpace}Under the assumptions of Definition
\ref{def:SpaceWeightedDecompositionSpace} and for $\delta\in\left(0,\infty\right)$,
define the \textbf{weighted coefficient space} $C_{w,v,\delta}^{p,q}$
as
\[
C_{w,v,\delta}^{p,q}:=\left\{ c=\left(\smash{c_{k}^{\left(i\right)}}\right)_{i\in I,k\in\Z^{\dimension}}\with\left\Vert c\right\Vert _{C_{w,v,\delta}^{p,q}}:=\left\Vert \left(\left|\det T_{i}\right|^{\frac{1}{2}-\frac{1}{p}}\cdot w_{i}\cdot\left\Vert \left[v\left(\delta\cdot T_{i}^{-T}k\right)\cdot c_{k}^{\left(i\right)}\right]_{k\in\Z^{\dimension}}\right\Vert _{\ell^{p}}\right)_{i\in I}\right\Vert _{\ell^{q}}<\infty\right\} .\qedhere
\]
\end{defn}
The corresponding ``weighted version'' of Theorem \ref{thm:BanachFrameTheorem}
on the existence of Banach frames for decomposition spaces reads as
follows:
\begin{thm}
\label{thm:WeightedBanachFrameTheorem}Assume that $\CalQ$ satisfies
Assumption \ref{assu:CrashCourseStandingAssumptions}. Let $\Omega_{0},\Omega_{1}\in\left[1,\infty\right)$,
$K\in\left[0,\infty\right)$ and $\varepsilon,p_{0},q_{0}\in\left(0,1\right]$.
Let $v_{0}$ be $\left(\CalQ,\Omega_{0},\Omega_{1},K\right)$-regular.
Let $w=\left(w_{i}\right)_{i\in I}$ be a $\CalQ$-moderate weight
and let $v$ be $v_{0}$-moderate. Finally, let $p,q\in\left(0,\infty\right]$
with $p\geq p_{0}$ and $q\geq q_{0}$.

Define
\[
N:=\left\lceil K+\frac{\dimension+\varepsilon}{\min\left\{ 1,p\right\} }\right\rceil ,\qquad\tau:=\min\left\{ 1,p,q\right\} \qquad\text{ and }\qquad\sigma:=\tau\cdot\left(\frac{\dimension}{\min\left\{ 1,p\right\} }+K+N\right).
\]
Let $\gamma_{1}^{\left(0\right)},\dots,\gamma_{n}^{\left(0\right)}:\R^{\dimension}\to\Compl$
be given and define $\gamma_{i}:=\gamma_{k_{i}}^{\left(0\right)}$
for $i\in I$. Assume that the following conditions are satisfied:

\begin{enumerate}
\item We have $\gamma_{k}^{\left(0\right)}\in L_{\left(1+\left|\mybullet\right|\right)^{K}}^{1}\left(\R^{\dimension}\right)$
and $\Fourier\gamma_{k}^{\left(0\right)}\in C^{\infty}\left(\R^{\dimension}\right)$
for all $k\in\underline{n}$, where all partial derivatives of $\Fourier\gamma_{k}^{\left(0\right)}$
are polynomially bounded.
\item We have $\left[\Fourier\gamma_{k}^{\left(0\right)}\right]\left(\xi\right)\neq0$
for all $\xi\in\overline{Q_{0}^{\left(k\right)}}$ and all $k\in\underline{n}$.
\item We have $\gamma_{k}^{\left(0\right)}\in C^{1}\left(\R^{\dimension}\right)$
and $\nabla\gamma_{k}^{\left(0\right)}\in L_{v_{0}}^{1}\left(\R^{\dimension}\right)\cap L^{\infty}\left(\R^{\dimension}\right)$
for all $k\in\underline{n}$.
\item We have
\[
C_{1}:=\sup_{i\in I}\:\sum_{j\in I}M_{j,i}<\infty\quad\text{ and }\quad C_{2}:=\sup_{j\in I}\:\sum_{i\in I}M_{j,i}<\infty,
\]
where
\[
\qquad\qquad M_{j,i}:=\left(\frac{w_{j}}{w_{i}}\right)^{\tau}\cdot\left(1+\left\Vert T_{j}^{-1}T_{i}\right\Vert \right)^{\sigma}\cdot\max_{\left|\beta\right|\leq1}\left(\left|\det T_{i}\right|^{-1}\cdot\int_{Q_{i}}\:\max_{\left|\alpha\right|\leq N}\left|\left(\left[\partial^{\alpha}\widehat{\partial^{\beta}\gamma_{j}}\right]\left(T_{j}^{-1}\left(\xi-b_{j}\right)\right)\right)\right|\d\xi\right)^{\tau}.
\]
\end{enumerate}
Then there is some $\delta_{0}=\delta_{0}\left(p,q,w,v,v_{0},\varepsilon,\left(\gamma_{i}\right)_{i\in I}\right)>0$
such that for arbitrary $0<\delta\leq\delta_{0}$, the family 
\[
\left(L_{\delta\cdot T_{i}^{-T}k}\:\widetilde{\gamma^{\left[i\right]}}\right)_{i\in I,k\in\Z^{\dimension}}\quad\text{ with }\quad\gamma^{\left[i\right]}=\left|\det T_{i}\right|^{1/2}\cdot M_{b_{i}}\left[\gamma_{i}\circ T_{i}^{T}\right]\quad\text{ and }\quad\widetilde{\gamma^{\left[i\right]}}\left(x\right)=\gamma^{\left[i\right]}\left(-x\right)
\]
forms a \textbf{Banach frame} for $\DecompSp{\CalQ}p{\ell_{w}^{q}}v$.
Precisely, this means the following:

\begin{itemize}[leftmargin=0.7cm]
\item The \textbf{analysis operator} 
\[
A^{\left(\delta\right)}:\DecompSp{\CalQ}p{\ell_{w}^{q}}v\to C_{w,v,\delta}^{p,q},f\mapsto\left(\left[\smash{\gamma^{\left[i\right]}}\ast f\right]\left(\delta\cdot T_{i}^{-T}k\right)\right)_{i\in I,k\in\Z^{\dimension}}
\]
is well-defined and bounded for each $\delta\in\left(0,1\right]$.
Here, the convolution $\gamma^{\left[i\right]}\ast f$ is defined
as in equation \eqref{eq:SpecialConvolutionDefinition}, where now
the series converges normally in $L_{\left(1+\left|\mybullet\right|\right)^{-K}}^{\infty}\left(\R^{\dimension}\right)$
and thus absolutely and locally uniformly, for each $f\in\DecompSp{\CalQ}p{\ell_{w}^{q}}v$.
Of course, the simplified expression from Lemma \ref{lem:SpecialConvolutionClarification}
still holds if $f\in L^{2}\left(\R^{\dimension}\right)\subset Z'\left(\CalO\right)$.
\item For $0<\delta\leq\delta_{0}$, there is a bounded linear \textbf{reconstruction
operator} $R^{\left(\delta\right)}:C_{w,v,\delta}^{p,q}\to\DecompSp{\CalQ}p{\ell_{w}^{q}}v$
satisfying $R^{\left(\delta\right)}\circ A^{\left(\delta\right)}=\identity_{\DecompSp{\CalQ}p{\ell_{w}^{q}}v}$.
\item We have the following \textbf{consistency property}: If $\CalQ$-moderate
weights $w^{\left(1\right)}=\left(\smash{w_{i}^{\left(1\right)}}\right)_{i\in I}$
and $w^{\left(2\right)}=\left(\smash{w_{i}^{\left(2\right)}}\right)_{i\in I}$
and exponents $p_{1},p_{2},q_{1},q_{2}\in\left(0,\infty\right]$,
as well as two $v_{0}$-moderate weights $v_{1},v_{2}:\R^{\dimension}\to\left(0,\infty\right)$
are chosen such that the assumptions of the current theorem are satisfied
for $p_{1},q_{1},w^{\left(1\right)},v_{1}$, as well as for $p_{2},q_{2},w^{\left(2\right)},v_{2}$
and if $0<\delta\leq\min\left\{ \delta_{0}\left(p_{1},q_{1},w^{\left(1\right)},v_{1},v_{0},\varepsilon,\left(\gamma_{i}\right)_{i\in I}\right),\delta_{0}\left(p_{2},q_{2},w^{\left(2\right)},v_{2},v_{0},\varepsilon,\left(\gamma_{i}\right)_{i\in I}\right)\right\} $,
then we have the following equivalence:
\[
\forall f\in\DecompSp{\CalQ}{p_{2}}{\ell_{w^{\left(2\right)}}^{q_{2}}}{v_{2}}:\quad f\in\DecompSp{\CalQ}{p_{1}}{\ell_{w^{\left(1\right)}}^{q_{1}}}{v_{1}}\Longleftrightarrow\left(\left[\smash{\gamma^{\left[i\right]}}\ast f\right]\left(\delta\cdot T_{i}^{-T}k\right)\right)_{i\in I,k\in\Z^{\dimension}}\in C_{w^{\left(1\right)},v_{1},\delta}^{p_{1},q_{1}}.
\]
\end{itemize}
Finally, there is an estimate for the size of $\delta_{0}$ which
is independent of the choice of $p\geq p_{0}$ and $q\geq q_{0}$
and of $v,v_{0}$: There is a constant $L=L\left(p_{0},q_{0},K,\varepsilon,\dimension,\CalQ,\Phi,\Omega_{0},\Omega_{1},\gamma_{1}^{\left(0\right)},\dots,\gamma_{n}^{\left(0\right)}\right)>0$
such that we can choose 
\[
\delta_{0}=\left(1+L\cdot C_{\CalQ,w}^{4}\cdot\left(C_{1}^{1/\tau}+C_{2}^{1/\tau}\right)^{2}\right)^{-1}.\qedhere
\]
\end{thm}
\begin{proof}
For brevity, set $N_{0}:=\left\lceil K+p_{0}^{-1}\cdot\left(\dimension+\varepsilon\right)\right\rceil $
and note $N\leq N_{0}$.

First of all, we verify that the family $\Gamma=\left(\gamma_{i}\right)_{i\in I}$
satisfies \cite[Assumption 3.6]{StructuredBanachFrames}. To this
end, we want to apply \cite[Lemma 3.7]{StructuredBanachFrames} (with
$N=n$). Recall that $\gamma_{i}=\gamma_{k_{i}}^{\left(0\right)}$
and from Assumption \ref{assu:CrashCourseStandingAssumptions} that
$Q_{i}'=Q_{0}^{\left(k_{i}\right)}$ for all $i\in I$. Thus, in the
notation of \cite[Lemma 3.7]{StructuredBanachFrames}, we have for
$k\in\underline{n}$ that
\[
Q^{\left(k\right)}=\bigcup\left\{ Q_{i}'\with i\in I\text{ and }k_{i}=k\right\} \subset Q_{0}^{\left(k\right)}.
\]
But by our assumption, by continuity of $\Fourier\gamma_{k}^{\left(0\right)}$
and by compactness of the sets $\overline{Q_{0}^{\left(k\right)}}$,
there is some $c>0$ satisfying $\left|\left[\Fourier\smash{\gamma_{k}^{\left(0\right)}}\right]\left(\xi\right)\right|\geq c$
for all $\xi\in\overline{Q_{0}^{\left(k\right)}}\supset Q^{\left(k\right)}$
and all $k\in\underline{n}$. Consequently, \cite[Lemma 3.7]{StructuredBanachFrames}
shows that $\Gamma$ satisfies \cite[Assumption 3.6]{StructuredBanachFrames}
and also yields the estimate $\Omega_{2}^{\left(p,K\right)}\leq\Omega_{3}$
for a constant $\Omega_{3}=\Omega_{3}\left(\CalQ,\gamma_{1}^{\left(0\right)},\dots,\gamma_{n}^{\left(0\right)},p_{0},K,\dimension\right)>0$.
Here, $\Omega_{2}^{\left(p,K\right)}$ is a constant defined in \cite[Assumption 3.6]{StructuredBanachFrames}.
To obtain this estimate, we used that $p\geq p_{0}$.

Now, since the family $\Gamma$ satisfies \cite[Assumption 3.6]{StructuredBanachFrames},
the assumptions of the present theorem easily imply that all assumptions
of \cite[Corollary 6.6]{StructuredBanachFrames} are satisfied. This
uses the special structure of the family $\Gamma=\left(\gamma_{i}\right)_{i\in I}$,
i.e., that $\gamma_{i}=\gamma_{k_{i}}^{\left(0\right)}$ for each
$i\in I$.

In particular, \cite[Corollary 6.6]{StructuredBanachFrames} shows
that the operators $\overrightarrow{A}$ and $\overrightarrow{B}$
from \cite[Assumption 3.1 and Assumption 4.1]{StructuredBanachFrames}
are well-defined and bounded with $\vertiii{\smash{\overrightarrow{A}}}^{\max\left\{ 1,\frac{1}{p}\right\} }\leq L_{1}^{\left(0\right)}\cdot\left(C_{1}^{1/\tau}+C_{2}^{1/\tau}\right)$
and $\vertiii{\smash{\overrightarrow{B}}}^{\max\left\{ 1,\frac{1}{p}\right\} }\leq L_{1}^{\left(0\right)}\cdot\left(C_{1}^{1/\tau}+C_{2}^{1/\tau}\right)$
for
\begin{align*}
L_{1}^{\left(0\right)} & =\Omega_{0}^{K}\Omega_{1}\cdot\dimension^{1/\min\left\{ 1,p\right\} }\cdot\left(4\dimension\right)^{1+2N}\cdot\left(\varepsilon^{-1}\cdot s_{\dimension}\right)^{1/\min\left\{ 1,p\right\} }\cdot\max_{\left|\alpha\right|\leq N}C^{\left(\alpha\right)}\\
 & \leq\Omega_{0}^{K}\Omega_{1}\cdot\dimension^{1/p_{0}}\cdot\left(4\dimension\right)^{1+2N_{0}}\cdot\left(1+\varepsilon^{-1}\cdot s_{\dimension}\right)^{1/p_{0}}\cdot\max_{\left|\alpha\right|\leq N_{0}}C^{\left(\alpha\right)}=:L_{1}.
\end{align*}
Note that $L_{1}=L_{1}\left(\dimension,\varepsilon,\CalQ,\Phi,p_{0},\Omega_{0},\Omega_{1},K\right)$,
since the constants $C^{\left(\alpha\right)}$ from Definition \ref{def:RegularPartitionOfUnity}
only depend on $\alpha,\CalQ,\Phi$.

Since \cite[Corollary 6.6]{StructuredBanachFrames} is applicable
to $\Gamma$, we see that $\Gamma$ satisfies \cite[Assumption 4.1]{StructuredBanachFrames}.
Therefore, \cite[Lemma 4.3]{StructuredBanachFrames} shows that the
series in equation \eqref{eq:SpecialConvolutionDefinition} converges
normally in $L_{\left(1+\left|\mybullet\right|\right)^{-K}}^{\infty}\left(\R^{\dimension}\right)$
for all $f\in\DecompSp{\CalQ}p{\ell_{w}^{q}}v$. Since each of the
summands of the series is a continuous functions, this yields absolute
and locally uniform convergence of the series.

Next, since $\overrightarrow{A}$ and $\overrightarrow{B}$ are bounded,
\cite[Theorem 4.7]{StructuredBanachFrames} is applicable. This shows
that the family $\left(L_{\delta\cdot T_{i}^{-T}k}\:\widetilde{\gamma^{\left[i\right]}}\right)_{i\in I,k\in\Z^{\dimension}}$
yields a Banach frame for $\DecompSp{\CalQ}p{\ell_{w}^{q}}v$ as in
the statement of the current theorem, as soon as $0<\delta\leq\delta_{00}$
for $\delta_{00}=1\big/\left(1+2\cdot\smash{\vertiii{F_{0}}^{2}}\vphantom{\vertiii{F_{0}}}\right)$,
where the operator $F_{0}$ is defined in \cite[Lemma 4.6]{StructuredBanachFrames}.
That lemma also yields the estimate
\begin{align*}
\vertiii{F_{0}} & \leq2^{\frac{1}{q}}C_{\CalQ,\Phi,v_{0},p}^{2}\cdot\vertiii{\Gamma_{\CalQ}}^{2}\cdot\left(\vertiii{\smash{\overrightarrow{A}}}^{\max\left\{ 1,\frac{1}{p}\right\} }+\vertiii{\smash{\overrightarrow{B}}}^{\max\left\{ 1,\frac{1}{p}\right\} }\right)\cdot L_{2}^{\left(0\right)}\\
 & \leq2^{\frac{1}{q_{0}}}C_{\CalQ,\Phi,v_{0},p}^{2}\cdot\vertiii{\Gamma_{\CalQ}}^{2}\cdot\left(C_{1}^{1/\tau}+C_{2}^{1/\tau}\right)\cdot2L_{1}\cdot L_{2}^{\left(0\right)},
\end{align*}
where
\[
C_{\CalQ,\Phi,v_{0},p}=\sup_{i\in I}\left[\left|\det T_{i}\right|^{\frac{1}{\min\left\{ 1,p\right\} }-1}\cdot\left\Vert \Fourier^{-1}\varphi_{i}\right\Vert _{L_{v_{0}}^{\min\left\{ 1,p\right\} }}\right]
\]
and where $\Gamma_{\CalQ}:\ell_{w}^{q}\left(I\right)\to\ell_{w}^{q}\left(I\right)$
is the \textbf{$\CalQ$-clustering map} given by $\Gamma_{\CalQ}\left(c_{i}\right)_{i\in I}=\left(c_{i}^{\ast}\right)_{i\in I}$
where $c_{i}^{\ast}=\sum_{\ell\in i^{\ast}}c_{\ell}$. Further, with
$M:=\left\lceil K+\frac{\dimension+1}{\min\left\{ 1,p\right\} }\right\rceil \leq\left\lceil K+\frac{\dimension+1}{p_{0}}\right\rceil =:M_{0}$,
the constant $L_{2}^{\left(0\right)}$ is given by
\begin{align*}
L_{2}^{\left(0\right)} & =\begin{cases}
\frac{\left(2^{16}\cdot768/\dimension^{\frac{3}{2}}\right)^{\frac{\dimension}{p}}}{2^{42}\cdot12^{\dimension}\cdot\dimension^{15}}\cdot\left(2^{52}\cdot\dimension^{\frac{25}{2}}\cdot M^{3}\right)^{M+1}\cdot N_{\CalQ}^{2\left(\frac{1}{p}-1\right)}\left(1+R_{\CalQ}C_{\CalQ}\right)^{\dimension\left(\frac{4}{p}-1\right)}\cdot\Omega_{0}^{13K}\Omega_{1}^{13}\Omega_{2}^{\left(p,K\right)}, & \text{if }p<1,\\
\left(2^{12+6\left\lceil K\right\rceil }\cdot\sqrt{\dimension}\right)^{-1}\cdot\left(2^{17}\cdot\dimension^{5/2}\cdot M\right)^{\left\lceil K\right\rceil +\dimension+2}\cdot\left(1+R_{\CalQ}\right)^{\dimension}\cdot\Omega_{0}^{3K}\Omega_{1}^{3}\Omega_{2}^{\left(p,K\right)}, & \text{if }p\geq1
\end{cases}\\
 & \leq\begin{cases}
2^{26\dimension/p_{0}}\cdot\left(2^{52}\cdot\dimension^{13}\cdot M_{0}^{3}\right)^{M_{0}+1}\cdot N_{\CalQ}^{\frac{2}{p_{0}}}\cdot\left(1+R_{\CalQ}C_{\CalQ}\right)^{\frac{4\dimension}{p_{0}}}\cdot\Omega_{0}^{13K}\Omega_{1}^{13}\Omega_{3}, & \text{if }p<1,\\
\left(2^{17}\cdot\dimension^{5/2}\cdot M_{0}\right)^{\left\lceil K\right\rceil +\dimension+2}\cdot\left(1+R_{\CalQ}\right)^{\dimension}\cdot\Omega_{0}^{3K}\Omega_{1}^{3}\Omega_{3}, & \text{if }p\geq1
\end{cases}\\
\left({\scriptstyle \text{since }C_{\CalQ}\geq1}\right) & \leq2^{26\dimension/p_{0}}\cdot\left(2^{52}\cdot\dimension^{13}\cdot M_{0}^{3}\right)^{M_{0}+1}\cdot N_{\CalQ}^{\frac{2}{p_{0}}}\cdot\left(1+R_{\CalQ}C_{\CalQ}\right)^{\frac{4\dimension}{p_{0}}}\cdot\Omega_{0}^{13K}\Omega_{1}^{13}\Omega_{3}=:L_{2}.
\end{align*}
Here, the last step used that $M_{0}=\left\lceil K+p_{0}^{-1}\cdot\left(\dimension+1\right)\right\rceil \geq\left\lceil K+\dimension+1\right\rceil =\left\lceil K\right\rceil +\dimension+1$,
as well as $\Omega_{0},\Omega_{1}\geq1$. Note as above that $L_{2}=L_{2}\left(\dimension,p_{0},\CalQ,\Omega_{0},\Omega_{1},\Omega_{3},K,M_{0}\right)=K_{2}\left(\dimension,p_{0},\CalQ,\Omega_{0},\Omega_{1},K,\smash{\gamma_{1}^{\left(0\right)},\dots,\gamma_{n}^{\left(0\right)}}\right)$.

As seen in \cite[Lemma 4.13]{DecompositionEmbedding}, we have $\vertiii{\Gamma_{\CalQ}}\leq C_{\CalQ,w}\cdot N_{\CalQ}^{1+\frac{1}{q}}\leq C_{\CalQ,w}\cdot N_{\CalQ}^{1+\frac{1}{q_{0}}}$.
Furthermore, \cite[Corollary 6.5]{StructuredBanachFrames} shows that
there is a function $\varrho\in\TestFunctionSpace{\R^{\dimension}}$
(which only depends on $\CalQ$) such that
\begin{align*}
C_{\CalQ,\Phi,v_{0},p} & \leq\Omega_{0}^{K}\Omega_{1}\cdot\left(4\dimension\right)^{1+2N}\cdot\left(\frac{s_{\dimension}}{\varepsilon}\right)^{1/\min\left\{ 1,p\right\} }\cdot2^{N}\cdot\lambda_{\dimension}\left(\overline{\bigcup_{i\in I}Q_{i}'}\right)\cdot\max_{\left|\alpha\right|\leq N}\left\Vert \partial^{\alpha}\varrho\right\Vert _{\sup}\cdot\max_{\left|\alpha\right|\leq N}C^{\left(\alpha\right)}\\
 & \leq\Omega_{0}^{K}\Omega_{1}\cdot\left(8\dimension\right)^{1+2N_{0}}\cdot\left(1+\frac{s_{\dimension}}{\varepsilon}\right)^{1/p_{0}}\cdot\left(2R_{\CalQ}\right)^{\dimension}\cdot\max_{\left|\alpha\right|\leq N_{0}}\left\Vert \partial^{\alpha}\varrho\right\Vert _{\sup}\cdot\max_{\left|\alpha\right|\leq N_{0}}C^{\left(\alpha\right)}=:L_{3}.
\end{align*}
Here, we used that $Q_{i}'\subset\overline{B_{R_{\CalQ}}}\left(0\right)\subset\left[-R_{\CalQ},R_{\CalQ}\right]^{\dimension}$
for all $i\in I$. Since the constants $C^{\left(\alpha\right)}=C^{\left(\alpha\right)}\left(\Phi,\CalQ\right)$
from Definition \ref{def:RegularPartitionOfUnity} only depend on
$\alpha,\Phi,\CalQ$ and since $\varrho$ only depends on $\CalQ$,
we see $L_{3}=L_{3}\left(\dimension,\varepsilon,p_{0},\CalQ,\Phi,\Omega_{0},\Omega_{1},K\right)$.

All in all, we arrive at
\[
\vertiii{F_{0}}\leq2^{1+\frac{1}{q_{0}}}N_{\CalQ}^{2+\frac{2}{q_{0}}}\cdot L_{1}L_{2}L_{3}^{2}\cdot C_{\CalQ,w}^{2}\cdot\left(C_{1}^{1/\tau}+C_{2}^{1/\tau}\right)=L_{4}\cdot C_{\CalQ,w}^{2}\cdot\left(C_{1}^{1/\tau}+C_{2}^{1/\tau}\right)
\]
for a suitable constant $L_{4}=L_{4}\left(\dimension,\varepsilon,p_{0},q_{0},\CalQ,\Phi,\gamma_{1}^{\left(0\right)},\dots,\gamma_{n}^{\left(0\right)},\Omega_{0},\Omega_{1},K\right)$,
so that the family $\left(L_{\delta\cdot T_{i}^{-T}k}\:\widetilde{\gamma^{\left[i\right]}}\right)_{i\in I,k\in\Z^{\dimension}}$
yields a Banach frame for $\DecompSp{\CalQ}p{\ell_{w}^{q}}v$ as soon
as $0<\delta\leq\delta_{0}$ for $\delta_{0}:=\left(1+2\left[L_{4}\cdot C_{\CalQ,w}^{2}\cdot\left(C_{1}^{1/\tau}+C_{2}^{1/\tau}\right)\right]^{2}\right)^{-1}$,
since $\delta_{0}\leq\delta_{00}$. Now, setting $L:=2\cdot L_{4}^{2}$
yields the claim.
\end{proof}
Finally, we present a ``weighted version'' of Theorem \ref{thm:AtomicDecompositionTheorem}
concerning the existence of atomic decompositions for decomposition
spaces.
\begin{thm}
\label{thm:WeightedAtomicDecompositionTheorem}Assume that $\CalQ$
satisfies Assumption \ref{assu:CrashCourseStandingAssumptions}. Let
$\Omega_{0},\Omega_{1}\in\left[1,\infty\right)$, $K\in\left[0,\infty\right)$
and $\varepsilon,p_{0},q_{0}\in\left(0,1\right]$. Let $v_{0}$ be
$\left(\CalQ,\Omega_{0},\Omega_{1},K\right)$-regular. Let $w=\left(w_{i}\right)_{i\in I}$
be a $\CalQ$-moderate weight and let $v$ be $v_{0}$-moderate. Finally,
let $p,q\in\left(0,\infty\right]$ with $p\geq p_{0}$ and $q\geq q_{0}$.

Define
\[
N:=\left\lceil K+\frac{\dimension+\varepsilon}{\min\left\{ 1,p\right\} }\right\rceil ,\qquad\tau:=\min\left\{ 1,p,q\right\} ,\qquad\vartheta:=\left(\frac{1}{p}-1\right)_{+}\:,\qquad\text{ and }\qquad\varUpsilon:=K+1+\frac{\dimension}{\min\left\{ 1,p\right\} },
\]
as well as
\[
\sigma:=\begin{cases}
\tau\cdot N, & \text{if }p\in\left[1,\infty\right],\\
\tau\cdot\left(p^{-1}\cdot\dimension+K+N\right), & \text{if }p\in\left(0,1\right).
\end{cases}
\]
Let $\gamma_{1}^{\left(0\right)},\dots,\gamma_{n}^{\left(0\right)}:\R^{\dimension}\to\Compl$
be given and define $\gamma_{i}:=\gamma_{k_{i}}^{\left(0\right)}$
for $i\in I$. Assume that there are functions $\gamma_{1}^{\left(0,j\right)},\dots,\gamma_{n}^{\left(0,j\right)}$
for $j\in\left\{ 1,2\right\} $ such that the following conditions
are satisfied:

\begin{enumerate}
\item We have $\gamma_{k}^{\left(0,1\right)}\in L_{\left(1+\left|\mybullet\right|\right)^{K}}^{1}\left(\R^{\dimension}\right)$
for all $k\in\underline{n}$.
\item We have $\gamma_{k}^{\left(0,2\right)}\in C^{1}\left(\R^{\dimension}\right)$
for all $k\in\underline{n}$.
\item We have
\[
\Omega^{\left(p\right)}:=\max_{k\in\underline{n}}\left\Vert \gamma_{k}^{\left(0,2\right)}\right\Vert _{\varUpsilon}+\max_{k\in\underline{n}}\left\Vert \nabla\gamma_{k}^{\left(0,2\right)}\right\Vert _{\varUpsilon}<\infty,
\]
where $\left\Vert f\right\Vert _{\varUpsilon}=\sup_{x\in\R^{\dimension}}\left(1+\left|x\right|\right)^{\varUpsilon}\cdot\left|f\left(x\right)\right|$
for $f:\R^{\dimension}\to\Compl^{\ell}$ and (arbitrary) $\ell\in\N$.
\item We have $\Fourier\gamma_{k}^{\left(0,j\right)}\in C^{\infty}\left(\R^{\dimension}\right)$
and all partial derivatives of $\Fourier\gamma_{k}^{\left(0,j\right)}$
are polynomially bounded for all $k\in\underline{n}$ and $j\in\left\{ 1,2\right\} $.
\item We have $\gamma_{k}^{\left(0\right)}=\gamma_{k}^{\left(0,1\right)}\ast\gamma_{k}^{\left(0,2\right)}$
for all $k\in\underline{n}$.
\item We have $\left\Vert \gamma_{k}^{\left(0\right)}\right\Vert _{\varUpsilon}<\infty$
for all $k\in\underline{n}$.
\item We have $\left[\Fourier\gamma_{k}^{\left(0\right)}\right]\left(\xi\right)\neq0$
for all $\xi\in\overline{Q_{0}^{\left(k\right)}}$ and all $k\in\underline{n}$.
\item We have
\[
K_{1}:=\sup_{i\in I}\:\sum_{j\in I}N_{i,j}<\infty\qquad\text{ and }\qquad K_{2}:=\sup_{j\in I}\:\sum_{i\in I}N_{i,j}<\infty,
\]
where $\gamma_{j,1}:=\gamma_{k_{j}}^{\left(0,1\right)}$ for $j\in I$
and
\[
\qquad\qquad N_{i,j}:=\left(\frac{w_{i}}{w_{j}}\cdot\left(\left|\det T_{j}\right|\big/\left|\det T_{i}\right|\right)^{\vartheta}\right)^{\tau}\!\!\cdot\left(1\!+\!\left\Vert T_{j}^{-1}T_{i}\right\Vert \right)^{\sigma}\!\cdot\left(\left|\det T_{i}\right|^{-1}\!\cdot\int_{Q_{i}}\:\max_{\left|\alpha\right|\leq N}\left|\left[\partial^{\alpha}\widehat{\gamma_{j,1}}\right]\left(T_{j}^{-1}\left(\xi\!-\!b_{j}\right)\right)\right|\d\xi\right)^{\tau}.
\]
\end{enumerate}
Then there is some $\delta_{0}\in\left(0,1\right]$ such that the
family 
\[
\Psi_{\delta}:=\left(L_{\delta\cdot T_{i}^{-T}k}\:\gamma^{\left[i\right]}\right)_{i\in I,\,k\in\Z^{\dimension}}\qquad\text{ with }\qquad\gamma^{\left[i\right]}=\left|\det T_{i}\right|^{1/2}\cdot M_{b_{i}}\left[\gamma_{i}\circ T_{i}^{T}\right]
\]
forms an \textbf{atomic decomposition} of $\DecompSp{\CalQ}p{\ell_{w}^{q}}v$,
for all $\delta\in\left(0,\delta_{0}\right]$. Precisely, this means
the following:

\begin{itemize}[leftmargin=0.7cm]
\item The \textbf{synthesis map}
\[
S^{\left(\delta\right)}:C_{w,v,\delta}^{p,q}\to\DecompSp{\CalQ}p{\ell_{w}^{q}}v,\left(\smash{c_{k}^{\left(i\right)}}\right)_{i\in I,\,k\in\Z^{\dimension}}\mapsto\sum_{i\in I}\:\sum_{k\in\Z^{\dimension}}\left[c_{k}^{\left(i\right)}\cdot L_{\delta\cdot T_{i}^{-T}k}\:\gamma^{\left[i\right]}\right]
\]
is well-defined and bounded for every $\delta\in\left(0,1\right]$.
\item For $0<\delta\leq\delta_{0}$, there is a bounded linear \textbf{coefficient
map} $C^{\left(\delta\right)}:\DecompSp{\CalQ}p{\ell_{w}^{q}}v\to C_{w,v,\delta}^{p,q}$
satisfying 
\[
S^{\left(\delta\right)}\circ C^{\left(\delta\right)}=\identity_{\DecompSp{\CalQ}p{\ell_{w}^{q}}v}.
\]
\end{itemize}
Finally, there is an estimate for the size of $\delta_{0}$ which
is independent of $p\geq p_{0}$, $q\geq q_{0}$ and of $v,v_{0}$:
There is a constant $L=L\left(p_{0},q_{0},\varepsilon,\dimension,\CalQ,\Phi,\gamma_{1}^{\left(0\right)},\dots,\gamma_{n}^{\left(0\right)},\Omega_{0},\Omega_{1},K\right)>0$
such that we can choose\vspace{-0.1cm}
\[
\delta_{0}=\min\left\{ 1,\,1\big/\left[L\cdot\Omega^{\left(p\right)}\cdot\left(K_{1}^{1/\tau}+K_{2}^{1/\tau}\right)\right]\right\} .\qedhere
\]
\end{thm}
\begin{rem*}
Convergence of the series defining $S^{\left(\delta\right)}c$ has
to be understood as in the remark to Theorem \ref{thm:AtomicDecompositionTheorem}.
Also as in that remark, the action of the coefficient map $C^{\left(\delta\right)}$
on a given $f\in\DecompSp{\CalQ}p{\ell_{w}^{q}}v$ is \emph{independent}
of the precise choice of $p,q,v,w$, as long as $C^{\left(\delta\right)}f$
is defined at all.
\end{rem*}
\begin{proof}
For brevity, set $N_{0}:=\left\lceil K+p_{0}^{-1}\cdot\left(\dimension+\varepsilon\right)\right\rceil $.
As in the proof of Theorem \ref{thm:WeightedBanachFrameTheorem},
we see as a consequence of \cite[Lemma 3.7]{StructuredBanachFrames}
that $\Gamma=\left(\gamma_{i}\right)_{i\in I}$ satisfies \cite[Assumption 3.6]{StructuredBanachFrames},
with $\Omega_{2}^{\left(p,K\right)}\leq\Omega_{3}$ for a suitable
constant $\Omega_{3}=\Omega_{3}\left(\CalQ,\gamma_{1}^{\left(0\right)},\dots,\gamma_{n}^{\left(0\right)},p_{0},\dimension,K\right)>0$.
For brevity, set $\Omega_{5}:=\Omega_{0}^{16K}\Omega_{1}^{16}\Omega_{3}$.

Now, since we have $\gamma_{i}=\gamma_{k_{i}}^{\left(0\right)}=\gamma_{k_{i}}^{\left(0,1\right)}\ast\gamma_{k_{i}}^{\left(0,2\right)}$
for all $i\in I$, it is easy to see that all assumptions of \cite[Corollary 6.7]{StructuredBanachFrames}
are satisfied. Consequently, \cite[Corollary 6.7]{StructuredBanachFrames}
shows that the operator $\overrightarrow{C}$ defined in \cite[Assumption 5.1]{StructuredBanachFrames}
is well-defined and bounded, with
\[
\vertiii{\smash{\overrightarrow{C}}}^{\max\left\{ 1,\frac{1}{p}\right\} }\leq L_{1}^{\left(0\right)}\cdot\left(K_{1}^{1/\tau}+K_{2}^{1/\tau}\right),
\]
where
\[
L_{1}^{\left(0\right)}=\Omega_{0}^{K}\Omega_{1}\cdot\left(4\dimension\right)^{1+2N}\cdot\left(\frac{s_{\dimension}}{\varepsilon}\right)^{1/\min\left\{ 1,p\right\} }\cdot\max_{\left|\alpha\right|\leq N}C^{\left(\alpha\right)}\leq\Omega_{0}^{K}\Omega_{1}\left(4\dimension\right)^{1+2N_{0}}\cdot\left(1+\frac{s_{\dimension}}{\varepsilon}\right)^{1/p_{0}}\cdot\max_{\left|\alpha\right|\leq N_{0}}C^{\left(\alpha\right)}=:L_{1},
\]
where the constants $C^{\left(\alpha\right)}=C^{\left(\alpha\right)}\left(\CalQ,\Phi\right)$
are as in Definition \ref{def:RegularPartitionOfUnity}. Thus, $L_{1}=L_{1}\left(\dimension,p_{0},\varepsilon,\CalQ,\Phi,\Omega_{0},\Omega_{1},K\right)$.

Finally, \cite[Corollary 6.7]{StructuredBanachFrames} shows that
$\Gamma$ satisfies all assumptions of \cite[Theorem 5.6]{StructuredBanachFrames},
so that the family $\Psi_{\delta}$ defined in the statement of the
theorem yields an atomic decomposition of $\DecompSp{\CalQ}p{\ell_{w}^{q}}v$
as soon as $0<\delta\leq\min\left\{ 1,\delta_{00}\right\} $, where
$\delta_{00}>0$ is defined by
\begin{align*}
\delta_{00}^{-1}\! & :=\!\begin{cases}
\!\frac{2s_{\dimension}}{\sqrt{\dimension}}\cdot\left(2^{17}\!\cdot\!\dimension^{2}\!\cdot\!\left(K+2+\dimension\right)\right)^{K+\dimension+3}\!\!\!\cdot\left(1+R_{\CalQ}\right)^{\dimension+1}\cdot\Omega_{0}^{4K}\Omega_{1}^{4}\Omega_{2}^{\left(p,K\right)}\Omega^{\left(p\right)}\cdot\vertiii{\smash{\overrightarrow{C}}}\,, & \text{if }p\geq1,\\
\frac{\left(2^{14}/\dimension^{\frac{3}{2}}\right)^{\!\frac{\dimension}{p}}}{2^{45}\cdot\dimension^{17}}\!\cdot\!\left(\frac{s_{\dimension}}{p}\right)^{\!\frac{1}{p}}\left(2^{68}\dimension^{14}\!\cdot\!\left(K+1+\frac{\dimension+1}{p}\right)^{\!3}\right)^{K+2+\frac{\dimension+1}{p}}\!\!\!\cdot\!\left(1\!+\!R_{\CalQ}\right)^{1+\frac{3\dimension}{p}}\!\cdot\!\Omega_{0}^{16K}\Omega_{1}^{16}\Omega_{2}^{\left(p,K\right)}\Omega^{\left(p\right)}\cdot\vertiii{\smash{\overrightarrow{C}}}^{\frac{1}{p}}, & \text{if }p<1
\end{cases}\\
 & \leq\begin{cases}
2s_{\dimension}\cdot\left(2^{17}\!\cdot\!\dimension^{2}\!\cdot\!\left(K+2+\dimension\right)\right)^{K+\dimension+3}\!\!\!\cdot\left(1+R_{\CalQ}\right)^{\dimension+1}\cdot\Omega_{0}^{4K}\Omega_{1}^{4}\Omega_{3}\Omega^{\left(p\right)}\cdot L_{1}\cdot\left(K_{1}^{1/\tau}+K_{2}^{1/\tau}\right)\,, & \text{if }p\geq1,\\
2^{14\dimension/p_{0}}\left(1\!+\!\frac{s_{\dimension}}{p_{0}}\right)^{\!\frac{1}{p_{0}}}\!\!\cdot\left(2^{68}\dimension^{14}\!\cdot\!\left(K\!+\!1\!+\!\frac{\dimension+1}{p_{0}}\right)^{\!3}\right)^{K+2+\frac{\dimension+1}{p_{0}}}\!\!\!\cdot\!\left(1\!+\!R_{\CalQ}\right)^{1+\frac{3\dimension}{p_{0}}}\!\cdot\!\Omega_{5}\Omega^{\left(p\right)}\cdot L_{1}\cdot\left(\!K_{1}^{1/\tau}\!+\!K_{2}^{1/\tau}\right), & \text{if }p<1
\end{cases}\\
 & \leq L_{2}\cdot\Omega^{\left(p\right)}\cdot\left(K_{1}^{1/\tau}+K_{2}^{1/\tau}\right)
\end{align*}
for\vspace{-0.1cm}
\[
L_{2}:=\left[2^{14\dimension}\cdot\left(1+p_{0}^{-1}\cdot s_{\dimension}\right)\right]^{1/p_{0}}\cdot\left(2^{68}\dimension^{14}\!\cdot\!\left(K+1+\left(\dimension+1\right)\cdot p_{0}^{-1}\right)^{\!3}\right)^{K+2+\frac{\dimension+1}{p_{0}}}\!\!\!\cdot\!\left(1+R_{\CalQ}\right)^{1+\frac{3\dimension}{p_{0}}}\!\cdot\!\Omega_{5}\cdot L_{1}.
\]
Here, our application of \cite[Theorem 5.6]{StructuredBanachFrames}
implicitly used that the constant $\Omega^{\left(p\right)}$ from
the statement of Theorem \ref{thm:WeightedAtomicDecompositionTheorem}
satisfies $\Omega^{\left(p\right)}=\Omega_{4}^{\left(p,K\right)}$
with $\Omega_{4}^{\left(p,K\right)}$ as in \cite[Assumption 5.1]{StructuredBanachFrames}.

Note that $L:=L_{2}=L_{2}\left(\dimension,\varepsilon,p_{0},\CalQ,\Phi,\gamma_{1}^{\left(0\right)},\dots,\gamma_{n}^{\left(0\right)},\Omega_{0},\Omega_{1},K\right)$
and finally observe that if\vspace{-0.1cm}
\[
\delta_{0}=\min\left\{ 1,\,\left[L\cdot\Omega^{\left(p\right)}\cdot\left(K_{1}^{1/\tau}+K_{2}^{1/\tau}\right)\right]^{-1}\right\} 
\]
is defined as in the statement of Theorem \ref{thm:WeightedAtomicDecompositionTheorem},
then $\delta_{0}\leq\min\left\{ 1,\delta_{00}\right\} $, so that
the family $\Psi_{\delta}$ indeed yields an atomic decomposition
of $\DecompSp{\CalQ}p{\ell_{w}^{q}}v$ as soon as $\delta\in\left(0,\delta_{0}\right]$.
Finally, the remark associated to \cite[Theorem 5.6]{StructuredBanachFrames}
shows that convergence of the series in the definition of $S^{\left(\delta\right)}$
occurs as claimed in the remark after Theorem \ref{thm:AtomicDecompositionTheorem}
and that the action of $C^{\left(\delta\right)}$ on a given $f\in\DecompSp{\CalQ}p{\ell_{w}^{q}}v$
is independent of the precise choice of $p,q,w,v$, as claimed in
the remark to Theorem \ref{thm:WeightedAtomicDecompositionTheorem}.
\end{proof}

\subsection{Cartoon approximation with \texorpdfstring{$\alpha$}{α}-shearlets
and polynomial search depth}

In view of the results in the preceding subsection, we first define
a new variant of the $\alpha$-shearlet smoothness spaces:
\begin{defn}
\label{def:SpaceWeightedAlphaShearletSmoothnessSpaces}Let $\alpha\in\left[0,1\right]$,
let $\omega_{0}$ be the standard weight from Example \ref{exa:StandardShearletSpaceWeight}
and let $u=\left(u_{v}\right)_{v\in V^{\left(\alpha\right)}}$ be
as in Definition \ref{def:UnconnectedAlphaShearletCovering}. For
$p,q\in\left(0,\infty\right]$ and $s,\kappa\in\R$, we define the
\textbf{(weighted) $\alpha$-shearlet smoothness space} as
\[
\mathscr{S}_{\alpha,s,\kappa}^{p,q}\left(\R^{2}\right):=\DecompSp{\CalS_{u}^{\left(\alpha\right)}}p{\ell_{u^{s}}^{q}}{\omega_{0}^{\kappa}}.\qedhere
\]
\end{defn}
In this section, we will only consider exponents $\kappa\geq0$, for
which clearly $\mathscr{S}_{\alpha,s,\kappa}^{p,q}\left(\R^{2}\right)\hookrightarrow\mathscr{S}_{\alpha,s}^{p,q}\left(\R^{2}\right)\hookrightarrow\Schwartz'\left(\R^{2}\right)$,
cf.\@ Lemma \ref{lem:AlphaShearletIntoTemperedDistributions}. Now,
for $0\leq\kappa\leq\kappa_{0}$, Example \ref{exa:StandardShearletSpaceWeight}
shows that the weight $\omega_{0}^{\kappa}$ used above is $\omega_{0}^{\kappa_{0}}$-moderate
and that $\omega_{0}^{\kappa_{0}}$ is $\left(\smash{\CalS_{u}^{\left(\alpha\right)}},3,1,\kappa_{0}\right)$-regular.
Then, by repeating the proofs of Theorems \ref{thm:NicelySimplifiedAlphaShearletFrameConditions}
and \ref{thm:ReallyNiceShearletAtomicDecompositionConditions} for
the modified values of $N,\sigma,\tau$ or $N,\sigma,\tau,\varUpsilon$,
one easily sees that Theorems \ref{thm:NicelySimplifiedUnconnectedAlphaShearletFrameConditions}
and \ref{thm:ReallyNiceUnconnectedShearletAtomicDecompositionConditions}
remain valid (with the proper modifications) for the more general
spaces $\mathscr{S}_{\alpha,s,\kappa}^{p,q}\left(\R^{2}\right)$,
cf.\@ Theorems \ref{thm:WeightedAlphaShearletBanachFrameCondition}
and \ref{thm:WeightedAlphaShearletAtomicDecompositionCondition} below.

The only nontrivial modification in the proof is the following: In
the proof of Theorem \ref{thm:ReallyNiceShearletAtomicDecompositionConditions},
Proposition \ref{prop:ConvolutionFactorization} (with $N=N_{0}$)
is used to obtain factorizations $\varphi=\varphi_{1}\ast\varphi_{2}$
and $\psi=\psi_{1}\ast\psi_{2}$, where one still has a certain control
over $\varphi_{1},\varphi_{2},\psi_{1},\psi_{2}$. Indeed, Proposition
\ref{prop:ConvolutionFactorization} ensures that $\varphi_{2},\psi_{2},\nabla\varphi_{2},\nabla\psi_{2}$
decay faster than any polynomial, so that the constant $\Omega^{\left(p\right)}$
from Theorem \ref{thm:WeightedAtomicDecompositionTheorem} is finite.
But Theorem \ref{thm:WeightedAtomicDecompositionTheorem} requires
$\varphi_{1},\psi_{1}\in L_{\left(1+\left|\mybullet\right|\right)^{\kappa_{0}}}^{1}\left(\R^{2}\right)$,
whereas Theorem \ref{thm:AtomicDecompositionTheorem} only required
$\varphi_{1},\psi_{1}\in L^{1}\left(\R^{2}\right)$. But this is still
guaranteed by Proposition \ref{prop:ConvolutionFactorization}, since
it implies $\left\Vert \varphi_{1}\right\Vert _{N_{0}},\left\Vert \psi_{1}\right\Vert _{N_{0}}<\infty$,
where now $N_{0}=\left\lceil \kappa_{0}+p_{0}^{-1}\cdot\left(2+\varepsilon\right)\right\rceil \geq\kappa_{0}+2+\varepsilon>\kappa_{0}+2$,
from which we easily get $\varphi_{1},\psi_{1}\in L_{\left(1+\left|\mybullet\right|\right)^{\kappa_{0}}}^{1}\left(\R^{2}\right)$.

The precise statements of the ``weighted'' versions of Theorems
\ref{thm:NicelySimplifiedUnconnectedAlphaShearletFrameConditions}
and \ref{thm:ReallyNiceUnconnectedShearletAtomicDecompositionConditions}
are as follows:
\begin{thm}
\label{thm:WeightedAlphaShearletBanachFrameCondition}Let $\alpha\in\left[0,1\right]$,
$\varepsilon,p_{0},q_{0}\in\left(0,1\right]$, $\kappa_{0}\in\left[0,\infty\right)$
and $s_{0},s_{1}\in\R$ with $s_{0}\leq s_{1}$. Assume that $\varphi,\psi:\R^{2}\rightarrow\Compl$
satisfy the following:

\begin{itemize}[leftmargin=0.6cm]
\item $\varphi,\psi\in L_{\left(1+\left|\mybullet\right|\right)^{\kappa_{0}}}^{1}\left(\R^{2}\right)$
and $\widehat{\varphi},\widehat{\psi}\in C^{\infty}\left(\R^{2}\right)$,
where all partial derivatives of $\widehat{\varphi},\widehat{\psi}$
have at most polynomial growth.
\item $\varphi,\psi\in C^{1}\left(\R^{2}\right)$ and $\nabla\varphi,\nabla\psi\in L_{\left(1+\left|\mybullet\right|\right)^{\kappa_{0}}}^{1}\left(\R^{2}\right)\cap L^{\infty}\left(\R^{2}\right)$.
\item We have 
\begin{align*}
\widehat{\psi}\left(\xi\right)\neq0 & \text{ for all }\xi=\left(\xi_{1},\xi_{2}\right)\in\R^{2}\text{ with }\left|\xi_{1}\right|\in\left[3^{-1},3\right]\text{ and }\left|\xi_{2}\right|\leq\left|\xi_{1}\right|,\\
\widehat{\varphi}\left(\xi\right)\ne0 & \text{ for all }\xi\in\left[-1,1\right]^{2}.
\end{align*}
\item $\varphi,\psi$ satisfy equation \eqref{eq:ShearletFrameFourierDecayCondition}
for all $\theta\in\N_{0}^{2}$ with $\left|\theta\right|\leq N_{0}$,
where $N_{0}:=\left\lceil \kappa_{0}+p_{0}^{-1}\cdot\left(2+\varepsilon\right)\right\rceil $
and
\begin{align*}
K & :=\varepsilon+\max\left\{ \frac{1-\alpha}{\min\left\{ p_{0},q_{0}\right\} }+2\left(\frac{2}{p_{0}}+\kappa_{0}+N_{0}\right)-s_{0},\,\frac{2}{\min\left\{ p_{0},q_{0}\right\} }+\frac{2}{p_{0}}+\kappa_{0}+N_{0}\right\} ,\\
M_{1} & :=\varepsilon+\frac{1}{\min\left\{ p_{0},q_{0}\right\} }+\max\left\{ 0,\,s_{1}\right\} ,\\
M_{2} & :=\max\left\{ 0,\,\varepsilon+\left(1+\alpha\right)\left(\frac{2}{p_{0}}+\kappa_{0}+N_{0}\right)-s_{0}\right\} ,\\
H & :=\max\left\{ 0,\,\varepsilon+\frac{1-\alpha}{\min\left\{ p_{0},q_{0}\right\} }+\frac{2}{p_{0}}+\kappa_{0}+N_{0}-s_{0}\right\} .
\end{align*}
\end{itemize}
Then there is some $\delta_{0}\in\left(0,1\right]$ such that for
$0<\delta\leq\delta_{0}$ and all $p,q\in\left(0,\infty\right]$ and
$\kappa,s\in\R$ with $p\geq p_{0}$, $q\geq q_{0}$ and $s_{0}\leq s\leq s_{1}$,
as well as $0\leq\kappa\leq\kappa_{0}$, the following is true: The
family 
\[
{\rm SH}_{\alpha}\!\left(\smash{\tilde{\varphi},\tilde{\psi}};\delta\right)\!=\!\left(\!L_{\delta\cdot B_{v}^{-T}k}\:\widetilde{\gamma^{\left[v\right]}}\right)_{\!v\in V^{\left(\alpha\right)},k\in\Z^{2}}\;\text{ with }\;\widetilde{\gamma^{\left[v\right]}}\left(x\right)=\gamma^{\left[v\right]}\!\left(-x\right)\;\text{ and }\;\gamma^{\left[v\right]}\!:=\!\begin{cases}
\left|\det B_{v}\right|^{\frac{1}{2}}\cdot\left(\psi\circ B_{v}^{T}\right), & \text{if }v\in V_{0}^{\left(\alpha\right)},\\
\varphi, & \text{if }v=0
\end{cases}
\]
forms a Banach frame for $\mathscr{S}_{\alpha,s,\kappa}^{p,q}\left(\R^{2}\right)$.

The precise interpretation of this statement is as in Theorem \ref{thm:NicelySimplifiedAlphaShearletFrameConditions},
with the obvious changes. In particular, the coefficient space $C_{u^{s}}^{p,q}$
needs to be replaced by $C_{u^{s},\omega_{0}^{\kappa},\delta}^{p,q}$.
\end{thm}

\begin{thm}
\label{thm:WeightedAlphaShearletAtomicDecompositionCondition}Let
$\alpha\in\left[0,1\right]$, $\varepsilon,p_{0},q_{0}\in\left(0,1\right]$,
$\kappa_{0}\in\left[0,\infty\right)$ and $s_{0},s_{1}\in\R$ with
$s_{0}\leq s_{1}$. Assume that $\varphi,\psi:\R^{2}\rightarrow\Compl$
satisfy the following:

\begin{itemize}[leftmargin=0.6cm]
\item We have $\left\Vert \varphi\right\Vert _{\kappa_{0}+1+\frac{2}{p_{0}}}<\infty$
and $\left\Vert \psi\right\Vert _{\kappa_{0}+1+\frac{2}{p_{0}}}<\infty$,
where $\left\Vert g\right\Vert _{\Lambda}=\sup_{x\in\R^{2}}\left(1+\left|x\right|\right)^{\Lambda}\left|g\left(x\right)\right|$
for $g:\R^{2}\to\Compl^{\ell}$ (with arbitrary $\ell\in\N$) and
$\Lambda\geq0$.
\item We have $\widehat{\varphi},\widehat{\psi}\in C^{\infty}\left(\R^{2}\right)$,
where all partial derivatives of $\widehat{\varphi},\widehat{\psi}$
are polynomially bounded.
\item We have
\begin{align*}
\widehat{\psi}\left(\xi\right)\neq0 & \text{ for all }\xi=\left(\xi_{1},\xi_{2}\right)\in\R^{2}\text{ with }\left|\xi_{1}\right|\in\left[3^{-1},3\right]\text{ and }\left|\xi_{2}\right|\leq\left|\xi_{1}\right|,\\
\widehat{\varphi}\left(\xi\right)\ne0 & \text{ for all }\xi\in\left[-1,1\right]^{2}.
\end{align*}
\item $\varphi,\psi$ satisfy equation \eqref{eq:ShearletAtomicDecompositionFourierDecayCondition}
for all $\xi=\left(\xi_{1},\xi_{2}\right)\in\R^{2}$ and all $\beta\in\N_{0}^{2}$
with $\left|\beta\right|\leq N_{0}:=\left\lceil \kappa_{0}+p_{0}^{-1}\cdot\left(2+\varepsilon\right)\right\rceil $,
where
\begin{align*}
\qquad\qquad\Lambda_{0} & :=\begin{cases}
3+2\varepsilon+\max\left\{ 2,\,\frac{1-\alpha}{\min\left\{ p_{0},q_{0}\right\} }+N_{0}+s_{1}\right\} , & \text{if }p_{0}=1,\\
3+2\varepsilon+\max\left\{ 2,\,\frac{1-\alpha}{\min\left\{ p_{0},q_{0}\right\} }+\frac{1-\alpha}{p_{0}}+\kappa_{0}+N_{0}+1+\alpha+s_{1}\right\} , & \text{if }p_{0}\in\left(0,1\right),
\end{cases}\\
\qquad\qquad\Lambda_{1} & :=\varepsilon+\frac{1}{\min\left\{ p_{0},q_{0}\right\} }+\max\left\{ 0,\,\left(1+\alpha\right)\left(\frac{1}{p_{0}}-1\right)-s_{0}\right\} ,\\
\qquad\qquad\Lambda_{2} & :=\begin{cases}
\varepsilon+\max\left\{ 2,\,\left(1+\alpha\right)N_{0}+s_{1}\right\} , & \text{if }p_{0}=1,\\
\varepsilon+\max\left\{ 2,\,\left(1+\alpha\right)\left(1+\frac{1}{p_{0}}+\kappa_{0}+N_{0}\right)+s_{1}\right\} , & \text{if }p_{0}\in\left(0,1\right),
\end{cases}\\
\qquad\qquad\Lambda_{3} & :=\begin{cases}
\varepsilon+\max\left\{ \frac{1-\alpha}{\min\left\{ p_{0},q_{0}\right\} }+2N_{0}+s_{1},\,\frac{2}{\min\left\{ p_{0},q_{0}\right\} }+N_{0}\right\} , & \text{if }p_{0}=1,\\
\varepsilon+\max\left\{ \frac{1-\alpha}{\min\left\{ p_{0},q_{0}\right\} }+\frac{3-\alpha}{p_{0}}+2\kappa_{0}+2N_{0}+1+\alpha+s_{1},\,\frac{2}{\min\left\{ p_{0},q_{0}\right\} }+\frac{2}{p_{0}}+\kappa_{0}+N_{0}\right\} , & \text{if }p_{0}\in\left(0,1\right).
\end{cases}
\end{align*}
\end{itemize}
Then there is some $\delta_{0}\in\left(0,1\right]$ such that for
all $0<\delta\leq\delta_{0}$ and all $p,q\in\left(0,\infty\right]$
and $\kappa,s\in\R$ with $p\geq p_{0}$, $q\geq q_{0}$ and $s_{0}\leq s\leq s_{1}$,
as well as $0\leq\kappa\leq\kappa_{0}$, the following is true: The
family 
\[
{\rm SH}_{\alpha}\left(\varphi,\psi;\delta\right)=\left(L_{\delta\cdot B_{v}^{-T}k}\:\gamma^{\left[v\right]}\right)_{v\in V^{\left(\alpha\right)},\,k\in\Z^{2}}\quad\text{ with }\quad\gamma^{\left[v\right]}:=\begin{cases}
\left|\det B_{v}\right|^{1/2}\cdot\left(\psi\circ B_{v}^{T}\right), & \text{if }v\in V_{0}^{\left(\alpha\right)},\\
\varphi, & \text{if }v=0
\end{cases}
\]
forms an atomic decomposition for $\mathscr{S}_{\alpha,s,\kappa}^{p,q}\left(\R^{2}\right)$.
Precisely, this has to be understood as in Theorem \ref{thm:ReallyNiceShearletAtomicDecompositionConditions},
with the obvious changes. In particular, the coefficient space $C_{u^{s}}^{p,q}$
needs to be replaced by $C_{u^{s},\omega_{0}^{\kappa},\delta}^{p,q}$.
\end{thm}
\begin{rem}
\label{rem:WeightedBFDConditionsForSeparable}Of course, Remark \ref{rem:NiceTensorConditionsForUnconnectedCovering}
(cf.\@ Corollaries \ref{cor:ReallyNiceAlphaShearletTensorAtomicDecompositionConditions}
and \ref{cor:ReallyNiceAlphaShearletTensorBanachFrameConditions})
also applies in the current setting; one simply needs to replace the
old values of $N_{0}$ and $K,M_{1},M_{2},H$ or $\Lambda_{0},\dots,\Lambda_{3}$
with the modified ones.
\end{rem}
We can now finally show that the approximation rate stated in Theorem
\ref{thm:CartoonApproximationWithAlphaShearlets} can also be achieved
when restricting to polynomial search depth:
\begin{thm}
\label{thm:CartoonApproximationWithPolynomialSearchDepth}Let $\beta\in\left(1,2\right]$
be arbitrary and set $\alpha:=\beta^{-1}\in\left[0,1\right]$. Let
$\varepsilon\in\left(0,1\right]$ be arbitrary and set $\pi\left(x\right):=40000\cdot x^{14+4\left\lceil 1/\varepsilon\right\rceil }$
for $x\in\R$. There is an enumeration $\varrho:\N\to V^{\left(\alpha\right)}\times\Z^{2}$,
with the index set $V^{\left(\alpha\right)}$ from Definition \ref{def:UnconnectedAlphaShearletCovering},
such that the following is true:

Assume that $\varphi,\psi$ satisfy the assumptions of Theorem \ref{thm:WeightedAlphaShearletAtomicDecompositionCondition}
for the choices $p_{0}=q_{0}=\frac{2}{1+\beta}$, $\kappa_{0}=\varepsilon$
and $s_{0}=0$, as well as $s_{1}:=\frac{1}{2}\left(1+\beta\right)$
and for $\varepsilon$ as above. Then there is some $\delta_{0}\in\left(0,1\right]$
such that every $0<\delta\leq\delta_{0}$ satisfies the following:
If $\left(\gamma^{\left[v,k\right]}\right)_{v\in V^{\left(\alpha\right)},k\in\Z^{2}}={\rm SH}_{\alpha}\left(\varphi,\psi;\delta\right)$
denotes the $\alpha$-shearlet system generated by $\varphi,\psi$,
then there is for each $f\in\mathcal{E}^{\beta}\left(\R^{2}\right)$
and each $N\in\N$ a function $f_{N}$ which is a linear combination
of $N$ elements of the set $\left\{ \gamma^{\left[\varrho\left(n\right)\right]}\with n=1,\dots,\pi\left(N\right)\right\} $
and such that for all $\sigma,\nu>0$ there is a constant $C=C\left(\varphi,\psi,\delta,\varepsilon,\sigma,\nu,\beta\right)>0$
(independent of $f,N$) satisfying
\[
\left\Vert f-f_{N}\right\Vert _{L^{2}}\leq C\cdot N^{-\left(\frac{\beta}{2}-\sigma\right)}\qquad\forall f\in\mathcal{E}^{\beta}\left(\R^{2};\nu\right)\text{ and all }N\in\N.\qedhere
\]
\end{thm}
\begin{rem*}
Using Remark \ref{rem:WeightedBFDConditionsForSeparable}, one can
show similarly to Remark \ref{rem:CartoonApproximationConstantSimplification}
that the above theorem is applicable (with a suitable choice of $\varepsilon$),
if $\varphi,\psi$ satisfy the assumptions stated in Remark \ref{rem:CartoonApproximationConstantSimplification}.
\end{rem*}
\begin{proof}
Let $N\in\N$ be arbitrary and choose $n\in\N_{0}$ with $2^{n}\leq N<2^{n+1}$,
i.e., $n=\left\lfloor \log_{2}N\right\rfloor $. For $v\in V^{\left(\alpha\right)}$,
we denote by $s\left(v\right)$ the \emph{scale} encoded by $v$,
i.e., $s\left(0\right):=-1$ and $s\left(j,m,\iota\right):=j$ for
$\left(j,m,\iota\right)\in V_{0}^{\left(\alpha\right)}$. Then, we
define
\begin{equation}
W_{N}:=\left\{ \left(v,k\right)\in V^{\left(\alpha\right)}\times\Z^{2}\with s\left(v\right)\leq4n\text{ and }\left|B_{v}^{-T}k\right|\leq2^{2\left\lceil n/\varepsilon\right\rceil }\right\} .\label{eq:SearchDepthSpecialSetDefinition}
\end{equation}
Now, note that if $\left(v,k\right)=\left(\left(j,m,\iota\right),k\right)\in\left(\smash{V_{0}^{\left(\alpha\right)}}\times\smash{\Z^{2}}\right)\cap W_{N}$,
then $j\leq4n$ and $\left|m\right|\leq\left\lceil 2^{\left(1-\alpha\right)j}\right\rceil \leq2\cdot2^{\left(1-\alpha\right)j}$,
so that we get
\begin{align*}
\left|k\right| & =\left|B_{v}^{T}B_{v}^{-T}k\right|\leq\left\Vert B_{v}^{T}\right\Vert \cdot2^{2\left\lceil n/\varepsilon\right\rceil }\leq2^{2\left(1+n/\varepsilon\right)}\cdot\left\Vert \left(\begin{smallmatrix}2^{j} & 0\\
2^{\alpha j}m & 2^{\alpha j}
\end{smallmatrix}\right)\right\Vert \\
 & \leq4\cdot2^{2n/\varepsilon}\cdot\left(2^{j}+2^{\alpha j}+2^{\alpha j}\left|m\right|\right)\leq16\cdot2^{2n/\varepsilon}\cdot2^{j}\leq16\cdot2^{\left(4+2\left\lceil 1/\varepsilon\right\rceil \right)n},
\end{align*}
and thus $k\in\left\{ -16\cdot2^{n_{0}n},\dots,16\cdot2^{n_{0}n}\right\} ^{2}$,
where we defined $n_{0}:=4+2\left\lceil 1/\varepsilon\right\rceil \in\N$
for brevity. Furthermore, clearly $\left|m\right|\leq\left\lceil 2^{\left(1-\alpha\right)j}\right\rceil \leq2^{j}\leq2^{4n}$
and thus $m\in\left\{ -2^{4n},\dots,2^{4n}\right\} $. Finally, in
case of $\left(v,k\right)=\left(0,k\right)\in V^{\left(\alpha\right)}\times\Z^{2}$,
we get $\left|k\right|=\left|B_{v}^{-T}k\right|\leq2^{2\left\lceil n/\varepsilon\right\rceil }\leq2^{n_{0}n}$
and hence $k\in\left\{ -2^{n_{0}n},\dots,2^{n_{0}n}\right\} ^{2}$.
All in all, we have shown
\[
W_{N}\subset\left[\left\{ 0\right\} \times\left\{ -2^{n_{0}n},\dots,2^{n_{0}n}\right\} ^{2}\right]\cup\bigcup_{j=0}^{4n}\left[\left\{ j\right\} \times\left\{ -2^{4n},\dots,2^{4n}\right\} \times\left\{ 0,1\right\} \times\left\{ -16\cdot2^{n_{0}n},\dots,16\cdot2^{n_{0}n}\right\} ^{2}\right],
\]
and thus
\begin{align*}
\left|W_{N}\right| & \leq\left(1+2\cdot2^{n_{0}n}\right)^{2}+\left(1+4n\right)\cdot\left(1+2\cdot2^{4n}\right)\cdot2\cdot\left(1+32\cdot2^{n_{0}n}\right)^{2}\\
 & \leq\left(3\cdot2^{n_{0}n}\right)^{2}+5n\cdot3\cdot2^{4n}\cdot2\cdot\left(33\cdot2^{n_{0}n}\right)^{2}\\
\left({\scriptstyle \text{since }n\leq2^{n}\text{ and }2^{n}\leq N}\right) & \leq9\cdot2^{2n_{0}n}+30\cdot33^{2}\cdot2^{\left(5+2n_{0}\right)n}\leq40000\cdot2^{\left(5+2n_{0}\right)n}\leq40000\cdot N^{5+2n_{0}}.
\end{align*}

Next, note for arbitrary $\left(v,k\right)\in V^{\left(\alpha\right)}\times\Z^{2}$
that there is some $n\in\N_{0}$ with $s\left(v\right)\leq4n$ and
$\left|B_{v}^{-T}k\right|\leq2^{2\left\lceil n/\varepsilon\right\rceil }$,
so that $\left(v,k\right)\in W_{N}$ for $N=2^{n}$. Hence, $W:=V^{\left(\alpha\right)}\times\Z^{2}=\bigcup_{N\in\N}W_{N}$.
Now, choose the enumeration $\varrho:\N\to W$ such that $\varrho$
first enumerates $W_{1}$ (in an arbitrary way), then $W_{2}\setminus W_{1}$
(again arbitrarily), then $W_{3}\setminus\left(W_{1}\cup W_{2}\right)$,
and so on. Formally, if we define $M_{0}:=0$ and $M_{N}:=\left|W_{N}\setminus\bigcup_{\ell=1}^{N-1}W_{\ell}\right|\in\N_{0}$,
then $\varrho$ satisfies $\varrho\left(\underline{\sum_{\ell=1}^{N}M_{\ell}}\right)=\bigcup_{\ell=1}^{N}W_{\ell}$
for all $N\in\N$. Because of $\sum_{\ell=1}^{N}M_{\ell}\leq\sum_{\ell=1}^{N}\left|W_{\ell}\right|\leq40000\cdot\sum_{\ell=1}^{N}\ell^{5+2n_{0}}\leq40000\cdot N^{6+2n_{0}}=\pi\left(N\right)$,
we thus have $\varrho\left(\smash{\underline{\pi\left(N\right)}}\right)\supset\bigcup_{\ell=1}^{N}W_{\ell}\supset W_{N}$
for all $N\in\N$. For brevity, let us set $Z_{N}:=\varrho\left(\smash{\underline{\pi\left(N\right)}}\right)\subset W$
for $N\in\N$.

\medskip{}

We have thus constructed the enumeration $\varrho:\N\to V^{\left(\alpha\right)}\times\Z^{2}$
from the statement of the theorem. Now, let $\varphi,\psi$ be as
in the assumptions of the theorem. Then Theorem \ref{thm:WeightedAlphaShearletAtomicDecompositionCondition}
yields some $\delta_{0}\in\left(0,1\right]$ such that if $0<\delta\leq\delta_{0}$,
then the system ${\rm SH}_{\alpha}\left(\varphi,\psi;\delta\right)$
forms an atomic decomposition simultaneously for all $\alpha$-shearlet-smoothness
spaces $\mathscr{S}_{\alpha,s,\kappa}^{p,q}\left(\R^{2}\right)$ for
$p,q\geq p_{0}$, $0=s_{0}\leq s\leq s_{1}$ and $0\leq\kappa\leq\kappa_{0}=\varepsilon$.
Let $0<\delta\leq\delta_{0}$ be arbitrary and let $S^{\left(\delta\right)},C^{\left(\delta\right)}$
be the associated synthesis and coefficient operators. As noted in
Theorem \ref{thm:WeightedAlphaShearletAtomicDecompositionCondition}
(see Theorem \ref{thm:ReallyNiceShearletAtomicDecompositionConditions}),
the domain and codomain of these operators strictly speaking depend
on the choice of $p,q,s,\kappa$, but the \emph{action} of these operators
does not. Hence, we commit the weak notational crime of not indicating
this dependence.

For $f\in\mathcal{E}^{\beta}\left(\R^{2}\right)\subset L^{2}\left(\R^{2}\right)=\mathscr{S}_{\alpha,0,0}^{2,2}\left(\R^{2}\right)$,
let $c^{\left(f\right)}:=\left(\smash{c_{w}^{\left(f\right)}}\right)_{w\in W}:=C^{\left(\delta\right)}f\in C_{u^{0},\omega_{0}^{0},\delta}^{2,2}=\ell^{2}\left(W\right)$
and choose a subset $J_{N}^{\left(f\right)}\subset Z_{N}$ satisfying
$\left|\smash{J_{N}^{\left(f\right)}}\right|=N$ and $\left|\smash{c_{j}^{\left(f\right)}}\right|\geq\left|\smash{c_{i}^{\left(f\right)}}\right|$
for all $j\in J_{N}^{\left(f\right)}$ and all $i\in Z_{N}\setminus J_{N}^{\left(f\right)}$.
Such a choice is possible, since $Z_{N}$ is finite with $\left|Z_{N}\right|=\pi\left(N\right)\geq N$.
Finally, set 
\[
f_{N}:=S^{\left(\delta\right)}\left(\smash{\Indicator_{J_{N}^{\left(f\right)}}}\cdot\smash{c^{\left(f\right)}}\right).
\]
By definition of $S^{\left(\delta\right)}$, $f_{N}$ is then a linear
combination of $N$ elements of the set $\left\{ \gamma^{\left[\varrho\left(\ell\right)\right]}\with\ell=1,\dots,\pi\left(N\right)\right\} $,
as desired. It remains to verify the claimed approximation rate. Thus,
let $\sigma,\nu>0$ be arbitrary.

\medskip{}

We start with some preliminary considerations: In view of Remark \ref{rem:WeightedBFDConditionsForSeparable},
we see that there are symmetric, real-valued functions $\varphi_{0},\psi_{0}\in C_{c}\left(\R^{2}\right)$
which satisfy the assumptions of Theorem \ref{thm:WeightedAlphaShearletBanachFrameCondition}
for the choices of $p_{0},q_{0},\kappa_{0},s_{0},s_{1},\varepsilon$
from the current theorem. Hence, there is $\tau>0$ such that the
$\alpha$-shearlet system$\left(\theta^{\left[v,k\right]}\right)_{v\in V^{\left(\alpha\right)},k\in\Z^{2}}:={\rm SH}_{\alpha}\left(\varphi_{0},\psi_{0};\tau\right)$
forms a Banach frame for all $\alpha$-shearlet smoothness spaces
$\mathscr{S}_{\alpha,s,\kappa}^{p,q}\left(\R^{2}\right)$, for the
same range of parameters as above. Note that the distinction between
${\rm SH}_{\alpha}\left(\smash{\tilde{\varphi_{0}}},\smash{\tilde{\psi_{0}}};\tau\right)$
and ${\rm SH}_{\alpha}\left(\varphi_{0},\psi_{0};\tau\right)$ does
not matter by symmetry of $\varphi_{0},\psi_{0}$. As a consequence
of Lemma \ref{lem:SpecialConvolutionClarification} and of the symmetry
and real-valuedness of $\varphi_{0},\psi_{0}$, we then see that the
analysis operator $A^{\left(\delta\right)}$ from Theorem \ref{thm:WeightedAlphaShearletBanachFrameCondition}
satisfies $A^{\left(\delta\right)}f=\left(\left\langle f,\,\theta^{\left[v,k\right]}\right\rangle _{L^{2}}\right)_{v\in V^{\left(\alpha\right)},k\in\Z^{2}}$
for all $f\in L^{2}\left(\R^{2}\right)=\mathscr{S}_{\alpha,0,0}^{2,2}\left(\R^{2}\right)$
and thus in particular for $f\in\mathcal{E}^{\beta}\left(\R^{2}\right)$.

Now, for $v\in V^{\left(\alpha\right)}$ and $f\in\mathcal{E}^{\beta}\left(\R^{2};\nu\right)$,
we have $\left\Vert f\right\Vert _{L^{\infty}}\leq C_{1}=C_{1}\left(\nu\right)$
and thus
\[
\left|\left\langle f,\,\smash{\theta^{\left[v,k\right]}}\right\rangle _{L^{2}}\right|\leq\left\Vert f\right\Vert _{L^{\infty}}\cdot\left\Vert \smash{\theta^{\left[v,k\right]}}\right\Vert _{L^{1}}\leq C_{1}\cdot\left|\det\smash{B_{v}^{\left(\alpha\right)}}\right|^{-1/2}\cdot\max\left\{ \left\Vert \varphi_{0}\right\Vert _{L^{1}},\left\Vert \psi_{0}\right\Vert _{L^{1}}\right\} =C_{1}C_{2}\cdot u_{v}^{-\frac{1+\alpha}{2}},
\]
with $C_{2}:=\max\left\{ \left\Vert \varphi_{0}\right\Vert _{L^{1}},\left\Vert \psi_{0}\right\Vert _{L^{1}}\right\} $.
By the consistency statement of Theorem \ref{thm:WeightedAlphaShearletBanachFrameCondition}
(see Theorem \ref{thm:NicelySimplifiedAlphaShearletFrameConditions}),
this shows $f\in\mathscr{S}_{\alpha,0}^{\infty,\infty}\left(\R^{2}\right)$
with $\left\Vert f\right\Vert _{\mathscr{S}_{\alpha,0}^{\infty,\infty}}\leq C_{3}\cdot\left\Vert A^{\left(\delta\right)}f\right\Vert _{C_{u^{0}}^{\infty,\infty}}\leq C_{1}C_{2}C_{3}$
with $C_{3}:=\vertiii{R^{\left(\delta\right)}}_{C_{u^{0}}^{\infty,\infty}\to\mathscr{S}_{\alpha,0}^{\infty,\infty}}$,
with the reconstruction operator $R^{\left(\delta\right)}$ provided
by Theorem \ref{thm:WeightedAlphaShearletBanachFrameCondition} (for
$\varphi_{0},\psi_{0}$). Here, we used the easily verifiable identity
$C_{u^{0}}^{\infty,\infty}=\ell_{u^{\left(1+\alpha\right)/2}}^{\infty}\left(V^{\left(\alpha\right)}\times\Z^{2}\right)$,
where $u^{\left(1+\alpha\right)/2}=\left(\smash{u_{v}^{\left(1+\alpha\right)/2}}\right)_{v\in V^{\left(\alpha\right)}}$
is interpreted as a weight on $V^{\left(\alpha\right)}\times\Z^{2}$
in the obvious way.

Now, choose $A\geq1$ with $\supp\varphi_{0},\supp\psi_{0}\subset\left(-A,A\right)^{2}$.
Further, note that $R=R^{-1}=R^{T}=\left(\begin{smallmatrix}0 & 1\\
1 & 0
\end{smallmatrix}\right)$ preserves the $\ell^{\infty}$-norm, so that every $\left(j,m,\iota\right)\in V_{0}^{\left(\alpha\right)}$
satisfies 
\[
\left\Vert B_{j,m,\iota}^{-T}\right\Vert _{\ell^{\infty}\to\ell^{\infty}}=\left\Vert \left(\begin{smallmatrix}2^{-j} & 0\\
0 & 2^{-\alpha j}
\end{smallmatrix}\right)\left(\begin{smallmatrix}\vphantom{2^{-j}}1 & -m\\
0 & \vphantom{2^{-\alpha j}}1
\end{smallmatrix}\right)\right\Vert _{\ell^{\infty}\to\ell^{\infty}}=\left\Vert \left(\begin{smallmatrix}2^{-j} & -2^{-j}m\\
0 & 2^{-\alpha j}
\end{smallmatrix}\right)\right\Vert _{\ell^{\infty}\to\ell^{\infty}}\leq2^{-j}+2^{-\alpha j}+\left|-2^{-j}m\right|\leq3,
\]
since $\left|m\right|\leq\left\lceil 2^{\left(1-\alpha\right)j}\right\rceil \leq2^{j}$.
Further, clearly $\left\Vert B_{0}^{-T}\right\Vert _{\ell^{\infty}\to\ell^{\infty}}=\left\Vert \identity\right\Vert _{\ell^{\infty}\to\ell^{\infty}}=1\leq3$.
Now, since each $f\in\mathcal{E}^{\beta}\left(\R^{2}\right)$ satisfies
$\supp f\subset\left[-1,1\right]^{2}$, we see that $\left\langle f,\theta^{\left[v,k\right]}\right\rangle _{L^{2}}\neq0$
can only hold if
\begin{align*}
\emptyset & \subsetneq\left[-1,1\right]^{2}\cap\supp\theta^{\left[v,k\right]}\\
\left({\scriptstyle \text{with }\theta_{v}=\psi\text{ for }v\in V_{0}^{\left(\alpha\right)}\text{ and }\theta_{0}=\varphi}\right) & =\left[-1,1\right]^{2}\cap\supp L_{\tau\cdot B_{v}^{-T}k}\left[\theta_{v}\circ B_{v}^{T}\right]\\
 & \subset\left[-1,1\right]^{2}\cap\left[\tau\cdot B_{v}^{-T}k+B_{v}^{-T}\left(-A,A\right)^{2}\right],
\end{align*}
which implies $\tau\cdot B_{v}^{-T}k\in\left[-1,1\right]^{2}-B_{v}^{-T}\left(-A,A\right)^{2}\subset\left[-1,1\right]^{2}+3\left(-A,A\right)^{2}\subset\left[-4A,4A\right]^{2}$,
since $A\geq1$.

Hence, $\omega_{0}^{\kappa_{0}}\left(\tau\cdot B_{v}^{-T}k\right)=\left(1+\left|\tau\cdot B_{v}^{-T}k\right|\right)^{\varepsilon}\leq\left(1+8A\right)^{\varepsilon}\leq9A$
for all $\left(v,k\right)\in W$ with $\left\langle f,\,\theta^{\left[v,k\right]}\right\rangle _{L^{2}}\neq0$,
since $\varepsilon\leq1$. But Proposition \ref{prop:CartoonLikeFunctionsBoundedInAlphaShearletSmoothness}
shows because of $1\in\left(2/\left(1+\beta\right),2\right]$ that
$\mathcal{E}^{\beta}\left(\R^{2};\nu\right)\subset\mathscr{S}_{\alpha,\left(1+\alpha\right)\left(1-2^{-1}\right)}^{1,1}\left(\R^{2}\right)$
is bounded, i.e., $\left\Vert f\right\Vert _{\mathscr{S}_{\alpha,\left(1+\alpha\right)\left(1-2^{-1}\right)}^{1,1}\left(\R^{2}\right)}\leq C_{4}=C_{4}\left(\beta,\nu\right)$.
Since the associated coefficient space is $C_{u^{\left(1+\alpha\right)/2}}^{1,1}=\ell^{1}\left(W\right)$,
this implies $\left\Vert A^{\left(\delta\right)}f\right\Vert _{\ell^{1}}\leq C_{5}=C_{5}\left(\beta,\nu\right)$.
But since we just saw that $\omega_{0}^{\kappa_{0}}\left(\tau\cdot B_{v}^{-T}k\right)\leq9A$
for those $\left(v,k\right)\in W$ for which $\left(A^{\left(\delta\right)}f\right)_{v,k}\neq0$,
this implies $\left\Vert A^{\left(\delta\right)}f\right\Vert _{C_{u^{\left(1+\alpha\right)/2},\omega_{0}^{\kappa_{0}},\tau}^{1,1}}\leq9A\cdot C_{5}$
for all $f\in\mathcal{E}^{\beta}\left(\R^{2};\nu\right)$, as one
can see directly from Definition \ref{def:WeightedCoefficientSpace}.
By the consistency statement of Theorem \ref{thm:WeightedAlphaShearletBanachFrameCondition}
(see Theorem \ref{thm:NicelySimplifiedAlphaShearletFrameConditions}),
this shows as above that $f\in\mathscr{S}_{\alpha,\frac{1+\alpha}{2},\kappa_{0}}^{1,1}\left(\R^{2}\right)$
with $\left\Vert f\right\Vert _{\mathscr{S}_{\alpha,\frac{1+\alpha}{2},\kappa_{0}}^{1,1}}\leq C_{6}=C_{6}\left(\beta,\nu,\kappa_{0},\varphi_{0},\psi_{0},\tau\right)$,
for all $f\in\mathcal{E}^{\beta}\left(\R^{2};\nu\right)$. Here, we
used that $s_{0}=0\leq\frac{1+\alpha}{2}\leq\frac{1+\beta}{2}=s_{1}$,
since $\alpha\leq1\leq\beta$.

\medskip{}

Now, we continue with the proof of the approximation rate: Since we
have $p^{-1}-2^{-1}\to\beta/2$ as $p\downarrow\frac{2}{1+\beta}$
and $\frac{\beta}{2}-\sigma<\frac{\beta}{2}$, there is some $p=p\left(\beta,\sigma\right)\in\left(2/\left(1+\beta\right),\,2\right]$
with $p^{-1}-2^{-1}>\frac{\beta}{2}-\sigma$. By Proposition \ref{prop:CartoonLikeFunctionsBoundedInAlphaShearletSmoothness},
$\mathcal{E}^{\beta}\left(\R^{2};\nu\right)\subset\mathscr{S}_{\alpha,\left(1+\alpha\right)\left(p^{-1}-2^{-1}\right)}^{p,p}\left(\R^{2}\right)$
is bounded and the associated coefficient space to this $\alpha$-shearlet
smoothness space is $C_{u^{\left(1+\alpha\right)\left(p^{-1}-2^{-1}\right)}}^{p,p}=\ell^{p}\left(W\right)$,
so that we get $\left\Vert c^{\left(f\right)}\right\Vert _{\ell^{p}}=\left\Vert C^{\left(\delta\right)}f\right\Vert _{C_{u^{\left(1+\alpha\right)\left(p^{-1}-2^{-1}\right)}}^{p,p}}\leq C_{7}=C_{7}\left(\varphi,\psi,\beta,\delta,p,\nu\right)$.
Here, we used that $s_{0}=0\leq\left(1+\alpha\right)\left(p^{-1}-2^{-1}\right)\leq\left(1+\alpha\right)\left(\frac{1+\beta}{2}-\frac{1}{2}\right)=\frac{\beta+1}{2}=s_{1}$
and $p\geq p_{0}=\frac{2}{1+\beta}$, so that $\mathscr{S}_{\alpha,\left(1+\alpha\right)\left(p^{-1}-2^{-1}\right)}^{p,p}\left(\R^{2}\right)$
is in the ``allowed'' range.

Likewise, our considerations from above showed that $\mathcal{E}^{\beta}\left(\R^{2};\nu\right)$
is a bounded subset of $\mathscr{S}_{\alpha,0}^{\infty,\infty}\left(\R^{2}\right)$,
and of $\mathscr{S}_{\alpha,\frac{1+\alpha}{2},\kappa_{0}}^{1,1}\left(\R^{2}\right)$,
so that there are constants $C_{8},C_{9}$ (only dependent on $\varphi,\psi,\delta,\beta,\nu,\varepsilon$)
with $\left\Vert c^{\left(f\right)}\right\Vert _{C_{u^{\left(1+\alpha\right)/2},\omega_{0}^{\kappa_{0}},\delta}^{1,1}}\leq C_{8}$
and $\left\Vert c^{\left(f\right)}\right\Vert _{\ell_{u^{\left(1+\alpha\right)/2}}^{\infty}}\leq C_{9}$,
since $C_{u^{0}}^{\infty,\infty}=\ell_{u^{\left(1+\alpha\right)/2}}^{\infty}\left(V^{\left(\alpha\right)}\times\Z^{2}\right)$.
Finally, set $C_{10}:=\vertiii{S^{\left(\delta\right)}}_{\ell^{2}\to L^{2}}$.

Because of $S^{\left(\delta\right)}\circ C^{\left(\delta\right)}=\identity_{\mathscr{S}_{\alpha,0}^{2,2}}=\identity_{L^{2}}$
and since $c^{\left(f\right)}=C^{\left(\delta\right)}f$, we have
\begin{align}
\left\Vert f-f_{N}\right\Vert _{L^{2}} & =\left\Vert S^{\left(\delta\right)}\left[\smash{c^{\left(f\right)}}-\smash{\Indicator_{J_{N}^{\left(f\right)}}\cdot c^{\left(f\right)}}\right]\right\Vert _{L^{2}}\leq C_{10}\cdot\left\Vert \smash{c^{\left(f\right)}}-\smash{\Indicator_{J_{N}^{\left(f\right)}}\cdot c^{\left(f\right)}}\right\Vert _{\ell^{2}\left(W\right)}\nonumber \\
\left({\scriptstyle \text{since }J_{N}^{\left(f\right)}\subset Z_{N}}\right) & \leq C_{10}\cdot\left(\left\Vert c^{\left(f\right)}\right\Vert _{\ell^{2}\left(W\setminus Z_{N}\right)}+\left\Vert c^{\left(f\right)}-\Indicator_{J_{N}^{\left(f\right)}}\cdot c^{\left(f\right)}\right\Vert _{\ell^{2}\left(Z_{N}\right)}\right).\label{eq:PolynomialSearchDepthSubdivision}
\end{align}
Now, our choice of the set $J_{N}^{\left(f\right)}$, together with
Stechkin's estimate (see e.g.\@ \cite[Proposition 2.3]{RauhutCompressiveSensing}),
shows 
\[
\left\Vert c^{\left(f\right)}-\Indicator_{J_{N}^{\left(f\right)}}\cdot c^{\left(f\right)}\right\Vert _{\ell^{2}\left(Z_{N}\right)}\leq N^{-\left(\frac{1}{p}-\frac{1}{2}\right)}\cdot\left\Vert c^{\left(f\right)}\right\Vert _{\ell^{p}\left(Z_{N}\right)}\leq C_{7}\cdot N^{-\left(\frac{1}{p}-\frac{1}{2}\right)}\leq C_{7}\cdot N^{-\left(\frac{\beta}{2}-\sigma\right)},
\]
since $p^{-1}-2^{-1}\geq\frac{\beta}{2}-\sigma$, so that it suffices
to further estimate the first term in equation \eqref{eq:PolynomialSearchDepthSubdivision}.

But for $\left(v,k\right)\in W\setminus Z_{N}\subset W\setminus W_{N}$,
we have $s\left(v\right)\geq4n$ (and thus in particular $v\in V_{0}^{\left(\alpha\right)}$),
or $\left|B_{v}^{-T}k\right|>2^{2\left\lceil n/\varepsilon\right\rceil }$,
where we recall that $2^{n}\leq N<2^{n+1}$. In the first case, we
have $\left|c_{v,k}^{\left(f\right)}\right|\leq C_{9}\cdot u_{v}^{-\left(1+\alpha\right)/2}\leq C_{9}\cdot u_{v}^{-1/2}\leq C_{9}\cdot2^{-2n}$
and in the second case, we get $\omega_{0}^{\kappa_{0}}\left(\delta\cdot B_{v}^{-T}k\right)=\left(1+\left|\delta\cdot B_{v}^{-T}k\right|\right)^{\varepsilon}\geq\delta^{\varepsilon}\cdot2^{2n}\geq\delta\cdot2^{2n}$
and thus 
\[
\left|c_{v,k}^{\left(f\right)}\right|^{2}\leq C_{9}\cdot\left|c_{v,k}^{\left(f\right)}\right|\leq C_{9}\cdot\frac{2^{-2n}}{\delta}\cdot\omega_{0}^{\kappa_{0}}\left(\delta\cdot B_{v}^{-T}k\right)\cdot\left|c_{v,k}^{\left(f\right)}\right|.
\]
Therefore,
\begin{align*}
\left\Vert c^{\left(f\right)}\right\Vert _{\ell^{2}\left(W\setminus Z_{N}\right)}^{2} & \leq\sum_{\substack{v\in V_{0}^{\left(\alpha\right)}\\
\text{ with }s\left(v\right)\geq4n
}
}\;\sum_{k\in\Z^{2}}\left|c_{v,k}^{\left(f\right)}\right|^{2}+\sum_{v\in V_{0}^{\left(\alpha\right)}}\:\sum_{\substack{k\in\Z^{2}\\
\text{with }\left|B_{v}^{-T}k\right|>2^{2\left\lceil n/\varepsilon\right\rceil }
}
}\left|c_{v,k}^{\left(f\right)}\right|^{2}\\
 & \leq C_{9}\cdot2^{-2n}\sum_{v\in V^{\left(\alpha\right)}}\;\sum_{k\in\Z^{2}}\left|c_{v,k}^{\left(f\right)}\right|+\frac{C_{9}}{\delta}\cdot2^{-2n}\cdot\sum_{v\in V^{\left(\alpha\right)}}\:\sum_{k\in\Z^{2}}\omega_{0}^{\kappa_{0}}\left(\delta\cdot B_{v}^{-T}k\right)\left|c_{v,k}^{\left(f\right)}\right|\\
\left({\scriptstyle \text{since }\omega_{0}^{\kappa_{0}}\geq1}\right) & \leq C_{9}\cdot\left(1\!+\!\delta^{-1}\right)\cdot2^{-2n}\cdot\left\Vert \left(\!\left|\det B_{v}\right|^{\frac{1}{2}-\frac{1}{1}}\cdot u_{v}^{\frac{1+\alpha}{2}}\cdot\left\Vert \left(\omega_{0}^{\kappa_{0}}\left(\delta\cdot B_{v}^{-T}k\right)\cdot c_{v,k}^{\left(f\right)}\right)_{\!k\in\Z^{2}}\right\Vert _{\ell^{1}}\right)_{\!\!v\in V^{\left(\alpha\right)}}\right\Vert _{\ell^{1}}\\
 & =C_{9}\cdot\left(1+\delta^{-1}\right)\cdot2^{-2n}\cdot\left\Vert c^{\left(f\right)}\right\Vert _{C_{u^{\left(1+\alpha\right)/2},\omega_{0}^{\kappa_{0}},\delta}^{1,1}}\\
 & \leq C_{8}C_{9}\cdot\left(1+\delta^{-1}\right)\cdot2^{-2n}\\
\left({\scriptstyle \text{since }N\leq2^{n+1}\text{ and }\frac{\beta}{2}-\sigma\leq\frac{\beta}{2}\leq1}\right) & \leq4C_{8}C_{9}\cdot\left(1+\delta^{-1}\right)\cdot N^{-2}\leq4C_{8}C_{9}\cdot\left(1+\delta^{-1}\right)\cdot N^{-2\left(\frac{\beta}{2}-\sigma\right)}.
\end{align*}
Taking the square root and recalling equation \eqref{eq:PolynomialSearchDepthSubdivision}
finishes the proof.
\end{proof}
\bibliographystyle{plain}
\bibliography{felixbib}

\end{document}